\documentclass[12pt]{article}
\usepackage{graphicx}
\usepackage{color}
\usepackage{amsmath}
\usepackage{amssymb}
\usepackage{amscd}
\usepackage{amsthm}
\usepackage{amsopn}
\usepackage{xspace}
\usepackage{verbatim}
\usepackage{amsfonts}
\usepackage{epic}
\usepackage{eepic}

\usepackage{empheq}
\usepackage[T3,T1]{fontenc}
\DeclareSymbolFont{tipa}{T3}{cmr}{m}{n}
\DeclareMathAccent{\invbreve}{\mathalpha}{tipa}{16}
\usepackage[active]{srcltx}
\usepackage[a4paper,margin=3cm]{geometry}
\definecolor{red}{rgb}{1,0,0}
\definecolor{green}{rgb}{0,1,0}
\definecolor{SeaGreen}{RGB}{46,139,87}
\definecolor{Maroon}{RGB}{128,0,0}
\definecolor{blush}{rgb}{0.87, 0.36, 0.51}

\newcommand{\C}{{\mathbb{C}}}
\newcommand{\R}{{\mathbb{R}}}

\newcommand{\A}{{\mathcal A}}
\newcommand{\B}{\mathcal B}
\newcommand{\CC}{\mathcal C}

\newcommand{\LL}{\mathcal L}

\def\mf {{\mathfrak m}}
\def\Mg {{\mathcal M}}
\newcommand{\OO}{\mathcal O}

\newcommand{\QQ}{\mathcal Q}

\def\Sg {{\mathcal S}}

\def\Tg {{\mathcal T}}

\def\Jf {{\mathfrak J}}

\def\Df {{\mathfrak D}}
\def\Nf {{\mathfrak N}}

\def\curl{\text{\rm curl}}

\def\curl{\text{\rm curl\,}}

\def\Div{\text{\rm div\,}}

\newcommand{\supp}{\rm Supp\,}

\def\Ai{\text{\rm Ai\,}}

\def\0{\mathbf  0}

\def\XXint#1#2#3{{\setbox0=\hbox{$#1{#2#3}{\int}$ }
\vcenter{\hbox{$#2#3$ }}\kern-.6\wd0}}

 

\errorcontextlines=0 \numberwithin{equation}{section}
\theoremstyle{plain}
\newtheorem{theorem}{Theorem}[section]
\newtheorem{lemma}[theorem]{Lemma}

\newtheorem{proposition}[theorem]{Proposition}

\newtheorem{remark}[theorem]{Remark}
\newtheorem{corollary}[theorem]{Corollary}

\title{On the stability of symmetric flows in a two-dimensional channel}
 \author{ Y. Almog, Department of
  Mathematics, \\ Braude College of Engineering, \\ 
    Carmiel 2161002, Israel \\~\\
  and \\~\\
\noindent   B. Helffer, Laboratoire de Math\'ematiques Jean Leray, \\CNRS and  Nantes Universit\'e, \\
 44000 Nantes  Cedex France}

\begin{document}
 \maketitle
\begin{abstract}
We consider the stability of symmetric flows in a two-dimensional
  channel (including the Poiseuille flow). In 2015 Grenier, Guo, and
  Nguyen have established instability of these flows in a particular
  region of the parameter space, affirming formal asymptotics results
  from the 1940's. We prove that these flows are stable outside this
  region in parameter space. More precisely we show that the
  Orr-Sommerfeld operator
\end{abstract}



\newpage 
\tableofcontents
\newpage

{\bf Acknowledgements:} We thank the referees for their time, efforts, and comments.

\section{Introduction}
\label{sec:1}
\subsection{Main results}
   Consider the incompressible Navier-Stokes equations in the
   two-dimensional pipe  $D=\R\times (-1,1)$
   \begin{equation}
   \label{eq:1}
   \begin{cases}
   \partial_t {\mathbf v} - \epsilon \Delta {\mathbf v} + {\mathbf v} \cdot \nabla {\mathbf v} = -
   \nabla p & \text{in } \R_+\times D  \\
   {\mathbf v}=v_b\; \hat{i}_1  & \text{on } \R_+\times\partial D  \,,
   \end{cases} 
   \end{equation}
   where  $\hat{i}_1 = (1,0)$, ${\mathbf v}=(v_1,v_2)$ is the fluid
   velocity, and $p$ is the pressure. \\
   The parameter
    \begin{equation} \label{defreynolds}
    R :=\frac 1 \epsilon
    \end{equation}  is the Reynolds
    number of the  flow and 
    \begin{displaymath}
      v_b:\partial D\to\R
    \end{displaymath}
is the boundary velocity.\\ Since the
    flow is incompressible we must have
   \begin{displaymath}
     \Div {\mathbf v}=0\,.
   \end{displaymath}
   We linearize \eqref{eq:1} near the laminar flow (cf. \cite{almog2019stability}) 
   \begin{displaymath}
     {\mathbf v}=U(x_2)\hat{i}_1 \,,
   \end{displaymath}
   to obtain the linearized equation 
\begin{displaymath}
  {\bf u}_t-\mathcal T_0({\bf u},q)=0
\end{displaymath}
   where  ${\bf u}=(u_1,u_2)$ and $q$  are defined on $\mathbb R_+\times D$, and $\mathcal T_0$ is the map
     \begin{equation}
   \label{eq:2}
   ({\bf u},q) \mapsto  {\mathcal T}_0 ({\bf u} , q ):=  -  \epsilon \,\Delta {\mathbf u} + U\, \frac{\partial{\mathbf
       u}}{\partial x_1}+ \, u_2\, U^\prime\, \hat{i}_1 - \nabla q\,.
   \end{equation}
   We proceed with a formal derivation of the Orr-Sommerfeld equation,
   intentionally skipping the definitions of ${\bf v}$, $p$, ${\bf u}$
   and $q$. Interested readers can read the entire derivation in
   \cite{almog2019stability}. The associated resolvent equation for
 $\mathcal T_0$ assumes the form
   \begin{equation}
   \label{eq:3}
     \Tg_0({\bf u},q)-\Lambda{\bf u}={\bf f} \,,
   \end{equation}
   where $\Div {\bf u}=0$ and  $\Lambda\in\C$ is the spectral parameter.\\ 
    Hence, we may define a stream function
   \begin{displaymath}
     {\mathbf u}=\nabla_\perp\psi=(-\psi_{x_2},\psi_{x_1}) \,.
   \end{displaymath}
  Substituting the above into \eqref{eq:3} and then taking the curl of
  the ensuing equation for $\psi$ yields
  \begin{equation}
  \label{eq:4}
    \Big(-  \epsilon \Delta^2 + U\frac{\partial}{\partial x_1}\Delta -
    U^{\prime\prime}\frac{\partial}{\partial x_1}  - \Lambda \,  \Delta\Big)\psi=F \,,
  \end{equation}
where $F=\curl f$. 

  We consider $U\in C^2([-1,1])$ (we later restrict ourselves to
  $U\in C^4([-1,1])$) satisfying the following
  \begin{subequations}
  \label{eq:5}
    \begin{equation}
      U(\pm1)=0 \quad ; \quad \max_{x\in[-1,1]}U^{\prime\prime}(x)<0 \quad ; \quad
      U(-x)=U(x)\,. \tag{\ref{eq:5}a,b,c}
    \end{equation}
  We normalize $U$ so that
  \begin{equation}
  U^\prime(\pm1)=\mp1 \,.  \tag{\ref{eq:5}d}
  \end{equation}
  \end{subequations}
  Substituting $\psi(x_1,x_2)=\phi(x_2) \, e^{i\alpha x_1}$ into
  \eqref{eq:4} yields for $ \phi:(-1,1)\to\C$  the equation
  \begin{subequations}
  \begin{equation}
    \B_{\lambda,\alpha,\beta}^{\mathfrak D}\phi=f \,, 
  \end{equation}
  where (setting $x_2=x$)
   \begin{equation}
\label{eq:6}
     \B_{\lambda,\alpha,\beta}^{\mathfrak D} =(\LL_\beta -\beta\lambda)\Big(\frac{d^2}{dx^2}-\alpha^2\Big)  -i\beta U^{\prime\prime} \,,
   \end{equation}
   where
  \begin{equation}
  \label{eq:p29}
    \LL_\beta  = -\frac{d^2}{dx^2}+i\beta U\,.
  \end{equation}
  \end{subequations}
In the above 
\begin{equation}\label{defbeta}
\beta = \alpha \epsilon^{-1} =\alpha R
\end{equation}
 ($R$ being the Reynolds
     number introduced in \eqref{defreynolds}), and, for $\beta \neq 0$, 
     \begin{equation}
\label{eq:569}
     \lambda = \beta^{-1} \Big( \frac{\Lambda}{\epsilon} -\alpha^2\Big)\,.
     \end{equation}
     We refer to Section 3 in \cite{almog2019stability} for the
     details of the derivation.  We use the pair of parameters
     $(\alpha,\beta)$ instead of $(\alpha,R)$ since the asymptotic limit we
     consider in the sequel is
     $\beta\to\infty$. \\
     We define $ \B_{\lambda,\alpha,\beta}^{\mathfrak D}$ on
  \begin{equation}\label{eq:7}
  D(\B_{\lambda,\alpha,\beta}^{\mathfrak D})=\{u\in H^4(-1,1)\,,\, u(1)=u^\prime(1)= u(-1) =u^\prime (-1) =0 \},
  \end{equation}

  We focus our interest on the restriction
  $\B^{\mathfrak D,sym}_{\lambda,\alpha,\beta}$ of the operator $\B^{\mathfrak
    D}_{\lambda,\alpha,\beta}$ to functions that are symmetric with respect to
  the reflection $x\mapsto -x$. Hence we are led to consider the equivalent
  restricted operator $\B_{\lambda,\alpha,\beta}^{\mathfrak N, \mathfrak D}$ on
  $(0,1)$ whose domain is
  \begin{equation}\label{eq:p29ab}
 D(\B_{\lambda,\alpha,\beta}^{\mathfrak N, \mathfrak D})=\{u\in H^4(0,1)\,,\,  u^\prime(0)=u^{(3)}(0)=0  \mbox{ and }\,  u(1)=u^\prime(1)=0\}.
 \end{equation}
 We leave the discussion of anti-symmetric modes to future research.
 Note that since
 \begin{displaymath}
   \B_{\lambda,\alpha,\beta}^U =  \overline{\B_{\bar{\lambda},\alpha,\beta}^{-U}} \,,
 \end{displaymath}
(where $\bar{\B}$ denotes the complex conjugate of $\B$) 
 the analysis in the sequel applies to the case where
 $\min_{x\in[-1,1]}U^{\prime\prime}(x)>0$ as well. Clearly, the Poiseuille flow associated with 
 $U(x)=(1-x^2)/2$ meets all the criteria in \eqref{eq:5}.  
Another example mentioned in \cite{grenier2016spectral} is given by  $ U(x)=(2/\pi) \cos
\pi x/2$. Note that $U$ is decreasing on $(0,1)$. \\

In \cite[Theorem 1.1]{grenier2016spectral} it has been established by Grenier-Guo-Nguyen
that for sufficiently large $R$ and for each $\alpha$ satisfying
  \begin{displaymath}
      C_L\, R^{-1/7}\leq\alpha\leq C_R\, R^{-1/11}\,,
    \end{displaymath} 
or equivalently when $\beta$ is large and
\begin{displaymath}
   C_L^{7/6}\, \beta^{-1/6}\leq\alpha\leq C_R^{11/10}\, \beta^{-1/10} \,,
 \end{displaymath}
    there exists $\lambda\in\C$ with negative real part such that
    $\B_{\lambda,\alpha,\beta}^{\mathfrak N,\Df}$ is not invertible.  For the case
    of a Poiseuille flow, the positive constants $C_L$ and $C_R$ have
    been determined from well-known formal
    asymptotic calculations (cf.  the book \cite{drre04}  by P.~G. Drazin and W.~H. Reid).   \\
     
 In the present contribution we consider the converse problem, i.e.
 we attempt to show that, for any $\delta>0$, $(\B_{\lambda,\alpha,\beta}^{\Df,sym})^{-1}$
 is bounded for $\Re\lambda\leq0$ when
 \begin{displaymath}
   \alpha\geq \,\beta^{-1/10+\delta} \text{ or } 0\leq\alpha\leq \alpha_L\, \beta^{-1/6}\,.
 \end{displaymath}
 Note that unlike \cite{grenier2016spectral} we do not provide the
 precise estimate derived by formal asymptotics, i.e, $\alpha_L<C_L$ and
 $\beta^{-1/10+\delta} >\beta^{1/10}$. The determination of the precise curves
 is left to future research.

Recall from \eqref{defbeta} that $\beta=\alpha/\epsilon$. Our main results are
 \begin{theorem}
 \label{thm:small-alpha}
   Let $U\in C^4([-1,1])$ satisfy \eqref{eq:5}. Then there exist positive
   $\alpha_L$, $C$, $\Upsilon$, and $\beta_0>1$ such that for all $\beta>\beta_0$
   it holds that
   \begin{equation}
     \label{eq:8}
        \sup_{
          \begin{subarray}{c}
           0\leq \alpha \leq \alpha_L\beta^{-1/6}\\
       \Re \lambda <\Upsilon\beta^{-1/2} 
             \end{subarray}}
  \big\|(\B_{\lambda,\alpha,\beta}^{\Df,sym})^{-1}\big\|+
        \big\|\frac{d}{dx}\, (\B_{\lambda,\alpha,\beta}^{\Df,sym})^{-1}\big\|\leq
          C\, \beta^{-1/2} \log \beta \,,
   \end{equation}
 \end{theorem}
The condition $\alpha\in[0,\alpha_L\beta^{-1/6}]$ can be rephrased in terms
  of the pair $(\alpha,R)$ as $\alpha\in[0,\alpha_L^{6/7}R^{-1/7}]$.

 \begin{theorem}
 \label{thm:large-alpha}
 Let $U\in C^4([-1,1])$ satisfy \eqref{eq:5}. {Let further
   $\hat{\mu}_m>0$ be given by \cite[Eq. (6.57)]{almog2019stability}.}
 Then for any $\delta >0$ { and any $\hat \Upsilon >0$,} there exist positive
 $C$, $\Upsilon$, and $\beta_0$ such that for all $\beta>\beta_0$ we have
   \begin{equation}
  \label{eq:9}
        \sup_{
          \begin{subarray}{c}
          \beta^{-1/10 +\delta}\leq \alpha \\
 { \Re\lambda\leq\min(\Upsilon\beta^{-1/2},\beta^{-1/3}[\hat{\mu}_m-\hat \Upsilon-\alpha^2\beta^{-2/3}/2])}
             \end{subarray}}
  \big\|(\B_{\lambda,\alpha,\beta}^{\Df,sym})^{-1}\big\|+
        \Big\|\frac{d}{dx}\, (\B_{\lambda,\alpha,\beta}^{
          \Df,sym})^{-1}\Big\|\leq  C\beta^{-1/2+\delta} \,.
    \end{equation}
 \end{theorem}
The condition $\alpha\geq\beta^{-1/10 +\delta}$ can be rephrased in terms
  of the pair $(\alpha,R)$ as $\alpha\geq R^{(-1+10\delta)/(11-10\delta)}$. {  Note that
    the condition $\Re\lambda\leq\beta^{-1/3}[\hat{\mu}_m-\hat \Upsilon-\alpha^2\beta^{-2/3}/2])$
  guarantees, by (\ref{eq:569}) that $\B_{\lambda,\alpha,\beta}^{\Df,sym}$ is
  invertible for $\Lambda\leq\alpha\beta^{-1/3}[\hat{\mu}_m+ \alpha^2\beta^{-2/3}/2-\hat \Upsilon]$ and hence the stability of
  the laminar flow even for $\alpha\gtrsim\beta^{1/3}$.}

In the above
  \begin{displaymath}
    \big\|(\B_{\lambda,\alpha,\beta}^{\Df,sym})^{-1}\big\|=
\sup_{
  \begin{subarray}{c}
 f\in L^2(0,1) \\
\|f\|_2=1   
  \end{subarray}}
\|(\B_{\lambda,\alpha,\beta}^{\Df,sym})^{-1}f\|_2
  \end{displaymath}
and 
\begin{displaymath}
   \Big\|\frac{d}{dx}\, (\B_{\lambda,\alpha,\beta}^{\Df,sym})^{-1}\Big\|
   =\sup_{
  \begin{subarray}{c}
 f\in L^2(0,1) \\
\|f\|_2=1   
  \end{subarray}}\Big\|\frac{d}{dx}\,
   (\B_{\lambda,\alpha,\beta}^{\Df,sym})^{-1}f\Big\|_2 \,,
\end{displaymath}
where $\|\cdot\|_2$ denotes the standard $L^2(0,1)$ norm. 

 In the recent years, hydrodynamic stability of shear flows has attracted
 significant attention. For the case of a Couette flow we mention only a
 partial list of rigorous analytical results
 \cite{BGM,Mas17,chen2018transition}.  In \cite{almog2019stability} we
 have established similar estimates for the Orr-Sommerfeld operator,
 together with semigroup estimates for the linearized Navier-Stokes
 operator in the case where $|U^\prime|>0$ in $[-1,1]$ (see also the works
 of Chen, Wei, and Zhang
 \cite{chen2022linear} and of Jia \cite{jia2023uniform} for recent
 generalizations). In contrast with the present case the
 Orr-Sommerfeld operator has, when $|U^{\prime\prime}|>0$, a bounded resolvent
 in the half-plane $\Re\lambda\leq0$.\\

The hydrodynamic stability of symmetric flows in a channel has been
considered extensively in the Physics literature (cf. for instance
\cite{drre04,orszag1971accurate,li44,to47}). These works, just like
that of  \cite{grenier2016spectral}, all attempt to determine as function of $\beta$ the
region in the $(\alpha,\Im\lambda)$ plane where the Orr-Sommerfeld is
unstable. In a recent work \cite{ding2022enhanced}, the stability
of Poiseuille flow has been established in the case of a Navier-slip boundary
condition. This means that the boundary condition $u^\prime(\pm1)=0$ in
\eqref{eq:p29ab} is replaced by  $u^{\prime\prime}(\pm1)=0$. The stability of a 
pipe Poiseuille flow has also been addressed in
\cite{chen2019linear}. 

\subsection{Proof strategy}
\label{sec:proof-strategy}
In the following informal discussion, we
  present the main ingredients of the proofs of Theorems
  \ref{thm:small-alpha} and \ref{thm:large-alpha}.  Some of the
  definitions appearing in the discussion will remain slightly vague and will
  be reformulated more precisely in the next sections. The reader may
  be interested in reviewing the relevant part of this
  presentation before diving into the technical details of any
  part of the analysis in the sequel.

   We begin with a
  rather heuristic discussion. Consider the equation
  \begin{equation}
\label{eq:924}
    \B_{\lambda,\alpha,\beta}\phi=f \,,
  \end{equation}
where $(\phi,f)\in D( \B_{\lambda,\alpha,\beta}^{\Nf,\Df})\times L^2(0,1)$ and $
\B_{\lambda,\alpha,\beta}= \B_{\lambda,\alpha,\beta}^{\Nf,\Df}$. \\
Our goal is to estimate  $  \B_{\lambda,\alpha,\beta}^{-1}$ 
 under specific conditions on the parameters $(\lambda,\alpha,\beta)\in \mathbb C \times \mathbb R_+\times \mathbb R_+$.
 We may rewrite the above
equation in the form
\begin{subequations}\label{eq:defphiv}
\begin{equation}
  \A_{\lambda,\alpha}\phi=v \,,
\end{equation}
where
\begin{equation}
  v=\beta^{-1}[f+ \phi^{(4)}-\alpha^2\phi^{\prime\prime}] \,,
\end{equation}
and
\begin{equation}
   \A_{\lambda,\alpha}=(U+i\lambda)\Big(-\frac{d^2}{dx^2}+ \alpha^2\Big)+U^{\prime\prime}\,,
\end{equation}
is the inviscid (Rayleigh) operator whose study will be the main object of Section~2. We define $ \A_{\lambda,\alpha}$  for $\Re\lambda\neq0$ or when $\Re\lambda=0$ and
$\Im\lambda\notin[0,U(0)]$, on 
\begin{equation}
  D(\A_{\lambda,\alpha}^{\mathfrak N,\Df})=\{\phi\in H^2(0,1)\,|\,\phi^\prime(0)=0 \mbox{ and }\phi(1)=0 \;\} \,.
\end{equation}
\end{subequations}
It intuitively appears that $v$ should tend to $0$ as $\beta\to\infty$, and thus,
we adopt the following strategy of proof
\begin{enumerate}
\item We prove that $v$ becomes small as $\beta\to\infty$.
\item We obtain a bound for $ \|\A_{\lambda,\alpha}^{-1}v\|_{1,2}$ (where
    $\|\cdot\|_{1,2}$ denotes the standard $H^1(0,1)$ norm).
\end{enumerate}
After successfully completing the above stages we expect to obtain an
inequality of the form
\begin{displaymath}
  \|\phi^\prime\|_2\leq \delta_1(\beta)\|f\|_2+\delta_2(\beta)\|\phi^\prime\|_2 \,.
\end{displaymath}
If for sufficiently large $\beta$ it holds that $\delta_2(\beta)<1/2$, we can
conclude from here  an estimate for
$\|\B_{\lambda,\alpha,\beta}^{-1}\|+\|\frac{d}{dx} \B_{\lambda,\alpha,\beta}^{-1}\|$\,. \\

\noindent {\em Estimation of $\A_{\lambda,\alpha}^{-1}$:}\\

We use, in Section 2, a similar procedure to the one used in
\cite{almog2019stability}. Let $\lambda=\mu+i\nu$. Given that $|U^{\prime\prime}|>0$
in $[0,1]$ and since 
\begin{equation}
\label{eq:920}
  \Im\Big\langle\phi,\frac{\A_{\lambda,\alpha}\phi}{U+i\lambda}\Big\rangle=-\mu
  \Big\|\frac{|U^{\prime\prime}|^{1/2}\phi}{U+i\lambda}\Big\|_2^2\,,
\end{equation}
we easily obtain that
\begin{equation}
\label{eq:923}
  \Big\|\frac{\phi}{U+i\lambda}\Big\|_2\leq \frac{C}{|\mu|}\|v\|_2\,.
\end{equation}
From the above (accompanied by a rather straightforward integration by
parts) it is not difficult to show that 
\begin{displaymath}
  \|\A_{\lambda,\alpha}^{-1}\|+\Big\|\frac{d}{dx}\A_{\lambda,\alpha}^{-1}\Big\|\leq\frac{C}{|\mu|}\,.
\end{displaymath}

The above estimate is unsatisfactory in the limit $\mu\to0$, and hence
finer estimates need to be established. We use the fact that
$\A_{i\nu,\alpha}$ is self-adjoint. Thus, for \break $\nu\not\in(0,U(0)]$
 (see Subsections \ref{sec:case--nu-negative} and
\ref{sec:case--large-nu}) we may write
\begin{displaymath}
  \Big\langle\frac{\phi}{U-\nu},\A_{i\nu,\alpha}\phi\Big\rangle=\Big\|(U-\nu)\Big(\frac{\phi}{U-\nu}\Big)^\prime\Big\|_2^2 + \alpha^2\|\phi\|_2^2\,,
\end{displaymath}
and obtain from it an estimate for $\|\phi^\prime\|_2$ in the case where
$|\mu|$ is small. \\
For $\nu\in (0,U(0)]$ we have to address the singularity
where $U=\nu$. Given the fact that $U$ is increasing in $[0,1]$ there
exists a unique $x_\nu\in[0,1)$ where $U(x_\nu)=\nu$. Let
$\chi\in C^\infty([0,1],[0,1])$ denote a cutoff function supported on
$[x_\nu/2,1]$. Setting $\varphi=\phi-\phi(x_\nu)\chi$ we may write
\begin{equation}
\label{eq:922}
 \Big\|(U-\nu)\Big(\frac{\varphi}{U-\nu}\Big)^\prime\Big\|_2^2 + \alpha^2\|\phi\|_2^2=
 \Big\langle\frac{\varphi}{U-\nu},\A_{i\nu,\alpha}\varphi\Big\rangle-\phi(x_\nu)\Big[\Big\langle
  \frac{\varphi}{U-\nu},U^{\prime\prime}\chi\Big\rangle
  -\langle\varphi,\chi^{\prime\prime}-\alpha^2\chi\rangle\Big]  \,. 
\end{equation}
For the above balance to become useful for the purpose of obtaining
estimates for $\|\phi^\prime\|_2$,  we need to obtain an estimate for
$\phi(x_\nu)$. To this end we use \eqref{eq:920} to obtain (for $\mu\neq0$)
\begin{displaymath}
  \mu|\phi(x_\nu)|^2 \Big\|\frac{1}{U+i\lambda}\|_2^2 \leq
  \|\phi\|_\infty\Big\|\frac{v}{U+i\lambda}\Big\|_1 + C|\mu|
  \Big\|\frac{\phi-\phi(x_\nu)}{U-\nu}\Big\|_2^2  \,.  
\end{displaymath}
Under the condition in \cite{almog2019stability} on $U$, which is
assumed to be strictly monotone, the above estimates leads to
\begin{displaymath}
  |\phi(x_\nu)|\leq  \|\phi\|_\infty\Big\|\frac{v}{U+i\lambda}\Big\|_1
  +C|\mu|\,\|\phi^\prime\|_2\,.
\end{displaymath}
Substituting the above into \eqref{eq:922} (properly amended to
account for small values of $|\mu|$) yields an estimate for
$\|\phi^\prime\|_2$.

To adapt the above method to the present context we need to overcome
several difficulties:
\begin{enumerate}
\item It holds that $\A_{0,0}U=0$ and since $U\in D(\A_{0,0})$,
  $(\A_{\lambda,\alpha})^{-1}$ becomes strongly singular in the limit
  $(\lambda,\alpha)\to(0,0)$.
\item The boundary condition at $x=0$ is a Neumann condition in
  contrast with the Dirichlet condition in
  \cite{almog2019stability}. Thus, we have to write \eqref{eq:922}
  separately on the intervals $(0,x_\nu)$ and $(x_\nu,1)$. On $(x_\nu,1)$
  we may use the same method used in  \cite{almog2019stability}. 
  However, on $(0,x_\nu)$, considering $\phi/(U-\nu)$, we obtain bounds
of this quotient  for small values of $\alpha$ that are significantly greater than
those obtained for larger values of $\alpha$.
\item The quadratic behavior of $U-U(0)$ as opposed to the linear
  behavior  considered in  \cite{almog2019stability}.
\end{enumerate}
The first pair of difficulties is addressed by the same techniques:
\begin{itemize}
\item 
For small values of $\alpha$ we use the fact that we can consider
$(\A_{0,\lambda})^{-1}$ as an integral operator to obtain satisfactory
estimates for its norm (see Subsection \ref{lem:lambda-zero}). 
\item For
larger values of $\alpha$ we use again \eqref{eq:922} (see Subsection
\ref{sec:case-lambda-small-alpha-large}).
\end{itemize}
 In Subsection \ref{ss2.6} we
present a different analysis, which is valid for any $\alpha\geq0$ with
stronger singularity in the limit $\lambda\to0$. In all cases, we use the
orthogonal decomposition $\phi=C_\parallel(U-\nu)+\phi_\perp$ to obtain separate
estimates for $C_\parallel$ and
$\phi_\perp$, estimates of the latter being significantly 
smaller than the former. To overcome the last difficulty we simply use
\eqref{eq:923} in Subsection \ref{sec:case--quadratic}. In Subsection
\ref{sec:case-near-quadratic} we consider the transition between the
quadratic behavior of $U$ near $x_\nu$ and the linear behavior
considered in Subsection \ref{ss2.6}. \\

{ \em Estimate of $\B_{\lambda,\alpha,\beta}^{-1}$}\\

To obtain an estimate of $v$ (see (\ref{eq:defphiv}b)) we set
\begin{equation}
\label{eq:929}
    v_\Df :=\A_{\lambda,\alpha}\phi + (U+i\lambda)\phi^{\prime\prime}(1)\hat{\psi} \,,
\end{equation}
where
\begin{displaymath}
  \hat{\psi}(x)= \frac{{\rm Ai}\big(
  \beta^{1/3}e^{ -i\pi/6}[(1- x)- i\lambda]\big)} 
{{\rm Ai}\big( e^{ -i2\pi/3}\beta^{1/3}\lambda\big)}\eta(x) \,.
\end{displaymath}
Here $\rm Ai$ is the Airy function and $\eta\in C^\infty([0,1],[0,1])$ is supported on $(1/4,1]$, and satisfies
$\eta\equiv1$ on $[1/2,1]$.  Note that $v_\Df(1)=v_\Df^\prime(0)=0$ and that
$\hat{\psi}$ is
a good approximation for the $L^2(-\infty,1)$ solution of
\begin{equation}
\label{eq:926}
  \begin{cases}
     (-\frac{d^2}{dx^2}+i\beta[(1-x)+i\lambda])u=0 & \text{in } (-\infty,1) \\
     u(1)=1 \,.
  \end{cases}
\end{equation}
We can now
rewrite \eqref{eq:924} in the form
\begin{displaymath}
  (\LL_\beta-\beta\lambda)v_\Df=g_\Df\,,
\end{displaymath}
where
\begin{displaymath}
  \LL_\beta=-\frac{d^2}{dx^2}+i\beta U
\end{displaymath}
is defined on
\begin{displaymath}
  D(\LL_\beta)=\{u\in H^2(0,1)\,|\,u(1)=u^\prime(0)=0\,\}\,.
\end{displaymath}
While the precise form of $g_\Df$ need not concern us in this brief
summary of the proof we still need to obtain an estimate of its
$L^2(0,1)$ norm. Thus, we get an estimate of  
$v$ by working through the following steps
\begin{enumerate}
\item Estimate of $\phi^{\prime\prime}(1)$.
\item Estimate of $g_\Df$.
\item Estimate of $v_\Df$.
\item Estimate of $\A_{\lambda,\alpha}^{-1}v_\Df$ and of $\phi^{\prime\prime}(1) \A_{\lambda,\alpha}^{-1} (U+i\lambda)\hat{\psi} $.
\end{enumerate}

We use two different methods for the estimation of $\phi^{\prime\prime}(1)$.\\
{\bf  For
$\alpha$ values that are not too small, }we rewrite \eqref{eq:924} in the form
\begin{displaymath}
   \Big(-\frac{d^2}{dx^2}+i\beta[U+i\lambda]\Big)(\phi^{\prime\prime}-\alpha^2\phi)=i\beta U^{\prime\prime}\phi+f\,.
\end{displaymath}
Given that $\phi(1)=\phi^\prime(1)=\phi^\prime(0)=0$, we may conclude that for
${\mathfrak z}(x)=\cosh(\alpha x)/\cosh \alpha$ it holds that
\begin{subequations}\label{eq:heuristiczeta}
\begin{equation}
  \langle{\mathfrak z},\phi^{\prime\prime}-\alpha^2\phi\rangle=0\,.
\end{equation}
Consequently, we define the Schr\"odinger operator $\LL_\beta^{\mathfrak
  z}$ with the same differential  operator as for  $\LL_\beta$ but with the following
domain
\begin{equation}
  D(\LL_\beta^{\mathfrak z})=\{
  u\in H^2(0,1)\,|\,u^\prime(0)=0\,,\;\langle{\mathfrak z},u\rangle=0 \}\,.
\end{equation}
\end{subequations}
Let $(v,g)\in D(\LL_\beta^{\mathfrak z})\times L^2(0,1)$ satisfy
\begin{equation}
\label{eq:927}
  (\LL_\beta^{\mathfrak z}-\beta\lambda)v=g \,.
\end{equation}
In Section~4 we obtain estimates for $v(1)$ that we later use in
Section 5 (except for Subsections \ref{sec:5.5} and
\ref{sec:resolv-estim-large-1}) to obtain an estimate for
$\phi^{\prime\prime}(1)$.  Again we have to distinguish between the quadratic
case (Subsection \ref{sec:resolvent-estimates-quadratic}) and the
linear case (Subsection \ref{sec:resolvent-estimates-linear}).

{\bf For smaller values of $\alpha$ and $|\lambda|$,} the estimate of $\phi^{\prime\prime}(1)$,
obtained by the above technique becomes deficient, given the
singularity of $\A_{0,0}$. We thus integrate \eqref{eq:924} for $\alpha=0$
to obtain
\begin{equation}
\label{eq:925}
    \phi^{(3)}(1)=-\int_0^1f(x)\,dx \,.
\end{equation}
Then we use the identity
\begin{multline*}
    \|(U^{\prime\prime})^{-1/2}\phi^{(3)}\|_2^2 = -  \Re\langle(U^{\prime\prime})^{-1} \phi^{\prime\prime},\B_{\lambda,0,\beta}\phi\rangle-
      \frac{1}{U^{\prime\prime}(1)}\Re (\bar{\phi}^{\prime\prime}(1)\phi^{(3)}(1))  \\
      - \Re\langle[(U^{\prime\prime})^{-1}]^\prime\phi^{\prime\prime},\phi^{(3)}\rangle +
      \mu \beta \|(U^{\prime\prime})^{-1/2}\phi^{\prime\prime}\|_2^2 \,.
\end{multline*}
to obtain a proper bound for $\|\phi^{(3)}\|_2$, which together with
Sobolev embedding (skipping, of course, some of the details) leads to
a satisfactory bound for $|\phi^{\prime\prime}(1)|$. Note that the effectiveness
of this technique is lost  when $\alpha$ is not small since \eqref{eq:925}
is no longer valid. We use it only in Subsection \ref{sec:5.5}. 

{\bf Finally, we note that for $\alpha\gtrsim\beta^{-1/3}$,}  ${\mathfrak z}$ undergoes
significant changes through an $\OO(\beta^{-1/3})$ boundary layer near
$x=1$.  Hence, we can no longer make any
good use of \eqref{eq:926} as an estimate for the behavior near $x=1$
of the solution of \eqref{eq:927}. Instead, we need to use the same
method developed in \cite{almog2019stability} for this case. Note
that, since ${\mathfrak z}$ is localized near $x=1$ for large values of
$\alpha$, the effect of the different boundary conditions at $x=0$ here and
in  \cite{almog2019stability} is exponentially small. Resolvent
estimates for $\LL_\beta^{\mathfrak z}$ in this case are brought in Subsection
\ref{sec:resolv-estim-large} whereas  estimates for the inverse of 
$\B_{\lambda,\alpha,\beta}$ are given in Subsection \ref{sec:resolv-estim-large-1}.

We skip  the rather technical stage of estimating $g_\Df$. Once
it is done,  we need to estimate $(\LL_\beta-\beta\lambda)^{-1}$ in $\LL(L^2,L^p)$
for $1\leq p\leq\infty$ in order to obtain
an appropriate estimate for $v_\Df$. These estimates are obtained in
Section 3 for various ranges of $\lambda$ values. Again we need to
distinguish between the linear behavior of $U-U(x_\nu)$ near $x=x_\nu$
(Subsection \ref{ss3.1}) and the quadratic behavior near $x=0$ for $\nu=U(0)$ (Subsection
\ref{sec:schrod-quad}). Special attention is also devoted to
$\LL(L^2,L^1)$ and $\LL(H^1,L^1)$ estimates
(Subsection \ref{sec:l1-estimates}). 

Next, we estimate $\A_{\lambda,\alpha}^{-1}v_\Df$ using the aforementioned
techniques of Section 2. To estimate $\A_{\lambda,\alpha}^{-1} (U+i\lambda)\hat{\psi}
$ we use the exponential decay of $\hat{\psi}$ away from $x=1$ to
obtain, in a rather straightforward manner, except in the case
$(\lambda,\alpha)\to(0,0)$. These estimates are addressed in Subsection
\ref{sec:5.1}.\\

 Finally, in
Section \ref{sec:6} we summarize the results of Section 5 and prove the main
theorems.

\section{The inviscid operator}
\label{sec:2}

\subsection{Preliminaries}
 We begin by presenting the notation, frequently used in the sequel
  \begin{displaymath}
    \langle f,g\rangle_{L^2(a,b)}=\int_a^b\bar{f}(x)g(x)\,dx\,.
  \end{displaymath}
Note that when $(a,b)=(0,1)$ the abbreviated notation $\langle f,g\rangle$ is
used instead of $\langle f,g\rangle_{L^2(0,1)}$.

Next, recall that the differential expression of the inviscid operator (also
called the Rayleigh operator) is given by
\begin{equation}
\label{eq:2.1aa}
  \A_{\lambda,\alpha} \overset{def}{=} (U+i\lambda)\Big(-\frac{d^2}{dx^2}+ \alpha^2\Big)+U^{\prime\prime}\,,
\end{equation}
where $\lambda \in \mathbb C$ and $\alpha \in \mathbb R$.\\ Note that,  for any
$\phi\in D(\B_{\lambda,\alpha,\beta}^\Df)$, 
$$  \A_{\lambda,\alpha}\phi=\lim_{\beta\to\infty}\beta^{-1}\B_{\lambda,\alpha,\beta}^\Df\phi\,.$$
We consider here only spaces of even functions in $(-1,1)$,  hence, as explained 
in the introduction,  we restrict the operator to $(0,1)$ and consider
the Neumann condition at $0$ and the Dirichlet condition at $1$. We
thus define $\A_{\lambda,\alpha}^{\mathfrak N,\mathfrak D}$
\begin{subequations}\label{eq:domND}
\begin{itemize}
\item   for $\Re\lambda\neq0$ or when $\Re\lambda=0$ and
$\Im\lambda\notin[0,U(0)]$, on 
\begin{equation}
  D(\A_{\lambda,\alpha}^{\mathfrak N,\Df})=\{\phi\in H^2(0,1)\,|\,\phi^\prime(0)=0 \mbox{ and }\phi(1)=0 \;\} \,,
\end{equation}
 \item for {$\Re\lambda=0$ and $\Im\lambda\in[0,U(0)]$}
\begin{equation}
  D(\A_{0,\alpha}^{\mathfrak N,\Df})=\{\phi\in H^2((0,1), (U-\Im\lambda)^2 dx)\,|\,\phi^\prime(0)=0 \mbox{ and }\phi(1)=0 \;\} \,.
\end{equation}
which is equipped  with the norm
  \begin{displaymath}
    \|u\|_{ D(\A_{0,\alpha}^{\mathfrak N,\Df})}=\int_0^1
    [|u^{\prime\prime}|^2+[|u^\prime|^2+|u|^2](U-\Im\lambda)^2 \,dx \,.
  \end{displaymath}
\end{itemize}
\end{subequations}
Hence, in the following, we restrict attention to the interval $[0,1]$
assuming  in this section that $U \in C^3([0,1])$  and satisfies the condition:  
  \begin{subequations}
  \label{eq:10}
    \begin{equation}
      U(1)=0 \quad ; \quad \max_{x\in[0,1]}U^{\prime\prime}(x)<0 \quad ; \quad
      U^\prime(0)=0 \,. \tag{\ref{eq:10}a,b,c}
    \end{equation}
  We normalize $U$ so that
  \begin{equation}
  U^\prime(1)=-1 \,.  \tag{\ref{eq:10}d}
  \end{equation}
  \end{subequations}
For convenience of notation we omit the superscripts $\mathfrak N$ and $\mathfrak
D$ in the sequel whenever there is not any fear of ambiguity.  Since
\begin{displaymath}
  \A_{\lambda,\alpha}^U =  \overline{\A_{\bar{\lambda},\alpha}^{-U}} \,,
\end{displaymath}
the analysis in this section applies to the case where
$\min_{x\in[-1,1]}U^{\prime\prime}(x)>0$ replaces (\ref{eq:10}b) as well.  

In this section we obtain a variety  of estimates for $\A_{\lambda,\alpha}^{-1}$ that are
  necessary in order to obtain bounds  in the same parameter regime
  for  $\B_{\lambda,\alpha}^{-1}$.  We note that since
  \begin{displaymath}
    \A_{\lambda,\alpha}=\overline{  \A_{-\bar{\lambda},\alpha}}
  \end{displaymath}
the results in this section do not depend on the sign of $\Re \lambda$. 

Let $\lambda=\mu+i\nu$.  We begin by summarizing the results of this section
in Figure 1. We map in this figure the regions in the $(|\mu|,\nu)$
plane where the various estimates of $\A_{\lambda,\alpha}^{-1}$ can be found.
We refer the reader to Subsection \ref{sec:proof-strategy} where the
consideration that led us to split the $\lambda$ plane to these regimes are
briefly explained.  We note that the results of Subsections 2.6--2.11 are
valid for any $\alpha\geq0$. Subsection 2.3 addresses the case $\alpha=\lambda=0$,
Subsection 2.4 addresses small $\alpha$ values, and Subsection 2.5 relatively
large values of $\alpha$.  The constants determining the boundaries of the
various regimes satisfy: $0<\nu_1<U(0)$, $\nu_0 <0$, $0<\kappa_0$ must be
sufficiently large, and $\mu_0$ and $\lambda_0$ must be sufficiently small.
\hoffset=0in


\begin{figure}
\includegraphics[width=200mm]{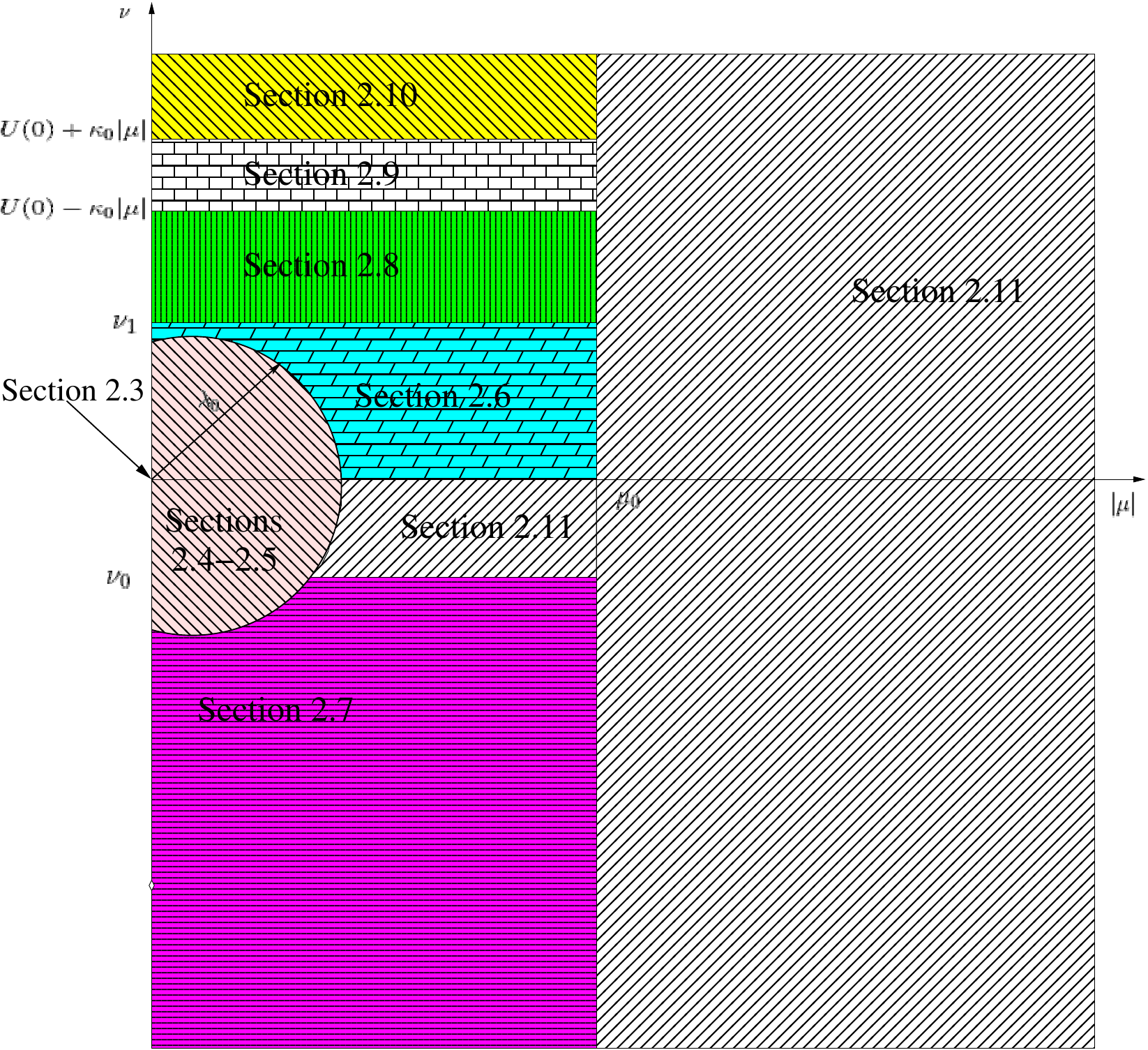}
  \caption{Summary of the results in Section \ref{sec:2}. For each
    zone in $\mathbb C$ for the parameter $\lambda$, we indicate the
    subsection where the basic inequality is proved. Recall that 
    $\mu=\Re\lambda$, and $\nu=\Im\lambda$, $\mu_0,\,\nu_0,\,\nu_1,\,\lambda_0,\,\kappa_0$ are
    positive constants which are introduced in the various subsections.}
  \label{fig:1}
\end{figure}
\hoffset=-1.1in

\clearpage
\hoffset=0in

\subsection{Some more preliminaries}\label{ss2.2}
We begin by defining the following (formally) self-adjoint unbounded operator  on

\begin{displaymath}
H^0_U (0,1):=L^2((0,1), U^2 dx)
\end{displaymath}
by  
\begin{subequations}
  \label{eq:11}
  \begin{equation}
\Mg_U= - U^{-2} \frac{d}{dx}U^2\frac{d}{dx} \,,
\end{equation}
which is naturally associated via the Lax-Migram Lemma with the
quadratic form $ \QQ_U$, defined on
\begin{equation}
  H^1_U(0,1)= \{u\in L^2_{loc} (0,1)\,|\, Uu\in L^2(0,1)  \mbox{ and }  Uu^\prime \in L^2(0,1)\,\} \,.
\end{equation}
by
\begin{equation}
  \QQ_U(w)=\|Uw^\prime\|_2^2\,.
\end{equation}
\end{subequations}
Recall that $U$ satisfies \eqref{eq:10}.
\begin{remark}
  When $U$ is replaced by $e^{-\phi}$, we arrive at a well-known problem
  considered in Statistical Mechanics and Morse theory (Witten
  Laplacians), see for example \cite{HebookWitten} and the references
  therein.  As observed in \cite{BenBen}, we arrive at  a singular
  case of this theory because $U$ is vanishing at $x=1$.
\end{remark}

We can now state
\begin{lemma} $
\mathcal M_U$ is a self-adjoint operator with compact
  resolvent on $L^2((0,1), U^2\, dx)$.  If $ \{\kappa_n\}_{n=1}^\infty$
  denotes the non decreasing sequence of which the spectrum
  $\sigma(\Mg_U)$ is consisted, then
$$\kappa_1=0\,,$$ 
and 
  \begin{equation}
\label{eq:12}
    \kappa_2\geq \lambda_2^N \, |U(0)|^2\,,
  \end{equation}
  where $\lambda_2^N$ denotes the second eigenvalue of the radially symmetric
Neumann-Laplacian (i.e, the Laplacian reduced to the radially symmetric functions satisfying the Neumann condition.)  in the unit ball in $\R^3$.  
\end{lemma}
\begin{proof}
  The proof that $\kappa_1=0$ is trivial (the associated eigenfunction is
  the non-vanishing constant function). To prove the lower bound for
  $\kappa_2$ we first observe that, by the variational max-min
  characterization of the second eigenvalue, 
  \begin{equation}
\label{eq:13}
     \kappa_2= \sup _{\psi \in H^1_U(0,1)} \inf_{
      \begin{subarray}\strut
        w\in H^1_U(0,1) \\
        \langle w,U^2\psi\rangle=0
      \end{subarray}}\frac{\|Uw^\prime\|_2^2}{\|Uw\|_2^2}    \,=\, 
 \sup _{\phi \in H^1_U(0,1)} 
\inf_{\begin{subarray}\strut
        w\in H^1_U(0,1) \\
        \langle w,(1-x)^2\phi\rangle=0
      \end{subarray}}
\frac{\|Uw^\prime\|_2^2}{\|Uw\|_2^2}\,.
  \end{equation}
The second equality can be proved by writing $\psi = \frac{(1-x)^2}{U^2}
\phi$. Then, we use the fact that by the concavity of $U$ (see
(\ref{eq:10}b)) 
\begin{equation}\label{eq:14}
  U(0)(1-x) \leq U(x)\leq (1-x)\,,
\end{equation}
to obtain
\begin{equation}
\label{eq:15}
  |U(0)|^{-2} \frac{\|(1-x)w^\prime\|_2^2}{\|(1-x)w\|_2^2} \geq
  \frac{\|Uw^\prime\|_2^2}{\|Uw\|_2^2}\geq
  |U(0)|^2\frac{\|(1-x)w^\prime\|_2^2}{\|(1-x)w\|_2^2}\,. 
\end{equation}
Let $U_0(x)=1-x$ and $\kappa_n^0$ denote the $n$'th eigenvalue of
$\Mg_{U_0}$. By \eqref{eq:15} and \cite[Theorem 11.12]{Hebookthsp})
we have that
\begin{equation}
\label{eq:16}
   |U(0)|^{2} \kappa_2^0 \leq \kappa_2 \leq   |U(0)|^{-2}  \kappa_2^0\,.
\end{equation}
Setting $\rho(x)=1-x$ we obtain that
\begin{displaymath}
  \Mg_{U_0}= - \rho^{-2} \frac{d}{d\rho}\rho^2\frac{d}{d\rho}\,,
\end{displaymath}
which is defined on (the Neumann condition is the natural boundary
condition associated with $Q_{U_0}$)
\begin{displaymath} 
  D(\Mg_{U_0})=\{u\in H^2([0,1];\rho^2d\rho)\,|\, u^\prime(1)=0\} \,.
\end{displaymath}
Hence,   $\Mg_{U_0}$ is the radially symmetric Neumann Laplacian and
we may conclude that 
\begin{displaymath}
   \kappa_2^0=\lambda_2^N \,.
\end{displaymath}
The above, together with \eqref{eq:16} yields \eqref{eq:12}.
\end{proof}
 We next recall Hardy's inequality on finite intervals (see for example 
\cite[Eq. (1.25)]{kupe03} )
\begin{lemma}
   Let $w\in H^1(a,b)$ satisfy $w(b)=0$. Then,  we have
\begin{equation}
\label{eq:17}
  \|([x-a]w)^\prime\|_2^2 =\|[x-a]w^\prime\|_2^2\geq \frac{1}{4}\|w\|_2^2 \,.
\end{equation}
\end{lemma}
\begin{proof}
  Let $\tilde{w}\in H^1(a,\infty)$ be given by
  \begin{displaymath}
    \tilde{w}(x)=
    \begin{cases}
      w(x) & x\in(a,b) \\
      0 & x\geq b \,.
    \end{cases}
  \end{displaymath}
  Then Hardy's inequality on $\R_+$ applied to $\tilde w$ implies that
\begin{displaymath}
 \|([x-a]w)^\prime\|_{L^2(a,b)}^2=  \|([x-a]\tilde{w})^\prime\|_{L^2(a,\infty)}^2\geq
   \frac{1}{4}\|\tilde{w}\|_{L^2(a,\infty)}^2= \frac{1}{4}\|w\|_{L^2(a,b)}^2 \,. 
\end{displaymath}
To complete the proof we first write
\begin{displaymath}
  \|([x-a]w)^\prime\|_2^2= \|[x-a]w^\prime\|_2^2+\|w\|_2^2+2\Re\langle w,[x-a]w^\prime\rangle\,.
\end{displaymath}
An integration by parts then yields
\begin{displaymath}
  2\Re\langle w,[x-a]w^\prime\rangle=\Re\langle[x-a],(|w|^2)^\prime\rangle=-\|w\|_2^2 \,.
\end{displaymath}
Hence 
\begin{displaymath}
    \|([x-a]w)^\prime\|_2^2= \|[x-a]w^\prime\|_2^2 \,,
\end{displaymath}
which completes the proof of the lemma. 
\end{proof}

If we drop the requirement that $w(b)=0$ we can still state the
following 
\begin{lemma}
   Let $w\in H^1(a,b)$. Then,  we have
\begin{equation}
\label{eq:18}
  \|([x-a]w)^\prime\|_2^2\geq \frac{1}{4}\|w\|_2^2 \,.
\end{equation}
\end{lemma}
\begin{proof}
  We use Hardy's inequality in $\R_+$ for the extension
  \begin{displaymath}
    \tilde{w}(x)=
    \begin{cases}
      w(x) & x\in(a,b) \\
      w(b)\frac{b-a}{x-a} & x\geq b \,.
    \end{cases}
  \end{displaymath}
\end{proof} 
\subsection{ Estimates in the case $\alpha=\lambda=0$}

We begin with the simplest possible case, for which $\alpha=\lambda=0$. We
recall that
\begin{displaymath} 
\mathcal A_{0,0}:= - U \frac{d^2}{dx^2} + U^{\prime\prime}\,,
\end{displaymath}
is defined on
\begin{displaymath} 
  D(\A_{0,0})=\{u\in H^2((0,1), U^2 dx)\,|\,u(1)=u^\prime(0)=0\,\}\,,
\end{displaymath}
corresponding to a Dirichlet-Neumann problem on $(0,1)$.\\
Next, let $W^{1,p}(0,1)$ denote the normed space 
\begin{displaymath}
  W^{1,p} (0,1):=\{u\in L^p(0,1)\,|\, u^\prime \in L^p(0,1)\}\,,
\end{displaymath}
  with its natural norm denoted by $\| \cdot \|_{1,p}$. \\

 Observing that $U$ belongs to  the kernel of $\A_{0,0}$,   we set
\begin{subequations}
    \label{eq:19}
   \begin{equation}
    \phi= c_\parallel  U+\phi_\perp   \,,
  \end{equation}
where
\begin{equation}
 c_{\parallel}  = \frac{\langle U,\phi\rangle }{\|U\|_2^2}\,,
\end{equation}	
in which $\langle \cdot ,\cdot \rangle$ denotes the natural  $L^2(0,1)$ inner product.
\end{subequations}
\begin{lemma}
\label{lem:lambda-alpha-zero}
Let $U\in C^2([0,1])$ satisfy \eqref{eq:10}. There exists
$C>0$ such that for any $(\phi,v)\in D(\A_{0,0})\times  W^{1,2}(0,1)$
satisfying
\begin{subequations}\label{eq:20}
\begin{equation}
  \A_{0,0}\,\phi=v\,,
\end{equation}
\begin{equation}
\langle1,v\rangle=0 \mbox{ and } v(1)=0\,,
\end{equation}
and
 \begin {equation}
 \lim_{
   \begin{subarray}\strut
    x\to1 \\
    x<1
   \end{subarray}}\phi^\prime(x)=0\,,
 \end{equation}
 \end{subequations}
we have
\begin{subequations}
  \label{eq:21}
  \begin{equation}
\|\phi_\perp\|_{1,2}\leq C \, \|v\|_2\,,
\end{equation}
 and
\begin{equation}
  |c_\parallel|\leq C \, \|v\|_2^{1/2}\|v\|_{1,2}^{1/2}\,.
\end{equation}
\end{subequations}
\end{lemma}
\begin{proof}~\\
Let $w=\phi/U$. Clearly, $w= c_\parallel  + w_\perp$  with $w_\perp=\phi_\perp/U$, and 
\begin{equation}\label{eq:22}
\mathcal A_{0,0} \, \phi= U^2\mathcal M_U\, w=  U^2\mathcal M_U \, w_\perp=v \,.
\end{equation}
{\em  Step 1: Estimate of $\|\phi_\perp^\prime\|_2$}.~\\
Taking the inner product with $w_\perp$ yields
\begin{equation}
\label{eq:23}
  \|Uw_\perp^\prime\|_2^2=\langle w_\perp,v\rangle \,.
\end{equation}

Since $\langle w_\perp,U^2\rangle =\langle \phi_\perp,U \rangle  =0$,  we now use \eqref{eq:12}, 
\eqref{eq:23}, and  Hardy's inequality \eqref{eq:17} 
to obtain that
\begin{equation}
\label{eq:24}
   \|\phi_\perp\|_2^2=\|Uw_\perp\|_2^2\leq \kappa_2  \|Uw_\perp^\prime\|_2^2 \leq  
   \kappa_2  \, \Big\|\frac{\phi_\perp}{ U}\Big\|_2\|v\|_2\leq C\, \|\phi_\perp^\prime\|_2\|v\|_2 \,.
\end{equation}
Note that by  \eqref{eq:10} and  \eqref{eq:17}  it holds that
\begin{multline}
  \label{eq:25}
\|w_\perp\|_2 \leq \Big\|\frac{\phi_\perp}{U(0)(1-x)}\Big\|_2\leq
\Big\|\frac{\tilde{\phi}_\perp}{U(0)(1-x)}\Big\|_{L^2(-\infty,1)}\\\leq
\frac{2}{U(0)}\|\tilde{\phi}_\perp^\prime\|_{L^2(-\infty,1)}=\frac{2}{U(0)}\|\phi_\perp^\prime\|_2\,,
\end{multline}
 where $\tilde{\phi}_\perp\in H^1_{loc}((-\infty, 1])$ is given by 
\begin{displaymath}
  \tilde{\phi}_\perp(x)=
  \begin{cases}
   \phi_\perp(x) & x\in[0,1] \\
   \phi_\perp(0) &x <0\,.
  \end{cases}
\end{displaymath}
Integration by parts yields
\begin{equation}
\label{eq:26}
  \|\phi_\perp^\prime\|_2^2 = \|(Uw_\perp)^\prime\|_2^2 = \|Uw_\perp^\prime\|_2^2
  -\langle Uw_\perp,U^{\prime\prime}w_\perp\rangle \leq \|Uw_\perp^\prime\|_2^2+ C\,\|\phi_\perp\|_2\|w_\perp\|_2 \,.
\end{equation}
Using \eqref{eq:23} and \eqref{eq:25}, we obtain
\begin{displaymath}
  \|\phi_\perp^\prime\|_2\leq C\, (\|v\|_2+ \|\phi_\perp\|_2)\,.
\end{displaymath}
By \eqref{eq:24} we then obtain
\begin{equation}
\label{eq:27}
  \|\phi_\perp^\prime\|_2\leq C\, \|v\|_2 \,.
\end{equation}
Note that to obtain \eqref{eq:27} we have used the mere fact that $v\in
L^2(0,1)$ .\\

{\em Step 2: Estimate of $c_\parallel$\,.\\}

By (\ref{eq:20}a) and the fact that $\mathcal A_{0,0}  \phi_\perp =v$, it holds that
\begin{displaymath}
  \|\phi_\perp^{\prime\prime}\|_2\leq C \, \Big(\Big\|\frac{\phi_\perp}{U}\Big\|_2+\Big\|\frac{v}{U}\Big\|_2\Big)\,.
\end{displaymath}
By Hardy's inequality~\eqref{eq:18} we then obtain that
\begin{displaymath}
   \|\phi_\perp^{\prime\prime}\|_2\leq C(\|\phi_\perp^\prime\|_2+ \|v^\prime\|_2)\,.
\end{displaymath}
Using \eqref{eq:27} we may conclude that
\begin{equation}
\label{eq:28}
  \|\phi_\perp^{\prime\prime}\|_2\leq C\|v\|_{1,2} \,. 
\end{equation}
Using Sobolev embedding,  (\ref{eq:10}d) and (\ref{eq:20}c) yields
\begin{displaymath}
  |c_\parallel|=|\phi_\perp^\prime(1)|\leq  \|\phi_\perp^\prime\|_\infty \leq
  \|\phi_\perp^{\prime\prime}\|_2^{1/2}\|\phi_\perp^\prime\|_2^{1/2}\,.
\end{displaymath}
We can now conclude  (\ref{eq:21}b) from \eqref{eq:27} and
\eqref{eq:28}. 
\end{proof}

\subsection{ Estimate of $(\A_{\lambda,\alpha})^{-1}$ for  $\Re \lambda \neq 0$ and
  $\alpha\ll|\lambda|^{1/2}$.}  
We continue with the following estimate of $(\A_{\lambda,0})^{-1}$ when
$\Re \lambda \neq 0$. From \eqref{eq:2.1aa}, we recall that 
\begin{equation*}
  \A_{\lambda,0} \overset{def}{=} - (U+i\lambda) \frac{d^2}{dx^2}+U^{\prime\prime}\,,
\end{equation*}
and that its domain is defined in (\ref{eq:domND}a):
\begin{displaymath} 
  D(\A_{\lambda,0})=\{u\in H^2((0,1), U^2 dx)\,|\,u(1)=u^\prime(0)=0\,\}\,.
\end{displaymath}
We shall then consider $\mathcal A_{\lambda,\alpha}^{-1}$ for $\alpha$ small enough.
\begin{proposition}
\label{lem:lambda-zero}
  Let $p>1$ and $U\in C^3([0,1]) $ 
satisfy \eqref{eq:10}. There exists $C>0$ such that, for  $\lambda\in\C$
  satisfying  $\Re\lambda\neq0$ and $|\lambda|< U(3/4)$, it holds for any $ (\phi,v)\in
  D(\A_{\lambda,0})\times  L^p(0,1)$ satisfying $\A_{\lambda,0} \, \phi=v$ that
\begin{equation}
  \label{eq:29}
\|\phi\|_{1,2}\leq C\Big(\frac{1}{|\lambda|}{\Big|\int_0^1v\,dx\Big|}+  \Big\|\frac{v}{U+i\lambda}\Big\|_1+\|v\|_p  \Big)\,.
\end{equation}
\end{proposition}
\begin{proof}~\\
  {\em Step 1:  We prove that
  \begin{equation}
    \label{eq:30}
    \|\phi\|_\infty\leq C\Big(\frac{1}{|\lambda|}{\Big|\int_0^1v\,dx\Big|}+ \Big\|\frac{v}{U+i\lambda}\Big\|_1+\|v\|_p 
       \Big)\,. 
  \end{equation}
  }
  
{\em Step  1.1: The estimate on $[1/2,1]$.}\\
As, for any $\phi\in D(\A_{\lambda,0})$,
\begin{displaymath}
  \A_{\lambda,0}\phi=-[(U+i\lambda)\phi^\prime -U^\prime\phi]^\prime \,,
\end{displaymath}
we may conclude that
\begin{equation}
\label{eq:31}  
  \frac{\phi(x)}{U(x)+i\lambda}=- \int_x^1[K_2(x,\lambda)-K_2(t,\lambda)]v(t)dt
  +K_2(x,\lambda)\int_0^1v(t)\,dt \,,
\end{equation}
where
\begin{displaymath}
  K_2(x,\lambda)=\int_x^1\frac{ds}{(U+i\lambda)^2(s)} \,.
\end{displaymath}
A first integration by parts yields
\begin{displaymath}
   K_2(x,\lambda)=\frac{1}{U^\prime(U+i\lambda)}+\frac{1}{i\lambda}-\int_x^1\frac{U^{\prime\prime}ds}{(U^\prime)^2(U+i\lambda)}\,.
\end{displaymath}
An additional integration by parts further gives
\begin{displaymath}
 - \int_x^1\frac{U^{\prime\prime}ds}{(U^\prime)^2(U+i\lambda)} =
  \frac{U^{\prime\prime}}{(U^\prime)^3}\log(U+i\lambda) - U^{\prime\prime}(1)\log(i\lambda)+
  \int_x^1\Big(\frac{U^{\prime\prime}}{(U^\prime)^3}\Big)^\prime\log(U+i\lambda) \,ds\,.
\end{displaymath}
For $|\lambda|<U(3/4)$  (a bounded set in $\mathbb C$), there
  exists $C>0$ such that for all \break $x\in[1/2,1]$  (where $U^\prime(x)\neq
  0$) 
  \begin{equation}
\label{eq:logpz}
    \Big|\int_x^1\Big(\frac{U^{\prime\prime}}{(U^\prime)^3}\Big)^\prime\log(U+i\lambda)\,ds\Big|\leq C\,.
  \end{equation}
  To prove \eqref{eq:logpz},  we introduce, for $\nu >0$, the real $x_\nu$, which is defined by
  \begin{equation}
\label{eq:defxnu}
  U(x_\nu)=\nu \mbox{ for } 0 < \nu < U(0)\,,\, x_\nu=1 \mbox{ if } \nu\leq 0 \mbox{ and } x_\nu =0 \mbox{ if } \nu > U(0)\,.
  \end{equation}
Note that, for $ 0< \nu\leq U(0)$, $x_\nu\in[0,1)$,   is indeed uniquely defined  by the
  assumed monotonicity of $U$ (see \eqref{eq:10}).\\
Then we use the fact that for $|\lambda|<U(3/4)$ (implying $|\nu| <
U(3/4)$ and $|U^\prime(x_\nu)| >0$),   there exists $C>0$ such that 
\begin{displaymath}
  |\log(U(x)+i\lambda)|\leq C\, [1+\log|x-x_\nu|^{-1}]\,.
\end{displaymath}
Hence
 \begin{equation}\label{eq:logp}
 \log (U+i \lambda) \mbox{  is uniformly bounded  in }  L^q (0,1) \,,\,  \forall\, 1 \leq  q < +\infty\,,
\end{equation}
readily verifying \eqref{eq:logpz} .\\
Consequently 
it holds that 
\begin{equation}
  \label{eq:32}
\Big|K_2(x,\lambda)-\frac{1}{U^\prime(U+i\lambda)}-\frac{1}{i\lambda}-\frac{U^{\prime\prime}}{(U^\prime)^3}\log(U+i\lambda)
+ U^{\prime\prime}(1)\log(i\lambda)\Big|\leq C\,.
\end{equation}
Let 
\begin{equation}\label{eq:33}
  \hat{K}_2(x,\lambda)=K_2(x,\lambda)-\frac{1}{i\lambda}+U^{\prime\prime}(1)\log(i\lambda)\,.
\end{equation}
By \eqref{eq:32}, we get the existence of $C >0$ such that, for
  any $x\in[1/2,1]$ and \break $|\lambda|<U(3/4)$ we have  
\begin{equation}
\label{eq:34}
  \Big|(U+i\lambda)(x)\hat{K}_2(x,\lambda)\Big|\leq C \,.
\end{equation}
 We now write, for
  any $x\in[1/2,1]$, 
\begin{displaymath}
  \Big|\int_x^1\frac{v}{U^\prime(U+i\lambda)}dt\Big|\leq C\, \Big\|\frac{v}{U+i\lambda}\Big\|_1
  \,,
\end{displaymath}
and hence, by \eqref{eq:logp} and \eqref{eq:32}, for all $x\in[1/2,1]$
and $|\lambda|\leq U(3/4)$ we have, for $(p,q)$ satisfying $\frac 1p+\frac 1q
=1$, 
\begin{multline}
\label{eq:35}
\Big|\int_x^1 \hat{K}_2(t,\lambda)v(t)dt\Big| \leq
   C\Big(\Big\|\frac{v}{U+i\lambda}\Big\|_1+ \| \log (U+i\lambda) \|_q \|v\|_p +
   \|v\|_1 \Big)\\ \leq \widehat C \Big(\Big\|\frac{v}{U+i\lambda}\Big\|_1+\|v\|_p
   \Big)\,. 
\end{multline}
Substituting \eqref{eq:33}, \eqref{eq:34} and \eqref{eq:35} into
\eqref{eq:31}, given that $K_2(x,\lambda)-K_2(t,\lambda)
=\hat{K}_2(x,\lambda)-\hat{K}_2(t,\lambda)$ yields
\begin{equation}
  \label{eq:36}
\|\phi\|_{L^\infty(1/2,1)} \leq
C\Big(\frac{1}{|\lambda|} \Big|\int_0^1v\,dx\Big| +
\Big\|\frac{v}{U+i\lambda}\Big\|_1+\|v\|_p\Big)  \,, 
\end{equation}
which completes the estimate of $\phi$ in $[\frac 12,1]$.\\[1.5ex]

{\em Step  1.2: Estimate of $\|\phi\|_{L^\infty(1/2,1)}$.}\\[1.5ex]

Next, we consider the case where $x\in[0,1/2)$. In this interval we
need to address the fact that $U^\prime(0)=0$, as is assumed in
(\ref{eq:10}a). Here, we use the assumption $|\lambda|<U(3/4)$ to obtain
that for $x\in (0,1/2)$
\begin{equation}
\label{eq:37} 
   \Big|K_2(x,\lambda) - K_2\Big(\frac 12,\lambda\Big) \Big| =
   \Big|\hat{K}_2(x,\lambda) - \hat{K}_2\Big(\frac 12,\lambda\Big) \Big|\leq C
   \,. 
\end{equation}
 
Since by \eqref{eq:34}
\begin{displaymath} 
 \Big|(U(1/2)+i\lambda) \hat K_2(1/2,\lambda)\Big|\leq C \,,
\end{displaymath}
we obtain that
\begin{equation}
\label{eq:38}
  \Big|\hat{K}_2(1/2,\lambda)\Big|\leq  C \,.
\end{equation}
 Combining \eqref{eq:38}  with \eqref{eq:37} and \eqref{eq:33} yields
 \begin{displaymath}
     \|K_2(\cdot,\lambda)\|_{L^\infty(0,1/2)}\leq \frac{C}{|\lambda|}\,.
 \end{displaymath}
 Hence, for all $x\in[0,1/2]$ it holds that
 \begin{equation}
 \label{eq:39}
    |K_2(x,\lambda)(U+i\lambda)(x)|\leq \frac{C}{|\lambda|} \,.
 \end{equation}
  
Furthermore, by   \eqref{eq:33}, \eqref{eq:37} and \eqref{eq:38}  we have that
\begin{multline*}
   \Big|\int_x^1\Big(K_2(t,\lambda)-K_2(x,\lambda)\Big)v(t)dt\Big|\\ \leq
   \Big|\int_x^{1/2}\Big(\hat{K}_2(t,\lambda)-\hat{K}_2(x,\lambda)\Big)v(t)dt\Big| + 
   \Big|\int_{1/2}^1\Big(\hat{K}_2(t,\lambda)-\hat{K}_2(x,\lambda)\Big)v(t)dt\Big| \\
   \leq
   \Big|\int_{1/2}^1\Big(\hat{K}_2(t,\lambda)-\hat{K}_2(x,\lambda)\Big)v(t)dt\Big| +C\|v\|_1 \\
   \leq\widehat  C \Big(\frac{1}{|\lambda|}
   \Big|\int_0^1v\,dx\Big|+\Big\|\frac{v}{U+i\lambda}\Big\|_1+\|v\|_p\Big)\,,
\end{multline*}
where, to obtain the last inequality, we used \eqref{eq:35}.\\
We can then conclude that
 \begin{multline*}
   \Big|(U+i\lambda)(x)\,\int_x^1\Big(K_2(t,\lambda)-K_2(x,\lambda)\Big)v(t)dt\Big|\\ \leq C \,  \Big(\frac{1}{|\lambda|}
   \Big|\int_0^1v\,dx\Big|+\Big\|\frac{v}{U+i\lambda}\Big\|_1+\|v\|_p\Big)
\end{multline*} 
Substituting the above, together with \eqref{eq:39} into
\eqref{eq:31} yields
\begin{displaymath}
\|\phi\|_{L^\infty(0,1/2)} \leq C \Big(\frac{1}{|\lambda|}
   \Big|\int_0^1v\,dx\Big|+\Big\|\frac{v}{U+i\lambda}\Big\|_1+\|v\|_p\Big)
\end{displaymath}
Combined with \eqref{eq:36} the above readily yields \eqref{eq:30}. \\

{\em Step 2: We prove \eqref{eq:29}. }\\

\noindent We begin by rewriting $\A_{\lambda,0} \, \phi=v$ in the form
\begin{displaymath}
- \phi^{\prime\prime} = - \frac{U^{\prime\prime}}{U+i \lambda} \phi + \frac{v}{U+i\lambda} \,.
\end{displaymath}
 Taking the inner product with $\phi$ in $L^2(0,1)$, integration by parts
yields, as $\phi^\prime(0)=\phi(1)=0$, 
\begin{equation}
\label{eq:40}
  \|\phi^\prime\|_2^2 = \Re \Big\langle\phi,\frac{U^{\prime\prime}\phi}{U+i\lambda}\Big\rangle + \Re \Big\langle\phi,\frac{v}{U+i\lambda}\Big\rangle \,.
\end{equation}
 Let $\hat{\chi}\in C^\infty(\R,[0,1])$ satisfy
 \begin{equation}
\label{eq:41}
   \hat{\chi}(x)=
   \begin{cases}
     0 & |x|<\frac{1}{4} \\
     1  & |x|>\frac{1}{2} \,.
   \end{cases}
 \end{equation}
Let further $\tilde{\chi}=1-\hat{\chi}$. The first 
 term on the right-hand-side  of \eqref{eq:37}  can be rewritten after an integration by part as 
\begin{equation}
\label{eq:42}
   -\Re \Big\langle\phi,\frac{U^{\prime\prime}\phi}{U+i\lambda}\Big\rangle =\Re
   \Big\langle\Big(\frac{U^{\prime\prime}|\phi|^2\hat{\chi}}{U^\prime}\Big)^\prime,\log(U+i\lambda)\Big\rangle -\Re \Big\langle\tilde{\chi}\phi,\frac{U^{\prime\prime}\phi}{U+i\lambda}\Big\rangle \,.
\end{equation}
For the first term on the right-hand-side  of \eqref{eq:42} we write, 
\begin{displaymath}
\Big|  \Re\Big\langle\Big(\frac{U^{\prime\prime}|\phi|^2\hat{\chi}}{U^\prime})^\prime,\log(U+i\lambda)\Big\rangle\Big | \leq
  C\|\phi\|_\infty\|[|\phi|+|\phi^\prime|]\log(U+i\lambda)\|_1\,.
\end{displaymath}
Using \eqref{eq:logp} together with 
Poincar\'e's inequality leads to
\begin{equation}\label{eq:43}
| \Re\Big\langle\Big(\frac{U^{\prime\prime}|\phi|^2\hat{\chi}}{U^\prime})^\prime,\log(U+i\lambda)\Big\rangle | \leq \widehat C\, \|\phi\|_\infty\|\phi^\prime\|_2\,.
\end{equation}
For the second term on the right-hand-side  of \eqref{eq:42}   we
have, observing that \break $U(x)+i \lambda$ does not vanish for $x$ in the
support of $\tilde \chi$ and $|\lambda| < U(3/4)$,  
\begin{equation}\label{eq:44} 
  \Big|\Re \Big\langle\tilde{\chi}\phi,\frac{U^{\prime\prime}\phi}{U+i\lambda}\Big\rangle\Big| \leq C\, \|\phi\|_2^2 \,.
\end{equation}
Substituting \eqref{eq:43} and \eqref{eq:44} into \eqref{eq:42}
yields, again with the aid of Poincar\'e's inequality,
\begin{equation}
\label{eq:45}
   \Big|\Re \Big\langle\phi,\frac{U^{\prime\prime}\phi}{U+i\lambda}\Big\rangle\Big|    \leq
   C\, \|\phi^\prime\|_2\|\phi\|_\infty \,.
\end{equation}
  For the second term on the right-hand-side  of \eqref{eq:40} we use
  Poincar\'e's inequality and Sobolev's embeddings to obtain   
\begin{equation}
\label{eq:46}
  | \Re \Big\langle\phi,\frac{v}{U+i\lambda}\Big\rangle| \leq  \|\phi\|_\infty
   \Big\|\frac{v}{U+i\lambda}\Big\|_1\leq \|\phi^\prime\|_2
   \Big\|\frac{v}{U+i\lambda}\Big\|_1
\end{equation}
Substituting \eqref{eq:46} together with \eqref{eq:45} into
\eqref{eq:40} yields
\begin{displaymath}
   \|\phi^\prime\|_2\leq C\,\Big(\Big\|\frac{v}{U+i\lambda}\Big\|_1+\|\phi\|_\infty\Big)\,.
\end{displaymath}
We can now conclude \eqref{eq:29} from \eqref{eq:30}.  
\end{proof}

We can now state as a corollary, an estimate for $\mathcal A_{\lambda,\alpha}^{-1}$ when $\alpha$ is relatively small.
  \begin{corollary}\label{cor:estimate-a_lambda-alpha-small}
    Under the assumptions of Proposition \ref{lem:lambda-zero}, there
    exist $C>0$ and $\delta_0>0$ such that, for all $\lambda\in\C$ satisfying
    $\Re\lambda\neq0$ and $|\lambda|< U(3/4)$, for all $\alpha$ such that
    $$  |\alpha|\leq \delta_0\min (|\lambda|^{1/2},\log^{-1/2}|\mu|^{-1})\,,$$ 
    and
    for any pair $(\phi,v)\in D(\A_{\lambda,\alpha})\times L^p(0,1)$ satisfying
\begin{equation}
  \label{eq:47}
\mathcal A_{\lambda,\alpha}\phi=v\,,
\end{equation}
 it holds that
\begin{displaymath} 
  \|\phi\|_{1,2}\leq C\Big(\frac{1}{|\lambda|}{\Big|\int_0^1v\,dx\Big|}+ 
    \Big\|\frac{v}{U+i\lambda}\Big\|_1+\|v\|_p  \Big)\,. 
\end{displaymath}
  \end{corollary}
The proof is obtained by rewriting \eqref{eq:47} in the manner
\begin{displaymath}
  \mathcal A_{\lambda,0}\phi=v+\alpha^2\phi\,,
\end{displaymath}
 or
\begin{displaymath}
  \phi=\A_{\lambda,0}^{-1}v+\alpha^2\A_{\lambda,0}^{-1}\phi\,.
\end{displaymath}
 We now use \eqref{eq:29} to obtain
\begin{displaymath}
  \|\phi\|_{1,2}\leq C\Big(\frac{1}{|\lambda|}{\Big|\int_0^1v\,dx\Big|}+
    \Big\|\frac{v}{U+i\lambda}\Big\|_1+\|v\|_p  \Big)\,+ C 
  \alpha^2\Big(\frac{1}{|\lambda|}\|\phi\|_1+\Big\|\frac{\phi}{U+i\lambda}\Big\|_1+\|\phi\|_p\Big)\,. 
\end{displaymath}  
Since $|\nu| \leq U(3/4)<1$ \\
\begin{equation}
\label{eq:48}
   \Big\|\frac{1}{U+i\lambda}\Big\|_1\leq C\log |\mu|^{-1}\,,
\end{equation}
to obtain that
\begin{displaymath}
  \Big\|\frac{\phi}{U+i\lambda}\Big\|_1\leq
  \Big\|\frac{1}{U+i\lambda}\Big\|_1\,\|\phi\|_\infty\leq C \log|\mu|^{-1} \|\phi\|_\infty \,.
\end{displaymath}
Sobolev embeddings, given that $|\alpha|<\delta_0\, \min (|\lambda|^{1/2}, \log^{-1/2}|\mu|^{-1})$, complete the
proof of the corollary when $p<\infty$ for sufficiently small $\delta_0$.  

\subsection{Small $|\lambda|$ and $ \alpha>\|U\|_2^{-1}(\Im
  \lambda)_+^{1/2}$.}
\label{sec:case-lambda-small-alpha-large}
Set for any $\lambda\in\C$ and $\delta>0$
\begin{equation}
\label{eq:51} 
  \alpha_{\lambda,\delta}=\|U\|_2^{-1}\left( (\Im\lambda)_+(1+2\delta)\right)^{1/2}\,.
\end{equation}
In this subsection we attempt to prove the invertibility of $\mathcal
A_{\lambda,\alpha}$ as defined in \eqref{eq:domND} for sufficiently small  $|\lambda|$ 
and $\alpha \geq \alpha_{\lambda,\delta}$.

To be able to state the result of this subsection we define, for $p>1$
and $\Re \lambda \neq 0$, on $W^{1,p}(0,1)$ the maps 
\begin{subequations}
\label{eq:49}
  \begin{equation}
  v \mapsto N_0(v,\lambda) :=\min \Big(\Big\|(1-x)^{1/2}\frac{v}{U+i\lambda}\Big\|_1,\|v\|_{1,p}\Big)\,,
\end{equation}
and
\begin{equation} 
\label{eq:50}
  N_1(v,\lambda)=\Big| \Big\langle1,\frac{v}{U+i\lambda}\Big\rangle \Big| \,.
\end{equation}
\end{subequations}
We can now state and prove the following
\begin{proposition}
\label{prop:edge}
Let $r>1$, $p>1$, and $\delta>0$ and $U\in C^3([0,1]) $
satisfy \eqref{eq:10}. There exist $\lambda_0>0$ and $C>0$ such that for $0<|\lambda|<\lambda_0$, $\alpha\geq\alpha_{\lambda,\delta} $
and \break $(\phi,v)\in D(\A_{\lambda,\alpha})\times W^{1,p}(0,1)$ satisfying  $\mathcal A_{\lambda,\alpha}
\phi=v$,  we have, with $c_\parallel = \langle U,\phi\rangle /\|U\|_2^2$ and $\phi_\perp = \phi - c_\parallel U $,
\begin{subequations}
\label{eq:52}
\begin{equation}
    |c_\parallel |\leq \frac{1+C|\lambda|^2\log
 |\lambda|^{-1}}{|\alpha^2\|U\|_2^2+i\lambda|}\big[\|v\|_1+ C |\lambda| N_1(v,\lambda)\big] \,,
  \end{equation}
and
\begin{equation}
   \|\phi_\perp\|_{1,2}\leq
   C\Big[N_0(v,\lambda)+\frac{|\lambda|}{|\alpha^2\|U\|_2^2+i\lambda|} \Big(\|v\|_1+
   |\lambda|  N_1(v,\lambda) \Big)\Big] 
\end{equation} 
\end{subequations}
\end{proposition}
\begin{proof}~\\
    {\em Step 1: We prove the existence of $\lambda_0 >0$ and $C>0$ such
    that,   for $|\lambda| \leq \lambda_0$ and $\alpha \geq \alpha_{\lambda,\delta}$ it holds that 
\begin{equation}
  \label{eq:53}
|c_\parallel |\leq \frac{\quad 1+C|\lambda|^2\log
 |\lambda|^{-1}}{|\alpha^2\|U\|_2^2+i\lambda|}\Big(\|v\|_1+C|\lambda|\|\phi_\perp\|_{1,2}+|\lambda| N_1(v,\lambda) \Big) \,.
\end{equation}
}

As
\begin{equation}\label{eq:54}
  U(-\phi^{\prime\prime}+\alpha^2\phi)-U^{\prime\prime}\phi= v- i\lambda(-\phi^{\prime\prime}+\alpha^2\phi)
\end{equation}
or equivalently, by \eqref{eq:20}  and \eqref{eq:22}
\begin{equation}\label{eq:55}
  U^2(\Mg_U+\alpha^2)w= \frac{Uv }{U+i\lambda}+
  i\lambda\frac{U^{\prime\prime}\phi}{U+i\lambda}\,,
\end{equation}
where $\Mg_U$ is given by \eqref{eq:11} and $w=U^{-1}\phi$. \\
Taking the inner product with $1$  and integrating by parts then
yields,
\begin{equation} \label{eq:56}
\alpha^2\|U\|_2^2\, c_\parallel  = \Big\langle1,\frac{Uv
}{U+i\lambda}\Big\rangle+i\lambda\Big\langle1,\frac{U^{\prime\prime}\phi }{U+i\lambda}\Big\rangle
\end{equation}
We now write
\begin{equation}
\label{eq:57}
  \Big\langle1,\frac{U^{\prime\prime}\phi
  }{U+i\lambda}\Big\rangle=c_\parallel \Big\langle1,\frac{U^{\prime\prime}U}{U+i\lambda}\Big\rangle+\Big\langle1,\frac{U^{\prime\prime}\phi_\perp}{U+i\lambda}\Big\rangle
\end{equation}
For the first term on the right-hand-side  we write, as
$U^\prime(0)=0$ and $U^\prime(1)=-1$,
\begin{equation}
\label{eq:58}
  c_\parallel \Big\langle1,\frac{U^{\prime\prime}U}{U+i\lambda}\Big\rangle=-c_\parallel \Big(1+i\lambda\Big\langle1,\frac{U^{\prime\prime}}{U+i\lambda}\Big\rangle\Big)
\end{equation}
Since,  for $\Im\lambda<U(1/2)$, 
\begin{multline*}
  \Big\langle1,\frac{U^{\prime\prime}}{U+i\lambda}\Big\rangle_{L^2(1/2,1)}=-U^{\prime\prime}(1)\log
  (i\lambda) -\frac{U^{\prime\prime}(1/2)}{U^\prime(1/2)}\log(U(1/2)+i\lambda)\\ - \Big\langle
\Big(\frac{U^{\prime\prime}}{U^\prime}\Big)^\prime,\log (U+i\lambda)\Big\rangle_{L^2(1/2,1)}\,,
\end{multline*}
from which we may conclude the existence of $C>0$ and $0 < \lambda_0 < U(1/2)$, such that, for $|\lambda|\leq \lambda_0$,
\begin{displaymath}
    \Big|\Big\langle1,\frac{U^{\prime\prime}}{U+i\lambda}\Big\rangle_{L^2(1/2,1)}\Big|\leq C\log |\lambda|^{-1}\,.
\end{displaymath}
As
\begin{displaymath}
    \Big|\Big\langle1,\frac{U^{\prime\prime}}{U+i\lambda}\Big\rangle_{L^2(0,1/2)}\Big|\leq C\,,
\end{displaymath}
we obtain 
\begin{equation}
\label{eq:59}
  \Big|\Big\langle1,\frac{U^{\prime\prime}}{U+i\lambda}\Big\rangle\Big|\leq C\log |\lambda|^{-1}\,.
\end{equation}
 
Next, we consider the second term on the right-hand-side  of
\eqref{eq:57} (note that $\phi_\perp(1)=0$)
\begin{displaymath} 
\begin{array}{ll}
  \Big|\Big\langle1,\frac{U^{\prime\prime}\phi_\perp}{U+i\lambda}\Big\rangle_{L^2(1/2,1)}\Big|&
  =\Big | -\frac{U^{\prime\prime}\phi_\perp}{U^\prime}\log (U(1/2)+i\lambda))_{|_{x=1/2}}\,
  + \, \Big\langle\Big(\frac{U^{\prime\prime}\phi_\perp}{U^\prime}\Big)^\prime,\log
  (U+i\lambda)\Big\rangle\Big|\\ &  \leq 
  C\, \|\phi_\perp\|_{1,2} \,.
  \end{array}
\end{displaymath}
For $\Im \lambda < U(1/2)$ we can write,
\begin{displaymath}
\Big|  \Big\langle1,\frac{U^{\prime\prime}\phi_\perp}{U+i\lambda}\Big\rangle_{L^2(0,1/2)}\Big|\leq
  C\|\phi_\perp\|_1\,,
\end{displaymath}
and hence we may conclude that
\begin{equation}
\label{eq:60}
 \Big|   \Big\langle1,\frac{U^{\prime\prime}\phi_\perp}{U+i\lambda}\Big\rangle\Big|\leq
  C\|\phi_\perp\|_{1,2}\,.
\end{equation}
Substituting the above, together with \eqref{eq:58}  and \eqref{eq:59} 
into \eqref{eq:57} yields
\begin{equation}
\label{eq:61}
 \Big| \Big\langle1,\frac{U^{\prime\prime}\phi  }{U+i\lambda}\Big\rangle+c_\parallel \Big|\leq C|\lambda|\log
 |\lambda|^{-1}c_\parallel +C\|\phi_\perp\|_{1,2} \,.
\end{equation}
We next rewrite \eqref{eq:56} in the form
\begin{equation}\label{eq:62}
(\alpha^2\|U\|_2^2 +i \lambda) \, c_\parallel  = \Big\langle1,\frac{Uv
}{U+i\lambda}\Big\rangle+i\lambda \left(\Big\langle1,\frac{U^{\prime\prime}\phi }{U+i\lambda}\Big\rangle+ c_\parallel \right)\,,
\end{equation}
and then observe that
\begin{displaymath}
  \Big\langle1,\frac{Uv}{U+i\lambda}\Big\rangle= \Big|\langle1,v\rangle
  -i\lambda\Big\langle1,\frac{v}{U+i\lambda}\Big\rangle\Big|\leq\|v\|_1+|\lambda|N_1(v,\lambda)\,.
\end{displaymath}
Substituting the above together with \eqref{eq:61} into \eqref{eq:62} yields
\begin{displaymath}
|c_\parallel |\leq \frac{1}{|\alpha^2\|U\|_2^2+i\lambda|}\Big(\|v\|_1+C|\lambda|^2\log
 |\lambda|^{-1}c_\parallel +C|\lambda|\|\phi_\perp\|_{1,2}+|\lambda| N_1(v,\lambda) \Big) \,.
\end{displaymath}
Given the fact that when $\Im \lambda >0$,  $\alpha\geq\alpha_{\lambda,\delta} $ we have 
\begin{displaymath}
    \frac{1}{|\alpha^2\|U\|_2^2+i\lambda|}\leq\frac{1}{\delta|\lambda|}\,,
  \end{displaymath} 
and that when $\Im\lambda\leq 0$ we have
\begin{displaymath}
    \frac{1}{|\alpha^2\|U\|_2^2+i\lambda|}\leq\frac{1}{|\lambda|}\,,
  \end{displaymath}
  we obtain  \eqref{eq:53} for
sufficiently small $\lambda_0>0$.

{\em Step 2: We prove \eqref{eq:52}.}

{\em Step  2.1:} {\em For $w_\perp=U^{-1}\phi_\perp$, we prove that
\begin{equation}
\label{eq:928}
   \|Uw_\perp^\prime\|_2^2 + \alpha^2\|\phi_\perp\|_2^2 \leq C\|\phi_\perp^\prime\|_2(N_0(v,\lambda)+ |\lambda|\, |c_\parallel|)+ \Big|
  i\lambda\Big\langle\frac{\phi_\perp}{U},\frac{U^{\prime\prime}\phi_\perp}{U+i\lambda}\Big\rangle\Big|\,.  
\end{equation}}

Taking the inner product in \eqref{eq:54} with $w_\perp$ yields (note that
$w^\prime=w_\perp^\prime$)
\begin{equation}
\label{eq:63}
  \|Uw_\perp^\prime\|_2^2 + \alpha^2\|\phi_\perp\|_2^2 
=  \Big\langle \phi_\perp,\frac{v}{U+i\lambda}\Big\rangle 
  +i\lambda\Big\langle w_\perp,\frac{U^{\prime\prime}\phi}{U+i\lambda}\Big\rangle \,. 
\end{equation}
We now turn to estimate the right-hand-side  of \eqref{eq:63}. For the
first term we have,  using the fact that $|\phi_\perp(x)|\leq (1-x)^{1/2}
\|\phi_\perp^\prime\|_2$, 
\begin{equation}
\label{eq:64}
  \Big| \Big\langle \phi_\perp,\frac{v}{U+i\lambda}\Big\rangle\Big | \leq  \|\phi_\perp^\prime\|_2\,
  \Big\|(1-x)^{1/2}\frac{v}{U+i\lambda}\Big\|_1\,.
\end{equation} 
Furthermore, splitting the domain of integration $(0,1)$ into two
intervals: $(0,1/4)$ and $(1/4,1)$ and then integrating by parts on
$(1/4,1)$ yield, for $0 < |\lambda| <\lambda_0\leq U(1/2)$,
\begin{multline}
\label{eq:65}
    \Big\langle \phi_\perp,\frac{v}{U+i\lambda}\Big\rangle =
    -\overline{\phi}_\perp v\log(U+i\lambda)\big|_{x=\frac{1}{4}} \\ -\int_{1/4}^1
    \log\,(U+i\lambda)\Big(\frac{\overline{\phi_\perp }\, v}{U^\prime}\Big)^\prime\,dx +\int_0^{1/4}
    \overline{\phi_\perp}\frac{v}{U+i\lambda} \,dx\,.
\end{multline}
 Using a Sobolev embedding and the Poincar\'e inequality yields for $|\lambda| \leq U(1/2)$
    \begin{equation}
\label{eq:66}
      \big|\overline{\phi}_\perp
      v\log(U+i\lambda)\big|_{x=\frac{1}{4}}\big|\leq C\|\phi_\perp\|_\infty\|v\|_\infty \leq
      C\|\phi_\perp^\prime\|_2\|v\|_{1,p}\,.
    \end{equation}
Furthermore, it holds that
\begin{equation}
\label{eq:67}
  \Big|\int_0^{1/4}   \overline{\phi_\perp}\frac{v}{U+i\lambda}
  \,dx\Big|\leq C \, \|v\|_2\|\phi_\perp\|_2 \,,
\end{equation}
and, as in the proof of \eqref{eq:45},
\begin{displaymath}
  \Big|\int_{1/4}^1
    \log\,(U+i\lambda)\Big(\frac{\overline{\phi_\perp }\, v}{U^\prime}\Big)^\prime\,dx
    \Big|\leq C\, (\|v\|_{1,p}\|\phi_\perp\|_\infty+\|v\|_\infty \|\phi^\prime_\perp\|_2)  
\end{displaymath}
 Substituting the above, together with \eqref{eq:67} and
\eqref{eq:66} into \eqref{eq:65} we can conclude, with the aid of Poincar\'e's inequality and
Sobolev's embeddings, that
\begin{equation} 
\label{eq:68}
   \Big| \Big\langle \phi_\perp,\frac{v}{U+i\lambda}\Big\rangle\Big | \leq  C\, \|\phi_\perp^\prime\|_2\|v\|_{1,p}\,.
\end{equation}
Combining \eqref{eq:68}  with \eqref{eq:64} and \eqref{eq:49})  yields
\begin{equation}
  \label{eq:69}
  \Big| \Big\langle \phi_\perp,\frac{v}{U+i\lambda}\Big\rangle\Big | \leq  C\|\phi_\perp^\prime\|_2N_0(v,\lambda)\,.  
\end{equation}
For the second term we  have, using the decomposition \eqref{eq:19}
\begin{equation}
\label{eq:70}
  i\lambda \Big\langle w_\perp,\frac{U^{\prime\prime}\phi}{U+i\lambda}\Big\rangle=ic_\parallel \lambda\Big\langle
  \phi_\perp,\frac{U^{\prime\prime}}{U+i\lambda}\Big\rangle + i\lambda\Big\langle
  w_\perp,\frac{U^{\prime\prime}\phi_\perp}{U+i\lambda}\Big\rangle \,.
\end{equation}
To estimate  the first term in \eqref{eq:70} we use  \eqref{eq:60} to obtain
\begin{equation} 
\label{eq:71}
\Big| c_\parallel \lambda\Big\langle
  \phi_\perp,\frac{U^{\prime\prime}}{U+i\lambda}\Big\rangle \Big| \leq  C \, |\lambda|\, |c_\parallel|  \|\phi_\perp\|_{1,2}\,.
\end{equation}
 Substituting the above together with \eqref{eq:70} and
  \eqref{eq:69} into \eqref{eq:63} yields \eqref{eq:928}.  The
estimate of  the last term in \eqref{eq:928} is the object of the  
next step. 

{\em Step  2.2: We prove that for every $\epsilon>0$ there exists $\lambda_0>0$
such that for all  $|\lambda|\leq \lambda_0$ it holds that
\begin{equation}
\label{eq:72}
  \Big|
  i\lambda\Big\langle\frac{\phi_\perp}{U},\frac{U^{\prime\prime}\phi_\perp}{U+i\lambda}\Big\rangle\Big|\leq\epsilon\, \|\phi^\prime_\perp\|_2^2 \,.  
\end{equation}
}
 
\noindent Clearly,
\begin{equation}
\label{eq:73}
  i\lambda\Big\langle\frac{\phi_\perp}{U},\frac{U^{\prime\prime}\phi_\perp}{U+i\lambda}\Big\rangle= i\lambda \int_0^1
    \frac{U^{\prime\prime}|\phi_\perp|^2}{U(U+i\lambda)}\,dx \,.
\end{equation}
Recall the definition of $x_\nu$ in \eqref{eq:defxnu}, and let
$1-x_\nu\leq d<1/2$.

\paragraph{The integral over $(1-2d,1)$.} ~\\
We attempt, see step 2 of the proof of
  \cite[Proposition 4.14]{almog2019stability},  to prove that for
  $d<\frac 12$ there exists $\hat C >0$ and $\lambda_0$ such that for $|\lambda|
  \leq \lambda_0$ 
\begin{equation}
\label{eq:74}
  \Big|i\lambda\int_{1-2d}^1
    \frac{U^{\prime\prime}|\phi_\perp|^2}{U(U+i\lambda)}\,dx\Big| \leq \hat C\, d\log
   |d|^{-1}\|\phi_\perp^\prime\|_2^2\,. 
\end{equation}
To estimate this  integral 
we use the identity
\begin{displaymath}
  \frac{1}{U(U+i\lambda)}=\frac{1}{i\lambda}\Big[\frac{1}{U}-\frac{1}{U+i\lambda}\Big] \,,
\end{displaymath}
to obtain that
\begin{equation}
\label{eq:75} 
   i\lambda\int_{1-2d}^1
    \frac{U^{\prime\prime}|\phi_\perp|^2}{U(U+i\lambda)}\,dx =
   \int_{1-2d}^1\frac{U^{\prime\prime}|\phi_\perp|^2}{U}\,dx- \int_{1-2d}^1\frac{U^{\prime\prime}|\phi_\perp|^2}{U+i\lambda}\,dx \,.
\end{equation}
Integration by parts yields
\begin{equation}
\label{eq:76}
  \int_{1-2d}^1\frac{U^{\prime\prime}|\phi_\perp|^2}{U+i\lambda}\,dx = \left(\frac{U^{\prime\prime}}{U^\prime}
  |\phi_\perp|^2\,|\log\,(U+i\lambda)\Big|\right)_{\Big | x= 1-2d } \\ -
\int_{1-2d}^1\Big(\frac{U^{\prime\prime}}{U^\prime}|\phi_\perp|^2\Big)^\prime \,\,|\log\,(U+i\lambda)|\,dx \,.
\end{equation}
Next we observe that
\begin{equation}
\label{eq:77}
   |\phi_\perp (x)|^2 \leq 2d \,\|\phi_\perp^\prime\|_2^2  \mbox{ and } 0 \leq U(x) \leq 2d\,, \, \forall x \in [1-2d,1]\,.
\end{equation}
We also note that, for $\lambda_0$ small enough and $|\lambda| \leq \lambda_0$, 
\begin{equation}
|\log\,(U (1-2d) +i\lambda)\Big| \leq C\log|d|^{-1}\,.
\end{equation}
Hence, the first term on the right hand side of \eqref{eq:76} can be
estimated as follows
\begin{equation}\label{eq:78}
\Big| \left(\frac{U^{\prime\prime}}{U^\prime}
  |\phi_\perp|^2\,|\log\,(U+i\lambda)\Big|\right)_{x= 1-2d } \Big| \leq C\, d \log|d|^{-1}\,  \|\phi_\perp^\prime\|_2^2\,.
  \end{equation}
For the second term, we write
\begin{equation}\label{eq:79}
\begin{array}{l}
  \Big|\int_{1-2d}^1\Big(\frac{U^{\prime\prime}}{U^\prime}|\phi_\perp|^2\Big)^\prime
  \,\,|\log\,(U+i\lambda)|\,dx \Big|\,\\ \qquad  \leq
  \Big\|\frac{U^{\prime\prime}}{U^\prime}\Big\|_{W^{1,\infty}(1-2d,1)}\int_{1-2d}^1(|\phi_\perp|^2+2|\phi_\perp|\,|\phi_\perp^\prime|)
  \,\,|\log\,(U+i\lambda)|\,dx\\ \qquad  \leq  C\int_{1-2d}^1(|\phi_\perp|^2+ 
  |\phi_\perp|\,|\phi^\prime_\perp| )
  \,\,|\log\,(U+i\lambda)|\,dx \,.
  \end{array}
\end{equation}
As
\begin{equation}
  \|\log (U+i\lambda)\|_{L^p(1-2d,1)}\leq Cd^{1/p}\log |d|^{-1}\quad\mbox{ for } 
  p\in\{1,2\}\,,
\end{equation}
we obtain,  using  \eqref{eq:77}, from \eqref{eq:79}  that
\begin{displaymath}
  \Big|\int_{1-2d}^1\Big(\frac{U^{\prime\prime}}{U^\prime}{|\phi_\perp |}^2\Big)^\prime
  \,\,|\log\,(U+i\lambda)|\,dx \Big|\,\leq C d\log |d|^{-1}\|\phi_\perp^\prime\|_2^2\,.
\end{displaymath}
Substituting the above, together with~\eqref{eq:77} into \eqref{eq:76}
then yields
\begin{equation}
\label{eq:80}
 \Big|\int_{1-2d}^1\frac{U^{\prime\prime}|\phi_\perp|^2}{U+i\lambda}\,dx\Big| \leq  C d\log
   |d|^{-1}\|\phi_\perp^\prime\|_2^2\,. 
\end{equation}
We now estimate the first term in the right hand side of
\eqref{eq:75}.  Employing \eqref{eq:25} and Poincar\'e's inequality
yields
\begin{displaymath} 
   \Big|\int_{1-2d}^1\frac{U^{\prime\prime}|\phi_\perp|^2}{U}\,dx\Big|\leq C  \| \phi_\perp\| _{L^2(1-2d,1)} \|\phi_\perp/U\|_{L^2(1-2d,1)} \leq 
   \tilde{C} \, d \,\|\phi_\perp^\prime\|_2^2\,.
\end{displaymath}
Substituting the above together with \eqref{eq:80} into \eqref{eq:75} yields \eqref{eq:74}.

\paragraph{The integral over $[0,1-2d]$.} ~\\
By \eqref{eq:10} there exists $C>0$ such that for all $1-x_\nu\leq d<1/2$
\begin{displaymath}
  \Big\|\frac{1}{U+i\lambda}\Big\|_{L^\infty(0,1-2d)}\leq \frac{1}{U(1-2d)-\nu}\leq
  \frac{C}{d} \,.
\end{displaymath}
Hence, given that $U^\prime(1) <0$, 
  \begin{displaymath}
    \Big| \lambda \int_{0}^{1-2d} \frac{U^{\prime\prime}|\phi_\perp|^2}{U(U+i\lambda)}\,dx\Big|
    \leq C \frac{|\lambda|}{ d^2} \| \phi_\perp \|^2_2\,.
  \end{displaymath}
Combining the above with \eqref{eq:74} yields that for all $1-x_\nu\leq
d$, with the aid of Poincar\'e's inequality 
\begin{equation}\label{eq:81}
      \Big| \lambda \int_0^1 \frac{U^{\prime\prime}|\phi_\perp|^2}{U(U+i\lambda)}\,dx\Big| \leq C
      \Big(\frac{|\lambda|}{ d^2}+d\log |d|^{-1}\Big) \| \phi_\perp^\prime \|^2_2\,.
\end{equation}
Let $\epsilon>0$. Clearly, there exists $d(\epsilon)>0$ such that for $d\in
(0,d(\epsilon))$ 
\begin{displaymath}
  d\log |d|^{-1}\leq\frac{\epsilon}{2C} \,.
\end{displaymath}
Furthermore, for $(\epsilon,d)$ as above we add the condition
\begin{displaymath}
  \lambda_0 \leq\min\Big(\frac{ \epsilon d^2}{2C}, { d\, U(0)} \Big)\,.
\end{displaymath}
As
\begin{displaymath}
  U(0)(1-x_\nu)\leq|\nu|<|\lambda|<\lambda_0\leq  d \, U(0) \,,
\end{displaymath}
we obtain that $d\geq1-x_\nu$, therefore, \eqref{eq:72} can now be
verified with the aid of \eqref{eq:81}.\\

{\em Step  2.3: We complete the proof of \eqref{eq:52}.}\\

\noindent  Substituting \eqref{eq:72} into \eqref{eq:928}
yields for any $\epsilon >0$,  the existence of $C>0$ and $\lambda_\epsilon >0$, such that for  $|\lambda| < \lambda_\epsilon\,$, 
it holds that 
\begin{equation}
\label{eq:82} 
\Big\|Uw_\perp^\prime\Big\|_2^2+\alpha^2\|\phi_\perp\|_2^2
\leq C\|\phi_\perp^\prime\|_2(N_0 (v,\lambda)+ |\lambda|
|c_\parallel |) + \epsilon\|\phi^\prime_\perp\|_2^2 \,.  
\end{equation}
We now attempt to bound $\|\phi_\perp^\prime\|_2$. As  
\begin{displaymath}
\phi^\prime_\perp = U w^\prime_\perp + U^\prime w_\perp\,,
\end{displaymath} 
and since $\|U^\prime\|_\infty \leq 1$ we may use \eqref{eq:25} to obtain
\begin{equation}\label{eq:83}
  \|\phi_\perp^\prime\|_2^2  \leq 2\|Uw_\perp^\prime\|_2^2+ C  \|\phi_\perp\|_2^2 \,.
\end{equation}
On the other hand, by \eqref{eq:12} and \eqref{eq:13} we have that
\begin{displaymath}
  \|\phi_\perp\|_2^2 \leq C\|Uw_\perp^\prime\|_2^2\,,
\end{displaymath}
and hence, combining with \eqref{eq:83}, we obtain
\begin{displaymath}
    \|\phi_\perp^\prime\|_2  \leq C\|Uw_\perp^\prime\|_2\,.
\end{displaymath}
Substituting the above into \eqref{eq:82} yields, with the aid of
Poincar\'e's inequality  and a suitable choice of $\epsilon$,  the existence
of  $\lambda_0$ and $C>0$  such that for  $|\lambda| \leq \lambda_0$,   
\begin{equation}
\label{eq:84}
  \|\phi_\perp^\prime\|_2\leq C(N_0(v,\lambda)+|\lambda|\,|c_\parallel | )\,.  
\end{equation}
We now combine \eqref{eq:84}  with \eqref{eq:53} to obtain (\ref{eq:52}a,b).\\
\end{proof}

 While a direct use of \eqref{eq:52} will be made in the proof of
\eqref{lem:other}, we shall also need to transform $N_1(v,\lambda)$ into a
more conventional bound, which is precisely what we achieve in the 
next lemma.
 \begin{lemma}
\label{rem:integral-over-0}
Let $U\in C^2([0,1])$ satisfy \eqref{eq:10}, $p >1$, $
0<\nu_0<U(1/2)$, and $\mu_0>0$.  There exist $C>0$ and $\beta_0>0$ such that
for all $\beta\geq \beta_0$, $|\mu|< \mu_0$ and $|\nu| \leq \nu_0$ it holds that
\begin{equation}
  \label{eq:85}
N_1(v,\lambda)\leq |v(1)|\log|\lambda|+C\|v\|_{1,p} \,,
\end{equation}
where $N_1$ is introduced in \eqref{eq:50}.
\end{lemma}
\begin{proof}
  Clearly,
    \begin{displaymath}
          \Big\langle1,\frac{v}{U+i\lambda}\Big\rangle=
    \Big\langle1,\frac{v}{U+i\lambda}\Big\rangle_{L^2(0,1/2)}+  \Big\langle1,\frac{v}{U+i\lambda}\Big\rangle_{L^2(1/2,1)}
    \end{displaymath}
  Integration by parts yields
  \begin{displaymath}
  \Big\langle1,\frac{v}{U+i\lambda}\Big\rangle_{L^2(1/2,1)}
    =v(1)\log(i\lambda)-\frac{v(1/2)}{U^\prime(1/2)}\log(U(1/2)+i\lambda)- \int_{1/2}^1\Big(\frac{v}{U^\prime}\Big)^\prime \log(U+i\lambda)\,dx \,.
  \end{displaymath}
 Furthermore, as $\nu<U(1/2)$ it holds that
\begin{displaymath}
  \Big|\Big\langle1,\frac{v}{U+i\lambda}\Big\rangle_{L^2(0,1/2)}\Big|\leq\frac{\|v\|_1}{|U(1/2)+i\lambda|}
\end{displaymath}
Consequently, by \eqref{eq:50}, we can conclude \eqref{eq:85} for any
$p>1$.
\end{proof}

\subsection{The case $0< \Im \lambda < U(0)$}\label{ss2.6}

Propositions \ref{lem:lambda-alpha-zero},  \ref{lem:lambda-zero} and
 \ref{prop:edge} address some of the cases where $|\lambda|$ is
small.  We now consider the case $ 0<\nu<\nu_1$ for some
$\nu_1<U(0)$, (recall that $\nu =\Im \lambda$). For later reference  (see Lemma
\ref{lem:inviscid-decay} and Proposition \ref{lem:right-curve}),  we also
consider the case where $\nu$ is small  using a different approach
  than that of the previous subsection.  We set, for $p>1$ and $v\in W^{1,p} (0,1)$,
\begin{multline}
\label{eq:86}
 N(v,\lambda) := \min
  \Big(\Big\|[(1-x)^{1/2}+\nu^{-1/2}(1-x)]\frac{v}{U+i\lambda}\Big\|_1, \\ |\nu|^{1/2}(\|v^\prime\|_p+
      |v(1)|\log|\nu|^{-1})+\|v\|_2+\nu^{-1/2}\|v\|_1\Big)\,.
\end{multline}
 For small values of $|\lambda|$, since $\A_{i\nu,0}(U-\nu)=0$  and since
$U(1)-\nu=\nu\ll1$, we expect $\A_{\lambda,0}$ to be almost singular, and that
the norm of $\phi=\A_{\lambda,0}^{-1}v$, would be much greater in the space
spanned by $U-\nu$. We  thus use the decomposition 
\begin{displaymath}
  \phi=c_\parallel^\nu(U-\nu)+ [\phi-c_\parallel^\nu(U-\nu)] \,,
\end{displaymath}
where $c_\parallel^\nu$ is defined by \eqref{eq:88}. 

The proof is somewhat similar to the proof of \cite[Proposition
4.15]{almog2019stability}. In addition to the above-mentioned
difference, resulting from the non-invertibility  of $\A_{0,0}$, we need to
address here the Neumann condition at $x=0$, which 
complicates the  estimate of $\|\phi^\prime\|_{L^2(0,x_\nu)}$, where $x_\nu$
is given by \eqref{eq:defxnu}. This estimate, which is addressed in
step 3 of the proof, significantly contributes to its length and
complexity.
\begin{proposition}
\label{prop:linear}
Let $p>1$, 
$U\in C^3([0,1])$  satisfy
\eqref{eq:10}, and $0<\nu_1<U(0)$. 
There exist $\mu_0>0$ and $C>0$ such that
for all $\lambda=\mu+i \nu$ such that $ 0<\nu<\nu_1$, and $0<|\mu|\leq \mu_0$ and $\alpha \geq 0$,  we have,  
for all  $(\phi,v)\in D(\A_{\lambda,\alpha})\times W^{1,p}(0,1)$ satisfying  $\mathcal
A_{\lambda,\alpha} \phi=v$, (where $\A_{\lambda,\alpha}$ is defined in \eqref{eq:2.1aa})
   \begin{equation}
\label{eq:87}
\|\phi-c_\parallel^\nu(U-\nu)\|_{1,2}  + \nu^{1/2}|c_\parallel^\nu|  \leq \frac{C}{\nu} \,  N (v,\lambda)   \,,
\end{equation}
where 
\begin{equation}
  \label{eq:88}
c_\parallel^\nu=\frac{\langle\phi-\phi(x_\nu),U-\nu\rangle_{L^2(0,x_\nu)}}{\|U-\nu\|_{L^2(0,x_\nu)}^2}\,,
\end{equation}
in which $x_\nu$ is defined by \eqref{eq:defxnu}.
\end{proposition}
\begin{proof}~\\
  {\em Step 1:
We prove that there exists $C>0$ such that,
for all $0<|\mu|\leq 1$ it holds that 
 \begin{equation}
\label{eq:89} 
   |\phi(x_\nu)| \leq C \Big|\Big\langle\phi,\frac{v}{U+i\lambda}\Big\rangle\Big|^{1/2}\,.
 \end{equation}
 for all pairs $(\phi,v)\in D(\A_{\lambda,\alpha})\times W^{1,p}(0,1)$ satisfying
 \eqref{eq:47}. \\ }

\noindent It can be easily  verified (since $\mathcal A_{\lambda,\alpha} \phi=v$) that
\begin{equation}
\label{eq:90}
  \Im\Big\langle\phi,\frac{v}{U-\nu+i\mu}\Big\rangle=-\mu\, \Big\langle\frac{U^{\prime\prime}\phi}{(U-\nu)^2+\mu^2},\phi\Big\rangle\,.
\end{equation}
As 
\begin{displaymath}
  |\phi (x)|^2\geq \frac{1}{2}|\phi(x_\nu)|^2 - |\phi(x)-\phi(x_\nu)|^2 \,,
\end{displaymath}
we may use \eqref{eq:90} to obtain, observing that $-U^{\prime\prime} >0$, 
\begin{equation} 
\label{eq:91}
 \Big|  \Im\Big\langle\phi,\frac{v}{U-\nu+i\mu}\Big\rangle\Big| \geq |\mu| \,
   \Big\langle\frac{|U^{\prime\prime}|}{(U-\nu)^2+\mu^2},\frac 12 |\phi(x_\nu)|^2 - |\phi(x)-\phi(x_\nu)|^2 \Big\rangle\,.
\end{equation}
Hence
 \begin{equation}  
\label{eq:92}
   \frac{|\mu| }{2} |\phi(x_\nu)|^2 \Big\|\frac{|U^{\prime\prime}|^{1/2}}{U+i\lambda}\Big\|_2^2 \leq
   |\mu| \sup |U^{\prime\prime}| \Big\|\frac{\phi-\phi(x_\nu)}{U+i\lambda}\Big\|_2^2 
  +\Big|\Big\langle\phi,\frac{v}{U-\nu+i\mu}\Big\rangle\Big| \,.
 \end{equation}
 Since  $|U(x)-\nu| \leq |x-x_\nu|$, and $|U^{\prime\prime}| >0$  we can conclude, that
 for some positive $C$
 \begin{displaymath}
    \Big\|\frac{|U^{\prime\prime}|^{1/2}}{U+i\lambda}\Big\|_2^2  \geq   \frac 1C
   \int_0^1  \frac{1}{(x-x_\nu)^2+\mu^2} dx\,.
 \end{displaymath}
 As $|\mu|\leq 1$ and $x_\nu \in [x_{\nu_1},1]$, we obtain after the change
 of variable $y= (x_\nu-x)/|\mu|$
\begin{displaymath}
 \int_0^1  \frac{1}{(x-x_\nu)^2+\mu^2} dx = \frac{1}{|\mu|}
 \int_{\frac{x_\nu-1}{|\mu|}}^{\frac{x_\nu}{|\mu|}} \frac{1}{1+y^2} dy
 \geq\frac{1} {|\mu|} 
 \int_0^{\frac{x_{\nu_1}}{|\mu|}} \frac{1}{1+y^2} dy\,. 
\end{displaymath}
 Consequently, there exists $\hat C >0$, such that for all $|\mu| \leq 1$ and $\nu \in (0,\nu_1)$, 
    \begin{equation}
\label{eq:93}
    \Big\|\frac{|U^{\prime\prime}|^{1/2}}{U+i\lambda}\Big\|_2^2 \geq   \frac {1}{\hat C} |\mu|^{-1}  \,.
 \end{equation}
A similar argument is employed in the proof of \cite[Proposition
4.14]{almog2019stability} (see between Eq. (4.59) and (4.60) there).  
By \eqref{eq:92} we then have 
 \begin{equation}
\label{eq:94}
    |\phi(x_\nu)|^2 \leq C\Big[ |\mu|
    \Big\|\frac{\phi-\phi(x_\nu)}{U+i\lambda}\Big\|_2^2 + \Big|\Big\langle\phi,\frac{v}{U-\nu+i\mu}\Big\rangle\Big| \Big]\,.
 \end{equation}
To estimate the first  term on the right-hand-side  of
\eqref{eq:94} we first observe that for some $C=C(\nu_1)>0$ we have, for all $\lambda=\mu+ i\nu$ such that $0 < \nu <\nu_1$
\begin{displaymath}
  \Big|\frac{1}{U(x)+i\lambda}\Big| \leq \frac{C}{| x -x_\nu|}\,,\,  \forall x \in (x_{\nu_1}/4  ,1) \,,
\end{displaymath}
where $x_{\nu_1}=x_\nu\big|_{\nu=\nu_1}$,
and 
\begin{displaymath}
  \Big|\frac{1}{U(x)+i\lambda}\Big| \leq  C  \,,\,  \forall x\in (0,x_{\nu_1}/4) \,.
\end{displaymath}
Consequently, we may write
\begin{equation}
\label{eq:95}
  \Big|\frac{1}{U(x)+i\lambda}\Big| \leq \frac{C}{|x-x_\nu|}\,,\, \forall x\in (0,1) \,.
\end{equation}
We now apply Hardy's inequality \eqref{eq:18} to
$w=(x-x_\nu)^{-1}(\phi-\phi(x_\nu))$ in $(x_\nu,1)$. It follows that
\begin{equation}\label{hardyw1}
  \Big\|\frac{\phi-\phi(x_\nu)}{x-x_\nu}\Big\|_{L^2(x_\nu,1)}^2 \leq
  \frac{1}{4}\|\phi^\prime\|_{L^2(x_\nu,1)}^2 \,.
\end{equation}
A similar bound can be the interval $(0,x_\nu)$:
\begin{equation}\label{hardyw1a}
  \Big\|\frac{\phi-\phi(x_\nu)}{x-x_\nu}\Big\|_{L^2(0,x_\nu)}^2 \leq
  \frac{1}{4}\|\phi^\prime\|_{L^2(0,x_\nu)}^2 \,.
\end{equation}
Consequently,
\begin{equation}
\label{eq:96}
  \Big\|\frac{\phi-\phi(x_\nu)}{x-x_\nu}\Big\|_{L^2(0,1)}^2 \leq
  \frac{1}{4}\|\phi^\prime\|_{L^2(0,1)}^2 \,.
\end{equation}
from which we easily conclude, in view of \eqref{eq:95}, 
\begin{equation}
\label{eq:97}
\Big\|\frac{\phi-\phi(x_\nu)}{U+i\lambda}\Big\|_2^2 \leq C \,  \| \phi^\prime\|_2^2\,.
\end{equation}
Substituting \eqref{eq:97}  into \eqref{eq:94} readily 
  yields 
   \begin{equation}
\label{eq:98}
       |\phi(x_\nu)|^2 \leq  C \left( \Big|\Big\langle\phi,\frac{v}{U+i\lambda}\Big\rangle\Big| +|\mu|\,\|\phi^\prime\|_2^2\right) \,.
  \end{equation}
Using the positivity  of $-U^{\prime\prime}$ on $[0,1]$ and \eqref{eq:90} for the second inequality we
  then obtain that
\begin{equation}
\label{eq:99}
  \Big\|\frac{\phi}{U+i\lambda}\Big\|_2^2\leq  C \int_0^1 \frac{ (-U^{\prime\prime})}{(U-\nu)^2+\mu^2}|
  \phi|^2  dx  \leq\frac{\widehat C}{|\mu|} \Big|\Big\langle\phi,\frac{v}{U-\nu+i\mu}\Big\rangle\Big| \,.  
\end{equation}
Since 
\begin{displaymath}
   \|\phi^\prime\|_2^2 + \alpha^2\|\phi\|_2^2=  \Re\Big\langle\phi,\frac{\A_{\lambda,\alpha}\phi}{U+i\lambda}\Big\rangle 
 -  \Re\Big\langle U^{\prime\prime}\phi,\frac{\phi}{U+i\lambda}\Big\rangle  \,,
\end{displaymath}
we may conclude from \eqref{eq:99}, Poincar\'e's inequality, and
\eqref{eq:47} that
\begin{multline}\label{eq:100}
  \|\phi^\prime\|_2^2 + \alpha^2\|\phi\|_2^2\leq    \Big|\Big\langle\phi,\frac{v}{U+i\lambda}\Big\rangle\Big|+ 
  C\Big\|\frac{\phi}{U+i\lambda}\Big\|_2\|\phi\|_2\\ \leq  \Big|\Big\langle\phi,\frac{v}{U+i\lambda}\Big\rangle\Big|+\frac{\widehat
    C}{|\mu|^{1/2}}\Big|\Big\langle\phi,\frac{v}{U+i\lambda}\Big\rangle\Big|^{1/2}\|\phi^\prime\|_2 \,,  
\end{multline}
from which we conclude, given that $|\mu|\leq1$, 
\begin{equation}
\label{eq:101}
  \|\phi^\prime\|_2^2 \leq \frac{C}{|\mu|}\Big|\Big\langle\phi,\frac{v}{U+i\lambda}\Big\rangle\Big| \,.
\end{equation}
Substituting \eqref{eq:101}  into \eqref{eq:98} yields \eqref{eq:89}. \\

{\em Step 2: We prove that for any $A
  >0$, there exists $C$ and $\mu_A$ such that, for $\alpha^2 \leq A$ and $\lambda$ such that  $|\mu|\leq \mu_A$ and $\nu \in (0,\nu_1)$ 
\begin{equation}
  \label{eq:102}
\|\phi\|_{H^1(x_\nu,1)}\leq  C
\Big[\nu^{-1/2}\Big|\Big\langle\phi,\frac{v}{U+i\lambda}\Big\rangle\Big|^{1/2} +N(v,\lambda)\Big] \,. 
\end{equation}
holds for any pair $(\phi,v)\in D(\A_{\lambda,\alpha})\times W^{1,p}(0,1)$ satisfying
\eqref{eq:47}. \\ }

 Let
$\chi\in C^\infty(\R,[0,1])$ be given by 
\begin{equation}
\label{eq:103}
  \chi(x)=
  \begin{cases}
    1 & x<1/2 \\
    0 & x>3/4 \,.
  \end{cases}
\end{equation}
With $d=1-x_\nu$, let $\chi_d(x) = \chi((x-x_\nu)/d)$ and set
\begin{equation}
\label{eq:104}
  \phi=\varphi + \phi(x_\nu)\chi_d \,.
\end{equation}
Note that by the choice of $d$, $\varphi$ satisfies also the boundary
conditions at $x\in\{0, 1\}\,$.\\
It can be  easily verified that
\begin{displaymath}
  \A_{\lambda,\alpha}\varphi =v + \phi(x_\nu)\big((U+i\lambda)(\chi_d^{\prime\prime}-\alpha^2\chi_d) -U^{\prime\prime}\chi_d\big)
  \,.
\end{displaymath}
 By \eqref{eq:104}  we have that 
 \begin{equation}\label{defneww}
  w:=(U-\nu)^{-1}\varphi\in H^2(0,1)\,,
  \end{equation}
and hence we can rewrite the above equality (using \eqref{eq:47}
twice)  in the form  
\begin{multline}
\label{eq:105}
  -\Big((U-\nu)^2\Big(\frac{\varphi}{U-\nu}\Big)^\prime\Big)^\prime+\alpha^2(U-\nu)\varphi 
\\ \quad = v  +  \phi(x_\nu)\big((U-\nu)(\chi_d^{\prime\prime}-\alpha^2\chi_d) -U^{\prime\prime}\chi_d\big) +
  i\mu(\phi^{\prime\prime}-\alpha^2\phi) \\ \quad = \frac{(U-\nu)v}{U+i\lambda}+
  \phi(x_\nu)\big((U-\nu)(\chi_d^{\prime\prime}-\alpha^2\chi_d) -U^{\prime\prime}\chi_d\big)
  +i\mu\frac{U^{\prime\prime}\phi}{U+i\lambda} \,.
\end{multline}
Taking the inner product with $w$ and integrating by
parts,  exploiting the fact that $\varphi(x_\nu)=0$, then yields
\begin{multline} 
\label{eq:106}
  \|(U-\nu)w^\prime\|_{L^2(x_\nu,1)}^2 + \alpha^2\|\varphi\|_{L^2(x_\nu,1)}^2 =\Big\langle
  \varphi, \frac{v}{U+i\lambda}\Big\rangle _{L^2(x_\nu,1)}-\langle
  w,\phi(x_\nu)U^{\prime\prime}\chi_d\rangle_{L^2(x_\nu,1)}\\
  +  \phi(x_\nu)\langle\varphi,\chi_d^{\prime\prime}-\alpha^2\chi_d\rangle_{L^2(x_\nu,1)}  +i\mu\Big\langle w,\frac{U^{\prime\prime}\phi}{U+i\lambda}\Big\rangle_{L^2(x_\nu,1)}\,.
\end{multline}

We now estimate the four terms appearing in the right-hand-side of
\eqref{eq:106}, using precisely the same procedure as in the proof of
\cite[Proposition 4.13]{almog2019stability}.  For the first term on
the right-hand-side of \eqref{eq:106} we obtain with the aid of \eqref{eq:104}
\begin{multline}\label{eq:107}
  \Big|\Big\langle \varphi, \frac{v}{U+i\lambda}\Big\rangle_{L^2(x_\nu,1)}\Big|\leq  \Big|\Big\langle
  \phi-\phi(x_\nu), \frac{v}{U+i\lambda}\Big\rangle_{L^2(x_\nu,1)}\Big| \\ + |\phi(x_\nu)|\Big|\Big\langle
  1-\chi_d, \frac{v}{U+i\lambda}\Big\rangle_{L^2(x_\nu,1)}\Big| 
\end{multline}
Since the integration is carried over $(x_\nu,1)$ we can estimate the
first term on the right-hand-side  of \eqref{eq:107} by using Hardy's
inequality \eqref{eq:18} and \eqref{eq:95}
\begin{multline}
   \Big|\Big\langle
  \phi-\phi(x_\nu), \frac{v}{U+i\lambda}\Big\rangle_{L^2(x_\nu,1)}\Big| \leq
  \Big\|\frac{\phi-\phi(x_\nu)}{x-x_\nu}\Big\|_{L^2(x_\nu,1)}
  \Big\|\frac{ (x-x_\nu) v }{U+i\lambda}\Big\|_{L^2(x_\nu,1)} \\ \leq C\,
  \|\phi^\prime\|_{L^2(x_\nu,1)}\|v\|_2 \,.
\end{multline}
For the second term  on
the right-hand-side of \eqref{eq:107}  we  first note that since 
\begin{displaymath}
\phi(x_\nu)=-\int_{x_\nu}^1\phi^\prime(x) \,dx
\end{displaymath}
  we may conclude that
  \begin{equation}
\label{eq:108}
   |\phi(x_\nu)|\leq
  d^{1/2}\|\phi^\prime\|_2\,.
  \end{equation} 
By the definition of $\chi_d$,
\begin{displaymath}
  \Big\|\frac{1-\chi_d}{U+i\lambda}\Big\|_\infty\leq \frac{C}{d} \,,
\end{displaymath}
and hence
\begin{displaymath}
   \Big\|\frac{1-\chi_d}{U+i\lambda}\Big\|_{L^2(x_\nu,1)}=
   \Big\|\frac{1-\chi_d}{U+i\lambda}\Big\|_2\leq\frac{C}{d^{1/2}}\,.
\end{displaymath}
Hence, by the above and \eqref{eq:108},
\begin{displaymath}
  |\phi(x_\nu)|\Big|\Big\langle
  1-\chi_d, \frac{v}{U+i\lambda}\Big\rangle_{L^2(x_\nu,1)}\Big| \leq
  d^{1/2}\|\phi^\prime\|_2\Big\|\frac{1-\chi_d}{U+i\lambda}\Big\|_2\|v\|_2\leq C\|\phi^\prime\|_2\|v\|_2\,.
\end{displaymath}
 Hence,
\begin{equation}\label{eq:2.101n}
  \Big|\Big\langle \varphi, \frac{v}{U+i\lambda}\Big\rangle_{L^2(x_\nu,1)} \Big|\leq C\, \|\phi^\prime\|_2\|v\|_2
\end{equation}
In addition, we may write, observing that $\varphi(1)=0$, 
 \begin{displaymath} 
    \Big|\Big\langle \varphi,
    \frac{v}{U+i\lambda}\Big\rangle_{L^2(x_\nu,1)}\Big|\leq\|\varphi^\prime\|_{L^2(x_\nu,1)}\,
    \Big\|(1-x)^{1/2}\frac{v}{U+i\lambda}\Big\|_1 
 \end{displaymath}
Using \eqref{eq:108},  yields
\begin{displaymath}
  \|\varphi^\prime\|_{L^2(x_\nu,1)}\leq \|\phi^\prime\|_{L^2(x_\nu,1)} + |\phi(x_\nu)|\,\|\chi_d^\prime\|_2 \leq C\|\phi^\prime\|_2\,.
\end{displaymath}
which leads to 
\begin{equation}\label{eq:2.102n} 
  \Big|\Big\langle \varphi, \frac{v}{U+i\lambda}\Big\rangle_{L^2(x_\nu,1)} \Big|\leq C
  \|\phi^\prime\|_{L^2(x_\nu,1)}\,  \Big\|(1-x)^{1/2}\frac{v}{U+i\lambda}\Big\|_1
\end{equation}
Combining \eqref{eq:2.101n} and \eqref{eq:2.102n}  yields the
existence of $C>0$ such that 
\begin{equation}
\label{eq:109}
\Big|\Big\langle \varphi, \frac{v}{U+i\lambda}\Big\rangle_{L^2(x_\nu,1)}\Big|\leq C\, \|\phi^\prime\|_{L^2(x_\nu,1)}N(v,\lambda)\,.
\end{equation}

To estimate the second term $\langle
  w,\phi(x_\nu)U^{\prime\prime}\chi_d\rangle_{L^2(x_\nu,1)}$  on the right-hand-side   of \eqref{eq:106}, we note that by
Hardy's inequality \eqref{eq:18} and \eqref{eq:108},  we have  
\begin{multline} 
\label{eq:110}
  \|w\|_{L^2(x_\nu,1)}\leq C \|\varphi^\prime\|_{L^2(x_\nu,1)}\\ \leq \widehat C \Big(\|\phi^\prime\|_{L^2(x_\nu,1)}+
  \frac{1}{d^{1/2}}|\phi(x_\nu)|\Big)\leq \widetilde C\, \|\phi^\prime\|_{L^2(x_\nu,1)} \,.
\end{multline}
From \eqref{eq:110} we then get
\begin{equation}
\label{eq:111}
| \langle w,\phi(x_\nu)U^{\prime\prime}\chi_d\rangle_{L^2(x_\nu,1)}| \leq C\,|\phi(x_\nu)|  \|\phi^\prime\|_{L^2(x_\nu,1)}\,.
\end{equation}

Next, we write for the third term
($\phi(x_\nu)\langle\varphi,\chi_d^{\prime\prime}-\alpha^2\chi_d\rangle_{L^2(x_\nu,1)})$ on the
right-hand-side  of \eqref{eq:106}, using integration by parts 
  (note that $\chi^\prime_d (x_\nu)=0=\chi^\prime_d(1)$) and the fact that $\alpha^2\leq A$
\begin{displaymath}
\begin{array}{ll}
  |\langle\varphi,\chi_d^{\prime\prime}-\alpha^2\chi_d\rangle_{L^2(x_\nu,1)}|&  \leq\|\varphi^\prime\|_{L^2(x_\nu,1)}\|\chi_d^\prime\|_2 + C_A\,  \|\varphi\|_{L^2(x_\nu,1)}\\ &  \leq \widehat C_A \left(
  \frac{1}{d^{1/2}}\|\varphi^\prime\|_{L^2(x_\nu,1)} +\|\varphi\|_{L^2(x_\nu,1)} \right)\,.
  \end{array}
\end{displaymath}
(For convenience we drop the notation referring to the dependence on
$A$ in the sequel.)  Consequently, by \eqref{eq:104}, 
\begin{multline*}
  |\phi(x_\nu)| \,|\langle\varphi,\chi_d^{\prime\prime}-\alpha^2\chi_d\rangle_{L^2(x_\nu,1)}|\leq C |\phi(x_\nu)| \Big(  \frac{1}{d^{1/2}}\Big(\|\phi^\prime\|_{L^2(x_\nu,1)}+
  \frac{1}{d^{1/2}}|\phi(x_\nu)|\Big) \\+ (\|\phi\|_{L^2(x_\nu,1)}+d^{1/2}
  |\phi(x_\nu)|)\Big)\,. 
\end{multline*}
Hence, using Poincar\'e's inequality and \eqref{eq:108}, yields 
\begin{equation} 
\label{eq:112}
  |\phi(x_\nu)| \,| \langle\varphi,\chi_d^{\prime\prime}-\alpha^2\chi_d\rangle_{L^2(x_\nu,1)} |\leq  C  \frac{ |\phi(x_\nu)|}{d^{1/2}}\|\phi^\prime\|_{L^2(x_\nu,1)}\,.
\end{equation}

To estimate the last term ($i\mu\Big\langle
w,\frac{U^{\prime\prime}\phi}{U+i\lambda}\Big\rangle_{(x_\nu,1)}$) on the right-hand-side 
of \eqref{eq:106}, we first write
\begin{displaymath}
    \Big|\Big\langle w,U^{\prime\prime}\frac{\phi}{U+i\lambda}\Big\rangle_{L^2(x_\nu,1)}\Big|\leq\Big|\Big\langle
    w,U^{\prime\prime}\frac{\phi-\phi(x_\nu)}{U+i\lambda}\Big\rangle_{L^2(x_\nu,1)}\Big| +\Big|\Big\langle
    w,U^{\prime\prime}\frac{\phi(x_\nu)}{U+i\lambda}\Big\rangle_{L^2(x_\nu,1)}\Big| \,.
  \end{displaymath}
We then use \eqref{eq:110}, \eqref{eq:108} and
    \cite[Eq. (4.38)]{almog2019stability}, which reads, for $\nu \in [0,\nu_1]$, with $\nu_1 < U(0)$, 
    \begin{equation} 
\label{eq:113}
         \Big\|\frac{1}{U+i\lambda}\Big\|_2 \leq \check C_{\nu_1}  |\mu|^{-1/2}\,,
    \end{equation}
    to obtain that 
\begin{equation}
\label{eq:114}
  \Big|\Big\langle w,U^{\prime\prime}\frac{\phi(x_\nu)}{U+i\lambda}\Big\rangle_{L^2(x_\nu,1)}\Big|\leq
  C\, \|w\|_{L^2(x_\nu,1)}|\phi(x_\nu)|\Big\|\frac{1}{U+i\lambda}\Big\|_2 \leq \hat C\, |\mu|^{-1/2}d^{1/2}\|\phi^\prime\|_{L^2(x_\nu,1)}^2\,.
\end{equation}
Furthermore, we have, by \eqref{eq:97} and \eqref{eq:110}, 
\begin{displaymath}
  \Big|\Big\langle
    w,U^{\prime\prime}\, \frac{\phi-\phi(x_\nu)}{U+i\lambda}\Big\rangle_{L^2(x_\nu,1)}\Big| \leq C\, \|\phi^\prime\|_{L^2(x_\nu,1)}^2 \,.
\end{displaymath}
Substituting the above and \eqref{eq:114} together with \eqref{eq:109}, \eqref{eq:111}, and
  \eqref{eq:112} into \eqref{eq:106} yields that there exists $C>0$
  such that 
\begin{multline}
  \label{eq:115} 
  \|(U-\nu)w^\prime\|_{L^2(x_\nu,1)}^2 + \alpha^2\|\varphi\|_{L^2(x_\nu,1)}^2 \\ \leq C \,
  \Big([|\mu|^{1/2}d^{1/2}+|\mu|] \|\phi^\prime\|_{L^2(x_\nu,1)}^2 + 
 \Big[ \frac{|\phi(x_\nu)|}{d^{1/2}}+  N(v,\lambda) \Big] \|\phi^\prime\|_{L^2(x_\nu,1)}\Big)\,.
\end{multline}
As $|U(x)-\nu|\geq \frac 1 C |x-x_\nu|$ for all $x\in(x_\nu,1)$, we can
apply Hardy's inequality \eqref{eq:17} to $(U-\nu)w^\prime$ on $(x_\nu,1)$
to obtain
\begin{multline}
\label{eq:116} 
  \|w\|_{L^2(x_\nu,1)}^2 \leq   \widehat C\, \|(U-\nu)w^\prime\|_{L^2(x_\nu,1)}^2\\  \leq
 \widetilde C \,
  \Big(\big [|\mu|^{1/2}d^{1/2}+|\mu|\big ] \|\phi^\prime\|_{L^2(x_\nu,1)}^2 + 
 \big[ \frac{|\phi(x_\nu)|}{d^{1/2}}+  N(v,\lambda) \big] \|\phi^\prime\|_{L^2(x_\nu,1)}\Big)\,.
\end{multline}
 Continuing as in step 2 of the proof of \cite[Proposition
4.14]{almog2019stability} we write, using the definition of $w$ and
$\varphi$, 
\begin{equation}
\label{eq:117}
  \|\phi^\prime\|_{L^2(x_\nu,1)}\leq \|(U-\nu)w^\prime\|_{L^2(x_\nu,1)} + \|U^\prime w\|_{L^2(x_\nu,1)}+ C \,  d^{-1/2}|\phi(x_\nu)|\,,
\end{equation}
from which  we conclude with the aid of \eqref{eq:115} 
and \eqref{eq:116} that, for sufficiently small $\mu_A$,
\begin{equation}
\label{eq:118}
  \|\phi^\prime\|_{L^2(x_\nu,1)}\leq  C\left(  \frac{|\phi(x_\nu)|}{d^{1/2}}+  N(v,\lambda)\right) \,.
\end{equation}
 From \eqref{eq:118}  we can conclude \eqref{eq:102} with the aid of Poincar\'e's
inequality, the fact that $d \geq \frac 1C \nu$,  and \eqref{eq:89}.\\

{\em Step 3:  We prove that for any $A
  >0$, there exist $C$ and $\mu_A$ such that, for $\alpha^2 \leq A$, $|\mu|\leq
  \mu_A$,  and $\nu \in [0,\nu_1)$
\begin{equation}
\label{eq:119}
  |c_\parallel^\nu| + \|\phi\|_{H^1(0,x_\nu)}\leq  C\Big(\nu^{-1} \Big|\Big\langle\phi,\frac{v}{U+i\lambda}\Big\rangle\Big|^{1/2} + \nu^{-1/2}N(v,\lambda)\Big) \,,
\end{equation}
holds for any pair $(\phi,v)\in D(\A_{\lambda,\alpha})\times W^{1,p}(0,1)$ satisfying
\eqref{eq:47}.}\\ 

Here we need to obtain an estimate for $\|w\|_{L^2(0,x_\nu)}$, where we
recall from \eqref{defneww} that $ w:=(U-\nu)^{-1}\varphi$.

To this end we need an estimate for $w(\hat x_0)$ for some $\hat
x_0\in(x_\nu/2,x_\nu)$, to be determined at later stage.  Clearly, there
exists $ \hat x_1\in((1+x_\nu)/2,1)$ such that
\begin{displaymath}
 |\phi^\prime(\hat x_1)| \leq \frac{ \sqrt{2} }{d^{1/2}}\|\phi^\prime\|_{L^2(x_\nu,1)}\,.
\end{displaymath}
and 
\begin{displaymath}
   |\phi(\hat x_1)| \leq d^{1/2}\|\phi^\prime\|_{L^2(x_\nu,1)}\,.
\end{displaymath}
Furthermore, it holds for all $x\in(x_\nu/2,x_\nu)$ that
\begin{displaymath} 
   |\phi^\prime(x)| \leq   |\phi^\prime(\hat x_1)|  +
   \Big|\int_x^{\hat x_1}\phi^{\prime\prime}(t)\,dt\Big| \,.
\end{displaymath}
Consequently,  as
\begin{displaymath}
  |w(\hat x_0)|=\Big|\frac{\phi(\hat x_0)-\phi(x_\nu)}{U(\hat x_0)-\nu}\Big|\leq
  \frac{1}{|U(\hat x_0)-\nu|}\int_{\hat x_0}^{x_\nu} |\phi^\prime(x)|\,dx
\end{displaymath}
we may conclude that
\begin{displaymath}
  |w(\hat x_0)|\leq \frac{1}{|U(\hat x_0)-\nu|}\int_{\hat x_0}^{x_\nu} \Big[|\phi^\prime(\hat x_1)|  +
   \Big|\int_x^{\hat x_1}\phi^{\prime\prime}(t)\,dt\Big|\Big]\,dx\,.
\end{displaymath}
With the aid of  \eqref{eq:47} we then have
\begin{equation}
\label{eq:120}
 |w(\hat x_0)|\leq \frac{1}{|U(\hat x_0)-\nu|}\int_{\hat x_0}^{x_\nu} \Big(|\phi^\prime(\hat x_1)|  +
   \Big|\int_x^{\hat x_1}\Big[\frac{U^{\prime\prime}\phi-v}{U+i\lambda}+\alpha^2\phi\Big]\,dt\Big|\Big)\,dx \,.
\end{equation}
We now write
\begin{equation}\label{eq:121}
  \int_x^{\hat x_1}\frac{U^{\prime\prime}\phi}{U+i\lambda}\,dt
  = \phi(x_\nu)\int_x^{\hat x_1}\frac{U^{\prime\prime}}{U+i\lambda}\,dt +
  \int_x^{\hat x_1}\frac{U^{\prime\prime}[\phi-\phi(x_\nu)]}{U+i\lambda}\,dt\,.
\end{equation}
To facilitate the estimate of the integral appearing in the first term
in the right-hand-side of \eqref{eq:121}   we use an integration by parts to obtain 
\begin{displaymath}
  \int_x^{\hat x_1}\frac{U^{\prime\prime}}{U+i\lambda}\,dt 
= \left( \frac{U^{\prime\prime}}{U^\prime} \log (U +i\lambda)\right) \Big|_{\hat x_1}^{x} 
- \int_x^{\hat x_1} \left( \frac{U^{\prime\prime}}{U^\prime}\right)^\prime \log (U+i\lambda)\,dt \,.
\end{displaymath}

As $ 0 <\nu < \nu_1$, $U^{\prime\prime}/U^\prime$ and
$\left(\frac{U^{\prime\prime}}{U^\prime}\right)^\prime$ are uniformly
bounded in $(x_\nu/2,1)$ and in view of the inequality
\begin{displaymath}
   |\log (U+i\lambda)|\leq\log|x-x_v|^{-1}+C\,,
\end{displaymath}
  we observe that  the  $L^1$-norm of
  $\log (U(x) +i\lambda)$ is bounded and that  
  \begin{equation} \label{estlog}
   |\log(U+i\lambda)(\hat x_1)|  \leq  C ( |\log (d^2+\mu^2)| +1)\,.
   \end{equation}  
Hence, we have 
\begin{displaymath}
   \Big|\phi(x_\nu)\int_x^{\hat x_1}\frac{U^{\prime\prime}(t)}{U (t) +i \lambda}\,dt\Big| \leq C
   |\phi(x_\nu)|[1 +   |\log (d^2+\mu^2)|    + |\log(U+i\lambda)(\hat x)|  ] \,.
\end{displaymath}

For the second term in the right-hand-side of \eqref{eq:121},   we have by
\eqref{eq:97} 
\begin{displaymath}
\begin{array}{ll}
  \Big|\int_{x_\nu}^{\hat
    x_1}\frac{U^{\prime\prime}(t) [\phi(t)-\phi(x_\nu)]}{U(t)+i\lambda}\,dt\Big|& \leq C|\hat
  x_1-x_\nu|^{1/2}\Big[\int_{x_\nu}^{\hat
    x_1}\Big|\frac{\phi(t)-\phi(x_\nu)}{U(t)+i\lambda}\Big|^2\,dt\Big]^{1/2}\\ & 
\leq \widetilde C \, d^{1/2}\|\phi^\prime\|_{L^2(x_\nu,\hat x_1)}\,.  
\end{array}
\end{displaymath}
In a similar manner we obtain that 
\begin{displaymath}
  \Big|\int_x^{x_\nu}\frac{U^{\prime\prime}[\phi(t)-\phi(x_\nu)]}{U(t)+i\lambda}\,dt\Big|\leq
  C(x_\nu-x)^{1/2}\|\phi^\prime\|_{L^2(x,x_\nu)} \,.
\end{displaymath}
Consequently,
\begin{displaymath}
  \Big|\int_x^{\hat x_1}\frac{U^{\prime\prime}[\phi(t)-\phi(x_\nu)]}{U(t)+i\lambda}\,dt\Big|
  \leq C\, \big(\|\phi^\prime\|_{L^2(x_\nu,\hat x_1)}+(x_\nu-x)^{1/2}\|\phi^\prime\|_{L^2(x,x_\nu)}\big)\,.
\end{displaymath}
Hence, 
\begin{multline}
  \label{eq:122}
  \Big|\int_x^{\hat x_1}\frac{U^{\prime\prime}\phi}{U+i\lambda}\,dt\Big| \leq
  C\, \Big(\|\phi^\prime\|_{L^2(x_\nu,\hat
    x_1)}+(x_\nu-x)^{1/2}\|\phi^\prime\|_{L^2(x,x_\nu)} \\ +
  |\phi(x_\nu)|[1+   |\log (d^2+\mu^2)|  +|\log(U+i\lambda)(x)|]\Big) \,. 
\end{multline}
Next, we write
\begin{displaymath}
   \Big|\int_x^{\hat x_1}\phi(t)\,dt\Big|\leq C\, d^{1/2}\, \|\phi\|_{L^2(x_\nu,1)}+
  \Big|\int_x^{x_\nu}\phi(t)\,dt\Big| \,.
\end{displaymath}
Then, with the aid of  Poincar\'e's inequality we obtain
\begin{equation*} 
  \Big|\int_x^{\hat x_1}\phi(t)\,dt\Big|  \leq  C\, d\, \|\phi^\prime\|_{L^2(x_\nu,1)}+
  |\phi(x_\nu)|(x_\nu-x)+C(x_\nu-x)^{3/2}\|\phi^\prime\|_{L^2(x,x_\nu)}\,. 
\end{equation*}
Substituting the above, together with \eqref{eq:122}, into
\eqref{eq:120} yields
\begin{multline*}
    |w(\hat x_0)|\leq \frac{C(1+\alpha^2)}{|U(\hat x_0)-\nu|}\int_{\hat x_0}^{x_\nu} \Big[d^{-1/2}
    \|\phi^\prime\|_{L^2(x_\nu,1)}+\\ + (x_\nu-x)^{1/2}\|\phi^\prime\|_{L^2(x,x_\nu)} +
   |\phi(x_\nu)|[1+|\log(U+i\lambda)(x)| +  |\log (d^2+\mu^2) | ] \Big]\,dx\,.
\end{multline*}
We now write,  using \eqref{estlog}, 
\begin{displaymath}
\begin{array}{ll}
  \frac{1}{|U(\hat x_0)-\nu|}\int_{\hat x_0}^{x_\nu}|\log(U+i\lambda)(x)|\,dx
 & \leq\frac{C}{x_\nu-\hat{x}_0}\int_{\hat
    x_0}^{x_\nu}[1+|\log(x_\nu-x)|]\,dx\\
    &  \leq \widehat C[1+\log \left( |\hat x_0-x_\nu|^{-1}\right) ] \,.
    \end{array}
\end{displaymath}
Consequently, since $|(\hat x_0 -x_\nu)(U(\hat x)-\nu)^{-1}|$ is uniformly
bounded for $\hat x_0\in(x_\nu,x_\nu/2)$ and $\alpha^2 \leq A$, 
\begin{multline}
\label{eq:123}
 |w(\hat x_0)|\leq C[|\phi(x_\nu)|(1+\log|\hat x_0-x_\nu|^{-1} + |\log (d^2 +\mu^2) | )\\+(x_\nu-\hat
 x_0)^{1/2}\|\phi^\prime\|_{L^2(\hat x_0,x_\nu)} +
 d^{-1/2}\|\phi^\prime\|_{L^2(x_\nu,1)}]\,. 
\end{multline}
 We can now apply Hardy's inequality \eqref{eq:17} to $w-w(\hat
  x_0)$ on the interval $(\hat x_0,x_\nu)$ to obtain 
\begin{displaymath}
  \|w-w(\hat x_0)\|_{L^2(\hat x_0,x_\nu)}^2 \leq
  4 \, \|([x-x_\nu][w-w(\hat x_0)])^\prime\|_{L^2(\hat x_0,x_\nu)}^2\leq C(\nu_1)\|(U-\nu)w^\prime\|_{L^2(\hat x_0,x_\nu)}^2
\end{displaymath}
By \eqref{eq:123} and \eqref{eq:89} we then have
\begin{displaymath}
\begin{array}{ll}
   \|w\|_{L^2(\hat x_0,x_\nu)}&  \leq
   C\, \big(\|(U-\nu)w^\prime\|_{L^2(\hat x_0,x_\nu)}+d^{-1/2}\|\phi^\prime\|_{L^2(x_\nu,1)} \\ &\quad \quad  +[\|\phi^\prime\|_2^{1/2} N(v,\lambda)^\frac 12  +
|\mu|^{1/2}\|\phi^\prime\|_2](1+\log \left( |\hat x_0-x_\nu|^{-1}\right) +  |\log (d^2+\mu^2)| )\\
&\quad \quad  \quad \quad +
(x_\nu-\hat x_0)^{1/2}\|\phi^\prime\|_{L^2(\hat x_0,x_\nu)}\big)\,. 
\end{array}
\end{displaymath}
Note that $C$ is independent of $\hat x_0 \in (x_\nu/2, x_\nu)$.\\
On the other hand by Poincar\'e's inequality we have,
\begin{displaymath}
\|w-w(\hat x_0)\|_{L^2(0,\hat x_0)}^2  \leq C  \|  \, w^\prime \|_{L^2(0,\hat x_0)}^2 \,.
\end{displaymath}
 Observing
  that $\frac 12 x_{\nu_1}  \leq \hat x_0 \leq x_\nu \leq 1$ we obtain for all $x\in (0,\hat x_0)$  
\begin{displaymath}
U(x) -\nu \geq U(\hat x_0) -U(x_\nu) \geq  |U^\prime(\frac 12 x_{\nu_1})|\, |\hat x_0-x_\nu|\,.
\end{displaymath}
Consequently,
\begin{equation} 
\label{eq:2.116new}
   \|w-w(\hat x_0)\|_{L^2(0,\hat x_0)}^2 \leq C  (x_\nu-\hat{x}_0)^{-2} \|(U-\nu)w^\prime\|_{L^2(0,\hat x_0)}^2 \,,
\end{equation}
and hence, as $\hat x_0 \leq  x_\nu$, 
\begin{multline*}
    \|w\|_{L^2(0,x_\nu)} \leq
   C \big(
   (x_\nu-\hat{x}_0)^{-1}  \|(U-\nu)w^\prime\|_{L^2(0,x_\nu)}+d^{-1/2}\|\phi^\prime\|_{L^2(x_\nu,1)}+
    \\
  \quad +  [\|\phi^\prime\|_2^{1/2}N(v,\lambda)^\frac 12 +
|\mu|^{1/2}\|\phi^\prime\|_2](1+\log|\hat x_0-x_\nu|^{-1}+   |\log (d^2+\mu^2)|)\\ +
   (x_\nu-\hat x_0)^{1/2}\|\phi^\prime\|_{L^2(\hat x_0,x_\nu)}\big)\,. 
\end{multline*}
Continuing as in Step 2 of the proof of \cite[Proposition
4.14]{almog2019stability} we establish that
\begin{multline}
\label{eq:124}
    \|\phi^\prime\|_{L^2(0,x_\nu)} \leq C\big(
     (x_\nu-\hat{x}_0)^{-1}  \|(U-\nu)w^\prime\|_{L^2(0,x_\nu)} +d^{-1/2}\|\phi^\prime\|_{L^2(x_\nu,1)}+
    \\ 
 \quad  +  [\|\phi^\prime\|_2^{1/2}N(v,\lambda)^\frac 12 +| \mu|^{1/2}\|\phi^\prime\|_2](1+ \log\left( |\hat x_0-x_\nu|^{-1}\right)\\+   |\log (d^2+\mu^2)|)+
   (x_\nu-\hat x_0)^{1/2}\|\phi^\prime\|_{L^2(\hat x_0,x_\nu)}\big)\,. 
\end{multline}

Taking the inner product in $L^2(0,x_\nu)$ of \eqref{eq:105} with $w$
yields, as in \eqref{eq:106} (note that $\chi_d\equiv1$ in $(0,x_\nu)$)
\begin{multline} 
\label{eq:125}
  \|(U-\nu)w^\prime\|_{L^2(0,x_\nu)}^2 + \alpha^2\|\varphi\|_{L^2(0,x_\nu)}^2 =\Big\langle
  \varphi, \frac{v}{U+i\lambda}\Big\rangle_{L^2(0,x_\nu)} -\langle
  w,\phi(x_\nu)U^{\prime\prime} \rangle_{L^2(0,x_\nu)}\\ 
  -\alpha^2 \phi(x_\nu)\langle\varphi,1\rangle_{L^2(0,x_\nu)}   +i\mu\Big\langle w,\frac{U^{\prime\prime}\phi}{U+i\lambda}\Big\rangle_{L^2(0,x_\nu)} \,.
\end{multline}

 As in the proof of  (\ref{eq:109})
we obtain below  for the first term in the right-hand-side of \eqref{eq:125} 
\begin{equation}\label{eq:126} 
 \Big| \Big\langle
  \varphi, \frac{v}{U+i\lambda}\Big\rangle_{L^2(0,x_\nu)}\Big| \leq C\,   \|\phi^\prime\|_{L^2(0,x_\nu)}\,N(v,\lambda) \,.
\end{equation}
Indeed, as $\varphi^\prime\equiv\phi^\prime$ in $(0,x_\nu)$  and $\varphi(x_\nu)=0$, we get for $x\in (0,x_\nu)$
  \begin{displaymath}
    |\varphi(x)|\leq\|\varphi^\prime\|_{L^2(0,x_\nu)}(x_\nu-x)^{1/2}\leq
\|\phi^\prime\|_{L^2(0,x_\nu)}(1-x)^{1/2}\,,
  \end{displaymath}
from which we conclude that
\begin{equation}\label{eq:127}
  \Big| \Big\langle
  \varphi, \frac{v}{U+i\lambda}\Big\rangle_{L^2(0,x_\nu)}\Big| \leq
  \|\phi^\prime\|_{L^2(0,x_\nu)}\,\Big\|(1-x)^{1/2}\frac{v}{U+i\lambda}\Big\|_1 \,.
\end{equation}
In addition, we can write
\begin{displaymath}
  \Big| \Big\langle
  \varphi, \frac{v}{U+i\lambda}\Big\rangle_{L^2(0,x_\nu)}\Big| =  \Big| \Big\langle
  \phi-\phi(x_\nu), \frac{v}{U+i\lambda}\Big\rangle_{L^2(0,x_\nu)}\Big| = \Big| \Big\langle
\frac{  (\phi-\phi(x_\nu))}{(x-x_\nu)}, \frac{x-x_\nu}{U+i\lambda}v \Big\rangle_{L^2(0,x_\nu)}\Big|
\end{displaymath}
 and then use \eqref{hardyw1a} and \eqref{eq:95}  to obtain 
\begin{equation}\label{eq:128}
  \Big| \Big\langle
  \varphi, \frac{v}{U+i\lambda}\Big\rangle_{L^2(0,x_\nu)}\Big| \leq C\, \|\phi^\prime\|_2 \|v\|_2\,.
\end{equation}
Using the definition of $N(v,\lambda)$ in \eqref{eq:86} we can now conclude
\eqref{eq:126} from \eqref{eq:127} and \eqref{eq:128}.

By \eqref{eq:18} which reads in this case
\begin{equation}\label{eq:129}
\| w\|_{L^2(0,x_\nu)} \leq C \|
\phi^\prime\|_{L^2(0,x_\nu)} \,,
\end{equation}
we have for the second term on the right-hand-side of  \eqref{eq:125} that
\begin{displaymath}
  |\langle  w,\phi(x_\nu)U^{\prime\prime} \rangle_{L^2(0,x_\nu)}|\leq C\,|\phi(x_\nu)|  \|\phi^\prime\|_{L^2(0,x_\nu)}\,.
\end{displaymath}
For the third term we have, using the fact that $\alpha^2 \leq A\,$, 
\eqref{eq:104} and Poincar\'e's inequality,  
\begin{displaymath}
  |\alpha^2 \phi(x_\nu)\langle\varphi,1\rangle_{L^2(0,x_\nu)}|\leq
   C\, |\phi(x_\nu)|\,\big(\|\phi\|_{L^2(0,x_\nu)}+|\phi(x_\nu)|\big)\leq \widehat C\, |\phi(x_\nu)|\,\|\phi^\prime\|_{L^2(0,x_\nu)} \,.
\end{displaymath}
Finally, we obtain for the last term  of \eqref{eq:125}   using  \eqref{eq:95}, 
\eqref{hardyw1a} (as in the proof of \eqref{eq:97} but on the interval
$(0,x_\nu)$),  and \eqref{eq:129}
\begin{displaymath}
  \Big|\mu\Big\langle w,\frac{U^{\prime\prime}(\phi-\phi(x_\nu)) }{U+i\lambda}\Big\rangle_{L^2(0,x_\nu)}\Big|\leq
  C\, |\mu| \, \|\phi^\prime\|_{L^2(0,x_\nu)}^2 \,.
\end{displaymath} 
Hence, with the aid of \eqref{eq:113} we conclude, as in the
proof of \eqref{eq:114},
\begin{displaymath}
\begin{array}{ll}
  \Big|\mu\Big\langle w,\frac{U^{\prime\prime} \phi(x_\nu)}{U+i\lambda}\Big\rangle_{L^2(0,x_\nu)}\Big|& \leq
  C|\mu| \,  \| (U+i\lambda)^{-1} \|_{L^2(0,x_\nu)} \, |\phi(x_\nu)|\,   \|\phi^\prime\|_{L^2(0,x_\nu)} \\
  & \leq   \widehat  C \, |\mu|^\frac 12  \, |\phi(x_\nu)|  \|\phi^\prime\|_{L^2(0,x_\nu)}  
 \\
 & \leq \widehat  C\, \big( |\mu| \|\phi^\prime\|_{L^2(0,x_\nu)}^2 +  |\phi(x_\nu)|^2\big)\\
 & \leq  \widetilde  C\, \big( |\mu| \|\phi^\prime\|_{L^2(0,x_\nu)} +  |\phi(x_\nu)|\big)\|\phi^\prime\|_{L^2(0,x_\nu)}   \,.
  \end{array}
\end{displaymath}  
Combining the above starting from \eqref{eq:125} yields 
\begin{equation}
\label{eq:130}
\begin{array}{l} 
   \|(U-\nu)w^\prime\|_{L^2(0,x_\nu)}^2 + \alpha^2\|\varphi\|_{L^2(0,x_\nu)}^2\\
   \qquad \leq C \,
    \big( |\mu|\,\|\phi^\prime\|_{L^2(0,x_\nu)}^2+ 
 [ |\phi(x_\nu)|+ N(v,\lambda)]\|\phi^\prime\|_{L^2(0,x_\nu)}\big) \,.
 \end{array}
\end{equation}
Hence, by \eqref{eq:124} and \eqref{eq:89}, there exists
  $C>0$ such that for any $\hat x_0 \in (x_\nu/2,x_\nu)$ we have, with
  $\varepsilon= x_\nu-\hat x_0 $,   
\begin{displaymath} 
    \|\phi^\prime\|_2 \leq  C\, (N(v,\lambda) +[1+|\mu|^{1/2}+d^{-1/2}]\|\phi^\prime\|_{L^2(x_\nu,1)})
      + C [|\mu|^{1/2}(\varepsilon^{-1}+   |\log (d^2+\mu^2)|) + \varepsilon^{1/2}]\|\phi^\prime\|_2
\end{displaymath}
We can now choose $\hat x_0$ such that $\varepsilon = \inf \left( (\frac{1}{4C}
  )^2, x_{\nu_1} /4 \right)$. Then under the condition $|\mu|^{1/2}
(\varepsilon^{-1} + |\log (d^2+\mu^2)|) \leq \frac{1}{4C}$, which is valid for
$\mu_A$ which is small enough, we obtain
 \begin{equation}
\label{eq:131}
\|\phi^\prime\|_2\leq  3C\, \big(N(v,\lambda) +d^{-1/2}\|\phi^\prime\|_{L^2(x_\nu,1)}\big)\leq
\hat{C}\, \big(N(v,\lambda) +\nu^{-1/2}\|\phi^\prime\|_{L^2(x_\nu,1)}\big)  .
\end{equation} 
 Combining \eqref{eq:131}  with \eqref{eq:118}  yields, 
\begin{equation}
  \label{eq:132}
\|\phi^\prime\|_2\leq  C \, \Big(\frac{|\phi(x_\nu)|}{d}+ d^{-1/2}N(v,\lambda)\Big) \,,
\end{equation}
from which we can conclude a bound on $\phi$ in $H^1(0,x_\nu)$ as stated
in \eqref{eq:119}, upon use of Poincar\'e's inequality and
\eqref{eq:89}.\\
To complete the proof of \eqref{eq:119}, we need to estimate $c_\parallel^\nu$,
which is defined in \eqref{eq:88} by
\begin{equation*}
c_\parallel^\nu=\frac{\langle\phi-\phi(x_\nu),U-\nu\rangle_{L^2(0,x_\nu)}}{\|U-\nu\|_{L^2(0,x_\nu)}^2}\,.
\end{equation*}
We note  that  for $0 < \nu < \nu_1$ the denominator in the definition of $c_\parallel^\nu$  satisfies
\begin{equation}\label{eq:133}
\| (U-\nu)  \|_{L^2 (0,x_\nu)}^2 \geq  \| (U-\nu)   \|_{L^2 (x_{\nu_1/4},x_{\nu_1/2})}^2 \geq \frac{1}{C} \,.
\end{equation}
Hence,
\begin{equation}
\label{eq:134}
  |c_\parallel^\nu|\leq C \, 
  \Big\|\frac{\phi-\phi(x_\nu)}{U-\nu}\Big\|_{L^2(0,x_\nu)}
  \leq \widehat C \,\|\phi^\prime\|_{L^2(0,x_\nu)} \,,
\end{equation}
which together with \eqref{eq:132} yields
\eqref{eq:119}. \\
 
{\em Step 4:} {\it We prove that for any $A
  >0$, there exists $C$ and $\mu_A$ such that, for $\alpha^2 \leq A$,
  $|\mu|\leq \mu_A$ and $\nu \in (0,\nu_1)$, such that
\begin{equation}
\label{eq:135}
\|\phi-c_\parallel^\nu(U-\nu)\|_{1,2}\leq  C\nu^{-1/2}\Big(
   \Big|\Big\langle\phi,\frac{v}{U+i\lambda}\Big\rangle\Big|^{1/2} + N(v,\lambda)\Big)\,.
\end{equation}
holds for any pair $(\phi,v)\in D(\A_{\lambda,\alpha})\times W^{1,p}(0,1)$ satisfying
\eqref{eq:47}.}\\

Since \eqref{eq:12} remains valid if we replace $U$ by $U-\nu$, $1-x$ by
$x_\nu-x$ and $(0,1)$ by $(0,x_\nu)$ we may conclude  from \eqref{eq:12}
that there exists $C>0$ such that
\begin{equation}
\label{eq:136}
  \|(U-\nu)w^\prime\|_{L^2(0,x_\nu)}^2 \geq \frac 1 C \|\varphi_\perp^\nu\|_{L^2(0,x_\nu)}^2 \,,
\end{equation}
where, in the interval  $ (0,x_\nu)$, $\varphi_\perp^\nu$ is defined by
\begin{equation}\label{eq:137}
\varphi_\perp^\nu =\phi-\phi(x_\nu)-c_\parallel^\nu(U-\nu)  = \varphi - c_\parallel^\nu (U-\nu)
\end{equation} 
and $c_\parallel^\nu$ is defined in \eqref{eq:88}.  Note
that by construction
\begin{equation}\label{eq:138}
\langle\varphi_\perp^\nu,U-\nu\rangle_{(0,x_\nu)}=0\,.
\end{equation}
Furthermore, by \eqref{eq:138} and \eqref{eq:133}
\begin{equation}\label{eq:139}
\frac 1C |c_\parallel^\nu|^2 \leq |c_\parallel^\nu|^2 \| (U-\nu)\|_{L^2(0,x_\nu)}^2 \leq \| \varphi \|^2_{L^2(0,x_\nu)} \,.
\end{equation}
Substituting \eqref{eq:139} and \eqref{eq:136} into \eqref{eq:130}
(recall again that $\varphi^\prime\equiv\phi^\prime$ in $(0,x_\nu)$) yields, with the aid
of \eqref{eq:101} and \eqref{eq:137}, for a new constant $C$
\begin{equation}\label{eq:140}
\begin{array}{ll}
  \| \varphi_\perp^\nu  \|_{L^2(0,x_\nu)}^2+  \alpha^2|c_\parallel^\nu|^2  & \leq 
  C\Big( \Big|\Big\langle\phi,\frac{v}{U+i\lambda}\Big\rangle\Big| \\ & \quad \quad\quad  +
 \big[ |\phi(x_\nu)|  +N(v,\lambda)\big]\big(\| (\varphi^\nu_\perp)^\prime\|_{L^2(0,x_\nu)}+|c_\parallel^\nu|\big)\Big)\,.
 \end{array}
\end{equation}
Let $w_\perp^\nu:=(U-\nu)^{-1}\varphi_\perp^\nu$. 
As in \eqref{eq:26} we obtain that
\begin{equation}\label{eq:141}  
   \|(\varphi_\perp^\nu)^\prime\|_{L^2(0,x_\nu)}^2\leq C\, \Big( \|(U-\nu)(w_\perp^\nu)^\prime\|_{L^2(0,x_\nu)}^2+ \|\varphi_\perp^\nu\|_{L^2(0,x_\nu)}\|w_\perp^\nu\|_{L^2(0,x_\nu)}\Big) \,.
\end{equation}
By Hardy's inequality \eqref{eq:18}  applied to $w_\perp^\nu$,  and since $(w_\perp^\nu)^\prime=w^\prime$,  we
may conclude from \eqref{eq:141}, with 
the aid of \eqref{eq:136}, that
\begin{multline}
\label{eq:142}
  \|(\varphi_\perp^\nu)^\prime\|_{L^2(0,x_\nu)}\leq C(\|(U-\nu)w^\prime\|_{L^2(0,x_\nu)}+
  \|\varphi_\perp^\nu\|_{L^2(0,x_\nu)}) \\ \leq  \widehat C\, \|(U-\nu)w^\prime\|_{L^2(0,x_\nu)}\,.
\end{multline}
 We now rewrite \eqref{eq:130} with the aid of \eqref{eq:101} in
  the form
  \begin{displaymath}
    \|(U-\nu)w^\prime\|_{L^2(0,x_\nu)}^2 \leq C \,
    \Big(  \Big|\Big\langle\phi,\frac{v}{U+i\lambda}\Big\rangle\Big| + 
 [ |\phi(x_\nu)|+ N(v,\lambda)]\|\phi^\prime\|_{L^2(0,x_\nu)}\Big) \,.
  \end{displaymath}
Combining the above with \eqref{eq:142} and \eqref{eq:137} yields
\begin{equation}
\label{eq:143}
\begin{array}{ll}
  \|(\varphi_\perp^\nu)^\prime\|_{L^2(0,x_\nu)}^2 &  \leq 
  C\, \Big( \Big|\Big\langle\phi,\frac{v}{U+i\lambda}\Big\rangle\Big| \\ & +
 [ |\phi(x_\nu)|+  N(v,\lambda) ](\| (\varphi^\nu_\perp)^\prime \|_{L^2(0,x_\nu)}+|c_\parallel^\nu|)\Big)\,.
 \end{array}
\end{equation}
This yields by \eqref{eq:119} and \eqref{eq:89} that
\begin{equation}
  \label{eq:144}
 \|(\varphi_\perp^\nu)^\prime\|_{L^2(0,x_\nu)}^2\leq C \, \nu^{-1} \Big(
   \Big|\Big\langle\phi,\frac{v}{U+i\lambda}\Big\rangle\Big| + N(v,\lambda)^2\Big)\,.
\end{equation}
Clearly, by Poincar\'e's inequality
\begin{equation*}
 \|\phi-c_\parallel^\nu(U-\nu)\|_{L^2(0,x_\nu)}^2 \leq 2 (\|\varphi_\perp^\nu\|_{L^2(0,x_\nu)}^2
 +|\phi(x_\nu)|^2)\leq C (\|(\varphi_\perp^\nu)^\prime\|_{L^2(0,x_\nu)}^2+|\phi(x_\nu)|^2)\,,
\end{equation*}
which implies 
\begin{equation}\label{eq:145}
 \|\phi-c_\parallel^\nu(U-\nu)\|_{H^1(0,x_\nu)}^2 \leq C (\|(\varphi_\perp^\nu)^\prime\|_{L^2(0,x_\nu)}^2+|\phi(x_\nu)|^2)\,.
\end{equation}
Combining \eqref{eq:145} with \eqref{eq:119},  \eqref{eq:89}, and 
\eqref{eq:144} then yields
\begin{equation}
\label{eq:146}
  \|\phi-c_\parallel^\nu(U-\nu)\|_{H^1(0,x_\nu)} \leq C\nu^{-1/2}\Big(
   \Big|\Big\langle\phi,\frac{v}{U+i\lambda}\Big\rangle\Big|^{1/2} +N(v,\lambda)]\Big)\,.
\end{equation}
Next, we write with the aid of  \eqref{eq:102} and \eqref{eq:119}
(the bound on $|c^\nu_\parallel|$),  
\begin{equation}\label{eq:147}
  \|\phi-c_\parallel^\nu(U-\nu)\|_{H^1(x_\nu,1)} \leq
  \|\phi\|_{H^1(x_\nu,1)}+d^{1/2} |c_\parallel^\nu|  \leq
  C\Big(\nu^{-1/2}\Big|\Big\langle\phi,\frac{v}{U+i\lambda}\Big\rangle\Big|^{1/2}+
  N(v,\lambda)\Big) \,. 
\end{equation}
Combining \eqref{eq:147} with \eqref{eq:146} and \eqref{eq:89} yields, 
\begin{displaymath}
   \|\phi-c_\parallel^\nu(U-\nu)\|_{1,2} \leq C\nu^{-1/2}\Big(
   \Big|\Big\langle\phi,\frac{v}{U+i\lambda}\Big\rangle\Big|^{1/2} + N(v,\lambda)\Big)\,,
\end{displaymath}
verifying \eqref{eq:135}. \\

{\em Step 5: We prove that for any $0<\nu_1<U(0)$ and  $A >0$, there exists $\mu_A
>0$ and $C_A >0$ such that \eqref{eq:87} holds for $\alpha^2\leq A$ 
  and $|\mu|\leq\mu_A$.} \\

For $\nu \in(0,\nu_1) $, let $\zeta_\nu \in
C^\infty(\R_+,[0,1])$ satisfy
 \begin{equation}\label{eq:zetanu}
   \zeta_\nu(x)=
   \begin{cases}
     0 & x<\frac{x_\nu}{4} \\
     1  & x>\frac{x_\nu}{2} 
   \end{cases}
 \end{equation}
 and 
\begin{displaymath}
  |\zeta_\nu^\prime(x)|\leq C(\nu_1)\,,\, \forall \nu \in (0,\nu_1)\,,\, \forall x\in (0,1)  \,.
\end{displaymath}
  Let further 
  \begin{equation}
\label{eq:tildezetanu}
  \tilde{\zeta_\nu}=1-\zeta_\nu\,.
  \end{equation}
  We may now write, omitting the reference to $\nu$ for $\zeta_\nu$ and $\tilde \zeta_\nu$,
 \begin{equation}
\label{eq:148}
   \Big|\Big\langle\zeta\phi,\frac{v}{U+i\lambda}\Big\rangle\Big| \leq|\phi(x_\nu)|
   \Big|\Big\langle\zeta,\frac{v}{U+i\lambda}\Big\rangle\Big| +
   \Big|\Big\langle\zeta\big(\phi-\phi(x_\nu)\big),\frac{v}{U+i\lambda}\Big\rangle\Big| \,. 
 \end{equation}

For the first term on the right-hand-side  of \eqref{eq:148} we
begin by  writing that 
\begin{displaymath}
    \Big\langle\zeta,\frac{v}{U+i\lambda}\Big\rangle  = -   \Big\langle \Big( \frac{\zeta \bar
      v}{U^\prime}\Big)^\prime  , \log (U+i\lambda)\Big\rangle  + \frac{ \bar  v(1) \log
      (U(1)+i\lambda)}{U^\prime(1)}  \,,
 \end{displaymath}
 Then we observe that 
\begin{displaymath}
|\log (U(1)+i\lambda)|= |\log (U(1)- U(x_\nu)+i\mu)| \leq C (\log (d^{-1}+1)\,,
\end{displaymath}
and since $d \leq 1-x_{\nu_1}$ we obtain
\begin{displaymath}
|\log (U(1)+i\lambda)|= |\log (U(1)- U(x_\nu)+i\mu)| \leq \widehat C \,  (\log (d^{-1})\,.
\end{displaymath} 
We can then conclude that
\begin{displaymath}
  \Big|\Big\langle\zeta,\frac{v}{U+i\lambda}\Big\rangle\Big| \leq   \Big|\Big\langle\Big(\frac{\zeta\bar{v}}{U^\prime}\Big)^\prime,\log\,(U+i\lambda)\Big\rangle\Big| + C|v(1)|\log|d|^{-1}\Big]\,.
\end{displaymath}
In view of \eqref{eq:logp} we can use H\"older inequality
 and the fact that 
\begin{displaymath}
 \|v\|_p \leq  \|v\|_\infty\leq |v(1)|+\|v^\prime\|_p
\end{displaymath}
 to obtain, with the aid of \eqref{eq:logp}, that for any $p>1$, we
 have 
  \begin{displaymath}
  \Big|\Big\langle\Big(\frac{\zeta\bar{v}}{U^\prime}\Big)^\prime,\log\,(U+i\lambda)\Big\rangle\Big|   \leq    C\, ( |v(1)| + \|v^\prime\|_p)\,.
\end{displaymath}
Consequently, there exist $C>0$  such that
 \begin{equation}\label{eq:149}
 \Big|\Big\langle\zeta,\frac{v}{U+i\lambda}\Big\rangle\Big| \leq C\, \big(\|v^\prime\|_p+ |v(1)|\log|d|^{-1}\big)\,.
 \end{equation}

   For the second term on the right-hand-side  of
 \eqref{eq:148} we have, using \eqref{eq:137} and Hardy's
   inequality \eqref{eq:96}
 \begin{multline}
\label{eq:150}
 \Big|\Big\langle\zeta(\phi-\phi(x_\nu)),\frac{v}{U+i\lambda}\Big\rangle\Big| \leq
  |c_\parallel^\nu|  \Big|\Big\langle\zeta(U-\nu),\frac{v}{U+i\lambda}\Big\rangle\Big| +
  \Big|\Big\langle\zeta\varphi_\perp^\nu,\frac{v}{U+i\lambda}\Big\rangle\Big| \leq\\ \leq C  |c_\parallel^\nu|
  \,\|v\|_1+ \Big\|\frac{\varphi_\perp^\nu}{x-x_\nu}\Big\|_2\Big\|\frac{(x-x_\nu) \, v\, \zeta}{U+i\lambda}\Big\|_2\leq
\widehat C \, \big(|c_\parallel^\nu|\,\|v\|_1+\|(\varphi_\perp^\nu)^\prime\|_2\|v\|_2\big)\,.
 \end{multline}
Substituting \eqref{eq:149}
and \eqref{eq:150} into \eqref{eq:148} yields
\begin{multline}
  \label{eq:151}
\Big|\Big\langle\zeta\phi,\frac{v}{U+i\lambda}\Big\rangle\Big|
\leq C\Big[|\phi(x_\nu)|(\|v^\prime\|_p+ |v(1)|\log|d|^{-1}) \\ + |c_\parallel^\nu|\,\|v\|_1+\|(\varphi_\perp^\nu)^\prime\|_2\|v\|_2\Big]\,.
\end{multline}
As  $0<\nu <\nu_1< U(0)$  and ${\rm supp\,} {\tilde \zeta} \subset  [0, x_\nu/4]\,$, it holds that  $|\tilde{\zeta}(U+i\lambda)^{-1}| \leq C$
  and hence, 
 \begin{equation} \label{eq:152}
    \Big|\Big\langle\tilde{\zeta}\phi,\frac{v}{U+i\lambda}\Big\rangle\Big| \leq
    C\|\phi\|_\infty\|v\|_1\leq C\, \|\phi^\prime\|_2\|v\|_1\,.
 \end{equation}

 Combining \eqref{eq:152}  with \eqref{eq:151} and \eqref{eq:137} then yields
 \begin{equation*}
      \Big|\Big\langle\phi,\frac{v}{U+i\lambda}\Big\rangle\Big| \leq   C\,\big [|\phi(x_\nu)|\big(\|v^\prime\|_p+
      |v(1)|\log|d|^{-1}\big)+|c_\parallel^\nu|\,\|v\|_1+\|(\varphi_\perp^\nu)^\prime\|_2\|v\|_2\big]\,.
 \end{equation*}
With the aid of  \eqref{eq:89} we then obtain that
 \begin{equation}
\label{eq:153}
      \Big|\Big\langle\phi,\frac{v}{U+i\lambda}\Big\rangle\Big|   \leq C\, \big[\|v^\prime\|_p^2+
      |v(1)|^2\log^2|d|^{-1}+|c_\parallel^\nu|\,\|v\|_1+\|(\varphi_\perp^\nu)^\prime\|_2\|v\|_2\big] 
    \end{equation}
Substituting \eqref{eq:153}  into \eqref{eq:119} leads to
\begin{multline}
\label{eq:154}
 |c_\parallel^\nu|\leq C\big(\nu^{-1}[\|v^\prime\|_p+
      |v(1)|\log|d|^{-1}+\|v\|_1+\|(\varphi_\perp^\nu)^\prime\|_2^{1/2}\|v\|_2^{1/2}]\\ + \nu^{-1/2}N(v,\lambda)\big).
\end{multline}
Next, we substitute \eqref{eq:153} into \eqref{eq:135} to obtain, in
view of \eqref{eq:137} 
\begin{multline}
\label{eq:155} 
  \|(\varphi_\perp^\nu)^\prime\|_2\leq C\big(\nu^{-1/2}[\|v^\prime\|_p+
      |v(1)| \log \nu
  ^{-1}+|c_\parallel^\nu|^{1/2}\,\|v\|_1^{1/2}+\nu^{-1/2}\|v\|_2]\\+ \nu^{-1/2} N(v,\lambda)\big).
\end{multline}
Combining \eqref{eq:155} with \eqref{eq:154} yields
\begin{multline}
\label{eq:156}
 \nu^{1/2}\|(\varphi_\perp^\nu)^\prime\|_2+ \nu|c_\parallel^\nu|\leq C\big(\|v^\prime\|_p+
      |v(1)|\log \nu^{-1}+\nu^{-1}\|v\|_1 \\ +\nu^{-1/2}\|v\|_2+ N(v,\lambda)\big)\,.
\end{multline}
Substituting \eqref{eq:156} and \eqref{eq:89}  into \eqref{eq:153} leads to
\begin{multline}
\label{eq:157}
   |\phi(x_\nu)|\leq   \Big|\Big\langle\phi,\frac{v}{U+i\lambda}\Big\rangle\Big|^{1/2} \leq  \\ \leq C\big(\|v^\prime\|_p+
      |v(1)|\log \nu^{-1}+\nu^{-1}\|v\|_1+\nu^{-1/2}\|v\|_2+  N(v,\lambda)\big)\,.
\end{multline}
 On the other hand, given the fact that $\phi(1)=0$ and in view of
 \eqref{eq:137}, it holds  that 
 \begin{displaymath}
   |\phi(x)|=\Big|\int_x^1\phi^\prime(t)\,dt\Big|\leq|c_\parallel^\nu|(1-x)+\|(\varphi_\perp^\nu)^\prime\|_2(1-x)^{1/2} \,.
 \end{displaymath}
Consequently,
 \begin{equation}\label{eq:158}
       \Big|\Big\langle\phi,\frac{v}{U+i\lambda}\Big\rangle\Big| \leq 
       |c_\parallel^\nu| \Big\|(1-x)\, \frac{v}{U+i\lambda}\Big\|_1+
      \|(\varphi_\perp^\nu)^\prime\|_2 \Big\|(1-x)^{1/2}\, \frac{v}{U+i\lambda}\Big\|_1 \,.
 \end{equation}

Substituting \eqref{eq:158}  into 
\eqref{eq:119} yields
\begin{multline}
\label{eq:159}
    |c_\parallel^\nu| \leq \nu^{-2}\Big\|(1-x)\, \frac{v}{U+i\lambda}\Big\|_1+\\
    + \nu^{-1} \|(\varphi_\perp^\nu)^\prime\|_2^{1/2}\Big\|(1-x)^{1/2}\,
      \frac{v}{U+i\lambda}\Big\|_1^{1/2} +\nu^{-1/2}N(v,\lambda)\,.
\end{multline} 
Then, substituting \eqref{eq:158} into \eqref{eq:135} leads to
\begin{multline}
\label{eq:160}
   \|(\varphi_\perp^\nu)^\prime\|_2\leq   \nu^{-1/2}|c_\parallel^\nu|^{1/2}\Big\|(1-x)\, \frac{v}{U+i\lambda}\Big\|_1^{1/2}+\\
     +  \nu^{-1}\Big\|(1-x)^{1/2}\,
      \frac{v}{U+i\lambda}\Big\|_1+ \nu^{-1/2}N(v,\lambda)\,.
\end{multline}
Combining \eqref{eq:159} and \eqref{eq:160} we may conclude 
\begin{multline}
\label{eq:161}
  \nu^{1/2}\|(\varphi_\perp^\nu)^\prime\|_2+ \nu|c_\parallel^\nu|\leq C\,  \Big( \nu^{-1}\Big\|(1-x)\, \frac{v}{U+i\lambda}\Big\|_1^{1/2}+\\
     +  \nu^{-1/2}\Big\|(1-x)^{1/2}\frac{v}{U+i\lambda}\Big\|_1+ N(v,\lambda)\Big)\,.
\end{multline}
Combining \eqref{eq:161} with \eqref{eq:157} and \eqref{eq:86} 
yields \eqref{eq:87} for $\alpha^2\leq A$. \\

{\em Step 6: There exists $A_0 \geq 0$ and $\widehat C$ such that if $\alpha^2 \geq A_0$, $|\mu| \leq 1$, $\nu \in (0,\nu_1)$, then 
\begin{equation}
  \label{eq:162}
\|\phi\|_{H^1(0,1)}\leq \widehat C  \, N(v,\lambda)\,.
\end{equation} for any pair $(\phi,v)\in D(\A_{\lambda,\alpha})\times W^{1,p}(0,1)$ satisfying \eqref{eq:47}.} \\

We preliminarily observe that 
  \begin{displaymath}
    \delta_2 = \sup_{\nu \in (0,\nu_1)}\Big\|\zeta_\nu \frac{U^{\prime\prime}}{U^\prime}\Big\|_{1,\infty} < 
    +\infty  \,,
  \end{displaymath}
\begin{equation*}
\label{eq:163}
  \hat C_0 = \sup_{
    \begin{subarray}{c}
      |\mu|\leq 1  \\
      \nu \in (0,\nu_1)
    \end{subarray}}\|\log\,(U+i\lambda)\|_2 < +\infty\,,
\end{equation*}
and
\begin{equation*}
  \label{eq:164}
\hat C_1=\sup_{
    \begin{subarray}{c}
      |\mu|\leq 1  \\
      \nu \in (0,\nu_1)
    \end{subarray}} \Big\|\tilde{\zeta_\nu }\frac{U^{\prime\prime}}{U+i\lambda}\Big\|_\infty < +\infty  \,,
\end{equation*}
where $\zeta_\nu$ and $\tilde \zeta_\nu$ are defined in \eqref{eq:zetanu}-\eqref{eq:tildezetanu}.
\vspace{2ex}

Taking  (as in \eqref{eq:40} for the
case when $\alpha=0$)  the inner product of \eqref{eq:20} with $\phi$ yields for the real part
\begin{equation}
\label{eq:165}
  \|\phi^\prime\|_2^2 + \alpha^2\|\phi\|_2^2=  \Re\Big\langle\phi,\frac{v}{U+i\lambda}\Big\rangle 
 -  \Re\Big\langle \zeta U^{\prime\prime}\phi,\frac{\phi}{U+i\lambda}\Big\rangle - \Re\Big\langle
  \tilde{\zeta}U^{\prime\prime}\phi,\frac{\phi}{U+i\lambda}\Big\rangle \,. 
\end{equation}
For the first term in the right hand side, we can use \eqref{eq:628} to
obtain
\begin{displaymath}
| \Re\Big\langle\phi,\frac{v}{U+i\lambda}\Big\rangle |\nobreakspace\leq  C\, \|\phi^\prime\|_2 N(v,\lambda) \,.
\end{displaymath}
For the second term we apply Poincar\'e's inequality, the finiteness of
$\hat C_0$ and $\delta_2$,  and Sobolev embedding   
\begin{displaymath}
\| \phi\|_\infty^2 \leq 2 \|\phi \|_2 \|\phi^\prime\|_2\,,
\end{displaymath}
to conclude that for some new constant $\widehat C$
 \begin{equation*}\begin{array}{ll}
  \Big|\Big\langle \zeta U^{\prime\prime}\phi,\frac{\phi}{U+i\lambda}\Big\rangle\Big| 
 &  \leq  \Big|\Big\langle
  \Big(\frac{U^{\prime\prime}}{U^\prime}\zeta|\phi|^2\Big)^\prime,\log\, (U+i\lambda)\Big\rangle\Big| \\ &\leq 
  \|\log\,(U+i\lambda)\|_2\|\phi\|_\infty\Big(2\Big\|\zeta\frac{U^{\prime\prime}}{U^\prime}\Big\|_\infty
  \|\phi^\prime\|_2+
  \Big\|\Big(\zeta\frac{U^{\prime\prime}}{U^\prime}\Big)^\prime\Big\|_\infty\|\phi\|_2\Big) \\ &\leq 
 \widehat C \, \|\phi^\prime\|_2^{3/2}\|\phi\|_2^{1/2} \,.
  \end{array}
\end{equation*}

For the last term on the right-hand-side  of \eqref{eq:165} we have, 
\begin{displaymath}
  \Big|\Big\langle
  \tilde{\zeta}U^{\prime\prime}\phi,\frac{\phi}{U+i\lambda}\Big\rangle\Big|\leq
  \|\phi\|_2^2\Big\|\tilde{\zeta}\frac{U^{\prime\prime}}{U+i\lambda}\Big\|_\infty
  \leq \widehat C_1 \, \|\phi\|_2^2
\end{displaymath}
Consequently, 
\begin{displaymath}
  \|\phi^\prime\|_2^2 +\alpha^2\|\phi\|_2^2\leq \hat C \left( 
  \|\phi^\prime\|_2^{3/2}\|\phi\|_2^{1/2} + \|\phi\|_2^2+
  \|\phi^\prime\|_2N(v,\lambda)\right) \,.
\end{displaymath}
Using Young's inequality we obtain, for some $A_0\geq 0$ and $\widehat C >0$
\begin{equation}
\label{eq:166} 
  \frac{1}{2}\|\phi^\prime\|_2^2\leq \Big( A_0 -\alpha^2\Big)\|\phi\|_2^2 + \widehat  C\, N(v,\lambda)^2
\end{equation}
Hence, for $\alpha^2\geq A_0$,   \eqref{eq:162} follows immediately from the above
inequality in conjunction with Poincar\'e's inequality.\\

{\em Conclusion:} Observing that $\frac 1C \nu \leq d \leq C \nu $, the
proof of \eqref{eq:87} follows from \eqref{eq:119}, \eqref{eq:135}, and
\eqref{eq:162}.
\end{proof}

\subsection{The case $\Im \lambda<0$.}
\label{sec:case--nu-negative}
We now consider the case where $\Im\lambda$ is negative. Due to the the
non-invertibility of $\A_{0,0}$ the estimates of $\mathcal
A_{\lambda,\alpha}^{-1}$ become challenging in the limit $\Im\lambda\to0$ and
the bounds necessarily include negative powers of $\nu=\Im \lambda$.
\begin{proposition}
  Let $U\in C^3([0,1]) $ satisfy \eqref{eq:10}. Then there exist $C>0$
  such that for any $\alpha \geq 0$, and $(\phi,v)\in D(\A_{\lambda,\alpha})\times
  W^{1,p}(0,1)$ (where $\A_{\lambda,\alpha}$ is defined in
    \eqref{eq:2.1aa}) satisfying \eqref{eq:47}, we have for all
  $\nu<0$,
\begin{subequations}
  \label{eq:170}
  \begin{equation}
 \|\phi\|_{1,2} \leq  C \,  (1+|\nu|^{-1})  \,  \Big\|(1-x)^{1/2}\frac{v}{U+i\lambda}\Big\|_1\,,
\end{equation}
and, for $-1/2<\nu<0$
\begin{equation}
  \|\phi^\prime\|_{L^2(1-|\nu|^{1/2},1)}\leq
  C\, |\nu|^{-3/4} \Big\|(1-x)^{1/2}\frac{v}{U+i\lambda}\Big\|_1\,.
\end{equation}
\end{subequations}
Furthermore, it holds that for all $\nu < 0$ 
\begin{equation}
  \label{eq:171}
|\phi(x)|\leq C\, (1-x)^{1/2}[1+|\nu|^{-1/2}(1-x)^{1/2}] \Big|\Big\langle\phi,\frac{v}{U+i\lambda}\Big\rangle\Big|^{1/2} \,.
\end{equation}
\end{proposition}
\begin{proof}
  We begin by rewriting $\A_{\lambda,\alpha}\phi=v$ in the form
\begin{displaymath}
  -\big((U-\nu)^2w^\prime\big)^\prime+\alpha^2(U-\nu)^2w
    = v-i\mu\frac{v-U^{\prime\prime}\phi}{U+i\lambda}=\frac{(U-\nu)v+i\mu U^{\prime\prime}\phi}{U+i\lambda}\,,
\end{displaymath}
where $$w=\phi/(U-\nu)\,.
$$
Taking the inner product with $w$ on the left yields (see  \eqref{eq:106} and \eqref{eq:125})
\begin{displaymath}
   \|(U-\nu)w^\prime\|_2^2 + \alpha^2\|\phi\|_2^2 
=  \Big\langle\phi,\frac{v}{U+i\lambda}\Big\rangle +i\mu\Big\langle w,\frac{U^{\prime\prime}\phi}{U+i\lambda}\Big\rangle \,. 
\end{displaymath}
For the last term on the right-hand-side we have, since $U^{\prime\prime} <0$ and $U-\nu >0$,
\begin{displaymath}
  \Re\Big(i\mu\Big\langle w,\frac{U^{\prime\prime}\phi}{U+i\lambda}\Big\rangle\Big)
  =|\mu|^2\Big\langle\frac{\phi}{U-\nu},\frac{U^{\prime\prime}\phi}{|U+i\lambda|^2}\Big\rangle<0 \,. 
\end{displaymath}
Hence
\begin{equation}
  \label{eq:172}
   \|(U-\nu)w^\prime\|_2^2 + \alpha^2\|\phi\|_2^2 
\leq \Re   \Big\langle\phi,\frac{v}{U+i\lambda}\Big\rangle 
\end{equation}
We now write, using the fact that $U(x)-\nu  \geq C^{-1}(1-x-\nu)$,
\begin{multline}
\label{eq:173} 
  |w(x)|\leq
  \Big|\int_x^1 w^\prime(t) \,dt\Big|
  \leq \Big[\int_x^1\frac{1}{(U-\nu)^2}\,dt\Big]^{1/2}\,
\|(U-\nu)w^\prime\|_2 \\
  \leq C  \frac{(1-x)^{1/2}}{|\nu|^{1/2}(1-x-\nu)^{1/2}} \,\|(U-\nu)w^\prime\|_2 \,.
\end{multline}
Using this time the bound $U(x)-\nu \leq1-x-\nu$ together with
\eqref{eq:172} and \eqref{eq:173},  we obtain \eqref{eq:171}. 
Integrating \eqref{eq:173} squared over $[0,1]$ yields
\begin{equation}
\label{eq:174}
  \|w\|_2^2\leq C \, |\nu|^{-1}\, \|(U-\nu)w^\prime\|_2^2 \,.
\end{equation}
By writing $\phi= (U-\nu) \, w $, we obtain that
  \begin{equation*}
  \|\phi^\prime\|_2^2 \leq 2  \, \|(U-\nu)w^\prime\|_2^2 + 2\, 
  \|U^\prime w\|_2^2  \,,
\end{equation*}
which leads, together with \eqref{eq:174},  to
\begin{equation}
\label{eq:175}
\|\phi^\prime\|_2^2 \leq C (|\nu|^{-1} +1) \,  \|(U-\nu)w^\prime\|_2^2\,.
\end{equation}
We can now establish (\ref{eq:170}a) by combining \eqref{eq:175} with
\eqref{eq:172}  and the fact that $|\phi(x)|\leq\|\phi^\prime\|_2(1-x)^{1/2}$. 

To obtain (\ref{eq:170}b) we write  for $-1/2 < \nu < 0$, 
\begin{equation}
\label{eq:176}
  \|\phi^\prime\|_{L^2(1-|\nu|^{1/2},1)}^2
   \leq 2\, \|(U-\nu)w^\prime\|_{L^2(1-|\nu|^{1/2},1)}^2 +  
2\, \|U^\prime w\|_{L^2(1-|\nu|^{1/2},1)}^2
\end{equation}
By \eqref{eq:173} and the fact that for all $x\in[0,1]$
  \begin{displaymath}
    0\leq \frac{1-x}{|\nu|(1-x-\nu)} \leq \frac{1}{|\nu|}
  \end{displaymath}
 we obtain via integration over $(1-|\nu|^{1/2},1)$  that
\begin{displaymath}
  \|U^\prime w\|_{L^2(1-|\nu|^{1/2},1)}^2 \leq
  C\, \|w\|_{L^2(1-|\nu|^{1/2},1)}^2\leq \widehat C |\nu|^{-1/2}\|(U-\nu)w^\prime\|_2^2 \,.
\end{displaymath}
Substituting the above into 
\eqref{eq:176} yields
\begin{equation}
\label{eq:177} 
   \|\phi^\prime\|_{L^2(1-|\nu|^{1/2},1)}^2
  \leq C\, |\nu|^{-1/2}\|(U-\nu)w^\prime\|_2^2\,.
\end{equation}
By  \eqref{eq:172} and  (\ref{eq:170}a) we then obtain that for $-1/2 < \nu < 0$, 
\begin{displaymath}
   \|\phi^\prime\|_{L^2(1-|\nu|^{1/2},1)}^2
  \leq C\, |\nu|^{-1/2}\|\phi^\prime\|_2
  \,\Big\|(1-x)^{1/2}\frac{v}{U+i\lambda}\Big\|_1\leq \widehat C\, |\nu|^{-3/2}  \,\Big\|(1-x)^{1/2}\frac{v}{U+i\lambda}\Big\|_1^2\,,
\end{displaymath}
readily verifying  (\ref{eq:170}b).
\end{proof}

\subsection{The case  $\nu_1\leq\Im \lambda <U(0)- \kappa_0 | \Re \lambda| $, $|\Re \lambda|$
  small.}
\label{sec:case-near-quadratic}
In the following we establish estimates, similar to \eqref{eq:52}, for
$\nu_1\leq\nu<U(0)- \kappa_0 |\mu|$, for sufficiently small $U(0)-\nu_1$ and
$|\mu|$. 
 \begin{proposition}
   \label{prop:near-quadratic} 
 Let $U \in C^4([0,1])$ satisfy
   \eqref{eq:10} and $U^{\prime\prime\prime}(0)=0$.  Let further $p\geq2$.  There exist $
   0<\nu_1<U(0)$, $\mu_0>0$, $\kappa_0>0$, and $C>0$ such that for $\lambda $
   s.t.  $ 0<|\mu|\leq\mu_0$ and $ \nu_1<\nu<U(0)-\kappa_0|\mu|$, for all $\alpha \geq 0$
   we have, for all pair $(\phi,v)\in D(\A_{\lambda,\alpha})\times L^\infty(0,1)$ satisfying
   \eqref{eq:47},
\begin{subequations}
\label{eq:178}
    \begin{equation}
    \|\phi \|_{1,2} \leq \frac{C}{|\mu|^{\frac{1}{p}} \, x_\nu^{1/2-1/p}} \|v\|_p \,,
\end{equation}
and
  \begin{equation}
\|\phi \|_{1,2} \leq
C \frac{\log\frac{x_\nu}{|\mu|^{1/2}}}{x_\nu^{1/2}}\|v\|_\infty 
 \,.
\end{equation}
\end{subequations}
\end{proposition}

\begin{proof}
Since $0<\nu<U(0)$ and since by \eqref{eq:defxnu} $U(x_\nu)=\nu$, we must
have that $x_\nu\in(0,1)$.
 
{\em Step 1: We prove that there exist $C>0$, $\mu_0 >0$,  and 
$\nu_1 < U(0)$, such that, for all $\lambda =\mu + i \nu $ such that $
\nu_1<\nu<U(0)-|\mu|$ and $0<|\mu|\leq \mu_0$ it holds that
\begin{equation}
\label{eq:179}
   |\phi(x_\nu)|^2\leq Cx_\nu\Big[\frac{\mu}{x_\nu^2}\|\phi^\prime\|_2^2+ \Big|\Big\langle\phi,\frac{v}{U+i\lambda} \Big\rangle\Big|\Big] \,.
\end{equation}
for all pairs $(\phi,v)\in D(\A_{\lambda,\alpha})\times L^\infty(0,1)$ satisfying
\eqref{eq:47}.}\\ 
 
\noindent As in \eqref{eq:91} we write
\begin{equation}\label{eq:180}
\Big|  \Im\Big\langle\phi,\frac{v}{U-\nu+i\mu}\Big\rangle\Big| \geq |\mu| \,
   \Big\langle\frac{|U^{\prime\prime}|}{(U-\nu)^2+\mu^2},\frac 12 |\phi(x_\nu)|^2 - |\phi(x)-\phi(x_\nu)|^2 \Big\rangle\,.
\end{equation}
We begin by observing that for
some $C>0$ 
\begin{equation}
\label{eq:181}
  \Big|U(x)-\nu-\frac{1}{2}U^{\prime\prime}(0)[2x_\nu(x-x_\nu)+(x-x_\nu)^2]\Big|\leq
  C[(x-x_\nu)^4+x_\nu^3 | x-x_\nu| ]\,.
\end{equation}
 By \eqref{eq:181} we have 
\begin{equation}\label{eq:182}
\begin{array}{l}
\frac{1}{2}U^{\prime\prime}(0) |x^2-x_\nu^2| -C[(x-x_\nu)^4+x_\nu^3
  | x-x_\nu|]\\
\qquad   \qquad  \leq  |U-\nu| \leq \frac{1}{2}U^{\prime\prime}(0) |x^2-x_\nu^2| + C[(x-x_\nu)^4+x_\nu^3
  | x-x_\nu|] \,. 
  \end{array}
\end{equation}
From the right inequality  in  \eqref{eq:182}, as $(x-x_\nu)^3
+x_\nu^3 \leq   2 (x+x_\nu)$,  we conclude that there exists $C_1>0$ such that
\begin{equation}
\label{eq:183}
   |U-\nu| \leq C_1 |x^2-x_\nu^2|\,.
\end{equation}
From the left inequality in \eqref{eq:182}, there exist $a_0>0$ and
$\nu_1 < U(0)$ such that for all $\nu_1 < \nu < U(0)$ and $x +x_\nu \leq a_0
$ it holds that
\begin{equation}
\label{eq:184}
  \frac{1}{2}U^{\prime\prime}(0) |x^2-x_\nu^2| -C[(x-x_\nu)^4+x_\nu^3
  | x-x_\nu|] \geq  \frac{1}{4}U^{\prime\prime}(0) |x^2-x_\nu^2| \,,
\end{equation}
from which we can conclude that whenever $x +x_\nu \leq
a_0 $ we have  
\begin{equation}\label{eq:185}
|U-\nu| \geq \frac 1C  |x^2-x_\nu^2| \,.
\end{equation}
On the other hand we have that, for $\nu_1 < \nu < U(0)$ such that $x_{\nu_1} < \frac{a_0}{2}$,
  \begin{displaymath}
    \inf_{x\geq a_0-x_\nu}\frac{|U-\nu|}{x^2-x_\nu^2}\geq
    \frac{1}{a_0}\,\inf_{x\geq a_0-x_\nu}\frac{|U-\nu|}{x-x_\nu}\geq
    \frac{|U^\prime(a_0-x_{\nu_1})|}{a_0}>0 \,.
  \end{displaymath}
Combining the above with \eqref{eq:185} and \eqref{eq:183} yields
the existence of $0< \nu_1< U(0)$ and $C>0$ for which
\begin{equation}
\label{eq:186}
  \frac{1}{C}(x^2-x_\nu^2)^2\leq (U(x) -\nu)^2\leq C(x^2-x_\nu^2)^2\,,
\end{equation}
for all $x\in[0,1]$ and $\nu_1 < \nu < U(0)$. 

From \eqref{eq:186} we can get that
\begin{displaymath}
  \int_0^1\frac{|U^{\prime\prime}|}{(U-\nu)^2+\mu^2}\,dx \geq
  \frac 1C \int_0^1\frac{1}{(x-x_\nu)^2(x+x_\nu)^2+\mu^2}\,dx \,.
\end{displaymath}
As
\begin{displaymath}
 \sup_{x\in[1,\infty)} \frac{x^4}{(x-x_\nu)^2(x+x_\nu)^2+\mu^2}  \leq  \frac{1}{(1-x_{\nu_1}^2)^2}\,,
\end{displaymath}
we obtain that 
\begin{displaymath}
\begin{array}{ll}
  \int_0^1\frac{dx}{(x-x_\nu)^2(x+x_\nu)^2+\mu^2}& \geq
  \int_0^\infty\frac{dx}{(x-x_\nu)^2(x+x_\nu)^2+\mu^2}-  \frac{1}{(1-x_{\nu_1}^2)^2}\, \int_1^\infty\frac{dx}{x^4}\\
   &  =  \int_0^\infty\frac{dx}{(x-x_\nu)^2(x+x_\nu)^2+\mu^2} -C_2\,,
     \end{array}
\end{displaymath}
where $C_2 := \frac{1}{3 (1-x_{\nu_1}^2)^2}\  $.\\
Using the substitution $\xi=x/x_\nu$ it can be easily verified that
\begin{displaymath}
  \int_0^\infty\frac{dx}{(x-x_\nu)^2(x+x_\nu)^2+\mu^2}=\frac{1}{2x_\nu^3}\int_{-\infty}^\infty \frac{d\xi}{(\xi^2-1)^2+\hat a^2} \,,
\end{displaymath}
where $\hat a=|\mu|/x_\nu^2$. \\
Hence, there exists $C>0$ such that
\begin{equation}
\label{eq:187}
   \int_0^1\frac{|U^{\prime\prime}|}{(U-\nu)^2+\mu^2}\,dx \geq
   \frac 1 C \Big[\frac{1}{x_\nu^3}\int_{-\infty}^\infty\frac{dx}{(x^2-1)^2+\hat a^2} - 2 C_2\Big]\,, 
\end{equation}
 Making use of the residue theorem yields 
\begin{multline}
\label{eq:188}
  \int_{-\infty}^\infty\frac{dx}{(x^2-1)^2+\hat a^2}=2\pi i\Big[{\rm
    Res}\Big(\frac{1}{(z^2-1)^2+\hat a^2},\sqrt{1+i \hat a}\Big)\\\qquad \qquad
  \qquad  +  {\rm
    Res}\Big(\frac{1}{(z^2-1)^2+\hat a^2},-\sqrt{1-i\hat a}\Big)\Big]\\ 
    =\frac{\pi\Re(\sqrt{1+i\hat a})}{\hat a\sqrt{\hat a^2+1}} \,.
\end{multline}
Consequently, given that, by \eqref{eq:183} applied with $x=0$, it holds, for
sufficiently small $U(0)-\nu$, that $$\mu<U(0)-\nu\leq
C_1x_\nu^2\,,$$ hence $x_\nu^2\geq c_0 \mu$ with $c_0 = 1/C_1$ and finally 
that $\hat a\leq \frac{1}{c_0}$.\\
 Consequently, there exists $0<\nu_1<U(0)$
and $\mu_0>0$ such that for all $\nu_1 < \nu < U(0)-|\mu|$ and $0 < |\mu |\leq \mu_0$ we
have for some $\hat{C}>0$
\begin{equation}
  \label{eq:189} 
\int_0^1\frac{|U^{\prime\prime}|}{(U-\nu)^2+\mu^2}\,dx \geq \frac{1}{C}
\Big[\frac{1}{x_\nu^3} \frac{\pi\Re(\sqrt{1+i\hat a})}{\hat a\sqrt{\hat a^2+1}}
- 2 C_2 \Big]\geq  \frac{\hat{C}} { |\mu | x_\nu}\,.
\end{equation}
In a similar manner we can also show the existence of a positive $C$
such that
\begin{equation}
  \label{eq:190}
\int_0^1\frac{|U^{\prime\prime}|}{(U-\nu)^2+\mu^2}\,dx \leq \tilde{C}\int_0^\infty\frac{dx}{(x-x_\nu)^2(x+x_\nu)^2+\mu^2}\,dx \leq  \frac{C}{|\mu| x_\nu}\,.
\end{equation}

We proceed with the estimation of the right-hand-side of
\eqref{eq:180} by observing that, in view of \eqref{eq:186}
\begin{equation}
\label{eq:191}
     \Big\langle\frac{|U^{\prime\prime}|}{(U-\nu)^2+\mu^2},|\phi(x)-\phi(x_\nu)|^2
     \Big\rangle\leq C\Big\|\frac{\phi(x)-\phi(x_\nu)}{x^2-x_\nu^2}\Big\|_2^2
\end{equation}
Applying Hardy's inequality yields
\begin{equation}
\label{eq:192}
  \Big\|\frac{\phi(x)-\phi(x_\nu)}{x^2-x_\nu^2}\Big\|_2^2\leq
  C\Big\|\Big(\frac{\phi(x)-\phi(x_\nu)}{x+x_\nu}\Big)^\prime\Big\|_2^2 \,.
\end{equation}
 As
\begin{displaymath}
  \Big\|\Big(\frac{\phi(x)-\phi(x_\nu)}{x+x_\nu}\Big)^\prime\Big\|_2^2
  \leq 2\Big\|\frac{\phi^\prime}{x+x_\nu}\Big\|_2^2  
  + 2 \Big\|\frac{\phi-\phi(x_\nu)}{(x+x_\nu)^2}\Big\|_2^2 \,,
\end{displaymath}
and since by Hardy's inequality \eqref{eq:18}
\begin{displaymath}
   \Big\|\frac{\phi-\phi(x_\nu)}{(x+x_\nu)^2}\Big\|_2^2 \leq \frac{1}{x_\nu^2}
   \Big\|\frac{\phi-\phi(x_\nu)}{x-x_\nu}\Big\|_2^2 \leq
   \frac{C}{x_\nu^2}\|\phi^\prime\|_2^2 \,,
\end{displaymath}
we obtain that
\begin{displaymath}
   \Big\|\Big(\frac{\phi(x)-\phi(x_\nu)}{x+x_\nu}\Big)^\prime\Big\|_2^2\leq   \frac{C}{x_\nu^2}\|\phi^\prime\|_2^2\,.
\end{displaymath}
Substituting the above into \eqref{eq:192} and then into \eqref{eq:191} yields
\begin{equation}
\label{eq:193}
   \Big\langle\frac{|U^{\prime\prime}|}{(U-\nu)^2+\mu^2},|\phi(x)-\phi(x_\nu)|^2 \Big\rangle \leq
   \frac{C}{x_\nu^2}\|\phi^\prime\|_2^2\,.
\end{equation}
which, when substituted into \eqref{eq:180} together with
\eqref{eq:189}
leads to \eqref{eq:179}. \\

{\em Step 2: We prove that there exist $0<\nu_1<U(0)$, positive
$\mu_0$, $C$, and $C_0$, and $\kappa_0 \geq 1$ such that, for all $\alpha \geq
C_0/x_\nu$, $\nu_1<\nu<U(0)-\kappa_0|\mu|$, and $0<|\mu|\leq \mu_0$\,, the inequality
\begin{equation}
\label{eq:194}
\|\phi\|_{1,2}\leq  C\Big|\Big\langle\phi,\frac{v}{U+i\lambda}\Big\rangle\Big|^{1/2}\,,
\end{equation}
holds for any pair $(\phi,v)\in D(\A_{\lambda,\alpha})\times L^\infty(0,1)$ satisfying
\eqref{eq:47}.\\
}

Let
\begin{displaymath}
  \tilde{U}=U(0)-U \mbox{ and }  \kappa= \sqrt{(i\lambda+U(0))} \,.
\end{displaymath}
 Note that $|\kappa|>|\mu|^{1/2}>0\,$.  
Clearly,
\begin{displaymath}
  \frac{1}{U+i\lambda}=-\frac{1}{(\tilde{U}^{1/2}-\kappa)(\tilde{U}^{1/2}+\kappa)}=
  -\frac{1}{2\kappa}\Big[\frac{1}{\tilde{U}^{1/2}-\kappa} - \frac{1}{\tilde{U}^{1/2}+\kappa}\Big] \,.
\end{displaymath}
We now write
\begin{displaymath}
   \int_0^1\frac{U^{\prime\prime}}{U+i\lambda} \,dx
  =-\frac{1}{4\kappa}\int_0^1U^{\prime\prime}
  \Big[\frac{1}{\tilde{U}^{1/2}-\kappa} -
  \frac{1}{\tilde{U}^{1/2}+\kappa}\Big] \,dx    \,.
\end{displaymath}
 An integration by parts now yields
\begin{multline*}
  \int_0^1U^{\prime\prime}
  \Big[\frac{1}{\tilde{U}^{1/2}-\kappa} -
  \frac{1}{\tilde{U}^{1/2}+\kappa}\Big]  \,dx \\
=  \frac{1}{2}\frac{U^{\prime\prime}}{\tilde{U}^{-1/2}U^\prime}
  \log\frac{\tilde{U}^{1/2}-\kappa}{\tilde{U}^{1/2}+\kappa}\Big|_0^1  +
  \frac{1}{2}\int_0^1\Big(\frac{U^{\prime\prime}}{\tilde{U}^{-1/2}U^\prime}\Big)^\prime 
  \log\frac{\tilde{U}^{1/2}-\kappa}{\tilde{U}^{1/2}+\kappa}\,dx
\end{multline*}
Since $\max_{x\in[0,1]}(\tilde{U}^{-1/2}U^\prime)(x)<0$ and since
$\tilde{U}^{-1/2}U^\prime\in C^2([0,1])$ we can conclude that there
exist positive $C_1,C_2,C$ such that
\begin{displaymath}
  \Big|\int_0^1\frac{U^{\prime\prime}}{U+i\lambda} \,dx\Big| \leq C_1 +\frac{C_2}{\kappa}
\int_0^1[|\log\tilde{U}^{1/2}-\kappa |+ \log\tilde{U}^{1/2}+\kappa |]\,dx \leq
\frac{C}{\kappa} \,.
\end{displaymath}
Since by \eqref{eq:185} it holds that $\kappa\geq\sqrt{U(0)-\nu}\geq x_\nu/C$,
we conclude from the above that there exists $\widehat C>0$ such that
\begin{equation}
\label{eq:195}
  \Big|\int_0^1\frac{U^{\prime\prime}}{U+i\lambda} \,dx\Big| \leq \frac{\widehat C}{x_\nu} \,.
\end{equation}
Taking the inner product of \eqref{eq:47} with $\phi /(U+i\lambda)$ yields
for the real part
\begin{equation}
\label{eq:196}
  \Re\Big\langle\phi,\frac{v}{U+i\lambda}\Big\rangle= \|\phi^\prime\|_2^2 + \alpha^2\|\phi\|_2^2+
  \Re\Big\langle U^{\prime\prime}\phi,\frac{\phi}{U+i\lambda}\Big\rangle\,. 
\end{equation}
We then write
\begin{multline}
\label{eq:197}
  \|\phi^\prime\|_2^2 + \alpha^2\|\phi\|_2^2 \leq  \Re\Big\langle\phi,\frac{v}{U+i\lambda}\Big\rangle+
   \\[1.5ex] + \Big[\int_0^1\frac{
    |U^{\prime\prime}(x)|\, \big| \,|\phi(x)|^2-|\phi(x_\nu)|^2\,\big|}{|U+i\lambda|} \,dx +|\phi(x_\nu)|^2\,\bigg|\int_0^1\frac{
     U^{\prime\prime}(x)}{U+i\lambda} \,dx\bigg|  \,\Big]\,.
\end{multline}
By \eqref{eq:195} it holds that
\begin{equation}
\label{eq:198}
  |\phi(x_\nu)|^2\,\bigg|\int_0^1\frac{
     U^{\prime\prime}(x)}{U+i\lambda} \,dx\bigg| \leq
   \frac{\widehat C}{x_\nu}|\phi(x_\nu)|^2 \,.
\end{equation}
To estimate the first term on the right-hand-side of \eqref{eq:198} we
use the inequality
\begin{displaymath}
  \big| \,|\phi(x)|^2-|\phi(x_\nu)|^2\,\big| \leq [|\phi(x)|+|\phi(x_\nu)|]\,|\phi(x)-\phi(x_\nu)|
\end{displaymath}
together with Hardy's inequality and \eqref{eq:186} to obtain
that for some $C_0>0$ 
\begin{equation}
\label{eq:199} 
\begin{array}{ll}
  \int_0^1\frac{|U^{\prime\prime}(x)|\big|
    \,|\phi(x)|^2-|\phi(x_\nu)|^2\,\big|}{|U+i\lambda|} \,dx &  \leq \Big\| \frac{x^2
    -x_\nu^2} {U+i\lambda}\Big\|_\infty \,\Big\|\frac{|\phi|+|\phi(x_\nu)|}{x+x_\nu}\Big\|_2\Big\|\,  
   \Big\|\frac{\phi-\phi(x_\nu)}{x-x_\nu}\Big\|_2
    \\ & \leq C_0
   [ x_\nu^{-1}\|\phi\|_2 + x_\nu^{-1/2}|\phi(x_\nu)| ] \,\|\phi^\prime\|_2 \,. 
   \end{array}
\end{equation}
Note that to obtain the last inequality we need the estimate
  \begin{displaymath}
    \Big\|\frac{\phi(x_\nu)}{x+x_\nu}\Big\|_2\leq C\,  \frac{|\phi(x_\nu)|}{x_\nu^{1/2}}\,. 
  \end{displaymath}
Substituting  \eqref{eq:199} together with \eqref{eq:198} into
\eqref{eq:197} yields the existence of $\hat C_0 >0$ and $C>0$ such that, 
for $\alpha\geq \hat C_0/x_\nu\,$,  
\begin{displaymath}
  \|\phi^\prime\|_2^2 \leq  \frac{C}{x_\nu}|\phi(x_\nu)|^2 +2 \Big|\Big\langle\phi,\frac{v}{U+i\lambda}\Big\rangle\Big|\,.
\end{displaymath}
{Substituting  \eqref{eq:179} into the above yields
\begin{displaymath}
   \|\phi^\prime\|_2^2 \leq  \frac{|\mu|}{x_\nu^2}|\phi(x_\nu)|^2 +
   (2+Cx_\nu)\Big|\Big\langle\phi,\frac{v}{U+i\lambda}\Big\rangle\Big|\,. 
\end{displaymath}
For sufficiently large $\kappa_0$ (or equivalently for sufficiently small
$|\mu|/x_\nu^2$) we obtain \eqref{eq:194}. }

{\em Step 3: With  $N_2(v,\lambda)$  given by
\begin{equation}
  \label{eq:200}
N_2(v,\lambda)=\Big\|(1-x)^{1/2}\frac{v}{U+i\lambda}\Big\|_1\,,
\end{equation}
we prove that, 
for any $\hat C_0>0$, there exist $C>0$, $\mu_0>0$,  $\kappa_0>0$, and $\nu_1>0$
such that, for $\nu_1 < \nu < U(0)-\kappa_0|\mu|$, $|\alpha| \leq \hat C_0/x_\nu$, and $|\mu|\leq \mu_0$\,, we have 
\begin{equation}
\label{eq:201}
\|\phi^\prime\|_{L^2(0,x_\nu)}\leq
C\Big(\Big|\Big\langle\phi,\frac{v}{U+i\lambda}\Big\rangle\Big|^{1/2}+
x_\nu^{1/2}N_2(v,\lambda)\Big) \,,
\end{equation}
holds for any pair $(\phi,v)\in D(\A_{\lambda,\alpha})\times W^{1,p}(0,1)$ satisfying
\eqref{eq:47}. }\\

We seek an estimate for $ \|\phi^\prime\|_{L^2(0,x_\nu)}$, depending on $
\|\phi\|_{L^2(0,x_\nu)}$. We begin, to this end, by obtaining an $L^\infty$ estimate of $\phi^\prime$.
\paragraph{Estimate of $\|\phi^\prime\|_{L^\infty}$.}~\\
We separately consider  the subintervals $(0,x_\nu/2)$ and
$(x_\nu/2,x_\nu)$.
\paragraph{Estimate on $(0,x_\nu/2)$.}
To obtain an estimate for $\|\phi^\prime\|_{L^\infty(0,x_{\nu}/2)}$ we integrate
the relation $\mathcal A_{\lambda,\alpha}\phi=v$, to obtain for all $x\in
(0,x_\nu)$,
\begin{equation*}
  |\phi^\prime(x)|=\Big|\int_0^x\phi^{\prime\prime}(t)\,dt\Big| 
  \leq\Big|\int_0^x \left(\frac{U^{\prime\prime}\phi-v}{U+i\lambda}+\alpha^2\phi\right) \,dt\Big|\,,
\end{equation*}
 which leads to
\begin{equation}
\label{eq:202}
  |\phi^\prime(x) \leq \Big|\int_0^x \left(\frac{U^{\prime\prime}\phi}{U+i\lambda}\right)
  \,dt\Big| + \Big\| \frac{v}{U+i \lambda}\Big\|_{L^1(0,x)} +\alpha^2 \|\phi\|_{L^1(0,x)}
  \,. 
\end{equation}
We then use the following decomposition
\begin{equation}\label{eq:203}
  \int_0^x\frac{U^{\prime\prime}\phi}{U+i\lambda} \,dt=
  \int_0^x\frac{U^{\prime\prime}[\phi-\phi(x_\nu)]}{U+i\lambda} \,dt+  \phi(x_\nu)\int_0^x\frac{U^{\prime\prime}}{U+i\lambda} \,dt
\end{equation}
To estimate the first integral on the right-hand-side of
\eqref{eq:203} we need the following bound which follows from
\eqref{eq:181}
\begin{displaymath}
  \Big|\frac{U^{\prime\prime}}{U-\nu}-\frac{2}{x^2-x_\nu^2}\Big|  \leq
  \Big|\frac{U^{\prime\prime}(0)}{U-\nu}-\frac{2}{x^2-x_\nu^2}\Big| +
  \Big|\frac{U^{\prime\prime}(x)-U^{\prime\prime}(0) }{U-\nu}\Big| \leq
  C\frac{x_\nu}{x_\nu^2-x^2}\,. 
\end{displaymath}
Consequently
\begin{equation}\label{eq:204}
  \Big|\int_0^x\frac{U^{\prime\prime}[\phi-\phi(x_\nu)]}{U+i\lambda} \,dt\Big|  \leq  \int_0^x \frac{|U^{\prime\prime}|\, |\phi-\phi(x_\nu)|}{|U-\nu|} \,dt \leq 
  (2+Cx_\nu) \int_0^x\frac{|\phi(t)-\phi(x_\nu)|}{x_\nu^2-t^2} \,dt 
  \,,
\end{equation}
and hence, for sufficiently small $U(0)-\nu_1$, 
\begin{multline}
\label{eq:205}
  \Big|\int_0^x\frac{U^{\prime\prime}[\phi-\phi(x_\nu)]}{U+i\lambda}
  \,dt\Big|\leq(2+Cx_\nu) \,\log
  \frac{x_\nu+x}{x_\nu} \|\phi^\prime\|_{L^\infty(0,x)} \\ \leq (1+Cx_\nu)\log(9/4)\|\phi^\prime\|_{L^\infty(0,x)}\,.
\end{multline}
To obtain the last inequality we have used the fact that $x\in
(0,x_\nu/2)$.  Note that $\log(9/4)<1$. Hence the coefficient of
$|\phi^\prime\|_{ L^\infty(0,x)}$ is smaller than one for sufficiently small
$U(0)-\nu_1$.  Next, we write for the second integral in the
  right-hand-side of \eqref{eq:203}
\begin{equation*}
   \Big|\phi(x_\nu)\int_0^x\frac{U^{\prime\prime}}{U+i\lambda} \,dt\Big|
   \leq   |\phi(x_\nu)|  \Big| \int_0^x\frac{|U^{\prime\prime}|}{|U-\nu|} \,dt\Big| 
   \leq(2+Cx_\nu)|\phi(x_\nu)|\int_0^x\frac{dt}{x_\nu^2-t^2}\,,\end{equation*}
which leads to 
\begin{equation}
\label{eq:206}
   \Big|\phi(x_\nu)\int_0^x\frac{U^{\prime\prime}}{U+i\lambda} \,dt\Big|
   \leq   \hat C\,   \frac{|\phi(x_\nu)|}{x_\nu} \, \log \frac{x_\nu+x}{x_\nu-x} \,. 
\end{equation}
Substituting the above, together with \eqref{eq:206} and \eqref{eq:205}
into \eqref{eq:202} yields for all $t\in (0,x] \subset (0,x_\nu/2] $
\begin{displaymath}
   |\phi^\prime(t)|\leq \log \frac{9}{4} (1+ C x_\nu)  \|\phi^\prime\|_{L^\infty(0,t)}  + 
   C  \frac{|\phi(x_\nu)|}{x_\nu} \, \log \frac{x_\nu+t}{x_\nu-t} 
   +\Big\|\frac{v}{U+i\lambda}\Big\|_{L^1(0,x)} + \alpha^2\|\phi\|_{L^1(0,x)}\,.
\end{displaymath}
Taking the supremum over $t\in(0,x]$ yields 
\begin{multline*} 
   \|\phi^\prime(t)\|_{L^\infty(0,x)}  \leq \log \frac{9}{4} (1+ C x_\nu)  \|\phi^\prime\|_{L^\infty(0,x)}  + 
   C  \frac{|\phi(x_\nu)|}{x_\nu} \, \log \frac{x_\nu+x}{x_\nu-x} 
   \\+\Big\|\frac{v}{U+i\lambda}\Big\|_{L^1(0,x)} + \alpha^2\|\phi\|_{L^1(0,x)}\,.
\end{multline*}
Hence, for sufficiently small $U(0)-\nu_1$, we obtain for all $\nu_1 < \nu <
U(0)$ and all  $x\in[0,x_\nu/2]$,
\begin{equation}
  \label{eq:207} 
  \|\phi^\prime\|_{L^\infty(0,x)}\leq C\left(  \left(\log
  \frac{x_\nu+x}{x_\nu-x}\right)\, \frac{|\phi(x_\nu)|}{x_\nu} +
  \alpha^2\|\phi\|_{L^1(0,x)} +\Big\|\frac{v}{U+i\lambda}\Big\|_{L^1(0,x)}\right) \,.  
\end{equation}
\paragraph{Estimate on $(x_\nu/2,x_\nu)$.}
 To obtain a bound for $\|\phi^\prime\|_{L^\infty(x_\nu/2,x)}$ for
$x\in(x_\nu/2,x_\nu)$ we write
\begin{displaymath}
   |\phi^\prime(x)|\leq |\phi^\prime(x_\nu/2)| + \Big|\int_{x_\nu/2}^x\phi^{\prime\prime}(t)\,dt\Big|\,.
\end{displaymath}
The first term can be estimated by using \eqref{eq:209}. 
To obtain a bound for the second term, we follow the same path as in
\eqref{eq:204}. We get
\begin{equation}\label{eq:208}
  \Big|\int_{x_\nu/2}^x\frac{U^{\prime\prime}[\phi-\phi(x_\nu)]}{U+i\lambda} \,dt\Big|  \leq
  (2+Cx_\nu) \int_{x_\nu. }^x\frac{|\phi(t)-\phi(x_\nu)|}{x_\nu^2-t^2} \,dt  
  \,.
\end{equation}
As in \eqref{eq:205} we can conclude that
  \begin{displaymath}
  \begin{array}{ll}
     (2+Cx_\nu) \int_{x_\nu/2}^x\frac{|\phi(t)-\phi(x_\nu)|}{x_\nu^2-t^2} \,dt
    &  \leq  (2+Cx_\nu)\|\phi^\prime\|_{L^\infty(x_\nu/2,x)} \log
     \frac{x+x_\nu} {3x_\nu/2}\\ &  \leq   (2+Cx_\nu) \log (4/3) \|\phi^\prime\|_{L^\infty(x_\nu/2,x)}  \,.
     \end{array}
  \end{displaymath}
Hence the coefficient of $\|\phi^\prime\|_{L^\infty(x_\nu/2,x)}$ is again smaller
  than $1$ by a suitable choice of $\nu_1$. 
 Repeating the above other steps then yields, for $x\geq x_\nu/2$,  
 \begin{equation*}
\|\phi^\prime\|_{L^\infty(x_\nu/2,x)}\leq C \Big(  \frac{1}{x_\nu} |\phi(x_\nu)| \, \log
\frac{x_\nu+x}{x_\nu-x} +
\alpha^2\|\phi\|_{L^1(0,x)} +\Big\|\frac{v}{U+i\lambda}\Big\|_{L^1(0,x)}   + |\phi^\prime(x_\nu/2)|\Big)\,.  
\end{equation*}
which, combined with \eqref{eq:207} for $x=x_\nu/2$, finally gives
\begin{equation*}
\|\phi^\prime\|_{L^\infty(x_\nu/2,x)}\leq \hat C \Big(  \frac{1}{x_\nu} |\phi(x_\nu)| \, \log
\frac{x_\nu+x}{x_\nu-x} +
\alpha^2\|\phi\|_{L^1(0,x)} +\Big\|\frac{v}{U+i\lambda}\Big\|_{L^1(0,x)} \Big)\,.  
\end{equation*}
Combining the above and  \eqref{eq:207} lead to the existence of $C>0$ such that 
\begin{equation}
\label{eq:209}
\|\phi^\prime\|_{L^\infty(0,x)}\leq  C \Big(  \frac{1}{x_\nu} |\phi(x_\nu)| \, \log
\frac{x_\nu+x}{x_\nu-x} +
\alpha^2\|\phi\|_{L^1(0,x)} +\Big\|\frac{v}{U+i\lambda}\Big\|_{L^1(0,x)} \Big)\,,\,\forall x\in (0,x_\nu)\,.  
\end{equation}

\paragraph{ Estimate of $\|\phi^\prime\|_{L^2(0,x_\nu)}$.} \strut
Observing that
\begin{equation}
  \int_0^{x_\nu}\log^2\frac{x_\nu+x}{x_\nu-x}\,dx \leq 
  x_\nu\int_0^1\log^2\frac{1+t}{1-t}\,dt\leq Cx_\nu \,, 
\end{equation}
we may conclude from \eqref{eq:209} by integrating over $(0,x_\nu)$,  that
\begin{equation*} 
\|\phi^\prime\|_{L^2(0,x_\nu)}\leq C\Big( x_\nu^{-1/2} |\phi(x_\nu)|
+\alpha^2x_\nu\|\phi\|_{L^2(0,x_\nu)} + x_\nu^{1/2}\Big\|\frac{v}{U+i\lambda}\Big\|_{L^1(0,x_\nu)}\Big)\,. 
\end{equation*}
Note that
\begin{equation}\label{eq:210}
  \Big\|\frac{v}{U+i\lambda}\Big\|_{L^1(0,x_\nu)}\leq (1-x_\nu)^{-\frac 12} N_2(v,\lambda)\,,
\end{equation}
which leads to
\begin{equation}
\label{eq:211} 
\|\phi^\prime\|_{L^2(0,x_\nu)}\leq C\big( x_\nu^{-1/2} |\phi(x_\nu)|
+\alpha^2x_\nu\|\phi\|_{L^2(0,x_\nu)} + x_\nu^{1/2} N_2(v,\lambda) \big)\,. 
\end{equation}

\paragraph{ Estimate of $\|\phi\|_{L^2(0,x_\nu)}$.} \strut \\
Set
\begin{displaymath}
  \phi=\varphi + \phi(x_\nu)  \,,
\end{displaymath} 
and recall from \eqref{eq:105} that for all $x\in(0,x_\nu)$ 
\begin{multline*}
  -\Big((U-\nu)^2\Big(\frac{\varphi}{U-\nu}\Big)^\prime\Big)^\prime+\alpha^2(U-\nu)\varphi 
 \\ = \frac{(U-\nu)v}{U+i\lambda}
  -\alpha^2\phi(x_\nu)\big((U-\nu)-U^{\prime\prime}\big)
  +i\mu\frac{U^{\prime\prime}\phi}{U+i\lambda} \,.
\end{multline*}
Taking the inner product, in $L^2(0,x_\nu)$, with $w$ defined as in
Equation \eqref{defneww} by $ w:=(U-\nu)^{-1}\varphi$, yields as in
\eqref{eq:106}
\begin{multline} 
\label{eq:212}
  \|(U-\nu)w^\prime\|_{L^2(0,x_\nu)}^2 + \alpha^2\|\varphi\|_{L^2(0,x_\nu)}^2=\\ \qquad =\Big\langle
  \varphi, \frac{v}{U+i\lambda}\Big\rangle_{L^2(0,x_\nu)}   - \langle w,\phi(x_\nu)U^{\prime\prime}\rangle_{L^2(0,x_\nu)} 
  \\ -  \alpha^2\phi(x_\nu)\langle\varphi,1\rangle_{L^2(0,x_\nu)}   +i\mu\Big\langle w,\frac{U^{\prime\prime}\phi}{U+i\lambda}\Big\rangle_{L^2(0,x_\nu)} \,.
\end{multline}

We now turn to estimate the various terms on the right-hand-side  of
\eqref{eq:212}. \\
For the second term in \eqref{eq:212} we use the fact that by Hardy's
inequality and \eqref{eq:186}  it holds that
\begin{equation}\label{eq:213}
  \|w\|_{L^1(0,x_\nu)} \leq C \,  \Big\|\frac{\phi-\phi(x_\nu)}{x-x_\nu}\Big\|_{L^2(0,x_\nu)}
  \Big\|\frac{1}{x+x_\nu}\Big\|_{L^2(0,x_\nu)}\leq \frac{\hat C}{x_\nu^{1/2}}\, \|\phi^\prime\|_{L^2(0,x_\nu)}\,.
\end{equation}
Consequently,
\begin{equation}
\label{eq:214}
  |\langle w,\phi(x_\nu)U^{\prime\prime}\rangle_{L^2(0,x_\nu)}|\leq C\frac{|\phi(x_\nu)|}{x_\nu^{1/2}}\|\phi^\prime\| _{L^2(0,x_\nu)}\,.
\end{equation}
For the third term in \eqref{eq:212},  it follows from that 
\begin{equation}
\label{eq:215}
  \alpha^2|\phi(x_\nu)\langle\varphi,1\rangle | \leq  C\alpha^2x_\nu^{1/2}\left(\|\phi\|_{L^2(0,x_\nu)}
  +x_\nu^{1/2}|\phi(x_\nu)|\right)|\phi(x_\nu)|\,.
\end{equation}
Finally, for the last term in \eqref{eq:212}, proceeding as in the proof of
\eqref{eq:199},  we obtain by using
\eqref{eq:213} and  Hardy's inequality 
\begin{equation*}
\begin{array}{l}
 \Big| \Big\langle w,\frac{U^{\prime\prime}\phi}{U+i\lambda}\Big\rangle_{L^2(0,x_\nu)}\Big| \\\qquad  \leq
 |\phi(x_\nu)|\,\|w\| _{L^1(0,x_\nu)}\Big\|
 \frac{U^{\prime\prime}}{U+i\lambda}\Big\|_{L^\infty(0,x_\nu)}
  +  \|w\|_{L^2(0,x_\nu)}
 \Big\|\frac{U^{\prime\prime}(\phi-\phi(x_\nu))}{U+i\lambda}\Big\|_{L^2(0,x_\nu)} \\ \qquad \leq
 C\Big[\frac{|\phi(x_\nu)| }{|\mu|x_\nu^{1/2}}\|\phi^\prime\|_{L^2(0,x_\nu)}
 +\frac{\|\phi^\prime\|_{L^2(0,x_\nu)}^2}{x_\nu^2}\Big]\,.
 \end{array}
\end{equation*}
Substituting the above, together with \eqref{eq:215} and \eqref{eq:214}
into \eqref{eq:212} yields
\begin{multline*}
 \alpha^2\|\varphi\|_{L^2(0,x_\nu)}^2 \leq C\Big[\frac{|\phi(x_\nu)|}{x_\nu^{1/2}}\|\phi^\prime\|_{L^2(0,x_\nu)}
 +\frac{|\mu|}{x_\nu^2}\|\phi^\prime\|_{L^2(0,x_\nu)}^2\Big] \\ \qquad + C\alpha^2(x_\nu^{1/2}\|\phi\|_{L^2(0,x_\nu)}
  +x_\nu|\phi(x_\nu)|)|\phi(x_\nu)|\\ + \Big|\Big\langle
  \varphi, \frac{v}{U+i\lambda}\Big\rangle_{L^2(0,x_\nu)}\Big| \qquad \qquad \qquad\,,
\end{multline*}
from which we  easily  conclude that
\begin{multline}
\label{eq:216}
  \alpha^2\|\phi\|_{L^2(0,x_\nu)}^2 \leq C\Big[\frac{ |\phi(x_\nu)|}{x_\nu^{1/2}}\|\phi^\prime\|_{L^2(0,x_\nu)}
 +\frac{|\mu|}{x_\nu^2}\|\phi^\prime\|_{L^2(0,x_\nu)}^2+ \\
+  \alpha^2x_\nu|\phi(x_\nu)|^2\Big] + \Big|\Big\langle
  \varphi, \frac{v}{U+i\lambda}\Big\rangle_{L^2(0,x_\nu)}\Big| \,.    
\end{multline}
Substituting \eqref{eq:216} into \eqref{eq:211} we obtain
\begin{multline*}
  \|\phi^\prime\|_{L^2(0,x_\nu)}\leq C\Big(\frac{|\phi(x_\nu)|}{x_\nu^{1/2}}
+ \alpha^2 x_\nu^{3/2}|\phi(x_\nu)| + \alpha|\mu|^{1/2}\|\phi^\prime\|_{L^2(0,x_\nu)} \\
+ x_\nu^{1/2}N_2(v,\lambda)+\alpha x_\nu\Big|\Big\langle
  \varphi, \frac{v}{U+i\lambda}\Big\rangle_{L^2(0,x_\nu)}\Big|^{1/2}\Big) \,. 
\end{multline*}
As $|\alpha|\leq C_0/x_\nu\,$, we obtain for sufficiently large  $\kappa_0$ (implying
that both $|\mu|^{1/2} /x_\nu$ and $ \alpha|\mu|^{1/2}$ are sufficiently small) 
\begin{equation}
\label{eq:217}
  \|\phi^\prime\|_{L^2(0,x_\nu)}\leq \widehat C\Big(\frac{|\phi(x_\nu)|}{x_\nu^{1/2}}+
  x_\nu^{1/2}N_2(v,\lambda)+\Big|\Big\langle
  \varphi, \frac{v}{U+i\lambda}\Big\rangle_{L^2(0,x_\nu)}\Big|^{1/2}
  \Big)\,. 
\end{equation}
Note that, by \eqref{eq:210}, 
\begin{displaymath} 
  \Big|\Big\langle  \varphi, \frac{v}{U+i\lambda}\Big\rangle_{L^2(0,x_\nu)}\Big|\leq
  \|\phi-\phi(x_\nu)\|_{L^\infty(0,x_\nu)}\,\Big\|\frac{v}{U+i\lambda}\Big\|_{L^1(0,x_\nu)}\leq 2x_\nu^{1/2}\|\phi^\prime\|_{L^2(0,x_\nu)}N_2(v,\lambda)
\end{displaymath}
Combining the above with \eqref{eq:217} and \eqref{eq:179} yields
\begin{displaymath}
 \|\phi^\prime\|_{L^2(0,x_\nu)}\leq 
 C\Big[\frac{\mu^{1/2}}{x_\nu}\|\phi^\prime\|_2+x_\nu^{1/2}N_2(v,\lambda)+ \Big|\Big\langle
  \phi , \frac{v}{U+i\lambda}\Big\rangle\Big|^{1/2}
\Big] \,. 
\end{displaymath}
For sufficiently large $\kappa_0$ we easily obtain \eqref{eq:201}.\\

{\em Step 4:  We prove that, 
for any $C_0>0$, there exist $C>0$, $\mu_0>0$,  $\kappa_0>0$, and $\nu_1>0$
such that, for $\nu_1 < \nu < U(0)-\kappa_0|\mu|$, $|\alpha| \leq C_0/x_\nu$,   and $|\mu|\leq \mu_0$\,, we have 
\begin{equation}
\label{eq:218}
\|\phi^\prime\|_2\leq  
C\Big(\Big|\Big\langle\phi,\frac{v}{U+i\lambda}\Big\rangle\Big|^{1/2}+
x_\nu^{1/2}N_2(v,\lambda)\Big) \,,
\end{equation}
for any pair $(\phi,v)$ satisfying \eqref{eq:47}.}\\

Observing  that $U^{\prime\prime} (U-\nu)> 0$ on $(x_\nu,1))$ implies
\begin{displaymath}
   \Re\Big\langle U^{\prime\prime}\phi,\frac{\phi}{U+i\lambda}\Big\rangle_{L^2(x_\nu,1)}=
   \int_{x_\nu}^1\frac{|\phi|^2U^{\prime\prime}(U - \nu)}{|U+i\lambda|^2}\, dx  \geq 0\,,
\end{displaymath}
and hence we may conclude from \eqref{eq:196},   that
\begin{equation}
\label{eq:219}
   \|\phi^\prime\|_2^2 + \alpha^2\|\phi\|_2^2\leq    - \Re\Big\langle
   U^{\prime\prime}\phi,\frac{\phi}{U+i\lambda}\Big\rangle_{L^2(0,x_\nu)} +\Re\Big\langle\phi,\frac{v}{U+i\lambda}\Big\rangle\,. 
\end{equation}
To estimate $|\Re\Big\langle U^{\prime\prime}\phi,\frac{\phi}{U+i\lambda}\Big\rangle_{L^2(0,x_\nu)}
|$, we now proceed as in the proof of \eqref{eq:198}-\eqref{eq:199}
to obtain
   \begin{multline}\label{eq:220}
\Big|  \Re\Big\langle   U^{\prime\prime}\phi,\frac{\phi}{U+i\lambda}\Big\rangle_{L^2(0,x_\nu)}\Big|
  \leq  \frac{C}{x_\nu} |\phi(x_\nu)|^2 + \\C \left(  x_\nu^{-1}\|\phi\|_{L^2(0,x_\nu)}+ x_\nu^{-1/2}|\phi(x_\nu)| \right)  \,\|\phi^\prime\|_{L^2(0,x_\nu)}\,.
\end{multline}
Using Poincar\'e's inequality applied in $(0,x_\nu)$ to $(\phi-\phi(x_\nu))$ yields
 \begin{multline}\label{eq:221}
\Big|  \Re\Big\langle   U^{\prime\prime}\phi,\frac{\phi}{U+i\lambda}\Big\rangle_{L^2(0,x_\nu)}\Big|
  \leq  \frac{C}{x_\nu} |\phi(x_\nu)|^2 + \\ + \hat C \left(\|\phi^\prime\|_{L^2(0,x_\nu)}+ x_\nu^{-1/2}|\phi(x_\nu)| \right)  \,\|\phi^\prime\|_{L^2(0,x_\nu)}\,.
\end{multline}
Substituting \eqref{eq:221} together with  \ \eqref{eq:201} and \eqref{eq:179}
into \eqref{eq:219} yields
\begin{displaymath}
    \|\phi^\prime\|_2^2 + \alpha^2\|\phi\|_2^2\leq C\Big[
    \frac{|\mu|}{x_\nu^2}\|\phi^\prime\|_2^2+\Big|
    \Re\Big\langle\phi,\frac{v}{U+i\lambda}\Big\rangle| +x_\nu N_2(v,\lambda)^2\Big]\,.
\end{displaymath}
For sufficiently large $\kappa_0$ we readily obtain \eqref{eq:218}. \\

{\em Step 5: We prove \eqref{eq:178}. }\\

We begin by deriving two conclusions of  \eqref{eq:218} and \eqref{eq:194}
under the assumptions of the proposition.  Since  by \eqref{eq:158} 
\begin{displaymath}
     \Big|\Big\langle \phi,
  \frac{v}{U+i\lambda}\Big\rangle\Big|\leq \|\phi^\prime\|_2N_2(v,\lambda)\,,
\end{displaymath}
where $N_2$ is given by \eqref{eq:200}, we obtain by \eqref{eq:218} (for $|\alpha| \geq \frac{C_0}{x_\nu}$)
and \eqref{eq:194}   (for $|\alpha| \leq \frac{C_0}{x_\nu}$) that, under the assumption of the proposition,
\begin{equation}
\label{eq:222}
  \|\phi^\prime\|_2 \leq CN_2(v,\lambda)\,,
\end{equation}
which combined with \eqref{eq:179} yields, for bounded $|\mu|/x_\nu^2$,
\begin{equation}
  \label{eq:223}
 |\phi(x_\nu)|\leq Cx_\nu^{1/2}N_2(v,\lambda) \,.
\end{equation}

\paragraph{Proof of (\ref{eq:178}a).}

We begin by obtaining a bound on $\|(U+i\lambda)^{-1}\|_q$. 
By \eqref{eq:186}  we have  for $q>1$ and $\nu_1 < \nu < U(0)$
\begin{equation*}
   \Big\|\frac{1}{U+i\lambda}\Big\|_{L^q(0,1)}^q  \leq\frac{C}{|\mu|^{q-1/2}}\int_{\R_+} \frac{ds}{[(s^2-a^2)^2
     +1]^{q/2}}\,,
\end{equation*}
where $a=x_\nu|\mu|^{-1/2}$. \\
 We estimate the integral on the
right-hand-side in the following manner
\begin{multline}
\label{eq:224}
  \int_{\R_+} \frac{ds}{[(s^2-a^2)^2+1]^{q/2}}\leq \int_{\R_+}
  \frac{ds}{[a^2(s-a)^2+1]^{q/2}}\leq \\ \leq \int_\R \frac{d\tau}{[a^2\tau^2+1]^{q/2}}
  \leq \frac{1}{a}\int_\R \frac{dt}{[t^2+1]^{q/2}}\leq \frac{C}{a}\,.
\end{multline}
It follows that   for $q>1$ and $\nu_1 \leq \nu < U(0)$
\begin{equation}
\label{eq:225}
\Big\|\frac{1}{U+i\lambda}\Big\|_{L^q(0,1)}^q \leq\frac{C}{x_\nu |\mu|^{q-1}}\,.
    \end{equation}
Since for $\nu <\nu_1$ we may establish \eqref{eq:225}, using
\eqref{eq:95}, as in \cite{almog2019stability} we may conclude that
\eqref{eq:225} holds for any $0\leq \nu < U(0)$. 

We continue by estimating $\langle\phi,(U+i\lambda)^{-1}v\rangle$. To thus end we write
\begin{equation}
\label{eq:226}
   \Big|\Big\langle\phi,\frac{v}{U+i\lambda} \Big\rangle\Big| \leq
   |\phi(x_\nu)|\, \Big\|\frac{v}{U+i\lambda} \Big\|_1
   +\Big\|\frac{\phi-\phi(x_\nu)}{x-x_\nu} \Big\|_2 \Big\|\frac{(x-x_\nu)\,v }{U+i\lambda} \Big\|_2
\end{equation}
Suppose first that $v\in L^p(0,1)$ for some $p\in[2,\infty)$. Then,
\begin{displaymath}
   \Big\|\frac{v}{U+i\lambda}\Big\|_1\leq \|v\|_p \Big\|\frac{1}{U+i\lambda}\Big\|_q\,,
\end{displaymath}
where $q=p/(p-1)$. \\
Consequently, by \eqref{eq:225}
\begin{equation}
\label{eq:227}
  \Big\|\frac{v}{U+i\lambda}\Big\|_1\leq \frac{C }{|\mu|^{\frac{1}{p}} x_\nu^{1-\frac{1}{p} }} \|v\|_p \,.
\end{equation}
Next, we estimate the second term on the right-hand-side  of
\eqref{eq:226}. Consider first the case $p=2$. Here we write with the
aid of \eqref{eq:186}
\begin{equation}
\label{eq:228}
  \Big\|\frac{(x-x_\nu) v }{U+i\lambda} \Big\|_2\leq \|v\|_2
  \Big\|\frac{x-x_\nu}{U+i\lambda} \Big\|_\infty \leq \frac{C}{x_\nu} \|v\|_2\,.
\end{equation}
For $p>2$ we have
\begin{displaymath}
  \Big\|\frac{(x-x_\nu)v}{U+i\lambda} \Big\|_2\leq \|v\|_p\Big\|\frac{x-x_\nu}{U+i\lambda}\Big\|_{\tilde{q}}\,,
\end{displaymath}
where $\tilde{q}=2p/(p-2)$\,.\\
 As above we write 
\begin{equation}\label{eq:229}
  \Big\|\frac{x-x_\nu}{U+i\lambda}\Big\|_{\tilde{q}}^{\tilde{q}}\leq C 
  \int_0^1
  \frac{dx}{[x+x_\nu]^{\tilde{q}}} \leq \frac{\hat C}{x_\nu^{\tilde{q}-1}} \,.
\end{equation}
Consequently,
\begin{equation}
\label{eq:230}
 \Big\|\frac{(x-x_\nu)v }{U+i\lambda} \Big\|_2\leq
 \frac{C}{x_\nu^{\frac{1}{2}+\frac{1}{p}}}\|v\|_p \,,
\end{equation}
which is in accordance with \eqref{eq:228} for $p=2$. \\
Using \eqref{eq:226} once again, we deduce from \eqref{eq:227} and
\eqref{eq:230} that
\begin{displaymath}
   \Big|\Big\langle\phi,\frac{v}{U+i\lambda} \Big\rangle\Big| \leq C\, \left( \frac{1}{|\mu|^{\frac{1}{p}}x_\nu^{1-\frac{1}{p}}}  
   |\phi(x_\nu)|\, 
   +  \frac{1}{x_\nu^{\frac{1}{2}+\frac{1}{p}}} \Big\|\frac{\phi-\phi(x_\nu)}{x-x_\nu} \Big\|_2 \right)\,  \|v\|_p 
\end{displaymath}
 By Hardy's inequality we then have
\begin{equation}
  \label{eq:231}
   \Big|\Big\langle\phi,\frac{v}{U+i\lambda} \Big\rangle\Big| \leq C\, \left(  \frac{1}{|\mu|^{\frac{1}{p}}x_\nu^{1-\frac{1}{p}}}  
   |\phi(x_\nu)|\, 
   +  \frac{1}{x_\nu^{\frac{1}{2}+\frac{1}{p}}} \|\phi^\prime\|_2 \right)\,  \|v\|_p 
\end{equation}
By \eqref{eq:200} and \eqref{eq:227}, we have that
\begin{equation}\label{eq:232}
  N_2(v,\lambda)\leq\Big\|\frac{v}{U+i\lambda}\Big\|_1\leq \frac{C}{|\mu|^{\frac{1}{p}}x_\nu^{1-\frac{1}{p}}}   \|v\|_p \,.
\end{equation}
The above combined with \eqref{eq:222} yields
\begin{equation}
\label{eq:233}
  \|\phi^\prime\|_2 \leq  \frac{C}{|\mu|^{\frac{1}{p}}x_\nu^{1-\frac{1}{p}}}   \|v\|_p \,,
\end{equation}
which is weaker than (\ref{eq:178}a). \\
We obtain a better estimate in
the following manner.
By \eqref{eq:231}, \eqref{eq:232} and \eqref{eq:218} it holds that
\begin{displaymath}
  \|\phi^\prime\|_2\leq C\, \left( \bigg[\frac{1}{|\mu|^{\frac{1}{2p}}x_\nu^{\frac{1}{2}-\frac{1}{2p}}}  
   |\phi(x_\nu)|^{1/2}\, 
   +  \frac{1}{x_\nu^{\frac{1}{4}+\frac{1}{2p}}} \|\phi^\prime\|_2^{1/2}  \bigg]\|v\|_p^{1/2} + |\mu|^{-\frac{1}{p}}x_\nu^{-\frac{1}{2}+\frac{1}{p}} \|v\|_p \right)\,,
\end{displaymath}
from which we get
\begin{displaymath}
\begin{array}{ll}
  \|\phi^\prime\|_2 & \leq      C \frac{1}{x_\nu^{\frac{1}{4}+\frac{1}{2p}}
  }\,\|\phi^\prime\|_2^\frac 12  \|v\|_p^{\frac 12}   +   C\,
  \left(x_\nu^{-1/2} 
   |\phi(x_\nu)|\, 
   + |\mu|^{-\frac{1}{p}}x_\nu^{-\frac{1}{2}+\frac{1}{p}}  \|v\|_p \right)\\
   & \leq   \hat  C  |\mu|^{\frac{1}{p}} \, x_\nu^{-\frac 2p}  \|\phi^\prime\|_2 + \hat C   \left(x_\nu^{-1/2} 
   |\phi(x_\nu)|\, 
   + |\mu|^{-\frac{1}{p}}x_\nu^{-\frac{1}{2}+\frac{1}{p}}  \|v\|_p \right)\,.
   \end{array}
\end{displaymath}
 Using the fact that $|\mu|^{1/2}x_\nu^{-1}$
  can be assumed to be small for a suitable choice of $\kappa_0$, we can then conclude that
\begin{equation}
\label{eq:234}
  \|\phi^\prime\|_2\leq C\big(x_\nu^{-1/2}|\phi(x_\nu)|+  |\mu|^{-\frac{1}{p}}x_\nu^{-\frac{1}{2}+\frac{1}{p}}  \|v\|_p\big)\,.
\end{equation}
By \eqref{eq:179}  
and \eqref{eq:231} it holds that
\begin{displaymath}
  |\phi(x_\nu)|^2\leq Cx_\nu\bigg[\frac{\mu}{x_\nu^2}\|\phi^\prime\|_2^2+ \bigg(|\mu|^{-\frac{1}{p}}x_\nu^{-1+\frac{1}{p}} 
   |\phi(x_\nu)|\, 
   +  \frac{1}{x_\nu^{\frac{1}{2}+\frac{1}{p}}} \|\phi^\prime\|_2 \bigg)\,  \|v\|_p \bigg] \,,
\end{displaymath}
and hence we obtain that for any $\delta>0$ there exist $C_\delta>0$ and
$\kappa_0(\delta)$ such that, under the conditions of the proposition with
$\kappa_0=\kappa_0(\delta)$,
\begin{equation}\label{eq:235}
   |\phi(x_\nu)|\leq x_\nu^{1/2}\delta\|\phi^\prime\|_2+C_\delta\frac{x_\nu^{1/p}}{|\mu|^{1/p}}\|v\|_p \,.
\end{equation}
Substituting \eqref{eq:235} into \eqref{eq:234} yields
\begin{displaymath}
    \|\phi^\prime\|_2\leq C  |\mu|^{-\frac{1}{p}}x_\nu^{-\frac{1}{2}+\frac{1}{p}} \|v\|_p\,.
\end{displaymath}
As $\phi(1)=0$, we may use Poincar\'e's inequality to establish
(\ref{eq:178}a).

\paragraph{Proof of (\ref{eq:178}b).}

Suppose now that $v\in L^\infty(0,1)$. Then,
\begin{displaymath}
  \Big\|\frac{v}{U+i\lambda}\Big\|_1\leq \|v\|_\infty \, \Big\|\frac{1}{U+i\lambda}\Big\|_1\,.
\end{displaymath}
Then, we may use \eqref{eq:186} to obtain
\begin{equation}
\label{eq:236}
  \Big\|\frac{1}{U+i\lambda}\Big\|_1\leq C\int_0^1 \frac{dx}{[(x^2-x_\nu^2)^2
     +\mu^2]^{1/2}}\leq\frac{C}{|\mu|^{1/2}}\int_{\R_+} \frac{ds}{[(s^2-a^2)^2
     +1]^{1/2}}\,,
\end{equation}
Then, we write, using the fact that for sufficiently large $\kappa_0$ we
have $a \geq 2$, 
\begin{multline}
\label{eq:237}
  \int_{\R_+} \frac{ds}{[(s^2-a^2)^2
     +1]^{1/2}}\leq \int_0^{2a}\frac{ds}{[a^2(s-a)^2 
     +1]^{1/2}} \\ + \int_{2a}^\infty \frac{ds}{[(s-a)^4 
     +1]^{1/2}}\leq C \left(\frac{\log a }{a} \right ) \,.
\end{multline}
Combining the above yields for $\nu_1<\nu < U(0)$
\begin{equation}
\label{eq:238}
\Big\|\frac{1}{U+i\lambda}\Big\|_1 \leq C\frac{\log\frac{x_\nu}{|\mu|^{1/2}}}{x_\nu}\,.  
\end{equation} 
Note that by \eqref{eq:95} the above estimate holds for $\nu \leq \nu_1$ as
well (see \cite{almog2019stability}).  From \eqref{eq:238}, we deduce
immediately
\begin{equation}
\label{eq:239}
   x_\nu^{1/2}\Big\|\frac{v}{U+i\lambda}\Big\|_1\leq C\frac{\log\frac{x_\nu}{|\mu|^{1/2}}}{x_\nu^{1/2}}\|v\|_\infty\,.
\end{equation}

Next, we estimate the second term on the right-hand-side  of
\eqref{eq:226} in the case $p=\infty$. Using \eqref{eq:229} we obtain that
\begin{equation}\label{eq:240}
  \Big\|\frac{(x-x_\nu) v }{U+i\lambda} \Big\|_2\leq \|v\|_\infty 
  \Big\|\frac{x-x_\nu}{U+i\lambda} \Big\|_2 \leq \frac{C}{x_\nu^{1/2}} \|v\|_\infty
  \,.
\end{equation}
Combining \eqref{eq:240} with \eqref{eq:239}, \eqref{eq:194},
\eqref{eq:218}, and \eqref{eq:201} yields (\ref{eq:178}b).
Note that by
\eqref{eq:179} and \eqref{eq:226} we obtain that
\begin{equation}
  \label{eq:241}
|\phi(x_\nu)|\leq C\log\Big(\frac{x_\nu}{|\mu|^{1/2}}\Big)\|v\|_\infty
\end{equation}
\end{proof}
\subsection{The case $ U(0)-\kappa_0|\Re \lambda| \leq \Im \lambda \leq U(0)+ \kappa_0 |\Re \lambda|
  $.}
\label{sec:case--quadratic} 
 In the following, we consider the case where $\nu$ is very close to
$U(0)$. Here we need to address the quadratic behavior of $U-U(0)$
near $x=0$.  This case deserves special attention whenever $|\nu-U(0)|\lesssim|\mu|$.
  
\begin{proposition}
\label{prop:quadratic} 
 Let $p\in (2,+\infty]$, $\kappa_0>0$ and $U\in C^3([0,1])$ satisfy
\eqref{eq:10}. There exist $\mu_0>0$ and $C >0$ such that, for any
$\alpha\geq 0$ and any $\lambda$ for which $ 0<|\mu|\leq\mu_0 $ and $ U(0)-\kappa_0|\mu|\leq
\nu \leq U(0)+\kappa_0|\mu|$, we have, for every pair $(\phi,v)\in D(\A_{\lambda,\alpha})\times 
L^\infty(0,1)$ satisfying \eqref{eq:47} 
    \begin{equation}
\label{eq:242}
|\mu|^{\frac{1}{2p}+\frac{1}{4}}\|\phi\|_{1,2}
\leq C\|v\|_p  \,.
\end{equation}
\end{proposition}
\begin{proof}
For $\nu\leq U(0)$ we choose $x_\nu\in[0,1)$ so that
$U(x_\nu)=\nu$. In the case $\nu>U(0)$ we set $x_\nu=0$ and proceed in a
similar manner. Obviously, the assumptions made on $U$ and $\lambda$ imply
that there exists $C>0$ such that  
\begin{equation}\label{eq:243}
x_\nu<C|\mu|^{1/2} \mbox{ for all }  0< |\mu|\leq 1\,.
\end{equation}
 
  {\em Step 1:}  {\em We prove that there exist $C>0$ and $\mu_0>0$  such that,
for all $\lambda$ such that  $0<|\mu|\leq \mu_0$ and   $ U(0)-\kappa_0|\mu|\leq \nu \leq  U(0)+\kappa_0|\mu|$ it holds that 
 \begin{equation}
\label{eq:244}
  |\phi(x_\nu)|^2\leq C|\mu|^{1/2}\Big[\|\phi^\prime\|_2^2+ \Big|\Big\langle\phi,\frac{v}{U+i\lambda}
  \Big\rangle\Big|\Big] \,. 
 \end{equation}
  for all pairs $(\phi,v)\in D(\A_{\lambda,\alpha})\times  L^\infty(0,1)$ satisfying \eqref{eq:47}.}\\

We note that  \eqref{eq:180} can be rewritten in the form   
\begin{equation}
\label{eq:245}
\begin{array}{ll}
\frac {|\mu|}{2} \,  |\phi(x_\nu)|^2 \,
   \int _0^1\frac{ |U^{\prime\prime}|}{(U-\nu)^2+\mu^2}\,dx &  \leq \Big|  \Im\Big\langle\phi,\frac{v}{U-\nu+i\mu}\Big\rangle\Big| \\ & \qquad +  |\mu| \,
   \Big\langle\frac{|U^{\prime\prime}|}{(U-\nu)^2+\mu^2} |\phi(x)-\phi(x_\nu)|^2 \Big\rangle\,.
   \end{array}
\end{equation}
 Since 
\begin{displaymath}
  (U-\nu)^2\leq C(x^2+|\nu-U(0)|)^2\leq C(x^2+\kappa_0|\mu|)^2\leq\hat{C}(x^4+|\mu|^2)\,,
\end{displaymath}
we obtain that 
\begin{displaymath}
  \int_0^1\frac{|U^{\prime\prime}|}{(U-\nu)^2+\mu^2}\,dx \geq
  \frac{1}{\hat C} \int_0^1\frac{dx}{x^4+\mu^2} \,.
\end{displaymath}
Using the substitution $x=|\mu|^{1/2}\xi$ yields
\begin{displaymath}
  \int_0^1\frac{dx}{x^4+\mu^2}=
  |\mu|^{-3/2}\int_0^{|\mu|^{-1/2}}\frac{d\xi}{\xi^4+1} =
  |\mu|^{-3/2}\Big[\int_0^\infty\frac{d\xi}{\xi^4+1}-\int_{|\mu|^{-1/2}}^\infty\frac{d\xi}{\xi^4+1}\Big]  \,.
\end{displaymath}
As
\begin{displaymath}
  \int_{|\mu|^{-1/2}}^\infty\frac{d\xi}{\xi^4+1}\leq C|\mu|^{3/2} \,,
\end{displaymath}
we obtain the existence of $\mu_0>0$ and $\hat{C}$ such that, under the conditions of this step
\begin{equation}
 \label{eq:246}
\int_0^1\frac{|U^{\prime\prime}|}{(U-\nu)^2+\mu^2}\,dx \geq \frac{1}{\hat C|\mu|^{3/2}} \,.
\end{equation}
 By \eqref{eq:193}
we have that
\begin{displaymath}
   \Big\langle\frac{|U^{\prime\prime}|}{(U-\nu)^2+\mu^2},|\phi(x)-\phi(x_\nu)|^2 \Big\rangle \leq
  \frac{C}{|\mu|}\Big\|\frac{\phi-\phi(x_\nu)}{|U-\nu|^{1/2}}\Big\|_2^2 
\end{displaymath}
By \eqref{eq:186} we have for all $\nu_1<\nu<U(0)$ for some positive
$\nu_1>0$ that (note for sufficiently small $\mu_0$ we clearly have $\nu>U(0)-\kappa_0|\mu|>\nu_1$)
\begin{equation}\label{eq:247}
  |U-\nu| \geq \frac{1}{C}(x^2-x_\nu^2) \geq   \frac{1}{C}(x-x_\nu)^2 \,, 
\end{equation}
which remains valid also for $\nu\geq U(0)$ given that 
\begin{equation}\label{eq:248}
|U-\nu|\geq|U-U(0)|\geq \frac{1}{C} x^2=\frac{1}{C} (x-x_\nu)^2 \,.
\end{equation}
Hence, by Hardy's inequality \eqref{eq:96},
\begin{displaymath}
   \Big\langle\frac{|U^{\prime\prime}|}{(U-\nu)^2+\mu^2},|\phi(x)-\phi(x_\nu)| \Big\rangle \leq
   \frac{C}{|\mu|}\Big\|\frac{\phi(x)-\phi(x_\nu)}{x-x_\nu}\Big\|_2^2  \leq
     \frac{C}{|\mu|}\|\phi^\prime\|_2^2
\end{displaymath}
Combining the above with \eqref{eq:246} and \eqref{eq:245} yields
\eqref{eq:244}. \\

{\em Step 2:} {\it We prove that for any $\kappa_1>0$ there exists positive $C$, and
  $\mu_0$ such that, for  all  $\nu>U(0)-\kappa_1|\mu|$  and $|\mu|\leq
  \mu_0$ it holds that
\begin{equation}
 \label{eq:249}
\|\phi^\prime\|_2\leq  C\Big[\mu^{1/4}\Big\|\frac{v}{U+i\lambda}\Big\|_1 +
  \Big\|\frac{(x-x_\nu)v}{U+i\lambda}\Big\|_2 
  \Big]\,.
\end{equation}
holds for any pair $(\phi,v)\in D(\A_{\lambda,\alpha})\times L^\infty(0,1) $ satisfying  \eqref{eq:47}. }\\

We begin by restating \eqref{eq:219}
\begin{equation}\label{eq:250}
   \|\phi^\prime\|_2^2 + \alpha^2\|\phi\|_2^2\leq    - \Re\Big\langle
   U^{\prime\prime}\phi,\frac{\phi}{U+i\lambda}\Big\rangle_{L^2(0,x_\nu)} +\Re\Big\langle\phi,\frac{v}{U+i\lambda}\Big\rangle\,. 
\end{equation}
Then, we write
\begin{displaymath}
  \Re \Big\langle
  \phi,\frac{U^{\prime\prime}\phi}{U+i\lambda}\Big\rangle_{L^2(0,x_\nu)}= \Re \Big\langle
  \phi(x_\nu) ,\frac{U^{\prime\prime}\phi}{U+i\lambda}\Big\rangle_{L^2(0,x_\nu)}+ \Re \Big\langle
  \phi-\phi(x_\nu) ,\frac{U^{\prime\prime}\phi}{U+i\lambda}\Big\rangle_{L^2(0,x_\nu)}\,.
\end{displaymath}
For the first term on the right-hand-side we use 
\eqref{eq:244}  to obtain 
\begin{displaymath}
  \Big|\Big\langle \phi(x_\nu)
  ,\frac{U^{\prime\prime}\phi}{U+i\lambda}\Big\rangle_{L^2(0,x_\nu)}\Big|\leq C|\mu|^{1/4}x_\nu^{1/2}\Big[\|\phi^\prime\|_2+ \Big|\Big\langle\phi,\frac{v}{U+i\lambda}
  \Big\rangle\Big|^{1/2}\Big]  \Big\|\frac{\phi}{U+i\lambda}\Big\|_2\,.
\end{displaymath}
Using \eqref{eq:90} and the fact that $U^{\prime\prime} < 0$, we obtain that
\begin{displaymath}
  \Big\|\frac{\phi}{U+i\lambda}\Big\|_2^2 = \int_0^1\frac{ |\phi|^2}{(U-\nu)^2
    +\mu^2 } dx \leq C \, \int_0^1\frac{ - U^{\prime\prime} |\phi|^2}{(U-\nu)^2 +\mu^2 } dx
  \leq \frac{C}{\mu}\, \Big|\Big\langle\phi,\frac{v}{U+i\lambda}   \Big\rangle\Big|\,. 
\end{displaymath}
Hence it holds that
\begin{equation}
\label{eq:251}
  \Big\|\frac{\phi}{U+i\lambda}\Big\|_2\leq C|\mu|^{-1/2}\Big|\Big\langle\phi,\frac{v}{U+i\lambda}
  \Big\rangle\Big|^{1/2} \,.
\end{equation}
and we can conclude that
\begin{equation}
\label{eq:252}
   \Big|\Big\langle \phi(x_\nu)
  ,\frac{U^{\prime\prime}\phi}{U+i\lambda}\Big\rangle_{L^2(0,x_\nu)}\Big| \leq C\Big[\|\phi^\prime\|_2+ \Big|\Big\langle\phi,\frac{v}{U+i\lambda}
  \Big\rangle\Big|^{1/2}\Big] \Big|\Big\langle\phi,\frac{v}{U+i\lambda}
  \Big\rangle\Big|^{1/2} \,. 
\end{equation}
For the second term on the right-hand-side we use \eqref{eq:251}, 
\eqref{hardyw1a}  and \eqref{eq:243} to obtain
\begin{displaymath}
\begin{array}{ll}
  \Big|\Big\langle  \phi-\phi(x_\nu)
  ,\frac{U^{\prime\prime}\phi}{U+i\lambda}\Big\rangle_{L^2(0,x_\nu)}\Big|&  =\Big|\Big\langle
  \frac{\phi-\phi(x_\nu)}{x-x_\nu}  , \frac{(x-x_\nu)
    U^{\prime\prime}\phi}{U+i\lambda}\Big\rangle_{L^2(0,x_\nu)}\Big|\\ 
&   \leq
  C \, x_\nu\|\phi^\prime\|_2  \Big\|\frac{\phi}{U+i\lambda}\Big\|_2\\
  & \leq \hat C\, \|\phi^\prime\|_2\,\Big|\Big\langle\phi,\frac{v}{U+i\lambda}
  \Big\rangle\Big|^{1/2}\,.
  \end{array}
\end{displaymath}
Hence,
\begin{equation}\label{eq:253}
  \Big|\Big\langle  \phi-\phi(x_\nu) ,\frac{U^{\prime\prime}\phi}{U+i\lambda}\Big\rangle_{L^2(0,x_\nu)}\Big| \leq \hat C\, \|\phi^\prime\|_2\,\Big|\Big\langle\phi,\frac{v}{U+i\lambda}
  \Big\rangle\Big|^{1/2}\,.
\end{equation}
Substituting \eqref{eq:253}  together with \eqref{eq:252} into
\eqref{eq:250}  yields
\begin{equation}
\label{eq:254}
\|\phi^\prime\|_2^2\leq  C\Big|\Big\langle\phi,\frac{v}{U+i\lambda}
  \Big\rangle\Big|\,.
\end{equation} 
Note that by combining
\eqref{eq:254} with \eqref{eq:244} we can also conclude that  
 \begin{equation}
\label{eq:255}
  |\phi(x_\nu)|^2\leq C|\mu|^{1/2}\Big|\Big\langle\phi,\frac{v}{U+i\lambda}
  \Big\rangle\Big| \,. 
 \end{equation}
To prove \eqref{eq:249} we now write, with the aid of Hardy's
inequality 
\begin{displaymath}
\begin{array}{ll}
  \Big|\Big\langle\phi,\frac{v}{U+i\lambda} \Big\rangle\Big|& 
   \leq   \Big|\Big\langle\phi(x_\nu),\frac{v}{U+i\lambda}
  \Big\rangle +  \Big|\Big\langle\phi-\phi(x_\nu),\frac{v}{U+i\lambda}
  \Big\rangle\Big|\\ & \leq |\phi(x_\nu)|\,  \Big\|\frac{v}{U+i\lambda}\Big\|_1 +
  \|\phi^\prime\|_2\,\Big\|\frac{(x-x_\nu)v}{U+i\lambda}\Big\|_2  \,.
  \end{array}
\end{displaymath}
Combining the above with \eqref{eq:255} gives
\begin{equation}\label{eq:256}
 \Big|\Big\langle\phi,\frac{v}{U+i\lambda} \Big\rangle\Big| \leq C \left(  |\mu|^\frac 12   \Big\|\frac{v}{U+i\lambda}\Big\|_1^2  + 
  \|\phi^\prime\|_2\,\Big\|\frac{(x-x_\nu)v}{U+i\lambda}\Big\|_2 \right)\,.
\end{equation}
Then,  \eqref{eq:256} 
and \eqref{eq:254} imply \eqref{eq:249}.\\

{\em Step 3: We estimate $ \Big\|\frac{1}{U+i\lambda}\Big\|_q$ for $q\geq1$ and
$\Big\|\frac{x-x_\nu}{U+i\lambda}\Big\|_q$  for $q\geq 2$ for  $\nu \in (U(0)-\kappa_0|\mu|, U(0) + \kappa_0
|\mu|) $.\\}

We consider two separate cases by splitting 
$(U(0)-\kappa_0|\mu|, U(0)+\kappa_0|\mu| )$ into two subintervals.

\paragraph{The case $\nu \in   (U(0)-\kappa_0|\mu|, U(0) - \delta |\mu|) $, $\delta\in (0,\kappa_0)$.}~\\
In this case, observing that $x_\nu \geq C^{-1} |\delta \mu|^\frac 12$\, we
use \eqref{eq:225} to establish that for any $\delta>0$ there exists $C$
such that for all $\delta|\mu|<U(0)-\nu<\kappa_0|\mu|$ it holds for all $q>1$
\begin{equation}
\label{eq:257}
   \Big\|\frac{1}{U+i\lambda}\Big\|_q^q\leq \frac{C}{|\delta|^{\frac{q}{2}}\, |\mu|^{q-1/2}}\,.
\end{equation}
In a similar manner we obtain from \eqref{eq:236} and \eqref{eq:237}
(note that $a=x_\nu|\mu|^{-1/2}\geq C^{-1} \delta^\frac 12$ in the present
regime of $\nu$ values) that there exists $C_\delta>0$ such that
\begin{equation}
\label{eq:258}
   \Big\|\frac{1}{U+i\lambda}\Big\|_1\leq \frac{C_\delta}{|\mu|^{1/2}}\,.
\end{equation}
Finally, we use \eqref{eq:229} and the fact that $x_\nu \geq C^{-1} |\delta
\mu|^\frac 12$  to obtain for all $q\geq2$
\begin{equation}
\label{eq:259}
 \Big\|\frac{x-x_\nu}{U+i\lambda} \Big\|_q^q\leq
 \frac{C_\delta}{|\mu|^{(q-1)/2}}  \,.
\end{equation}

\paragraph{The case $\nu \in   ( U(0) - \delta |\mu|, U(0)+\kappa_0|\mu|) $, $\delta\in (0,\kappa_0)$.}~\\

In this case  we use \eqref{eq:186} to
obtain that
\begin{displaymath}
  (U-\nu)^2\geq \frac 1C (x^2-x_\nu^2)^2\geq \frac{1}{C} \Big(\frac{x^4}{2}-x_\nu^4\Big)
  \geq \frac{1}{C_1} x^4-C_2\delta^2|\mu|^2 \,.
\end{displaymath}
The above inequality implies that there exist $C$ and $\delta_0 \leq \kappa_0$
such that, for $\delta \in (0,\delta_0)$ and $\nu \in ( U(0) - \delta |\mu|,
U(0)+\kappa_0|\mu|)$,
\begin{equation}
\label{eq:260}
  \frac{1}{(U(x)-\nu)^2+|\mu|^2} \leq \frac{C}{x^4+|\mu|^2} 
\end{equation} 
for all $x\in[0,1]$.\\
 Consequently, for all $q\geq1$,  using the substitution
 $x=|\mu|^{1/2}\xi$, we obtain
\begin{equation}\label{eq:261}
   \Big\|\frac{1}{U+i\lambda}\Big\|_q^q\leq\int_0^1\frac{C}{[x^2+|\mu|]^q}\,dx
   \leq\frac{C}{|\mu|^{q-1/2}}\int_0^\infty\frac{d\xi}{[\xi^2+1]^q}\leq
   \frac{C}{|\mu|^{q-1/2}} \,.
\end{equation}
Finally, we use  \eqref{eq:260} to obtain that
\begin{equation}
\label{eq:262}
   \Big\|\frac{x-x_\nu}{U+i\lambda} \Big\|_q^q\leq C\Big(|\mu|^{q/2}
   \Big\|\frac{1}{U+i\lambda}\Big\|_q^q+  \Big\|\frac{x}{x^2+|\mu|} \Big\|_q^q\Big)\,.
\end{equation}
We now observe that for all $q>1$
\begin{equation}\label{eq:263}
   \Big\|\frac{x}{x^2+|\mu|} \Big\|_q^q\leq
   \frac{1}{|\mu|^{(q-1)/2}}\int_0^\infty\frac{\xi^q}{[\xi^2+1]^q}\,d\xi\leq\frac{C} {|\mu|^{(q-1)/2}} \,.
\end{equation}
Together with \eqref{eq:263} and \eqref{eq:261}, 
\eqref{eq:262} yields the existence of $C>0$
\begin{equation}\label{eq:264}
 \Big\|\frac{x-x_\nu}{U+i\lambda} \Big\|_q^q\leq
 \frac{C}{|\mu|^{(q-1)/2}}  \,.
\end{equation}
\paragraph{The general case.}~\\
Combining \eqref{eq:257} and \eqref{eq:258}, \eqref{eq:261} yields,
for $q\geq 1$, the existence of $C>0$, such that for
$|U(0)-\nu|<\kappa_0|\mu|$
\begin{equation}
  \label{eq:265}
  \Big\|\frac{1}{U+i\lambda}\Big\|_q^q\leq \frac{C}{|\mu|^{q-1/2}}\,.
\end{equation}
By \eqref{eq:259} and \eqref{eq:264} we may conclude, for all  $q\geq1$,  that there exists $C>0$  such that, 
for $|U(0)-\nu|<\kappa_0|\mu|$
\begin{subequations}
\label{eq:266}
  \begin{equation}
 \Big\|\frac{x-x_\nu}{U+i\lambda} \Big\|_q^q\leq
 \frac{C}{|\mu|^{(q-1)/2}}  \,.
\end{equation}
 Note that
  \begin{multline*}
     \Big\|\frac{x-x_\nu}{U+i\lambda} \Big\|_\infty \leq C
     \Big\|\frac{x-x_\nu}{|x^2-x_\nu^2|+|\mu|} \Big\|_\infty \\ \leq C\Big(
     \Big\|{\mathbf 1}_{|x-x_\nu|<|\mu|^{1/2}}\frac{x-x_\nu}{|\mu|} \Big\|_\infty +
     \Big\|{\mathbf
       1}_{|x-x_\nu|\geq|\mu|^{1/2}}\frac{x-x_\nu}{|x^2-x_\nu^2|} \Big\|_\infty
     \Big)  \leq\frac{\widehat C}{|\mu|^{1/2}}\,,
  \end{multline*}
  hence 
  \begin{equation} 
     \Big\|\frac{x-x_\nu}{U+i\lambda} \Big\|_\infty \leq  \frac{\widehat C}{|\mu|^{1/2}}\,.
  \end{equation}
  
\end{subequations}

{\it Step 4: We prove \eqref{eq:242}. }\\

The proof is similar to Step 4 of the proof of Proposition
\ref{prop:near-quadratic}. 
We estimate the right-hand-side of \eqref{eq:249} separately for
$p\in[2,+\infty)$ and for $p=\infty$. 
Suppose first that $v\in L^p(0,1)$ for some $p\in[2,+\infty)$. \\
For the first term in the r.h.s., we deduce from \eqref{eq:265}
 \begin{equation}
\label{eq:267}
  |\mu|^{1/4} \Big\|\frac{v}{U+i\lambda}\Big\|_1\leq C\frac{\|v\|_p}{|\mu|^{\frac{1}{2p}+\frac{1}{4}}} \,.
\end{equation}
To estimate the second term we use \eqref{eq:266} to obtain  for all
$p\geq2$
\begin{equation}
\label{eq:268}
   \Big\|\frac{(x-x_\nu)v }{U+i\lambda} \Big\|_2\leq
\frac{C}{|\mu|^{\frac{1}{2p}+\frac{1}{4}}}\|v\|_p \,.
\end{equation}
Suppose now that $v\in L^\infty(0,1)$. Then, by \eqref{eq:265}
we may conclude for the first term on the r.h.s. of \eqref{eq:249} that
\begin{equation}
\label{eq:269}
   |\mu|^{1/4}\Big\|\frac{v}{U+i\lambda}\Big\|_1\leq \frac{C}{|\mu|^{1/4}}\|v\|_\infty\,.
\end{equation}
Next, we estimate the second term on the right-hand-side of
\eqref{eq:249}. Using \eqref{eq:266} we obtain that
\begin{displaymath}
  \Big\|\frac{(x-x_\nu)v}{U+i\lambda} \Big\|_2\leq \frac{C}{|\mu|^{1/4}} \|v\|_\infty
  \,.
\end{displaymath}
Together with \eqref{eq:269} the above yields (\ref{eq:242})  for
$p=+\infty$. 
\end{proof}
\begin{remark}
\label{rem:quadratic-bound}
Note that, under the assumptions of Proposition \ref{prop:quadratic}, 
by \eqref{eq:255}, \eqref{eq:256}, (\ref{eq:266}a)  and \eqref{eq:265}
it holds that 
  \begin{displaymath}
    |\phi(x_\nu)|^2\leq C\,\big(|\phi(x_\nu)|\,\|v\|_\infty +|\mu|^{1/4} \|\phi^\prime\|_2\|v\|_\infty\big)\,.
  \end{displaymath}
Using  (\ref{eq:242})  for
$p=+\infty$ then yields the existence of $C>0$ such that
\begin{equation}
  \label{eq:270}
 |\phi(x_\nu)|\leq C\, \|v\|_\infty \,.
\end{equation}
\end{remark}

\subsection{The case $\Im \lambda >U(0)$}
\label{sec:case--large-nu}
In the case where $\Im \lambda=\nu>U(0)$, we get a better estimate of
$\A_{\lambda,\alpha}^{-1}$, measured by
a negative power of $ |\mu| + (\nu -U(0))$. More precisely,
\begin{proposition}
\label{prop:non-factorable}
Let $p\in[2,+\infty]$. 
There exist $\mu_0>0$ and $C>0$ such that for all
$U\in C^3([0,1])$ satisfying \eqref{eq:10}, $\nu>U(0)$, $|\mu|<\mu_0$,
$\alpha\geq0$, and $(\phi,v)\in D(\A_{\lambda,\alpha})\times L^\infty(0,1)$ satisfying
\eqref{eq:47} it holds that
 \begin{equation}
\label{eq:271}
\|\phi\|_{1,2}
\leq \frac{C}{[|\mu|+|\nu-U(0)|]^{\frac{1}{2p}+\frac{1}{4}}}\|v\|_p  \,,
\end{equation}
\end{proposition}
\begin{proof}
 We begin by restating \eqref{eq:196} 
\begin{equation}\label{eq:272}
  \|\phi^\prime\|_2^2+\alpha^2\|\phi\|_2^2+
  \Big\langle\frac{U^{\prime\prime}(U-\nu)\phi}{(U-\nu)^2+\mu^2},\phi\Big\rangle=
  \Re\Big\langle\phi,\frac{v}{U+i\lambda}\Big\rangle\,. 
\end{equation}
Since $\nu>U(0)$, it holds by \eqref{eq:10} that the third term on the left-hand-side  is
positive, and hence we can conclude that
\begin{equation}
  \label{eq:273}
  \|\phi^\prime\|_2^2\leq   \Re\Big\langle\phi,\frac{v}{U+i\lambda}\Big\rangle\,. 
\end{equation}
We split the proof into two separate parts in accordance with the
magnitude of $(U(0)-\nu)^2+\mu^2$.\\

{\em Step 1: The case  $(U(0)-\nu)^2+\mu^2$ small}.\\

We consider here the case where
\begin{displaymath}
  (U(0)-\nu)^2+\mu^2<\varepsilon 
\end{displaymath}
for some sufficiently small $\varepsilon>0$. 

  To properly bound the right-hand-side of \eqref{eq:273} in that case
  we need an estimate for $|\phi(0)|$. To this end we use
  \eqref{eq:272} once again to obtain
\begin{displaymath}
  \Big\langle\frac{U^{\prime\prime}(U-\nu)\phi}{(U-\nu)^2+\mu^2},\phi\Big\rangle\leq 
  \Re\Big\langle\phi,\frac{v}{U+i\lambda}\Big\rangle\,,
\end{displaymath}
from which we can conclude that
\begin{multline}
  \label{eq:274}
 \Big\langle\frac{U^{\prime\prime}(U-\nu)}{(U-\nu)^2+\mu^2},1\Big\rangle|\phi(0)|^2\leq \\
 2\Big\langle\frac{U^{\prime\prime}(U-\nu)(\phi-\phi(0))}{(U-\nu)^2+\mu^2},\phi-\phi(0)\Big\rangle + 2 \Re\Big\langle\phi,\frac{v}{U+i\lambda}\Big\rangle\,.
\end{multline}
We continue by bounding from below the left-hand-side of
\eqref{eq:274}. To this end we observe that since
\begin{displaymath}
  \min_{x\in[0,1]}|U^{\prime\prime}(x)|\geq \frac {1}{C_0} >0 \text{ and } U(0)-U(x) \geq
  \frac{1}{2C_0} x^2
\end{displaymath}
we can conclude that
\begin{equation*}
\begin{array}{ll}
   \Big\langle\frac{U^{\prime\prime}(U-\nu)}{(U-\nu)^2+\mu^2},1\Big\rangle|& \geq \frac 1C \int_0^1
   \frac{x^2+\nu-U(0)}{[x^2+\nu-U(0)]^2 + \mu^2}\,dx\\
   & \geq \frac 1C \int_0^1
   \frac{x^2}{[x^2+\nu-U(0)]^2 + \mu^2}\,dx\\
    & \geq \frac 1{2C}  \int_0^1
   \frac{x^2}{x^4+(\nu-U(0)^2 + \mu^2}\,dx\\
   &  \geq \frac 1{2C} \Big[\int_0^\infty 
   \frac{x^2}{x^4+[\nu-U(0)]^2 + \mu^2}\,dx -1\Big]\,. 
   \end{array}
\end{equation*}
Setting 
\begin{displaymath}
  x=([\nu-U(0)]^2 + \mu^2)^{1/4}s
\end{displaymath}
yields
\begin{multline*}
  \Big\langle\frac{U^{\prime\prime}(U-\nu)}{(U-\nu)^2+\mu^2},1\Big\rangle\geq \frac{1}{2C} \left( 
  \frac{1}{([\nu-U(0)]^2 + \mu^2)^{1/4}}\int_0^\infty
   \frac{s^2}{s^4+1}\,ds  -1\right) \\ \geq  \frac{1}{\hat C \{[\nu-U(0)]^2 +
     \mu^2\}^{1/4}}-\hat C\,.  
\end{multline*}
For sufficiently small $\varepsilon$ we obtain that 
\begin{equation}
\label{eq:275}
  \Big\langle\frac{U^{\prime\prime}(U-\nu)}{(U-\nu)^2+\mu^2},1\Big\rangle|\geq  \frac{1}{2\hat C \{[\nu-U(0)]^2 +
     \mu^2\}^{1/4}}\,.
\end{equation}

To estimate the first term on the right-hand-side  of \eqref{eq:274}
we first write 
\begin{displaymath}
  \Big\langle\frac{U^{\prime\prime}(U-\nu)(\phi-\phi(0))}{(U-\nu)^2+\mu^2},\phi-\phi(0)\Big\rangle \leq
  \Big\|\frac{U^{\prime\prime}(U-\nu)x^2}{(U-\nu)^2+\mu^2}\Big\|_\infty
  \Big\|\frac{\phi-\phi(0)}{x}\Big\|_2^2 \,. 
\end{displaymath}
As
 \begin{displaymath}
\Big\|\frac{U^{\prime\prime}(U-\nu)x^2}{(U-\nu)^2+\mu^2}\Big\|_\infty\leq C\Big\|\frac{x^4+[\nu-U(0)]x^2}{[x^2+\nu-U(0)]^2+\mu^2}\Big\|_\infty\leq
\widehat C\,,
\end{displaymath}
we may use Hardy's inequality to obtain that
\begin{equation}\label{eq:276}
  \Big\langle\frac{U^{\prime\prime}(U-\nu)(\phi-\phi(0))}{(U-\nu)^2+\mu^2},\phi-\phi(0)\Big\rangle \leq C \, \|\phi^\prime\|_2^2 \,.
\end{equation}
\eqref{eq:276}  together with \eqref{eq:273} and \eqref{eq:275} yields, when
substituted into \eqref{eq:274},
\begin{equation}
  \label{eq:277}
|\phi(0)|^2\leq
C(|\mu|^{1/2}+|\nu-U(0)|^{1/2})\Re\Big\langle\phi,\frac{v}{U+i\lambda}\Big\rangle \,.
\end{equation}

To complete the proof we now estimate the right-hand-side  of
\eqref{eq:273}. To this end we write
\begin{displaymath}
  \Re\Big\langle\phi,\frac{v}{U+i\lambda}\Big\rangle \leq |\phi(0)|
  \Big\|\frac{v}{U+i\lambda}\Big\|_1 + \Big\|\frac{\phi-\phi(0)}{x}\Big\|_2
  \Big\|\frac{xv}{U+i\lambda}\Big\|_2 \,. 
\end{displaymath}
Using Hardy's inequality together with \eqref{eq:277} yields
\begin{displaymath}
\label{eq:278}
    \Re\Big\langle\phi,\frac{v}{U+i\lambda}\Big\rangle \leq
    C\Big[(|\mu|^{1/2}+|\nu-U(0)|^{1/2})\Big\|\frac{v}{U+i\lambda}\Big\|_1^2 +
    \|\phi^\prime\|_2  \Big\|\frac{xv}{U+i\lambda}\Big\|_2\Big]  
\end{displaymath}
Using \eqref{eq:273} once again yields
\begin{equation}
\label{eq:279}
  \|\phi^\prime\|_2 \leq C\Big[(|\mu|^{1/4}+|\nu-U(0)|^{1/4})\Big\|\frac{v}{U+i\lambda}\Big\|_1+
    \Big\|\frac{xv}{U+i\lambda}\Big\|_2\Big]  \,.
\end{equation}

The proof of \eqref{eq:271} is now verified by following the same path 
as in the proof of Step 4 of Propositions \ref{prop:near-quadratic}
and \ref{prop:quadratic}. Thus, since
\begin{displaymath}
  (U-\nu)\geq C(x^2+\nu-U(0) )\,,
\end{displaymath}
we obtain as in \eqref{eq:261} that for all $q\geq1$
\begin{multline}
\label{eq:280}
   \Big\|\frac{1}{U+i\lambda}\Big\|_q^q\leq \int_0^1\frac{C}{[x^2+\nu-U(0)+|\mu|]^q}\,dx\\
   \leq\frac{C}{[|\mu+\nu-U(0)]|^{q-1/2}}\int_0^\infty\frac{d\xi}{[\xi^2+1]^q}\\\leq
   \frac{\hat C}{[|\mu|+{\nu -U(0)}]^{q-1/2}} \,.
\end{multline}
Hence, for all $p\in[2,+\infty]$,
\begin{equation}
  \label{eq:281}
\Big\|\frac{v}{U+i\lambda}\Big\|_1\leq
  C  [|\mu|+\nu-U(0)]^{-\frac{1}{2p}-\frac{1}{2}} \|v\|_p \,. 
\end{equation}

As in \eqref{eq:262} and \eqref{eq:263} we then write for all $q>1$
\begin{multline*}
  \Big\|\frac{x}{U+i\lambda} \Big\|_q^q  \leq C\, \Big\|\frac{x}{x^2+\nu-U(0)
    +|\mu|} \Big\|_q^q \\
  \leq   \frac{\widehat C }{[|\mu|+\nu-U(0)]^{(q-1)/2}}
  \int_0^\infty\frac{\xi^q}{[\xi^2+1]^q}\,d\xi  \\ \leq\frac{\widetilde C}
  {[|\mu|+\nu-U(0)]^{(q-1)/2}} \,. 
\end{multline*}
Note also that for $q=+\infty$ we have
\begin{displaymath}
  \Big\|\frac{x}{U+i\lambda} \Big\|_\infty\leq\frac{C}
  {[|\mu|+\nu-U(0)]^{1/2}}
\end{displaymath}
Consequently, for all $p\in[2,+\infty]$
\begin{equation}
\label{eq:282}
  \Big\|\frac{xv }{U+i\lambda} \Big\|_2\leq
  \frac{C}{[|\mu|+\nu-U(0)]^{\frac{1}{2p}+\frac{1}{4}}} {\|v\|_p }\,.
\end{equation}
Substituting the above, together with \eqref{eq:281}, into
\eqref{eq:279} yields  \eqref{eq:271} for sufficiently small $\varepsilon >0$.\\

{\em Step 2: The case where, for some $\varepsilon>0$,
\begin{displaymath}
  (U(0)-\nu)^2+\mu^2\geq \varepsilon \,.
\end{displaymath}
}

By Poincar\'e's inequality and \eqref{eq:281}, we obtain
\begin{multline*}
  |  \Re\Big\langle\phi,\frac{v}{U+i\lambda}\Big\rangle|\leq \|\phi\|_\infty
  \Big\|\frac{v}{U+i\lambda}\Big\|_1 \\\leq 
  \frac{C}{[(U(0)-\nu)^2+\mu^2]^{1/2}}\|\phi^\prime\|_2\|v\|_p\\ \leq
  \frac{C}{\varepsilon^{\frac{2-p}{4} }\, [(U(0)-\nu)^2+\mu^2]^{\frac{1}{2p}+\frac{1}{4}}}\|\phi^\prime\|_2\|v\|_p\,, 
\end{multline*}
which together with \eqref{eq:273} readily gives
\eqref{eq:271} in this case.
\end{proof}

{\begin{remark}
\label{rem:large-nu-bound}
As in Remark \ref{rem:quadratic-bound} we may use \eqref{eq:277},
\eqref{eq:278}, and \eqref{eq:279}, under the assumptions of
Proposition \ref{prop:non-factorable}, to obtain
\begin{equation*}
|\phi(0)|^2\leq
C(|\mu|+|\nu-U(0)|) \,
\Big\|\frac{v}{U+i\lambda}\Big\|_1^2+(|\mu|+|\nu-U(0)|)^{1/2}  \Big\|\frac{xv}{U+i\lambda}\Big\|_2^2\,.
\end{equation*}
From \eqref{eq:281} we then obtain
the existence of $C>0$ such that
\begin{equation}
 \label{eq:283}
 |\phi(0)|\leq C\, \|v\|_\infty \,.
\end{equation}
\end{remark}

 \subsection{The case $|\Re \lambda | \geq \mu_1 >0$. }
The results in the preceding subsections were all obtained under the
assumption that $|\mu| \leq \mu_0$ for some sufficiently small $\mu_0$.
Hence it remains to treat the case when $|\mu| \geq \mu_1$ where $\mu_1>0$ is arbitrary.
We complement Proposition \ref{prop:linear} by addressing the large
$|\mu|$ case.
  \begin{proposition}\label{rem:large-mu}
    For any $\mu_1>0$ and $p>1$ there exists $C>0$ such that for all
    $(\phi,v)\in D(\A_{\lambda,\alpha})\times W^{1,p}(0,1)$ satisfying $\mathcal A_{\lambda,\alpha}\phi=v$\,, 
    for all $|\mu|\geq\mu_1$, $\nu\in\R$ and $\alpha\geq0$, 
\begin{equation}
\label{eq:167}
      \|\phi\|_{H^1(0,1)}\leq C\|(1-x)^{1/2}v\|_1\,.
    \end{equation}
  \end{proposition}
  \begin{proof} 
    Since \eqref{eq:100} holds true for any $\mu\neq0$ we can conclude
    that
\begin{equation}
\label{eq:168}
   \|\phi^\prime\|_2 \leq\Big(1+
   \frac{C}{|\mu|^{1/2}}\Big)\Big|\Big\langle\phi,\frac{v}{U+i\lambda}\Big\rangle\Big|^{1/2} \,.
\end{equation}
Hence, there exists $C$ such that for $|\mu| \geq \mu_1$
\begin{equation}
\label{eq:169}
   \|\phi^\prime\|_2^2 \leq  C \, \Big|\Big\langle\phi,\frac{v}{U+i\lambda}\Big\rangle\Big| \,.
\end{equation}
Consequently, as $|\phi(x)|\leq \|\phi^\prime\|_2(1-x)^{1/2}$ and $|\mu|>\mu_1$,
\begin{displaymath}
   \|\phi^\prime\|_2^2 \leq
   C\|\phi^\prime\|_2\|(1-x)^{1/2}v\|_1\Big\|\frac{1}{U+i\lambda}\Big\|_\infty
   \leq \frac{C}{\mu_1} \|\phi^\prime\|_2\|(1-x)^{1/2}v\|_1 \,,
\end{displaymath}
from which \eqref{eq:167} follows with the aid of Poincar\'e's
inequality.  
  \end{proof}

\section{Neumann-Dirichlet Schr\"odinger operators}
\label{sec:3}
The first part of this  section is devoted to resolvent estimates for the operator
\begin{subequations} \label{eq:285}
  \begin{equation}
    \LL_\beta^{\Nf,\Df}=-\frac{d^2}{dx^2}+i\beta \, U \,,
  \end{equation}
  which is defined on
  \begin{equation}
    D(\LL_\beta^{\Nf,\Df})=\{ u\in H^2(0,1)\,|\, u(1)=u^\prime(0)=0\} \,.
  \end{equation}
\end{subequations}
We note that to estimate the inverse norm of the Orr-Sommerfeld
operator \eqref{eq:6} we need a bound of both $\A_{\lambda,\alpha}^{-1}$ and
$(\LL_\beta^{\Nf,\Df}-\lambda)^{-1}$, and in addition, to obtain resolvent
estimates for the special Schr\"odinger operators of the next section. 

For convenience of notation we omit in the sequel the reference  to the
Dirichlet condition at $x=1$ and use $\LL_\beta^{\Nf}$ instead of $
\LL_\beta^{\Nf,\Df}$. }

 Let $\lambda=\mu+i\nu$. Recall the definition of $x_\nu$ from
  \eqref{eq:defxnu}.  From \cite{AGH}, for instance, we know that the
  main contribution to the resolvent norm comes from a small region
  near $x=x_\nu$. We begin this section by estimating the
  resolvent norm of $\LL_\beta^{\Nf}$ in the case where $U(0)-\nu
  \gg\beta^{-1/2}$. In this case one may approximate $U-\nu$ by a linear
  potential of the form $U^\prime(x_\nu)(x-x_\nu)$. In Subsection
  \ref{sec:schrod-quad}, we consider the case $|U(0)-\nu|\lesssim\beta^{-1/2}$
  where $U-\nu$ will be approximated by the quadratic potential
  $U^{\prime\prime}(0)x^2/2+U(0)-\nu$. Finally, Subsection \ref{sec:l1-estimates} is
  devoted to some $L^1$ estimates that are necessary in the Section
  \ref{sec:4}.

\subsection{Resolvent estimates for $ U(0)-\Im \lambda \gg\beta^{-1/2}$.}\label{ss3.1}
With $x_\nu$ defined in \eqref{eq:defxnu}, we introduce 
\begin{equation}
\label{eq:286}
{\mathfrak J_\nu}=|U^\prime(x_\nu)| \,.
\end{equation}
We further define $\hat \kappa_1\in\C$  to be the  leftmost eigenvalue of
\begin{equation}\label{eq:defkappa1}
  \LL_+=-\frac{d^2}{dx^2} + ix 
\end{equation}
in $H^2(\R_+)\cap H^1_0(\R_+)$. 
The first proposition is similar to \cite[Proposition
5.2]{almog2019stability}.  Unlike \cite{almog2019stability} which
defines the problem on $(-1,+1)$ with Dirichlet conditions at $x=\pm1$,
we consider the operator on $(0,1)$ with a Neumann condition at $x=0$,
in accordance with a restriction to even function space on $(-1,1)$,
and a Dirichlet condition at $x=1$.  Furthermore, the velocity field
is not strictly monotone as in \cite{almog2019stability}.
Nevertheless, for $x_\nu\gg\beta^{-1/4}$ (or equivalently for $U(0) -\nu
\gg\beta^{-1/2}$) we can still make good use of the 
estimates in \cite{almog2019stability}.
\begin{proposition}
\label{prop5.3}  
Let $U\in C^2([0,1])$ satisfy \eqref{eq:10} 
and $ p \in (1,2]$. Then
there exist positive $\Upsilon$, $a$, $C$, $C_p$, and $\beta_0$ such that, for
all $\beta\geq \beta_0$, $U(0)-\nu>a\beta^{-1/2}$, and $f\in L^\infty(0,1)$,
   \begin{subequations}
\label{eq:287}
       \begin{multline} 
 \sup_{ \mu\leq\Upsilon  {\mathfrak J_\nu}^{2/3}\beta^{-1/3}}
\|(\LL_\beta^\Nf-\beta\lambda)^{-1}f\|_2  \leq C\min([{\mathfrak
J_\nu}\beta]^{-2/3}\|f\|_2,[{\mathfrak J_\nu}\beta]^{-5/6}\|f\|_\infty)\,.
  \end{multline}
Furthermore, for  $f\in L^2(0,1),$
\begin{equation}
 \sup_{ \mu \leq\Upsilon  {\mathfrak J_\nu}^{2/3}\beta^{-1/3} }
\Big\| \frac{d}{dx} (\LL_\beta^\Nf-\beta\lambda)^{-1}f\Big\|_p\leq \frac{C_p}{[{\mathfrak J}_\nu\beta]^{\frac{2+p}{6p}}}\|f\|_2\,.
\end{equation}   
\end{subequations}
\end{proposition}
\begin{proof}~\\
 {\em Step 1: For $\Upsilon>0$, we prove that there exist positive $\beta_0$, $a_0$,  and $C$
  such that for all $\beta \geq \beta_0$,  $U(0)-\nu \geq a \beta^{-\frac 12}$, $a\geq
  a_0$  and $  \mu  \leq\Upsilon  {\mathfrak J_\nu}^{2/3}\beta^{-1/3}$ 
  we have
\begin{equation}
\label{eq:288}
  \|(\LL_\beta^\Nf-\beta\lambda)^{-1}f\|_2+ [{\mathfrak
 J_\nu}\beta]^{-1/3}\|\frac{d}{dx}\, (\LL_\beta^\Nf-\beta\lambda)^{-1}f\|_2 \leq C \, [{\mathfrak
J_\nu}\beta]^{-2/3}\|f\|_2\,.
\end{equation}
}

As  $U(0) -\nu >a\beta^{-1/2}$, it holds that  
\begin{equation}\label{eq:289}
x_\nu \geq \frac{1}{C}\,
a^{1/2} \beta^{-\frac 14} \,.
\end{equation} 
For future reference we note, in addition,  that 
\begin{equation}\label{eq:290}
\frac 1C x_\nu \leq \mathfrak J_\nu \leq C x_\nu\,. 
\end{equation}
We split the proof of \eqref{eq:288} into two parts according to the sign of $\nu-U(\frac 12)$.\\[1.5ex]

{\em Step  1.1: The case when  $U(1/2)<\nu<U(0)-a\beta^{-1/2}$. \\}

We note that in this case $x_\nu \in (0,\frac 12)$.  Let $\hat{\chi}$ be given as in \eqref{eq:41} by 
 \begin{equation}
\label{eq:291}
   \hat{\chi}(x)=
   \begin{cases}
     0 & |x|<\frac{1}{4} \\
     1  & |x|>\frac{1}{2} \,.
   \end{cases}
 \end{equation}
  Set further for $x\in [0,1]$
  \begin{equation}
\label{eq:292}
    \chi_\nu^\pm(x)=\hat{\chi}(x/x_\nu-1){\mathbf 1}_{\R_+}(\pm
    (x-x_\nu)) \,.
  \end{equation}
  and $\hat \chi$ can be chosen such that 
  \begin{equation}\label{eq:293}
  \tilde{\chi}_\nu:=\sqrt{1-(\chi_\nu^+)^2-(\chi_\nu^-)^2}\in C^\infty([0,1])\,.
  \end{equation}
 Note that $\chi_\nu^-$ is supported on $(0,3x_\nu/4)$ whereas
  $\chi_\nu^+$ is supported on $(5x_\nu/4,1)$.  The complementary cutoff
  function $\tilde{\chi}_\nu$ is supported on $(x_\nu/2,3x_\nu/2)$. 
   The cutoff functions defined in \eqref{eq:292} allow us to obtain
  estimates for $v$ separately on the intervals $(0,x_\nu/2)$,
  $(3x_\nu/2,1)$ (via integration by parts), and $(x_\nu/2,3x_\nu/2)$.  
  
More precisely,  let $(v,f)\in D(\LL_\beta^\Nf)\times L^2(0,1)$ satisfy
   $(\LL_\beta^\Nf-\lambda)v=f$.
Set
 \begin{equation}
\label{eq:294}
    \hat{U}_\nu(x)=
    \begin{cases}
      U(x) & x\in(x_\nu/2,3x_\nu/2) \\
      U(x_\nu/2)+U^\prime(x_\nu/2)(x-x_\nu) & x\leq\frac{x_\nu}{2} \\
      U(3x_\nu/2)+U^\prime(3x_\nu/2)(x-x_\nu) & x\geq \frac{3x_\nu}{2}
    \end{cases}
    \,.
  \end{equation}
Let then 
\begin{subequations}
\label{eq:295}
  \begin{equation}
   \hat{\LL}_{\beta,\nu,\R} = -\frac{d^2} {dx^2} + i\beta \hat{U}_\nu  \,,
  \end{equation}
be defined on the domain
\begin{equation}
D( \hat{\LL}_{\beta,\nu,\R})=\{u\in H^2(\R)\,|\,xu\in L^2(\R)\}\,.
\end{equation}
\end{subequations}

We now apply the unitary dilation operator (corresponding to the
change of variable $y= x /x_\nu$)
\begin{equation}
\label{eq:296}
  {\Tg_\nu}u(x)=x_\nu^{-1/2}u( x/ x_\nu) \,,
\end{equation}
to obtain
  \begin{equation}
\label{eq:297}
  \Tg_\nu^{-1} \hat {\LL}_{\beta,\nu,\R} \Tg_\nu=
  x_\nu^{-2}\Big(-\frac{d^2}{dy^2}+ i \tilde \beta_\nu  \tilde {U}_\nu (y)\Big)\,,
\end{equation}
where
\begin{equation}\label{eq:298}
  \tilde {U}_\nu (y)=\frac{\hat{U}_\nu(x_\nu y)}{x_\nu^2} \mbox{ and }\tilde \beta_\nu = \beta x_\nu^4 \,.
\end{equation}
As there exist positive $m$, $M$ such that for $\nu$ satisfying the assumptions of our proposition 
\begin{equation}
\label{eq:299}
  0<m\leq|\tilde {U}_\nu^\prime(y)|\leq M\,,\;\forall y \in\R\,,
\end{equation}
and in view of the uniform bound 
\begin{equation}\label{eq:300}
\|\tilde{U}_\nu^{\prime\prime}\|_{L^\infty(\R)}\leq C\,,
\end{equation}
we may apply \cite[Proposition 5.1]{almog2019stability} to  the family of operators 
\begin{equation}\label{eq:301}
\tilde \LL_{\tilde \beta,\mathbb R} :=-\frac{d^2}{dy^2}+i \tilde  \beta \, \tilde {U}_\nu (y) 
\end{equation}
to obtain for $\hat \beta \geq \hat \beta_0$
\begin{equation*}
   \sup_{\Re \tilde \lambda\leq\Upsilon \tilde \beta^{-1/3}}\|( \tilde {\LL}_{\tilde \beta,\R}-\tilde \beta \tilde \lambda)^{-1}\| +
  \tilde \beta^{-1/3} \Big\|\frac{d}{dy}( \tilde {\LL}_{\tilde \beta,\R}- \tilde \beta \tilde \lambda)^{-1}\Big\|\leq
   \frac{C}{\tilde \beta^{2/3}}\,, 
\end{equation*}
where 
\begin{equation}
\tilde \beta =\tilde \beta_\nu= \beta x_\nu^4 \mbox{  and } \tilde \lambda
=x_\nu^{-2}\lambda\,.
\end{equation}
 We observe that, by \eqref{eq:289}, for any given
$\hat \beta_0>0$ there exists $a_0>0$ such that $\tilde \beta>\hat \beta_0$ is
satisfied for $a\geq a_0$. Taking the inverse dilation transformation,
we obtain
\begin{equation*}
   \sup_{\Re \lambda\leq\Upsilon  x_\nu^{2/3} \beta^{-1/3}}\|( \hat {\LL}_{\beta,\nu, \R}- \beta \lambda)^{-1}\| +
   (\beta  x_\nu)^{-1/3} \Big\|\frac{d}{dx}( \hat  {\LL}_{ \beta,\nu, \R}- \beta \lambda)^{-1}\Big\|\leq
   \frac{C}{(x_\nu  \beta)^{2/3}}\,.  
\end{equation*}

Given that \cite[Proposition 5.1]{almog2019stability} allows for an
arbitrary $\Upsilon$, we can replace $x_\nu$ by $\mathfrak J_\nu$ to obtain
\begin{equation}
\label{eq:302}
   \sup_{\Re \lambda\leq\Upsilon  \mathfrak J_\nu^{2/3} \beta^{-1/3}}\|( \hat {\LL}_{\beta,\nu, \R}- \beta \lambda)^{-1}\| +
   (\beta  \mathfrak J_\nu)^{-1/3} \Big\|\frac{d}{dx}( \hat  {\LL}_{ \beta,\nu, \R}- \beta \lambda)^{-1}\Big\|\leq
   \frac{C}{(\mathfrak J_\nu  \beta)^{2/3}}\,.  
\end{equation}
We now write 
\begin{equation}
\label{eq:303}
  ( \hat {\LL}_{\beta,\nu,\R} -\beta\lambda)(\tilde{\chi}_\nu v)=\tilde{\chi}_\nu f-2\tilde{\chi}_\nu^\prime v^\prime- \tilde{\chi}_\nu^{\prime\prime}v \,.
\end{equation}
We can then conclude from \eqref{eq:302} and \eqref{eq:303} that
\begin{equation}
\label{eq:304} 
  \|\tilde{\chi}_\nu v\|_2+[\beta\mathfrak J_\nu ]^{-1/3} \|(\tilde{\chi}_\nu
  v)^\prime\|_2\leq\frac{C}{[\beta\mathfrak J_\nu ]^{2/3}}(\|\tilde{\chi}_\nu f\|_2+2\|\tilde{\chi}_\nu^\prime v^\prime\|_2+
  \|\tilde{\chi}_\nu^{\prime\prime}v\|_2)\,. 
\end{equation}
  Note that
\eqref{eq:304} implies, by \eqref{eq:290},  that
\begin{equation}
\label{eq:305} 
  \|\tilde{\chi}_\nu v\|_2+[\beta x_\nu ]^{-1/3} \|(\tilde{\chi}_\nu
  v)^\prime\|_2\leq\frac{C}{[\beta x_\nu ]^{2/3}}(\| f\|_2+ x_\nu^{-1} \| v^\prime\|_2+ x_\nu^{-2} 
  \|v\|_2)\,. 
\end{equation}

To estimate $v$ on $(0,3x_\nu/4)$ and $(5x_\nu/4,1)$ we write 
\begin{equation}
\label{eq:306}
  (\LL_\beta^\Nf-\beta\lambda)(\chi^\pm_\nu v)=\chi^\pm_\nu f-2(\chi^\pm_\nu)^\prime v^\prime- (\chi^\pm_\nu)^{\prime\prime}v \,.
\end{equation}
The real part of the inner product with $ \chi^\pm_\nu v$, after
integration by parts, is given by
\begin{equation}
\label{eq:307}
  \|(\chi_\nu^\pm v)^\prime\|_2^2=\|(\chi_\nu^\pm)^\prime v\|_2^2+\mu\beta\|\chi_\nu^\pm v\|_2^2 +
  \Re\langle\chi_\nu^\pm v,\chi_\nu^\pm f\rangle \,, 
\end{equation}
whereas the imaginary part assumes the form
\begin{equation}
\label{eq:308}
  \mp \beta\| |U-\nu|^{1/2}\chi_\nu^\pm v\|_2^2+2\Im\langle(\chi_\nu^\pm)^\prime v,(\chi_\nu^\pm
  v)^\prime\rangle =  +\Im \langle\chi_\nu^\pm v,\chi_\nu^\pm f\rangle \,.
\end{equation}
As, by \eqref{eq:185},   $|U-\nu|^{1/2}\chi_\nu^\pm\geq \frac 1 C \,
x_\nu\chi_\nu^\pm$, 
\eqref{eq:308} yields first 
\begin{displaymath}
\frac 1C \beta x_\nu^2  \|\chi_\nu^\pm v\|_2^2 \leq \|\chi_\nu^\pm v\|_2\, \| \chi_\nu^\pm
f \|_2 + \frac{C}{x_\nu} \|v\|_2\,  \| (\chi_\nu^\pm v)^\prime\|_2\,. 
\end{displaymath}
Combining the above with \eqref{eq:307} we obtain
 \begin{multline*}
\frac 1C \beta x_\nu^2  \|\chi_\nu^\pm v\|_2^2 \leq \|\chi_\nu^\pm v\|_2\, \| \chi_\nu^\pm f \|_2\\ +
\frac{C}{x_\nu} \|v\|_2\, \Big( \frac{1}{x_\nu} \|v\|_2  +  \mu_{\beta,+}^\frac
  12  \|\chi_\nu^\pm v\|_2 + \sqrt{\|\chi_\nu^\pm v\|_2\, \| \chi_\nu^\pm f\|_2}  \Big) \,, 
\end{multline*}
where
\begin{equation}
\label{eq:284}
  \mu_{\beta,+}=\max(\mu\beta,0)\,.
\end{equation}
For  $\mu\leq\Upsilon  {\mathfrak J_\nu}^{2/3}\beta^{-1/3}$, we may conclude,
using  \eqref{eq:290}, that
 \begin{displaymath}
\frac 1C \beta x_\nu^2  \|\chi_\nu^\pm v\|_2^2 \leq \|\chi_\nu^\pm v\|_2\, \| \chi_\nu^\pm f \|_2 +
\frac{C}{x_\nu} \|v\|_2\, \left( \frac{1}{x_\nu} \|v\|_2 + (x_\nu \beta)^\frac
  13  \|\chi_\nu^\pm v\|_2 + \sqrt{\|\chi_\nu^\pm v\|_2\,  \|\chi_\nu^\pm  f\|_2}  \right) \,,
\end{displaymath}
which implies
\begin{displaymath}
 \|\chi_\nu^\pm v\|_2^2 \leq \frac{C}{(\beta x_\nu^2)^2} \| \chi_\nu^\pm f\|^2_2 + C
 \frac{1}{\beta x_\nu^2}\,  \left(\frac{1}{x_\nu^2} + \frac{1}{\beta^\frac 13
     x_\nu^{\frac{10}{3} }}\right) \|v\|^2_2 \,.
\end{displaymath}
By \eqref{eq:289},  there exists $C>0$ such that
\begin{displaymath}
 \frac{1}{\beta^\frac 13   x_\nu^{\frac{10}{3} }} \leq
 Ca^{-2/3} x_\nu^{-2}  \,.
\end{displaymath}
Hence we may conclude that there exists $C>0$ such that, if $a \geq1$
and \break $U(0) -\nu >a\beta^{-1/2}$,
\begin{equation}
\label{eq:309}
  \|\chi_\nu^\pm v\|_2 \leq \frac{C}{\beta x_\nu^2}[\|\chi_\nu^\pm f\|_2+\beta^{1/2} \| v\|_2]\,.
\end{equation}
Combining \eqref{eq:309} and \eqref{eq:305} leads to, with the aid
of \eqref{eq:289}, 
\begin{displaymath}
\begin{array}{ll}
  \|v\|_2& \leq
  C\big(([\beta x_\nu^2]^{-1}+[\beta x_\nu]^{-2/3})\|f\|_2+ ( [\beta
  x_\nu^4]^{-1/3} + [\beta x_\nu^4]^{-2/3}) \|v\|_2+[\beta^2x_\nu^5]^{-1/3}\|v^\prime\|_2\big)\\
  &\leq   \check C\big([\beta x_\nu]^{-2/3}\|f\|_2+[\beta^2x_\nu^5]^{-1/3}\|v^\prime\|_2\big)  + \widehat C a^{-\frac 23} \|v\|_2 \,.
  \end{array}
\end{displaymath}
Thus, there exists $a_0\geq 1$  and $C>0$ such that for $a \geq a_0$ we obtain
\begin{equation}
\label{eq:310}
   \|v\|_2\leq   C\big([\beta x_\nu]^{-2/3}\|f\|_2+[\beta^2x_\nu^5]^{-1/3}\|v^\prime\|_2\big)\,. 
\end{equation}
We now use \eqref{eq:305} together with  \eqref{eq:289} to
establish that
 \begin{equation} 
\label{eq:311}
\|(\tilde{\chi}_\nu  v)^\prime\|_2 \leq  C  ([\beta \, x_\nu ]^{- 1/3}  \|f\|_2+[\beta x_\nu^7]^{-1/3}
  \|v\|_2)  + \widehat C a^{-\frac 23}\,  \| v^\prime\|_2  \,. 
\end{equation}
By \eqref{eq:307}, as $ \mu\beta \leq C (x_\nu \beta)^\frac 23$,   it holds that
\begin{equation*}
  \|(\chi_\nu^\pm v)^\prime\|_2 \leq  C \left( \frac{1}{x_\nu} \| v\|_2 + x_\nu^\frac 13 \beta^\frac 13 \|\chi_\nu^\pm v\|_2 +
 | \Re\langle\chi_\nu^\pm v,\chi_\nu^\pm f\rangle |^\frac 12\right) 
\end{equation*}
which leads to
\begin{equation}
  \|(\chi_\nu^\pm v)^\prime\|_2 \leq   \widehat C \left( \frac{1}{x_\nu} \| v\|_2
    + (\beta x_\nu)^\frac 13  \|\chi_\nu^\pm v\|_2 + (\beta x_\nu)^{-\frac 13 } \|f\|_2\right)
 \,. 
\end{equation}
Then we use \eqref{eq:309}, to get first that 
\begin{equation*}
  \|(\chi_\nu^\pm v)^\prime\|_2   \leq C  \Big( \frac{1}{x_\nu} +\frac{\beta^\frac 12\,
    (\beta x_\nu)^\frac 13}{\beta x_\nu^2} \Big)\| v\|_2 + C  \left( (\beta
    x_\nu)^{-\frac 13 }+ \frac{ (\beta x_\nu)^\frac 13}{ \beta x_\nu^2} \right)
  \|f\|_2\,,
\end{equation*}
 and then conclude from \eqref{eq:289} that
 \begin{equation}
\label{eq:312}
  \|(\chi_\nu^\pm v)^\prime\|_2 \leq C (x_\nu^{-1} \| v\|_2 + (\beta x_\nu)^{-\frac 13 } \|f\|_2)\,.
\end{equation}
Combining \eqref{eq:311} and \eqref{eq:312}  yields
\begin{displaymath}
   \|v^\prime\|_2\leq C\big((x_\nu^{-1}+[\beta x_\nu^7]^{-1/3})\|v\|_2+[\beta
   x_\nu]^{-1/3}\|f\|_2\big) + \widehat C a^{-\frac 23}\,  \| v^\prime\|_2  \,. 
\end{displaymath}
Using again \eqref{eq:289} we obtain the existence of  $a_0$ that for $a\geq a_0$
\begin{displaymath}
   \|v^\prime\|_2\leq C(x_\nu^{-1}\|v\|_2+[\beta x_\nu]^{-1/3}\|f\|_2) \,.
\end{displaymath}
Substituting \eqref{eq:310} into the above yields the existence of
$a_0>0$ that for all $a\geq a_0$ 
\begin{equation}\label{eq:313}
  \|v^\prime\|_2\leq C[\beta x_\nu]^{-1/3}\|f\|_2 \,.
\end{equation}
 By \eqref{eq:310} and \eqref{eq:313} we then obtain
\begin{displaymath}
  \|v\|_2\leq   C\big([\beta x_\nu]^{-2/3}+[\beta^2x_\nu^5]^{-1/3}[\beta x_\nu]^{-1/3}\big)\|f\|_2\,, 
\end{displaymath}
which implies,  using \eqref{eq:289},
\begin{equation}\label{eq:314}
  \|v\|_2\leq   C\, [\beta x_\nu]^{-2/3} \|f\|_2\,. 
\end{equation}
Having in mind \eqref{eq:290} we finally obtain from
\eqref{eq:313} and \eqref{eq:314} 
\begin{equation}
  \label{eq:315}
\|v\|_2 + [\Jf_\nu \beta]^{-1/3}\|v^\prime\|_2 \leq C[{\mathfrak
J}_\nu \beta]^{-2/3}\|f\|_2\,,
\end{equation}
which is precisely  \eqref{eq:288}.\\[1.5ex]
{\em Step  1.2: The case when $\nu\leq U(1/2)$. \\}

We recall that $x_\nu =1$ for $\nu < 0$ and observe that $x_\nu \geq \frac
12$ in this Step . Hence, we need to define only a pair of cutoff
functions. We thus set
\begin{equation}\label{eq:defchi2}
  \chi_2(x)=\hat{\chi}(2x)\,,
\end{equation}
and $\tilde{\chi}_2=\sqrt{1- \chi_2^2}$, which is supported on $[0,1/4]$. \\
Then, we
may write as in \eqref{eq:309}
\begin{equation}
\label{eq:316}
  \|\tilde{\chi}_2v\|_2 \leq C \beta^{-1} [\|f\|_2+\beta^{1/2} \| v\|_2]\,.
\end{equation}
Similarly, we obtain as in \eqref{eq:312}
\begin{displaymath}
  \|(\tilde{\chi}_2v)^\prime\|_2 \leq \frac{C}{\beta^{1/2}}[\|f\|_2+\beta\| v\|_2]\,.
\end{displaymath}
Suppose that 
\begin{equation}
\label{eq:317}
 \Upsilon<[U^\prime(1/4)/|U^\prime(1)|]\Re\hat{\kappa}_1\,.
 \end{equation}
 For later reference we note that \eqref{eq:317} implies that
 $\Upsilon<[U^\prime(1/4)/\Jf_\nu]^{2/3}\Re\hat{\kappa}_1$.  As in \eqref{eq:304} we can
 also conclude, from \cite[Proposition 5.2]{almog2019stability}, that
\begin{displaymath}
    \|\chi_2v\|_2+\beta^{-1/3} \|(\chi_2v)^\prime\|_2\leq\frac{C}{\beta^{2/3}}(\|f\|_2+\|v^\prime\|_2+
  \|v\|_2)\,. 
\end{displaymath}
Combining the above we may proceed as in the step 1.1 to conclude
\eqref{eq:288} and the $L^2$-bound on the right hand side in
(\ref{eq:287}a).\\ 

To complete the proof of (\ref{eq:287}a) we need to establish an
$\LL(L^2,L^\infty)$ bound for $(\LL_\beta^\Nf-\beta\lambda)^{-1}$. \\

{\em Step 2: For
\begin{equation}
\label{eq:318}\Upsilon<[|U^\prime(1/16)|/|U^\prime(1)|]^{2/3}\Re\hat{\kappa}_1\,,
\end{equation}
we prove that there exist positive $\beta_0$, $a_0$, and $C$ such that
for all $\beta \geq \beta_0$, \break $U(0)-\nu \geq a \beta^{-\frac 12}$, $a\geq a_0$ and $\mu
\leq\Upsilon {\mathfrak J_\nu}^{2/3}\beta^{-1/3}$ we have
\begin{equation}
\label{eq:319}
  \|(\LL_\beta^\Nf-\beta\lambda)^{-1}f\|_2+ [{\mathfrak
 J_\nu}\beta]^{-1/3}\|\frac{d}{dx}\, (\LL_\beta^\Nf-\beta\lambda)^{-1}f\|_2 \leq C \, [{\mathfrak
J_\nu}\beta]^{-5/6}\|f\|_\infty\,.
\end{equation}}

{\em Step  2.1: We consider the case  $0 < x_\nu\leq1/2$.\\}

Considering a pair $(v,f)\in D(\LL_\beta^\Nf)\times L^\infty(0,1)$ satisfying
$(\LL_\beta^\Nf-\beta\lambda)v=f$,  we then 
write as in \eqref{eq:303} 
 \begin{equation}
\label{eq:320}
  (\hat {\LL}_{\beta,\nu,\R}-\beta\lambda)(\tilde{\chi}_\nu v)=\tilde{\chi}_\nu f-2\tilde{\chi}_\nu^\prime v^\prime- \tilde{\chi}_\nu^{\prime\prime}v \,.
\end{equation}
Let $w_1\in D(\hat{\LL}_{\beta,\nu,\R})$ satisfy 
\begin{equation}
\label{eq:321}
   (\hat {\LL}_{\beta,\nu, \R}-\beta\lambda)w_1=\tilde{\chi}_\nu f \,.
\end{equation}
Let $ \tilde {\LL}_{\tilde \beta,\R}$ be defined by \eqref{eq:301}.  We
now apply \cite[Lemma 5.5]{almog2019stability} to the operator \break $\tilde
{\LL}_{\tilde \beta,\R}-i\tilde{\beta}\tilde{U}_\nu(1)$. Note that, due to to
\eqref{eq:299} and \eqref{eq:300}, there
exists $r>1$ such that the potential 
$y \mapsto \tilde U_\nu(y)-\tilde U_\nu(1)$ belongs to $\Sg_r^2$ (see \cite[Eq.
(2.32)]{almog2019stability} for the definition of this class).  Note
further that \cite[Lemma 5.5]{almog2019stability} holds under the
assumption $\Upsilon >0$.  Hence, for any $\tilde a >0$ there exists of
$\tilde C$ such that
\begin{equation}
\label{eq:322}
   \sup_{\Re \tilde \lambda\leq\Upsilon \tilde \beta^{-1/3}}\|( \tilde {\LL}_{\tilde
     \beta,\R}-\tilde \beta \tilde \lambda)^{-1}\tilde g \|_{L^2(-\tilde a
     ,\tilde a) } \leq 
   \frac{\tilde C}{\tilde \beta^{5/6}} \|\tilde g\|_\infty \,.  
\end{equation}
We apply \eqref{eq:322} with $\tilde a =2$,  $\tilde g (y) =
x_\nu^{1/2} ( \chi_\nu  f) ( x_\nu y)$, $\tilde \lambda= x_\nu^{-2}  \lambda   $ and $\tilde
\beta = \beta x_\nu^4 $ 
 to establish after a change of variable that
\begin{equation}
\label{eq:323}  
  \|w_1\|_{L^2(0,2 x_\nu)} \leq\frac{C}{[\beta\Jf_\nu]^{5/6}}\|f\|_\infty\,. 
\end{equation}
Let further $w_2\in D(\hat{\LL}_{\beta,\nu, \R})$ satisfy
\begin{displaymath}
(\hat {\LL}_{\beta,\nu, \R}-\beta\lambda)w_2=-2\tilde{\chi}_\nu^\prime v^\prime- \tilde{\chi}_\nu^{\prime\prime}v \,.
\end{displaymath}
 From \eqref{eq:297}  we get that
\begin{equation}\label{eq:324} 
 \|w_2\|_2\leq  \frac{C}{[\beta\Jf_\nu]^{2/3}}(\|\tilde{\chi}_\nu^\prime v^\prime\|_2+
  \|\tilde{\chi}_\nu^{\prime\prime}v\|_2)\,. 
\end{equation}
Combining \eqref{eq:324}  with \eqref{eq:323} yields as $\tilde{\chi}_\nu
v=w_1+w_2$ and $\supp \tilde \chi_\nu \subset (0,2 x_\nu)$
 \begin{multline}
\label{eq:325} 
     \|\tilde{\chi}_\nu v\|_2 =   \|\tilde{\chi}_\nu v\|_{L^2(0,2x_\nu)}
     \leq C\,\Big([\beta x_\nu]^{-5/6}\|f\|_\infty \\ + 
     [\beta^2x_\nu^5]^{-1/3}[\|{\mathbf
       1}_{(\frac{x_\nu}{2},\frac{3x_\nu}{4})}v^\prime\|_2+\|{\mathbf
       1}_{(\frac{5x_\nu}{4},\frac{3x_\nu}{2}}v^\prime\|_2 
  + x_\nu^{-1}\|v\|_2]\Big)\,. 
 \end{multline}
 Given the support of $\chi_\nu^- $ it holds that
\begin{displaymath}
  \|\chi_\nu^- f\|_2\leq Cx_\nu^{1/2}\|f\|_\infty \,,
\end{displaymath} 
and hence we can conclude from \eqref{eq:309}   that 
\begin{equation}
\label{eq:326}
  \|\chi_\nu^- v\|_2 \leq \frac{C}{\beta
    x_\nu^2}\,\Big[x_\nu^{1/2}\|f\|_\infty+\beta^{1/2}\|v\|_2\Big]\,.
\end{equation}
To obtain a similar bound for $\chi_\nu^+ v$ we obtain
 with the aid \eqref{eq:186} 
\begin{displaymath}
   \|\chi_\nu^+|U-\nu|^{-1/2}\|_2^2 \leq
   C\int_{\frac{3x_\nu}{2}}^1\frac{dx}{x^2-x_\nu^2}\leq\frac{C}{x_\nu} \,.
\end{displaymath}
Consequently, we can conclude that
\begin{equation}
\label{eq:327}
  \|(\chi_\nu^+)^2 v\|_1 \leq   \|\chi_\nu^+|U-\nu|^{-1/2}\|_2
  \|\chi_\nu^+|U-\nu|^{1/2}v\|_2 \leq \frac{C}{x_\nu^{1/2}}   \|\chi_\nu^+|U-\nu|^{1/2}v\|_2\,.
\end{equation}

 We now use \eqref{eq:308} to obtain that
\begin{equation*}
  \|\chi_\nu^+|U-\nu|^{1/2}v\|_2^2 \leq C \beta^{-1} (x_\nu^{-1}\|v\|_2\|(\chi_\nu^+
  v)^\prime\|_2 + \|(\chi_\nu^+)^2 v\|_1\|f\|_\infty) \,,
\end{equation*}
which implies that for any $\eta >0$, we have
\begin{equation}
\label{eq:328}
  \|\chi_\nu^+|U-\nu|^{1/2}v\|_2^2 \leq C \beta^{-1} (\eta x_\nu^{-2}\|v\|_2^2 + \frac 1 \eta \|(\chi_\nu^+
  v)^\prime\|_2^2 + \|(\chi_\nu^+)^2 v\|_1\|f\|_\infty) \,.
\end{equation}
By \eqref{eq:307} and \eqref{eq:289} we have that
\begin{equation}\label{eq:329}
   \|(\chi_\nu^+v)^\prime\|_2^2\leq C[\beta x_\nu]^{2/3}\|v\|_2^2+ \|(\chi_\nu^+)^2
   v\|_1\|f\|_\infty \,.  
\end{equation}
 Note that by \eqref{eq:289} we can conclude that
$x_\nu^{-2}\leq[\beta x_\nu]^{2/3}$.
Substituting \eqref{eq:329} into \eqref{eq:328} yields for any $\eta >0$
\begin{displaymath}
  \|\chi_\nu^+|U-\nu|^{1/2}v\|_2^2 \leq  C \beta^{-1} \left((\eta x_\nu^{-2}+
    \frac 1 \eta (\beta x_\nu)^{2/3}) \|v\|_2^2 + (\frac 1 \eta+1)
    \|(\chi_\nu^+)^2 v\|_1\|f\|_\infty\right) \,.  
\end{displaymath}
Setting $\eta_\nu  = x_\nu (\beta x_\nu)^\frac 13$ we observe that $\eta_\nu \geq
\frac 1C$ by \eqref{eq:289}, and hence
\begin{equation}\label{eq:330}
    \|\chi_\nu^+|U-\nu|^{1/2}v\|_2^2 \leq C \beta^{-1} \left(\beta^{1/3}x_\nu^{-2/3}\|v\|_2^2+ \|(\chi_\nu^+)^2 v\|_1\|f\|_\infty\right ) \,. 
\end{equation}
We then obtain, for any $\rho >0$, 
\begin{displaymath}
   \|\chi_\nu^+|U-\nu|^{1/2}v\|_2^2 \leq    C\left([\beta
     x_\nu]^{-2/3}\|v\|_2^2+ \rho \beta^{-2} \|(\chi_\nu^+)^2 v\|_1 + \frac 1
     \rho \|f\|_\infty^2\right) 
\end{displaymath}
which leads with the aid of \eqref{eq:327} to
\begin{displaymath}
   \|\chi_\nu^+|U-\nu|^{1/2}v\|_2^2 \leq    \hat C\left([\beta
     x_\nu]^{-2/3}\|v\|_2^2+ \rho \beta^{-2} x_\nu^{-1}
     \|\chi_\nu^+|U-\nu|^{1/2}v\|_2^2 + \frac 1 \rho \|f\|_\infty^2\right)\,. 
\end{displaymath}
Setting $\rho = [2 \hat C]^{-1} \beta^{2} x_\nu$ finally leads to 
\begin{equation}
\label{eq:331}
   \|\chi_\nu^+|U-\nu|^{1/2}v\|_2^2 \leq
   C\left([\beta x_\nu]^{-2/3}\|v\|_2^2+[\beta^2x_\nu]^{-1}\|f\|_\infty^2\right)
\end{equation}
Since for some positive $C$ it holds by \eqref{eq:186} that
$\chi_\nu^+|U-\nu|^{1/2}\geq C^{-1}x_\nu\chi_\nu^+$ we can conclude that
\begin{equation}
\label{eq:332}
     \|\chi_\nu^+v\|_2\leq C(\beta^{-1/3}x_\nu^{-4/3}\|v\|_2+\beta^{-1}x_\nu^{-3/2}\|f\|_\infty)\,.
\end{equation}
Combining \eqref{eq:326}  with \eqref{eq:325} and \eqref{eq:332} then
yields with the aid of \eqref{eq:289}
\begin{displaymath}
\begin{array}{ll}
  \|v\|_2& \leq
  C\, \Big(([\beta x_\nu^{3/2}]^{-1}+[\beta x_\nu]^{-5/6})\|f\|_\infty + [\beta x_\nu^4]^{-1/3} \|v\|_2\\& \hspace{40mm} +[\beta^2x_\nu^5]^{-1/3}[ \|{\mathbf
       1}_{(\frac{x_\nu}{2},\frac{3x_\nu}{4})}v^\prime\|_2 + \|{\mathbf
       1}_{(\frac{5x_\nu}{4},\frac{3x_\nu}{2})}v^\prime\|_2] \Big)\\
       &\leq  
  \hat C\, \Big([\beta x_\nu]^{-5/6} \|f\|_\infty + a^{-1} \|v\|_2+[\beta^2x_\nu^5]^{-1/3}[ \|{\mathbf
       1}_{(\frac{x_\nu}{2},\frac{3x_\nu}{4})}v^\prime\|_2+ \|{\mathbf
       1}_{(\frac{5x_\nu}{4},\frac{3x_\nu}{2})}v^\prime\|_2] \Big) \,.
       \end{array}
\end{displaymath}
Hence, there exist $a_0 >0$ and $C>0$ such that for all $a\geq a_0$ 
\begin{equation}
\label{eq:333}
  \|v\|_2\leq  C\,\left([\beta x_\nu]^{-5/6}\|f\|_\infty + [\beta^2x_\nu^5]^{-1/3}[\|{\mathbf
       1}_{(\frac{x_\nu}{2},\frac{3x_\nu}{4})}v^\prime\|_2+ \|{\mathbf
       1}_{(\frac{5x_\nu}{4},\frac{3x_\nu}{2})}v^\prime\|_2] \right)\,.
\end{equation}
Set
\begin{displaymath}
  \check \chi_\nu^\pm(x)=\hat{\chi}(2(x/x_\nu-1)){\mathbf 1}_{\R_+}(\pm(x-x_\nu)) \,,
\end{displaymath}
where $\hat \chi$ is defined by \eqref{eq:41}.  We note that by its
definition $\check \chi_\nu^- =1$ on $[0,\frac{3x_\nu}{4}) $ and ${\rm
  supp\,} \check \chi_\nu^- \subset (-\infty, \frac{7 x_\nu}{8})$.  Similarly,
$\check \chi_\nu^+ =1$ on $[\frac{5x_\nu}{4},1]$ and ${\rm supp\,} \check
\chi_\nu^+ \subset (
\frac{9x_\nu}{8},+\infty)$.\\
Proceeding as in the proof of \eqref{eq:307} integration by parts
yields, since $v$ satisfies a Neumann condition at $x=0$ and Dirichlet
condition at $x=1$, and since we have $(\check \chi_\nu^-)^\prime(0)=0$,
 \begin{equation*}
  \|(\check \chi_\nu^\pm  v)^\prime\|_2^2=\|(\check \chi_\nu^\pm)^\prime
  v\|_2^2+\mu\beta\|\check \chi_\nu^\pm v\|_2^2 +
  \Re\langle \check \chi_\nu^\pm v,\check \chi_\nu^\pm f\rangle \,.
\end{equation*}
The above identity implies, as $\mu<\Upsilon \mathfrak J_\nu \beta^{-\frac 13} $,  
\begin{displaymath}
    \|(\check \chi_\nu^\pm v)^\prime\|_2^2\leq C([x_\nu\beta]^{2/3}+x_\nu^{-2})\|v\|_2^2+
  \||U-\nu|^{1/2}\check \chi_\nu^\pm v\|_2 \||U-\nu|^{-1/2}\check \chi_\nu^\pm\|_2 \|f\|_\infty  \,.
\end{displaymath}
 By \eqref{eq:186} there exists $0<\nu_1<U(0)$ such that for all
$\nu_1<\nu<U(0)-a\beta^{-1/2}$ it holds that
\begin{equation}
  \label{eq:334}
\||U-\nu|^{-1/2}\check \chi_\nu^-\|_2^2 \leq C\int_0^{7x_\nu/8}\frac{dx}{x_\nu^2-x^2}\leq\frac{C}{x_\nu}\,.
\end{equation}
 Similarly,
\begin{equation}
  \label{eq:335}
\||U-\nu|^{-1/2}\check \chi_\nu^+\|_2^2 \leq C\int_{9x_\nu/8}^1\frac{dx}{x_\nu^2-x^2}\leq\frac{C}{x_\nu}\,.
\end{equation}
For $0<\nu<\nu_1$ \eqref{eq:334} and \eqref{eq:335} still hold given the
support of $\check \chi_\nu^\pm$. Consequently, we may conclude that
\begin{equation}
\label{eq:336}
      \|(\check \chi_\nu^\pm  v)^\prime\|_2^2\leq C[([x_\nu\beta]^{2/3}+x_\nu^{-2})\|v\|_2^2+
  x_\nu^{-1/2}\||U-\nu|^{1/2}\check \chi_\nu^\pm v\|_2 \|f\|_\infty  \,. 
\end{equation}
As in \eqref{eq:308} it holds that
\begin{equation*}
  \mp\beta\| |U-\nu|^{1/2}\check \chi_\nu^\pm v\|_2^2+2\Im\langle(\check \chi_\nu^\pm)^\prime
  v,(\check \chi_\nu^\pm v)^\prime\rangle =  \Im \langle \check \chi_\nu^\pm v,\check \chi_\nu^\pm f\rangle \,,
\end{equation*}
which implies
\begin{displaymath}
   \beta\| |U-\nu|^{1/2}\check \chi_\nu^\pm v\|_2^2\leq
   \frac{C}{x_\nu}\|v\|_2\|(\check \chi_\nu^\pm
   v)^\prime\|_2 +  \||U-\nu|^{1/2}\check \chi_\nu^\pm v\|_2
   \||U-\nu|^{-1/2}\check \chi_\nu^\pm\|_2 \|f\|_\infty  \,.
\end{displaymath}
Consequently, by \eqref{eq:334}, 
\begin{equation}
\label{eq:337}
  \| |U-\nu|^{1/2}\check \chi_\nu^\pm v\|_2^2\leq C([\beta x_\nu]^{-1}\|v\|_2\|(\check \chi_\nu^\pm
   v)^\prime\|_2 +  [\beta^2x_\nu]^{-1}\|f\|_\infty^2 )\,.
\end{equation}
Substituting \eqref{eq:337} into \eqref{eq:336} then yields, with the
aid of \eqref{eq:289}, 
\begin{displaymath}
      \|(\check \chi_\nu^\pm v)^\prime\|_2^2\leq
      C\left([\beta x_\nu]^{2/3}\|v\|_2^2+[\beta x_\nu]^{-1}\|f\|_\infty^2+
       [\beta^{1/2}x_\nu]^{-1} \|v\|_2^{1/2}\|(\check \chi_\nu^\pm v)^\prime\|_2^{1/2}\|f\|_\infty\right)  \,.
\end{displaymath}
By the above inequality we may conclude, first, that 
\begin{displaymath}
      \|(\check \chi_\nu^\pm v)^\prime\|_2^2\leq
      C\left([\beta x_\nu]^{2/3}\|v\|_2^2+ 2 [\beta x_\nu]^{-1}\|f\|_\infty^2 + x_\nu^{-1}   ( \|(\check \chi_\nu^\pm v)^\prime\|_2  \|v\|_2) \right)  \,,
\end{displaymath}
and then, for any $\eta>0$, 
\begin{displaymath}
      \|(\check \chi_\nu^\pm v)^\prime\|_2^2\leq
      C\left([\beta x_\nu]^{2/3}\|v\|_2^2+ 2 [\beta x_\nu]^{-1}\|f\|_\infty^2+ \eta  \|(\check \chi_\nu^\pm v)^\prime\|_2^2 + \frac{1}{\eta} x_\nu^{-2} \|v\|_2^2)  \right) \,.
\end{displaymath}
Using again  \eqref{eq:289}, for  $\eta$ small enough, we finally obtain
\begin{displaymath}
      \|(\check \chi_\nu^\pm v)^\prime\|_2^2\leq
      C\left([\beta x_\nu]^{2/3}\|v\|_2^2+ 2 [\beta x_\nu]^{-1}\|f\|_\infty^2 \right) \,.
\end{displaymath}
From the above it can be easily verified that
\begin{multline}
  \label{eq:338}
  \|{\mathbf 1}_{(\frac{x_\nu}{2},\frac{3x_\nu}{4})}v^\prime\|_2  +
  \|{\mathbf 1}_{(\frac{5x_\nu}{4},\frac{3x_\nu}{2})}v^\prime\|_2     \\ \leq 
  \|(\check \chi_\nu^+ v)^\prime\|_2 + \|(\check \chi_\nu^- v)^\prime\|_2\\ \leq C([\beta x_\nu]^{1/3}\|v\|_2+[\beta x_\nu]^{-1/2}\|f\|_\infty)\,. 
\end{multline}
 Substituting \eqref{eq:338} into \eqref{eq:333} yields, using \eqref{eq:289}
\begin{displaymath}
   \|v\|_2\leq C \left([\beta x_\nu]^{-5/6}\|f\|_\infty+[\beta x_\nu^4]^{-1/3}\|v\|_2\right)\,.
\end{displaymath}
Hence,  there exists  $a_0>0$ such that we obtain  for $a\geq a_0$
\begin{equation}
\label{eq:339}
    \|v\|_2\leq C\, [\beta x_\nu]^{-5/6}\|f\|_\infty\,.
\end{equation}

{\em Step  2.2: The case  $x_\nu\geq 1/8\,$.}\\
Let 
 \begin{equation}\label{eq:340}
 \Upsilon<[|U^\prime(1/16)|/|U^\prime(1)|]^{2/3}\Re\hat{\kappa}_1\,.
 \end{equation}
 We set
\begin{equation}
\label{eq:341}
   \hat{\eta}_\nu =\chi(-(x-x_\nu)/x_\nu)\,,
\end{equation}
which is supported on $(x_\nu/4,1]$ and satisfies $ \hat{\eta}_\nu \equiv1$ on
$(x_\nu/2,1]$.  Then, as
\begin{equation}
\label{eq:342}
  (\LL_\beta^\Nf-\beta\lambda)(\hat{\eta}_\nu v)=\hat{\eta}_\nu f-2\hat{\eta}_\nu^\prime v^\prime- \hat{\eta}_\nu^{\prime\prime}v \,,
\end{equation}
we may use \cite[Proposition 5.2 and Proposition
5.4]{almog2019stability} (both  hold  for  $U\in \Sg_r^2$  though stated
for $U\in \Sg_r^4$), to obtain that
 \begin{equation}
\label{eq:343}
     \|\hat{\eta}_\nu v\|_2\leq\frac{C}{\beta^{2/3}}(\beta^{-1/6}\|f\|_\infty+\|\hat{\eta}_\nu^\prime v^\prime\|_2+
  \|\hat{\eta}_\nu^{\prime\prime}v\|_2)\,. 
 \end{equation}
We note that \eqref{eq:340} implies
$\Upsilon<[|U^\prime(x_\nu/4)|/\Jf_\nu]^{2/3}\Re\hat{\kappa}_1$ for all $1/8\leq x_\nu\leq1$.
Let
$\tilde{\eta}_\nu=\sqrt{1-\hat{\eta}_\nu^2}\in C^\infty(\R,[0,1])$. Note that
$\tilde{\eta}_\nu$ is supported on $[0,x_\nu/2]$. Consequently, we may
obtain, as in \eqref{eq:326} but for $x_\nu \geq \frac 18$, 
\begin{displaymath}
  \|\tilde{\eta}_\nu v\|_2 \leq C \beta^{-1} [\|f\|_2+\beta^{1/2}\|v\|_2]\,.
\end{displaymath}
Combining the above with \eqref{eq:343} yields 
\begin{displaymath}
  \|v\|_2\leq \frac{C}{\beta^{2/3}}\left(\beta^{-1/6}\|f\|_\infty+\|{\mathbf 1}_{(x_\nu/4,x_\nu/2)}v^\prime\|_2\right)\,.
\end{displaymath}
We now use a variant of \eqref{eq:338}, (which is valid also for
$x_\nu>1/8$) to bound $\|{\mathbf 1}_{(x_\nu/4,x_\nu/2)}v^\prime\|_2$ to
obtain, with the aid of \eqref{eq:289},
\begin{displaymath}
  \|v\|_2\leq \frac{C}{\beta^{5/6}}\, \|f\|_\infty\,.
\end{displaymath}
Together with \eqref{eq:288} the above inequality establishes
(\ref{eq:287}a). \\

{\em Step 3: We prove (\ref{eq:287}b), when
\begin{equation}
\label{eq:344}
\Upsilon < \inf_{x_\nu\in[0,1]}(|U^\prime(x_\nu/2)|/|U^\prime(x_\nu)|)^{2/3} \Re \hat
\kappa_1\,.
\end{equation}}

Note that for $p=2$ (\ref{eq:287}b) readily follows from
\eqref{eq:288}.  In the following we then assume $p\in(1,2)$.  Suppose
first that $x_\nu<1/2$. As above, we consider a pair $(v,f)$ in
$D(\mathcal L_\beta^\Nf)\times L^2(0,1)$ satisfying $(\mathcal L_\beta^{\Nf} - \lambda
\beta)v=f$. Let
\begin{displaymath}
  \hat{\LL}^\Df_\beta:H^2(x_\nu/2,3x_\nu/2)\cap
H^1_0(x_\nu/2,3x_\nu/2)\to L^2(x_\nu/2,3x_\nu/2)
\end{displaymath}
be associated with the
same differential operator as $\LL^\Nf_\beta$.  Let
\begin{displaymath}
  \tilde \LL_{\tilde
  \beta}^\Df:H^2(1/2,3/2)\cap H^1_0(1/2,3/2)\to L^2(1/2,3/2)
\end{displaymath} 
be associated with the same differential operator as $\tilde
\LL_{\tilde \beta,\mathbb R} $ in \eqref{eq:301}. We recall from
\cite[Proposition 5.2]{almog2019stability} that for any $g\in
L^2(1/2,3/2)$ it holds
\begin{equation}
\label{eq:345}
     \sup_{\Re \tilde \lambda\leq\Upsilon \tilde \beta^{-1/3}}
  \tilde \beta^{-1/3} \Big\|\frac{d}{dy}( \tilde {\LL}_{\tilde \beta}^\Df- \tilde \beta \tilde \lambda)^{-1}g\Big\|_p\leq
   \frac{C}{\tilde \beta^{\frac{2+p}{6p}}}\|g\|_2\,.
\end{equation}
As in \eqref{eq:303} it holds that
\begin{displaymath}
   ( \hat {\LL}_\beta^\Df -\beta\lambda)(\tilde{\chi}_\nu v)=\tilde{\chi}_\nu f-2\tilde{\chi}_\nu^\prime v^\prime- \tilde{\chi}_\nu^{\prime\prime}v \,,
\end{displaymath}
and hence by applying the inverse of the dilation \eqref{eq:296} to
\eqref{eq:345} we can conclude that
\begin{equation}
  \|(\tilde{\chi}_\nu v)^\prime\|_p\leq C [\beta
    x_\nu]^{-\frac{2+p}{6p}}\, (\|f\|_2+\|\tilde{\chi}_\nu^\prime v^\prime\|_2+
  \|\tilde{\chi}_\nu^{\prime\prime}v\|_2)\,.
\end{equation}
  By \eqref{eq:288} we then obtain
\begin{displaymath}
   \|(\tilde{\chi}_\nu v)^\prime\|_p \leq C [\beta
    x_\nu]^{-\frac{2+p}{6p}} \, (1+\beta^{-1/3}x_\nu^{-4/3}+\beta^{-2/3}x_\nu^{-8/3})\|f\|_2\,. 
\end{displaymath}
From \eqref{eq:289} we easily conclude that
\begin{equation}
\label{eq:346}
    \|(\tilde{\chi}_\nu v)^\prime\|_p \leq C [\beta
    x_\nu]^{-\frac{2+p}{6p}} \, \|f\|_2\,.
\end{equation}
We now seek an estimate for $\chi_\nu^\pm  v^\prime$. To this end, we use
integration by parts to obtain 
\begin{multline}\label{eq:347}
  \Re \langle(\chi_\nu^\pm)^2(U-\nu)v,(\LL_\beta^\Nf-\beta\lambda)v\rangle
  = \mp \|\chi_\nu^\pm |U-\nu|^{1/2}v^\prime\|_2^2 + \\
  \Re\langle \chi_\nu^\pm(U^\prime \chi_\nu^\pm +2(U-\nu)(\chi_\nu^\pm)^\prime)v ,v^\prime\rangle-\mu\beta\|\chi_\nu^\pm |U-\nu|^{1/2}v\|_2^2\,.
 \end{multline}
 Since
\begin{equation}
\label{eq:348}
  |U^\prime(x)|\leq C|U(x)-U(0)|^{1/2}\leq C\left(|U(x)-\nu|^{1/2}+|U(0)-\nu|^{1/2}\right) \,,
\end{equation}
we can conclude that 
\begin{displaymath}
   \|(\chi_\nu^\pm)^2U^\prime v\|_2\leq C\left(x_\nu\|v\|_2 + \|\chi_\nu^\pm |U-\nu|^{1/2}v\|_2\right)\,.
\end{displaymath}
 
Furthermore, given the support of $(\chi_\nu^\pm)^\prime$ we obtain by
\eqref{eq:183}
\begin{displaymath}
   \|\chi_\nu^\pm (U-\nu)(\chi_\nu^\pm)^\prime\|_\infty \leq C
   \|(x^2-x_\nu^2)(\chi_\nu^\pm)^\prime\|_\infty \leq Cx_\nu \,,
\end{displaymath} 
Combining the above  with \eqref{eq:347} yields  that
\begin{multline}
\label{eq:349}
   \|\chi_\nu^\pm |U-\nu|^{1/2}v^\prime\|_2^2\leq \|(U-\nu)v\|_2\|f\|_2  \\
 +  C\,\left([\beta x_\nu]^{2/3}\|\chi_\nu^\pm|U-\nu|^{1/2}v\|_2^2+ [x_\nu\|v\|_2+ \|\chi_\nu^\pm |U-\nu|^{1/2}v\|_2]\|v^\prime\|_2\right)\,.
\end{multline}

As
\begin{equation}
\label{eq:350}
     \Im \langle(U-\nu)v,(\LL_\beta^\Nf-\beta\lambda)v\rangle=
     \beta\|(U-\nu)v\|_2^2+ \Im\langle U^\prime v,v^\prime\rangle\,, 
\end{equation}
we obtain by \eqref{eq:348}  that
\begin{displaymath}
  \beta\|(U-\nu)v\|_2^2 \leq
  C\left(\beta^{-1}\|f\|_2^2+[\||U(x)-\nu|^{1/2}v\|_2+x_\nu\|v\|_2]\|v^\prime\|_2\right)\,.
\end{displaymath}
Furthermore, since
\begin{equation}
\label{eq:351}
  \||U-\nu|^{1/2}v\|_2^2\leq \frac{1}{2}\big[x_\nu^{-2}\|(U-\nu)v\|_2^2+ x_\nu^2\|v\|_2^2\big]\,,
\end{equation}
we can write
\begin{displaymath}
  \beta\|(U-\nu)v\|_2^2 \leq
  C\left(\beta^{-1}\|f\|_2^2+[x_\nu\|v\|_2+x_\nu^{-1}\|(U-\nu)v\|_2]\|v^\prime\|_2\right)\,.
\end{displaymath}
Hence,
\begin{equation}
\label{eq:352}
  \beta\|(U-\nu)v\|_2^2 \leq
  C\left( \beta^{-1} \|f\|_2^2+ x_\nu\|v\|_2 \|v^\prime\|_2 +  \beta^{-1} x_\nu^{-2} \|v^\prime\|_2^2\right)  \,.
\end{equation}
Using  \eqref{eq:288} gives the following estimates for $\|v\|_2$ and
$\|v^\prime\|_2$ 
\begin{equation}
\label{eq:353}
\| v^\prime\|_2 \leq C \, 
[x_\nu\beta]^{-1/3}\|f\|_2\,\mbox{ and } 
\| v\|_2 \leq C \, 
[x_\nu\beta]^{-2/3}\|f\|_2\,.
\end{equation}
Substituting the above  into \eqref{eq:352} yields
\begin{displaymath}
  \beta\|(U-\nu)v\|_2^2 \leq
  C( \beta^{-1} +  \beta^{-1} x_\nu^{-2} [x_\nu\beta]^{-2/3}     )\|f\|_2^2 \,.
\end{displaymath}
From the above, recalling that $\beta^\frac 14 x_\nu$ is bounded from
below, we conclude that
\begin{equation}
\label{eq:354}
  \|(U-\nu)v\|_2\leq C \beta^{-1} \|f\|_2\,.
\end{equation}
Next we write, using \eqref{eq:353} and  (\ref{eq:354}),
\begin{equation}\label{eq:355}
   \||U-\nu|^{1/2}v\|_2^2\leq
   \frac{1}{2}[\beta^{1/3}x_\nu^{-2/3}\|(U-\nu)v\|_2^2+\beta^{-1/3}x_\nu^{2/3}
   \|v\|_2^2]\leq\frac{C}{\beta^{5/3}x_\nu^{2/3}}\|f\|_2^2\,.
\end{equation}
Substituting \eqref{eq:355}, together with \eqref{eq:354} into
\eqref{eq:349} yields, with the aid of \eqref{eq:353} 
\begin{equation}
\label{eq:356}
   \|\chi_\nu^\pm |U-\nu|^{1/2}v^\prime\|_2^2\leq C \beta^{-1} \|f\|_2^2 \,.
\end{equation}
We now observe that
\begin{equation}\label{eq:357}
\||U-\nu|^{-1/2}\chi_\nu^-\|_q^q \leq C\int_0^{3x_\nu/4}\frac{dx}{[x_\nu^2-x^2]^{q/2}}\leq\frac{C}{x_\nu^{q-1}}\,.
\end{equation}
 Similarly,
\begin{equation}\label{eq:358}
\||U-\nu|^{-1/2}\chi_\nu^+\|_2^2 \leq C\int_{5x_\nu/4}^1\frac{dx}{[x_\nu^2-x^2]^{q/2}}\leq\frac{C}{x_\nu^{q-1}}\,,
\end{equation}
which is obtained with the aid of \eqref{eq:186} for $\nu >\nu_1$ (for
$\nu \leq \nu_1$ the above bounds are trivial).  Consequently, we obtain
that
\begin{multline}
\label{eq:359}
    \|(\chi_\nu^\pm) v^\prime\|_p\leq
    \|\chi_\nu^\pm |U-\nu|^{1/2}v^\prime\|_2\|[{\mathbf 1}_{[0,3x_\nu/4]}+ \\ {\mathbf
    1}_{[5x_\nu/4,1]}|U-\nu|^{-1/2}\|_{\frac{2p}{2-p}}\leq
    \frac{C}{\beta^{1/2}x_\nu^{\frac{3p-2}{2p}}}\|f\|_2 \,.
\end{multline}
 Similarly, we write that
\begin{displaymath}
  \|(\chi_\nu^\pm)^\prime v\|_p\leq
  \|(\chi_\nu^\pm)^\prime|U-\nu|^{1/2}v\|_2\|[{\mathbf 1}_{[0,3x_\nu/4]}+{\mathbf
    1}_{[5x_\nu/4,1]}|U-\nu|^{-1/2}\|_{\frac{2p}{2-p}}\leq 
  \frac{C}{x_\nu^{\frac{17}{6}-\frac{1}{p}}\beta^{5/6}}\|f\|_2\,.
\end{displaymath}
To obtain the second inequality we have used \eqref{eq:355}.
Together with \eqref{eq:359} and \eqref{eq:289} the above yields 
\begin{equation}
\label{eq:360}
   \|(\chi_\nu^\pm  v)^\prime\|_p\leq  C \beta^{-1/2}x_\nu^{-\frac{3p-2}{2p}}\, \|f\|_2 \,.
\end{equation}
Combining the above with \eqref{eq:346} yields the existence of $a_0$
such that  (\ref{eq:287}b) holds  for all $a \geq a_0$.

Consider next the case $x_\nu\geq1/2$ (in which no dilation
transformation is necessary). Let
  \begin{equation}
\label{eq:361} 
\Upsilon <(|U^\prime(1/4)|/|U^\prime(1)|)^{2/3} \Re\hat \kappa_1\,. 
\end{equation}
We now
set
\begin{displaymath}
  \iota_\nu =\sqrt{\tilde{\chi}_\nu^2+(\chi_\nu^+)^2}\,.
\end{displaymath}
Then, as
  \begin{displaymath}
   ( \hat {\LL}_\beta^\Df -\beta\lambda)(\iota_\nu v)=\iota_\nu f-2\iota_\nu^\prime v^\prime- \iota_\nu^{\prime\prime}v \,,
\end{displaymath}
We obtain using  \cite[Prop. 5.2]{almog2019stability} and
\eqref{eq:288} that
\begin{displaymath}
     \|(\iota_\nu v)^\prime\|_p \leq
   \frac{C}{\beta^{\frac{2+p}{6p}}}\|f\|_2\,.
\end{displaymath}
Since \eqref{eq:360} holds true for $\chi_\nu^-v$ in the case
$x_\nu\geq1/2\,$, we can combine it with the above to extend
the validity (\ref{eq:287}b)  to this case as well.\\

Given \eqref{eq:317}, \eqref{eq:318},
  \eqref{eq:340},  \eqref{eq:344} and \eqref{eq:361}
  it follows that there exists $\Upsilon >0$ for which Proposition  \ref{prop5.3}
holds true.
\end{proof}

\begin{remark}
  \label{rem:negative-mu}
As in \cite[Proposition 5.1]{almog2019stability} we can obtain better
estimates for the case where $\mu<0$. Thus setting $\chi_\nu^\pm\equiv1$ in
\eqref{eq:307} yields for $\mu<0$
\begin{equation}\label{eq:362}
  \|v^\prime\|_2^2+|\mu|\beta\|v\|_2^2 =
  \Re\langle v,f\rangle \,. 
\end{equation}
From this  we conclude that
\begin{equation}
  \label{eq:363}
\|v\|_2\leq [|\mu|\beta]^{-1} \|f\|_2 \,,
\end{equation}
which is stronger that \eqref{eq:287} when $|\mu|\gg\beta^{-1/3}$. \\
Note that for $\mu < 0$
\begin{equation}
  \label{eq:364}
\|v^\prime\|_2\leq  [|\mu| \beta]^{-1/2} \|f\|_2 \,.
\end{equation}
\end{remark}

\subsection{Resolvent estimates for $| U(0)-\Im \lambda| =\OO
  (\beta^{-1/2})$\,.}
  \label{sec:schrod-quad}
   In this case we will
  approximate $U-\nu$ by its quadratic potential  
  $$x \mapsto U^{\prime\prime}(0)x^2/2+U(0)-\nu$$ and then use a proper resolvent estimate
  established by R. Henry  in  \cite{Hen2}.\\
   More precisely, we prove.
\begin{proposition}
  \label{lem:schrod-quad} Let $U\in C^3([0,1])$ satisfy
  \eqref{eq:10},
  $a>0$ and  $\Upsilon<\sqrt{-U^{\prime\prime}(0)}/2$.  \\
  Then there exist $C>0$ and $\beta_0>0$ such that, for all
  $\beta\geq \beta_0$, 
  \begin{itemize}
  \item 
  if $f\in L^\infty(0,1)$,
   \begin{subequations}
\label{eq:365}
       \begin{multline}
\sup_{
  \begin{subarray}{c}
   \mu  \leq\Upsilon\beta^{-1/2} \\
U(0)-a\beta^{-1/2}<\nu <U(0)+a\beta^{-1/2}
  \end{subarray}} \Big(
\|(\LL_\beta^{\Nf,\Df}-\beta\lambda)^{-1}f\|_2  +
\beta^{-1/4}\| \frac{d}{dx} (\LL_\beta^{\Nf,\Df}-\beta\lambda)^{-1}f\|_2 \\
\qquad \qquad +\beta^{1/8}\|(\LL_\beta^{\Nf,\Df}-\beta\lambda)^{-1}f\|_1\Big) \\ \leq   C\min(\beta^{-1/2}\|f\|_2,\beta^{-5/8}\|f\|_\infty)\,, 
  \end{multline}
\item  if  $(x-x_\nu)^{-1}f\in L^2(0,1)$
\begin{multline}
\sup_{\begin{subarray}{c}
  \mu \leq\Upsilon\beta^{-1/2} \\
    U(0)-a\beta^{-1/2}<\nu <U(0)+a\beta^{-1/2}
  \end{subarray}}\Big(
\Big\|(\LL_\beta^{\Nf,\Df}-\beta\lambda)^{-1}f\Big\|_2 +
\beta^{-1/4}\Big\| \frac{d}{dx} (\LL_\beta^{\Nf,\Df}-\beta\lambda)^{-1}f\Big\|_2  \\ \qquad \qquad + \beta^{1/8}\Big\|(\LL_\beta^{\Nf,\Df} -\beta\lambda)^{-1}f\Big\|_1\Big) \\ \leq C\beta^{-3/4}\Big\|\frac{f}{x-x_\nu}\Big\|_2 \,.
\end{multline}   
\end{subequations}
\end{itemize}
\end{proposition}
\begin{proof}~\\
All the estimates established in this proof assume that
\begin{equation}\label{eq:366}
\mu\leq\Upsilon\beta^{-1/2} \mbox{ and }  U(0)-a\beta^{-1/2}<\nu <U(0)+a\beta^{-1/2}\,.
  \end{equation}
By the second condition it holds that
  \begin{equation} \label{eq:367} 
0 \leq x_\nu \leq C_a \beta^{ -\frac 14}\,.
  \end{equation}
Consequently, for any $\nu_1 < U(0)$,  there exists $\beta_0$ such that,
for all $\beta \geq \beta_0$, we have
   \begin{equation}\label{eq:368}
  \nu > \nu_1 \mbox{ and } x_\nu < 1/4\,.
  \end{equation}

{\it Step 1: $\LL(L^2,L^2)$ estimate.}\\[1.5ex]
  By \cite[Theorem 1.3]{Hen2} we immediately obtain that, under \eqref{eq:366}, 
\begin{equation}
\label{eq:369}
  \|(\LL_\beta^\Nf-\beta\lambda)^{-1}f\|_2\leq C\beta^{-1/2}\|f\|_2\,.
\end{equation}
To prove the second inequality  in (\ref{eq:365}a), 
let $f\in L^2(0,1)$ and
$v\in D(\LL_\beta^\Nf)$ satisfy $(\LL_\beta^\Nf-\beta\lambda)v=f$. Taking the scalar product with $v$, an integration by
parts  yields for the real part, with the aid of \eqref{eq:366}, 
\begin{displaymath}
   \|v^\prime\|_2^2= \mu\beta\|
   v\|_2^2+\Re\langle v,f\rangle\leq C(\beta^{1/2}\|v\|_2^2+\beta^{-1/2}\|f\|_2^2)\,.
\end{displaymath}
By \eqref{eq:369} we can then conclude that
\begin{equation}
  \label{eq:370}
 \Big\|\frac{d}{dx}(\LL_\beta^\Nf-\beta\lambda)^{-1}f\Big\|_2\leq C\beta^{-1/4}\|f\|_2\,,
\end{equation}
which together with \eqref{eq:369} establishes the $\LL(L^2,L^2)$
estimate in (\ref{eq:365}a). \\

{\it Step 2: $\LL(L^\infty,L^2)$ estimate.}\\[1.5ex]
Next, we obtain an $\LL(L^\infty,L^2)$ estimate for
$(\LL_\beta^\Nf-\beta\lambda)^{-1}$ under \eqref{eq:366}. 
Let $\hat{\chi}$ be given \eqref{eq:41}. As before, we set 
\begin{displaymath}
\chi_\gamma(x)=\hat{\chi}(\gamma\beta^{1/4}x)\,, 
\end{displaymath}
for some positive 
\begin{displaymath}
\gamma<[8 C_a+1]^{-1} < 1\,.
\end{displaymath} 
In particular $\gamma$ satisfies  
\begin{displaymath}
0 < \gamma < [8 \beta^{\frac  14} x_\nu +1]^{-1}\,.
\end{displaymath}
   We note that $\chi_\gamma$  satisfies
  \begin{equation}\label{eq:371}
   {\rm supp} \chi_\gamma \subset [2 x_\nu + \frac 14 \beta^{-\frac 14},1) \mbox{ and }  | \chi^\prime_\gamma| \leq C \gamma \beta^{\frac 14}\,.
   \end{equation}
 Set further
$\tilde{\chi}_\gamma=\sqrt{1-\chi_\gamma^2}\in C^\infty(\R)$ and  note that $\tilde \chi_\gamma$  satisfies
  \begin{equation}\label{eq:372}
   {\rm supp} \tilde \chi_\gamma \subset [0,  \frac 1{2\gamma}  \beta^{-\frac 14})  \mbox{ and }  |\tilde \chi^\prime_\gamma| \leq C \gamma \beta^{\frac 14}\, \,.
   \end{equation} 
Let $f\in L^2(0,1)$ and
$v\in D(\LL_\beta^\Nf)$ satisfy $(\LL_\beta^\Nf-\beta\lambda)v=f$.\\
 We begin by estimating $\chi_\gamma v$. An integration by
parts yields
\begin{equation}
\label{eq:373}
  \|(\chi_\gamma v)^\prime\|_2^2-\|\chi_\gamma^\prime v\|_2^2-\mu\beta\|\chi_\gamma v\|_2^2=\Re\langle\chi_\gamma v,\chi_\gamma f\rangle\,,
\end{equation}
from which we conclude, given that $\mu\beta^{1/2}$ and $\gamma$  are bounded
from above,
\begin{equation}
\label{eq:374}
    \|(\chi_\gamma v)^\prime\|_2^2\leq \|\chi_\gamma v\|_1\|f\|_\infty +C\beta^{1/2}\|v\|_2^2 \,.
\end{equation}
Furthermore, we have that 
\begin{equation}
\label{eq:375}
  -\beta\|\,|U-\nu|^{1/2}\chi_\gamma v\|_2^2+2\Im\langle\chi_\gamma^\prime v,(\chi_\gamma v)^\prime\rangle
  =\Im\langle\chi_\gamma v,\chi_\gamma f\rangle \,.
\end{equation}
Hence, with the aid of \eqref{eq:374}, we can conclude that 
\begin{equation}
\label{eq:376}
  \beta\||U-\nu|^{1/2}\chi_\gamma v\|_2^2\leq
  C \,\left(\gamma\beta^{1/2}\|v\|_2^2+\|\chi_\gamma v\|_1\|f\|_\infty\right) \,. 
\end{equation}
We now write (note that $\chi_\gamma \chi_{2\gamma}=\chi_\gamma$) 
\begin{equation}
\label{eq:377}
  \|\chi_\gamma v\|_1\leq
  \|\chi_{2\gamma}|U-\nu|^{-1/2}\|_2\||U-\nu|^{1/2}\chi_\gamma v\|_2\,.
\end{equation}
For  $x \in [ x_\nu + \frac 18 \beta^{-\frac
  14},1)$, it holds by \eqref{eq:247}, \eqref{eq:248} and \eqref{eq:368}
\begin{equation}
\label{eq:378}
  |U(x) -\nu| \geq \frac 1C (x-x_\nu)^2 \geq \frac{1}{4C}  \beta^{-\frac
  12}\,,
  \end{equation}
  which implies, for $\beta \geq \beta_0$, 
\begin{displaymath}
   \int_{x_\nu + \frac{1}{8} \beta^{-\frac
       {1}{4}}}^1|U-\nu|^{-1}\,dx\leq C \int_{\frac 14 \beta^{ -\frac 14}}^{1} y^{-2} \,dy \leq \hat C \beta^{1/4}\,.
\end{displaymath}
We can then  conclude, using \eqref{eq:371},  that 
\begin{equation}
\label{eq:379}
  \|\chi_{2\gamma}|U-\nu|^{-1/2}\|_2\leq C\beta^{1/8} \,.
\end{equation}
Hence, by \eqref{eq:376} and  \eqref{eq:377}  we obtain that
\begin{displaymath}
  \|\chi_\gamma v\|_1\leq C\beta^{-3/8} (\gamma^{\frac 12} \beta^{1/4}\|v\|_2+\|\chi_\gamma v\|_1^{1/2}\|f\|_\infty^{1/2}) \,,
\end{displaymath}
from which we conclude that
\begin{displaymath}
  \|\chi_\gamma v\|_1\leq C\beta^{-1/8}(\gamma^{1/2}\|v\|_2+ \beta^{-5/8}\|f\|_\infty) \,. 
\end{displaymath}
Substituting the above into \eqref{eq:376} then yields
\begin{equation}
\label{eq:380}
  \||U-\nu|^{1/2}\chi_\gamma v\|_2\leq
  C\beta^{-1/4}(\gamma^{1/2}\|v\|_2+\beta^{-5/8}\|f\|_\infty) \,.
\end{equation}
Since by \eqref{eq:378}, 
\begin{equation}
\label{eq:381}
|U-\nu|^{1/2}\chi_\gamma\geq \frac 1 C \beta^{-1/4}\chi_\gamma\,,
\end{equation} 
we may write
\begin{equation}
  \label{eq:382}
 \|\chi_\gamma v\|_2\leq  C\, (\gamma^{1/2}\|v\|_2+\beta^{-5/8}\|f\|_\infty) \,.
\end{equation}

We now attempt to estimate $\tilde{\chi}_\gamma v$.  As
\begin{equation}
\label{eq:383}
   (\LL_\beta^\Nf-\beta\lambda)(\tilde{\chi}_\gamma v)=\tilde{\chi}_\gamma f-2\tilde{\chi}_\gamma^\prime
   v^\prime- \tilde{\chi}_\gamma^{\prime\prime}v \,,
\end{equation}
we may conclude from \cite[Theorem 1.3]{Hen2},  (which can be used since  $U\in C^3([0,1])$ and 
$\Upsilon<[- U^{\prime\prime}(0)]^{1/2}/2$), that
\begin{displaymath}
  \|\tilde{\chi}_\gamma v\|_2 \leq C\beta^{-1/2}\left(\|\tilde{\chi}_\gamma f\|_2+
  \|\tilde{\chi}_\gamma^\prime v^\prime\|_2 +
  \|\tilde{\chi}_\gamma^{\prime\prime}v\|_2\right)\leq \hat C \left(\beta^{-5/8}\|f\|_\infty+\gamma\beta^{-1/4}\|v^\prime\|+\gamma^2\|v\|_2\right)\,.
\end{displaymath}
To obtain the second inequality we used the fact that ${\rm
  supp}\,\tilde \chi_\gamma\subseteq(0,\check C \beta^{-\frac 14})$.  Combining the
above with \eqref{eq:382} yields the existence of $\gamma_0>0$ such that
for all $\gamma \in (0,\gamma_0)$,
\begin{equation}
  \label{eq:384}
 \|v\|_2 \leq C\, (\beta^{-5/8}\|f\|_\infty+\gamma\beta^{-1/4}\|v^\prime\|_2)\,.
\end{equation}
As in \eqref{eq:374} (replacing $\chi_\gamma$ by $1$), we obtain that 
\begin{equation}
\label{eq:385}
    \|v^\prime\|_2^2\leq \|v\|_1\|f\|_\infty +C\beta^{1/2}\|v\|_2^2 \,.
\end{equation}
 By \eqref{eq:265}, applied with $q=1$ and $\mu=
 \beta^{-\frac 12}$, and \eqref{eq:366} it holds that
  \begin{displaymath}
   \|(U-\nu+i\beta^{-1/2})^{-1/2}\|_2^2=  \|(U-\nu+i\beta^{-1/2})^{-1}\|_1 \leq C\beta^{1/4}\,.
  \end{displaymath}
From the above we conclude that
\begin{multline*}
  \|v\|_1\leq \||U-\nu+i\beta^{-1/2}|^{-1/2}\|_2
  \||U-\nu+i\beta^{-1/2}|^{1/2}v\|_2  \\ \leq C\beta^{1/8} \left( \||U-\nu|^{1/2}\chi_\gamma v\|_2+
  \||U-\nu|^{1/2}\tilde{\chi}_\gamma v\|_2+\beta^{-1/4}\|v\|_2\right)\,.
\end{multline*}
By \eqref{eq:186} (which is valid by \eqref{eq:368})
\begin{displaymath}
  |U-\nu|^{1/2}\tilde{\chi}_\gamma\leq C\sup_{x\in(0,\check C \beta^{-\frac
      14})}|x^2-x_\nu^2|^{1/2}\leq C\beta^{-1/4}\,,
\end{displaymath}
we obtain, with the aid of
\eqref{eq:380} that
\begin{equation}
\label{eq:386}
  \|v\|_1\leq C\, (\beta^{-1/8}\|v\|_2 +\beta^{-3/4}\|f\|_\infty)\,. 
\end{equation}
Substituting \eqref{eq:386}  into \eqref{eq:385} yields
\begin{displaymath}
    \|v^\prime\|_2\leq C\, (\beta^{1/4}\|v\|_2+\beta^{-3/8}\|f\|_\infty )\,,
\end{displaymath}
which when substituted into \eqref{eq:384} yields for sufficiently
small $\gamma_0$ and $\gamma \in (0,\gamma_0)$
\begin{equation}\label{eq:387}
   \|v\|_2\leq C\beta^{-5/8}\|f\|_\infty \,, 
\end{equation}
and then
\begin{equation}\label{eq:388}
\|v^\prime\|_2\leq C\beta^{-3/8}\|f\|_\infty \,.
\end{equation}
Substituting \eqref{eq:387}    into \eqref{eq:386}  yields
\begin{equation}\label{eq:389}
   \|v\|_1\leq C\beta^{-3/4}\|f\|_\infty\,. 
\end{equation}
By \eqref{eq:375} it holds that
\begin{displaymath}
   \beta\||U-\nu|^{1/2}\chi_\gamma v\|_2^2\leq
  C(\gamma\beta^{1/2}\|v\|_2^2+\|\chi_\gamma v\|_2\|f\|_2) \,,
\end{displaymath}
from which we conclude by combining it with  \eqref{eq:369} 
\begin{displaymath}
  \||U-\nu|^{1/2}\chi_\gamma v\|_2\leq C\beta^{-3/4}\|f\|_2 \,.
\end{displaymath}
Consequently
\begin{displaymath}
  \|\chi_\gamma^2v\|_1\leq   \||U-\nu|^{-1/2}\chi_\gamma\|_2\||U-\nu|^{1/2}\chi_\gamma v\|_2\leq C\beta^{-5/8}\|f\|_2 \,.
\end{displaymath}
Use of \eqref{eq:379} has been made to obtain the second inequality.

Since by \eqref{eq:369}
\begin{equation}
\label{eq:390}
    \|\tilde{\chi}_\gamma^2v\|_1\leq C\beta^{-1/8}\|v\|_2\leq C\beta^{-5/8}\|f\|_2 \,,
\end{equation}
we may conclude that
\begin{equation}\label{eq:391}
    \|v\|_1\leq C\beta^{-5/8}\|f\|_2 \,,
\end{equation}
which together with \eqref{eq:389} completes the proof of
(\ref{eq:365}a).\\

{\it  Step 3: Proof of (\ref{eq:365}b) }\\
To prove (\ref{eq:365}b) we set
\begin{displaymath}
  f=(x-x_\nu)g \,,
\end{displaymath}
where $g\in L^2(0,1)$.\\
 Then, as in \eqref{eq:374}, we use \eqref{eq:373}
to obtain
\begin{displaymath}
   \|(\chi_\gamma v)^\prime\|_2^2\leq \|(x-x_\nu)\chi_\gamma v\|_2\|g\|_2 +C\beta^{1/2}\|v\|_2^2 \,.
\end{displaymath}
By \eqref{eq:371} and \eqref{eq:378} there exists $C>0$ such that, for all $x\in[0,1]$,
\begin{equation}
\label{eq:392} 
  0\leq  (x-x_\nu)\chi_\gamma(x)\leq C\, (\nu-U(x))^{1/2}\chi_\gamma(x)\,.
\end{equation}
Hence,
\begin{displaymath}
   \|(\chi_\gamma v)^\prime\|_2^2\leq C \left(  \||U-\nu|^{1/2}\chi_\gamma v\|_2\|g\|_2 +\beta^{1/2}\|v\|_2^2 \right)\,.
\end{displaymath}
Next we use \eqref{eq:375} to obtain, as in \eqref{eq:376}, with the
aid of the above and \eqref{eq:392}
\begin{displaymath}
   \beta\||U-\nu|^{1/2}\chi_\gamma v\|_2^2\leq C\, \left(\gamma\beta^{1/2}\|v\|_2^2+\|\chi_\gamma|U-\nu|^{1/2}v\|_2\|g\|_2\right) \,,
 \end{displaymath}
 from which  we conclude that
\begin{equation}
\label{eq:393}
  \||U-\nu|^{1/2}\chi_\gamma v\|_2\leq C\beta^{-1/4}(\gamma^{1/2}\|v\|_2+  \beta^{-3/4} \|g\|_2) \,. 
\end{equation}
Combining the above with  \eqref{eq:381} yields
\begin{equation}
\label{eq:394}
  \|\chi_\gamma v\|_2\leq C\, (\gamma^{1/2}\|v\|_2+\beta^{-3/4}\|g\|_2) \,. 
\end{equation}
Furthermore, with the aid of \eqref{eq:379} we can conclude that
\begin{multline}
\label{eq:395}
  \|\chi_\gamma^2 v\|_1\leq  \|\, |U-\nu|^{-1/2}\chi_{\gamma} \|_2 \,  \||U-\nu|^{1/2}\chi_\gamma
  v\|_2 \\ \leq C\,\left(\gamma^{1/2}\beta^{-1/8}\|v\|_2+\beta^{-7/8}\|g\|_2\right) \,.
\end{multline}
We now use \eqref{eq:383} to obtain, as in \eqref{eq:384}, 
\begin{equation}\label{eq:396}
   \|\tilde{\chi}_\gamma v\|_2 \leq C \left(\beta^{-3/4}\|g\|_2+\gamma\beta^{-1/4}\|v^\prime\| + \gamma^2 \|v\|_2 \right)\,.
\end{equation}
Combining \eqref{eq:396}  with \eqref{eq:394} yields for
sufficiently small $\gamma$
\begin{equation}
\label{eq:397}
     \|v\|_2 \leq C\, \left(\beta^{-3/4}\|g\|_2+\gamma\beta^{-1/4}\|v^\prime\|\right)\,.
\end{equation}
Then we write
\begin{equation}
\label{eq:398}
    \|v^\prime\|_2^2= \mu\beta\|v\|_2^2+\Re\langle v,f  \rangle =  \mu\beta\|v\|_2^2+\Re\langle(x-x_\nu)v,g\rangle\,.
\end{equation}
To estimate the second term in the r.h.s of \eqref{eq:398} we first
note that by \eqref{eq:392}, \eqref{eq:367}  and \eqref{eq:372}, 
\begin{displaymath}
  \|(x-x_\nu)v\|_2 \leq \|(x-x_\nu)\chi_\gamma v\|_2 +
  \|(x-x_\nu)\tilde{\chi}_\gamma v\|_2 \leq C(\||U-\nu|^{1/2}\chi_\gamma
  v\|_2+\beta^{-1/4}\|v\|_2) \,.
\end{displaymath}
With the aid of \eqref{eq:393} we then obtain
\begin{equation}
\label{eq:399}
  \|(x-x_\nu)v\|_2 \leq C\,\left(\beta^{-1/4}\|v\|_2+\beta^{-3/4}\|g\|_2\right) \,. 
\end{equation}
Hence, by \eqref{eq:398} and since $\mu<C\beta^{-1/2}$ we may conclude
that
\begin{equation}\label{eq:400}
     \|v^\prime\|_2\, \leq  \, C \, \left(\beta^{1/4}\|v\|_2+\beta^{-1/2}\|g\|_2\right)\,.
\end{equation}
Substituting  \eqref{eq:400} into \eqref{eq:397} yields for small
enough $\gamma$ , 
\begin{equation}\label{eq:401}
   \|v\|_2 \leq C\beta^{-3/4}\|g\|_2\,.
\end{equation}
 By \eqref{eq:401}, the first inequality of \eqref{eq:390}, and
 \eqref{eq:395} we obtain 
\begin{equation}
\label{eq:402}
     \|v\|_1\leq C\beta^{-7/8}\|g\|_2\,.
\end{equation}
Together with \eqref{eq:401} and \eqref{eq:400}, \eqref{eq:402}
verifies (\ref{eq:365}b).
\end{proof}

\subsection{$L^1$ estimates for $U(0)-\Im \lambda \gg\beta^{-1/2}$}
\label{sec:l1-estimates}
  It is not difficult to show that the resolvent of the operator $-d^2/dx^2+ix$ is
 not bounded in $\LL(L^\infty(\R),L^1(\R))$, a fact that can be easily
 established from the identity
 \begin{displaymath}
   \Big(-\frac{d^2}{dx^2}+ix\Big)\frac{1}{\sqrt{x^2+1}}=
   \frac{2x^2-1}{[x^2+1]^{5/2}}+i\frac{x}{x^2+1}\,.
 \end{displaymath}
For the resolvent of the operator $\LL_\beta^{\Nf,\Df}$ on $(0,1)$, this unboundedness manifests itself through a  logarithmic
dependence on $\beta$ as we can clearly see in the following proposition. 
\begin{proposition}
  \label{lem:Dirichlet-L1} Let $U\in C^2([0,1])$ satisfy
  \eqref{eq:10}.  There exist $\Upsilon>0$, $a>0$, $C>0$, and
  $\beta_0>1$ such that, for $\beta\geq \beta_0$, $\Im\lambda<U(0)-a\beta^{-1/2}$, $\Re\lambda
  \leq {\mathfrak J_\nu}^{2/3}\Upsilon\beta^{-1/3}$, and $f\in L^\infty(0,1)$ we have
\begin{equation}
\label{eq:403} 
  \| (\LL_\beta^{\Nf,\Df} -\beta\lambda)^{-1} f  \|_1 \leq C \min\left([{\mathfrak
      J_\nu}\beta]^{-5/6}\|f \|_2,[{\mathfrak J_\nu}\beta]^{-1}\, \log\beta\,
    \|f\|_\infty\right) \,. 
\end{equation}
\end{proposition}
\begin{proof}~\\
  We assume that $\Upsilon>0$ is sufficiently small so that Proposition
  \ref{prop5.3} holds true.  We begin by recalling that by
  \eqref{eq:289}, for any $N >0$, there exists $ a_0>0$ such that for
  all $a\geq a_0$, we have under the conditions of the proposition,
\begin{displaymath}
\beta x_\nu \geq  \beta x_\nu^4 \geq N\,,
\end{displaymath}
where $x_\nu$ is defined by \eqref{eq:defxnu}.
Let $(v,f)\in D(\LL_\beta^\Nf)\times L^\infty(0,1)$ satisfy $(\LL_\beta^\Nf-\beta\lambda)v=f$.
 By \eqref{eq:225} 
  applied with $\mu = \beta^{-1/3} x_\nu^{2/3} $ and $q=2$, it holds that
\begin{equation*} 
   \|(U-\nu+i [\beta x_\nu^{-2}]^{-1/3})^{-1}\|_2^2 \leq
   Cx_\nu^{-5/3}\beta^{1/3} \,. 
\end{equation*}
 We may then conclude that
  \begin{equation}
\label{eq:404}
  \begin{array}{ll}
    \|v\|_1 &\leq \|(U-\nu+i[\beta x_\nu^{-2}]^{-1/3})^{-1}\|_2
    \|(U-\nu+i[\beta x_\nu^{-2}]^{-1/3})v\|_2\\ & \leq C\beta^{1/6}x_\nu^{-5/6}\, [\|(U-\nu)v\|_2+[\beta x_\nu^{-2}]^{-1/3}\|v\|_2]\,.
    \end{array}
  \end{equation}
By \eqref{eq:354} and (\ref{eq:287}a) 
\begin{equation}
  \label{eq:405}
\|v\|_1 \leq C [\beta x_\nu]^{-5/6}\, \|f \|_2 \,.
\end{equation}
By \eqref{eq:238} 
we can conclude that
\begin{equation*}
   \|(U-\nu+i [\beta x_\nu^{-2}]^{-1/3})^{-1/2}\|_2^2= \|(U-\nu+i [\beta
   x_\nu^{-2}]^{-1/3})^{-1}\|_1\leq\frac{C}{x_\nu}\log (\beta x_\nu^4)\,.
\end{equation*}
Hence, we  can complete the proof of \eqref{eq:403} by writing 
\begin{equation*}
\begin{array}{ll}
    \|v\|_1&  \leq \|(U-\nu+i[\beta x_\nu^{-2}]^{-1/3})^{-1/2}\|_2
    \|(U-\nu+i[\beta x_\nu^{-2}]^{-1/3})^{1/2}v\|_2\\ &
    \leq\frac{C}{x_\nu^{1/2}}[\log (\beta x_\nu^4)]^{1/2} [\||U-\nu|^{1/2}v\|_2+[\beta x_\nu^{-2}]^{-1/6}\|v\|_2]\,,
    \end{array}
\end{equation*}
which implies
\begin{equation} \label{eq:406}
  \|v\|_1 \leq \frac{C}{x_\nu^{1/2}}[\log (\beta)]^{1/2} [\||U-\nu|^{1/2}v\|_2+[\beta x_\nu^{-2}]^{-1/6}\|v\|_2]\,.
  \end{equation}
We note that by (\ref{eq:287}a) (which holds for $a\geq a_0$
    with $a_0$ large enough)  and \eqref{eq:289} \\ 
we have 
\begin{displaymath}
[\beta x_\nu^{-2}]^{-1/6}\|v\|_2\leq C \beta^{-1} x_\nu^{-\frac 12} \|f\|_\infty   \,.
\end{displaymath}
Hence, we obtain from \eqref{eq:406}
\begin{equation} 
\label{eq:407}
  \|v\|_1 \leq \frac{C}{x_\nu^{1/2}}[\log (\beta)]^{1/2} \big[\||U-\nu|^{1/2}v\|_2+ \beta^{-1} x_\nu^{-1/2} \|f\|_\infty\big ]\,.
  \end{equation}
  To complete the proof we need an estimate for $\||U-\nu|^{1/2}v\|_2$.
  In a similar manner to \eqref{eq:292} we let
\begin{displaymath}
\chi_s^\pm(x)=\hat{\chi}(s(x-x_\nu)){\mathbf   1}_{\R_+}(\pm (x-x_\nu)) \mbox{  with } s= [\beta x_\nu]^{1/3}\,,
\end{displaymath}
where $\hat{\chi}$ is defined by \eqref{eq:291}.
 An integration by parts yields, as in
\eqref{eq:307},
\begin{equation}
  \|(\chi_s^\pm v)^\prime\|_2^2-\|(\chi_s^\pm)^\prime v\|_2^2-\mu\beta\|\chi_s^\pm v\|_2^2=\Re\langle\chi_s^\pm v,\chi_s^\pm f\rangle\,,
\end{equation}
from which we conclude, given that $\mu\beta \leq C s^2 $ 
\begin{equation}
    \|(\chi_s^\pm v)^\prime\|_2^2\leq \|\chi_s^\pm v\|_1\|f\|_\infty +\hat C  s^2 \|v\|_2^2 \,.
\end{equation}
Furthermore, (see \eqref{eq:308}), we have that 
\begin{equation}
\label{eq:408}
  \mp \beta\|\,|U-\nu|^{1/2}\chi_s^\pm v\|_2^2+2\Im\langle(\chi_s^\pm )^\prime v,(\chi_s^\pm v)^\prime\rangle
  =\Im\langle\chi_s^\pm v,\chi_s^\pm f\rangle \,,
\end{equation}
and hence, with the aid of above, we obtain that
\begin{displaymath}
  \||U-\nu|^{1/2}\hat \chi_s  v\|_2^2\leq
  C \beta^{-1} ([\beta x_\nu]^{2/3}\|v\|_2^2+\|\chi_s^\pm v\|_1\|f\|_\infty) \,. 
\end{displaymath}
By (\ref{eq:287}a), we then have
\begin{equation}
\label{eq:409}
   \||U-\nu|^{1/2}\chi_s^\pm v\|_2\leq
   C(x_\nu^{-1/2}\beta^{-1}\|f\|_\infty+ \beta^{-1/2}\|v\|_1^{1/2}\|f\|_\infty^{1/2} ) \,.
\end{equation}
Let $\tilde{\chi}_s=\sqrt{1-(\chi_s^+)^2-(\chi_s^-)^2}$\,. Since $ s= [\beta
x_\nu]^{1/3}$ it holds that   
\begin{displaymath}
{\rm supp } \tilde \chi_s \subset [x_\nu-\frac 12 (\beta x_\nu)^{-1/3},x_\nu+\frac 12 (\beta x_\nu)^{-1/3})\,.
\end{displaymath}
   As
\begin{displaymath}
  \||U-\nu|^{1/2}\tilde{\chi}_s v\|_2\leq C\,[\beta x_\nu]^{-1/6}x_\nu^{1/2}  \|\tilde{\chi}_s v\|_2\leq C\beta^{-1/6}x_\nu^{1/3}  \|v\|_2\,,
\end{displaymath}
we may use (\ref{eq:287}a) once again to obtain 
\begin{displaymath}
  \||U-\nu|^{1/2}\tilde{\chi}_s v\|_2\leq C x_\nu^{-\frac 12} \beta^{-1}  \|f\|_\infty  \,.
\end{displaymath}

Combining the above with \eqref{eq:409} yields
\begin{equation}\label{eq:410}
   \||U-\nu|^{1/2}v\|_2\leq
   C\left(x_\nu^{-1/2}\beta^{-1}\|f\|_\infty+ \beta^{-1/2}\|v\|_1^{1/2}\|f\|_\infty^{1/2} \right) \,.
\end{equation}
Substituting \eqref{eq:410}  into \eqref{eq:407} yields, for any $\eta >0$, 
\begin{displaymath}
  \begin{array}{ll}
\|v\|_1 & \leq  \frac{C}{x_\nu^{1/2}}[\log (\beta)]^{1/2} \left( \beta^{-1} x_\nu^{-1/2} \|f\|_\infty + \beta^{-1/2} \| v\|_1^{1/2} \|f\|_\infty^{1/2}\right)\\
&\leq \frac{C}{x_\nu^{1/2}}[\log (\beta)]^{1/2} \left( \beta^{-1} x_\nu^{-1/2} \|f\|_\infty + \beta^{-1/2} (\frac{1}{\eta}  \|f\|_\infty + \eta   \| v\|_1) \right)  \,.
\end{array}
\end{displaymath}
Choosing  $\eta >0$ such that 
\begin{displaymath}
C\,  \eta  \,\beta^{-1/2} x_\nu^{- 1/2}[\log (\beta)]^{1/2}  = \frac 12
\end{displaymath}
we obtain
\begin{equation}
\label{eq:411}
\|v\|_1 \leq \hat C  \,  \frac{1}{\beta x_\nu}  \log (\beta)   \,   \|f\|_\infty \,.
\end{equation}
Combining the above  with
\eqref{eq:405},  completes the proof of the proposition. 
\end{proof}
   If $U-\nu\neq0$ in $[0,1]$ it can be easily verified (see
  \eqref{eq:354}) that
$$\|(-d^2/dx^2+i\beta[U-\nu])^{-1}\|\lesssim\beta^{-1}\,.
$$
In contrast, when
$U(x)=\nu$ for some $x\in(0,1)$ the best estimate we can obtain
(see \eqref{eq:287}) is $$\|\big(-d^2/dx^2+i\beta[U-\nu]\big)^{-1}\|\lesssim\beta^{-2/3}\,.$$
The zero of $U-\nu$ at $x=x_\nu$, thus, has a significant effect on the
resolvent norm. Nevertheless, if $f(x_\nu)=0$ and $f$ is small in the
neighborhood of $x_\nu$ one may expect that
$\|(-d^2/dx^2+i\beta[U-\nu])^{-1}f\|_2$ would be smaller in that case.\\
This heuristic argument is manifested more precisely in the next
proposition.
 \begin{proposition}
\label{H1-estimate-0}  Let  $U\in C^2([0,1])$
satisfy \eqref{eq:10}. There
exist $\Upsilon>0$, $C>0$, $a>0$, and $\beta_0>0$ such that, for $\beta\geq \beta_0$,
  $\nu\leq U(0)-a\beta^{-1/2}$, $ \mu <\Upsilon{\mathfrak
  J_\nu}^{2/3}\beta^{-1/3}\,,$ and $f\in L^2(0,1)$ such that
$(x-x_\nu)^{-1}f\in L^2(0,1)$, we have
  \begin{equation}
\label{eq:412}
\|(\LL_\beta^\Nf-\beta\lambda)^{-1}f\|_2\leq C\,  (\mathfrak J_\nu \beta)^{-1}\, \Big\|\frac{f}{x-x_\nu}\Big\|_2\,,
  \end{equation}
and
  \begin{equation}
    \label{eq:413}
\Big\|\frac{d}{dx}(\LL_\beta^\Nf-\beta\lambda)^{-1}f\Big\|_2\leq C\beta^{-1/2} \, \Big\|\frac{f}{x-x_\nu}\Big\|_2\,.
  \end{equation}
\end{proposition}

\begin{proof}~\\

Let $(v,f,\lambda)\in D(\LL_\beta^{\Nf})\times L^2(0,1)\times\C$  satisfy $(\mathcal L_\beta^\Nf-\beta \lambda)v=f$.\\

{\em Step 1: We prove \eqref{eq:412}.\\}

 \noindent  Set 
\begin{displaymath}
   f=(U-\nu)g \,. 
\end{displaymath}
 Given that $ \mu <\Upsilon{\mathfrak   J_\nu}^{2/3}\beta^{-1/3}\,,$
  it follows from \eqref{eq:287} that there exists $u\in D(\mathcal L_\beta^\Nf)$ satisfying 
\begin{equation}
\label{eq:414}
  (\mathcal L_\beta^\Nf-\beta \lambda)u=g\,.
\end{equation}
Let
\begin{displaymath}
  w=(U-\nu)u\,.
\end{displaymath}
 Then, it holds that
\begin{displaymath}
  (\mathcal L_\beta^\Nf-\beta \lambda)w=f-2U^\prime u^\prime-U^{\prime\prime}u \,.
\end{displaymath}
Consequently,
\begin{equation}
\label{eq:415}
  v=w+(\mathcal L_\beta^\Nf-\beta \lambda)^{-1}(2U^\prime u^\prime+U^{\prime\prime}u)\,.
\end{equation}

We now recall \eqref{eq:348}
\begin{displaymath}
  |U^\prime|\leq C|U-U(0)|^{1/2}\leq C(|U-\nu|^{1/2}+x_\nu)\,.
\end{displaymath}
By the above and (\ref{eq:287}a) it holds that
\begin{multline}\label{eq:949abc}
  \|(\mathcal L_\beta^\Nf-\beta \lambda)^{-1}(2U^\prime u^\prime+U^{\prime\prime}u)\|_2\leq \\ C\, [{\mathfrak
J_\nu}\beta]^{-2/3}\,\big(x_\nu\|u^\prime\|_2+\||U-\nu|^{1/2}u^\prime\|_2+\|u\|_2 \big) \,.
\end{multline}
 Let $\chi_\nu^\pm$ be defined by \eqref{eq:292}. Clearly, in view of
 \eqref{eq:293} and the fact that $U'(0)=0$,  
\begin{equation}\label{eq:949aa}
 \||U-\nu|^{1/2}u^\prime\|_2\leq  \|\chi_\nu^+
 |U-\nu|^{1/2}u^\prime\|_2+\|\chi_\nu^-|U-\nu|^{1/2}u^\prime\|_2+Cx_\nu\|u^\prime\|_2\,.
\end{equation}
By \eqref{eq:356}, applied to the pair $(u, g )$, \eqref{eq:949aa},
and (\ref{eq:287}b)  we then have 
\begin{equation}\label{eq:949bbb}
x_\nu \|u'\|_2  +  \||U-\nu|^{1/2}u^\prime\|_2\leq  C(\beta^{-1/2}+ x_\nu^{2/3}\beta^{-1/3})\|g\|_2\,.
\end{equation}
Using (\ref{eq:287}a) together with \eqref{eq:949abc} and
\eqref{eq:414} then yields, as  $x_\nu\geq\beta^{-1/4}$,
\begin{equation}
\label{eq:416}
   \|(\mathcal L_\beta^\Nf-\beta \lambda)^{-1}(2U^\prime u^\prime+U^{\prime\prime}u)\|_2\leq C\beta^{-1}\|g\|_2 \,.
\end{equation}
By \eqref{eq:354}, applied to the pair $(u, g )$,  it holds that 
\begin{equation}\label{eq:417}
  \|w\|_2\leq C \beta^{-1} \,\|g\|_2\,.
\end{equation}
Substituting the above together with \eqref{eq:416} into
\eqref{eq:415} yields
\begin{displaymath}
  \|v\|_2\leq C\beta^{-1}\|g\|_2\,.
\end{displaymath}
 Since 
\begin{displaymath}
  |U(x)-\nu|\geq \frac 1C x_\nu |x- x_\nu| \,,
\end{displaymath}
it holds that
\begin{displaymath}
\|g\|_2 \leq C x_\nu^{-1}  \|(x-x_\nu)^{-1}f\|_2\,, 
\end{displaymath}
and hence, we can now conclude \eqref{eq:412} from the above and
\eqref{eq:417}. \\

{\em Step 2: We prove \eqref{eq:413}.\\}

Taking the real part of the scalar product with $\langle v,(\mathcal
L_\beta^\Nf-\beta \lambda)v\rangle$ we write
\begin{equation*}
    \|v^\prime\|_2^2= \mu\beta\|v\|_2^2+\Re\langle(x-x_\nu)v,(x-x_\nu)^{-1}f\rangle\,,
\end{equation*}
From this we deduce that 
\begin{equation}
\label{eq:418}
    \|v^\prime\|_2^2\leq \mu_+ \beta\|v\|_2^2+| \langle(x-x_\nu)v,(x-x_\nu)^{-1}f\rangle|\,,
\end{equation}
where 
  \begin{equation}
\label{eq:914}
 \mu_+=\max(\mu,0)\,.   
  \end{equation}

Since $|x-x_\nu|\leq C\,|U-\nu|^{1/2}$ we obtain by \eqref{eq:412},  as
$x_\nu\geq\beta^{-1/4}$ and \break $\mu_+\leq C \, \mathfrak J_\nu^{2/3}  \beta^{-1/3} $,
\begin{equation}
\label{eq:419} 
   \|v^\prime\|_2^2\leq C \Big(\beta^{-1} \|(x-x_\nu)^{-1}f\|_2 + \||U-\nu|^{1/2}v\|_2
   \|(x-x_\nu)^{-1}f\|_2\Big)\,.
\end{equation} 
By \eqref{eq:308} it holds that
\begin{displaymath}
  \beta\| |U-\nu|^{1/2}\chi_\nu^\pm v\|_2^2\leq C x_\nu^{-1} \|v\|_2\|v^\prime\|_2 +
  \|(x-x_\nu)v\|_2\|(x-x_\nu)^{-1}f\|_2 \,.
\end{displaymath}
Hence,  in view of \eqref{eq:293}, 
\begin{displaymath}
   \beta\| |U-\nu|^{1/2}v\|_2^2\leq C (x_\nu^{-1} \|v\|_2\|v^\prime\|_2 +
    \beta x_\nu^2 \|v\|_2^2) +
  \|(x-x_\nu)v\|_2\|(x-x_\nu)^{-1}f\|_2 \,.
\end{displaymath}
 Since by \eqref{eq:412} it holds that   
 \begin{displaymath}
  \beta x_\nu^2\|v\|_2^2\leq C\beta^{-1}\|(x-x_\nu)^{-1}f\|_2^2\,,
 \end{displaymath}
we may obtain that
\begin{multline*}
   \beta\| |U-\nu|^{1/2}v\|_2^2\leq \\ \quad  \leq  C \Big( x_\nu^{-2} \beta^{-1} \|(x-x_\nu)^{-1}f\|_2\|v^\prime\|_2 
  +\|(U-\nu)^{1/2}v\|_2\|(x-x_\nu)^{-1}f\|_2 + \beta^{-1}\|(x-x_\nu)^{-1}f\|_2^2 \Big) \,.
\end{multline*}
From the above we conclude, as $x_\nu\geq\beta^{-1/4}$,
\begin{equation}
\label{eq:420}
  \| |U-\nu|^{1/2}v\|_2^2\leq C(\beta^{-1}\|v^\prime\|_2^2 
  + \beta^{-2}\|(x-x_\nu)^{-1}f\|_2^2) \,.
\end{equation}
Substituting the above into \eqref{eq:419} yields 
\begin{displaymath}
  \|v^\prime\|_2\leq C\beta^{-1/2}\|(x-x_\nu)^{-1}f\|_2\,,
\end{displaymath}
verifying \eqref{eq:413}.
\end{proof}

 We now  seek an estimate for $( \mathcal L_\beta -\beta\lambda)^{-1}$ in
  $\LL(H^1,L^1)$. To this end we write $f=f-f(x_\nu)+f(x_\nu)$,
    and estimate first $({\mathcal L_\beta}-\beta\lambda)^{-1}(f-f(x_\nu))$  using
    \eqref{eq:412}. Then, we estimate $({\mathcal
      L_\beta}-\beta\lambda)^{-1}f(x_\nu)$ by observing first that the leading order
    term is  $-if(x_\nu)[\beta(U+i\lambda)]^{-1}$  for $|\mu|
    >x_\nu^{2/3}\beta^{-1/3}$. 

\begin{proposition}
  \label{Dirichlet-L1-H1-0} 
  Let $U\in C^2([0,1])$ satisfy \eqref{eq:10}.
  Then there exist $\Upsilon>0$, $C>0$, $a>0$, and $\beta_0>0$ such that, for $\beta\geq
  \beta_0$, $ U(0)-\nu>a\beta^{-1/2}$, and $ \mu \leq {\mathfrak
    J_\nu}^{2/3}\Upsilon\beta^{-1/3}$ and $f\in H^1(0,1)$ we have
\begin{equation}
\label{eq:421}
  \Big\| (\LL_\beta^\Nf-\beta\lambda)^{-1} f+i\frac{f(x_\nu)}{\beta[U-\nu- i \max(-\mu,x_\nu^{2/3}\beta^{-1/3})]} \Big\|_1
  \leq C\, [{\mathfrak J_\nu}\beta]^{-1}\|f\|_{1,2}\,.
\end{equation}
\end{proposition}
\begin{proof}
Let $u=(\LL_\beta^\Nf-\beta\lambda)^{-1} f$.~\\

{\it Step 1: We prove \eqref{eq:421} in the case $-\mu\leq
x_\nu^{2/3}\beta^{-1/3}$.\\}

\noindent We apply the decomposition
\begin{equation}
\label{eq:422} 
    v =u + i  \frac{f(x_\nu)}{\beta(U-\nu - ix_\nu^{2/3}\beta^{-1/3})} \,.
   \end{equation}
Then, 
\begin{displaymath}
  (\LL_\beta^\Nf -\beta\lambda)v=f+i \beta^{-1} f(x_\nu)\left(  (\LL_\beta^\Nf -\beta\lambda)(
    U-\nu-ix_\nu^{2/3}\beta^{-1/3} )^{-1}\right) \,.
\end{displaymath}
We next observe that
\begin{multline*}
   (\LL_\beta^\Nf -\beta\lambda)( U-\nu-ix_\nu^{2/3}\beta^{-1/3} )^{-1}  =\\ =\beta \frac{(-\lambda + i
     U)}{ ( U-\nu-ix_\nu^{2/3}\beta^{-1/3})}   -
  \frac{2 |U^\prime|^2}{(U-\nu-ix_\nu^{2/3} \beta^{-1/3})^3} +\frac{U^{\prime\prime}}{(U-\nu-ix_\nu^{2/3}\beta^{-1/3})^2}    \,,
\end{multline*}
and that
\begin{multline*}
-i(\lambda-i U) \, {( U-\nu-i x_\nu^{2/3} \beta^{-1/3} )}^{-1} +1\\ = ( U-\nu-ix_\nu^{2/3}\beta^{-1/3}
)^{-1} \left( -i[(\mu + i\nu) - U] +  U-\nu-ix_\nu^{2/3}\beta^{-1/3} \right).
\end{multline*}  
Consequently, it holds that
\begin{equation}
\label{eq:423}
  (\LL_\beta^\Nf -\beta\lambda)v=f-f(x_\nu) + f(x_\nu)h \,,
\end{equation}
where
\begin{multline}
\label{eq:424}
  h= i \frac{U^{\prime\prime}}{\beta(U-\nu-ix_\nu^{2/3}\beta^{-1/3})^2} - 2 i 
  \frac{ |U^\prime|^2}{\beta(U-\nu-ix_\nu^{2/3}\beta^{-1/3})^3} \\- i 
  \frac{\mu +x_\nu^{2/3}\beta^{-1/3}}{U-\nu-ix_\nu^{2/3}\beta^{-1/3}} \,. 
\end{multline} 
Since for  $U\in C^2([0,1])$ satisfying \eqref{eq:10}, we have
 \begin{equation}
\label{eq:425} 
 |U(x) - \nu| \geq \frac{1}{C} x_\nu |x-x_\nu|\,,
 \end{equation}
we may conclude that for $k>1$, 
\begin{equation} 
\label{eq:426}
  \int_0^1 \frac{dx}{|U-\nu-ix_\nu^{2/3}\beta^{-1/3}|^k}\leq C  \int_0^1
  \frac{dx}{x_\nu^k|x-x_\nu|^k+x_\nu^{2k/3}\beta^{-k/3}}\leq  \widehat C  x_\nu^{-k} [\beta x_\nu]^{\frac{k-1}{3}} \,.
\end{equation}
For later reference we mention that for $k=1$
\begin{multline} 
\label{eq:427}
  \int_0^1\frac{dx}{|U-\nu+ix_\nu^{2/3}\beta^{-1/3}|}\leq\nobreakspace\nobreakspace C\Big[
  \frac{1}{x_\nu}\int_0^{2x_\nu} \frac{dx}{|x-x_\nu|+
    x_\nu^{-1/3}\beta^{-1/3}}+\int_{2x_\nu}^1 \frac{dx}{|x-x_\nu|^2}\Big] 
\\ \leq \frac{\widehat C}{x_\nu}\big[\log(x_\nu^{4/3}\beta^{1/3})+1\big]\,.
\end{multline}

Using \eqref{eq:426} and the fact that $x_\nu \geq \frac 1C \beta^{1/4}$,
together with \eqref{eq:348}, it can be verified that there exist
positive $C$ and $\beta_0$ such that for $\beta \geq \beta_0$ and $|\mu|\leq
x_\nu^{2/3}\beta^{-1/3}$,
\begin{equation}
\label{eq:428}
  \|h\|_2 \leq C\, [\beta x_\nu]^{-1/6} \,.
\end{equation}
Consequently, by \eqref{eq:403}
\begin{equation}
\label{eq:429}
  \|(\LL_\beta^\Nf -\beta\lambda)^{-1}h\|_1\leq C \, (\beta x_\nu)^{-1} \,. \end{equation}
 By \eqref{eq:412} and Hardy's inequality \eqref{eq:18} it holds that
\begin{multline*}
  \|(\LL_\beta^\Nf
  -\beta\lambda)^{-1}(f-f(x_\nu))\|_1\leq \|(\LL_\beta^\Df
  -\beta\lambda)^{-1}(f-f(x_\nu))\|_2 \\
  \leq  C \, (\beta x_\nu)^{-1} \, \Big\|\frac{f-f(x_\nu)}{x-x_\nu}\Big\|_2  \leq  \widehat C\,  (\beta x_\nu)^{-1} \, \|f^\prime\|_2 \,. 
\end{multline*}
 Substituting the above, together with \eqref{eq:429} into
\eqref{eq:423} yields
\begin{displaymath}
  \|v\|_1\leq C \,  (\beta x_\nu)^{-1} (\|f^\prime\|_2+| f(x_\nu)|)\leq  \widehat C\,  (\beta x_\nu)^{-1}  \, \|f\|_{1,2}\,.
\end{displaymath}

{\it Step 2: We prove \eqref{eq:421} in the case $\mu\leq -
x_\nu^{2/3}\beta^{-1/3}$. \\}

\noindent In this case we consider instead  the decomposition
   \begin{displaymath}
     v=u + i  \frac{f(x_\nu)}{\beta(U+i\lambda)} \,.
   \end{displaymath}
Then we obtain  
\begin{displaymath}
  (\LL_\beta^\Nf -\beta\lambda)v=f +i   \beta^{-1} f(x_\nu)\left(  (\LL_\beta^\Nf -\beta\lambda)( U + i \lambda)^{-1}\right) \,.
\end{displaymath}
As
\begin{displaymath}
   (\LL_\beta^\Nf -\beta\lambda)( U + i \lambda  )^{-1}  = i  \beta    -
  \frac{2 |U^\prime|^2}{(U + i \lambda )^3} +\frac{U^{\prime\prime}}{(U+ i \lambda)^2 }    \,,
\end{displaymath}
we obtain that
\begin{equation}
  (\LL_\beta^\Nf -\beta\lambda)v=f-f(x_\nu) + f(x_\nu)\tilde h \,,
\end{equation}
where
\begin{equation}
\label{eq:430}
 \tilde  h=i \frac{U^{\prime\prime}}{\beta(U+i\lambda)^2} -2 i 
  \frac{ |U^\prime|^2}{(U+i\lambda)^3} \,. 
\end{equation} 
and proceed in a similar manner using the lower bound for $|\mu|$. 
\end{proof}  
An immediate consequence of   Proposition \ref{Dirichlet-L1-H1-0}  now follows
 by using \eqref{eq:427}
\begin{corollary}
\label{cor:l1}
  Under the conditions of Proposition \ref{Dirichlet-L1-H1-0},  it holds
(with sufficiently large $a$)
  that
  \begin{equation}
    \label{eq:431} 
  \| (\LL_\beta^\Nf-\beta\lambda)^{-1} f\|_1
  \leq C\,[{\mathfrak J_\nu}\beta]^{-1}\big( \|f\|_{1,2}+|f(x_\nu)|\log(x_\nu^{4/3}\beta^{1/3})\big)\,.
  \end{equation}
\end{corollary}
We conclude this section by another auxiliary estimate which will
be useful in Subsections \ref{sec:5.6} and \ref{sec:5.7}.
\begin{proposition}  Let $U\in C^3([0,1])$ satisfying
  \eqref{eq:10},
  $a>0$ and  $\Upsilon<\sqrt{-U^{\prime\prime}(0)}/2$. Then  there
  exist $C>0$, $\beta_0>0$ such that, for all $\beta \geq \beta_0$,
      \begin{multline} \label{eq:432}
\sup_{
  \begin{subarray}{c}
   \mu  \leq\Upsilon\beta^{-1/2} \\
\nu <U(0)+a\beta^{-1/2}
  \end{subarray}} \Big(
\|(\LL_\beta^{\Nf,\Df}-\beta\lambda)^{-1}(U-\nu) f\|_2  +
\beta^{-1/2}\| \frac{d}{dx} (\LL_\beta^{\Nf,\Df}-\beta\lambda)^{-1} (U-\nu) f \|_2 \Big) \\ \leq   C \beta^{-1}\|f\|_2\,. 
  \end{multline}
\end{proposition}
 Note that if we apply \eqref{eq:412}  (for $\nu<U(0)-a_0 \beta^{-1/2}$ with
some sufficiently large $a_0$) or (\ref{eq:365}b) (in the case $|\nu -
  U(0)| \leq  a_0\beta^{-1/2}$)  we obtain that 
\begin{displaymath}
  \|(\LL_\beta^{\Nf,\Df}-\beta\lambda)^{-1}(U-\nu)f\|_2\leq C\beta^{-3/4}[1+x_\nu\beta^{1/4}]^{-1} \|f\|_2\,,
\end{displaymath}
which is weaker than \eqref{eq:432}.

\begin{proof}~\\

Let $v \in D( \LL_\beta^{\Nf,\Df}) $ such that
  \begin{equation}
    \label{eq:433}
(\LL_\beta^{\Nf,\Df} -\beta\lambda)v=(U-\nu)f \,.
  \end{equation}
   Let $w\in D(\LL_\beta^{\Nf,\Df})$ satisfy
  \begin{displaymath}
    (\LL_\beta^{\Nf,\Df} -\beta\lambda)w=f \,.
  \end{displaymath}
It can be easily verified that
\begin{displaymath}
  (\LL_\beta^{\Nf,\Df} -\beta\lambda)([U-\nu]w)=(U-\nu)f-2U^\prime w^\prime-U^{\prime\prime}w\,.
\end{displaymath}
Hence,
\begin{equation}
\label{eq:434}
 v=(U-\nu)w+(\LL_\beta^{\Nf,\Df} -\beta\lambda)^{-1}(2U^\prime w^\prime+U^{\prime\prime}w)\,.
\end{equation} 
Let $ a_0>0$.  Consider first the case where
$ U(0)-a_0\beta^{-1/2}<\nu<U(0)+a_0\beta^{-1/2}$. By \eqref{eq:350} (with $v$
replaced by $w$) and since by \eqref{eq:10} we have $|U^\prime(x)|\leq x$,
it holds that
\begin{displaymath}
  \beta\|(U-\nu)w\|_2^2 \leq
  C\left(\beta^{-1}\|f\|_2^2+[\|(x-x_\nu)w\|_2+x_\nu\|w\|_2]\|w^\prime\|_2\right)\,.
\end{displaymath} 
We can now use \eqref{eq:399}, given that $x_\nu\leq C\beta^{-1/4}$, to obtain
\begin{displaymath}
    \beta\|(U-\nu)w\|_2^2 \leq
  C\left(\beta^{-1}\|f\|_2^2+[\beta^{-3/4}\|f\|_2+\beta^{-1/4}\|w\|_2]\|w^\prime\|_2\right)\,.
\end{displaymath}
We may now apply (\ref{eq:365}a) to the pair $(w,f)$ 
  to  conclude that
\begin{equation}
  \label{eq:435}
\|(U-\nu)w\|_2 \leq  C\beta^{-1}\|f\|_2\,.
\end{equation}
 In \eqref{eq:354} we have established that there exists $a>0$ such
that \eqref{eq:435} holds also whenever $\nu<U(0)-a\beta^{-1/2}$ under the
conditions of Proposition \ref{prop5.3}. 

Next, we use once again either (\ref{eq:365}a) (in the case when $|\nu
- U(0)| < a \beta^{-1/2}$) or (\ref{eq:287}a) in the case when
$\nu<U(0)-a\beta^{-1/2}$), with $f$ replaced by $U^{\prime\prime}w$, to obtain that
\begin{multline}
\label{eq:436}
  \|(\LL_\beta^{\Nf,\Df} -\beta\lambda)^{-1}(U^{\prime\prime}w)\|_2  \leq
  C\beta^{-1/2}[1+x_\nu\beta^{1/4}]^{-2/3}\|w\|_2 \\ \leq \widetilde C
    \beta^{-1}[1+x_\nu\beta^{1/4}]^{-4/3}\|f\|_2\leq \widehat C
  \beta^{-1}\|f\|_2\,. 
\end{multline} 

Finally, we write
\begin{displaymath}
  \|(\LL_\beta^{\Nf,\Df} -\beta\lambda)^{-1}(U^\prime w^\prime)\|_2\leq \|(\LL_\beta^{\Nf,\Df}
  -\beta\lambda)^{-1}([U^\prime-U^\prime(x_\nu)]w^\prime)\|_2+\|(\LL_\beta^{\Nf,\Df}
  -\beta\lambda)^{-1}(U^\prime(x_\nu)w^\prime)\|_2
\end{displaymath}
For the second term on the right-hand-side we have by (\ref{eq:365}a) 
 and \eqref{eq:287}
\begin{displaymath}
 \| (\LL_\beta^{\Nf,\Df}
  -\beta\lambda)^{-1}(U^\prime(x_\nu)w^\prime)\|_2 \leq C \beta^{-1/2}x_\nu[1+x_\nu\beta^{1/4}]^{-2/3}\|w^\prime\|_2\leq
  \widehat C\beta^{-1}\|f\|_2\,.
\end{displaymath}
For the first term we use instead either (\ref{eq:365}b) or
\eqref{eq:412} with $f=[U^\prime-U^\prime(x_\nu)]w^\prime$ and then (\ref{eq:365}a)
and \eqref{eq:287} for the second one to obtain
\begin{displaymath}
  \|(\LL_\beta^{\Nf,\Df}-\beta\lambda)^{-1}([U^\prime-U^\prime(x_\nu)]w^\prime)\|_2\leq C\beta^{-3/4}
  \Big\|\frac{U^\prime-U^\prime(x_\nu)}{x-x_\nu}w^\prime \Big\|_2\leq   \widehat C \beta^{-1}\|f\|_2\,.
\end{displaymath} 
Hence,
\begin{equation}\label{eq:437}
  \|(\LL_\beta^{\Nf,\Df} -\beta\lambda)^{-1}(U^\prime w^\prime)\|_2\leq C\beta^{-1}\|f\|_2\,.
\end{equation}
By  \eqref{eq:434},  \eqref{eq:435},  \eqref{eq:436}, \eqref{eq:436},
and \eqref{eq:437},  we 
then conclude
\begin{equation}\label{eq:438}
  \|v\|_2\leq C \beta^{-1} \|f\|_2\,.
\end{equation}
The estimate of $v^\prime$ in \eqref{eq:432} follows immediately from the
identity
\begin{displaymath}
  \|v^\prime\|_2^2=\mu\beta\|v\|_2^2 +\Re\langle v,(U-\nu)f\rangle\,, 
\end{displaymath}
together with \eqref{eq:438} and the fact that
  $\mu\leq \Upsilon \beta^{-1/2}$. 
\end{proof}

\section{No-Slip Schr\"odinger operators}
\label{sec:4}
\subsection{Preliminaries}  Given the fact that
  $-\phi^{\prime\prime}+\alpha^2\phi$ does not necessarily belong to $D(\LL_\beta^{\Nf,\Df})$, we
  establish, in this section, resolvent estimates for the same
  differential operator but with one boundary condition replaced by an
  integral condition which will be satisfied by $-\phi^{\prime\prime}+\alpha^2\phi$ (cf. also \cite[Section 6]{almog2019stability}
  and the discussion around \eqref{eq:heuristiczeta} in the
  introduction.)

This section follows the same path as in \cite[Section
6]{almog2019stability} but this time for symmetric flows in $(-1,+1)$
(so that $U^\prime(0)=0$), to which end we consider the interval $(0,1)$
and a Neumann condition at $x=0$.

Let
\begin{displaymath}
  \LL_\beta^\zeta =-\frac{d^2}{dx^2}+i\beta U \,,
\end{displaymath}
be defined on
  \begin{equation}
\label{eq:439}
    D(\LL_\beta^\zeta)= \{ u\in H^2(0,1)\,| \, \langle\zeta,u\rangle=0\,,\;u^\prime(0)=0\} \,,
  \end{equation}
  where $\zeta\in H^2(0,1)$. \\
  We will later (see \eqref{eq:602}) confine the discussion to the
    case where $\zeta_\alpha (x) =\cosh(\alpha x)/\cosh \alpha$.

  More precisely, we introduce
\begin{equation}\label{eq:440}
\mathfrak U_0 := \{ \zeta \in H^2(0,1)\,|\, \zeta^\prime(0)=0,\; \zeta(1)=1\}
\end{equation}
and 
for $\beta \geq 0$, $\gamma>0$, $\theta > 0$  and $\lambda \in \mathbb C$,  the subset
\begin{equation}
  \label{eq:441}
  {\mathfrak U}_1(\beta,\lambda,\gamma,\theta) =\{  \zeta \in \mathfrak U_0\;, \;\|\zeta\|_\infty\leq \beta^\gamma,\;
  \|\zeta^\prime\|_{L^2(1-\beta^{-\gamma},1)}\leq\theta\beta^{1/6}\lambda_\beta^{1/4}\}\,,
\end{equation}
where
\begin{equation}\label{eq:lambdabeta}
   \lambda_\beta=1+|\lambda|\beta^{1/3}\,.
\end{equation}
Let $\Ai$ denote Airy's function (See \cite[Section 10.4]{abst72})
and $A_0$ denotes the generalized Airy function 
\begin{equation}\label{eq:defA0}
\C\ni z \mapsto A_0(z) = e^{i\pi/6} \int_z^{+\infty} {\rm Ai}(e^{i\pi/6} t)\,dt\,.
\end{equation}
(See \cite[10.4]{abst72} or in \cite[appendix
A.2]{almog2019stability}). We then set
\begin{displaymath}
  \Sg  =\{\,z\,|\,A_0 (i z)=0\,\}\,.
\end{displaymath}
It is established in \cite[appendix A.2]{almog2019stability} that
$\Sg$ (which is denoted there by $\Sg_\lambda$) is non empty.  In \cite{Wa2} (cf. also \cite[Appendix A]{almog2019stability})
  it is shown that 
\begin{equation}\label{deftheta1r}
\vartheta_1^r:=  \inf_{z\in\Sg} \Re z >0\,.
\end{equation}

\subsection{Resolvent estimates for $  |U(0)-\nu|\gg\beta^{-1/2}$} 
\label{sec:resolvent-estimates-linear}
We can now state 
\begin{proposition}
\label{lem:integral-conditions}   
Let $U\in C^2([0,1])$  satisfy \eqref{eq:10}.  Let further $\gamma<1/4$.
 Then, there exist $\Upsilon >0$, $\beta_0>0$, $a>0$, and
$\theta_0>0$ such that, for all $\beta\geq \beta_0$, $\theta \in (0,\theta_0]$, and $\lambda\in \C$ 
satisfying
\begin{equation}
\label{eq:442} 
 \max(|U(0)-\nu|^{-1/3},1)\beta^{1/3}\Re\lambda \leq  \Upsilon \,,
\end{equation}
and 
\begin{equation}
\label{eq:443}
  | U(0)-\nu | >a\beta^{-1/2}\, , 
\end{equation}
there exists a constant $C>0$ such that, for any $f\in H^1(0,1)$, $\zeta
\in {\mathfrak U}_1(\beta,\lambda,\gamma,\theta)$ and $v\in D(\mathcal L_\zeta^\beta)$
satisfying
\begin{equation}
  \label{eq:444}
(\mathcal L^\beta_\zeta-\beta\lambda)v=f \,,
\end{equation} 
it holds that 
\begin{equation}
  \label{eq:445}
  \begin{array}{l}
|v(1)| \leq C\,\beta^{1/3} \,  \lambda_\beta^{1/2}  \, \|\zeta\|_\infty \, \times  \\ 
 \times 
 \min \left([x_\nu\beta]^{-5/6}\|f\|_2,[|U(0)-\nu|^{1/2}\beta]^{-1}[\|f\|_{1,2} +
     |f(x_\nu)|\log(1+x_\nu\beta^{1/4})]\right)\,,
  \end{array}
\end{equation}
where $x_\nu$ is defined by \eqref{eq:defxnu},
and that
\begin{equation}
  \label{eq:446}
|v(1)| \leq C\, \beta^{1/3}\,  \lambda_\beta^{1/2} \, \|\zeta(\LL^\Nf_\beta-\beta\lambda)^{-1}f\|_1\,.
\end{equation}
Furthermore, for any  $p>1$ and $\nu<U(0)-a\beta^{-1/2}$ it holds that
\begin{subequations}
 \label{eq:447}
  \begin{equation}
|v(1)| \leq C\, \lambda_\beta^{1/2}\beta^{1/3}[x_\nu\beta]^{-1}\big[\|\zeta\|_\infty\|f\|_{1,2}+ \|\zeta\|_{1,p}(1+|\log[\nu+ i\mf]|)\,|f(x_\nu)| \big]\,,
\end{equation}
where
\begin{equation}
\mf:= - \max(-\mu, x_\nu^{2/3} \beta^{-1/3})\,.
\end{equation}
\end{subequations}
Finally, if in addition $(x-x_\nu)^{-1}f\in L^2(0,1)$  and
  $\nu<U(0)$, then we have 
\begin{equation}
  \label{eq:448} 
|v(1)| \leq C \beta^{-2/3}\,  \lambda_\beta^{1/2} \, \|\zeta\|_\infty \,|U(0)-\nu|^{-1/4}\Big\|\frac{f}{x-x_\nu}\Big\|_2\,. 
\end{equation}
\end{proposition}
 Note that since ${\mathfrak
    J}_\nu=|U^\prime(x_\nu)|\sim x_\nu\sim[U(0)-\nu]^{1/2}$ for $\nu<U(0)$, \eqref{eq:442} is
  similar to the condition set on $\Re\lambda$ in \eqref{eq:302}.  We use here a different
  notation since we need to address the case $\nu>U(0)$ as well. 

  \begin{proof}
As in \cite[Eq. (6.20)]{almog2019stability} we begin by a
decomposition of $v$ into a boundary term associated with  $x=1$ and a solution of the same
equation satisfying a Dirichlet condition at $x=1$. We estimate the
boundary term by using a linear approximation of $U$ near $x=1$
(recall that by \eqref{eq:10} $U(1)=0$ and $U^\prime(1)=-1$).  Let for $\beta >0$ and $\lambda \in \mathbb C\,$, 
   \begin{equation}
\label{eq:449}
  \tilde{\psi}_{\lambda,\beta} (x)= e^{- i \pi/6} \frac{{\rm Ai}\big(\beta^{1/3}e^{ -i\pi/6}[(1- x)- i\lambda]\big)}
{\overline{A_0(i \beta^{1/3}\bar{\lambda})}}\,,
  \end{equation}
  which is (note that $\tilde{\psi}=\psi_+$, with $J_+=1$, in \cite[Eq.
  (6.8b)]{almog2019stability}) the decaying solution as $x\to - \infty$ of
      \begin{equation}\label{eq:450}
    \begin{cases}
      \Big(-\frac{d^2}{dx^2} + i \beta [(x-1) +i\lambda]\Big) \tilde{\psi} =0 &
      \text{in } (-\infty,1)\,,  \\
             \int_{-\infty}^1 \tilde \psi (x) \,dx =  \beta^{-1/3} \,. 
    \end{cases}
  \end{equation}
  Note further that by \eqref{deftheta1r} $\tilde \psi_{\lambda,\beta}$ is well-defined
    whenever $\Re\lambda<\vartheta_1^r$.  For later reference we recall from
  \cite[{Eq.  (8.87)}]{almog2019stability} that  for any $\hat
    \delta_1$,  there exists $C>0$ and $\beta_0$  such that,  for $\Re \lambda \leq
    (\vartheta_1^r -\hat \delta_1) \beta^{-1/3} $ and  $\beta \geq \beta_0\,$,
 \begin{equation}
   \label{eq:451} 
    \frac{1}{C} \lambda_\beta^{1/2}  \leq  |\tilde \psi_{\lambda,\beta}(1)| \leq C\,  \lambda_\beta^{1/2} \,.
\end{equation}
To guarantee that the Neumann condition at $x=0$ is satisfied we
further set
\begin{equation}
  \label{eq:452}
\psi_{\lambda,\beta}(x)=\tilde{\psi}_{\lambda,\beta} (x)\chi(1-x)\,,
\end{equation}
where $\chi$ is given by \eqref{eq:103}, which we recall here for the
convenience of the reader
\begin{displaymath}
    \chi(t)=
  \begin{cases}
    1 & t<1/2 \\
    0 & t >3/4 \,.
  \end{cases}
\end{displaymath} 
Consequently, $ \psi_{\lambda,\beta}$ is supported on $[1/4,+1]$ and $
\psi_{\lambda,\beta} = \tilde \psi_{\lambda,\beta}$ on $[1/2,1]$.  We omit the subscript
$(\lambda,\beta)$ when no ambiguity is expected.  Consider a pair
$(\tilde v ,h)\in L^2(0,1) \times D(\LL_\beta^\Nf)$ such that
 \begin{equation}\label{eq:453}
  h  = \Big(-\frac{d^2}{dx^2} +i\beta(U+i\lambda)\Big)\psi\,,
\end{equation}
and 
\begin{equation}
\label{eq:454}
   (\LL_\beta^\Nf-\beta\lambda)\,\tilde{v} = h\,.
\end{equation}
 We note that the assumptions of the proposition, and in particular
\eqref{eq:442} and \eqref{eq:443} allow us to apply Propositions
\ref{prop5.3}, \ref{lem:Dirichlet-L1}, \ref{H1-estimate-0} and
\ref{Dirichlet-L1-H1-0} throughout the proof.\\

 {\em Step 1: We prove that 
\begin{equation}
\label{eq:455}
 \beta^{-1/6}\|\tilde{v}\|_2 +  \|\tilde{v}\|_1 \leq
 C\max(\beta^{-2/3}\lambda_\beta^{-3/4},\beta^{-5/3})\,.
\end{equation}
}
By \eqref{eq:450} it holds that
\begin{equation}
\label{eq:456}
  h= i\beta\, [U-(1-x)]\,
  \psi+\chi^{\prime\prime}(1-x)\tilde{\psi}-2\chi^\prime(1-x)\tilde{\psi}^\prime \mbox{ in } (0,1)\,.
\end{equation}
We note that by \cite[Eq. (6.17)]{almog2019stability} 
 there exists $\Upsilon >0$ (in the statement of the proposition) such that
 whenever $\beta^{1/3}\Re\lambda \leq\Upsilon$  
\begin{equation}
\label{eq:457}
   \|(1-x)^k\, \psi_{\lambda,\beta}\|_2 \leq \|(1-x)^k\, \tilde{\psi}_{\lambda,\beta}\|_2 \leq C\, \lambda_\beta^{\frac{1-2k}{4}} \beta^{-(1+2k) /6}\,,
\end{equation}
for $k \in [0,4]$\,. \\
Furthermore, since by \cite[Proposition A.1]{almog2019stability} 
  (or more precisely by  \cite[Eq. (A.4), (A.6), (A.19), (A.20)]{almog2019stability} and the
  display below \cite[Eq. (A.29)]{almog2019stability})
it holds that
 \begin{displaymath}
   \Psi(x,\lambda):=\frac{\Ai(e^{i\pi/6}[x+i\lambda])}{\Ai(e^{i2\pi/3}\lambda)}= \Big[1+
   \frac{i}{4}  \left((-\lambda)^{- 1/2}x^2 -\lambda^{-1} x
   \right)\Big]e^{-(-\lambda)^{1/2}x} + w_1 \,,
 \end{displaymath}
where:
\begin{itemize}
\item  For $\mu_0>0$
\begin{displaymath}
  \lambda \in \mathcal V(\mu_0):=  \{ \Re \lambda \leq \mu_0\} \cap \{|\lambda|>3\mu_0\}\,,
\end{displaymath}
\item the square root of $-\lambda$ is chosen such that
\begin{displaymath}
  \, \Re(-\lambda)^{1/2}>0\,,
\end{displaymath}
\item and the remainder $w_1\in H^1(\mathbb R_+)$ satisfies 
\begin{displaymath}
 \|x^4w_1\|_{L^2(\R_+)} +   \|x^4w_1^\prime\|_{L^2(\R_+)} \leq C\, |\lambda|^{-9/4} \,.
\end{displaymath}
\end{itemize}
Consequently,
for all $\lambda\in\mathcal V(\mu_0)$ it holds by
\cite[Eq. (A.20)]{almog2019stability} that
\begin{displaymath}
\|x^4\Psi\|_2 +  |\lambda|^{-1/2} \|x^4\Psi^\prime\|_2 \leq C\, |\lambda|^{-9/4} \,.
\end{displaymath}

Let $0<\mu_0<\hat{\kappa}_1$ (given by \eqref{eq:defkappa1}).   Then, there
exists $C(\mu_0)>0$ such that
\begin{displaymath}
 \sup_{|\lambda|\leq3\mu_0} \|x^4\Ai(e^{i\pi/6}[x+i\lambda])\|_{1,2} \leq C \,,
\end{displaymath}
and, since all the zeroes of $\Ai(e^{i2\pi/3}\lambda)$ are located in the 
half-plane $\Re\lambda\geq\hat{\kappa}_1$,   we have, since $\mu_0<\hat{\kappa}_1$, that
\begin{displaymath}
  \sup_{  \begin{subarray} \quad |\lambda| \leq 3 \mu_0 \\
      \Re\lambda<\mu_0
    \end{subarray}}\Big| \frac{1}{\Ai(e^{i2\pi/3}\lambda)}\Big| \leq C\,.
\end{displaymath}
Consequently, by the above inequalities, relying on the sole condition
that $\Re \lambda \leq \mu_0$, there exists $C >0$ such that
  \begin{displaymath}
     \|x^4\Psi\|_2 +   [1+|\lambda|^2]^{-1/4}\|x^4\Psi^\prime\|_2 \leq C\, [1+|\lambda|^2]^{-9/8} \,.
  \end{displaymath}
Note that
\begin{displaymath}
  \frac{\tilde{\psi}_{\lambda,\beta}(x)}{\tilde{\psi}_{\lambda,\beta}(1)}=\Psi(\beta^{1/3}(1-x),\beta^{1/3}\lambda) \,. 
\end{displaymath}
Using dilation, translation and \eqref{eq:451} the above yields 
\begin{equation}
\label{eq:458}
  \|(1-x)^k\, \tilde{\psi}_{\lambda,\beta}^\prime\|_{L^2(-\infty,1)}\leq C\,
  \lambda_\beta^{\frac{3-2k}{4}} \beta^{(1-2k) /6}\,.
\end{equation}
Hence, rewriting \eqref{eq:456} in the form
\begin{equation*}
  h= i\beta\, [U-(1-x)]\,
  \psi+ ((1-x)^{-2}\chi^{\prime\prime}(1-x) ) (1-x)^2 \tilde{\psi}-2((1-x)^{-3} \chi^\prime(1-x))\, (1-x)^3 \tilde{\psi}^\prime \,,
\end{equation*}
we obtain
\begin{equation}
\label{eq:459}
  |h(x)|\leq C[\beta(1-x)^2|\psi(x)|+(1-x)^3 |\tilde{\psi}^\prime(x)|]\,,
\end{equation}
and hence by \eqref{eq:457} with $k=2$  and \eqref{eq:458} with
$k=3$, 
\begin{equation}
\label{eq:460}
  \|h\|_2\leq C\, \beta^{1/6} \lambda_\beta^{-3/4}\,. 
\end{equation}

Recall from \eqref{eq:defxnu} the definition of $x_\nu\,$:
\begin{equation*}
  U(x_\nu)=\nu \mbox{ for } 0 < \nu < U(0)\,,\, x_\nu=1 \mbox{ if } \nu\leq 0 \mbox{ and } x_\nu =0 \mbox{ if } \nu > U(0)\,.
  \end{equation*}
We split the proof of \eqref{eq:455} into two steps, depending on the
value of $x_\nu$. \\

{\em Step  1.1. $x_\nu>1/4$ . }\\

In this case we have by \eqref{eq:454}, (\ref{eq:287}a),
\eqref{eq:354}, and \eqref{eq:403}, (note that $\Jf_\nu \geq \Jf_{U^{-1}
  (1/4)}>0$ in this case) that
\begin{displaymath} 
\beta^{-1/3}\|\tilde{v}\|_2 +   \|(U-\nu)\tilde{v}\|_2 +\beta^{-1/6}\|\tilde{v}\|_1\leq C\, \beta^{-1}\, \|h\|_2 \,.
\end{displaymath}
By \eqref{eq:460} we then obtain
\begin{equation}
\label{eq:461}
 \beta^{-1/3}\|\tilde{v}\|_2 +  \|(U-\nu)\tilde{v}\|_2 +\beta^{-1/6}\|\tilde{v}\|_1\leq C\beta^{-5/6}\lambda_\beta^{-3/4}\,,
\end{equation}
 readily yielding \eqref{eq:455}. \\

{\em Step  1.2:  $0<x_\nu\leq1/4\,$.} \\
We recall  that  $x_\nu\geq Ca^{1/2}\beta^{-1/4}$.
We write
\begin{displaymath}
   (\LL_\beta^\Nf-\beta\lambda)(\chi\tilde{v})=\chi h-2\chi^\prime\tilde{v}^\prime- \tilde{\chi}^{\prime\prime}\tilde{v} \,,
\end{displaymath}
to obtain by (\ref{eq:287}a) and  \eqref{eq:403}
\begin{equation}
\label{eq:462}
  \|\chi\tilde{v}\|_2+ [\beta x_\nu]^{1/6} \|\chi\tilde{v}\|_1\leq
  \frac{C}{[\beta x_\nu]^{2/3}}(\|\chi h\|_2 +\|\tilde{v}^\prime\|_2+
  \|\tilde{v}\|_2)\,.
\end{equation}
Let $\tilde{\chi}=\sqrt{1-\chi^2}$ where $\chi$ is chosen such
 that $\tilde \chi \in C^\infty(\R)$.  We note that the support of $\tilde
 \chi$ belongs to $[1/2,+\infty)$.    Then, an integration by parts yields, as
 in \eqref{eq:307}-\eqref{eq:308}
\begin{equation}
\label{eq:463}
  \|(\tilde{\chi}\tilde{v})^\prime\|_2^2=\|\tilde{\chi}^\prime \tilde{v}\|_2^2+\mu\beta\|\tilde{\chi} \tilde{v}\|_2^2 +
  \Re\langle\tilde{\chi} \tilde{v},\tilde{\chi} h\rangle \,, 
\end{equation}
and, observing that  the sign of $(U-\nu)$ is constant on the support of $\tilde \chi$, 
\begin{equation}
\label{eq:464}
 - \beta\||U-\nu|^{1/2}\tilde{\chi}\tilde{v}\|_2^2+2\Im\langle\tilde{\chi}^\prime \tilde{v},(\tilde{\chi}
  \tilde{v})^\prime\rangle =  \Im \langle\tilde{\chi} \tilde{v},\tilde{\chi} h\rangle \,. 
\end{equation}

Combining \eqref{eq:463}, \eqref{eq:464},  given the support of $\tilde
\chi$ and the fact that $x_\nu\leq1/4$, yields
\begin{multline*}
  \beta\|\tilde{\chi}\tilde{v}\|_2^2 \leq  C ( \| \tilde{v}\|_2\, \|(\tilde{\chi}
  \tilde{v})^\prime\|_2+ \| \tilde{\chi} \tilde{v}\|_2 \| h \|_2) \\ \leq  C( \| \tilde{v}\|_2^2+\|(\tilde{\chi}
  \tilde{v})^\prime\|_2^2+  \| \tilde{\chi} \tilde{v}\|_2 \| h \|_2)\\  \leq 2 C( \|
  \tilde{v}\|_2^2+  \mu\beta\|\tilde{\chi}\tilde{v}\|_2^2 +\| \tilde{\chi} \tilde{v}\|_2 \| h \|_2)\,.
\end{multline*}
Observing that $ \mu \leq \Upsilon \mathfrak J_\nu \beta^{-\frac 13} \leq C
\beta^{-\frac 13}$, we may conclude the existence of  $\beta_0>0$ and $\hat
C>0$ such that for all $\beta>\beta_0$ 
\begin{equation}\label{eq:465}
   \|\tilde{\chi} \tilde{v}\|_2 \leq C \beta^{-1}\, \left(\|h\|_2+\beta^{1/2} \|\tilde{v}\|_2\right)  \,.
 \end{equation} 
 Combining \eqref{eq:465} with the control of $\|\chi \tilde v\|_2$
 given in \eqref{eq:462} yields for sufficiently large $\beta_0$
\begin{equation}
\label{eq:466}
   \|\tilde{v}\|_2 \leq C\, \left(\beta^{-1}\|h\|_2 +[\beta x_\nu]^{-2/3}(\|\chi h\|_2 +\|\tilde{v}^\prime\|_2)\right)\,.
\end{equation}
By \eqref{eq:398} (with $v=\tilde{v}$ and $f=h$) and the fact that
$\mu\leq C\beta^{-1/3}x_\nu^{2/3}$ we can conclude that
\begin{equation}
\label{eq:467}
  \|\tilde{v}^\prime\|_2\leq C\, ([\beta x_\nu]^{1/3}\|\tilde{v}\|_2+\|\tilde{v}\|_2^{1/2}\|h\|_2^{1/2})\,.
\end{equation}
Since by \eqref{eq:443} $\beta x_\nu$ is large for sufficiently large
$\beta_0$, substituting \eqref{eq:467} into \eqref{eq:466} yields, ,
\begin{equation*}
   \|\tilde{v}\|_2 \leq C[(\beta^{-1}+[\beta x_\nu]^{-4/3})\|h\|_2 +[\beta x_\nu]^{-2/3}\|\chi h\|_2]\,.
\end{equation*} 
Hence, using the fact that $x_\nu\geq\beta^{-1/4}$, 
\begin{equation}
\label{eq:468}
   \|\tilde{v}\|_2 \leq C(\beta^{-1} \|h\|_2 + \beta^{-1/2} \|\chi h\|_2]\,.
\end{equation} 
By \eqref{eq:456}, \eqref{eq:457} and  \eqref{eq:458}, we
obtain, since $\chi$ is supported on $[0,3/4]$,  
\begin{multline}
\label{eq:469} 
  \|\chi h\|_2 \leq C\,\big[\beta\|(1-x)^4\psi\|_2+  \|(1-x)^4\psi^\prime\|_2\big]
   \leq
   \widehat C
   \beta^{-1/2}\lambda_\beta^{-\frac{5}{4}}\big(\lambda_\beta^{-\frac{1}{2}}+\beta^{-2/3}\big) \,.
\end{multline}
Substituting \eqref{eq:469}  together with \eqref{eq:460}  into \eqref{eq:468} gives
\begin{equation}\label{eq:470} 
  \|\tilde{v}\|_1 \leq    \|\tilde{v}\|_2 \leq C \beta^{-5/6}\lambda_\beta^{-3/4} \,,
\end{equation}
completing the proof of \eqref{eq:455}, for $ 0<x_\nu\leq1/4$ as
well.\\

 {\em Step  1.3:  $U(0)-\nu\leq-a\beta^{-1/2}$.}\\
Since
\begin{equation}
\label{eq:471}
  \Im\langle\chi^2\tilde{v},
  (\LL_\beta^\Nf-\beta\lambda)\tilde{v}\rangle=-\beta\|(\nu-U)^{1/2}\chi\tilde{v}\|_2^2+2\Im\langle\chi^\prime\tilde{v},(\chi\tilde{v})^\prime\rangle\,,
\end{equation}
we may conclude that
\begin{equation}
\label{eq:472}
  \beta\|(\nu-U)^{1/2}\chi\tilde{v}\|_2^2 \leq \|\chi\tilde{v}\|_2\|\chi h\|_2+
  C\|\tilde{v}\|_2\|(\chi\tilde{v})^\prime\|_2\,. 
\end{equation}
As
\begin{displaymath}
   \Re\langle\chi^2\tilde{v},
  (\LL_\beta^\Nf-\beta\lambda)\tilde{v}\rangle=\|(\chi\tilde{v})^\prime\|_2^2-\mu\beta\|\chi\tilde{v}\|_2^2-\|\chi^\prime\tilde{v}\|_2^2\,,
\end{displaymath}
we can conclude that
\begin{displaymath}
  \|(\chi\tilde{v})^\prime\|_2\leq C\left(\mu_{\beta,+}^{1/2}\|\chi\tilde{v}\|_2+\|\tilde{v}\|_2 + \|\chi\tilde{v}\|_2^{1/2}\|\chi h\|_2^{1/2}\right) \,,
\end{displaymath}
where $\mu_{\beta,+}$ is defined in \eqref{eq:284}.
Substituting the above into \eqref{eq:472} yields
\begin{displaymath}
  \beta\|(\nu-U)^{1/2}\chi\tilde{v}\|_2^2 \leq\|\chi\tilde{v}\|_2\|\chi h\|_2 +C
  \left(\mu_{\beta,+}^{1/2}\|\chi\tilde{v}\|_2\|\tilde{v}\|_2 +
  \|\tilde{v}\|_2^2+\|\chi\tilde{v}\|_2\|\chi h\|_2\right) \,.
\end{displaymath}
Since $\nu-U\geq\nu-U(0)$ in $[0,1]$ we may now conclude that
\begin{equation}
\label{eq:473}
  \|\chi\tilde{v}\|_2 \leq
  \frac{C}{\beta(\nu-U(0))}(\|\chi h\|_2+(\mu_{\beta,+}^{1/2}+[\beta(\nu-U(0))]^{1/2})\|v\|_2 \,.
\end{equation}
Given that \eqref{eq:465} remains valid for $\nu>U(0)$ it holds that
\begin{displaymath}
  \|\tilde{v}\|_2\leq   \|\chi\tilde{v}\|_2 +   \|\tilde{\chi}\tilde{v}\|_2
  \leq   \|\chi\tilde{v}\|_2 +C \beta^{-1} \,\left(\|h\|_2+\beta^{1/2} \|\tilde{v}\|_2\right)  \,.
\end{displaymath}
Consequently, there exists $\beta_0$ such that for  $\beta\geq \beta_0$ we may write
\begin{equation}
\label{eq:474}
  \|\tilde{v}\|_2\leq   \|\chi\tilde{v}\|_2 + C \beta^{-1} \|h\|_2\,.
\end{equation}
Substituting the above into \eqref{eq:473} yields, for some
sufficiently large $\beta_0$ and $\beta\geq \beta_0$
\begin{equation}\label{eq:475}
    \|\chi\tilde{v}\|_2 \leq
    C\, \beta^{-1} (\nu-U(0))^{-1}\, (\|\chi h\|_2+[\nu-U(0)]^{1/2})\beta^{-1/2}\|h\|_2) \,.
\end{equation}
By \eqref{eq:469} 
and \eqref{eq:460} we then obtain (note that
$|\lambda|>U(0)$ as $\nu -U(0) \geq a \beta^{-1/2}$)  
\begin{equation}\label{eq:476}
    \|\chi\tilde{v}\|_2 \leq C \, \beta^{-4/3} \,.
\end{equation}
Substituting the above into \eqref{eq:474} yields
\begin{displaymath}
   \|\tilde{v}\|_1 \leq    \|\tilde{v}\|_2 \leq C\beta^{-5/6}\lambda_\beta^{-3/4}\,.
\end{displaymath}

{\em Step 2: We prove (\ref{eq:445})  and \eqref{eq:447}.\\}

\noindent {\em Step  2.1: We prove  (\ref{eq:445}) in the case $\nu<U(0)$.}

Consider the pairs $(v,f)\in D(\LL_\beta^\zeta) \times L^2(0,1)$ satisfying
\eqref{eq:444} and $(\psi, \tilde v)$ satisfying 
\eqref{eq:449}-\eqref{eq:454}. As in \cite[Eq.
(6.20)]{almog2019stability} 
there exists $(A,u) \in \mathbb C\times 
D(\LL_\beta^\Nf)$ such that
\begin{equation}
\label{eq:477}
  v=A(\psi-\tilde{v})+u \,,
\end{equation}
where 
\begin{equation}
\label{eq:478}
  u=(\LL_\beta^\Nf-\beta\lambda)^{-1}f\,.
\end{equation}
Taking the inner product with $\zeta$ yields in view of \eqref{eq:439}  
\begin{equation}
\label{eq:479}
  A\langle\zeta,(\psi-\tilde{v})\rangle =- \langle\zeta,u\rangle \,.
\end{equation}
By \eqref{eq:441} and \eqref{eq:461} it holds that for any
  $0<\gamma<1/4$  there exist positive $C$ and $\beta_0$ such that
\begin{equation}
\label{eq:480}
  |\langle\zeta,\tilde{v}\rangle|\leq C\beta^\gamma\, \|\tilde{v}\|_1\leq \hat C \beta^{-(2/3-\gamma)} \,,
\end{equation}
for all $\beta>\beta_0$.\\
Then, we write
\begin{equation}\label{eq:481}
  \langle\zeta,\psi\rangle=\langle1,\psi\rangle+\langle\zeta-1,\psi\rangle\,.
\end{equation}
For the first term in the r.h.s of \eqref{eq:481}, we may rely on  \cite[Eq.
(6.28)]{almog2019stability} to obtain
\begin{equation}
\label{eq:482}
  \langle1,\tilde \psi\rangle=-\beta^{-1/3}+\OO(\beta^{-4/3})\,.
\end{equation}
Observing that
\begin{displaymath}
|\langle1,\psi -\tilde  \psi\rangle|\leq C\|(1-x)^4\tilde{\psi}\|_2\leq C\lambda_\beta^{-7/4}\beta^{-3/2}\,,
\end{displaymath}
 we obtain 
\begin{equation}
\label{eq:483}
  \langle1, \psi\rangle=-\beta^{-1/3}+\OO(\beta^{-4/3})\,.
\end{equation}
For the second term  in the r.h.s of  \eqref{eq:481} we have  
\begin{equation}
\label{eq:484}
  |\langle\zeta-1,\psi\rangle|\leq
 \|\zeta^\prime\|_{L^2(1-\beta^{-\gamma},1)}\|
  \|(1-x)^{1/2}\psi\|_1 + (1+ \|\zeta\|_\infty)\|\psi\|_{L^1(0,1-\beta^{-\gamma})}\,.
\end{equation} 
By  \cite[Eq. (6.27)]{almog2019stability} it holds that
\begin{displaymath}
  \|\psi\|_{L^1(0,1-\beta^{-\gamma})}\leq\beta^{3\gamma}\|(1-x)^3\psi\|_1\leq C\beta^{3\gamma-4/3} \,,
\end{displaymath}
and that
\begin{equation}
\label{eq:485} 
    \|(1-x)^{1/2}\psi\|_1 \leq C\beta^{-1/2} \lambda_\beta^{-1/4} \,.
\end{equation}
 Consequently, we obtain from \eqref{eq:484}  that (recall that $\zeta \in  {\mathfrak U}_1(\beta,\lambda,\gamma,\theta)$) 
\begin{equation}
\label{eq:486}
   |\langle\zeta-1,\psi\rangle|\leq C_0 \, (\theta\beta^{-1/3}+\beta^{4\gamma-4/3})\,.
\end{equation}
As
\begin{displaymath}
\langle\zeta,(\psi-\tilde{v})\rangle=  \langle1,\psi\rangle +  \langle (\zeta- 1) ,\psi\rangle -\langle\zeta,\tilde{v})\rangle\,, 
\end{displaymath}
we can conclude from \eqref{eq:486}, \eqref{eq:483}, and
\eqref{eq:480} that
\begin{equation}\label{eq:487}
 | \langle\zeta,(\psi-\tilde{v})\rangle| \geq \beta^{-1/3} (1- C_0 \theta - \hat C \beta^{4\gamma -1} -  \hat C \beta^{-1}  -\hat C \beta^{\gamma -1/3})\,.
\end{equation}
We choose $\theta_0 =1/2 C_0$, and since $\gamma<1/4$, there exists $\beta_0$
such that under the assumptions of the proposition we can conclude
from \eqref{eq:479} and \eqref{eq:487} that
\begin{equation}
\label{eq:488}
  |A|\leq C\beta^{1/3} |\langle\zeta,u\rangle|\leq C\beta^{1/3}\|\zeta\|_\infty\|u\|_1 \,. 
\end{equation}
For $\nu<U(0)$ we may use \eqref{eq:431}, the fact that $x_\nu \geq C \beta^{-1/4}$,
and  the $\LL(L^1,L^2)$ estimate in \eqref{eq:403} to obtain that
\begin{equation}
\label{eq:489}
   |A|\leq C\, \|\zeta\|_\infty \, \min
  \left(x_\nu^{-5/6}\beta^{-1/2}\|f\|_2,x_\nu^{-1}\beta^{-2/3}[\|f\|_{1,2} +|f(x_\nu)|\log (x_\nu\beta^{1/4})]\right)\,.
 \end{equation}
 We can now conclude \eqref{eq:445} from \eqref{eq:489},
 \eqref{eq:451},  \eqref{eq:452}, and the fact (see
   \eqref{eq:477}) that $ v(1)=A\psi(1)$.\\  
 
{\em  Step  2.2. We prove  \eqref{eq:447}.  \\}

 To prove \eqref{eq:447} we now write
\begin{displaymath}
  \Big\langle\zeta,\frac{1}{U-\nu+i\mf}\Big\rangle= \zeta(x_\nu)\,
  \int_0^1\frac{ dx}{U-\nu+i\mf} +
  \int_0^1\frac{[\zeta- \zeta(x_\nu) ]\, dx}{U-\nu+i\mf}  \,,
\end{displaymath}
where $\mf$ is defined in (\ref{eq:447}b).  For the coefficient of
$\zeta(x_\nu)$ in the first term on the right hand side we write
\begin{equation}
\label{eq:490}
   \int_0^1\frac{dx}{U-\nu+i\mf}
   = 
   \frac{1}{2\tilde{x}_\nu}\int_0^1\Big[\frac{1}{[U(0)-U]^{1/2}-\tilde{x}_\nu}
   -\frac{1}{[U(0)-U]^{1/2}+\tilde{x}_\nu} \Big]\,dx\,,
   \end{equation}
where
\begin{displaymath}
  \tilde{x}_\nu=[U(0)-\nu+i\mf]^{1/2}\,.
\end{displaymath}
An integration by parts yields
\begin{multline}
\label{eq:491}
  \int_0^1\frac{dx}{[U(0)-U]^{1/2}\pm \tilde{x}_\nu}=
  \frac{2[U(0)-U]^{1/2}}{U^\prime}\log\big([U(0)-U]^{1/2}\pm \tilde{x}_\nu\big)\Big|_0^1\\
  -\int_0^1\Big(\frac{2[U(0)-U]^{1/2}}{U^\prime}\Big)^\prime\log\big([U(0)-U]^{1/2}\pm
  \tilde{x}_\nu\big)\,dx\,.
\end{multline} 
Given that
\begin{displaymath}
  \Big\|\Big(\frac{2[U(0)-U]^{1/2}}{U^\prime}\Big)^\prime\|_\infty \leq C \,,
\end{displaymath}
it holds,  for $\mu\geq-1$ that
\begin{equation}
\label{eq:492}
  \Big|\int_0^1\Big(\frac{2[U(0)-U]^{1/2}}{U^\prime}\Big)^\prime\log\big([U(0)-U]^{1/2}\pm \tilde{x}_\nu\big)\,dx\Big|\leq \widehat C\,. 
\end{equation}
We now observe that
\begin{multline}
\label{eq:493}
  \frac{2[U(0)-U]^{1/2}}{U^\prime}\log\big([U(0)-U]^{1/2}\pm\tilde{x}_\nu\big)\Big|_0^1=\\ -\Big[\frac{2}{U^{\prime\prime}(0)}\Big] \log\big(\pm\tilde{x}_\nu\big)+2\sqrt{U(0)}\log\big([U(0)]^{1/2}\pm\tilde{x}_\nu\big)\,.
\end{multline}
Given that
\begin{displaymath}
  \log\big(-\tilde{x}_\nu\big)-\log\big(+\tilde{x}_\nu\big)=i\pi \,,
\end{displaymath}
and 
\begin{displaymath}
  \big|[U(0)]^{1/2}-\tilde{x}_\nu\big|\geq \frac 1C \, |\nu+i\mf| \,,
\end{displaymath} 
we obtain, by substituting \eqref{eq:493} together with
\eqref{eq:492}
into \eqref{eq:491},  given that \break $\frac 1C x_\nu\leq\tilde{x}_\nu$,
\begin{equation}
\label{eq:494}
  \Big|\int_0^1\frac{\zeta(x_\nu) dx}{U-\nu+i\mf}\Big|\leq C \, x_\nu^{-1} [1+|\log|\nu+i\mf|^{-1}|] \|\zeta\|_\infty \,.
\end{equation}
 For $\mu\leq-1$ it holds that
  \begin{displaymath}
    \Big|\int_0^1\frac{\zeta(x_\nu) dx}{U-\nu+i\mf}\Big|\leq \frac{C}{|\mu|} \, \|\zeta\|_\infty 
  \end{displaymath}
in accordance with \eqref{eq:494}.

For the second term we have by  \eqref{eq:425}, for any $p>1$
\begin{displaymath}
\Big|  \int_0^1\frac{\zeta-\zeta(x_\nu) dx}{U-\nu+i\mf}\leq C  \|\zeta^\prime\|_p
  \int_0^1 \frac{|x-x_\nu|^{\frac{p-1}{p}}
    dx} {x_\nu |x-x_\nu|}\leq\frac{\widehat C}{x_\nu}\|\zeta^\prime\|_p\,.
\end{displaymath}
Combining the above with \eqref{eq:494} yields
\begin{equation}
\label{eq:495}
  \Big|\Big\langle\zeta,\frac{1}{U-\nu+i\mf}\Big\rangle\Big|\leq  \frac{C}{x_\nu}{ [1+|\log|\nu+i\mf|^{-1}|]}\|\zeta\|_{1,p}\,.
\end{equation}
By the first inequality of  \eqref{eq:488} it holds that
 \begin{equation}
\label{eq:488bis}
  |A|\leq C\beta^{1/3} \Big[\Big\|\zeta
  \Big[u+i\beta^{-1}\frac{f(x_\nu)}{U-\nu+i\mf}\Big]\Big\|_1 
+ \beta^{-1}|f(x_\nu)|\Big|\Big\langle\zeta,\frac{1}{U-\nu+i\mf}
\Big\rangle\Big|\Big]\,.
\end{equation} 
Substituting \eqref{eq:495} into \eqref{eq:488bis} yields with the aid of
\eqref{eq:421}
\begin{equation}\label{eq:495aa}
  |A|\leq C\beta^{-2/3}x_\nu^{-1}\big[\|\zeta\|_\infty \|f\|_{1,2}+ (1+|\log|\nu+i\mf|^{-1}|)\|\zeta\|_{1,p}|f(x_\nu)|\big] \,.  
\end{equation}
We can now conclude \eqref{eq:447}   from \eqref{eq:495aa}, 
 \eqref{eq:451}, \eqref{eq:452},  and the fact that \break $ v(1)=A\psi(1)$.\\ 

{\em Step  2.3: We prove \eqref{eq:445} in the case $\nu>U(0)$.}\\
In this case we write
  \begin{displaymath} 
    \|u\|_1\leq \|(\nu-U)^{-1/2}\|_2\, \|(\nu-U)^{1/2}u\|_2 \,,
  \end{displaymath}
and as
\begin{displaymath}
  \|(\nu-U)^{-1/2}\|_2^2 \leq C \int_0^1 \frac{dx}{x^2 + \nu - U(0)}\leq \hat C \, |\nu-U(0)|^{-1/2} \,,
\end{displaymath}
we may conclude that
  \begin{equation}
\label{eq:496}
    \|u\|_1\leq \check C \,   [\nu-U(0)]^{-1/4}\, \|(\nu-U)^{1/2}u\|_2 \,.
  \end{equation}
We now use \eqref{eq:471} with $\chi\equiv1$, $\tilde{v}=u$ and $h=f$ to
obtain that
\begin{equation}
\label{eq:497}
  \beta\|(\nu-U)^{1/2}u\|_2^2 = - \Im\langle u,f\rangle\,.
\end{equation}
Consequently,
\begin{displaymath}
  \|(\nu-U)^{1/2}u\|_2\leq \frac{1}{\beta^{1/2}}\|u\|_1^{1/2}\|f\|_\infty^{1/2}\,,
\end{displaymath}
and hence, by \eqref{eq:496} we obtain that
\begin{equation}
\label{eq:498}
   \|u\|_1\leq \frac{C}{\beta[\nu-U(0)]^{1/2}} \|f\|_\infty \leq \frac{\widehat C}{\beta[\nu-U(0)]^{1/2}} \|f\|_{1,2}  \,.
\end{equation}
Next, we use the fact that 
\begin{equation}\label{eq:499}
\nu-U(0)\leq\nu-U(x) \mbox{ for } x \in [0,1]\,,
\end{equation}
to obtain from
\eqref{eq:497} that
\begin{equation}
\label{eq:500}
  \|u\|_2 \leq\frac{1}{\beta[\nu-U(0)]}\|f\|_2 \,.
\end{equation}
As 
\begin{displaymath}
  \|(\nu-U)^{-1}\|_2^2\leq C \, \int_0^1\frac{dx}{[x^2+\nu-U(0)]^2}\leq
  \frac{\hat C}{|\nu-U(0)|^{3/2}} \,,
\end{displaymath}
we may write 
  \begin{equation}
\label{eq:501}
    \|u\|_1 \leq \|(U-\nu)^{-1}\|_2
    \|(U-\nu)u\|_2\leq C[\nu-U(0)]^{-3/4}\|(U-\nu)u\|_2\,.
  \end{equation}
  By \eqref{eq:350} (applied with $v=u$) and \eqref{eq:348} combined
  with \eqref{eq:499}, it holds that
\begin{equation}
\label{eq:502}
\begin{array}{ll}
  \beta\|(U-\nu)u\|_2^2&\leq
  \|(U-\nu)u\|_2\|f\|_2+C\|(U-\nu)^{1/2}u\|_2\|u^\prime\|_2 \\
  & 
  \leq\|(U-\nu)u\|_2\|f\|_2+C\|(U-\nu)u\|_2^{1/2}\|u\|_2^{1/2}\|u^\prime\|_2 \,.
  \end{array}
\end{equation}
 Furthermore by \eqref{eq:398}  and \eqref{eq:442}
 \begin{displaymath}
   \|u^\prime\|_2\leq  \mu_{\beta,+}^{1/2}\|u\|_2+
   \|u\|_2^{1/2}\|f\|_2^{1/2}\leq C (\nu-U(0))^{1/6} \beta^{1/3}  \|u\|_2+ \|u\|_2^{1/2}\|f\|_2^{1/2}\,.
\end{displaymath}
We now use \eqref{eq:500}  to deduce from above
\begin{equation*}
   \|u^\prime\|_2 \leq  C\, \big[(\nu-U(0))^{-5/6} \beta^{-2/3}+  (\nu-U(0))^{-1/2}
   \beta^{-1/2}\big]\|f\|_2\,,
   \end{equation*}
which implies  using \eqref{eq:443} 
\begin{equation}
\label{eq:503}
   \|u^\prime\|_2 \leq  \hat C\, (\nu-U(0))^{-1/2}\beta^{-1/2}\|f\|_2\,.
\end{equation}
Substituting \eqref{eq:500} and \eqref{eq:503}  into \eqref{eq:502}
yields 
\begin{displaymath}
   \beta\|(U-\nu)u\|_2^2
   \leq\|(U-\nu)u\|_2\|f\|_2+\frac{C}{\beta(\nu-U(0))}\|(U-\nu)u\|_2^{1/2}\|f\|_2^{3/2}  \,,
\end{displaymath}
from which we conclude using \eqref{eq:443} that
\begin{displaymath}
  \|(U-\nu)u\|_2\leq C \beta^{-1} [1+\beta^{-1/3}[\nu-U(0)]^{2/3}]\|f\|_2 \leq \frac{\hat C}{\beta}\|f\|_2 \,.
\end{displaymath}
Substituting the above  into
\eqref{eq:501} yields together with \eqref{eq:443}
\begin{displaymath}
    \|u\|_1 \leq
    \frac{C}{\beta[\nu-U(0)]^{3/4}}\|f\|_2 \leq \frac{\hat C}{(\beta[\nu-U(0)]^{1/2})^{5/6}}\|f\|_2 \,,
\end{displaymath}
which, together with \eqref{eq:498} proves that
\begin{displaymath}
      \|u\|_1 \leq
      C\min\left((\beta[\nu-U(0)]^{1/2})^{-5/6}\|f\|_2,(\beta[\nu-U(0)]^{1/2})^{-1}\|f\|_{1,2}\right)\,. 
\end{displaymath}
As $\psi_{\lambda,\beta}(1)=\tilde \psi_{\lambda,\beta}(1)$ we may infer from \eqref{eq:451}
 \begin{equation}
   \label{eq:504} 
    \frac{1}{C} \lambda_\beta^{1/2}  \leq  |\psi_{\lambda,\beta}(1)| \leq C\,  \lambda_\beta^{1/2} \,,
\end{equation}
and hence we can conclude by \eqref{eq:477} and \eqref{eq:488} (which remains
valid for $\nu>U(0)$) that
\begin{displaymath}
  |v(1)|=|A\psi(1)|\leq  C\, \lambda_\beta^{1/2} \min((\beta[\nu-U(0)]^{1/2})^{-5/6}\|f\|_2,x_\nu^{-1}\beta^{-2/3}\|f\|_{1,2})\,.
\end{displaymath}
which verifies \eqref{eq:445} for $\nu>U(0)$.  \\

{\em Step 3:  We prove \eqref{eq:446}. }\\ 
The proof of \eqref{eq:446} which reads
 \begin{equation*}
|v(1)| \leq C\,  \lambda_\beta^{1/2} \, \beta^{1/3}\, \|\zeta u \|_1\,,
\end{equation*}
follows immediately from the first inequality in \eqref{eq:488}, from
\eqref{eq:451}, and again from the fact that $ v(1)=A\psi(1)$.\\

{\it Step 4:  We prove   \eqref{eq:448}.}\\

\noindent To prove it for $x_\nu<1/4$ we set
\begin{displaymath}
  f=(x-x_\nu)g \,,
\end{displaymath}
and assume that $g\in L^2(0,1)$. 
 Recall the definition of $\chi_\nu^\pm$ and
$\tilde{\chi}_\nu$ from \eqref{eq:292}.  Since by
\eqref{eq:357}-\eqref{eq:358} 
\begin{displaymath}
  \||U-\nu|^{-1/2}\chi_\nu^\pm\|_2^2\leq
  C\int_{\frac{5x_\nu}{4}}^1\frac{dx}{x^2-x_\nu^2}+C
  \int_0^{\frac{3x_\nu}{4}}\frac{dx}{x_\nu^2-x^2}\leq\frac{\hat C}{x_\nu}\,,
\end{displaymath}
we can conclude that
\begin{displaymath}
 \|(\chi_\nu^\pm)^2  u\|_1\leq\||U-\nu|^{-1/2}\chi_\nu^\pm\|_2\,
 \||U-\nu|^{1/2}\chi_\nu^\pm u\|_2\leq Cx_\nu^{-1/2}  \||U-\nu|^{1/2}\chi_\nu^\pm u\|_2\,.
\end{displaymath}
By \eqref{eq:420} and \eqref{eq:413}  we then obtain
\begin{displaymath}
  \|(\chi_\nu^\pm )^2u\|_1\leq
  C([\beta x_\nu]^{-1/6}\|u\|_2+\beta^{-1}x_\nu^{-1/2}\|g\|_2)\,.
\end{displaymath}
Hence, by \eqref{eq:412} (which reads $
\|u\|_2\leq C\, [\mathfrak J_\nu \beta]^{-1}\, \|g\|_2\,,$) we can conclude that
\begin{displaymath}
    \|(\chi_\nu^\pm )^2u\|_1\leq C([\beta x_\nu]^{-7/6}+\beta^{-1}x_\nu^{-1/2}])\|g\|_2 \,. 
\end{displaymath}
Given that $x_\nu\geq\frac 1C \beta^{-1/4}$ we obtain that
\begin{equation}
\label{eq:505}
   \|(\chi_\nu^\pm)^2 u\|_1\leq C\beta^{-1}x_\nu^{-1/2}\|g\|_2 \,. 
\end{equation}
Employing again \eqref{eq:412} we write
\begin{displaymath}
    \|\tilde{\chi}_\nu u\|_2\leq\|u\|_2\leq C[\beta x_\nu]^{-1}\|g\|_2 \,. 
\end{displaymath}
Consequently, since $\tilde \chi_v$ is supported on $[x_\nu/2,3x_\nu/2]$
\begin{equation}
\label{eq:506}
  \|\tilde{\chi}_\nu^2u\|_1\leq x_\nu^{1/2} \|\tilde{\chi}_\nu u\|_2\leq
  C\beta^{-1}x_\nu^{-1/2}\|g\|_2 \,. 
\end{equation}
 
Combining \eqref{eq:506} with \eqref{eq:505}, \eqref{eq:446}, and \eqref{eq:488} yields
\eqref{eq:448} for the case \break $x_\nu<1/4$.\\

 In the case $x_\nu\geq1/4$, 
\eqref{eq:448} immediately follows from \eqref{eq:412} and the fact
that $\|u\|_1\leq\|u\|_2$. 
\end{proof}

\subsection{Resolvent estimates for $| U(0)-\Im \lambda| =\OO
  (\beta^{-1/2})\,.$}
\label{sec:resolvent-estimates-quadratic}
Here we introduce, for $\beta >0$, $\lambda \in \mathbb C$,  and $\theta >0$,
\begin{equation}
\label{eq:507}
\mathfrak U_2(\beta,\theta,\lambda) =\{ \zeta \in \mathfrak U_0\,|\,
\|\zeta^\prime\|_2\leq\theta\beta^{1/6}\lambda_\beta^{1/4} \}\,,
\end{equation}
where $\mathfrak U_0$ is introduced in \eqref{eq:440}.  In the present
context $\lambda$ lies in a bounded set and hence
\begin{displaymath}
  \|\zeta^\prime\|_2\leq C \theta\beta^{1/4} \,.
\end{displaymath}

\begin{proposition}
  \label{lem:integral-conditions-quadratic}
  Let $ U\in C^3([0,1])$ satisfy \eqref{eq:10}, $ \Upsilon
  <\sqrt{-U^{\prime\prime}(0)}/2$,  $\mu_1>0$, and $a>0$.  Then, there exist
  $\beta_0>0$ and $\theta_0>0$ such that for all $\beta\geq \beta_0$, $\theta \in (0,\theta_0]$,
  $\lambda\in \C$ satisfying
\begin{subequations}
\label{eq:508}
  \begin{equation}
  U(0)-a\beta^{-1/2}< \nu<U(0)+a\beta^{-1/2}
  \end{equation}
  and 
  \begin{equation} 
  -\mu_1 \leq\mu<\Upsilon\beta^{-1/2}\,, 
\end{equation}
\end{subequations}
for any $\zeta$ in $\mathfrak U_2(\beta,\theta,\lambda)$, and   $(f,v) \in H^1(0,1) \times  D(\mathcal
L^\beta_\zeta)$ satisfying \eqref{eq:444}, 
it holds that
\begin{equation}
 \label{eq:509}
|v(1)| \leq C\|\zeta\|_\infty  \,
  \min(\beta^{-1/8}\|f\|_2,\beta^{-1/4}\|f\|_\infty)\,.
\end{equation}
Furthermore, for $f$ satisfying $(x-x_\nu)^{-1}f\in
L^2(0,1)$ we have
\begin{equation}
  \label{eq:510}
|v(1)| \leq C \beta^{-3/8}\, \|\zeta\|_\infty \Big\|\frac{f}{x-x_\nu}\Big\|_2\,.
\end{equation}
\end{proposition}

\begin{proof}
  By \eqref{eq:453}, \eqref{eq:454}, \eqref{eq:460}, and
  (\ref{eq:365}a) in Proposition \ref{lem:schrod-quad} it holds that
  \begin{equation}\label{eq:511}
   \| \tilde v\|_2  + \beta^{1/8}  \| \tilde v\|_1 \leq C \beta^{-1/2} \| h \|_2 \leq \hat C \beta^{-1/3}\, \lambda_\beta^{-3/4}\,.
 \end{equation}
 Since $|\lambda|>U(0)/2$ we obtain for $\beta\geq \beta_0$ with $\beta_0$  large enough
  \begin{equation}
\label{eq:512}
   \| \tilde v\|_2 + \beta^{1/8}  \| \tilde v\|_1 \leq C \beta^{-1/3} [\beta^{1/3}]^{-3/4}= C \beta^{- 7/12}\,.
 \end{equation}
Given that  for $\zeta\in\mathfrak U_2(\beta,\theta,\lambda)$ it holds that
  \begin{equation}\label{eq:513}
    \|\zeta\|_\infty\leq (1+ C\|\zeta^\prime\|_2)\leq C (1 + \theta\beta^{1/6}  \lambda_\beta^{1/4})\,,
  \end{equation}
and hence  we can conclude, from \eqref{eq:509} and \eqref{eq:510}, that
\begin{equation}\label{eq:514}
\|\zeta\|_\infty \leq C\beta^{1/4}\,.
\end{equation}
We then obtain, using \eqref{eq:512}, 
\begin{equation}
  \label{eq:515}
  |\langle\zeta,\tilde{v}\rangle|\leq \|\zeta\|_\infty\|\tilde{v}\|_1\leq \hat C\,  \beta^{-11/24} \,.
\end{equation}
Furthermore, we have, using \eqref{eq:485} and the fact that
$\|\zeta^\prime\|_2\leq \theta\beta^{1/4}$  that
\begin{equation}\label{eq:516}
  |\langle\zeta-1,\psi\rangle|\leq
 \|\zeta^\prime\|_{L^2(0,1)}\|
  \|(1-x)^{1/2}\psi\|_1 \leq C\theta\beta^{-1/3}\,.   
\end{equation}

Since $v$ is still expressible by \eqref{eq:477}, we can now conclude,
as in \eqref{eq:488},  with \eqref{eq:480} and \eqref{eq:486}
respectively  replaced by \eqref{eq:515}
and \eqref{eq:516} that there exist $\theta_0$ and $\beta_0$ such that, for $\theta \leq \theta_0$
  and $\beta \geq \beta_0$, it holds  that
\begin{equation}
\label{eq:517}
  |A|\leq C\beta^{1/3} \|\zeta\|_\infty \|u\|_1 \,,
\end{equation}
where $u$ is given by \eqref{eq:478}.\\ 
 We now use \eqref{eq:365} in Proposition \ref{lem:schrod-quad} to
obtain that
\begin{displaymath}
    |A|\leq C\|\zeta\|_\infty\min(\beta^{-7/24}\|f\|_2,\beta^{-5/12}\|f\|_\infty,\beta^{-13/24}\|(x-x_\nu)^{-1}f\|_2)\,.
\end{displaymath}
Consequently, by  \eqref{eq:477} and \eqref{eq:451} we obtain that
\begin{displaymath}
 |v(1)| \leq C\,  \lambda_\beta^{1/2} \|\zeta\|_\infty\min(\beta^{-7/24}\|f\|_2,\beta^{-5/12}\|f\|_\infty,\beta^{-13/24}\|(x-x_\nu)^{-1}f\|_2)\,.
\end{displaymath}
Given that   \eqref{eq:508}  provides a uniform bound on $|\lambda|$,  we
have 
\begin{displaymath}
 \lambda_\beta^{1/2} \leq  C \beta^{1/6}\,,
\end{displaymath}
hence  we can conclude that
\begin{displaymath}
  |v(1)|\leq C\|\zeta\|_\infty\min(\beta^{-1/8}\|f\|_2,\beta^{-1/4}\|f\|_\infty,\beta^{-3/8}\|(x-x_\nu)^{-1}f\|_2)\,.
\end{displaymath}
\end{proof}

\subsection{Resolvent estimates for negative $\Re\lambda$\,.}
Although Propositions \ref{lem:integral-conditions} and
\ref{lem:integral-conditions-quadratic}  provide estimates when the
spectral parameter $\lambda$ belongs to domains in $\C$ that include
$\Re\lambda\leq0$, one can obtain a better estimate if we assume $\Re\lambda\leq-\mu_0$
for some fixed $\mu_0>0$, or at least $\Re\lambda\leq-C\, [U(0)-\nu]$ for
$\nu<U(0)$.

\begin{proposition}
  \label{lem:integral-conditions-negative} Let $U\in C^2([0,1])$
  satisfy \eqref{eq:10}.  Let further $a$, $\mu_0$,  and $\nu_0$ denote positive
  constants.  Then, there exist $C>0$, $\beta_0>0$ and $\theta_0>0$ such that \\
  for all $\beta\geq \beta_0$, $\theta \in (0,\theta_0]$ and $\lambda=\mu+i\nu\in \C$ satisfying
  \begin{equation} 
\label{eq:518} 
  -\nu_0 <\nu<U(0)+a\beta^{-1/2} \,,
  \end{equation}
and 
\begin{equation} 
\label{eq:519}
\mu\leq-\mu_0\,,
\end{equation}
 for any $\zeta\in \mathfrak
  U_2(\beta,\theta,\lambda)$, given by \eqref{eq:507}, and any pair $(f,v) \in L^2(0,1)\times D(\mathcal L^\beta_\zeta)$
  satisfying \eqref{eq:444},  it holds that
\begin{equation}
\label{eq:520}
|v(1)| \leq C\,  \beta^{-1/2}\, \|\zeta\|_\infty \|f\|_2\,.
\end{equation}
\end{proposition}
\begin{proof}
  As in \eqref{eq:477} we write
\begin{displaymath}
  v=A(\psi-\tilde{v})+u \,,
\end{displaymath}
where $A\in\C$, $\psi=\psi_{\lambda,\beta}$ is given by \eqref{eq:452}, $\tilde{v}$
by \eqref{eq:454}
and $u$ by \eqref{eq:478}. \\
As $\zeta \in\mathfrak U_2(\beta,\theta,\lambda)$, we obtain, given that $-\nu_0 <
\nu<U(0)- a\beta^{-1/2}$, and in view of \eqref{eq:519}, \eqref{eq:513},
and \eqref{eq:455}
\begin{equation}\label{eq:521}
   |\langle\zeta,\tilde{v}\rangle|\leq \|\zeta\|_\infty\|\tilde{v}\|_1\leq C \beta^{- 1/2}  \lambda_\beta^{-1/2}  \leq  \hat C\,  \beta^{- 2/3} \,.
  \end{equation}
  Note that while both Propositions \ref{lem:integral-conditions} and
\ref{lem:integral-conditions-quadratic} assume $\mu\geq-\mu_0$,  both
\eqref{eq:455} and \eqref{eq:513} are valid for $\mu<-\mu_0$ as well.\\
In the case 
  \begin{equation*}
  U(0)-a\beta^{-1/2}< \nu<U(0)+a\beta^{-1/2}\,,
  \end{equation*}
  we proceed as in the proof of \eqref{eq:521} but use
  \eqref{eq:511} instead of \eqref{eq:455} and \eqref{eq:460} which
  continues to hold in this case.  Hence, we obtain the weaker  estimate
\begin{equation*}
  |\langle\zeta,\tilde{v}\rangle|\leq \hat C\,  \beta^{-11/24} \,.
\end{equation*}
Combining the above with \eqref{eq:521} yields the existence of $C>0$
such that  for any $\lambda$ satisfying \eqref{eq:518} and \eqref{eq:519}
it holds that
\begin{equation}\label{eq:522}
   |\langle\zeta,\tilde{v}\rangle|\leq C\beta^{-11/24}\,.
\end{equation}
Furthermore, as in \eqref{eq:516} we write
\begin{equation*}
  |\langle\zeta-1,\psi\rangle|\leq
 \|\zeta^\prime\|_{L^2(0,1)}\|
  \|(1-x)^{1/2}\psi\|_1 \,,  
\end{equation*}
from which we conclude, using the fact that $\zeta\in\mathfrak U_2(\beta,\theta,\lambda)$ and \eqref{eq:485}
\begin{equation*}
  |\langle\zeta-1,\psi\rangle|\leq C \theta\beta^{1/6} \lambda_\beta^{1/4}\times \beta^{-1/2} \lambda_\beta^{-1/4} \,.   
\end{equation*}
Consequently, it holds that 
\begin{equation}
\label{eq:523}
  |\langle\zeta-1,\psi\rangle|\leq C\theta\beta^{-1/3}\,.   
\end{equation}
Hence, as in \eqref{eq:517} we obtain that, choosing $\theta_0$ small enough
\begin{equation}
\label{eq:524}
 |A|\leq C\beta^{1/3} \|\zeta\|_\infty \|u\|_1 \,.
\end{equation}
To estimate $\|u\|_1$ we observe that
\begin{equation}
\label{eq:525}
  \Re\langle u,(\LL_\beta^\Nf-\beta\lambda)u\rangle=\|u^\prime\|_2^2 -\mu\beta\|u\|_2^2 \,,
\end{equation}
where $\LL_\beta^\Nf$ is defined in \eqref{eq:285},
from which we conclude  that
\begin{equation}
\label{eq:526} 
  \|u\|_1\leq \|u\|_2 \leq \frac{1}{|\mu| \beta}\|f\|_2 \,.
\end{equation}
The proof of the proposition can now be completed by using \eqref{eq:451} and  the fact that $
v(1)=A\psi(1)$. 
Thus, by  \eqref{eq:451} 
we obtain from \eqref{eq:524} that
\begin{equation}\label{eq:527}
 |v(1)|\leq C\beta^{1/3} \lambda_\beta^{1/2}   \|\zeta\|_\infty \,  \|u\|_1 \,.
\end{equation}
Since for $\mu \leq -\mu_0$ it holds by  (\ref{eq:518})  that
\begin{displaymath}
 |\lambda| \leq |\mu| + |\nu| \leq  C  |\mu|\,,
\end{displaymath}
we may conclude that
\begin{displaymath}
 |v(1)|\leq C\beta^{1/2}  |\mu|^{1/2} \|\zeta\|_\infty  \,\|u\|_1 \,.
\end{displaymath}
Hence by \eqref{eq:526} we can conclude that
\begin{displaymath}
 |v(1)|\leq C\beta^{-1/2}  |\mu|^{-1/2} \|\zeta\|_\infty  \|f\|_2 \leq \hat  C\beta^{-1/2}  |\mu_0|^{-1/2} \,\|\zeta\|_\infty  \,\|f\|_2 \,.
\end{displaymath}
The proposition is proved.
\end{proof}

We next consider the case $-\mu_0<\mu<-\frac{|U(0)-\nu|}{\kappa_1}$ for some
$\kappa_1>0$.  While \eqref{eq:446} and \eqref{eq:509} hold true under
this assumption, it is necessary, in the next section, to obtain
better estimates since Proposition \ref{prop:near-quadratic} is
inapplicable in this case.
\begin{proposition}
\label{lem:integral-conditions-mild-negative} Let  $U\in
C^2([0,1])$ satisfy \eqref{eq:10}.  Let further $a$, $\kappa_1$,  $\nu_0$  and $\mu_0$ denote positive
constants.  Then, there exist $C>0$, $\beta_0>0$ and $\theta_0>0$ such 
that,  for all $\beta\geq \beta_0$, $\theta \in (0, \theta_0]$ and $\lambda=\mu+i\nu\in \C$
satisfying (see  (\ref{eq:443}) and  \eqref{eq:518})
  \begin{equation}\label{eq:528}
    -\nu_0 <\nu<U(0)+a\beta^{-1/2} \,,
  \end{equation}
  and 
\begin{equation}\label{eq:529} 
   -\mu_0<\mu<-\frac{|U(0)-\nu|}{\kappa_1}\,,
 \end{equation} 
  for any  $\zeta\in \mathfrak U_2(\beta,\theta,\lambda)$, and any pair  $(v,f) \in D(\mathcal L^\beta_\zeta)\times  L^2(0,1)$ satisfying
 \eqref{eq:444}, it holds that
\begin{equation}
\label{eq:530}
|v(1)| \leq C \, \|\zeta\|_\infty \min(|\mu|^{-1/2}\beta^{-1/2}\|f\|_\infty,|\mu|^{-3/4}\beta^{-1/2}\|f\|_2)\,.
\end{equation}
\end{proposition}

\begin{proof}~\\{\it Step 1:} We prove that
\begin{equation} \label{eq:531}
 |v(1)| \leq C\, | \mu|^{-3/4}\beta^{-1/2} \|\zeta\|_\infty   \|f\|_2 \,.
\end{equation}  
As in the proof of Proposition \ref{lem:integral-conditions-negative}
and since for sufficiently small $\theta_0$, \eqref{eq:527} still holds
under the assumptions of this proposition we obtain that
  \begin{equation}
\label{eq:532} 
 |v(1)|\leq C\beta^{1/3}  \lambda_\beta^{1/2}  \|\zeta\|_\infty \, \|u\|_1\leq  \hat C\beta^{1/2} \|\zeta\|_\infty \|u\|_1 \,,
  \end{equation}
  where $u$ is given by \eqref{eq:478}. Note that under
  \eqref{eq:528} and \eqref{eq:529}, $|\lambda|$ is bounded. To obtain an
  estimate for $\|u\|_1$ we now write
\begin{equation}\label{eq:533}
  \|u\|_1\leq \|(U+i\lambda)^{-1}\|_2\,\|(U+i\lambda)u\|_2\,.
\end{equation}
By \eqref{eq:354} and \eqref{eq:525} we have that
\begin{equation}\label{eq:534}
  \|(U+i\lambda)u\|_2\leq \|(U-\nu)u\|_2+|\mu|\,\|u\|_2 \leq \frac{\hat C}{\beta}\|f\|_2 \,.
\end{equation}
 By \eqref{eq:265} (with $q=2$) it holds that 
  \begin{displaymath}
    \|(U+i\lambda)^{-1}\|_2^2 \leq C\, |\mu|^{-3/2}\,. 
  \end{displaymath}
Consequently, we may conclude from \eqref{eq:533} and \eqref{eq:534}  that
\begin{displaymath}
   \|u\|_1\leq C\, |\mu|^{-3/4}\beta^{-1}\,\|f\|_2 \,.  
\end{displaymath}
Substituting into \eqref{eq:532} then yields \eqref{eq:531}.\\

{\em Step 2: We prove that
\begin{equation}
\label{eq:535}
|v(1)| \leq C \, |\mu|^{-1/2}\beta^{-1/2} \|\zeta\|_\infty \, \|f\|_\infty \,.
\end{equation}
}
Suppose that $f\in L^\infty(0,1)$. Then by \eqref{eq:525} it holds for
negative values of $\mu$ that
\begin{equation}
\label{eq:536}
 \|u^\prime\|_2^2 + |\mu|\beta \|u\|_2^2\leq \|u\|_1\,\|f\|_\infty \,.
\end{equation}
Set   
\begin{displaymath}
    \chi_\beta^\pm(x)=\hat{\chi}([|\mu|\beta]^{1/2}(x-x_\nu)){\mathbf 1}_{\R_+}(\pm
    (x-x_\nu)) \,,
  \end{displaymath}
  where $\hat{\chi}$ be given by \eqref{eq:41}. \\
   Note that $ \chi_\beta^+$ is
  supported in $(x_\nu+[|\mu|\beta]^{-1/2}/4,+\infty)$ whereas $ \chi_\beta^-$ is
  supported in $(-\infty,x_\nu-[|\mu|\beta]^{-1/2}/4)$ . Let further
  $\tilde{\chi}_\beta =\sqrt{1-(\chi_\beta^+)^2-(\chi_\beta^-)^2}$.  Note that by
  \eqref{eq:529} it follows that we can choose $\beta_0$ 
  such that for all $\beta>\beta_0$
  \begin{equation}\label{eq:537}
  x_\nu-[|\mu|\beta]^{-1/2}/2>0\,.
  \end{equation}
Hence the support of
  $\tilde{\chi}_\beta$ is contained in
  \begin{displaymath}
  (x_\nu- \frac 12 [|\mu|\beta]^{-1/2}\,,\,x_\nu+ \frac 12 [|\mu|\beta]^{-1/2})\subset[0,1]\,.    
  \end{displaymath}
  As in \eqref{eq:308} (with $\chi_\nu^\pm$ replaced by $\chi_ \beta^\pm$), we now
  obtain
  \begin{equation}
\label{eq:538}
    \beta\||U-\nu|^{1/2}\chi_\beta^\pm u\|_2^2\leq 2\, \|(\chi_\beta^\pm)^\prime u\|_2\|\chi_\beta^\pm u^\prime\|_2 +  \|u\|_1\|f\|_\infty  \,.
  \end{equation}
Consequently, by \eqref{eq:538} and the definition of $\chi_\beta^\pm$ we
conclude that
\begin{equation*}
    \beta\||U-\nu|^{1/2}\chi_\beta^\pm u\|_2^2\leq C |\mu \beta|^{-1/2} \| u\|_2\| u^\prime\|_2 +  \|u\|_1\|f\|_\infty  \,.
  \end{equation*}
By \eqref{eq:536} we then have
\begin{equation}
\label{eq:539}
  \||U-\nu|^{1/2}\chi_\beta^\pm u\|_2^2\leq C\, \beta^{-1} \, \|u\|_1\|f\|_\infty  \,.
\end{equation}
By \eqref{eq:528}
and \eqref{eq:529} we have that 
\begin{equation}
\label{eq:540}
\frac 1C \,   \beta^{-1/2} \leq x_\nu^2\leq C|\mu| \,.
\end{equation}
Given the definition of $\tilde{\chi}_\beta$ we obtain that
 \begin{displaymath}
  |U(x)-\nu| \leq C \, (\sup_{x\in {\rm supp} \tilde \chi_\beta} |U^\prime(x) |)\,  |\mu \beta|^{-1/2}\,.
\end{displaymath}
Hence, by \eqref{eq:540} and \eqref{eq:537} we can conclude that
\begin{displaymath}
   \||U-\nu|^{1/2}\tilde{\chi}_\beta u\|_2^2
   \leq Cx_\nu[|\mu|\beta]^{-1/2}\|u\|_2^2\,.
\end{displaymath}
Combining the above with \eqref{eq:539} now yields
\begin{displaymath}
   \||U-\nu|^{1/2}u\|_2\leq  C\, \beta^{-1/2}\, \|u\|_1^{1/2}\|f\|_\infty^{1/2} \,.
\end{displaymath}
 By \eqref{eq:265} (with $q=1$) it holds that
\begin{displaymath}
   \|(U+i\lambda)^{-1/2}\|_2^2= \|(U+i\lambda)^{-1}\|_1 \leq C\,|\mu|^{-1/2}\,.
\end{displaymath}
Consequently,
\begin{displaymath}
\begin{array}{ll}
  \|u\|_1 & \leq
  \|(U+i\lambda)^{-1/2}\|_2\,\|(U+i\lambda)^{1/2}u\|_2\\  &\leq
  C\, |\mu|^{-1/4}\, \big[\||U-\nu|^{1/2}u\|_2+|\mu|^{1/2}\|u\|_2\big]\\ & 
  \leq \hat C\, |\mu|^{-1/4}\,\beta^{-1/2}\, \|u\|_1^{1/2}\|f\|_\infty^{1/2}   \,.
  \end{array} 
\end{displaymath}
Hence,
\begin{displaymath}
  \|u\|_1\leq C\, |\mu|^{-1/2} \beta^{-1} \, \|f\|_\infty \,,
\end{displaymath}
which, when substituted into \eqref{eq:532}, establishes \eqref{eq:535}. \\
Together with
\eqref{eq:531}  the above inequality completes the proof of the proposition.
\end{proof}

\subsection{Rapidly decaying functions} 
 When considering large
  values of $\alpha$ in the next section, it is useful to consider, as
  in \cite{almog2019stability}, the operator $\LL_\zeta^\beta$ where $\zeta$
  decays rapidly away from the boundary at $x=1$.  Set then for  $\lambda \in
  \mathbb C$ and positive $\beta, \theta, \alpha$  
  \begin{equation} 
\label{eq:541} 
\mathfrak U_3(\beta,\theta,\alpha,\lambda) =\{ \zeta \in
   \mathfrak U_2(\beta,\theta,\lambda)\,|\, |\zeta(x)|\leq  
    e^{-\alpha(1-x)} \|\zeta\|_\infty\,, \,\forall x\in[0,1]\}\,, 
\end{equation} 
where $\mathfrak U_2(\beta,\theta,\lambda)$   is introduced in \eqref{eq:507}.  
\begin{proposition}
    \label{lem:integral-conditions-large -alpha} 
Let $U\in C^3([0,1])$
    satisfy \eqref{eq:10}. Let further  $a>0$, $\mu_0>0$, and $ \Upsilon <
   \sqrt{-U^{\prime\prime}(0)}/2\,$. Then, there exist $C>0$, $\beta_0>0$ and
    $\theta_0>0$ such that for all $\beta\geq \beta_0$, $\theta \in (0,\theta_0]$, all
    $\alpha \geq 1$ and $\lambda\in \C$ satisfying
    \begin{equation}
\label{eq:542}
 -\mu_0\leq \Re\lambda\leq \Upsilon\beta^{-1/2} \,, 
\end{equation} 
and
\begin{equation}
      \label{eq:543} 
\frac{1}{2}U(0)<\nu<U(0)+a\beta^{-1/2} \,,
    \end{equation} 
    for all $\zeta \in \mathfrak U_3(\beta,\theta,\alpha,\lambda) $, and for all pair $(f,v) \in
    L^2(0,1)\times D(\mathcal L^\beta_\zeta)$ satisfying \eqref{eq:444}, it holds
    that
    \begin{equation} 
\label{eq:544} 
|v(1)| \leq C \alpha^{-1/2}\,(\beta^{-1/2}+e^{-\alpha/C})  \|\zeta\|_\infty \, \|f\|_2 \,.  
\end{equation} 
\end{proposition}
\begin{proof}~\\
{\em Step 1: Control of $v(1)$.}\\
Since \eqref{eq:515} and \eqref{eq:516} remain valid
    under our assumptions \eqref{eq:542} and \eqref{eq:543}, we can follow the same
    procedure as in Proposition
      \ref{lem:integral-conditions-quadratic} to obtain
  \begin{equation}
\label{eq:545}
        |v(1)|\leq C\, \beta^{1/3} \lambda_\beta^{1/2}|\langle\zeta,u\rangle|\leq  C\beta^{1/2}|\langle\zeta,u\rangle|\,,
  \end{equation}
  where $u$ is given by \eqref{eq:478}.\\
  
  {\em Step 2: We prove under \eqref{eq:542} and \eqref{eq:543} that
\begin{equation}\label{eq:546}
\|u\|_2 \leq C\, \beta^{-1/2}\,  \|f\|_2\,.
\end{equation}
}
We first consider the case where 
\begin{equation}
\label{eq:547}
U(0)/2 < \nu < U(0)- a_1 \beta^{-1/2} \,,
    \end{equation} 
    where $a_1\geq a $  will be determined in the sequel. 
    In this case, we can use \eqref{eq:287} which reads (for
      $\Jf_\nu=|U^\prime(x_\nu)|$)
    \begin{displaymath}
    \| u\|_2 \leq C \, (\Jf_\nu \beta)^{-2/3}\|f\|_2\,.
\end{displaymath}
and holds under the condition $\mu\leq\Upsilon_0 \, {\mathfrak
  J_\nu}^{2/3}\, \beta^{-1/3}$ for some sufficiently small $\Upsilon_0>0$. 
 Note that for $\nu-U(0)>a_1 \beta^{-1/2}$ there exists $\hat C>0$ such that
\begin{displaymath}
  \Jf_\nu\geq \frac{1}{\hat C} \, a_1^{1/2}\,  \beta^{-1/4}\,.
\end{displaymath}
Consequently, there exists $C>0$ such that \eqref{eq:287} is
applicable for all
\begin{displaymath}
    \mu\leq Ca_1^{1/3}\beta^{-1/2}\,.
\end{displaymath}
For sufficiently large $a_1\geq a$ the above set of $\mu$ values contains
\eqref{eq:542}  and hence we can conclude \eqref{eq:546} when
\eqref{eq:547} holds true.

We now look at the case 
\begin{equation*}
U(0) - a_1 \beta^{-1/2}  <\nu<U(0) + a\beta^{-1/2} \,,
    \end{equation*}  
    Here we can apply \eqref{eq:365} (with $a$ replaced by $a_1$) to
    obtain \eqref{eq:546} which, combined with
    \eqref{eq:545}, leads to
     \begin{displaymath}
     |v(1) | \leq C \,  \|\zeta\|_\infty \|f\|_2\,.
\end{displaymath}
Note that at this stage it is sufficient  to assume that $\zeta \in
\mathfrak U_2(\beta,\theta,\alpha,\lambda) $.\\ 
    
   {\em Step 3:  With  $\hat{x}_\nu \in(0,1)$ satisfying
  \begin{equation}
\label{eq:548}
    U(\hat{x}_\nu)= \frac{\nu}{2} \,,
  \end{equation}
we prove that 
 \begin{equation}
   \label{eq:549} 
|\langle \zeta, u\nobreakspace\rangle | \leq  C\,\alpha^{-1/2}\,  \beta^{-1/2}  ( \beta^{-1/2} +   \,e^{-\alpha(1-\hat{x}_\nu)}  )\,  \|\zeta\|_\infty \, \|f\|_2 \,.
 \end{equation}
}

Consider the decomposition
  \begin{displaymath}
  \langle \zeta, u\rangle = \langle{\mathbf 1}_{L^2(0,\hat{x}_\nu)}\zeta,u\rangle + \langle{\mathbf 1}_{(\hat{x}_\nu,1)}\zeta,u\rangle\,.
\end{displaymath}
 We first obtain using \eqref{eq:546} that
\begin{equation}
\label{eq:550}
  |\langle{\mathbf 1}_{L^2(0,\hat{x}_\nu)}\zeta,u\rangle|\leq \|{\mathbf
    1}_{L^2(0,\hat{x}_\nu)}\zeta\|_2 \, \|u\|_2  \leq C\, \alpha^{-1/2}\, \beta^{-1/2}
  \,e^{-\alpha(1-\hat{x}_\nu)}\, \|\zeta\|_\infty \|f\|_2 \,.  
\end{equation}
Moreover, it holds that
\begin{equation}
\label{eq:551} 
   |\langle{\mathbf 1}_{(\hat{x}_\nu,1)}\zeta,u\rangle| \leq C\|\zeta\|_2 \,\|{\mathbf
     1}_{(\hat{x}_\nu,1)}u\|_2 \leq \frac{C}{\alpha^{1/2}} {\| \zeta\|_\infty\,}  \|{\mathbf
     1}_{(\hat{x}_\nu,1)}u\|_2   \,.
\end{equation}
Let
\begin{equation}
\label{eq:552}
 \check \chi_\nu(x)=\chi\Big(\frac{x-x_\nu}{\hat{x}_\nu-x_\nu}\Big){\mathbf
    1}_{\R_+}(x-x_\nu) \,,
\end{equation}
where $\chi$ is given by \eqref{eq:41}. Note that $\check \chi_\nu$ is
supported in the interval \break $ (x_\nu+(\hat{x}_\nu-x_\nu)/4,+\infty)$ and
equals $1$ on
$[(x_\nu+\hat x_\nu)/2,1]$. Integration by part yields 
\begin{subequations}
  \label{eq:553}
\begin{equation}
  \|(\check \chi_\nu u)^\prime\|_2^2=\|\check \chi_\nu^\prime u\|_2^2+\mu\beta\|\check \chi_\nu u\|_2^2 +
  \Re\langle\check \chi_\nu u,\check \chi_\nu f\rangle \,, 
\end{equation}
and, given that $(U-\nu)$ has constant sign on  the support of $\check \chi_\nu$,
\begin{equation}
  - \beta\||U-\nu|^{1/2}\check \chi_\nu u\|_2^2+2\Im\langle(\check \chi^\prime_\nu  u,(\check \chi_\nu u)^\prime\rangle =  \Im \langle\check \chi_\nu u,\check \chi_\nu f\rangle \,.
\end{equation}
\end{subequations}
Combining the above we obtain,  given the support of $\check \chi_\nu$ 
\begin{multline} \label{eq:554}
  \|\check \chi_\nu u\|_2^2\leq C \||U-\nu|^{1/2}\check \chi_\nu u\|_2^2 \leq
  C \beta^{-1} [\|u\|_2\|(\check \chi_\nu u)^\prime\|_2+ \|\check
  \chi_\nu u\|_2\|\check \chi_\nu f\|_2] \leq \\
 \leq  C \beta^{-1} [\|u\|_2(\|u\|_2+ \mu_{\beta,+}^{1/2} \|\check
  \chi_\nu u\|_2+\|\check  \chi_\nu u\|_2^{1/2}\|\check \chi_\nu f\|_2^{1/2})   
  + \|\check  \chi_\nu u\|_2\|\check \chi_\nu f\|_2]\,.
\end{multline}
From here we deduce that 
\begin{displaymath} 
  \|\check \chi_\nu u\|_2^2\leq
  C \beta^{-1} \Big[\|u\|_2\big(\|u\|_2+\beta^{1/4} \|\check   \chi_\nu
  u\|_2+\|\check  \chi_\nu u\|_2^{1/2}\| \check \chi_\nu f\|_2^{1/2}\big)     + \|\check
  \chi_\nu u\|_2\| \check \chi_\nu f\|_2\Big]\,, 
\end{displaymath}
which implies
\begin{equation}
\label{eq:555}
  \|\check \chi_\nu u\|_2^2 \leq   C \left( \frac{1}{\beta^2} \|\check \chi_\nu f\|^2_2 + \frac{1}{\beta} \|u\|^2_2\right)\,.
\end{equation}
   Combining \eqref{eq:555} and \eqref{eq:546} leads to 
\begin{equation}
\label{eq:556}
  \|{\mathbf   1}_{(\hat{x}_\nu,1)}u\|_2\leq \|\check \chi_\nu u\|_2 \leq  C\, \beta^{-1}\, \|f\|_2 \,.
\end{equation}
For later reference we note that by (\ref{eq:553}a) and \eqref{eq:556}
it holds that
\begin{equation}
  \label{eq:557}
 \|{\mathbf   1}_{(\hat{x}_\nu,1)}u^\prime\|_2\leq \|(\check \chi_\nu u)^\prime\|_2 \leq  C\, \beta^{-1}\, \|f\|_2 \,.
\end{equation}
Combining \eqref{eq:556} with \eqref{eq:550} and  \eqref{eq:551}   yields
\eqref{eq:549}, which, together  with 
\eqref{eq:545} yields
  \eqref{eq:544}. 
\end{proof}
\subsection{Auxiliary estimates}
We recall that for $(\lambda,\beta) \in \mathbb C \times \mathbb R_+$, $\psi=\psi_{\lambda,\beta}$
is  given by \eqref{eq:452}.  We now set, for $x\in (0,1)$  
\begin{equation}
\label{eq:558}
  \hat{\psi}_{\lambda,\beta}(x) =\frac{\psi_{\lambda,\beta}(x) }{\psi_{\lambda,\beta}(1)} \,.
\end{equation}
The following auxiliary estimate will
become useful in the next section.
\begin{lemma}
\label{lem:auxiliary-estimate-1}
 Let $\tilde{\nu}_1 \in (0,U(0))$, $ a>0$, and
$ \Upsilon<\sqrt{-U^{\prime\prime}(0)}/2$.  Then there exist $C$ and $\beta_0$
  such that, for $\beta \geq \beta_0\,$,  
  \begin{equation}\label{eq:559}
   \nu\in (U(0)-\tilde{\nu}_1, 
  U(0)+a\beta^{-1/2}) \mbox{ and }   \mu<\Upsilon\beta^{-1/2}\,,
  \end{equation}
such that
 \begin{equation}
\label{eq:560}
    \|(\LL^\Nf_\beta-\beta\lambda)^{-1}\hat{\psi}_{\lambda,\beta}\|_2 + \beta^{-1/2}\Big\|
    \frac{d}{dx}\, (\LL^\Nf_\beta-\beta\lambda)^{-1}\hat{\psi}_{\lambda,\beta}\Big\|_2
    \leq C\, \beta^{-5/4} \,. 
  \end{equation}
\end{lemma}
 For convenience of omit the subscript
  $(\lambda,\beta)$ from $\hat{\psi}_{\lambda,\beta}$ in the sequel.
\begin{proof}
  Let $v\in D(\LL^\Nf_\beta)$ satisfy
  \begin{equation}
\label{eq:561}
    (\LL^\Nf_\beta-\beta\lambda)v=\hat{\psi} \,.
  \end{equation}
   Let $\mu_0>0$. We begin by considering the case $\mu\geq-\mu_0$. Let
  further $\check \chi_\nu$ be given by \eqref{eq:552}. As in \eqref{eq:555}
  we obtain
\begin{displaymath}
  \|\check \chi_\nu v\|_2\leq C\, [\beta^{-1/2}\|v\|_2 +
  \beta^{-1}\|\hat{\psi}\|_2]\,.
\end{displaymath}
By  \eqref{eq:546} it holds that
\begin{displaymath}
  \|v\|_2 \leq C\beta^{-1/2} \|\hat{\psi}\|_2 \,.
\end{displaymath}
Furthermore, using \eqref{eq:457} (with $k=0$) and \eqref{eq:451} 
 we have for $|\lambda| \geq |\nu| \geq U(0)/2$
\begin{equation}
\label{eq:562} 
  \|\hat{\psi}\|_2 \leq C \, \lambda_\beta^{-1/4} \beta^{-1/6} \leq \widehat C\beta^{-1/4}\,.
\end{equation}
Hence we obtain that 
\begin{equation}
\label{eq:563}
   \|\check \chi_\nu v\|_2\leq C\beta^{-5/4} \,.
\end{equation}
Furthermore, since by (\ref{eq:553}a)
\begin{equation}\label{eq:564}
  \|(\check \chi_\nu v)^\prime\|_2\leq[\mu_{\beta,+}]^{1/2}\|\check \chi_\nu v\|_2+C(\|v\|_2+
  \|\check \chi_\nu v\|_2^{1/2} \|\hat{\psi}\|_2^{1/2})\,,
\end{equation}
where $\mu_{\beta,+}$ is given by \eqref{eq:284},
we can conclude from \eqref{eq:563} that
\begin{equation}
\label{eq:565}
   \|(\check \chi_\nu v)^\prime\|_2\leq C(\beta^{-3/4}+\|v\|_2) \,.
\end{equation}
Let 
\begin{equation}\label{eq:566}
\invbreve{\chi}_\nu=\sqrt{1-\check{\chi}_\nu^2}\,.
\end{equation} 
Clearly
\begin{equation}
\label{eq:567}
  (\LL^\Nf_\beta-\beta\lambda)(\invbreve{\chi}_\nu v)=-\invbreve{\chi}_\nu^{\prime\prime}v-2\invbreve{\chi}_\nu^\prime v^\prime+\invbreve{\chi}_\nu\hat{\psi}\,.
\end{equation}
By   \eqref{eq:563} it holds that
\begin{displaymath}
   \|v\|_2 \leq    C\,\left(\|\invbreve{\chi}_\nu v\|_2+\beta^{-5/4}\right) \,.
\end{displaymath}
Furthermore we have, by \eqref{eq:565} together with \eqref{eq:563} for the last line, that
\begin{displaymath}
\begin{array}{ll}
  \|v^\prime\|_2 & \leq   \|(\check \chi_\nu v)^\prime\|_2+ \|(\invbreve{\chi}_\nu v)^\prime\| \\ & \leq
  C\left(\beta^{-3/4}+\|v\|_2+ \|(\invbreve{\chi}_\nu v)^\prime\|\right) \\
  & \leq \hat C\left(\beta^{-3/4}+\|\invbreve{\chi}_\nu v\|_2 + \|(\invbreve{\chi}_\nu v)^\prime\|\right)    \,.
  \end{array}
\end{displaymath}
We now apply  either \eqref{eq:365} or \eqref{eq:287} (see the proof of Proposition \ref{lem:integral-conditions-large -alpha}, Step 2)    to \eqref{eq:567} to obtain
\begin{multline*}
  \|\invbreve{\chi}_\nu v\|_2+ \beta^{-1/4}\|(\invbreve{\chi}_\nu v)^\prime\|_2\leq
  C\beta^{-1/2}\big(\|v\|_2 +\|v^\prime\|_2+\|\invbreve{\chi}_\nu\hat{\psi}\|_2\big) \\ \leq
  \hat C\beta^{-1/2}\big(\beta^{-3/4}+\|(\invbreve{\chi}_\nu v)^\prime\|_2+\|(\invbreve{\chi}_\nu v)^\prime\|_2
  +\|\invbreve{\chi}_\nu\hat{\psi}\|_2\big)\,. 
\end{multline*}
Hence,
\begin{displaymath}
   \|\invbreve{\chi}_\nu v\|_2+ \beta^{-1/4}\|(\invbreve{\chi}_\nu v)^\prime\|_2\leq
   C\beta^{-1/2}(\beta^{-3/4}+\|\invbreve{\chi}_\nu\hat{\psi}\|_2) \,.
\end{displaymath}
By \eqref{eq:457} and \eqref{eq:451} we obtain that
\begin{displaymath}
  \|\invbreve{\chi}_\nu\hat{\psi}\|_2\leq C\beta^{-3/4} \,,
\end{displaymath}
and hence
\begin{displaymath}
   \|\invbreve{\chi}_\nu v\|_2+ \beta^{-1/4}\|(\invbreve{\chi}_\nu v)^\prime\|_2\leq C\beta^{-5/4} \,.
\end{displaymath}
Combining the above with \eqref{eq:563} and \eqref{eq:567} yields
\eqref{eq:560}. 

Consider now the case where $\mu<-\mu_0$. Here we use \eqref{eq:363} and
\eqref{eq:364}, applied to the pair $(v,\hat \psi)$, and then
\eqref{eq:562} to obtain that
  \begin{displaymath}
    \|v\|_2 +\beta^{-1/2}\|v^\prime\|_2\leq C\beta^{-1}\|\hat{\psi}\|_2\leq \widehat C\, \beta^{-5/4} \,,
  \end{displaymath}
establishing, thereby, \eqref{eq:560} for the case $\mu<-\mu_0$.
\end{proof}
\begin{remark}
  Let  for $(\lambda,\beta) \in \mathbb C \times \mathbb R_+$,
  $\hat{g}(x)=h(x)\chi(1-x)/\psi_{\lambda,\beta}(1)$ where $h$ is given by 
  \eqref{eq:456}.  By the same arguments used to establish
  \eqref{eq:560} we may conclude under the assumptions of Lemma
  \ref{lem:auxiliary-estimate-1} that there exists $C$ and $\beta_0$ such
  that for $\beta \geq \beta_0$,  
\begin{displaymath}
\nu\in (U(0)-\tilde{\nu}_1,   U(0)+a\beta^{-1/2}) \mbox{  and } -\mu_0\leq\mu<\Upsilon\beta^{-1/2}
\end{displaymath}
  we have 
  \begin{equation}
    \label{eq:568}
    \|(\LL^\Nf_\beta-\beta\lambda)^{-1}\hat{g}\|_2 + \beta^{-1/2}\Big\|
    \frac{d}{dx}\,(\LL^\Nf_\beta-\beta\lambda)^{-1}\hat{g}\Big\|_2 \leq C\, \beta^{-5/4}
    \,. 
  \end{equation}
\end{remark}

\subsection{Resolvent estimates for large $\alpha$}
\label{sec:resolv-estim-large}
In this subsection, we adapt the results of Section 6.3 in
\cite{almog2019stability} to the present setting, involving a Neumann
condition at $x=0$. For $\alpha >0$, we consider $\mathfrak
z_\alpha $ to be the solution of
  \begin{equation}
\label{eq:602}
  \begin{cases}
      -  \mathfrak z^{\prime\prime}+\alpha^2  \mathfrak z = 0  \mbox{ for } x\in(0,1) \\
        \mathfrak z (1)=1 \mbox{ and }   \mathfrak z^\prime (0)=0 \,.
  \end{cases}
\end{equation}
The solution of \eqref{eq:602} is given by
\begin{equation}\label{eq:603}
\mathfrak z_\alpha(x) =  \cosh (\alpha x)/\cosh(\alpha)\,,
 \end{equation}
 and hence, for large $\alpha$ decays exponentially fast away from $x=1$.

 \begin{proposition}
\label{lem:integral-conditions-large-alpha} 
Let $\theta_1 >0$, $U\in C^2([0,1])$ satisfy \eqref{eq:10}, and
$\Upsilon<\sqrt{-U^{\prime\prime}(0)}/2$. { Let  further $\hat{\mu}_m>0$ be given by
  \cite[Eq. (6.57)]{almog2019stability}, Then, for any   $\hat{\Upsilon}>0$, } there exist $\beta_0>0$ and $C>0$ such that, for $\beta\geq \beta_0$ and $\alpha\geq\theta_1\beta^{1/3}
$\,,
   \begin{equation}
\label{eq:604}
\sup_{ \Re\lambda\leq \min(\Upsilon\beta^{-1/2},\beta^{-1/3}[\hat{\mu}_m-\hat \Upsilon-\alpha^2\beta^{-2/3}/2])}\|(\LL_\beta^{\mathfrak z_\alpha}-\beta\lambda)^{-1}\| 
\leq \frac{C}{\beta^{1/2}[1+\beta^{1/6}|U(0)-\nu|^{1/3}]} \,.
  \end{equation}
\end{proposition}
\begin{proof}
  The proof follows the same lines of the proof of \cite[Proposition
  6.11]{almog2019stability}, and hence we bring only its main
  ingredients.  

Let $\theta=\alpha\beta^{-1/3}$ and 
    \begin{displaymath}
      F(\lambda,\theta)=\int_{\R_+}e^{-\theta x}\, \Ai (e^{i\pi/6}(x+i\lambda))\,dx\,.
     \end{displaymath}
Let further 
\begin{displaymath}
 \omega( \beta,\lambda,\theta):= \frac{F(\beta^{1/3}\lambda,0)}{F(\beta^{1/3}\lambda,\theta)}\,.
\end{displaymath}
We then define
\begin{displaymath}
     \psi_\theta= \omega ( \beta,\lambda,\theta)  \psi\,,
\end{displaymath}
where $\psi =\psi_{\lambda,\beta}$ is defined in \eqref{eq:452}, and
\begin{displaymath}
  h_\theta=\omega ( \beta,\lambda,\theta) h \,,
\end{displaymath}
where $h$ is defined by \eqref{eq:453}. Set
\begin{displaymath}
  \tilde{v}_\theta=   (\LL_\beta^\Nf-\beta\lambda)^{-1}h_\theta\,.
\end{displaymath}
By  \cite[Eq. (6.79)]{almog2019stability} { for any $\hat \Upsilon>0$ 
there exists  $C>0$}  such that 
\begin{equation}\label{eq:605} 
\sup_{ \Re\lambda\leq\beta^{-1/3}[\hat{\mu}_m-\hat \Upsilon-\alpha^2\beta^{-2/3}/2]}|\omega ( \beta,\lambda,\theta)| \leq C \, \theta\,,\, \forall \theta \geq \theta_1  \,,
\end{equation}
and hence by \eqref{eq:455} and
\eqref{eq:512} we obtain that 
\begin{equation}
\label{eq:606}
  \|\tilde{v}_\theta\|_1 \leq C\theta\beta^{-2/3}\,.
\end{equation}
Note that for any $\Upsilon>0$, there exist $\beta_0$ and $a$ such that for $\beta
\geq \beta_0$  and \break $\nu<U(0)-a \beta^{-1/2}$, we have
  \begin{displaymath}
    \frac{\sqrt{-U^{\prime\prime}(0)}}{2}\beta^{-1/2}\leq \Upsilon\min(1,|U(0)-\nu|^{1/3})\beta^{-1/3}   \,,
  \end{displaymath}
which allows for the application of \eqref{eq:455}.

Suppose now that $(v,g)\in D(\LL_\beta^{\mathfrak z_\alpha})\times L^2(0,1)$ satisfy
\begin{displaymath}
  (\LL_\beta^{\mathfrak z_\alpha}-\beta\lambda)v=g \,. 
\end{displaymath}
Then as in \cite[Eq.
(6.20)]{almog2019stability} or \eqref{eq:477}-\eqref{eq:478}  we may write $v$ in the form
\begin{equation}
\label{eq:607}
  v=A(\psi_\theta-\tilde{v}_\theta)+u \,,
\end{equation}
where
\begin{displaymath}
   u=(\LL_\beta^\Nf-\beta\lambda)^{-1}g\,.
\end{displaymath}
Taking the inner product of \eqref{eq:607} with $\mathfrak z_\alpha$ yields
\begin{displaymath}
  A\langle\mathfrak z_\alpha\,,\,(\psi_\theta-\tilde{v}_\theta)\rangle=\langle\mathfrak z_\alpha\,,\,u\rangle \,.
\end{displaymath}
As in \cite{almog2019stability}, using the approximation ${\mathfrak
  z}_\alpha\approx e^{-\alpha(1-x)}$ (see the display below \cite[Eq.
(6.95)]{almog2019stability} and its proof with minor changes), 
we then show that
\begin{displaymath}
  \langle\mathfrak z_\alpha\,,\,(\psi_\theta-\tilde{v}_\theta)\rangle=\beta^{-1/3}[1+\OO(\beta^{-1/3}]\,. 
\end{displaymath}
Consequently,
\begin{equation}
  \label{eq:608}
|A|\leq C\beta^{1/3}|\langle\mathfrak z_\alpha\,,\,u\rangle|\,.
\end{equation}

To complete the proof we need an estimate for $|\langle\mathfrak z_\alpha ,u\rangle|$.
The proof in the
case $\nu\leq U(1/2)$ is identical with the derivation of
\cite[Eq. (6.93)]{almog2019stability}, given that \eqref{eq:287} 
together with \eqref{eq:363} give  \cite[Eq. (5.4)]{almog2019stability} in this case.
Hence we consider here only the case where $\nu>U(1/2)$. 

We thus write, with the aid of either \eqref{eq:287} or \eqref{eq:369} 
\begin{equation}
\label{eq:609}
  |\langle\mathfrak z_\alpha ,{\mathbf
    1}_{[0,\hat{x}_\nu]}u\rangle|\leq Ce^{-\alpha/2}\|{\mathbf
    1}_{[0,\hat{x}_\nu]}u\|_2\leq Ce^{-\theta\beta^{1/3}/2}\|g\|_2\,,
\end{equation}
where $\hat{x}_\nu$ is given by \eqref{eq:548} so that
$U(\hat{x}_\nu)=\nu/2$. Furthermore,
\begin{equation}
\label{eq:610}
   |\langle\mathfrak z_\alpha\,,\,{\mathbf
    1}_{[\hat{x}_\nu,1]}u\rangle|\leq \|\mathfrak z_\alpha \|_1\|{\mathbf
    1}_{[\hat{x}_\nu,1]}u\|_\infty\leq\frac{C}{\theta\beta^{1/3}}\|{\mathbf
    1}_{[\hat{x}_\nu,1]}u\|_\infty
\end{equation}
Since $u(1)=0$ it holds that
\begin{displaymath}
  \|{\mathbf     1}_{[\hat{x}_\nu,1]}u\|_\infty\leq  \|{\mathbf
    1}_{[\hat{x}_\nu,1]}u\|_2 \|{\mathbf
    1}_{[\hat{x}_\nu,1]}u^\prime\|_2\,. 
\end{displaymath}
By \eqref{eq:556} and \eqref{eq:557},  and \eqref{eq:610} we then have
\begin{displaymath}
   |\langle\mathfrak z_\alpha \,,\,{\mathbf 1}_{[\hat{x}_\nu,1]}u\rangle| \leq
   \frac{C}{\theta\beta^{4/3}}\|g\|_2 \,.
\end{displaymath}
Substituting the above, together with \eqref{eq:609} into
\eqref{eq:608} yields
\begin{displaymath}
|A|\leq   \frac{C}{\theta\beta}\|g\|_2 \,.
\end{displaymath}
By \eqref{eq:607}, \eqref{eq:457},  \eqref{eq:455}, and \eqref{eq:512} we then have
\begin{equation}\label{eq:611}
  \|v\|_2 \leq \frac{C}{\theta\beta}(\|\psi_\theta\|_2+\|\tilde{v}_\theta\|_2) \|g\|_2+ \|u\|_2 
\end{equation}
 By \eqref{eq:457} with $k=0$ and \eqref{eq:605} it holds that
\begin{displaymath}
  \frac{1}{ \theta \beta} \|\psi_{\theta}\|_2  \leq C  \, \lambda_\beta^{\frac{1}{4}} \beta^{-7 /6} \,.
\end{displaymath}
By \eqref{eq:512} and \eqref{eq:605} it holds that
 \begin{displaymath}
  \frac{1}{ \theta \beta} \|\tilde v_{\theta}\|_2  \leq C  \, \beta^{-19 /12}\,.
\end{displaymath}
Finally,  \eqref{eq:287} and \eqref{eq:365} establish the existence of $\Upsilon >0$ such that
\begin{equation}\label{eq:612}
  \|u\|_2\leq   \frac{C}{\beta^{1/2}[1+\beta^{1/6}|U(0)-\nu|^{1/3}]} \|g\|_2\,,
\end{equation}
for all $U(0)/2<\nu < U (0) + a \beta^{-1/2}$. \\

For $\nu \geq  U (0) + a \beta^{-1/2}$
we can use (\ref{eq:553}b) with $\breve{\chi}_\nu\equiv1$ and $f=g$ to obtain 
\begin{displaymath}
   \|u\|_2\leq   \frac{C}{\beta|U(0)-\nu|} \|g\|_2\,,
\end{displaymath}
and hence \eqref{eq:612} 
holds true for all $\nu>U(0)/2\,$.\\
Combining the above yields
\begin{displaymath}
  \|v\|_2\leq   \frac{C}{\beta^{1/2}[1+\beta^{1/6}|U(0)-\nu|^{1/3}]} \|g\|_2\,,
\end{displaymath}
verifying thereby \eqref{eq:604}. 
\end{proof}

\section{The Orr-Sommerfeld operator}
\label{sec:5}
\subsection{ Introduction}
In this section we prove Theorems \ref{thm:small-alpha} and
\ref{thm:large-alpha} by obtaining inverse estimates for the
Orr-Sommerfeld operator \eqref{eq:6}.  As in
\cite{almog2019stability}, we use the estimates for the inviscid
operator $\mathcal A_{\lambda,\alpha}$ from Section \ref{sec:2} together with
the resolvent estimates for the Schr\"odinger operators $\mathcal
L_\beta^{\mathfrak N,\mathfrak D}$ and $\mathcal L_\zeta^\beta$ from Sections
\ref{sec:3} and \ref{sec:4}. In contrast with
\cite{almog2019stability} we need to consider here many different
cases depending on the values of $\Im \lambda$ and $\alpha$.

 \begin{figure}[htbp]
 \begin{center}
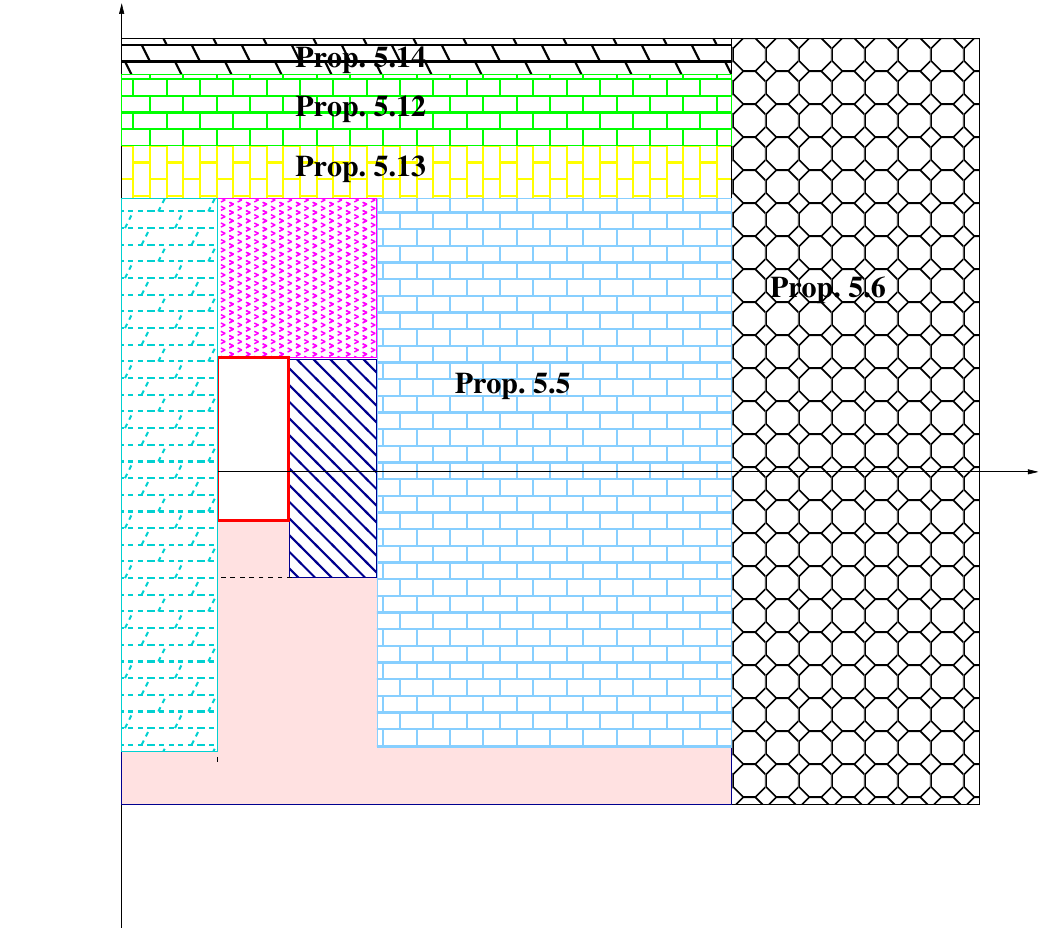
 \caption{Summary of the results in Section \ref{sec:5}.  }
 \label{fig:fig2}
 \end{center}
 \end{figure}

Figure 2 presents a rough sketch of the various domains where each
estimate is valid in the $(\alpha,\nu)$ plane ($\nu=\Im\lambda$). The blank domain
denotes the domain in the $(\alpha,\nu)$ plane where resolvent estimates
have not been obtained. We refer the reader to Subsection
\ref{sec:proof-strategy} for a brief explanation of the methods of the
proof.  

In the following we explain why the division of the $(\alpha,\nu)$ plane
into 10 subdomains is necessary. Propositions~\ref{lem:large-nu} and
\ref{large-nu-negative} deal with the case where $\nu\not\in[0,U(0)]$
making use of the invertibility of $\A_{i\nu,\alpha}$ in these cases. The
necessity of Proposition\ref{lem:large-alpha} which deals with the
case $\alpha\gtrsim\beta^{1/3}$, and Proposition~\ref{lem:zero-alpha} (and
\ref{cor:small-alpha}) which deals with the case $\alpha\ll\beta^{-1/6}$ is
explained in Section \ref{sec:proof-strategy}.
Proposition~\ref{lem:intermediate-alpha} deals with the case
$|\nu|<\nu_0<U(0)$ and $1\ll\alpha\ll\beta^{1/3}$. In this range of $\alpha$ values
we may effectively use the fact that $\|(-d^2/dx^2+\alpha^2)^{-1}\|$ is
small at the conclusion of the proof.
Proposition~\ref{lem:right-curve} deals with the case
$\nu\geq\beta^{-1/5+\delta}$ for any $0<\delta<1/5$ and $\alpha\lesssim1$. In the proof we use
the same methods as in \cite{almog2019stability}, till the value of
$\nu$ becomes to small due to the non invertibility of $\A_{0,0}$. For
$|\nu|\leq\beta^{-1/5+\delta}$ and $\beta^{-1/10+\delta/2}\ll\alpha\lesssim1$ we use
Proposition~\ref{lem:small-lambda-intermediate-alpha}. This range of
$\alpha$ values allows the application of Proposition~\ref{prop:edge}
towards the end of the proof. Proposition~\ref{lem:quadratic} deals
with the case where $|U(0)-\nu|\lesssim\beta^{-1/2}$. Here we can approximate
$U$ by a quadratic potential near $x=0$ and use the estimates in
Subsections \ref{sec:case--quadratic}, \ref{sec:schrod-quad}, and
\ref{sec:resolvent-estimates-quadratic}.  Finally,
Proposition~\ref{lem:nearly-quadratic} deals with the transition from
a linear behavior of $U-U(x_\nu)$ ($x_{\nu}$ is defined in
\eqref{eq:defxnu}) to a quadratic behavior near $x_\nu$.

\subsection{Preliminaries}
\label{sec:5.1}
We begin by recalling  from \eqref{eq:558} the definition of
the boundary terms  
\begin{equation}
\label{eq:613}
\hat{\psi}_{\lambda,\beta} (x)= \frac{{\rm Ai}\big(
  \beta^{1/3}e^{ -i\pi/6}[(1- x)- i\lambda]\big)} 
{{\rm Ai}\big( e^{ -i2\pi/3}\beta^{1/3}\lambda\big)}\chi(1-x)\,,
\end{equation}
where we recall that $\chi$ is given by \eqref{eq:103}\,.  We also
recall from \cite[Section 8.3.2, Eq: (8.91)]{almog2019stability} 
  that there exists $\Upsilon>0$ such that, for all  $\beta\geq 1$ and  $\Re\lambda<\Upsilon\beta^{-1/3}$, it
  holds that
\begin{equation} 
\label{eq:614}
\|(1-x)^s \hat{\psi}_{\lambda,\beta}\|_1 \leq C\,  \lambda_\beta^{-(s+1)/2} \,
\beta^{-(s+1)/3} \,,\, s\in [0,3]\,.
\end{equation}
Similarly, from \cite[Proposition A.8 and Eq.(A.43c,d)]{almog2019stability},  we can conclude the
existence of $C>0$ such that
\begin{equation}
  \label{eq:615}
\| (1-x)^s \hat{\psi}_{\lambda,\beta}\|_\infty\leq C \, \lambda_\beta^{- s/2} \,
\beta^{-s/3},  s\in [0,3]   \,.
\end{equation}
We further recall the definition of the inviscid operator in
\eqref{eq:2.1aa}, which is the  Neumann-Dirichlet realization in
$(0,1)$ of
\begin{equation*}
  \A_{\lambda,\alpha} \overset{def}{=} (U+i\lambda)\Big(-\frac{d^2}{dx^2}+ \alpha^2\Big)+U^{\prime\prime}\,,
\end{equation*}
for $\lambda\in \mathbb C$ and $\alpha \in \mathbb R$.  
We note that $\A_{\lambda,\alpha}$  is invertible when either  $\nu\not\in[0,U(0)]$
or  $|\mu|>0$, either by Proposition  \ref{prop:linear} or by Proposition 4.13 in
  \cite{almog2019stability} which holds true since
  $|U^{\prime\prime}|>0$.   We introduce in addition
  \begin{equation}\label{eq:616}
  \phi_{\lambda,\beta,\alpha}:=\A_{\lambda,\alpha}^{-1}(U+i\lambda)\hat{\psi}_{\lambda,\beta}\,.
  \end{equation}
  We dedicate this subsection to two extensions  of \cite[Lemma
8.1]{almog2019stability}. These are useful in order to establish
the contribution of the boundary terms in \eqref{eq:929}. The reader
is referred to Subsection \ref{sec:proof-strategy} for more details on
the necessity of these estimates. The first of them is the following
lemma which considers the case $\nu>\beta^{-1/5}$. The proof is
significantly more complex than the proof of  \cite[Lemma
8.1]{almog2019stability} in view of the non-injectivity of $\A_{0,0}$. 
 \begin{lemma}
   \label{lem:inviscid-decay} 
   Let $U\in C^3([0,1])$ satisfy \eqref{eq:10}.  There exist positive
   constants $\Upsilon$, $C$, $\widehat C$, $\nu_0$, and $\beta_0$ such that, for all $\alpha\geq0$,
   $\beta\geq \beta_0$, $\lambda\in\C$ for which
   $ \mu<\Upsilon\beta^{-1/3}$, 
   $\mu\neq0$, and $ \beta^{-1/5}<\nu<\nu_0$ it holds that
   \begin{equation}
\label{eq:617}
 \| \phi_{\lambda,\beta,\alpha} \|_{1,2} \leq  C \nu^{-1} \,[ |\lambda|\beta]^{-3/4} \,,
   \end{equation}
and
   \begin{equation}
\label{eq:618}
  |\phi_{\lambda,\beta,\alpha}(x_\nu)| \leq \widehat  C\, [|\lambda|\beta]^{-3/4}\,.
\end{equation}
 \end{lemma}
 \begin{proof}~\\
 
   {\em Step 1:
   We prove \eqref{eq:617} and \eqref{eq:618}   for the case $
   \alpha^2<\nu^{-1}$  and $0<|\mu|<1$. \\}

\noindent  By \eqref{eq:89} applied to the pair $(\phi,v)$ with $v= (U+i \lambda) \hat
 \psi$ (see \eqref{eq:616}), it holds that
\begin{equation}
\label{eq:619}
    |\phi(x_\nu)|^2 \leq C \,  |\langle\phi,\hat{\psi}\rangle| \,. 
\end{equation}
Let $\tilde{x}_\nu=(1+x_\nu)/2$.   To estimate the
right hand side of \eqref{eq:619}, we first obtain a 
bound for $\|\phi^\prime\|_{L^\infty(\tilde{x}_\nu,1)}$. To this end, we integrate the balance
$(U+i\lambda)^{-1}\A_{\lambda,\alpha}\phi=\hat{\psi}$ over $(\tilde{x}_\nu,1)$ to obtain
\begin{equation}\label{eq:620}
  \|\phi^{\prime\prime}\|_{L^1(\tilde{x}_\nu,1)}\leq\Big(\alpha^2+\frac{C}{\nu}\Big) \|\phi\|_{L^1(\tilde{x}_\nu,1)}+
  \|\hat{\psi}\|_{L^1(\tilde{x}_\nu,1)}\,.
\end{equation}
Since $\phi(1)=0$ it holds that
\begin{equation}\label{eq:621}
  \|\phi\|_{L^1(\tilde{x}_\nu,1)}\leq
  \|\phi^\prime\|_{L^\infty(\tilde{x}_\nu,1)}\|1-x\|_{L^1(\tilde{x}_\nu,1)}\leq C \, 
  \|\phi^\prime\|_{L^\infty(\tilde{x}_\nu,1)}\nu^2 \,.
\end{equation}
Using \eqref{eq:620}, \eqref{eq:621}, \eqref{eq:614} for $s =0$, and
the fact that (in this step) $\alpha^2\leq |\nu|^{-1}$, we obtain that
\begin{equation}\label{eq:622}
   \|\phi^{\prime\prime}\|_{L^1(\tilde{x}_\nu,1)}\leq
   C\,\big(\nu\|\phi^\prime\|_{L^\infty(\tilde{x}_\nu,1)} +[|\lambda|\beta]^{-1/2}\big) \,.
\end{equation}
Clearly, there exists $z_\nu\in [\tilde{x}_\nu,1]$ such that
\begin{displaymath}
  |\phi^\prime(z_\nu)|\leq|1-\tilde{x}_\nu|^{-1/2}\|\phi^\prime\|_{L^2(\tilde{x}_\nu,1)}\leq \widehat C\, |\nu|^{-1/2}\|\phi^\prime\|_{L^2(\tilde{x}_\nu,1)}\,.
\end{displaymath}
Since for any $x\in (\tilde{x}_\nu,1)$ it holds that
\begin{displaymath}
   |\phi^\prime(x)-\phi^\prime(z_\nu)|\leq  \|\phi^{\prime\prime}\|_{L^1(\tilde{x}_\nu,1)} \,,
\end{displaymath}
we can deduce that
\begin{equation}
\label{eq:623}
   \|\phi^\prime\|_{L^\infty(\tilde{x}_\nu,1)}\leq C\nu^{-1/2}\|\phi^\prime\|_{L^2(\tilde{x}_\nu,1)}+  \|\phi^{\prime\prime}\|_{L^1(\tilde{x}_\nu,1)} \,.
\end{equation}
We can then conclude, using \eqref{eq:622},  that we can choose
$\nu_0>0$ and $C>0$  such  that for all $0<\nu<\nu_0$
\begin{equation}
\label{eq:624}
   \|\phi^\prime\|_{L^\infty(\tilde{x}_\nu,1)}\leq C\big( \nu^{-1/2} \|\phi^\prime\|_{L^2(\tilde{x}_\nu,1)}+[|\lambda|\beta]^{-1/2}\big)\,.
\end{equation}
We now write 
\begin{multline}
\label{eq:625}
   |\langle\phi, \hat \psi\rangle|\leq |\langle\phi, \hat \psi\rangle_{L^2(\tilde{x}_\nu,1)}|+ |\langle\phi, \hat
   \psi\rangle_{L^2(0,\tilde{x}_\nu)}| \\ \leq 
   \|\phi^\prime\|_{L^\infty(\tilde{x}_\nu,1)}\|(1-x)\hat{\psi}\|_{L^1(\tilde{x}_\nu,1)} +
   \|\phi^\prime\|_2\|(1-x)^{1/2}\hat{\psi}\|_{L^1(0,\tilde{x}_\nu)}\,,  
\end{multline}
and then observe that
\begin{equation}\label{eq:626}
\begin{array}{ll}
  \|(1-x)^{1/2}\hat{\psi}\|_{L^1(0,\tilde{x}_\nu)}&  =   \|(1-x)^{-5/2} (1-x)^3\hat{\psi}\|_{L^1(0,\tilde{x}_\nu)} \\
   &  \leq (1-\tilde x_\nu)^{-5/2}  \|(1-x)^3\hat{\psi}\|_{L^1(0,\tilde{x}_\nu)} \\ & \leq 
  C\,  \nu^{-5/2} \, \|(1-x)^3\hat{\psi}\|_1\,.
  \end{array}
\end{equation}
By  \eqref{eq:87} with $v=(U+i\lambda)\hat{\psi}$ together with \eqref{eq:86} it holds that
  \begin{equation}\label{eq:627}
    \|\phi^\prime\|_2\leq \frac{C}{\nu^{3/2}}\|[(1-x)^{1/2}+\nu^{-1/2}(1-x)]\hat{\psi}\|_1 \,.
  \end{equation}
 Hence, by \eqref{eq:614} with $s=1/2,1,3$, \eqref{eq:626} and
\eqref{eq:627}  we obtain that 
\begin{multline*}
   \|\phi^\prime\|_2\|(1-x)^{1/2}\hat{\psi}\|_{L^1(0,\tilde{x}_\nu)}\\ \leq
  \frac{C}{\nu^4} \|[(1-x)^{1/2}+\nu^{-1/2}(1-x)]\hat{\psi}\|_1\|(1-x)^3\hat{\psi}\|_1 \\ \leq  
  \frac{\widehat C}{\nu^4}([|\lambda|\beta]^{-11/4}+\nu^{-1/2}[|\lambda|\beta]^{-3})\,.
\end{multline*}
Then, since by assumption $\nu>\beta^{-1/5}$, we obtain that
\begin{equation}
\label{eq:628}
   \|\phi^\prime\|_2\|(1-x)^{1/2}\hat{\psi}\|_{L^1(0,\tilde{x}_\nu)}\leq
   \widetilde C \beta^{-1/5}\, [|\lambda|\beta]^{-3/2}\,.   
\end{equation}
 By \eqref{eq:614} with $s=1$  and
\eqref{eq:624}  we have that
\begin{equation}\label{eq:629}
  \|\phi^\prime\|_{L^\infty(\tilde{x}_\nu,1)}
  \|(1-x)\hat{\psi}\|_{L^1(\tilde{x}_\nu,1)} \leq C\, \big( \nu^{-1/2}
  \|\phi^\prime\|_{L^2(\tilde{x}_\nu,1)}+[|\lambda|\beta]^{-1/2}\big) \, [|\lambda|\beta]^{-1}\,.
\end{equation}
Substituting \eqref{eq:629}, together with \eqref{eq:628} into
\eqref{eq:625} then yields
\begin{equation}
\label{eq:630}
 |\langle\phi, \hat \psi\rangle|   \leq 
 C(|\nu|^{-1/2}\|\phi^\prime\|_{L^2(x_\nu,1)}+[|\lambda|\beta]^{-1/2})[|\lambda|\beta]^{-1} \,.
\end{equation}
We now use \eqref{eq:102} and \eqref{eq:86} to
obtain that
\begin{equation}
\label{eq:631}
  \|\phi^\prime\|_{L^2(x_\nu,1)}\leq C\Big[ \nu^{-1/2}  |\langle\phi, \hat \psi\rangle|^{1/2} +
  \|[(1-x)^{1/2}+\nu^{-1/2}(1-x)]\hat{\psi}\|_1 \Big]
\end{equation}
Substituting \eqref{eq:631} into \eqref{eq:630} yields, with the aid of
\eqref{eq:614} and the fact that \break $\nu>\beta^{-1/5}$ 
\begin{displaymath} 
 |\langle\phi, \hat \psi\rangle|   \leq
 C(\nu^{-1}|\langle\phi, \hat \psi\rangle|^{1/2}+[|\lambda|\beta]^{-1/2})[|\lambda|\beta]^{-1} \,,
\end{displaymath}
which immediately implies
\begin{equation} \label{eq:632}
 |\langle\phi, \hat \psi\rangle|   \leq
 C\, (\nu^{-2} [|\lambda|\beta]^{-1}+[|\lambda|\beta]^{-1/2})[|\lambda|\beta]^{-1} \leq \widetilde C\,  [|\lambda|\beta]^{-3/2}  \,.
\end{equation}
From the above inequality, with the aid of \eqref{eq:619}, we
conclude \eqref{eq:618}.  Hence, by \eqref{eq:631}, we also get that
\begin{equation}
\label{eq:633}
  \|\phi^\prime\|_{L^2(x_\nu,1)}\leq C \, \nu^{-1/2}  [|\lambda|\beta]^{-3/4}\ \,.
\end{equation}
To obtain an effective bound for $\|\phi^\prime\|_2$ we use
\eqref{eq:131} and  \eqref{eq:86}   to
obtain, with the aid of \eqref{eq:631}, \eqref{eq:633}, and \eqref{eq:614}
\begin{multline}
\label{eq:634}
  \|\phi^\prime\|_2\leq  C \left(  \|[(1-x)^{1/2}+  \nu^{-1/2}(1-x)]\hat{\psi }\|_1  +
  \nu^{-1/2}\|\phi^\prime\|_{L^2(x_\nu,1)}\right)  \\ \leq
  \widetilde C\, \nu^{-1}[|\lambda|\beta]^{-3/4}  \,,
\end{multline}
from which we  conclude \eqref{eq:617} by using Poincar\'e's inequality.\\

{\it Step 2. We prove \eqref{eq:617} and \eqref{eq:618} for the case
$ \alpha^2 >\nu^{-1}$  and $|\mu|\leq 1$. \\}

\noindent To obtain \eqref{eq:617} for $\alpha^2>|\nu|^{-1}$ and $\nu < \nu_0$, we observe that for any
$A_0>0$ we can choose $\nu_0$ such that $\alpha^2 \geq A_0$ for $\nu <\nu_0$. and
consequently use \eqref{eq:162} in the form (with $v = (U+i\lambda) \hat
\psi$)
\begin{equation*}
\|\phi\|_{1,2}\leq  C   \Big\|[(1-x)^{1/2}+\nu^{-1/2}(1-x)]\frac{v}{U+i\lambda}\Big\|_1=  C  \, 
 \|[(1-x)^{1/2}+\nu^{-1/2}(1-x)]\hat \psi  \|_1\,.
\end{equation*} 
Using  \eqref{eq:614} and  the fact that $\nu>\beta^{-1/5}$, we obtain, 
\begin{equation*}
\|\phi\|_{1,2}\leq  \widehat C \,  \Big([|\lambda|\beta]^{-3/4}+\nu^{-1/2}
[|\lambda|\beta]^{-1}\Big) \leq \widetilde C \,  [|\lambda|\beta]^{-3/4}\,,
\end{equation*}
which implies \eqref{eq:617} for $\alpha^2 > \nu^{-1}$.  We now use
\eqref{eq:619} to obtain
\begin{equation}\label{eq:635}
   |\phi(x_\nu)|^2 \leq C\|\phi^\prime\|_2 \|(1-x)^{1/2}\hat \psi  \|_1\,.
\end{equation}
We may then conclude \eqref{eq:618} as well by using \eqref{eq:617}
and \eqref{eq:614}.  \\

{\it Step 3. We prove \eqref{eq:617} and \eqref{eq:618} for  $|\mu|\geq 1$. }\\[1.5ex] 
 The proof of \eqref{eq:617} in this case follows from \eqref{eq:100}
which yields
\begin{displaymath}
   \|\phi^\prime\|_2^2\leq C\, |\langle\phi,\hat{\psi}\rangle|\,.
\end{displaymath}
Consequently, by \eqref{eq:614} we obtain that
\begin{displaymath}
  \|\phi^\prime\|_2^2\leq C \|\phi^\prime\|_2 \|(1-x)^{1/2}\hat{\psi}\|_1  \leq  \widehat{C} \, [\|\lambda|\beta]^{-3/4}\|\phi^\prime\|_2 \,,
\end{displaymath}
yielding, thereby,
\begin{displaymath}
  \|\phi^\prime\|_2 \leq \widehat C \, [\|\lambda|\beta]^{-3/4} \,.
\end{displaymath}
  We can then complete the proof of \eqref{eq:617} by using 
   Poincar\'e's  inequality.   The proof of \eqref{eq:618} follows from
   \eqref{eq:617}  and Sobolev embeddings.
\end{proof}

 We next consider (as in Proposition \ref{prop:edge})   the case
$|\lambda|\ll1$ and $\alpha$ large enough, which will be sufficient to guarantee a satisfactory
estimate of $\|\A_{\lambda,\alpha}^{-1}\|$ despite the fact that $\A_{0,0}$ is
not injective.
 \begin{lemma}
   \label{lem:other} 
   Let $\delta>0$, and $U\in C^3([0,1])$ satisfy \eqref{eq:10}.  There exist
   positive constants $\Upsilon_0$, $\lambda_0$, $C$ and $\beta_0$ such that, for
   all $\beta\geq \beta_0$, all $ \lambda\in\C\setminus\{0\}$ for which $ -\lambda_0
   <\Re\lambda\leq\beta^{-1/3}\Upsilon_0$, $|\Im\lambda|<\lambda_0$, and $\alpha\geq\alpha_{\lambda,\delta} $, it
   holds that
 \begin{subequations}
 \label{eq:636}
      \begin{equation}
 |c_\parallel(\lambda,\beta,\alpha) |\leq \frac{C}{ |\alpha^2\|U\|_2^2+i\lambda|}(\lambda_\beta^{-1}\beta^{-2/3}+\lambda_\beta^{-1/2}|\lambda|\beta^{-1/3}) \,,
    \end{equation}
  and
 \begin{multline}
  \Big\|\phi_{\lambda,\beta,\alpha} - c_\parallel(\lambda,\beta,\alpha)  U\Big\|_{1,2}\leq
  C\Big[\lambda_\beta^{-3/4}\beta^{-1/2}+ \\ +
  \frac{|\lambda|}{ |\alpha^2\|U\|_2^2+i\lambda|}(\lambda_\beta^{-1}\beta^{-2/3}+\lambda_\beta^{-1/2}|\lambda|\beta^{-1/3})\Big] \,, 
 \end{multline}
 \end{subequations}
where, as in \eqref{eq:19},
 \begin{equation*}
 c_{\parallel}(\lambda,\beta,\alpha)   = \frac{\langle U,\phi_{\lambda,\beta,\alpha} \rangle }{\|U\|_2^2}\,,
\end{equation*}	 
  and   as  in 
  \eqref{eq:51} 
\begin{displaymath}
\alpha_{\lambda,\delta}=\|U\|_2^{-1}\left(|\Im\lambda|(1+2\delta)\right)^{1/2}\,.
\end{displaymath}
  \end{lemma}

  \begin{proof}~\\
  We write as in \eqref{eq:19} with $\phi= \phi_{\lambda,\beta,\alpha}$
  \begin{displaymath}
    \phi=c_\parallel  \, U+\phi_\perp \,.
  \end{displaymath}
 Then by (\ref{eq:52}a) and (\ref{eq:50}) , there exists $\lambda_1$ such that for $0 < |\lambda|< \lambda_1$, 
     \begin{equation}
 \label{eq:637}
 |c_\parallel |\leq \frac{1+C|\lambda|^2\log
  |\lambda|^{-1}}{ |\alpha^2\|U\|_2^2+i\lambda|}\big(\|(U+i\lambda)\hat{\psi}\|_1+C|\lambda| \|\hat{\psi}\|_1\big) \,.
     \end{equation}
It follows from (\ref{eq:614}) (with $s=0$) that for some
  positive $C$
  \begin{equation}
  \label{eq:638}
    \|\hat{\psi}\|_1 \leq C\, \lambda_\beta^{-1/2} \,\beta^{-1/3}\,.
  \end{equation}
 Furthermore, by \eqref{eq:638} and \eqref{eq:614} (with $s=1$)
 \begin{equation}
 \label{eq:639} 
    \|(U+i\lambda)\hat{\psi}\|_1 \leq|\lambda|\, \|\hat \psi\|_1 + C  \|(1-x)\hat \psi\|_1 
  \leq  \widehat C \, \big( |\lambda|\lambda_\beta^{-1/2}\beta^{-1/3} + \lambda_\beta^{-1}\beta^{-2/3}\big )\,.
 \end{equation}
 Consequently, for $\beta_0$ large enough, there exists $\lambda_0 <
 \lambda_1/\sqrt{2} $ such that for any $ |\lambda| \leq \sqrt{2} \lambda_0$  it holds
 by \eqref{eq:637}, that 
 \begin{displaymath} 
  |c_\parallel |\leq \frac{C}{|\alpha^2\,\|U\|_2^2 +i\lambda|}\big(\lambda_\beta^{-1}\beta^{-2/3}+|\lambda|\lambda_\beta^{-1/2}\beta^{-1/3}\big) \,,
 \end{displaymath}
 readily verifying (\ref{eq:636}a).\\
  We now apply (\ref{eq:52}b) to
 obtain  
  \begin{displaymath} 
    \|\phi_\perp\|_{1,2} \leq  C\Big[ \|(1-x)^{1/2}\hat{\psi} \|_1+\frac{|\lambda|}{|\alpha^2\|U\|_2^2 +i\lambda|}(\|U\hat \psi\|_1+C|\lambda| \|\hat \psi\|_1)\Big] \,.
  \end{displaymath}
  which, combined with \eqref{eq:614} (with $s=1/2$ and $s=1$), \eqref{eq:638},
 yields
  \begin{equation}\label{eq:640}
     \|\phi_\perp\|_{1,2} \leq     C\Big[\lambda_\beta^{-3/4}\beta^{-1/2} 
   +  \frac{|\lambda|}{|\alpha^2 \|U\|_2^2 +i\lambda|}(\lambda_\beta^{-1}\beta^{-2/3}+\lambda_\beta^{-1/2}|\lambda|\beta^{-1/3})\Big] \,,  
  \end{equation}
 establishing, thereby,  (\ref{eq:636}b).
   \end{proof}

  \subsection{Resolvent estimates and Fredholm property}
  \label{sec:5.2}
We recall from the introduction that 
  \begin{equation}\label{eq:641}
     \B_{\lambda,\alpha,\beta}^{\mathfrak N, \mathfrak D} =(\LL_\beta -\beta\lambda)\Big(\frac{d^2}{dx^2}-\alpha^2\Big)  -i\beta U^{\prime\prime} \,.
   \end{equation}
   on $(0,1)$ with domain (see \eqref{eq:p29ab})
  \begin{equation}\label{eq:642}
 D(\B_{\lambda,\alpha,\beta}^{\mathfrak N, \mathfrak D})=\{u\in H^4(0,1)\,,\,  u^\prime(0)=u^{(3)}(0)=0  \mbox{ and }\,  u(1)=u^\prime(1)=0\},
 \end{equation}
and 
  \begin{equation}
  \label{eq:p29a}
    \LL_\beta  = -\frac{d^2}{dx^2}+i\beta U\,.
  \end{equation}
  Note that this domain is independent of the parameters $(\lambda,\alpha,\beta)$,
  i.e., \break $D(\B_{\lambda,\alpha,\beta}^{\mathfrak N, \mathfrak
    D})=D(\B_{0,0,0}^{\mathfrak N, \mathfrak D})$. 

It can be easily verified that $\B_{0,0,0}^{\mathfrak N, \mathfrak D}$ is invertible.
Next we observe that
   \begin{displaymath}
\B_{\lambda,\alpha,\beta}^{\mathfrak N, \mathfrak D} \big(\B_{0,0,0}^{\mathfrak N, \mathfrak D}\big)^{-1}  = I + K_{\lambda,\alpha,\beta}\,,
   \end{displaymath}
   where $ K_{\lambda,\alpha,\beta} $ is a compact operator from $L^2(0,1)$ to
   $L^2(0,1)$. Hence, $ I + K_{\lambda,\alpha,\beta}$ is a Fredholm operator. 
   Considering again the family $\B_{\lambda,\alpha,\beta}^{\mathfrak N, \mathfrak
     D} =  ( I + K_{\lambda,\alpha,\beta})\,\B_{0,0,0}^{\mathfrak N, \mathfrak D}$, we can conclude that it is a Fredholm family from
   $D(\B_{0,0,0}^{\mathfrak N, \mathfrak D})$ into $L^2(0,1)$.  \\
Since
   its index depends continuously on $(\alpha,\beta,\lambda)$ and vanishes
   for $(\alpha,\beta,\lambda)=(0,0,0)$, it must be zero for all $(\lambda,\alpha,\beta)\in\C\times\R^2$.
    
   The rest of Section \ref{sec:5} is dedicated to the estimation of
   $(\B_{\lambda,\alpha,\beta}^{\Df,\Nf})^{-1}$. In practice, we first show in
   each subsection that for some subset of parameters $(\lambda,\alpha,\beta)$,
  \begin{equation}
\label{eq:921}
    \phi\in  D(\B_{\lambda,\alpha,\beta}^{\Df,\Nf})\Rightarrow \|\phi\|_{1,2}\leq
  C(\lambda,\alpha,\beta)\|f\|_2\,.
  \end{equation}
 where $f=\B_{\lambda,\alpha,\beta}^{\Df,\Nf}\phi$. It follows
  that $\B_{\lambda,\alpha,\beta}^{\Df,\Nf}$ is injective and since its index
  vanishes, we can conclude the existence of
  $(\B_{\lambda,\alpha,\beta}^{\Df,\Nf})^{-1}:L^2(0,1)\to D(\B_{0,0,0}^{\mathfrak N,
    \mathfrak D})$ together with the estimate  
    \begin{equation}
\label{eq:921a}
\|(\B_{\lambda,\alpha,\beta}^{\Df,\Nf})^{-1}\|+\Big\|\frac{d}{dx}(\B_{\lambda,\alpha,\beta}^{\Df,\Nf})^{-1}\Big\|\leq
  C(\lambda,\alpha,\beta)\,.
  \end{equation}

  \subsection{Resolvent estimates for $\beta^{-1/5}\ll|\Im\lambda|<U(0)$}

The next proposition is somewhat similar to \cite[Lemma
    8.8]{almog2019stability} albeit with a significant difference: the fact
    that $\A_{0,0}$ is not invertible, which makes the  estimates
    become significantly more complex. 
\begin{proposition}
  \label{lem:right-curve} Let $U\in C^4([0,1])$ satisfy
  \eqref{eq:10} and the assumption $U^{\prime\prime\prime}(0)=0$. Let $0<\delta<1/5$, 
$\nu_0<U(0)$ and $\alpha_0$
  denote positive constants. There exist $C>0$, and $\beta_0>0$ such that
  for all $\beta\geq \beta_0$, it holds that
  \begin{equation}
\label{eq:643}
      \sup_{ \begin{subarray}{c}
          0\leq \alpha\leq\alpha_0\\
    \Re\lambda \leq   \beta^{-2/5-\delta}\\
          \beta^{-1/5+\delta}\leq\Im\lambda<\nu_0
           \end{subarray}} \Im \lambda \Big(
\big\|(\B_{\lambda,\alpha,\beta}^{\mathfrak N,\Df})^{-1}\big\|+
      \big\|\frac{d}{dx}\, (\B_{\lambda,\alpha,\beta}^{\mathfrak N,\Df})^{-1}\big\|\Big)\leq
      C \beta^{-1/2+ \delta}\,.
  \end{equation}
\end{proposition}
\begin{proof}~\\
We assume throughout the proof, without any loss of generality, that $0 < \delta \leq 1/30$.

{\it Step 1: Preliminaries.}\\
  Let $\phi\in D(\B_{\lambda,\alpha,\beta}^{\mathfrak N,\Df})$ and $f= \B_{\lambda,\alpha,\beta}^{\mathfrak N,\Df}\, \phi$. Let
  further $v_\Df\in H^2(0,1)$ be defined by
  \begin{equation}
\label{eq:644}
    v_\Df =\A_{\lambda,\alpha}\phi + (U+i\lambda)\phi^{\prime\prime}(1)\hat{\psi} \,,
  \end{equation}
where $\hat{\psi}=\hat \psi_{\lambda,\beta}$ is given by \eqref{eq:613}.  Note that by
  \eqref{eq:p29ab}, \eqref{eq:10} and the fact that $U^{\prime\prime\prime}(0)=0$, we
  have
\begin{equation}\label{eq:645}
v_\Df(1)=v_\Df^\prime(0)=0\,,
\end{equation}
and
hence  $v_\Df\in D(\LL_\beta^{\Nf,\Df})$ and  we may introduce, as in
\cite[Lemma 8.8]{almog2019stability}, 
  \begin{subequations}
\label{eq:646} 
\begin{equation}
g_\Df:=  (\LL_\beta^{\Nf,\Df}-\beta\lambda)v_{\Df} \,,
\end{equation}
which is expressible in the alternative form (using the fact that $f=
\B_{\lambda,\alpha,\beta}^{\mathfrak N,\Df}\, \phi$) 
\begin{equation}
\label{eq:647}
  g_\Df=(U+i\lambda)( -f + \phi^{\prime\prime}(1)\hat{g}) -   (U^{\prime\prime}\phi)^{\prime\prime} 
  -2U^\prime \tilde{v}_\Df^\prime-U^{\prime\prime} \tilde{v}_\Df\,,
\end{equation}
  \end{subequations}
wherein
\begin{equation}
\label{eq:648}
  \hat{g}: =  \big( \LL_\beta -\beta \lambda \big)\hat{\psi}  \,,
\end{equation}
and
\begin{equation}
\label{eq:649}
   \tilde{v}_\Df : = \frac{v_{\Df}-U^{\prime\prime}\phi}{U+i\lambda}=  - \phi^{\prime\prime} + \alpha^2\phi + \phi^{\prime\prime}(1)\hat{\psi}\,. 
\end{equation}
We note that 
\begin{equation}
\label{eq:650} 
   (\LL_\beta^\Nf-\beta\lambda)\tilde{v}_{\Df}= i\beta U^{\prime\prime}(\phi-\phi(x_\nu))+ i\beta U^{\prime\prime}\phi(x_\nu)- f + 
   \phi^{\prime\prime}(1)\hat{g} \,,
\end{equation}
where $x_\nu$ is given by \eqref{eq:defxnu}.
Recall that $\LL_\beta^\Nf$ stands for    $\LL_\beta^{\Nf,\Df}$, which is
defined in \eqref{eq:285}. 

As in the proof of \cite[Lemma 8.8]{almog2019stability} (see Eq.
  (8.90) there) we can
integrate by parts to obtain 
\begin{multline}
\label{eq:651}
  \Re\langle(U^{\prime\prime})^{-1}\tilde{v}_{\Df},(\LL_\beta^\Nf-\beta\lambda)\tilde{v}_{\Df}-i\beta U^{\prime\prime}\phi\rangle=
  \|(U^{\prime\prime})^{-1/2}\tilde{v}_\Df^\prime\|_2^2+ \\
  +\Re\langle\big((U^{\prime\prime})^{-1}\big)^\prime\tilde{v}_\Df,\tilde{v}_\Df^\prime\rangle
 -\beta \mu\, \|(U^{\prime\prime})^{-1/2}\tilde{v}_\Df\|_2^2  +\beta\, \Re\langle\phi^{\prime\prime}(1)\hat{\psi},i\phi\rangle
\end{multline}
We begin the estimation  of $\tilde{v}_\Df^\prime$ by obtaining a
bound for the last term on the
right-hand-side of \eqref{eq:651}. \\[1.5ex]

{\em Step 2:  Estimate of $ \beta
  \Re\langle\phi^{\prime\prime}(1)\hat{\psi}, i \phi\rangle\,$.\\}

\noindent We begin by writing $\phi$  as the sum  $$\phi = \hat w + \phi^{\prime\prime}(1)w $$ with 
\begin{displaymath}
  w(x)=\int_x^1(\xi-x)\hat{\psi}(\xi)\,d\xi\,,
\end{displaymath} 
 and the remainder
\begin{displaymath}
\hat w(x) := \int_x^1(\xi-x)[\phi^{\prime\prime}(\xi)-\phi^{\prime\prime}(1)\hat{\psi}(\xi)]\,d\xi\,.
\end{displaymath}
Then, we separately
  estimate the contribution of $ \beta
  \Re\langle\phi^{\prime\prime}(1)\hat{\psi}, i \phi^{\prime\prime}(1) w \rangle\,$ and $ \beta
  \Re\langle\phi^{\prime\prime}(1)\hat{\psi}, i \hat w \rangle\,$.
 By \eqref{eq:615} it holds that
\begin{displaymath}
 |w(x)|\leq C\, |1-x|^2\,.
\end{displaymath}
Consequently,
\begin{displaymath}
  |\Re\langle\phi^{\prime\prime}(1)\hat{\psi},i\phi^{\prime\prime}(1)w\rangle|\leq C \, |\phi^{\prime\prime}(1)|^2 \|(1-x)^2\hat{\psi}\|_1 \,,
\end{displaymath}
and hence, by \eqref{eq:614} with $s=2$, we then obtain that
\begin{equation}
\label{eq:652}
 \beta  |\Re\langle\phi^{\prime\prime}(1)\hat{\psi},i\phi^{\prime\prime}(1)w\rangle|\leq C \, [1+
    |\lambda|\beta^{1/3}]^{-3/2}\,  |\phi^{\prime\prime}(1)|^2 
\end{equation}
(see  \cite[Eq. (8.90)]{almog2019stability}).\\
To estimate $ \beta
  \Re\langle\phi^{\prime\prime}(1)\hat{\psi}, i \hat w \rangle\,$, 
we first obtain by \eqref{eq:649}, for $x\in (0,1)$,
\begin{equation}\label{eq:abcd}
\begin{array}{ll}
  \Big|\overline{\hat{\psi}(x)} \hat w  (x) \Big|&=
  \Big|\overline{\hat{\psi}(x)}\int_x^1(\xi-x)[-{\tilde
    v}_\Df(\xi)+\alpha^2\phi(\xi)]\,d\xi\Big|\\[1.5ex] & \leq C \, (1-x)^{5/2}\,
  |\hat{\psi}(x)|(\|\tilde{v}_\Df^\prime\|_2+\alpha^2\|\phi^\prime\|_2)\,, 
  \end{array}
\end{equation}
where  to obtain the last inequality we used the fact that
  \begin{displaymath}
|\phi(x)|\leq(1-x)^{1/2}\|\phi^\prime\|_2 \mbox{ and }|\tilde{v}_\Df(x)|\leq(1-x)^{1/2}\|\tilde{v}_\Df^\prime\|_2\,.
\end{displaymath}
  Using \eqref{eq:614} with $s=5/2$, we obtain
  \begin{equation}\label{eq:653}
   \beta
  |\Re\langle\phi^{\prime\prime}(1)\hat{\psi}, i \hat w \rangle\,| \leq  C |\phi^{\prime\prime}(1)|  \lambda_\beta^{-7/4} \,
    \beta^{-1/6} \big(\|\tilde{v}_\Df^\prime\|_2+\alpha^2\|\phi^\prime\|_2\big)\,.
  \end{equation}
Consequently, from \eqref{eq:652} and \eqref{eq:653}, we thus get, as
$\alpha\leq \alpha_0\,$, 
\begin{multline}
\label{eq:654} 
  \beta|\, \Re\langle\phi^{\prime\prime}(1)\hat{\psi},i\phi\rangle|\leq C \Big( \big[1+
    |\lambda|\beta^{1/3}]^{-3/2} |\phi^{\prime\prime}(1)|^2 
    \\ \quad +  \lambda_\beta^{-7/4} \,
    \beta^{-1/6}|\phi^{\prime\prime}(1)|(\|\tilde{v}_\Df^\prime\|_2+ \|\phi^\prime\|_2)\Big)\,.
  \end{multline}
 (See  \cite[Eq. (8.93)]{almog2019stability}.)\\

{\em Step 3: Estimate $ |\phi^{\prime\prime}(1)|$.}\\[1.5ex]
Let $\mathfrak z=\mathfrak z_\alpha$ be given by
\eqref{eq:603}  and recall that $\phi \in  D(\B_{\lambda,\alpha,\beta}^{\mathfrak
  N,\Df})$.  As $\langle{\mathfrak z},\phi^{\prime\prime}-\alpha^2\phi\rangle=0\,$,
  $\phi^{\prime\prime}-\alpha^2\phi$  belongs  
to the domain of $\LL_\beta^{\mathfrak z}-\beta\lambda$. Hence we may write
\begin{equation}
\label{eq:655}
   (\LL_\beta^{\mathfrak z}-\beta\lambda)(\phi^{\prime\prime}-\alpha^2\phi)=i\beta U^{\prime\prime}\phi+f\,.
\end{equation}
We separately solve in $D(\LL_\beta^{\mathfrak z})$ the equations
$(\LL_\beta^{\mathfrak z}-\beta\lambda)v=f$ and \break $(\LL_\beta^{\mathfrak z}-\beta\lambda)v_1=i
\beta U^{\prime\prime}\phi:=f_1 $, so that $(\phi^{\prime\prime}-\alpha^2\phi)=v_1+v_2$, and then
apply  \eqref{eq:447}  to the pair
$(v_1,f_1)$ (note that the assumptions of Proposition
\ref{lem:integral-conditions} are satisfied) and \eqref{eq:445} to the
pair $(v,f)$  to obtain 
\begin{equation}
  \label{eq:656}
|\phi^{\prime\prime}(1)|
\leq C \,
\lambda_\beta^{1/2}\, \big(\beta^{1/3}[\|\phi\|_{1,2}+ |\phi(x_\nu)|\,\big|\log|\nu+i\mf|^{-1}\big|]+\beta^{-1/2}\|f\|_2\big)\,,
\end{equation}
where $\mf$ is defined in (\ref{eq:447}b).
Note that since $\alpha\leq\alpha_0$ it holds that $\|{\mathfrak z}\|_{1,p}\leq C(\alpha_0)$.
Note further that
\begin{displaymath}
   |\phi(x_\nu)|\,\big|\log|\nu+i\mf|^{-1}\big| \leq C\, \nu^{1/2} \big
   |\log|\nu+i\mf|^{-1}\big|\, \|\phi^\prime\|_2\leq \widehat C\, \|\phi^\prime\|_2\,. 
\end{displaymath}
Then, for any $\hat \nu_0 >0$, there exists a constant $C>0$ such that
for  $|\lambda|\geq \hat  \nu_0\beta^{-1/3}$, (in particular it holds for $\nu >
\beta^{-1/5}$ for sufficiently large $\beta_0$).
\begin{equation}
  \label{eq:657}
|\phi^{\prime\prime}(1)|
\leq 
C \, |\lambda|^{1/2}\big(\beta^{1/2}[\|\phi\|_{1,2}+\beta^{-1/3}\|f\|_2\big)\,.
\end{equation}

Substituting  \eqref{eq:656} into \eqref{eq:654} yields
\begin{multline*}
  \beta|\, \Re\langle\phi^{\prime\prime}(1)\hat{\psi},i\phi\rangle|\leq
  C\Big[ \lambda_\beta^{-1/2} \, \big(\beta^{2/3}\|\phi\|_{1,2}^2 +\beta^{-1}\|f\|_2^2\big) + \\ +
  \lambda_\beta^{-5/4}\big(\beta^{1/3}\|\phi\|_{1,2}+\beta^{-1/2}\|f\|_2\big)   
  \big(\|\tilde{v}_\Df^\prime\|_2+ \|\phi^\prime\|_2\big)\Big] \,, 
\end{multline*}
from which,  we conclude that for
any $\hat \delta>0$ there exists $C_{\hat \delta}>0$ such that 
\begin{equation}
  \label{eq:658}
  \beta|\, \Re\langle\phi^{\prime\prime}(1)\hat{\psi},i\phi\rangle|\leq
  C_{\hat \delta} \big(\lambda_\beta^{-1/2} (\beta^{2/3}\|\phi\|_{1,2}^2+\beta^{-1}\|f\|_2^2\big)+ \hat \delta \, \|\tilde{v}_\Df^\prime\|_2^2\,.
\end{equation}
 Since, as in the proof of \eqref{eq:459} (see also \cite[Eq.
(6.18)]{almog2019stability}), we have for $\hat g$, introduced in
\eqref{eq:648}, 
\begin{equation}
\label{eq:659}
   |\hat{g}(x)|\leq C\,  \big(\beta\,(1-x)^2 |\hat{\psi}(x)|+ (1-x)^3  |\hat{\psi}^\prime(x)|\big) \,,
\end{equation}
we obtain from   \eqref{eq:451},  \eqref{eq:457}, \eqref{eq:458},
and \eqref{eq:558} that  
\begin{equation}
\label{eq:660} 
\beta^{1/3} \lambda_\beta^{-1}  \|(U+i\lambda)\hat{g}\|_2+ \|\hat{g}\|_2\leq
 C\beta^{1/6}\lambda_\beta^{-5/4} \,,
\end{equation}
Using the fact that $|\lambda|>\beta^{-1/5}$ we can then conclude
\begin{equation}
\label{eq:661}
 |\lambda|^{-1}\|(U+i\lambda)\hat{g}\|_2+ \|\hat{g}\|_2 \leq C\, \beta^{-1/4}|\lambda|^{-5/4}\,.
\end{equation}

{\it Step 4: Estimate of $\|\tilde v^\prime_\Df\|$.}\\[1.5ex]
From \eqref{eq:651}, we obtain, using \eqref{eq:650} and the fact that
$U^{\prime\prime}\neq 0$  
\begin{multline}
\label{eq:662}
 \frac 1C   \|\tilde{v}_\Df^\prime\|_2^2 \leq  \Re\langle(U^{\prime\prime})^{-1}\tilde{v}_{\Df}, - f + \phi^{\prime\prime}(1) \hat g \rangle
  \\
  - \Re\langle\big((U^{\prime\prime})^{-1}\big)^\prime\tilde{v}_\Df,\tilde{v}_\Df^\prime\rangle
 + \beta \mu\, \|\tilde{v}_\Df\|_2^2  - \beta\, \Re\langle\phi^{\prime\prime}(1)\hat{\psi},i\phi\rangle
\end{multline}
Next we obtain from \eqref{eq:662} and \eqref{eq:658}
(for sufficiently small $\hat \delta$)  that
\begin{multline}
\label{eq:663}
   \|\tilde{v}_\Df^\prime\|_2^2 \leq C\, \Big[\|\tilde{v}_\Df\|_2\big(\|f\|_2
   +|\phi^{\prime\prime}(1)|\,\|\hat{g}\|_2\big) +\\
   + (\mu_{\beta,+}  +1)
   \|\tilde{v}_\Df\|_2^2+ \lambda_\beta^{-1/2} \big( \beta^{2/3}\|\phi\|_{1,2}^2+\beta^{-1}\|f\|_2^2\big)\Big]
   \,,
\end{multline}
where
\begin{displaymath}
\mu_{\beta,+} := \max(\mu \beta,0)\,.
\end{displaymath}
We now substitute \eqref{eq:656} and \eqref{eq:660}  into
\eqref{eq:663}  to obtain 
 \begin{multline}
\label{eq:664}
  \|\tilde{v}_\Df^\prime\|_2^2\leq C\,\Big(\|\tilde{v}_\Df\|_2\big[\|f\|_2 +  \lambda_\beta^{-3/4}\beta^{1/2}\|\phi\|_{1,2}\big]+
\\+ \lambda_\beta^{-1/2} (\beta^{2/3}\|\phi\|_{1,2}^2+\beta^{-1}\|f\|_2^2)
+  (\mu_{\beta,+} +1)  \, \|\tilde{v}_\Df\|_2^2 \Big)\,.
 \end{multline} 
Hence,  for $|\lambda|>\beta^{-1/5}$ we can conclude that
 \begin{multline}
  \label{eq:665}
  \|\tilde{v}_\Df^\prime\|_2^2 \leq C\,\Big[ \|\tilde{v}_\Df\|_2(\|f\|_2
  +\beta^{1/4}|\lambda|^{-3/4}\|\phi\|_{1,2}) \\ + (\mu_{\beta,+} +1) 
\|\tilde{v}_\Df\|_2^2 +|\lambda|^{-1/2}\beta^{1/2}\|\phi\|_{1,2}^2+|\lambda|^{-1/2}\beta^{ -7/6}\|f\|_2^2\Big]\,.
\end{multline}

 {\it Step 5:  Estimate of $\|\tilde v_\Df\|$.}\\[1.5ex]
Given that $\nu<\nu_0$,  we may apply \eqref{eq:412}, (\ref{eq:287}a),   and Hardy's
inequality \eqref{eq:96}, to \eqref{eq:650} to obtain
\begin{equation}
\label{eq:666}
  \|\tilde{v}_{\Df}\|_2
  \leq C\big(\|\phi^\prime\|_2+\beta^{1/6}|\phi(x_\nu)|+\beta^{-2/3}[\|f\|_2+|\phi^{\prime\prime}(1)|\,\|\hat{g}\|_2]\big)\,.
\end{equation}
Using again  \eqref{eq:660} and \eqref{eq:656} we obtain
\begin{equation*}
  \|\tilde{v}_{\Df}\|_2
  \leq C\Big(\|\phi^\prime\|_2+\beta^{1/6}|\phi(x_\nu)|+\beta^{-2/3}\big[\|f\|_2+ 
    \lambda_\beta^{-3/4}\big(\beta^{1/2}\|\phi\|_{1,2}+\beta^{-1/3}\|f\|_2\big)\big]\Big)\,,
\end{equation*}
which implies  for any $\nu<\nu_0$
\begin{equation}
\label{eq:667}
  \|\tilde{v}_{\Df}\|_2
  \leq C\big(\|\phi\|_{1,2} +\beta^{1/6}|\phi(x_\nu)|+\beta^{-2/3} \|f\|_2 \big)\,.
\end{equation}

{\it Step 6: Estimate of $\|g_\Df\|_2$.}\\[1.5ex]
Substituting \eqref{eq:667}  into \eqref{eq:664}
yields  that 
\begin{multline}
  \label{eq:668}
\|\tilde{v}_\Df^\prime\|_2 \leq C\,
\big[  (\beta^{1/3}\lambda_\beta^{-1/4}+1) \|\phi\|_{1,2}+\beta^{1/6} |\phi(x_\nu)| \\
+(1 + \mu_{\beta,+}^{1/2}  \beta^{-2/3}  ) \|f\|_2
+\mu_{\beta,+}^{1/2}(\beta^{1/6}|\phi(x_\nu|+ \|\phi^\prime\|_2)\big]\,.
\end{multline}
For $|\lambda| \geq \beta^{-1/5}$ and $\mu \leq \beta^{-2/5} $ we conclude from
\eqref{eq:668} that
 \begin{multline*}
\|\tilde{v}_\Df^\prime\|_2 \leq C\,
\big[ (\beta^{1/4}|\lambda|^{-1/4}+1) \|\phi\|_{1,2}+\beta^{1/6}|\phi(x_\nu)|
+ \|f\|_2\\
+\mu_{\beta,+}^{1/2}(\beta^{1/6}|\phi(x_\nu|)+ \|\phi^\prime\|_2)\big]\,.
\end{multline*}
Furthermore, as
\begin{displaymath}
  (U^{\prime\prime}\phi)^{\prime\prime}=
  -U^{\prime\prime}\big(\tilde{v}_\Df+\alpha^2\phi+|\phi^{\prime\prime}(1)|\,\hat{\psi}\big)+2U^{(3)}\phi^\prime+U^{(4)}\phi \,,
\end{displaymath}
we obtain from \eqref{eq:667}, the boundedness of $\alpha$,
\eqref{eq:656}, and \eqref{eq:562}, that 
 \begin{equation}
\label{eq:669}
  \|(U^{\prime\prime}\phi)^{\prime\prime}\|_2\leq
  C\lambda_\beta^{\frac{1}{4}}(\beta^{1/6}\|\phi\|_{1,2}+\beta^{-2/3}\|f\|_2)\,.
\end{equation}
For $|\lambda| \geq \beta^{-1/5}$ the above inequality implies
 \begin{equation}
\label{eq:670}
  \|(U^{\prime\prime}\phi)^{\prime\prime}\|_2\leq
  C \, |\lambda|^{1/4}  \beta^{1/4} \, (\|\phi\|_{1,2}+  \beta^{-5/6}\|f\|_2\big)\,.
\end{equation}

By (\ref{eq:646}b) it holds that 
 \begin{equation}\label{eq:671}
  \| g_\Df\| \leq C \Big( \| (U+i\lambda)f\|_2 + |\phi^{\prime\prime}(1)|\,  \| (U+i\lambda) \hat{g}\|  + \|  (U^{\prime\prime}\phi)^{\prime\prime}) \| + 
  \|  \tilde{v}_\Df^\prime\| + \| \tilde{v}_\Df \|\Big)\,.
\end{equation}
We now substitute into \eqref{eq:671} the estimates \eqref{eq:657}
and \eqref{eq:661}, \eqref{eq:670}, \eqref{eq:668}, and
\eqref{eq:665}, to obtain that
\begin{multline}
\label{eq:672}
  \|g_\Df\|_2\leq C\big[ (1+|\lambda|) \, \|f\|_2
  +\mu_{\beta,+}^{1/2}(\beta^{1/6}|\phi(x_\nu)|+\|\phi\|_{1,2}) \\
  +  \beta^{1/4}(|\lambda|^{-1/4}+|\lambda|^{1/4})\,\|\phi\|_{1,2} \big] \,. 
\end{multline}
{\it Step 7: Estimate the contribution of the boundary term
$\A_{\lambda,\alpha}^{-1}\big([U+i\lambda]\phi^{\prime\prime}(1)\hat{\psi} \big)$. }\\[1.5ex] 

We continue as in the proof of \cite[Lemma 8.8]{almog2019stability}. We
first write, in view of \eqref{eq:644}
\begin{subequations}\label{eq:673}
\begin{equation}
  \phi = \phi_\Df + \check{\phi} \,,
\end{equation}
where 
\begin{equation}
  \phi_\Df =\A_{\lambda,\alpha}^{-1}v_\Df \quad ; \quad
 \check{\phi}=-\A_{\lambda,\alpha}^{-1}\big([U+i\lambda]\phi^{\prime\prime}(1)\hat{\psi} \big) \,. 
\end{equation}
 Note here that 
\begin{equation}
\check{\phi}=-\phi^{\prime\prime}(1)\phi_{\lambda,\beta,\alpha}\,.
\end{equation}
\end{subequations}
By  \eqref{eq:617} and \eqref{eq:657} we have
\begin{equation}\label{eq:674}
  \|\check{\phi}\|_{1,2}\leq
  C \,  |\nu|^{-1}\beta^{-1/4} |\lambda|^{-1/4}\,\big(\|\phi\|_{1,2}+\beta^{-5/6}\|f\|_2\big)\,.
\end{equation}
Consequently,  as $|\nu|>\beta^{-1/5+\delta}$,  we obtain for sufficiently large  $\beta_0$
\begin{equation}
  \label{eq:675}
 \|\check{\phi}\|_{1,2}\leq   C \,  |\nu|^{-1}\beta^{-1/4} |\lambda|^{-1/4}\, \big(\beta^{-5/6}\|f\|_2+\|\phi_\Df\|_{1,2}\big)\,.
\end{equation}
 Furthermore, by \eqref{eq:618} and (\ref{eq:673}c), we have
\begin{displaymath}
  |\check{\phi}(x_\nu)|\leq C[|\lambda|\beta]^{-3/4}|\phi^{\prime\prime}(1)| \,.
\end{displaymath}
Hence, by \eqref{eq:657} we obtain
 \begin{displaymath}
  |\check{\phi}(x_\nu)|\leq C[|\lambda|\beta]^{-1/4} \, \big(\|\phi\|_{1,2} +\beta^{-5/6}\|f\|_2\big)   \,,
\end{displaymath}
 and, therefore,  by \eqref{eq:675} and Poincar\'e's inequality we may
 conclude for $\nu\geq \beta^{-1/5+\delta}$ and $\beta_0$ large enough 
\begin{equation}
  \label{eq:676}
  |\check{\phi}(x_\nu)|\leq C[|\lambda|\beta]^{-1/4} \, \big(\|\phi_\Df^\prime\|_{2} +\beta^{-5/6}\|f\|_2\big)  \,. 
\end{equation}
\vspace{4ex}

{\em Step 8: Estimate $\phi_\Df$. }\\ 

Substituting the above
into \eqref{eq:672} yields,  with the aid of
\eqref{eq:673}, \eqref{eq:675},  and \eqref{eq:676} 
 \begin{multline}
\label{eq:677}
   \|g_\Df\|_2\leq C\Big[(1+|\lambda|)\|f\|_2
  +\mu_{\beta,+}^{1/2}(\beta^{1/6}|\phi_\Df(x_\nu)|+ \|\phi_\Df\|_{1,2})  \\
  +\beta^{1/4}(|\lambda|^{-1/4}+|\lambda|^{1/4})\|\phi_\Df\|_{1,2}\Big] \,.    
\end{multline}

By either \eqref{eq:157} or \eqref{eq:167} (for $|\mu|$ large) applied
to the pair $(v_\Df,\phi_\Df)$ together with \eqref{eq:645}, it holds
for any $1<q<2$ that
\begin{equation}\label{eq:678}
  |\phi_\Df(x_\nu)|\leq C\, \big[\|v_\Df\|_{1,q}+\nu^{-1/2}\|v_\Df\|_2+\nu^{-1}\|v_\Df\|_1\big] \,.
\end{equation}
We now estimate $\|\phi_\Df\|_{1,2}$ by applying \eqref{eq:86} and
\eqref{eq:87} to the pair $(v_\Df,\phi_\Df)$. We then conclude that
there exists $\mu_0>0$ such that for all $|\mu| \leq \mu_0$ and $0< \nu \leq
\nu_0$ it holds that
\begin{equation}\label{eq:679}
\|\phi_\Df\|_{1,2} \leq C \, \nu^{-1} n_\Df  \,,
\end{equation}
where
\begin{displaymath}
 n_D:=(\|v_\Df^\prime\|_q +  \nu^{-1/2}\|v_\Df\|_2+\nu^{-1}\|v_\Df\|_1) \,.
\end{displaymath}
For $|\mu|>\mu_0$ we use \eqref{eq:167} to obtain  for $0< \nu \leq
\nu_0$
\begin{equation}\label{eq:680}
  \|\phi_\Df\|_{1,2} \leq C\|(1-x)^{1/2}v_\Df \|_1\leq  C \, \nu^{-1} n_\Df  \,.
\end{equation}
We can then substitute \eqref{eq:678} and \eqref{eq:679} or
\eqref{eq:680} into \eqref{eq:677} to obtain  for any
$1<q<2$
 \begin{multline}
  \label{eq:681}
  \|g_\Df\|_2\leq C\, \Big([1+|\lambda|]\|f\|_2 +
 \big[\mu_{\beta,+}^{1/2}[\beta^{1/6}+  \nu^{-1}] \\  +\beta^{1/4}\nu^{-1}
 (|\lambda|^{-1/4}+|\lambda|^{1/4})\big]\, n_\Df\Big) 
\end{multline}

{\it Step 9: Estimate $n_\Df$.} \\

By either (\ref{eq:287}a) for $ -\beta^{-1/3}\leq\mu<\beta^{-2/5 -\delta }$ or \eqref{eq:363} for
$\mu<-\beta^{-1/3}$ we have, 
\begin{equation}
  \label{eq:682}
\|v_\Df\|_2 \leq \frac{C}{\beta^{\frac{2}{3}}+|\mu|\beta}\|g_\Df\|_2 \,,
\end{equation}
whereas by (\ref{eq:287}b),  which holds for $\mu\leq C\beta^{-1/3}$  and
$\beta_0$ large,  we have for all $\beta>\beta_0$
\begin{equation}
\label{eq:683}
  \|v_\Df^\prime\|_q \leq C_q\, \beta^{-\frac{2+q}{6q}}\|g_\Df\|_2 \,.
\end{equation}
 Furthermore by \eqref{eq:403} it holds that
\begin{equation}
\label{eq:684}
  \|v_\Df\|_1 \leq C_q\, \beta^{-5/6}\|g_\Df\|_2 \,.
\end{equation}
Substituting \eqref{eq:682}, \eqref{eq:683}, and \eqref{eq:684} into \eqref{eq:681}
yields, for $\delta \leq 1/30$, $\beta_0$ large enough, and $q$ satisfying
\begin{equation}
\label{eq:685}
1<q<\frac{4}{4-15\delta}\,,
\end{equation}
we obtain  the existence of $C>0$ such that for all $\beta \geq \beta_0$
\begin{equation}\label{eq:686}
   \|g_\Df\|_2\leq C\, (1+|\lambda|)\|f\|_2\,. 
\end{equation}
 We now use \eqref{eq:682} to obtain
\begin{equation}
\label{eq:687}
  \|v_\Df\|_2 \leq
  C\frac{1+|\lambda|}{\beta^{\frac{2}{3}}+|\mu|\beta}\|f\|_2\leq \widehat C \beta^{-2/3}\|f\|_2 \,,
\end{equation}
which is valid for all $\mu<\beta^{-\frac{2}{5}-\delta}$.

We next use \eqref{eq:683} and \eqref{eq:681}  to obtain
\begin{equation}\label{eq:688} 
 \|v^\prime_\Df\|_q \leq  C \beta^{-\frac{2+q}{6q}} (1+|\lambda|) \|f\|_2\,.
\end{equation}
Finally, use of \eqref{eq:684} and \eqref{eq:681} yields
\begin{equation}\label{eq:689} 
 \|v_\Df\|_1 \leq  C \beta^{-5/6}  (1+|\lambda|) \|f\|_2\,.
\end{equation}
 For
$-\mu_0\leq\mu<\beta^{-\frac{2}{5}-\delta}$ we have by \eqref{eq:687},
\eqref{eq:688} and \eqref{eq:689} that  the dominant term is the
one involving $ \|v_\Df^\prime\|_q $ and hence
\begin{equation} \label{eq:690}
   n_\Df\leq C \, \beta^{-\frac{2+q}{6q}}\|f\|_2 \,.
\end{equation}

{\it Step 10: Prove \eqref{eq:643}.} \\[1.5ex]

By \eqref{eq:679} we obtain, for
$-\mu_0\leq\mu\leq\beta^{-\frac{2}{5}-\delta}$\,, 
\begin{equation}
\label{eq:691}
  \|\phi_\Df\|_{1,2}\leq C \nu^{-1} \, \beta^{-\frac{2+q}{6q}}\|f\|_2 \,.
\end{equation}
For $\mu<-\mu_0$ we use \eqref{eq:167} and the first inequality in
\eqref{eq:687} to obtain
\begin{equation}\label{eq:692}
   \|\phi_\Df\|_{1,2}\leq C\, \|v_\Df\|_2\leq \widehat C\beta^{-1}\|f\|_2 \,.
\end{equation}
Combining \eqref{eq:675} with  \eqref{eq:691} and \eqref{eq:692} yields
\begin{equation}\label{eq:693}
   \|\phi\|_{1,2}\leq  C \nu^{-1} \,\beta^{-\frac{2+q}{6q}}\|f\|_2 \,.
\end{equation}
\end{proof}

\subsection{Resolvent estimates for $|\Im\lambda|=\OO (\beta^{-1/5})$
  and large $\beta^{1/10}\alpha$} 
\label{sec:5.3}
In the previous subsection we considered the case where $\Im \lambda \gg
\beta^{-1/5}$, where the inverse estimates derived in Subsection \ref{ss2.6} for the
Rayleigh operator $\A_{\lambda,\alpha}$ become effective for all $\alpha\geq0$. Here, we
use the estimates obtained in Subsection
\ref{sec:case-lambda-small-alpha-large} assuming $\alpha^2\gg|\Im\lambda|$. Since we
consider here $|\Im\lambda|=\OO (\beta^{-1/5})$ we need to consider
$\alpha^2\gg\beta^{-1/5}$.

\begin{proposition}
\label{lem:small-lambda-intermediate-alpha}
Let $ U\in C^4([0,1])$ satisfy \eqref{eq:10} and $U^{\prime\prime\prime}(0)=0$. Let
further $ 0 < \delta< 1/15$ and $\alpha_0$ denote positive constants. There
exist $C>0$, $\beta_0>0$, and $a_0 >0$ such that for all $\beta\geq \beta_0$, it
holds that
  \begin{equation}
\label{eq:694}
      \sup_{
        \begin{subarray}{c}
          a_0 \beta^{-1/10+\delta/2}\leq \alpha\leq\alpha_0\\
   \Re\lambda \leq   \beta^{-1/3-\delta} \\
          |\Im\lambda|\leq \beta^{-1/5+\delta}
           \end{subarray}}
\Big[\big\|(\B_{\lambda,\alpha,\beta}^{\mathfrak N,\Df})^{-1}\big\|+
      \Big\|\frac{d}{dx}\, (\B_{\lambda,\alpha,\beta}^{\mathfrak
        N,\Df})^{-1}\Big\|\Big]\leq C\, \beta^{-1/2+\delta}\,.
  \end{equation}
\end{proposition}
\begin{proof}~\\
{\em Step 1: With $g_\Df$ given by \eqref{eq:647}, we prove that
\begin{multline}
\label{eq:695}
   \|g_\Df\|_2\leq C\big[ (1+|\lambda|) \, \|f\|_2\\
  +(\lambda_\beta^{-1/4}\beta^{1/3}+\beta^{1/3-\delta/2} + \lambda_\beta^{1/4}
  \beta^{1/6})\|\phi\|_{1,2}+ \beta^{1/2-\delta/2}|\phi(x_\nu)|) \big] \,. 
\end{multline}
}

    Let $\phi\in D(\B_{\lambda,\alpha,\beta}^{\mathfrak N,\Df})$, $f= \B_{\lambda,\alpha,\beta}^{\mathfrak N,\Df}\, \phi$ and
    $v_\Df\in H^2(0,1)$ defined by \eqref{eq:644}. 
As before we write
\begin{displaymath}
(\LL_\beta^\Nf-\beta\lambda)v_{\Df}=g_\Df \,,
\end{displaymath}
  Let $\tilde{v}_\Df$ be given by \eqref{eq:649}. For
the convenience of the reader we repeat here \eqref{eq:667}
\begin{displaymath}
   \|\tilde{v}_{\Df}\|_2
  \leq C\big(\|\phi\|_{1,2} +\beta^{1/6}|\phi(x_\nu)|+\beta^{-2/3} \|f\|_2 \big)\,,
\end{displaymath}
and \eqref{eq:668}
\begin{multline*}
\|\tilde{v}_\Df^\prime\|_2 \leq C\,
\big[  (\beta^{1/3}\lambda_\beta^{-1/4}+1) \|\phi\|_{1,2}+\beta^{1/6}|\phi(x_\nu)| \\
+ (1 + \mu_{\beta,+}^{1/2}  \beta^{-2/3}  ) \|f\|_2
+\mu_{\beta,+}^{1/2}(\beta^{1/6}|\phi(x_\nu|+ \|\phi\|_{1,2})\big]\,.
\end{multline*}
 For $\mu<\beta^{-1/3-\delta}$ with $\delta < 1/10$, we then obtain using Poincar\'e's inequality
\begin{equation}
\label{eq:696}
  \|\tilde{v}_\Df^\prime\|_2 \leq C\,
\big[  (\beta^{1/3}\lambda_\beta^{-1/4}+\beta^{1/3-\delta/2}) \|\phi\|_{1,2}
+ \|f\|_2
+ \beta^{1/2-\delta/2}|\phi(x_\nu)|\big]\,.
\end{equation}
We next estimate $\|g_\Df\|_2$, beginning by repeating \eqref{eq:671}
 \begin{displaymath}
    \| g_\Df\| \leq C ( \| (U+i\lambda)f\|_2 + |\phi^{\prime\prime}(1)|\,  \| (U+i\lambda) \hat{g}\|  + \|  (U^{\prime\prime}\phi)^{\prime\prime} \| + 
  \|  \tilde{v}_\Df^\prime\| + \| \tilde{v}_\Df \|)\,.
  \end{displaymath}
By \eqref{eq:656}, \eqref{eq:660},  \eqref{eq:669}, \eqref{eq:696}, and \eqref{eq:667}
 it holds,  that
\begin{multline*}
   \|g_\Df\|_2\leq C\big[ (1+|\lambda|) \, \|f\|_2\\
  +(\lambda_\beta^{-1/4}\beta^{1/3}+\beta^{1/3-\delta/2} + \lambda_\beta^{1/4}
  \beta^{1/6})\|\phi\|_{1,2}+\beta^{1/2-\delta/2}|\phi(x_\nu)|) \big] \,,
\end{multline*}
which is precisely \eqref{eq:695}.\\

{\em Step 2: We estimate $\|v_\Df\|_{1,q}$ and $\|v_{\Df}\|_1$.}\\

\noindent By \eqref{eq:695} it holds  for $\mu \geq -1$ that
\begin{equation}
\label{eq:697}
   \|g_\Df\|_2\leq C\, \big[  \, \|f\|_2\\
  +\beta^{1/3} \|\phi\|_{1,2}+ \beta^{1/2-\delta/2}|\phi(x_\nu)|) \big] \,. 
\end{equation}
By (\ref{eq:287}b)  we have, for any $1<q<2$ and given that
$|\nu|<\beta^{-1/5+\delta}$ and  $\delta<1/5$, 
\begin{equation*}
  \|v_\Df\|_{1,q} \leq
  C\beta^{-\frac{2+q}{6q}}\|g_\Df\|_2\,.
\end{equation*}
Hence by \eqref{eq:697} we obtain  for $\mu \geq -1$ \,
\begin{equation}
\label{eq:698}
  \|v_\Df\|_{1,q} \leq
  C(\beta^{-\frac{2+q}{6q}}\|f\|_2+\beta^{-\frac{2-q}{6q}}\|\phi\|_{1,2}+
    \beta^{\delta_1(q)} |\phi(x_\nu)|]) \,,
\end{equation}
where $\delta_1(q)=(q-1)/3q-\delta/2 \,.$ \\
Furthermore, we have by   \eqref{eq:403} and \eqref{eq:697},  for
  $\mu\geq -1$, 
\begin{equation}\label{eq:699}
    \|v_\Df\|_1 \leq C\beta^{-5/6}\|g_\Df\|_2\leq \widehat
    C\big(\beta^{-5/6}\|f\|_2+\beta^{-1/2}\|\phi\|_{1,2} + \beta^{-1/3 -\delta/2}|\phi(x_\nu)|]\big) \,.
\end{equation} 
For $\mu\leq-1$ we use \eqref{eq:363}
 to obtain
  \begin{displaymath}
    \|v_\Df\|_1 \leq \|v_\Df\|_2\leq \frac{C}{|\lambda|\beta}\|g_\Df\|_2\,.
  \end{displaymath}
 We then employ \eqref{eq:695} which implies,  for $\delta \leq 1/10$,   
  \begin{equation*}
   \|g_\Df\|_2\leq C\big[ (1+|\lambda|) \, \|f\|_2
  +(\beta^{1/3-\delta/2} + |\lambda|^{1/4} 
  \beta^{1/4})\|\phi\|_{1,2}+ \beta^{1/2-\delta/2}|\phi(x_\nu)|) \big] \,. 
\end{equation*}
and hence,
   \begin{equation}
\label{eq:700}
    \|v_\Df\|_1 \leq  C\big[
    \, \beta^{-1} \|f\|_2
  +\beta^{-2/3 -\delta/2}\|\phi\|_{1,2}+\beta^{-1/2-\delta/2}|\phi(x_\nu)|\big] \,. 
  \end{equation}
 Combining \eqref{eq:699} and \eqref{eq:700}  yields 
\begin{equation}\label{eq:701}
    \|v_\Df\|_1 \leq C \, \big(\beta^{-5/6}\|f\|_2+\beta^{-1/2}\|\phi\|_{1,2} + \beta^{-1/3 -\delta/2}|\phi(x_\nu)|]\big) \,.
\end{equation} 

{\em Step 3: We prove \eqref{eq:694}.} \\
 
 \noindent We continue as in the proof of Proposition \ref{lem:right-curve}. Recall from \eqref{eq:644} that
 \begin{equation*}
    v_\Df =\A_{\lambda,\alpha}\phi + (U+i\lambda)\phi^{\prime\prime}(1)\hat{\psi} \,.
  \end{equation*}
Set 
\begin{equation}\label{eq:decabcd}
  \phi = \phi_\Df + \check{\phi} \,,
\end{equation}
where 
\begin{displaymath}
  \phi_\Df =\A_{\lambda,\alpha}^{-1}v_\Df \quad \mbox{ and } \quad
 \check{\phi}=-\A_{\lambda,\alpha}^{-1}\big([U+i\lambda]\phi^{\prime\prime}(1)\hat{\psi} \big) \,. 
\end{displaymath}

{\em Step 3a: We estimate  $\|\check{\phi}\|_{1,2}$.}\\

\noindent Note that by \eqref{eq:616}, we have
\begin{equation}\label{eq:702}
\check \phi =-\phi^{\prime\prime}(1) \phi_{\lambda,\beta,\alpha}\,.
\end{equation}
By \eqref{eq:636}  
there exist $C>0$ and $\lambda_0>0$, and for any  $a_0  >\|U\|_2^{-1} (1+\delta)^{1/2} $, $\beta(a_0 )>0$
such that, for $-\lambda_0<\mu$,   $ |\nu|\leq \beta^{-1/5+\delta}$,  $\alpha \geq a_0 \beta^{-1/10+ \delta/2}$ and $\beta \geq
\beta(a_0 )$, we have
\begin{equation}\label{eq:703}
  \|\check{\phi}\|_{1,2}\leq
  C a_0^{-2}   \beta^{-1/3} \lambda_\beta^{-1/2}|\phi^{\prime\prime}(1)| \,.
\end{equation}
 For $\mu\leq-\lambda_0<0$ we obtain from \eqref{eq:169} applied to the
   pair $(\check \phi,  (U+i\lambda)\phi^{\prime\prime}(1)\hat{\psi}) $ 
  \begin{displaymath}
      \|\check{\phi}^\prime\|_2^2\leq C|\langle\check{\phi},\hat{\psi}\rangle\,\phi^{\prime\prime}(1)| \leq  C \,  \|\check{\phi}^\prime\|_2\|(1-x)^{1/2}\hat{\psi}\|_1\,|\phi^{\prime\prime}(1)|\,,
  \end{displaymath}
and hence by \eqref{eq:614} (with $s=1/2$) for the first inequality,
we conclude for the second inequality that there exists $\beta(a_0 )>0$
such that for $\beta \geq \beta(a_0 )$ 
\begin{displaymath}
      \|\check{\phi}^\prime\|_2\leq C\,  \lambda_\beta^{-3/4} \beta^{-1/2} |\phi^{\prime\prime}(1)| \leq
 \widehat  C \,  a_0^{-2}   \beta^{-1/3} \lambda_\beta^{-1/2}|\phi^{\prime\prime}(1)| \,.
\end{displaymath}
Hence, \eqref{eq:703} holds true for any $ \mu<\beta^{-1/3+\delta}$
and $|\nu| \leq \beta^{-1/5+\delta}$.  From \eqref{eq:656} we then deduce
\begin{displaymath}
  \|\check{\phi}\|_{1,2}\leq
  C a_0^{-2}   \beta^{-1/3} ( \beta^{1/3} \|\phi\|_{1,2} + \beta^{-1/2} \|f\|_2 )\leq
  \widehat C (\beta^{-5/6}\|f\|_2+a_0^{-2}\|\phi\|_{1,2})
\end{displaymath}
Hence, for sufficiently large $a_0 $ we conclude with the aid of
\eqref{eq:decabcd} 
\begin{equation}
\label{eq:704}
   \|\check{\phi}\|_{1,2} \leq C\,\big(\beta^{-5/6}\|f\|_2+a_0^{-2}\|\phi_\Df\|_{1,2}\big)\,.
\end{equation}

{\em Step 3b: We estimate  $|\check{\phi}(x_\nu)|$.}\\

\noindent Using \eqref{eq:702}, the decomposition 
\begin{displaymath}
\phi_{\lambda,\beta,\alpha} =(\phi_{\lambda,\beta,\alpha} - c_\parallel(\lambda,\beta,\alpha)U) + c_\parallel(\lambda,\beta,\alpha)U\,,
\end{displaymath}
where
 \begin{equation*}
 c_{\parallel}(\lambda,\beta,\alpha)   = \frac{\langle U,\phi_{\lambda,\beta,\alpha} \rangle }{\|U\|_2^2}\,,
\end{equation*} 
 and H\"older
inequality, we may write
\begin{equation}
\label{eq:705}
  |\check{\phi}(x_\nu)|\leq C\Big(|\nu|\, |c_\parallel(\lambda,\beta,\alpha) |+ |\nu|^{1/2}
  \Big\|\phi_{\lambda,\beta,\alpha} - c_\parallel(\lambda,\beta,\alpha)  U\Big\|_{1,2}\Big) \, |\phi^{\prime\prime}(1)| \,.
\end{equation}
Then, we obtain  by \eqref{eq:636}  and the fact that
$\alpha\geq a_0 \beta^{-1/10+\delta/2}$ 
 \begin{equation}\label{eq:706}
  |\check{\phi}(x_\nu)|\leq 
  C\Big[|\nu|^{1/2}\lambda_\beta^{-3/4}\beta^{-1/2} +
  \frac{|\mu|\nu^{1/2} +|\nu|}{  a_0^2 \beta^{-1/5+\delta}+|\mu|}\big(\lambda_\beta^{-1}\beta^{-2/3}+\lambda_\beta^{-1/2}|\lambda|\beta^{-1/3}\big) \Big] |\phi^{\prime\prime}(1)|\,. 
\end{equation}
Using  \eqref{eq:656} and the fact that $ |\lambda| \leq \beta^{-1/3} \lambda_\beta$ we obtain
\begin{equation}\label{eq:707}
 |\check{\phi}(x_\nu)|\leq C  \big(a_0^{-2} \beta^{-1/3} \lambda_\beta +|\nu|^{1/2}\lambda_\beta^{-1/4}\beta^{-1/6}\big) \big(\|\phi\|_{1,2}+\beta^{-5/6}\|f\|_2\big) \,. 
\end{equation}

 We now consider three different cases.
\begin{itemize}
\item For $-\beta^{-1/5+\delta} < \mu <\beta^{-1/3+\delta}$, we have $\lambda_\beta\lesssim
  \beta^{\delta}$ and since $|\nu| \leq |\lambda|$, we deduce from \eqref{eq:707}
\begin{equation} \label{eq:708}
 |\check{\phi}(x_\nu)|
     \leq C    \beta^{-1/5 + \delta} \big(\|\phi\|_{1,2}+\beta^{-5/6}\|f\|_2\big) \,.
\end{equation}
\item 
 For $-\beta^{-1/10+\delta/2}<\mu<-\beta^{-1/5+\delta}$ we have, since $|\nu| \leq |\mu|$,
\begin{displaymath} 
   \frac{|\mu|\,|\nu|^{1/2} +|\nu|}{  a_0^2 \beta^{-1/5+\delta}+|\mu|}|\lambda|\leq
   C(|\mu|\,|\nu|^{1/2} +|\nu|) \leq \widehat C\beta^{-1/5+\delta}
\end{displaymath}
and hence, as $\lambda_\beta \sim |\lambda| \beta^{1/3}$, we can conclude \eqref{eq:708}
in this case as well.

 \item 
 Finally, for $\mu\leq-\beta^{-1/10+\delta/2}$,  we use \eqref{eq:168} with $v=
 (U+i\lambda) \hat \psi$, (\ref{eq:673}c), and \eqref{eq:656} to obtain
 that 
 \begin{displaymath}
\|\check{\phi}^\prime\|_2\leq
C|\mu|^{-1}\|(1-x)^{1/2}\hat{\psi}\|_1|\phi^{\prime\prime}(1)| \leq \widehat C|\lambda|^{-5/4}\beta^{-1/4}\big(\|\phi\|_{1,2}+\beta^{-5/6}\|f\|_2\big)\,,
\end{displaymath}
which  implies
\begin{equation}
\label{eq:709}
    |\check{\phi}(x_\nu)|\leq C|\nu|^{1/2}\|\check{\phi}^\prime\|_2\leq \widehat C\beta^{-9/40-\delta/8}\big(\|\phi\|_{1,2}+\beta^{-5/6}\|f\|_2\big)\,.
\end{equation}
\end{itemize}
Combining \eqref{eq:709} with \eqref{eq:708} which holds in the two
first cases  yields for all \break $\mu<\beta^{-1/3-\delta}$
\begin{equation} \label{eq:710}
 |\check{\phi}(x_\nu)|
     \leq C    \beta^{-1/5 + \delta} \big(\|\phi\|_{1,2}+\beta^{-5/6}\|f\|_2\big) \,.
\end{equation}
Then, by \eqref{eq:decabcd} and  \eqref{eq:710}   it holds that
\begin{equation}
\label{eq:711}
   |\check{\phi}(x_\nu)|\leq C\beta^{-1/5+\delta}\big(\beta^{-5/6}\|f\|_2+\|\phi_\Df\|_{1,2}\big)\,.   
\end{equation}

 {\em Step 3c: We prove \eqref{eq:694} for $\mu > -\beta^{-\delta}$.}\\

\noindent From \eqref{eq:698} and \eqref{eq:711} we get, for $\delta \leq 1/15$ and
$\mu > -1$, that
\begin{equation*}
  \|v_\Df\|_{1,q} \leq C \,  \big(  \beta^{-\frac{2+q}{6q}}\|f\|_2+\beta^{-\frac{2-q}{6q}}\|\phi\|_{1,2}+ \beta^{\delta_1(q)} |\phi_\Df(x_\nu)|\big) \,.
\end{equation*}
Using \eqref{eq:704} we then obtain that
\begin{equation*}
  \|v_\Df\|_{1,q} \leq C \big( \beta^{-\frac{2+q}{6q}}\|f\|_2+\beta^{-\frac{2-q}{6q}}\|\phi_\Df\|_{1,2}+ \beta^{\delta_1(q)} |\phi_\Df(x_\nu)|\big) \,.
\end{equation*}
As
 \begin{displaymath}
    -\frac{2-q}{6q}+\frac{1}{6}=\frac{1}{3}-\frac{1}{3q}=\delta_1(q)+\delta/2 \,,
  \end{displaymath}
  we can finally  conclude, for $\mu>-1$, that
\begin{equation}
\label{eq:712}  
\|v_\Df\|_{1,q} 
  \leq C\, \big(\beta^{-\frac{2+q}{6q}}\|f\|_2+ \beta^{-\frac{2-q}{6q}} [\|\phi_\Df\|_{1,2} +\beta^{1/6}|\phi_\Df(x_\nu)|]\big) \,.
\end{equation}
From   \eqref{eq:701}, \eqref{eq:704}, and \eqref{eq:711}  we
obtain 
\begin{equation}
\label{eq:713}
    \|v_\Df\|_1 \leq
    C\,\big(\beta^{-5/6}\|f\|_2+\beta^{-1/2}\|\phi_\Df\|_{1,2}+ \beta^{-1/3-\delta/2} |\phi_\Df(x_\nu)|\big) \,.
\end{equation}
As in the proof of  \eqref{eq:705} we then write 
\begin{equation}\label{eq:714}
   |\phi_\Df (x_\nu)|\leq C\,\big(|\nu|\, |c_\parallel^\Df|+ |\nu|^{1/2} \|\phi_\Df- c_\parallel^\Df U\|_{1,2} \big)\,,
\end{equation}
where
\begin{displaymath}
  c_\parallel^\Df = \langle U,\phi_\Df\rangle /\|U\|_2^2\,.
\end{displaymath}
We may then conclude from (\ref{eq:52}a) applied to the pair $(\phi_\Df,v_\Df)$  that for all
  $-\beta^{-\delta}<\mu$ it holds that
\begin{equation}\label{eq:715}
   |\phi_\Df(x_\nu)|\leq C\,\big(\|v_\Df\|_1+ |\lambda| N_1(v_\Df,\lambda)+ |\nu|^{1/2}\|v_\Df\|_{1,q} \big) \,.
\end{equation}
where $ N_1(v_\Df,\lambda)$ is defined in \eqref{eq:50}.\
By \eqref{eq:85}
(and the fact that $v_\Df(1)=0$) we then obtain that
\begin{equation}\label{eq:716} 
   |\phi_\Df(x_\nu)|\leq C\,\big( \|v_\Df\|_1+ \beta^{-\delta} \|v_\Df\|_{1,q} \big)\,.
\end{equation}
Substituting \eqref{eq:712} and \eqref{eq:713} 
into \eqref{eq:716} then yields, for $\mu>-\beta^{-\delta}$, sufficiently
large $\beta_0$,  and $1<q<(1-3\delta)^{-1}$  
\begin{equation}
\label{eq:717}
   |\phi_\Df(x_\nu)|\leq    C\beta^{-\delta}\, \big(\beta^{-\frac{2+q}{6q}}\|f\|_2+
   \beta^{-\frac{2-q}{6q}} \|\phi_\Df\|_{1,2}\big) \,. 
\end{equation}

Substituting \eqref{eq:717} into \eqref{eq:712} then leads to
\begin{equation}
\label{eq:718}
  \|v_\Df\|_{1,q} 
  \leq C\, \big(\beta^{-\frac{2+q}{6q}}\|f\|_2+
  \beta^{-\frac{2-q}{6q}}\|\phi_\Df\|_{1,2} \big) \,.
\end{equation}
Similarly by substituting \eqref{eq:717} into \eqref{eq:713} we obtain
that for $\mu>-\beta^{-\delta}$ and $1<q<(1-3\delta)^{-1}$
 \begin{equation}
 \label{eq:719}
    \|v_\Df\|_1 \leq
     C(\beta^{-5/6}\|f\|_2+\beta^{-1/2}\|\phi_\Df\|_{1,2}) \,.
 \end{equation}
For sufficiently large $a_0$ and $|\nu|\leq\beta^{-1/5+\delta}$ we have 
\begin{displaymath}
  \frac{1+C|\lambda|^2\log
 |\lambda|^{-1}}{|\alpha^2\|U\|_2^2+i\lambda|}\leq \widetilde C\frac{1}{|\alpha^2\|U\|_2^2-\nu|}\leq \widehat C\beta^{1/5-\delta} \,,
\end{displaymath}
and (recall that $|\nu|<|\alpha^2\|U\|_2^2-\nu|$)
\begin{displaymath}
   \frac{|\lambda|(1+C|\lambda|^2\log
 |\lambda|^{-1}}{|\alpha^2\|U\|_2^2+i\lambda|}\leq \widetilde C\frac{|\nu|+|\mu|}{|\alpha^2\|U\|_2^2-\nu|+|\mu|}\leq \widehat C \,.
\end{displaymath}
 Hence, we obtain from  (\ref{eq:52}a) and \eqref{eq:85} that
 \begin{displaymath}
    |c_\parallel^\Df |\leq C\, \big[\beta^{1/5-\delta}\|v_\Df\|_1+\|v_\Df\|_{1,q} \big] \,,
 \end{displaymath}
and from (\ref{eq:52}b) we obtain that 
\begin{displaymath}
   \|\phi_\Df- c_\parallel^\Df U\|_{1,2} \leq C[\|v_\Df\|_1+\|v_\Df\|_{1,q} ] \,. 
   \end{displaymath}
Consequently, by \eqref{eq:718}   and  \eqref{eq:719} it holds that
\begin{displaymath}
  \|\phi_\Df\|_{1,2}\leq C(\|v_\Df\|_{1,q}+\beta^{1/5}\|v_\Df\|_1) \leq \widehat C\, \big(\beta^{-\frac{2+q}{6q}}\|f\|_2+\beta^{-\frac{2-q}{6q}}\|\phi_\Df\|_{1,2}\big) \,.
\end{displaymath}
It follows that, for sufficiently large $\beta_0$,
\begin{equation}\label{eq:720}
 \|\phi_\Df\|_{1,2}\leq C \beta^{-\frac{2+q}{6q}}\|f\|_2\,.
\end{equation}
Combining \eqref{eq:720}  with \eqref{eq:704} 
gives (note that $5/6>(2+q)/(6q)$) yields
\begin{equation}\label{eq:721}
   \|\check{\phi}\|_{1,2} \leq C\, \beta^{-\frac{2+q}{6q}} \|f\|_2 \,.
\end{equation}
As $\phi=\phi_\Df + \check \phi$, we can deduce from \eqref{eq:720} and
\eqref{eq:721} that for any $\delta \in (0,1/15)$ and $q \in (1,
(1-3\delta)^{-1})$ and $\mu>-\beta^{-\delta}$ we have
\begin{equation}\label{eq:722}
   \|\phi \|_{1,2} \leq C\, \beta^{-\frac{2+q}{6q}}\|f\|_2\leq C\beta^{1/2-\delta}\,.
\end{equation}

{\em Step 3d. We prove   \eqref{eq:694}  for  $\mu\leq-\beta^{-\delta}$.}\\

\noindent For $\mu\leq-\beta^{-\delta}$ we use \eqref{eq:168}, Sobolev embeddings, and
Poincar\'e's inequality to obtain that
\begin{displaymath}  
    \|\phi_\Df\|_{1,2}\leq \frac{C}{|\mu|}\|\phi_\Df\|_\infty^{1/2}\|v_\Df\|_1^{1/2}\leq
   \widehat  C \, \beta^{\delta}\|\phi^\prime_\Df \|_2^{1/2}\|v_\Df\|_1^{1/2} \,,
\end{displaymath}
which implies
\begin{equation}\label{eq:723}
    \|\phi_\Df\|_{1,2} \leq C\beta^{2\delta}\|v_\Df\|_1\,.
\end{equation}
Consequently, by \eqref{eq:701}, \eqref{eq:723},  and Sobolev embeddings we obtain
that  
\begin{displaymath}
    \|\phi_\Df\|_{1,2} \leq    C \beta^{2\delta}\,
    \big(\beta^{-5/6}\|f\|_2+ \beta^{-1/3 -\delta/2} \|\phi\|_{1,2} \big) \,.
\end{displaymath}
 Making use of \eqref{eq:704} we establish that for sufficiently large
 $a_0$ and $\beta_0$
\begin{displaymath}
    \|\phi\|_{1,2} \leq    C \, \beta^{2\delta-5/6}\|f\|_2 \,,
\end{displaymath}
which, together with \eqref{eq:722}, yields   \eqref{eq:694} for $\delta<1/15$.
\end{proof}

\subsection{Resolvent estimates for intermediate $\alpha$}
\label{sec:5.3a}

In this subsection we provide inverse estimates for $\B_{\lambda,\alpha,\beta}$ for
$1\ll\alpha\ll\beta^{1/3}$. Let ${\mathfrak z}_\alpha$ be given by \eqref{eq:603}.
Since $\|{\mathfrak z}_\alpha^\prime\|_2\ll\beta^{1/6}$ in this subsection we may
conclude by \eqref{eq:441} that ${\mathfrak z}_\alpha\in{\mathfrak U}_1$.
Consequently, we may still use \eqref{eq:445} in this subsection to
estimate $\phi^{\prime\prime}(1)$.  Furthermore, we can use the fact that
$\alpha\gg1$ to obtain a much simpler proof than in the previous subsection
(which is valid only for bounded values of $\alpha$).
\begin{proposition}
  \label{lem:intermediate-alpha}
  Let $U\in
C^4([0,1])$ satisfy \eqref{eq:10}  and  $\nu_2 < U(0)$  denote
  a positive constant.  There
exist 
  $C>0$, $\Upsilon >0$, $\beta_0>0$, $\alpha_0>0$, and $\alpha_1>0$ such that for all $\beta\geq
  \beta_0$, it holds that
  \begin{equation}
\label{eq:724}
      \sup_{
        \begin{subarray}{c}
          \alpha_0\leq \alpha\leq\alpha_1\beta^{1/3}\\
     \Re\lambda \leq  \Upsilon  \beta^{-1/3}\\
         |\Im\lambda|\leq \nu_2
           \end{subarray}}
\big\|(\B_{\lambda,\alpha,\beta}^{\mathfrak N,\Df})^{-1}\big\|+
      \Big\|\frac{d}{dx}\, (\B_{\lambda,\alpha,\beta}^{\mathfrak
        N,\Df})^{-1}\Big\|\leq  C \beta^{-5/6}\,.
      \end{equation}
\end{proposition}
\begin{proof}
   Let $f\in L^2(0,1)$, $\phi\in D(\B_{\lambda,\alpha,\beta}^{\Nf,\Df})$ satisfy
\begin{displaymath}
\B_{\lambda,\alpha,\beta}^{\Nf,\Df} \phi=f \,.
\end{displaymath} 
We first recall the definition of $\tilde v_\Df$  from \eqref{eq:649}
\begin{equation}\label{eq:649rem}
  \tilde{v}_\Df=-\phi^{\prime\prime}+\alpha^2\phi +\phi^{\prime\prime}(1)\hat{\psi}\,,
\end{equation}
and rewrite \eqref{eq:650} in the form 
\begin{equation}
\label{eq:725}
  (\LL_\beta^{\Nf,\Df}-\beta\lambda)\tilde{v}_\Df= f+ i\beta U^{\prime\prime}\phi+\phi^{\prime\prime}(1)\hat{g} \,,
\end{equation}
where $\hat{g}$ is given by \eqref{eq:648}.\\
 By \eqref{eq:403} (which is applicable for $|\nu|\leq \nu_2<U(0)$) it holds that 
\begin{equation*}
  \| (\LL_\beta^{\Nf,\Df} -\beta\lambda)^{-1}(f +\phi^{\prime\prime}(1)\hat{g})    \|_1 \leq C \beta^{-5/6} \| (f +\phi^{\prime\prime}(1)\hat{g}) \|_2 \,.
\end{equation*}
Furthermore, by \eqref{eq:421} applied with $f$ replaced by $U^{\prime\prime}\phi$ 
it holds that
\begin{displaymath}
   \Big\| (\LL_\beta^{\Nf,\Df} -\beta\lambda)^{-1}  ( U^{\prime\prime}
   \phi)  {
     +i\frac{U^{\prime\prime}(x_\nu)\phi(x_\nu)}{\beta(U-\nu+i\mf) }}\Big\|_1 \leq  C \beta^{-1}\| U^{\prime\prime} \phi\|_{1,2} \leq \widehat C 
   \beta^{-1}\|  \phi\|_{1,2} \,,
\end{displaymath}
where
\begin{displaymath}
  \mf = -\max(-\mu,x_\nu^{2/3}\beta^{-1/3})\,.
\end{displaymath}
 Hence, 
 \begin{equation}\label{eq:726-a}
   \Big\|\tilde{v}_\Df {-
     \frac{U^{\prime\prime}(x_\nu)\phi(x_\nu)}{(U-\nu+i\mf)}}\Big\|_1\leq 
   C\, \big(\beta^{-5/6}\|f\|_2+\|\phi\|_{1,2}+\beta^{-5/6}|\phi^{\prime\prime}(1)|\,\|\hat{g}\|_2\big)\,.
 \end{equation}

 We next use \eqref{eq:445} together with \eqref{eq:655} to obtain, as
 in \eqref{eq:656}, that
\begin{equation}
\label{eq:919}
  |\phi^{\prime\prime}(1)|
\leq C \,
\lambda_\beta^{1/2}\, \big(\beta^{1/3}\big[\|\phi\|_{1,2}+|\phi(x_\nu)|\log \beta \big] +\beta^{-1/2}\|f\|_2\big)\,.
\end{equation}
For $\alpha \leq \alpha_1\beta^{1/3}$ we may then use  \eqref{eq:919},
   which, combined with \eqref{eq:660}, yields, as 
   \begin{equation}\label{eq:726cba}
   |\phi(x_\nu)|\leq|\nu|^{1/2}\|\phi^\prime\|_2\,,
   \end{equation}
\begin{equation}
\label{eq:726}
|\phi^{\prime\prime}(1)|\,\|\hat{g}\|_2\leq C\lambda_\beta^{-3/4}
\big(\beta^{1/2}[1+|\nu|^{1/2}\log\beta]\|\phi\|_{1,2}+\beta^{-1/3}\|f\|_2\big) \,.
\end{equation}
Consequently, by substituting \eqref{eq:726} into \eqref{eq:726-a},
we obtain
\begin{equation}
\label{eq:727}
  \Big\|\tilde{v}_\Df  -
     \frac{U^{\prime\prime}(x_\nu)\phi(x_\nu)}{U-\nu+i\mf}\Big\|_1\leq 
   C\, \big(\beta^{-5/6}\|f\|_2+\|\phi\|_{1,2}\big)\,.
\end{equation}
 Taking the scalar product of \eqref{eq:649rem} with $\phi$, and
 integrating by parts gives
\begin{equation}
\label{eq:728}
  \Big\langle\phi,\tilde{v}_\Df
  -\frac{U^{\prime\prime}(x_\nu)\phi(x_\nu)}{U-\nu+i\mf}\Big\rangle 
     =\|\phi^\prime\|_2^2 +\alpha^2\|\phi\|_2^2 +\phi^{\prime\prime}(1)\langle\phi,\hat{\psi}\rangle 
  - \Big\langle\phi, \frac{U^{\prime\prime}(x_\nu)\phi(x_\nu)}{U-\nu +i \mf}\Big\rangle 
\end{equation}
Using \eqref{eq:727} we then conclude
\begin{equation}
\label{eq:729} 
   \Big|\Big\langle\phi,\tilde{v}_\Df - 
     \frac{U^{\prime\prime}(x_\nu)\phi(x_\nu)}{U-\nu+i\mf}\Big\rangle\Big|\leq C\, \big(\beta^{-5/6}\|f\|_2+\|\phi\|_{1,2}\big)\,\|\phi\|_\infty  
\end{equation}

We next use \eqref{eq:614}, with $s=1/2$, together with \eqref{eq:919}
and a Sobolev inequality to obtain that
\begin{displaymath}
 |\phi^{\prime\prime}(1)\langle\phi,\hat{\psi}\rangle| \leq C \,
\lambda_\beta^{-1/4}\, \big(\beta^{-1/6}[\|\phi\|_{1,2}+
|\phi(x_\nu)|\log\beta]+\beta^{-1}\|f\|_2\big)\|\phi^\prime\|_2\,.
\end{displaymath}
Using \eqref{eq:726cba}, as $\lambda_\beta^{-1/4} \nu^{1/2} \log \beta \leq 1$ for
$\beta \geq \beta_0$ with sufficiently large $\beta_0$,
  we can conclude that
\begin{equation}
  \label{eq:730}
|\phi^{\prime\prime}(1)\langle\phi,\hat{\psi}\rangle| \leq C \,\big(\beta^{-1/6}\|\phi\|_{1,2}^2+\beta^{-11/6}\|f\|_2^2\big)\,.
\end{equation}
For the last term on the right-hand side of \eqref{eq:728} we write 
\begin{displaymath}
    \Big\langle\phi,\frac{U^{\prime\prime}(x_\nu)\phi(x_\nu)}{U-\nu+i\mf}\Big\rangle  =  
\Big\langle\phi,\frac{U^{\prime\prime}(x_\nu)\phi(x_\nu)}{U-\nu+i\mf}\Big\rangle_{L^2(0,x_\nu/2)}
+ \Big\langle\phi,\frac{U^{\prime\prime}(x_\nu)\phi(x_\nu)}{U-\nu+i\mf}\Big\rangle_{L^2(x_\nu/2,1)}   
\end{displaymath}
For the first term on the right-hand side we have, since $|\nu| \leq \nu_2$, 
\begin{equation}
\label{eq:731}
  \Big|\Big\langle\phi,\frac{U^{\prime\prime}(x_\nu)\phi(x_\nu)}{U-\nu+i\mf}\Big\rangle_{L^2(0,x_\nu/2)}\Big|\leq  C\,  \|\phi\|_2\|\phi\|_\infty \,.
\end{equation}
For the second term we use integration by parts to obtain
\begin{multline*}
  \Big\langle\phi,\frac{U^{\prime\prime}(x_\nu)\phi(x_\nu)}{U-\nu+i\mf}\Big\rangle_{L^2(x_\nu/2,1)}   \\ =
U^{\prime\prime}(x_\nu)\phi(x_\nu)\frac{\phi}{U^\prime}\log(U-\nu+i\mf)\Big|_{x=x_\nu/2}\\
-U^{\prime\prime}(x_\nu)\phi(x_\nu)\Big\langle\Big(\frac{\phi}{U^\prime}\Big)^\prime, \log(U-\nu+i\mf)\Big\rangle_{L^2(x_\nu/2,1)}   \,,
\end{multline*}
from which we readily obtain
\begin{displaymath}
  \Big|\Big\langle\phi,\frac{U^{\prime\prime}(x_\nu)\phi(x_\nu)}{U-\nu+i\mf}\Big\rangle_{L^2(x_\nu/2,1)}\Big| \leq C \|\phi\|_{1,2} |\phi(x_\nu)| \leq C \|\phi\|_{1,2}\|\phi\|_\infty  \,. 
\end{displaymath}
In conjunction with \eqref{eq:731} the above inequality yields 
\begin{displaymath}
   \Big|\Big\langle\phi,\frac{U^{\prime\prime}(x_\nu)\phi(x_\nu)}{U-\nu+i\mf}\Big\rangle\Big| \leq C \|\phi\|_{1,2}\|\phi\|_\infty \,. 
\end{displaymath}
Substituting the above, together with \eqref{eq:730} and
\eqref{eq:729} into \eqref{eq:728} yields 
\begin{displaymath}
   \|\phi^\prime\|_2^2 +\alpha^2\|\phi\|_2^2 \leq C\, \big( \|\phi\|_{1,2} +
   \beta^{-5/6}\|f\|_2\big)\|\phi\|_\infty+\beta^{-1/6}\|\phi\|_{1,2}^2\,.
\end{displaymath}
Since $\phi(1)=0$ we have
$\|\phi\|_\infty^2\leq2\|\phi^\prime\|_2\|\phi\|_2$ and hence, for any $\epsilon>0$ there
exists $C_\epsilon>0$ such that
\begin{displaymath}
  \| \phi\|_\infty^2 \leq \epsilon \| \phi^\prime\|_2^2 + C_\epsilon \| \phi\|^2\,.
\end{displaymath}
By choosing sufficiently small $\epsilon$ and sufficiently large $\beta_0$ we can
then conclude, with the aid of Poincar\'e's inequality,
\begin{displaymath}
  \|\phi^\prime\|_2^2 +\alpha^2\|\phi\|_2^2 \leq 
 C\,\big(\beta^{-5/3}\|f\|_2^2+ \| \phi\|_2^2\big)\,.
\end{displaymath}
We then obtain for sufficiently large $\alpha_0$ and $\beta_0$ the existence
of $C>0$ such that, under the conditions of the proposition,
\begin{displaymath}
  \|\phi^\prime\|_2^2 +\alpha^2\|\phi\|_2^2\leq C \beta^{-5/3}\, \|f\|_2^2\,,
\end{displaymath}
from which \eqref{eq:724} readily follows. 
\end{proof}
\subsection{Resolvent estimates for large $\alpha$}
\label{sec:resolv-estim-large-1}
For $\alpha\gtrsim\beta^{1/3}$ we can no longer use the estimates of the
previous subsection, relying on \eqref{eq:445}. In the following we
thus establish  estimates for  the inverse of  the Orr-Sommerfeld operator,
relying on \eqref{eq:604}, which is valid for $\alpha\gtrsim\beta^{1/3}$. 
\begin{proposition}
  \label{lem:large-alpha}
  Let $U\in C^4([0,1])$ satisfying \eqref{eq:10}, and $\alpha_2$ denote a
  positive constant. For any $\Upsilon< \sqrt{-U^{\prime\prime}(0)}/2$ { and any $\hat \Upsilon >0$} there exist $C>0$
  and $\beta_0>0$ such that for all $\beta\geq \beta_0$ it holds that
  \begin{equation}
\label{eq:732}
      \sup_{
        \begin{subarray}{c}
          \alpha_2 \beta^{1/3}\leq\alpha \\
  { \Re\lambda\leq\beta^{-1/3}[\hat{\mu}_m-\hat \Upsilon-\alpha^2\beta^{-2/3}/2]}
           \end{subarray}}
\Big(\big\|(\B_{\lambda,\alpha,\beta}^{\mathfrak N,\Df})^{-1}\big\|+
      \Big\|\frac{d}{dx}\, (\B_{\lambda,\alpha,\beta}^{\mathfrak N,\Df})^{-1}\Big\| \Big)\leq C \beta^{ -5/6} \,.
  \end{equation}
\end{proposition}
\begin{proof}
 Let $ {\mathfrak z} =\mathfrak z_\alpha $ be given by
  \eqref{eq:602}-\eqref{eq:603}. Let $f\in L^2(0,1)$, $\phi\in D(\mathcal B_{\lambda,\alpha}^{\Nf,\Df})$ satisfy
  \begin{displaymath}
      \mathcal B_{\lambda,\alpha}^{\Nf,\Df} \phi = f\,.
  \end{displaymath}
  An integration by parts yields $\langle \mathfrak
  z_\alpha,-\phi^{\prime\prime}+\alpha^2\phi\rangle=0$, and hence we may conclude that
  $-\phi^{\prime\prime}+\alpha^2 \phi \in D (\LL_\beta^{\mathfrak z_\alpha})$. Furthermore, it
  holds that
\begin{displaymath}
  (\LL_\beta^{\mathfrak z_\alpha}-\beta\lambda)(-\phi^{\prime\prime}+\alpha^2\phi)=f+i\beta U^{\prime\prime}\phi \,.
\end{displaymath}
By \eqref{eq:604} we then have  
 \begin{displaymath} 
\|-\phi^{\prime\prime}+\theta^2\beta^{2/3}\phi\|_2\leq
C\,\big(\beta^{1/2}\|\phi\|_2+\beta^{-1/2}\|f\|_2\big)\,,
 \end{displaymath}
where $\theta=\alpha\beta^{-1/3}$. \\
Hence, 
\begin{equation}\label{eq:733}
  \|\phi^\prime\|_2^2 +\theta^2\beta^{2/3}\|\phi\|_2^2\, = \langle -\phi^{\prime\prime}+\theta^2\beta^{2/3}\phi,\phi\rangle  \leq
  C\, \big(\beta^{1/2}\|\phi\|_2^2+ \beta^{-1/2}\|f\|_2\|\phi\|_2 \big)\,.
\end{equation}
As $\theta\geq \alpha_2$, we obtain that for sufficiently large $\beta_0$,
\begin{displaymath} 
\|\phi^\prime\|_2\leq  C \beta^{-5/6}  \|f\|_2   \,.
\end{displaymath}
With the aid of Poincar\'e's inequality we then obtain \eqref{eq:732}. 
\end{proof}
\begin{remark}
  An improved version of \eqref{eq:732} can be obtained by introducing the effect
  of $|U(0)-\nu|$ from \eqref{eq:604}
  \begin{equation}
\label{eq:734}
      \sup_{
        \begin{subarray}{c}
          \alpha_2 \beta^{1/3}\leq\alpha \\
     \Re\lambda \leq   \Upsilon\beta^{-1/2}
           \end{subarray}}
\big\|(\B_{\lambda,\alpha,\beta}^{\mathfrak N,\Df})^{-1}\big\|+
      \Big\|\frac{d}{dx}\, (\B_{\lambda,\alpha,\beta}^{\mathfrak
        N,\Df})^{-1}\Big\|\leq  C \frac{\beta^{-5/6}}{1+|U(0)-\nu|\,\beta^{1/6}} \,.
  \end{equation}
\end{remark}

  \subsection{Resolvent estimates for small $\alpha$}
 \label{sec:5.5}

 We continue by considering for some positive $ \hat \alpha_0$ and $0<\hat
 \nu_0<U(0)$, the zone
 \begin{equation}
\label{eq:hyp5.5}
 0\leq \alpha\leq\hat \alpha_0\beta^{-1/6}\,, \,  |\Im
 \lambda|\leq \hat{\nu}_0\,,\, \Re \lambda \leq \beta^{-1/2}\,.
 \end{equation}
  We begin
 by considering the case $\alpha=0$ and then obtain estimates of
 $(\B_{\lambda,\alpha,\beta}^{\Nf,\Df})^{-1}$ for $0<\alpha\leq \hat \alpha_0\beta^{-1/6}$ by treating it
 as a perturbation of $(\B_{\lambda,0,\beta}^{\Nf,\Df})^{-1}$ . More precisely, we introduce
 the set
 \begin{displaymath}
\mathfrak W(\beta,\hat{\nu}_0) :=\{ \lambda \in \mathbb C\,:\, |\Im\lambda| \leq \hat{\nu}_0, \;
\mbox{ and } \mu \leq \beta^{-1/2}\}\,. 
\end{displaymath}
and prove the following:
\begin{proposition}
\label{lem:zero-alpha}
Let $U\in C^4([0,1])$ satisfy \eqref{eq:10}.   There exist $C>0$,
  $\hat \nu_0>0$ and $\beta_0>0$ such that 
for all $\beta\geq \beta_0$ and $ \lambda\in\mathfrak W(\beta,\hat{\nu}_0)$  it holds that
  \begin{equation}
\label{eq:735}
\Big\|(\B_{\lambda,0,\beta}^{\Nf,\Df})^{-1}\Big\|+
      \Big\|\frac{d}{dx}\, (\B_{\lambda,0,\beta}^{\Nf,\Df})^{-1}\Big\|+ 
      \lambda_\beta^{-1/2}\Big\|\frac{d^2}{dx^2}\,
      (\B_{\lambda,0,\beta}^{\Nf,\Df})^{-1}\Big\| \leq C\, \beta^{-2/3} \,. 
  \end{equation}
\end{proposition}
\begin{proof}~\\

{\it Step 1: Preliminaries.}\\

 Let $(\phi,f)\in D(\B_{\lambda,0,\beta}^{\Nf,\Df})\times L^2(0,1)$ satisfy
$\B_{\lambda,0,\beta}^{\Nf,\Df}\phi=f$. Setting $\alpha=0$ in \eqref{eq:644} yields
\begin{equation}\label{eq:736}
  v_\Df =-(U+i\lambda)\phi^{\prime\prime}+U^{\prime\prime}\phi +  (U+i\lambda)\phi^{\prime\prime}(1)\hat{\psi}\,.
\end{equation}

Note that $v_\Df(1)=0$ and hence
\begin{equation}
\label{eq:737}
  (\LL_\beta^\Nf-\beta\lambda)v_\Df= g_\Df \,,
\end{equation}
 where  $g_\Df$ is given by   \eqref{eq:647}, which we recall
  here for the benefit of the reader in the equivalent form 
  \begin{equation}\label{eq:738}
     g_\Df + (U+i\lambda)f =(U+i\lambda)  \phi^{\prime\prime}(1)\hat{g}  -   (U^{\prime\prime}\phi)^{\prime\prime} 
  -2U^\prime \tilde{v}_\Df^\prime-U^{\prime\prime} \tilde{v}_\Df\,.
  \end{equation}
In \eqref{eq:738}, $\tilde{v}_\Df$ is given by setting $\alpha=0$ in 
\eqref{eq:649}, i.e., 
\begin{equation}\label{eq:739} 
  \tilde{v}_\Df=  - \phi^{\prime\prime} + \phi^{\prime\prime}(1)\hat{\psi}\,. 
\end{equation}

   {\it Step 2: We estimate $|\phi^{\prime\prime}(1)|$.} \\
   
   \noindent 
    Let $(\phi,f)\in D(\B_{\lambda,0,\beta}^{\mathfrak N,\Df})\times L^2(0,1)$ satisfy
    $\B_{\lambda,0,\beta}^{\mathfrak N,\Df}\phi=f$. An integration by parts yields
    \begin{multline}\label{eq:740}
      \|(U^{\prime\prime})^{-1/2}\phi^{(3)}\|_2^2 = -  \Re\langle(U^{\prime\prime})^{-1} \phi^{\prime\prime},\B_{\lambda,0,\beta}\phi\rangle-
      \frac{1}{U^{\prime\prime}(1)}\Re (\bar{\phi}^{\prime\prime}(1)\phi^{(3)}(1))  \\
      - \Re\langle[(U^{\prime\prime})^{-1}]^\prime\phi^{\prime\prime},\phi^{(3)}\rangle +
      \mu \beta \|(U^{\prime\prime})^{-1/2}\phi^{\prime\prime}\|_2^2 \,.
    \end{multline}
    To estimate the second term on the right hand side of
    \eqref{eq:740} we use the identity (which is obtained via an
    integration by parts of the balance $\B_{0,0,\beta} \phi=f$)
\begin{equation}
\label{eq:741}
  \phi^{(3)}(1)=-\int_0^1f(x)\,dx \,.
\end{equation}    
Hence,
    \begin{displaymath}
      \|\phi^{(3)}\|_2^2\leq
      C\,\big(|\phi^{\prime\prime}(1)|\,\|f\|_2+\beta^{1/2}\|\phi^{\prime\prime}\|_2^2+\|f\|_2\|\phi^{\prime\prime}\|_2+\|\phi^{(3)}\|_2\|\phi^{\prime\prime}\|_2\big)\,,
    \end{displaymath}
    which implies
    \begin{equation}\label{eq:742}
      \|\phi^{(3)}\|_2^2\leq
      \widehat C\,\big(|\phi^{\prime\prime}(1)|\,\|f\|_2+\beta^{1/2}\|\phi^{\prime\prime}\|_2^2+\|f\|_2\|\phi^{\prime\prime}\|_2\big)\,.
    \end{equation}
    Sobolev embeddings yield
    \begin{equation}\label{eq:743}
      |\phi^{\prime\prime}(1)|^2\leq \big(\|\phi^{(3)}\|_2+\|\phi^{\prime\prime}\|_2\big)\, \|\phi^{\prime\prime}\|_2\,.
    \end{equation}
     Combining \eqref{eq:742} and \eqref{eq:743} leads to
    \begin{equation}
      \label{eq:744}
      |\phi^{\prime\prime}(1)|\leq C(\beta^{1/8}\|\phi^{\prime\prime}\|_2+\|f\|_2^{1/4}\|\phi^{\prime\prime}\|_2^{3/4}+\|f\|_2^{1/3}\|\phi^{\prime\prime}\|_2^{2/3})\,.
    \end{equation} 
 By \eqref{eq:649} and the left inequality of \eqref{eq:562} it holds
    that 
    \begin{equation}
      \label{eq:745}
      \|\phi^{\prime\prime}\|_2\leq C\lambda_\beta^{-1/4}\beta^{-1/6}|\phi^{\prime\prime}(1)|+\|\tilde{v}_\Df\|_2 \,.
    \end{equation}  
 Using \eqref{eq:744} we then obtain for $\beta_0$ large enough 
    \begin{equation}\label{eq:746}
      \|\phi^{\prime\prime}\|_2\leq 2 \|\tilde{v}_\Df\|_2 +C\lambda_\beta^{-3/4}\beta^{-1/2}\|f\|_2\,.
    \end{equation}
By \eqref{eq:660} and \eqref{eq:666} (note that \eqref{eq:666} results from a straightforward
      application of \eqref{eq:412} and \eqref{eq:287} to \eqref{eq:650})
    it holds that
    \begin{equation}
      \label{eq:747}
      \|\tilde{v}_\Df\|_2\leq
      C\, \big(\|\phi^\prime\|_2+\beta^{1/6}|\phi(x_\nu)|+\beta^{-2/3}\|f\|_2+\beta^{ -1/2}|\phi^{\prime\prime}(1)|\,\big)\,.
    \end{equation}
 Since 
\begin{equation}
\label{eq:748}
|\phi(x_\nu)|= |\phi(x_\nu)-\phi(1)| \leq |\nu|^{1/2} \|\phi^\prime\|_2\leq|\lambda|^{1/2}\|\phi^\prime\|_2\,,
\end{equation}
    we obtain from \eqref{eq:746}, \eqref{eq:747},  
    with the aid of \eqref{eq:744} that 
 \begin{equation}
   \label{eq:749}
 \|\phi^{\prime\prime}\|_2\leq
   C\big(\lambda_\beta^{1/2}\|\phi^\prime\|_2+[\lambda_\beta^{-3/4}\beta^{-1/2}+\beta^{-2/3}]\|f\|_2\,\big)\,.
 \end{equation}  
We now substitute \eqref{eq:749} into \eqref{eq:744} to
obtain
\begin{equation}
\label{eq:750}
  |\phi^{\prime\prime}(1)|\leq 
   C\big(\beta^{1/6}\lambda_\beta^{1/2}\|\phi^\prime\|_2+\beta^{-1/3}\|f\|_2\,\big)\,.
\end{equation}

{\it Step 3: Given some $\mu_0>0$, we estimate $\tilde v_\Df$ and
$\tilde v_\Df^\prime$ under the additional assumption $\mu \geq -\mu_0$.}\\

 \noindent We begin by recalling \eqref{eq:664}, which is still valid in the present case and reads
\begin{multline*}
  \|\tilde{v}_\Df^\prime\|_2^2\leq C\,\Big(\|\tilde{v}_\Df\|_2\big[\|f\|_2 +  \lambda_\beta^{-3/4}\beta^{1/2}\|\phi\|_{1,2}\big]+
\\+ \lambda_\beta^{-1/2} (\beta^{2/3}\|\phi\|_{1,2}^2+\beta^{-1}\|f\|_2^2)
+  (\mu_{\beta,+} +1)  \, \|\tilde{v}_\Df\|_2^2 \Big)\,.
\end{multline*}
 Observing that $\mu_{\beta,+} \leq \beta^{1/2}$, where $\mu_{\beta,+}$ is given by
 \eqref{eq:284},  we conclude that
\begin{multline}
\label{eq:751} 
\|\tilde{v}_\Df^\prime\|_2^2\leq 
C\,\Big(\|\tilde{v}_\Df\|_2\big[\|f\|_2 +  \lambda_\beta^{-3/4}\beta^{1/2}\|\phi\|_{1,2}\big]+
\\+ \lambda_\beta^{-1/2} (\beta^{2/3}\|\phi\|_{1,2}^2+\beta^{-1}\|f\|_2^2)
+  \beta^{1/2}  \, \|\tilde{v}_\Df\|_2^2 \Big)\,. 
\end{multline}
By \eqref{eq:747}, \eqref{eq:748}, and \eqref{eq:750} it holds that 
    \begin{equation}
\label{eq:752}
  \|\tilde{v}_{\Df}\|_2 \leq C\, \big(\beta^{-2/3}\|f\|_2 + \lambda_\beta^{1/2}\|\phi^\prime\|_2\big)\,.
\end{equation}
Substituting \eqref{eq:752}  into \eqref{eq:751}
then yields that for given $\mu_0>0$ there exists $C>0$ such that for 
  $-\mu_0 \leq   \mu \leq \beta^{-1/2}$ it holds that    
\begin{equation}\label{eq:753}
  \|\tilde{v}_\Df^\prime\|_2\leq C\,\big(  \beta^{-1/4} \|f\|_2+
    \beta^{5/12}\|\phi\|_{1,2}\big)\,.
\end{equation}

{\it Step 4: We estimate  $\|v_\Df\|_\infty$ under the assumptions of Step 3.}\\

\noindent We begin by estimating the $L^2$-norm of $ g_\Df + (U+i\lambda)f $ using
\eqref{eq:738}. By \eqref{eq:656},
\eqref{eq:660}, and \eqref{eq:669}, it holds  that
 \begin{equation}\label{eq:754}
   \|(U+i\lambda)  \phi^{\prime\prime}(1)\hat{g} \| +   \|(U^{\prime\prime}\phi)^{\prime\prime} \| \leq
   C\, \big( \beta^{1/4}\|\phi\|_{1,2}+  \beta^{-7/12}\|f\|_2\big)\,.
 \end{equation}
  
 Substituting \eqref{eq:754} together with \eqref{eq:749},  \eqref{eq:752} and
 \eqref{eq:753} into \eqref{eq:738} yields 
 \begin{equation}
 \label{eq:755}  
      \|g_\Df+(U+i\lambda)f\|_2\leq C\,\big(  \beta^{-1/4} \|f\|_2+ 
    \beta^{5/12}\|\phi\|_{1,2}\big)\,.
 \end{equation}
Since, by a Sobolev embedding, we have 
\begin{displaymath}
  \Big\|(\LL_\beta^\Nf-\beta\lambda)^{-1} (U-\nu) f \|_\infty\leq
\Big\|\frac{d}{dx}(\LL_\beta^\Nf-\beta\lambda)^{-1} (U-\nu) f \|_2^{1/2}
\|(\LL_\beta^\Nf-\beta\lambda)^{-1} (U-\nu) f \|_2^{1/2}\,,
\end{displaymath}
we can conclude from \eqref{eq:412}  and  \eqref{eq:413}   that
\begin{equation}
  \label{eq:756}
\Big\|(\LL_\beta^\Nf-\beta\lambda)^{-1} (U-\nu) f \|_\infty \leq C
\beta^{-3/4} \, \|f\|_2\,. 
\end{equation}
Furthermore, by \eqref{eq:287} (for $-\Upsilon \beta^{-1/3} \leq \mu \leq \beta^{-1/2}$) or
\eqref{eq:363} and \eqref{eq:364} (for \break $-\mu_0\leq \mu\leq-\Upsilon\beta^{-1/3}$) we obtain
\begin{equation}\label{eq:757}
  \Big\|(\LL_\beta^\Nf-\beta\lambda)^{-1} i\mu f \|_\infty\leq C\beta^{-3/4}\|f\|_2\,.
 \end{equation}
Consequently, by \eqref{eq:756}  and \eqref{eq:757} we may infer that for
$\mu\geq -\mu_0$
\begin{equation}
\label{eq:758}
  \Big\|(\LL_\beta^\Nf-\beta\lambda)^{-1} (U+i\lambda) f \|_\infty \leq C
\beta^{-3/4} \, \|f\|_2\,. 
\end{equation}
By \eqref{eq:755} and (\ref{eq:287}a) it holds that
\begin{displaymath}
  \|(\LL_\beta^\Nf-\beta\lambda)^{-1} (g_\Df +[U+i\lambda] f) \|_2 \leq C\, \big(
    \beta^{-11/12}\|f\|_2+  \beta^{-1/4}\|\phi\|_{1,2}\big)\,.
\end{displaymath}
Furthermore, \eqref{eq:755} and (\ref{eq:287}b) with $p=2$ yield 
\begin{displaymath}
  \Big\|\frac{d}{dx}(\LL_\beta^\Nf-\beta\lambda)^{-1} (g_\Df +[U+i\lambda] f)\Big \|_2
  \leq   C\, \big( \beta^{-7/12}\|f\|_2+  \beta^{1/12}\|\phi\|_{1,2}\big)\,.
\end{displaymath}
Hence by Sobolev's embeddings
\begin{equation}\label{eq:759}
    \|(\LL_\beta^\Nf-\beta\lambda)^{-1} (g_\Df +[U+i\lambda] f) \|_\infty \leq C\, \big(
      \beta^{-3/4}\|f\|_2+  \beta^{-1/12}\|\phi\|_{1,2}\big)\,.
\end{equation}
In view of \eqref{eq:737} we may combine \eqref{eq:759} with
\eqref{eq:758} to obtain, for $\mu>-\mu_0$ that
\begin{equation}
\label{eq:760}
  \|v_\Df\|_\infty  \leq C \, \big( \beta^{-3/4} \|f\|_2+ \beta^{-1/12}\|\phi\|_{1,2}\big) \,.
\end{equation}

 {\it Step 5: We prove \eqref{eq:735}.}\\

 Recall from
 \noindent \eqref{eq:hyp5.5} that   $ \mu\leq\beta^{-1/2}$ and $|\nu|\leq \hat \nu_0 < U(0)$. \\

 {\em Step 5a: With  $\mu_0>0$, we prove
  \eqref{eq:735} for $|\nu|\leq\hat{\nu}_0$ and for $\mu$ satisfying
 \begin{equation}
 \label{eq:761}
   \mu \in (-\mu_0, - e^{-\beta^{1/24}}) \cup (e^{ -\beta^{1/24}},  \beta^{-1/2})\,.
 \end{equation}
 }

\noindent Set 
\begin{displaymath}
  v= \A_{\lambda,0}\phi=-(U+i\lambda)\phi^{\prime\prime}+U^{\prime\prime}\phi \,,
\end{displaymath} 
and note from \eqref{eq:644} (with $\alpha=0$) 
that
 \begin{displaymath}
   v=v_\Df - \phi^{\prime\prime}(1)(U+i\lambda)\hat{\psi}\,.
 \end{displaymath}
 An integration by parts yields
 \begin{displaymath}
   \int_0^1v\,dx=0 \,,
 \end{displaymath}
 and hence by \eqref{eq:29} and \eqref{eq:48}
 it holds that
 \begin{displaymath}
   \|\phi^\prime\|_2 \leq C\,\big(\|v_\Df\|_\infty[1+\log |\mu|^{-1}]
   +|\phi^{\prime\prime}(1)|\|\hat{\psi}\|_1\big)\,.
 \end{displaymath}
By \eqref{eq:614},  
\eqref{eq:750}, and   \eqref{eq:760},
we obtain for $|\mu|<\mu_0$  that 
 \begin{displaymath}
    \|\phi^\prime\|_2 \leq C(\beta^{-1/24}\|\phi^\prime\|_2
   + \beta^{-2/3}\|f\|_2)\,.
 \end{displaymath}
 
 For sufficiently large $\beta_0$ we obtain for $\beta \geq \beta_0$
 \begin{equation}
 \label{eq:762}
   \|\phi\|_{1,2}\leq  C\,\beta^{-2/3} \|f\|_2 \,.
 \end{equation}
 To
 obtain an estimate for $\|\phi^{\prime\prime}\|_2$ we use \eqref{eq:749}
  to obtain
 \begin{equation}
   \label{eq:763}
 \|\phi^{\prime\prime}\|_2\leq C\lambda_\beta^{1/2}\beta^{-2/3}\|f\|_2
 \end{equation}

  {\em Step 5b:    With  $0<\hat{\nu}_0<U(0)$, we prove  \eqref{eq:735}  for  $|\nu|\leq\hat{\nu}_0$ 
  and $|\mu| \leq e^{-\beta^{1/24}}$ \,.} \\
  
  \noindent Here we write for some $0 < \tilde{\mu} \leq 1/2$
  \begin{equation}
  \label{eq:764}
    \B^{\Nf,\Df}_{\lambda+\tilde{\mu}\beta^{-1/2},0,\beta}\phi= -\beta^{1/2}\tilde{\mu}\phi^{\prime\prime} +f \,.
  \end{equation}
  Note that $\lambda+\tilde{\mu}\beta^{-1/2}$ meets the assumptions of Step 5a,
  and hence, we can use \eqref{eq:763} to obtain
  \begin{displaymath}
    \|\phi^{\prime\prime}\|_2 \leq C(\tilde{\mu}\lambda_\beta^{1/2}\beta^{-1/6}\|\phi^{\prime\prime}\|_2 +\beta^{-2/3}\|f\|_2) \,.
  \end{displaymath}
  For sufficiently small  $\tilde{\mu}$ and sufficiently large $\beta_0$ we obtain \eqref{eq:763} once
  again. Consequently,
  \begin{displaymath}
    \|-\beta^{1/2}\tilde{\mu}\phi^{\prime\prime}+f\|_2 \leq C\|f\|_2\,.
  \end{displaymath}
  Hence we can apply \eqref{eq:762} once again to \eqref{eq:764}
  to establish \eqref{eq:762} for $|\mu| \leq  e^{-\beta^{1/24}}$\,.\\

 {\em Step 5c:    With  $0<\hat{\nu}_0<U(0)$, we prove  that there
    exists $\mu_0>0$ such that \eqref{eq:735}  holds for
  $\mu\leq-\mu_0$ and $|\nu|  < \hat{\nu}_0$.  }\\
  
  \noindent Since $\mu < 0$, we have, after two integrations by parts, 
\begin{equation}
\label{eq:765}
\Re\langle\phi,\B_{\lambda,0,\beta}\phi\rangle=\|\phi^{\prime\prime}\|_2^2+|\mu|\beta
\,\|\phi^\prime\|_2^2+\beta\Im\langle U^\prime\phi,\phi^\prime\rangle \,.
  \end{equation}
   Consequently, using Poincar\'e's  inequality, we obtain
\begin{displaymath}
  \|\phi^\prime\|_2 \leq \frac{C}{|\mu|}(\beta^{-1}\|f\|_2+\|\phi^\prime\|_2)\,.
\end{displaymath}
For sufficiently large $\mu_0$ and $\beta_0$ we can then conclude
\begin{displaymath}
  \|\phi^\prime\|_2 \leq \frac{C}{|\lambda|\beta}\|f\|_2 \,.
\end{displaymath}
 Using \eqref{eq:749} completes the proof of \eqref{eq:735}.
  \end{proof}
Using a perturbation argument we now obtain: 
\begin{proposition}
  \label{cor:small-alpha}
  Let $0<\hat{\nu}_0<U(0)$. 
Under the conditions of Proposition \ref{lem:zero-alpha} there exist
  $C>0$, $\hat \alpha_0>0$, and $\beta_0>0$ such that for all $\beta\geq \beta_0$ it holds
  that
  \begin{equation}
\label{eq:766}
      \sup_{
        \begin{subarray}{c}
          \alpha<\hat \alpha_0\beta^{-1/6} \\
   |\nu|< \hat \nu_0 \\
     \mu<\beta^{-1/2} 
           \end{subarray}}
\big\|(\B_{\lambda,\alpha,\beta}^{\Nf,\Df})^{-1}\big\|+
      \big\|\frac{d}{dx}\, (\B_{\lambda,\alpha,\beta}^{\Nf,\Df})^{-1}\big\|\leq
      C \beta^{-2/3}\,.
  \end{equation}
\end{proposition}
\begin{proof} 
 Let $(\phi,f)\in D(\B_{\lambda,\alpha,\beta}^{\mathfrak N,\Df})\times  L^2(0,1)$ satisfy
$\B_{\lambda,\alpha,\beta}^{\mathfrak N,\Df}\phi=f$. We then write
\begin{equation}
\label{eq:767}
  \B_{\lambda,0,\beta}^{\Nf,\Df}\phi=-\alpha^2\Big(-\frac{d^2}{dx^2}+i\beta(U+i\lambda)\Big)\phi+f \,.
\end{equation}
By \eqref{eq:735} we obtain that
\begin{displaymath}
 \|\phi^{\prime\prime}\|_2\leq C\, [\hat \alpha_0^2\lambda_\beta^{1/2}(\beta^{-1}\|\phi^{\prime\prime}\|_2+\|\phi\|_2)
 +\lambda_\beta^{1/2}\beta^{-2/3}\|f\|_2]\,.  
\end{displaymath}
Hence, for sufficiently large $\beta$ we obtain that
\begin{displaymath}
  \|\phi^{\prime\prime}\|_2\leq C\, \big[\hat \alpha^2_0\beta^{1/6}\|\phi\|_2
 +\beta^{-1/2}\|f\|_2\big]\,.  
\end{displaymath}
We can thus conclude that
\begin{displaymath}
  \Big\|\Big(-\frac{d^2}{dx^2}+i\beta(U+i\lambda)\Big)\phi\Big\|_2 \leq
  C\, \big(\beta\|\phi\|_2 +\beta^{-1/2}\|f\|_2\big) \,.
\end{displaymath}
By \eqref{eq:735} and \eqref{eq:767} we then obtain 
\begin{displaymath}
  \|\phi\|_{1,2} \leq C\,\big(\hat \alpha_0^2\|\phi\|_2+\beta^{-2/3}\|f\|_2\big) \,.
\end{displaymath}
For sufficiently small $\hat \alpha_0$ we may now conclude \eqref{eq:766}. 
\end{proof}

\subsection{Some auxiliary results}
This section is devoted to the proof of two auxiliary results which
will become useful 
in the next two subsections.\\

\begin{lemma}\label{newtechnlemma}
Let $U\in C^4(0,1)$ satisfy \eqref{eq:10}. Let further $\kappa_0$ and $\nu_1$
denote positive constants. There exist positive $\beta_0$, $\Upsilon$, $\alpha_1$, and $C$
such that, for $\lambda=\mu+i\nu$ where $\nu$ and $\mu$ satisfy
$\nu_1<\nu<U(0)+\kappa_0\beta^{-1/2}$ and $\mu<\Upsilon\beta^{-1/2}$, $\beta>\beta_0$,
$0\leq\alpha\leq\alpha_1\beta^{1/3}$,   and any $(\phi,f)\in D(\B_{\lambda,\alpha,\beta}^{\Nf,\Df})\times L^2(0,1)$ satisfying
$\B_{\lambda,\alpha,\beta}^{\Nf,\Df}\phi=f$, it holds that
\begin{multline}
\label{eq:768}
|\phi^{\prime\prime}(1)|\leq C\,\big(\beta^{-1/3}[\beta^{-1/4}+x_\nu]^{-5/6}\|f\|_2 \\
+\beta^{1/2}\bigg[\beta^{-1/4}+\frac{x_\nu}{ \log(1+x_\nu\beta^{1/4})}\bigg]^{-1}|\phi(x_\nu)|+\beta^{1/2}[\beta^{-1/4}+x_\nu]^{-1/2}\|\phi^\prime\|_2\Big)\,,
\end{multline}
where $x_\nu$ is defined by \eqref{eq:defxnu}.
\end{lemma}
\begin{proof}
  Consider first the case $U(0)-\kappa_1\beta^{-1/2}<\nu<U(0)+\kappa_0\beta^{-1/2}$
  for some $\kappa_1>0$. In this case we have $x_\nu \leq C \beta^{-1/4}$.
  As in the proof of \eqref{eq:656} we use the $\LL(L^2,L^\infty)$
  estimate of \eqref{eq:509} applied to $f=i \beta U^{\prime\prime} \phi(x_\nu)$ and
  \eqref{eq:510} applied to $i \beta U^{\prime\prime} (\phi-\phi(x_\nu)$). We obtain
  \begin{displaymath}
  |\phi^{\prime\prime}(1)|\leq C\,\big(\beta^{-1/8}\|f\|_2+\beta^{3/4}|\phi(x_\nu)|
+\beta^{5/8}\|\phi^\prime\|_2\Big)\,.  
  \end{displaymath}
 For $\nu_1<\nu<U(0)-\kappa_1\beta^{-1/2}$ with sufficiently large $\kappa_1$ and
  $\beta_0$,  we
  use \eqref{eq:445} 
   and \eqref{eq:448}, applied to the pair $( \phi^{\prime\prime} -
  \alpha^2 \phi, i\beta U^{\prime\prime}\phi+f)$ using the decomposition 
  \begin{displaymath}
   i \beta U^{\prime\prime}\phi
  =i\beta U^{\prime\prime}(\phi -\phi(x_\nu)) + i\beta U^{\prime\prime}\phi(x_\nu)
  \end{displaymath} 
and Hardy's
  inequality, to obtain
\begin{displaymath}
   |\phi^{\prime\prime}(1)|\leq
    C\Big(\beta^{-1/3}x_\nu^{-5/6}\|f\|_2+\beta^{1/2}x_\nu^{-1} { \log(1+x_\nu\beta^{1/4})}|\phi(x_\nu)|+\beta^{1/2}x_\nu^{-1/2}\|\phi^\prime\|_2\Big)\,.
\end{displaymath}
Combining the above pair of inequalities yields \eqref{eq:768}.
\end{proof}

\begin{lemma}
\label{sec:aux-vdf}  Let $U\in C^4(0,1)$ satisfy \eqref{eq:10} and
$\kappa_0$, $\Upsilon$ and $\nu_1$
denote positive constants. Let further
\begin{equation}
  \label{eq:769}
 \check x_\mu(\Upsilon,\beta)=
   \begin{cases}
     \min(\Upsilon \mu_{\beta,+}^{-1/2} \,,\,\beta^{-1/8}) & \mu>0 \\
      \beta^{-1/8} & \mu<0 \,,
    \end{cases}
\end{equation}
where $\mu_{\beta,+}$ is defined by \eqref{eq:284}.  Suppose that $\lambda=\mu+i\nu$ where $\beta^{-1}<|\mu|$, 
  $\mu<\Upsilon\beta^{-1/2}$, $\nu_1<\nu<U(0)+\kappa_0\beta^{-1/2}$   and  
  \begin{equation}
\label{eq:770}
    x_\nu<\check x_\mu(\Upsilon,\beta)\,,
  \end{equation}
  where $x_\nu$ is defined by \eqref{eq:defxnu}.  Then, there exist
  positive $\Upsilon_0$, $\beta_0$, and $C$ such that for all $\beta>\beta_0$
  and $\Upsilon<\Upsilon_0$ it holds that
\begin{multline}
  \label{eq:771}
\|\tilde{v}_\Df\|_2\leq C\, \Big[ \Upsilon^{-5/2}\beta^{-1/4}\|f\|_2\\
  +  ( \Upsilon^{-3/8}\mu_{+,\beta}^{ 3/4}+\Upsilon^{3/2}\beta^{3/16})|\phi(x_\nu)| +  (\mu_{+,\beta}^{1/2}+\Upsilon^{2}\beta^{1/8}) \|\phi^\prime\|_2 \Big]\,.
\end{multline}
in which $\tilde{v}_\Df$ is given by  \eqref{eq:649}.  
\end{lemma}

\begin{proof}
  Let $\delta:= \delta(\beta,\Upsilon)\in (0,1/4)$ be much greater than $\beta^{-1/4}$. More precisely, we introduce 
  for sufficiently small $\Upsilon$ and $\beta \geq \beta_0(\Upsilon)$ with $\beta_0(\Upsilon)$ large enough
\begin{equation}
\label{eq:772}
\delta (\beta,\Upsilon) :=
\begin{cases}
  \min(\Upsilon^{1/4} \mu_{\beta,+}^{-1/2},\Upsilon^{-1}\beta^{-1/8}) & \mu>0 \\
  \Upsilon^{-1}\beta^{-1/8} & \mu<0 \,.
\end{cases}
\end{equation}
Recall the definition of $\chi\in C_0^\infty(\R,[0,1])$ from \eqref{eq:103}
\begin{displaymath}
  \chi(x)=
  \begin{cases}
    1 & x<1/2 \\
    0 & x>3/4 \,.
  \end{cases}
\end{displaymath}
We
further set $\chi_\delta(x)=\chi(x/\delta)$, and $\tilde{\chi}_\delta=1-\chi_\delta$. 
Note that $\chi_\delta$ is supported in $(0,3\delta/4)$ and that
  $\tilde{\chi}_\delta$  is supported in $ [\delta/2,+\infty)$.\\

{\em Step 1: We estimate  $\tilde{\chi}_\delta\tilde{v}_\Df$.}\\

\noindent Using \eqref{eq:650} we now write
\begin{subequations}\label{eq:773}
  \begin{equation}
 (\LL_\beta^{\Nf,\Df}-\beta\lambda)(\tilde{\chi}_\delta\tilde{v}_\Df) =f_\delta  +
 i\beta U^{\prime\prime}\tilde{\chi}_\delta \, \phi \,,
  \end{equation}
where
\begin{equation}
 f_\delta =     2\delta^{-1}\tilde{\chi}^\prime(\cdot/\delta)  \tilde{v}_\Df^\prime+\delta^{-2} \tilde \chi^{\prime\prime}(\cdot/\delta)  \tilde{v}_\Df
    +  \tilde{\chi}_\delta\, [- f+\phi^{\prime\prime}(1)\hat{g}]\,.
\end{equation}
\end{subequations}
Setting $\gamma=2^{-1} \beta^{-1/4}\delta^{-1}$ and $v=\tilde{\chi}_\delta\tilde{v}_\Df$ in
 \eqref{eq:375} yields
\begin{multline} \label{eq:774}
 \beta\|\,|U-\nu|^{1/2}\tilde{\chi}_\delta\tilde{v}_\Df\|_2^2\leq 
  \|\tilde{\chi}_\delta\tilde{v}_\Df\|_2 \,  \|f_\delta\|_2\\  +
  C\beta\, \|\,|U-\nu|^{1/2}\tilde{\chi}_\delta\tilde{v}_\Df\|_2\, (\|\phi^\prime\|_2+\|(U-\nu)^{-1/2}\tilde{\chi}_\delta\|\,|\phi(x_\nu)|)\,.
\end{multline}
where we have used the identities $\tilde{\chi}_\delta\chi_\gamma= \tilde{\chi}_\delta$
and $\tilde{\chi}_\delta\chi_\gamma^\prime=0$, and the inequality (relying on Hardy's inequality)
\begin{displaymath}
  |\langle\tilde{\chi}_\delta\tilde{v}_\Df,\tilde{\chi}_\delta[(\phi-\phi(x_\nu))+
  \phi(x_\nu)]\rangle| \leq C\||U-\nu|^{1/2}\tilde{\chi}_\delta\tilde{v}_\Df\|_2(\|\phi^\prime\|_2+\|(U-\nu)^{-1/2}\tilde{\chi}_\delta\|\,|\phi(x_\nu)|)\,.
\end{displaymath}
  Since by \eqref{eq:770} and \eqref{eq:772} there exist positive $C$
  and $\widehat  C$ such
 that 
  \begin{displaymath}
   (U(0)-\nu)_+ \leq Cx_\nu^2\leq \widehat C\Upsilon^{3/2}\delta^2\,,
 \end{displaymath}
 we have, for sufficiently small $\Upsilon$, the existence of $\widetilde C
 >0$ such that 
 \begin{equation}\label{eq:775}
  |U-\nu|^{1/2}\tilde{\chi}_\delta\geq \frac{ 1}{\widetilde C} \delta\, \tilde{\chi}_\delta\,.
 \end{equation}
Hence, by \eqref{eq:774} and \eqref{eq:775}, 
\begin{equation}
\label{eq:776}
  \|\,|U-\nu|^{1/2}\tilde{\chi}_\delta\tilde{v}_\Df\|_2^2\leq 
  \beta^{-1}\|\tilde{\chi}_\delta\tilde{v}_\Df\|_2 \,  \|f_\delta\|_2  +
  C(\|\phi^\prime\|_2^2+\delta^{-1}\,|\phi(x_\nu)|^2)\,,
\end{equation}
which implies, using again \eqref{eq:775},
\begin{equation}
\label{eq:777}
   \|\tilde{\chi}_\delta\tilde{v}_\Df\|_2\leq C\, \Big((\delta^2\beta)^{-1}\,
   \|f_\delta\|_2  +\delta^{-1}\|\phi^\prime\|_2+\delta^{-3/2}|\phi(x_\nu)| \Big)\,.
\end{equation}
Consequently,  by (\ref{eq:773}b) and \eqref{eq:661} 
\begin{multline}
\label{eq:778}
  \|\tilde{\chi}_\delta\tilde{v}_\Df\|_2\leq
  C\Big( [\delta^2\beta]^{-1} \big(\delta^{-1}\|{\mathbf 1}_{[\delta/2,\delta]}\tilde{v}_\Df^\prime\|_2  \\ + 
  \delta^{-2}\|{\mathbf 1}_{[\delta/2,\delta]}\tilde{v}_\Df\|_2 +\|f\|_2 \big)+
   \delta^{-3/2}|\phi(x_\nu)|+  \delta^{-1}\|\phi^\prime\|_2
  +\delta^{-2}\beta^{-5/4}|\phi^{\prime\prime}(1)| \Big)\,.
\end{multline}
Substituting \eqref{eq:799} into \eqref{eq:778} yields in view of
\eqref{eq:772} (note that $\delta^2 \mu_{\beta,+}\leq \Upsilon^{1/2}$)
  \begin{multline}
\label{eq:779}
  \|\tilde{\chi}_\delta\tilde{v}_\Df\|_2^2 \leq C \Big(  [\delta^2 \beta]^{-2} \big(\delta^{-2}\|{\mathbf 1}_{[\delta/2,\delta]}\tilde{v}_\Df^\prime\|_2^2   + 
  \delta^{-4}\|{\mathbf 1}_{[\delta/2,\delta]}\tilde{v}_\Df\|_2^2\big) \\  +   [\delta^2 \beta]^{-2}  \|f\|_2^2 +
   \delta^{-3}|\phi(x_\nu)|^2+  \delta^{-2} \|\phi^\prime\|_2^2\Big)\,.
\end{multline}

{\em Step 2: We estimate $\chi_{2\delta}\tilde{v}_\Df$. }\\

\noindent Taking the inner product of \eqref{eq:650} with $\chi_{2\delta}^2 (
U^{\prime\prime})^{-1} \tilde v_\Df$  we obtain (see also  \eqref{eq:651}) that 
\begin{multline}
\label{eq:780}
  \Re\langle\chi_{2\delta}^2(U^{\prime\prime})^{-1}\tilde{v}_{\Df},(\LL_\beta^\Nf-\beta\lambda)\tilde{v}_{\Df}-i\beta U^{\prime\prime}\phi\rangle=
  \|[\chi_{2\delta}(U^{\prime\prime})^{-1/2}\tilde{v}_\Df]^\prime\|_2^2 \\
 -  \|[\chi_{2\delta}(U^{\prime\prime})^{-1/2}]^\prime\tilde{v}_\Df\|_2^2
 -\beta \mu\, \|\chi_{2\delta}(U^{\prime\prime})^{-1/2}\tilde{v}_\Df\|_2^2 + \beta
\Re\langle\chi_{2\delta}^2\phi^{\prime\prime}(1)\hat{\psi},i\phi\rangle\,.
\end{multline}
As $|U^{(3)}\chi_{2\delta}|\leq C\delta$ (given that $U^{(3)}(0)=0$) we can conclude that
\begin{equation}\label{eq:781}
  \|[\chi_{2\delta}(U^{\prime\prime})^{-1/2}\tilde{v}_\Df]^\prime\|_2^2 \geq \frac 1C 
  \|[\chi_{2\delta}\tilde{v}_\Df]^\prime\|_2^2- C \delta^2\|\chi_{2\delta}\tilde{v}_\Df\|_2^2\,.
\end{equation}
 Furthermore, as $|(\chi_{2\delta})^\prime| \leq C\delta^{-1}\tilde{\chi}_\delta$, we
 obtain, using again the fact that $U^{(3)}(0)=0\,$, 
\begin{displaymath}
  \|[\chi_{2\delta}(U^{\prime\prime})^{-1/2}]^\prime\tilde{v}_\Df\|_2^2\leq C(\delta^2
  \|\chi_{2\delta}\tilde{v}_\Df\|_2^2 + \delta^{-2}\|\tilde{\chi}_\delta\tilde{v}_\Df\|_2^2)\,.
\end{displaymath}
Substituting the above, together with \eqref{eq:781} into
\eqref{eq:780}, recalling that by \eqref{eq:614}
\begin{displaymath}
  \beta|\langle\chi_{2\delta}^2\phi^{\prime\prime}(1)\hat{\psi},i\phi\rangle|\leq\beta
  |\phi^{\prime\prime}(1)|\,\|\phi^\prime\|_2 \|(1-x)^3\hat{\psi}\|_1\leq C\beta^{-1}|\phi^{\prime\prime}(1)|\,\|\phi^\prime\|_2 \,,
\end{displaymath}
and that by \eqref{eq:659},  \eqref{eq:614},  \eqref{eq:615},
  \eqref{eq:451}, and \eqref{eq:458}
\begin{displaymath}
  |\langle\chi_{2\delta}(U^{\prime\prime})^{-1}\tilde{v}_{\Df},\phi^{\prime\prime}(1)\tilde \chi_{2\delta} \hat{g}\rangle|\leq C\beta^{-3/4} \|\chi_{2\delta}\tilde{v}_\Df\|_2|\phi^{\prime\prime}(1)|
\end{displaymath}
yields 
\begin{multline*}
  \|[\chi_{2\delta}\tilde{v}_\Df]^\prime\|_2^2\leq C\,\Big[  \Upsilon^{-1}  \delta^{2}\|f\|_2^2
  + \beta^{-3/4} |\phi^{\prime\prime}(1)|\,\|\chi_{2\delta}\tilde{v}_\Df\|_2+  \\ + (\Upsilon \delta^{-2}+ \mu_{\beta,+})\|\chi_{2\delta}\tilde{v}_\Df\|_2^2+
  \delta^{-2}\|\tilde{\chi}_\delta\tilde{v}_\Df\|_2^2+
  \beta^{-1}|\phi^{\prime\prime}(1)|\,\|\phi^\prime\|_2\Big]\,.
\end{multline*}
By Poincar\'e's inequality we have
\begin{equation}
  \label{eq:782}
\| \chi_{2\delta}\tilde{v}_\Df\|_2^2 \leq C \delta^2  \|[\chi_{2\delta}\tilde{v}_\Df]^\prime\|_2^2\,.
\end{equation}
Hence, in view of \eqref{eq:772} and  \eqref{eq:782},  we obtain for
sufficiently small $\Upsilon$
\begin{multline}\label{eq:783}
\| \chi_{2\delta}\tilde{v}_\Df\|_2^2 + \delta^2
\|[\chi_{2\delta}\tilde{v}_\Df]^\prime\|_2^2\leq C\,\Big[  \Upsilon^{-1} \delta^4 \|f\|_2^2
  \\ + \|\tilde{\chi}_\delta\tilde{v}_\Df\|_2^2+
  \beta^{-1}\delta^2 |\phi^{\prime\prime}(1)|\,(\|\phi^\prime\|_2 + \delta^2  \beta^{-1/2}  |\phi^{\prime\prime}(1) |) \Big]\,.
\end{multline}
 Substituting \eqref{eq:768}
into \eqref{eq:783}  yields, in view of
\eqref{eq:772} 
\begin{equation}
\label{eq:784}
  \| \chi_{2\delta}\tilde{v}_\Df \|_2^2  +  \delta^2
  \|[\chi_{2\delta}\tilde{v}_\Df]^\prime\|_2^2\leq C \Big( \Upsilon^{-1} \delta^4  \|f\|_2^2
  + (\beta^{-1/4} \delta^2 +\delta^4)  \|\phi^\prime\|_2^2 
 +   \|\tilde{\chi}_\delta\tilde{v}_\Df\|_2^2\Big) \,.
\end{equation}
Combining \eqref{eq:779} with \eqref{eq:784}, and
\eqref{eq:772} we obtain
\begin{multline}
\label{eq:785} 
\|\chi_{2\delta}\tilde{v}_\Df\|_2^2
    +\delta^2
    \|[\chi_{2\delta}\tilde{v}_\Df]^\prime\|_2^2\leq C\Big[  \Upsilon^{-1} \delta^4\|f\|_2^2
   +
    \delta^{-6}\beta^{-2}\|{\mathbf 1}_{[\delta/2,\delta]}\tilde{v}_\Df^\prime\|_2^2\\
    +\delta^{-8}\beta^{-2}\|{\mathbf 1}_{[\delta/2,\delta]}\tilde{v}_\Df\|_2^2
    +\delta^{-3}|\phi(x_\nu)|^2+ \delta^{-2}\|\phi^\prime\|_2^2 \Big]\,.
\end{multline}
  As ${\mathbf
  1}_{[\delta/2,\delta]}\leq\chi_{2\delta}$, and $\delta^{-8}\beta^{-2}\leq \Upsilon^{-2} \beta^2
\mu_+^4+\Upsilon^{8}\beta^{-1}$, where $\mu_+=\max(\mu,0)$,
  we obtain
\begin{multline*}
  \|\chi_{2\delta}\tilde{v}_\Df\|_2^2
    +\delta^2
    \|[\chi_{2\delta}\tilde{v}_\Df]^\prime\|_2^2\leq C\Big[ \Upsilon^{-1} \delta^4\|f\|_2^2\\
  +
     (\Upsilon^{-2} \beta^2 \mu_+^4+\Upsilon^8\beta^{-1})\delta^2\|[\chi_{2\delta}\tilde{v}_\Df]^\prime\|_2^2\\
    + (\Upsilon^{-2} \beta^2 \mu_+^4+\Upsilon^8\beta^{-1})\|\chi_{2\delta}\tilde{v}_\Df\|_2^2
    +\delta^{-3}|\phi(x_\nu)|^2+ \delta^{-2}\|\phi^\prime\|_2^2 \Big]\,.
\end{multline*}
For sufficiently small $\Upsilon$  and $\beta_0^{-1}$ we then conclude (as
$\Upsilon^{-2} \beta^2 \mu_+^4\leq \Upsilon^2$)  that  
\begin{equation}
\label{eq:786}
  \|\chi_{2\delta}\tilde{v}_\Df\|_2^2+
 \delta^2\|[\chi_{2\delta}\tilde{v}_\Df]^\prime\|_2^2\leq C\big[  \Upsilon^{-1} \delta^4\|f\|_2^2
+  \delta^{-3}|\phi(x_\nu)|^2+  \delta^{-2}\|\phi^\prime\|_2^2 \big]\,.
\end{equation}
Similarly, by \eqref{eq:779} 
and \eqref{eq:772} (note that $\delta^{-4}\beta^{-2}\leq (\Upsilon^{-2} \beta^2 \mu_+^4+\Upsilon^{8}\beta^{-1})\delta^4$), it holds that
\begin{multline}
\label{eq:787}
   \|\tilde{\chi}_\delta\tilde{v}_\Df\|_2^2  \leq C \Big(
   (\Upsilon^{-2} \beta^2 \mu_+^4+\Upsilon^{8}\beta^{-1}) \delta^2\|[\chi_{2\delta}\tilde{v}_\Df]^\prime\|_2^2 +
(\Upsilon^{-2} \beta^{-2} \mu_+^4+\Upsilon^{8}\beta^{-1})  \|\chi_{2\delta}\tilde{v}_\Df\|_2^2 \\ +  (\Upsilon^{-2} \beta^2 \mu_+^4+\Upsilon^{8}\beta^{-1})\delta^4 \|f\|_2^2 +
   \delta^{-3}|\phi(x_\nu)|^2+  \delta^{-2} \|\phi^\prime\|_2^2\Big)\,.
\end{multline}
 Since  $\chi_{2\delta} + \tilde \chi_\delta > \frac 1C$ we obtain from
\eqref{eq:786} and   \eqref{eq:787} (as  $\Upsilon^{-1} \delta^4  \leq C \Upsilon^{-5} \beta^{-1/2}$ by
\eqref{eq:772} and since $|\mu|>\beta^{-1}$ ) that for $\Upsilon$ and $\beta_0^{-1}$ small enough
\begin{displaymath}
\|\tilde{v}_\Df\|_2\leq C\big[ \Upsilon^{-5/2}\beta^{-1/4}\|f\|_2
  +  ( \Upsilon^{-3/8} \mu_{+,\beta}^{3/4}+\Upsilon^{3/2}\beta^{3/16})|\phi(x_\nu)| +  (\mu_{+,\beta}^{1/2}+\Upsilon^{2}\beta^{1/8}) \|\phi^\prime\|_2 \big]\,,
\end{displaymath}
 which is precisely \eqref{eq:771}.\\
\end{proof}

\subsection{Resolvent estimates for $|\Im\lambda-U(0)|=\OO (\beta^{-1/2})$.}
\label{sec:5.6}

 We  consider,  for given positive $ \kappa_0,\,\alpha_1$, $\lambda=\mu+i\nu$,  and
 some positive $\Upsilon$, the zone  
 \begin{displaymath}
\label{eq:hyp5.6}
\mathcal E(\alpha_1,\beta_0, \Upsilon,\kappa_0):=
\begin{Bmatrix}
 (\lambda,\alpha,\beta)\in \mathbb C\times \mathbb
R_+^2\,,\,  0 \leq \alpha \leq \alpha_1 \beta^{1/3}\,, \beta \geq \beta_0  \\  \mu<\Upsilon\beta^{-1/2} \,,\, 
      U(0)-\kappa_0\beta^{-1/2}\leq\nu\leq   U(0)+\kappa_0\beta^{-1/2}
\end{Bmatrix}
\,.
 \end{displaymath}

\begin{proposition}
\label{lem:quadratic}
Let $U\in C^4([0,1])$ satisfy \eqref{eq:10} and $U^{(3)}(0)=0$. Let
further $\alpha_1>0$ and $\kappa_0>0$. Then, there exist positive $\Upsilon$, $\beta_0$
and $C$, such that for \break $(\lambda,\alpha,\beta)\in \mathcal E(\alpha_1,\beta_0, \Upsilon,\kappa_0)$, it holds
   \begin{equation}
\label{eq:788}
 \max(1,|\mu\beta|^{1/4}) \Big(\big\|(\B_{\lambda,\alpha,\beta}^{\mathfrak N,\Df})^{-1}\big\|+
       \Big\|\frac{d}{dx}\, (\B_{\lambda,\alpha,\beta}^{\mathfrak
         N,\Df})^{-1}\Big\|\Big)\leq C \beta^{-3/8}\,.
   \end{equation}
 \end{proposition}
 \begin{proof}~\\
{\em Step 1: Preliminaries.}\\

 \noindent We follow the same outlines as in the proof of Proposition
 \ref{lem:right-curve}.  Nevertheless, given that
 $|\nu-U(0)|\sim\OO(\beta^{-1/2})$, we need to address here the
 quadratic behavior of $U(x)-U(0)$ in the vicinity of $x=0$ (see
 Subsection 3.2 for instance). 
   
Let $\tilde{v}_\Df$ be given by \eqref{eq:649}.  For 
   convenience of the reader we repeat here \eqref{eq:651}
\begin{multline}\label{eq:789}
  \Re\langle(U^{\prime\prime})^{-1}\tilde{v}_{\Df},(\LL_\beta^\Nf-\beta\lambda)\tilde{v}_{\Df}+i\beta U^{\prime\prime}\phi\rangle=
  \|(U^{\prime\prime})^{-1/2}\tilde{v}_\Df^\prime\|_2^2+ \\
  +\Re\langle\big((U^{\prime\prime})^{-1}\big)^\prime\tilde{v}_\Df,\tilde{v}_\Df^\prime\rangle
 -\beta \mu\, \|(U^{\prime\prime})^{-1/2}\tilde{v}_\Df\|_2^2 + \beta
\Re\langle\phi^{\prime\prime}(1)\hat{\psi},i\phi\rangle\,,
\end{multline}
 where $\hat \psi= \hat \psi _{\lambda,\beta}$ is introduced in \eqref{eq:613} and
$\phi\in D(\B_{\lambda,\alpha,\beta}^{\mathfrak N,\Df})$ satisfies for  $f\in L^2(0,1)$
\begin{displaymath}
\B_{\lambda,\alpha,\beta}^{\mathfrak N,\Df}\phi=f\,.
\end{displaymath}  
We begin by estimating the last term on the right-hand-side of
\eqref{eq:789}. For technical reasons we distinguish between the case
$\mu < - \mu_0 $ (for some sufficiently small $\mu_0 >0$) and
$\mu>-\mu_0$.\\

{\em Step 2: We estimate
$\Re\langle\phi^{\prime\prime}(1)\hat{\psi},i\phi\rangle$
for $\mu\geq-\mu_0$. }\\

\noindent As in step 2 of the proof of Proposition \ref{lem:right-curve} we
write
\begin{displaymath}
\phi(x) = \int_x^1(\xi-x)\phi^{\prime\prime}(\xi)\,d\xi=\phi^{\prime\prime}(1)w+
\int_x^1(\xi-x)[\phi^{\prime\prime}(\xi)-\phi^{\prime\prime}(1)\hat{\psi}(\xi)]\,d\xi\,.
\end{displaymath}
where
\begin{displaymath}
  w(x)=\int_x^1(\xi-x)\hat{\psi}(\xi)\,d\xi\,.
\end{displaymath}
Then we write
\begin{equation}
\label{eq:790}
  \Re\langle\phi^{\prime\prime}(1)\hat{\psi},i\phi\rangle=-|\phi^{\prime\prime}(1)|^2\Im\langle\hat{\psi},w\rangle+\Re\langle\phi^{\prime\prime}(1)\hat{\psi},i(\phi-\phi^{\prime\prime}(1)w)\rangle \,.
\end{equation}
For the first term on the right-hand-side  we write, using the fact
that $w^{\prime\prime} = \hat \psi$\,, and integration by parts,
\begin{equation}
\label{eq:791}
  \Im\langle\hat{\psi},w\rangle=\Im\langle w^{\prime\prime},w\rangle=-\Im \{\bar{w}^\prime(0)w(0)\} \,. 
\end{equation}
We now use \cite[Proposition A.1]{almog2019stability}  to obtain the following improvement of \eqref{eq:614}:
\begin{equation}\label{eq:792}
  \hat{\psi}(x)=e^{-\beta^{1/2}(-\lambda)^{1/2}(1-x)}+\hat{\psi}_1(x) \,,
\end{equation}
where
\begin{equation}\label{eq:793}
  \|\hat{\psi}_1\|_1+\beta^{1/2}\|(1-x)\hat{\psi}_1\|_1\leq C\beta^{-1} \,.
\end{equation}
Next, we write, using  \eqref{eq:hyp5.6}, \eqref{eq:792}, \eqref{eq:793}   
\begin{multline}
\label{eq:794}
\overline{w}^\prime(0)=-\int_0^1\overline{\hat \psi}(\xi)\,d\xi
=-\int_0^1[e^{-\beta^{1/2}(-\bar \lambda)^{1/2}(1-\xi)}+\overline{\hat
  \psi}_1(\xi)]\,d\xi \\ =-\frac{1}{(-\beta \bar \lambda)^{1/2}}+\OO(\beta^{-1})=
-\frac{e^{-i\pi/4}\beta^{-1/2}}{[U(0)]^{1/2}}[1+\OO(|\mu|^{1/2})]
+\OO(\beta^{-1}) \,.
\end{multline}
To obtain \eqref{eq:794} we used the identities
\begin{displaymath}
\int_0^1 e^{-\beta^{1/2}(- \bar \lambda)^{1/2}(1-x)} dx =
\beta^{-1/2}(-\bar{\lambda})^{-1/2} (1 - e^{-\beta^{1/2}(-\lambda)^{1/2}}) \,,
\end{displaymath}
and 
  $$\beta^{-1/2}(-[\mu-i\nu])^{-1/2}=e^{-i\pi/4}\beta^{-1/2}(U(0)+[\nu-U(0)+i\mu])^{-1/2}\,,$$
 which shows the exponentially small behavior of $e^{-\beta^{1/2}(-\lambda)^{1/2}}$ if $\mu_0 >0$ is chosen small enough.\\
Furthermore,  it
holds that
\begin{equation*}
  w(0)=\int_0^1\xi\hat{\psi}(\xi)\,d\xi=
  -w^\prime(0)-\int_0^1(1-\xi)\hat{\psi}(\xi)\,d\xi 
\end{equation*}
which implies by \eqref{eq:614}, \eqref{eq:792}, and \eqref{eq:793} 
\begin{equation}
\label{eq:795}
  w(0)= -w^\prime(0) -\frac{i\beta^{-1}}{U(0)}[1+\OO(|\mu|^{1/2})]   
  +\OO(\beta^{-3/2}) \,. 
\end{equation}
Combining \eqref{eq:794} and \eqref{eq:795}  yields
\begin{equation}\label{eq:796}
  \Im
  \{\bar{w}^\prime(0)w(0)\}=\frac{\beta^{-3/2}}{[U(0)]^{3/2}\sqrt{2}}[1+\OO(|\mu|^{1/2})]+\OO(\beta^{-2}) \,.
\end{equation}
Substituting \eqref{eq:796}  into \eqref{eq:791} yields
\begin{displaymath}
   \beta\Re\langle\hat{\psi},iw\rangle=\frac{\beta^{-1/2}}{[U(0)]^{3/2}\sqrt{2}}[1+\OO(|\mu|^{1/2})]+\OO(\beta^{-1}) \,.
\end{displaymath}
For sufficiently large $\beta_0$ and sufficiently small
$\mu_0$,  $ \beta\Re\langle\hat{\psi},iw\rangle$ is positive and hence, by \eqref{eq:790}
we can conclude that 
\begin{equation}
  \label{eq:797}
\Re\langle\phi^{\prime\prime}(1)\hat{\psi},i\phi\rangle \geq \Re\langle\phi^{\prime\prime}(1)\hat{\psi},i (\phi-\phi^{\prime\prime}(1)w)\rangle  \,. 
\end{equation}
As in the proof of Proposition \ref{lem:right-curve} (step 2) we now
apply \eqref{eq:653},  recalling that $\alpha\leq\alpha_1\beta^{1/3}$, to obtain  (note
that $\lambda_\beta\geq\frac 12 |U(0)|\beta^{1/3}$ by \eqref{eq:hyp5.6})
\begin{multline}\label{eq:798}
  \beta\, \Re\langle\phi^{\prime\prime}(1)\hat{\psi},i\phi\rangle\geq -C \,\beta^{-1/6} \lambda_\beta^{-7/4}\,
  | \phi^{\prime\prime}(1)|(\|\tilde{v}_\Df^\prime\|_2+   \beta^{2/3}  \|\phi^\prime\|_2)\\
    \geq-\widehat C\beta^{-3/4}|\phi^{\prime\prime}(1)|(\|\tilde{v}_\Df^\prime\|_2+ 
      \beta^{2/3}\|\phi^\prime\|_2)\,.
\end{multline}

Let $x_\nu$ be defined by \eqref{eq:defxnu}.  By the assumption
on $\nu$ it holds that $x_\nu\leq C\beta^{-1/4}$. Hence,  by
\eqref{eq:768} 
\begin{equation}
\label{eq:799}
    |\phi^{\prime\prime}(1)|\leq
    C \,\big(\beta^{-1/8}\|f\|_2+\beta^{3/4}|\phi(x_\nu)|+\beta^{5/8}\|\phi^\prime\|_2\big)\,,
\end{equation}
from which we conclude
\begin{equation}\label{eq:800}
 \beta\, \Re\langle\phi^{\prime\prime}(1)\hat{\psi}_,i\phi\rangle\geq
 -C\,\big(\|\phi\|_{1,2}+\beta^{-7/8}\|f\|_2)(\|\tilde{v}_\Df^\prime\|_2+
   \beta^{2/3}\|\phi^\prime\|_2\big)\,.
\end{equation}

{\em Step 3: We estimate $\tilde{v}_\Df$ and  $\tilde{v}_\Df^\prime$.} \\

\noindent Here we follow  the Steps 4 and 5 in the proof of
Proposition \ref{lem:right-curve} with \eqref{eq:657} replaced by
\eqref{eq:799}.\\ 

By  \eqref{eq:365}, \eqref{eq:650}, \eqref{eq:661}, and \eqref{eq:799}, 
we obtain (compare with \eqref{eq:666}) that
\begin{equation}
\label{eq:801}
  \|\tilde{v}_\Df\|_2\leq
  C\, (\beta^{1/4}\|\phi^\prime\|_2+\beta^{3/8}|\phi(x_\nu)|+\beta^{-1/2}\|f\|_2)\,.
\end{equation}
 Substituting \eqref{eq:801}, together with \eqref{eq:799},  \eqref{eq:800} and \eqref{eq:661},
  into \eqref{eq:662} yields 
\begin{multline*}
    \frac 1C  \|\tilde{v}_\Df^\prime\|_2^2 \leq 
     (\beta^{1/4}\|\phi^\prime\|_2+\beta^{3/8}|\phi(x_\nu)|+\beta^{-1/2}\|f\|_2)(\|f\|_2
   +\beta^{3/8}\|\phi^\prime\|_2+\beta^{1/2}|\phi(x_\nu)|)\\ \qquad 
   +( 1+ \mu_{\beta,+} )  
   (\beta^{1/4}\|\phi^\prime\|_2+\beta^{3/8}|\phi(x_\nu)|+\beta^{-1/2}\|f\|_2)^2\\
   +(\|\phi\|_{1,2}+\beta^{-7/8}\|f\|_2)(\|\tilde{v}_\Df^\prime\|_2+      \beta^{2/3}\|\phi^\prime\|_2)\,.
\end{multline*}
Hence,
\begin{equation}
\label{eq:802}
  \|\tilde{v}_\Df^\prime\|_2 \leq C\, \Big(\big( \beta^{1/8}+
   \mu_{\beta,+}^{1/2}\big)\,\big(\beta^{1/4}\|\phi^\prime\|_2+\beta^{3/8}|\phi(x_\nu)|\big)
  +\beta^{-1/8}\|f\|_2\Big) \,.
\end{equation}
Recall that
\begin{displaymath}
   \mu_{\beta,+}=\beta\mu_+=\beta \max(\mu,0) \,.
\end{displaymath}

 Since \eqref{eq:802} is unsatisfactory, given that the coefficient of
 $\beta^{1/4}\|\phi^\prime\|_2+\beta^{3/8}|\phi(x_\nu)|$ is not necessarily small,  as will become clear 
  in the sequel, we obtain an improved estimate in the next step.\\

 {\em Step 4:  For $|\mu|>\beta^{-1}$ we prove under the assumptions of
 the proposition that 
\begin{multline}
  \label{eq:803}
\|\phi^{\prime\prime}-\phi^{\prime\prime}(1)\hat{\psi}\|_2\leq C\big[
\Upsilon^{-5/2}\beta^{-1/4}\|f\|_2 \\
  +  (\Upsilon^{-3/8}\mu_{+,\beta}^{3/4}+\beta^{1/3})|\phi(x_\nu)| +  (\mu_{+,\beta}^{1/2}+\Upsilon^{2}\beta^{1/8}) \|\phi^\prime\|_2 \big]\,.
\end{multline}
}

Using the definition of $\tilde v_\Df$ given in \eqref{eq:649}, an
integration by parts yields
\begin{equation}
\label{eq:804}
  \langle-\phi^{\prime\prime}+\phi^{\prime\prime}(1)\hat{\psi},\tilde{v}_\Df\rangle=\|\phi^{\prime\prime}-\phi^{\prime\prime}(1)\hat{\psi}\|_2^2+\alpha^2\|\phi^\prime\|_2^2 +\langle \phi^{\prime\prime}(1)\hat{\psi},\alpha^2\phi\rangle\,.
\end{equation}
By \eqref{eq:614} for $s=1/2$ and \eqref{eq:799}  it holds that
\begin{multline*}
  |\langle \phi^{\prime\prime}(1)\hat{\psi},\alpha^2\phi\rangle|\leq
  \alpha^2|\phi^{\prime\prime}(1)|\,\|(1-x)^{1/2}\hat{\psi}\|_1\,\|\phi^\prime\|_2\\ \leq C\alpha^2\,\big(\beta^{-7/8}\|f\|_2+|\phi(x_\nu)|+\beta^{-1/8}\|\phi^\prime\|_2\big)\, \|\phi^\prime\|_2\,.
\end{multline*}
Substituting the above into \eqref{eq:804} yields for sufficiently
large $\beta_0$
\begin{multline*}
  \|\phi^{\prime\prime}-\phi^{\prime\prime}(1)\hat{\psi}\|_2^2+\alpha^2\|\phi^\prime\|_2^2
  \\ \leq\|-\phi^{\prime\prime}+\phi^{\prime\prime}(1)\hat{\psi}\|_2\,\|\tilde{v}_\Df\|_2+C\alpha^2\big(\beta^{-7/8}\|f\|_2+|\phi(x_\nu)|+\beta^{-1/8}\|\phi^\prime\|_2\big)\|\phi^\prime\|_2\,.
\end{multline*}
For sufficiently large $\beta_0$ we then obtain that
\begin{equation}
\label{eq:805}
   \|\phi^{\prime\prime}-\phi^{\prime\prime}(1)\hat{\psi}\|_2^2
  \leq C\, \big[\|\tilde{v}_\Df\|_2^2+ \alpha^2(\beta^{-7/4}\|f\|_2^2+|\phi(x_\nu)|^2)\big]\,.
\end{equation}
We now obtain \eqref{eq:803} from \eqref{eq:771} and the fact that
$\alpha\leq\alpha_1\beta^{1/3}$.\\

{\em Step 5: We estimate $\|v_\Df\|_\infty$ under the assumption of the
proposition and the additional conditions $|\mu|>\beta^{-1}$ and $\mu\geq-\mu_0$.}\\

\noindent Let $v_\Df$ be given by \eqref{eq:644}. For the convenience of the
reader we recall here \eqref{eq:646} 
\begin{subequations}
\label{eq:806}
  \begin{equation}
(\LL_\beta^{\Nf,\Df} -\beta\lambda)v_{\Df}=g_\Df \,,
\end{equation}
where  (after reordering)
\begin{multline}
  g_\Df= (U+i\lambda)( -f + \phi^{\prime\prime}(1)\hat{g}) 
  -2U^\prime \tilde{v}_\Df^\prime  -  U^{\prime\prime}\phi^{\prime\prime}(1)\hat{\psi}  + \\ U^{\prime\prime}([\phi^{\prime\prime}(1)\hat{\psi}
  -\phi^{\prime\prime}]-\tilde{v}_\Df) 
 -  2U^{(3)}\phi^\prime -   U^{(4)}\phi\,. 
\end{multline}
\end{subequations}

Next, we obtain a bound for $\|(\LL_\beta^{\Nf,\Df} -\beta\lambda)^{-1}g_\Df\|_\infty$ by
separately estimating the contribution of each of the six terms on the
right-hand-side  of (\ref{eq:806}b).\\
To  obtain the $L^\infty$ estimates we repeatedly use the following Sobolev
embedding inequality 
\begin{equation}\label{eq:inter} 
  \|v_\Df\|_\infty\leq  \|v_\Df\|_{2}^{1/2}\, \|v_\Df^\prime \|_{2}^{1/2}\,.
\end{equation}

Writing $(U +i\lambda)=(U-\nu) + i \mu$ we obtain by (\ref{eq:363})-(\ref{eq:364}) (for $\mu <0$),
  (\ref{eq:365}) (for $\mu >0$), and \eqref{eq:432}, that 
\begin{multline*}
    \|(\LL_\beta^{\Nf,\Df} -\beta\lambda)^{-1}(U+i\lambda)( -f + \phi^{\prime\prime}(1)\hat{g})
    \|_\infty\\ \leq C\big( \beta^{-3/4} + |\mu|^{1/4} \beta^{-3/4} + \mu_+ \beta^{-3/8} \big)\big(\|f\|_2 +
    |\phi^{\prime\prime}(1)|\,\|\hat{g}\|_2\big)\,.
\end{multline*}
Since $-\mu_0 < \mu< \Upsilon \beta^{-1/2}$, we obtain
\begin{subequations} \label{eq:807}
\begin{equation}
    \|(\LL_\beta^{\Nf,\Df} -\beta\lambda)^{-1}(U+i\lambda)( -f + \phi^{\prime\prime}(1)\hat{g})
    \|_\infty \leq C \beta^{-3/4} (\|f\|_2 +
    |\phi^{\prime\prime}(1)|\,\|\hat{g}\|_2)\,.
\end{equation}
 Using (\ref{eq:365}a) and
(\ref{eq:365}b) we obtain, recalling that, for $|\nu-U(0)|\leq \kappa_0 \beta^{-1/2}$,
we have $x_\nu\leq C\beta^{-1/4}$,
\begin{multline}
  \|(\LL_\beta^{\Nf,\Df} -\beta\lambda)^{-1} (U^\prime \tilde{v}_\Df^\prime)\|_\infty=
  \|(\LL_\beta^{\Nf,\Df} -\beta\lambda)^{-1} (U^\prime(x_\nu)+
  (U^\prime(x)-U^\prime(x_\nu))\tilde{v}_\Df^\prime\|_\infty \\ \leq
  C\, \big(x_\nu\beta^{-3/8}+\beta^{-5/8}\big)\|\tilde{v}_\Df^\prime\|_2  \leq \widehat C\beta^{-5/8} \|\tilde{v}_\Df^\prime\|_2\,.
\end{multline}
By   \eqref{eq:560} and (\ref{eq:365}a) we have
\begin{multline*}
    \|(\LL_\beta^{\Nf,\Df} -\beta\lambda)^{-1} (\phi^{\prime\prime}(1)U^{\prime\prime}\hat{\psi})\|_\infty=
    \|(\LL_\beta^{\Nf,\Df}
    -\beta\lambda)^{-1} \phi^{\prime\prime}(1)[U^{\prime\prime}(1)+(U^{\prime\prime}(x)-U^{\prime\prime}(1))]\hat{\psi}\|_\infty\\
    \leq C\,\big(\beta^{-1}+\beta^{-3/8}\|(1-x)\hat{\psi}\|_2\big) \, |\phi^{\prime\prime}(1) |\,.
\end{multline*}
By  \eqref{eq:451}, \eqref{eq:457}, and \eqref{eq:558} it holds that
\begin{equation}
\label{eq:808}
   \|(1-x)^k\, \hat{\psi}\|_2 \leq 
   C\,  \lambda_\beta^{- (1+2k)/4} \beta^{-(1+2k) /6}\,.
\end{equation}
   Using \eqref{eq:808} with $k=1$ yields 
   \begin{equation}
     \|(\LL_\beta^{\Nf,\Df} -\beta\lambda)^{-1} (\phi^{\prime\prime}(1)U^{\prime\prime}\hat{\psi})\|_\infty \leq C \beta^{-1}|\phi^{\prime\prime}(1)|\,.
   \end{equation}
   For the next term we use   (\ref{eq:365}a) and  (\ref{eq:365}b) to obtain that
\begin{displaymath}
 \|(\LL_\beta^{\Nf,\Df} -\beta\lambda)^{-1} U^{\prime\prime}([\phi^{\prime\prime}(1)\hat{\psi}
  -\phi^{\prime\prime}]-\tilde{v}_\Df)\|_\infty \leq C\beta^{-3/8}\,\big(\|\phi^{\prime\prime}(1)\hat{\psi}
  -\phi^{\prime\prime}\|_2+\|\tilde{v}_\Df\|_2\big)\,.
\end{displaymath}
 We then use \eqref{eq:771} 
and \eqref{eq:803}
to obtain that 
\begin{multline}
\|(\LL_\beta^{\Nf,\Df} -\beta\lambda)^{-1} U^{\prime\prime}([\phi^{\prime\prime}(1)\hat{\psi}
  -\phi^{\prime\prime}]-\tilde{v}_\Df)\|_\infty \leq   C\beta^{-3/8}\, \Big[
  \Upsilon^{-5/2}\beta^{-1/4}\|f\|_2 \\
  +  ( \Upsilon^{-3/8}\mu_{+,\beta}^{ 3/4}+ \beta^{1/3})|\phi(x_\nu)| +  (\mu_{+,\beta}^{1/2}+\Upsilon^{2}\beta^{1/8}) \|\phi^\prime\|_2 \Big]\,.
\end{multline}
 As by \eqref{eq:649} 
\begin{displaymath}
  \phi^{\prime\prime}(1)\hat{\psi}
  -\phi^{\prime\prime}-\tilde{v}_\Df=\alpha^2\phi \,,
\end{displaymath}
we may also write
  \begin{multline*}
\|(\LL_\beta^{\Nf,\Df} -\beta\lambda)^{-1} U^{\prime\prime}([\phi^{\prime\prime}(1)\hat{\psi}
  -\phi^{\prime\prime}]-\tilde{v}_\Df)\|_\infty =  \|(\LL_\beta^{\Nf,\Df} -\beta\lambda)^{-1}
  U^{\prime\prime}\alpha^2\phi\|_\infty \leq   \\ \leq  \|(\LL_\beta^{\Nf,\Df}
  -\beta\lambda)^{-1} U^{\prime\prime}\alpha^2[\phi(x_\nu) ]\|_\infty +  \|(\LL_\beta^{\Nf,\Df}
  -\beta\lambda)^{-1} U^{\prime\prime}\alpha^2[\phi-\phi(x_\nu)]\|_\infty  \,.
\end{multline*}
Then we use   (\ref{eq:365}a) and  (\ref{eq:365}b) together with
Hardy's inequality for the second term to obtain, 
\begin{equation}
  \|(\LL_\beta^{\Nf,\Df} -\beta\lambda)^{-1} U^{\prime\prime}\alpha^2\phi\|_\infty \leq C\alpha^2\beta^{-1/2}\big(
  |\phi(x_\nu)| + \beta^{-1/8}\|\phi^\prime\|_2\big)\,. 
\end{equation}
 In the sequel we use (\ref{eq:807}e) for $\alpha\geq\beta^{1/8}$ and
(\ref{eq:807}f) for $\alpha<\beta^{1/8}$. \\
We estimate the next term as in (\ref{eq:807}a) (as
  $|U^{(3)}(x)|\leq Cx$)
\begin{equation}
   \|(\LL_\beta^{\Nf,\Df} -\beta\lambda)^{-1} (2U^{(3)}\phi^\prime) \|_\infty \leq
   C \, \beta^{-5/8} \|\phi^\prime\|_2\,.
\end{equation}
Finally, we estimate the last term as in (\ref{eq:807}e)
\begin{equation}
   \|(\LL_\beta^{\Nf,\Df} -\beta\lambda)^{-1} (2U^{(4)}\phi)\|_\infty \leq C\beta^{-1/2}(|\phi(x_\nu)|+\beta^{-1/8}\|\phi^\prime\|_2)\,.
\end{equation}
\end{subequations} 
 Combining (\ref{eq:807}a-h) then yields
\begin{multline}\label{eq:809}
  \|v_\Df\|_\infty\leq
  C \Big[ \beta^{-3/4}\gamma(\alpha,\beta)\|f\|_2+ 
    \beta^{-3/4} |\phi^{\prime\prime}(1)|\,\{\|\hat{g}\|_2+\beta^{-1/4}\}
  + \beta^{-5/8} \|\tilde{v}_\Df^\prime\|_2 \\ + (
  [\beta\Upsilon]^{-3/8}\mu_{+,\beta}^{3/4}+ \beta^{-1/24} )|\phi(x_\nu)|  +
  (\beta^{-3/8}\mu_{+,\beta}^{1/2}+ \beta^{-1/4} ) \|\phi^\prime\|_2 \Big]\,, 
\end{multline}
where
\begin{displaymath}
  \gamma(\alpha,\beta)=
  \begin{cases}
    1 & \alpha<\beta^{1/8} \\
   \Upsilon^{-5/2} \beta^{1/8} & \alpha\geq\beta^{1/8} \,.
  \end{cases}
\end{displaymath}
Substituting \eqref{eq:799},
\eqref{eq:802}, 
\eqref{eq:661}, and \eqref{eq:771} into
\eqref{eq:809}, yields, with the aid of \eqref{eq:562}
\begin{multline*}
  \|v_\Df\|_\infty\lesssim  \beta^{-3/4}\gamma(\alpha,\beta)\|f\|_2
    \\ + (
  [\beta\Upsilon]^{-3/8}\mu_{+,\beta}^{3/4} + \beta^{-1/24} )|\phi(x_\nu)|  
 +  (\beta^{-3/8}\mu_{+,\beta}^{1/2}+ \beta^{-1/4} ) \|\phi^\prime\|_2 \,, 
\end{multline*}
By the assumption on $\mu_+$, we may finally conclude
\begin{multline}
  \label{eq:810}
 \|v_\Df\|_\infty\leq C\, \Big(\gamma(\alpha,\beta)\beta^{-3/4} \|f\|_2+
 [\beta^{-1/4} + \mu_+^{1/2}\beta^{1/8}] \|\phi\|_{1,2}+ [\beta^{-1/24} +
   \mu_{+}^{3/8}\beta^{3/16}] |\phi(x_\nu)|\Big) \,.  
\end{multline}

{\em Step 6: We prove \eqref{eq:788} in the case $\mu >-\mu_0$ and
$\alpha\leq\alpha_0\beta^{1/8}$.}\\

\noindent {\em Step 6a: Preliminaries.}\\

 We continue as in the proof of \cite[Lemma 8.8]{almog2019stability}. We
first write, as in \eqref{eq:673}
\begin{subequations}
\label{eq:811}
\begin{equation}
  \phi = \phi_\Df + \check{\phi} \,,
\end{equation}
where 
\begin{equation}
  \phi_\Df =\A_{\lambda,\alpha}^{-1}v_\Df \quad , \quad
 \check{\phi}=-\A_{\lambda,\alpha}^{-1}\big([U+i\lambda]\phi^{\prime\prime}(1)\hat{\psi} \big)= - \phi^{\prime\prime}(1) \phi_{\lambda,\alpha,\beta}\,. 
\end{equation}
\end{subequations}
 By Propositions \ref{prop:near-quadratic} and \ref{prop:quadratic},
there exists (sufficiently small) $C>0$ so that we can use
\eqref{eq:223}, for $|\mu|\leq C\, x_\nu^2\,$, 
and \eqref{eq:255} for
$|\mu|\geq C x_\nu^2$, both holding for sufficiently small $\mu_0$. Hence, we can
conclude that for any pair $(\tilde v, \tilde \phi)\in W^{1,p}(0,1)\times
D(\A_{\lambda,\alpha})$ satisfying $\tilde v=\A_{\lambda,\alpha}\tilde \phi$,
 \begin{equation}
\label{eq:812} 
  |\tilde  \phi(x_\nu)|^2\leq C\Big(|\mu|^{1/2}\Big|\Big\langle\tilde \phi,\frac{\tilde  v}{U+i\lambda} \Big\rangle\Big| +
  x_\nu\Big\|(1-x)^{1/2}\frac{\tilde  v}{U+i\lambda}\Big\|_1^2\Big)\,.
 \end{equation}
 We apply the above inequality to the pair $( -(U+i\lambda)\phi^{\prime\prime}(1)\hat
 \psi,\check \phi)$ to obtain
\begin{equation*}
  |\check \phi(x_\nu)|^2\leq C |\phi^{\prime\prime}(1)|\, \Big(|\mu|^{1/2}\Big|\Big\langle\check \phi,  \hat \psi  \Big\rangle\Big| +
  x_\nu|\phi^{\prime\prime}(1)| \Big\|(1-x)^{1/2}  \hat \psi |_1^2\Big)\,.
 \end{equation*}
 Given that $\check \phi(1)=0$ we write
 \begin{equation}
\label{eq:813}
   |\langle\phi,\hat{\psi}\rangle|\leq\|\phi^\prime\|_2\|(1-x)^{1/2}\hat{\psi}\|_1\leq C\, [|\lambda|\beta]^{-3/4} \|\phi^\prime\|_2 \,,
 \end{equation}to obtain, with the aid of \eqref{eq:614} and
 \eqref{eq:673},
\begin{equation}
\label{eq:814}
  |\langle\check{\phi},\hat{\psi}\rangle|\leq\|\check{\phi}^\prime\|_2\|(1-x)^{1/2}\hat{\psi}\|_1\leq C\, [|\lambda|\beta]^{-3/4} \|\check{\phi}^\prime\|_2 \,,
\end{equation}
and consequently it holds that
\begin{equation}\label{eq:815}
  |\check{\phi}(x_\nu)|^2 \leq 
  \widehat
  C\,(|\mu|^{1/2}\beta^{-3/4}\|\check{\phi}^\prime\|_2|\phi^{\prime\prime}(1)|+x_\nu\beta^{-3/2}|\phi^{\prime\prime}(1)|^2) \,.
\end{equation}
Combining \eqref{eq:815} with \eqref{eq:799} yields 
 \begin{multline*}
  |\check{\phi}(x_\nu)|^2 \leq 
  \widehat C\,\Big(|\mu|^{1/2}\|\check{\phi}^\prime\|_2+
  x_\nu\big(\beta^{-7/8}\|f\|_2+|\phi(x_\nu)|+\beta^{-1/8}\|\phi^\prime\|_2\big)\Big) \\ \Big(\beta^{-7/8}\|f\|_2+|\phi(x_\nu)|+\beta^{-1/8}\|\phi^\prime\|_2\Big)\,. 
\end{multline*}
Since $|\mu|^{1/2}\|\check{\phi}^\prime\|_2|\phi(x_\nu)|\leq
C\,\big(\delta^{-2}|\mu|\|\check{\phi}^\prime\|_2^2+\delta^2|\phi(x_\nu)|^2\big)$ for any $\delta>0$ and
\break $x_\nu<C\beta^{-1/4}$, we can conclude that there exists $C>0$ such that for
any $\delta>0$
\begin{multline}
  \label{eq:816}
 |\check{\phi}(x_\nu)| \leq 
  C\,\Big(\delta^{-1}(|\mu|^{1/2}+\beta^{-1/8})[\|\check{\phi}^\prime\|_2+\|\phi_\Df^\prime\|_2] \\+
  (\delta+\beta^{-1/8})[|\check{\phi}(x_\nu)|+|\phi_\Df(x_\nu)|]+\delta^{-1}\beta^{-7/8}\|f\|_2 
  \Big)\,.  
\end{multline}

By  applying   to the pair $(\phi_\Df,v_\Df)$ \eqref{eq:241}  for
$U(0)-\kappa_0\beta^{-1/2}\leq\nu\leq U(0)-\kappa_0|\mu|$, \eqref{eq:270}  for
$U(0)-\kappa_0\min(|\mu|,\beta^{-1/2})\leq\nu\leq U(0)+\kappa_0\min(|\mu|,\beta^{-1/2})$,
and \eqref{eq:283} for $U(0)+\kappa_0|\mu|\leq\nu\leq
  U(0)+\kappa_0\beta^{-1/2}$,  we
obtain that
\begin{equation*} 
   |\phi_\Df(x_\nu)| \leq C 
     \Big(1+\log\frac{\max(x_\nu,|\mu|^{1/2})}{|\mu|^{1/2}}\Big)
   \|v_\Df\|_\infty   \,.
\end{equation*}
With the aid of \eqref{eq:810} we then obtain 
\begin{multline*}
   |\phi_\Df(x_\nu)| \leq C \,
     \Big(1+\log\frac{\max(x_\nu,|\mu|^{1/2})}{|\mu|^{1/2}}\Big) \\ \Big(
       \gamma(\alpha,\beta) \beta^{-3/4}\|f\|_2 +  [\beta^{-1/4} +
      \mu_+^{1/2}\beta^{1/8}]\|\phi\|_{1,2}+ [\beta^{-1/24} + \mu_+^{3/8}\beta^{3/16}]|\phi(x_\nu)|\Big) \,,
\end{multline*}
(where $\mu_+$ is given by \eqref{eq:914}),
 from which we conclude by (\ref{eq:811}) that
\begin{multline}
  \label{eq:817}\mu_+
  |\phi_\Df(x_\nu)| \leq C 
     \Big(1+\log\frac{\max(x_\nu,|\mu|^{1/2})}{|\mu|^{1/2}}\Big) \Big(
       \gamma(\alpha,\beta) \beta^{-3/4}\|f\|_2    \\ +
 (\beta^{-1/4}+\mu_+^{1/2}\beta^{1/8}) [\|\check{\phi}^\prime\|_2+\|\phi_\Df^\prime\|_2] \\ +
 ( \beta^{-1/24}+ \mu_+^{3/8}\beta^{3/16})[|\check{\phi}(x_\nu)|+|\phi_\Df(x_\nu)|] 
 \Big) \,.
\end{multline} 
Combining \eqref{eq:816}  and \eqref{eq:817} yields, 
\begin{multline}\label{eq:818}
    |\check{\phi}(x_\nu)|+  |\phi_\Df(x_\nu)| \leq C \, 
     \Big(1+\log\frac{\max(x_\nu,|\mu|^{1/2})}{|\mu|^{1/2}}\Big)\\ \Big(
 [ \beta^{-1/24}+  \mu_+^{3/8}\beta^{3/16} +\delta] [|\check{\phi}(x_\nu)|+
 |\phi_\Df(x_\nu)| ] \\
 +[\beta^{-1/4}+\mu_+^{1/2}\beta^{1/8}+\delta^{-1}|\mu|^{1/2}]
 [\|\check{\phi}^\prime\|_2+\|\phi_\Df^\prime\|_2]+  \gamma(\alpha,\beta)\beta^{-3/4} \|f\|_2\Big) \,.  
\end{multline}

{\em Step 6b: We prove the existence of $\Upsilon >0$, $\mu_0>0$, and $\beta_0 >0$ such that
\eqref{eq:788} holds for $\Upsilon^{-1} \beta^{-1}<\mu<\Upsilon\beta^{-1/2}$ or $
-\mu_0<\mu<- \Upsilon^{-1} \beta^{-1}$, with $\beta \geq \beta_0$.}\\ 

\noindent Let
\begin{displaymath}
 \delta = \hat \delta /
\Big(1+\log \frac{\max(x_\nu,|\mu|^{1/2})}{|\mu|^{1/2}}\Big)\,,
\end{displaymath}
where $\hat{\delta}>0$ is independent of $\beta$.\\
Note that for $\beta^{-1}<|\mu|<x_\nu^2$, we have, for any $s>0$,  
  \begin{equation}\label{eq:819}
  \big( \frac{|\mu|^{1/2}}{x_\nu}\big)^s \Big(1+\log\frac{\max(x_\nu,|\mu|^{1/2})}{|\mu|^{1/2}}\Big)\leq C_s\,.
  \end{equation}
  The above inequality implies (with $s=1$), since $x_\nu\leq
  C\beta^{-1/4}$, the existence of $\beta_0>0$ such that for all $\beta>\beta_0$,
  \begin{subequations}
\label{eq:820}
\begin{equation} 
  \delta^{-1}|\mu|^{1/2}\leq C\beta^{-1/4}\,.
\end{equation} 
Furthermore, \eqref{eq:820} with $s=1/2$ leads to 
\begin{equation}
 \Big(1+\log\frac{\max(x_\nu,|\mu|^{1/2})}{|\mu|^{1/2}}\Big)\,  \mu_+^{3/8}
  \beta^{3/16}  \leq C \mu_+^{1/8} x_\nu^{1/2} \beta^{3/16} \leq \widehat C \Upsilon^{1/8} \,,
\end{equation}
and with $s=1/12$ to
\begin{equation}
 \Big(1+\log\frac{\max(x_\nu,|\mu|^{1/2})}{|\mu|^{1/2}}\Big)\, \beta^{-1/24}
 \leq C_s  \beta^{-1/24} (x_\nu/|\mu|^{1/2})^{s}  \leq \widehat C \beta^{-1/48}
 \,. 
\end{equation}
\end{subequations}
Hence, we obtain from \eqref{eq:818} and \eqref{eq:820}
that, for sufficiently small $\Upsilon$, $\hat \delta$  and $\mu_0$ and sufficiently large
$\beta_0$, 
\begin{multline}
\label{eq:821}
    |\check{\phi}(x_\nu)|+  |\phi_\Df(x_\nu)|
    \leq C \Big(1+\log\frac{\max(x_\nu,|\mu|^{1/2})}{|\mu|^{1/2}}\Big) 
    \Big(  \gamma(\alpha,\beta) \beta^{-3/4} \|f\|_2 + \\+  [\beta^{-1/4}+\mu_+^{1/2}\beta^{1/8} +
     |\mu|^{1/2}]  
    [\|\check{\phi}^\prime\|_2+\|\phi_\Df^\prime\|_2]\Big) \,.
\end{multline}
Next, we apply  (\ref{eq:178}b), for
$U(0)-\kappa_0\beta^{-1/2}\leq\nu\leq U(0)-\kappa_0|\mu|$,  \eqref{eq:242}  for
$U(0)-\kappa_0\min(|\mu|,\beta^{-1/2})\leq\nu\leq U(0)+\kappa_0\min(|\mu|,\beta^{-1/2})$,
and \eqref{eq:271} for $ U(0)+\kappa_0|\mu|\leq\nu\leq
  U(0)+\kappa_0\beta^{-1/2}$, for $p=+\infty$, to the pair $(\phi_\Df, v_\Df)$
to obtain that, for $|\mu|>\Upsilon^{-1}\beta^{-1}$,
\begin{equation} \label{eq:822}
 \Big(1+\log\frac{\max(x_\nu,|\mu|^{1/2})}{|\mu|^{1/2}}\Big) \|\phi_\Df^\prime\|_2 \leq C|\mu|^{-1/4}\|v_\Df\|_\infty \,.
\end{equation}
Note that while applying (\ref{eq:178}b) we have, since $x_\nu^2>\frac 1C
|\mu|$ in this case, that
\begin{equation*}
\frac{\Big[\log\frac{x_\nu}{|\mu|^{1/2}}\Big]^2}{x_\nu^{1/2}}\leq \widehat C\, |\mu|^{-1/4}\,.
\end{equation*}
Then \eqref{eq:810} and \eqref{eq:822} yield with the aid of
(\ref{eq:811}a),
 \begin{multline}
\label{eq:823}
   \|\phi_\Df^\prime\|_2 \leq  C
   \Big(1+\log\frac{\max(x_\nu,|\mu|^{1/2})}{|\mu|^{1/2}}\Big)^{-1}
   \Big(|\mu|^{-1/4}  \gamma(\alpha,\beta) \beta^{-3/4}\|f\|_2 \\ 
+
(|\mu|^{-1/4}\beta^{-1/4}+\mu_+^{1/4}\beta^{1/8})[\|\check{\phi}^\prime\|_2+\|\phi_\Df^\prime\|_2]\\
+ (\beta^{-1/24}|\mu|^{-1/4}+\mu_+^{1/8}\beta^{3/16})
 [|\check{\phi}(x_\nu)|+  |\phi_\Df(x_\nu)|] \Big) \,.  
\end{multline}
Substituting \eqref{eq:821} into \eqref{eq:823} yields
\begin{multline}
  \label{eq:824}
\|\phi_\Df^\prime\|_2 \leq C\, \Big(|\mu|^{-1/4}  \gamma(\alpha,\beta) \beta^{-3/4} \|f\|_2\\ +
\big[|\mu|^{-1/4}\beta^{-1/4}+\mu_+^{1/4}\beta^{1/8}+ |\mu|^{1/4}\big] \big[\|\check{\phi}^\prime\|_2+\|\phi_\Df^\prime\|_2\big]\Big)\,.
\end{multline}

We now use \eqref{eq:254} for the pair
$(\check{\phi},\phi^{\prime\prime}(1)(U+i\lambda)\hat{\psi})$, together with \eqref{eq:814}   to
obtain that
\begin{equation}\label{eq:825}
\|\check{\phi}^\prime\|_2 \leq C \,[\lambda|\beta|]^{-3/4}  |\phi^{\prime\prime}(1)|\,.
\end{equation}
Using \eqref{eq:799}, we deduce from \eqref{eq:825} for sufficiently
large $\beta_0$
\begin{equation}
\label{eq:826}
\|\check{\phi}^\prime\|_2\leq C([|\check{\phi}(x_\nu)|+
|\phi_\Df(x_\nu)|]+\beta^{-1/8}[\|\check{\phi}^\prime\|_2+\|\phi_\Df^\prime\|_2] +\beta^{-7/8}\|f\|_2) \,.
\end{equation}
We then obtain from \eqref{eq:821}
\begin{multline} \label{eq:827}
   \|\check{\phi}^\prime\|_2\leq C
   \Big(1+\log\frac{\max(x_\nu,|\mu|^{1/2})}{|\mu|^{1/2}}\Big)
   \\ \big(\gamma(\alpha,\beta) \beta^{-3/4}\|f\|_2+  [\beta^{-1/4}+\mu_+^{1/2}\beta^{1/8} +
     |\mu|^{1/2}]   [\|\check{\phi}^\prime\|_2+\|\phi_\Df^\prime\|_2]\big)
   \,. 
\end{multline}
To obtain the coefficient of $\|f\|_2$, we set $s=1/2$ in
\eqref{eq:819} to conclude for $|\mu| < x_\nu^2$
\begin{displaymath}
  \big(1+\log\frac{\max(x_\nu,|\mu|^{1/2})}{|\mu|^{1/2}}\big) \leq C
  x_\nu^{1/2} |\mu|^{-1/4} \leq \widetilde C  \beta^{-1/4}  |\mu|^{-1/4} \leq
  \widehat C |\mu|^{-1/4}\,. 
\end{displaymath}
Hence, combining \eqref{eq:827} with \eqref{eq:824} yields
\begin{multline*}
\|\check{\phi}^\prime\|_2+\|\phi_\Df^\prime\|_2  \leq C \Big(|\mu|^{-1/4} 
    \gamma(\alpha,\beta) \beta^{-3/4}\|f\|_2 \\ +  [\beta^{-1/4}+\mu_+^{1/2}\beta^{1/8} +
     |\mu|^{1/2}]  \big(1+\log\frac{\max(x_\nu,|\mu|^{1/2})}{|\mu|^{1/2}}\big)   [\|\check{\phi}^\prime\|_2+\|\phi_\Df^\prime\|_2] \Big)\,.
   \end{multline*} 
   Hence, with the aid of \eqref{eq:820}, we obtain that there
   exist $\Upsilon>0$ and $\beta_0>0$ (so that $\Upsilon +\beta_0^{-1}$ is small enough)
   such that for either $\Upsilon^{-1}\beta^{-1}\leq\mu\leq\Upsilon\beta^{-1/2}$ or $-\mu_0\leq
   \mu\leq-\Upsilon^{-1}\beta^{-1}$ we have
\begin{equation}
\label{eq:828}
  \|\phi^\prime\|_2 \leq
  C\, |\mu|^{-1/4}\gamma(\alpha,\beta)\beta^{-3/4}\|f\|_2 \leq \widehat C
  \gamma(\alpha,\beta)\beta^{-1/2} \|f\|_2 \,. 
\end{equation}
Combined with Poincar\'e's inequality \eqref{eq:828} yields
\eqref{eq:788}. \\

 In the next step we use a shifting argument and hence
it is necessary to obtain first an estimate for $
\|\phi^{\prime\prime}-\alpha^2\phi\|_2$.  By \eqref{eq:821} and the first
  inequality of \eqref{eq:828}, we obtain that 
\begin{equation}
\label{eq:829}
   |\phi(x_\nu)| \leq|\check{\phi}(x_\nu)|+  |\phi_\Df(x_\nu)| \leq  C\gamma(\alpha,\beta)
     \beta^{-3/4} \log\beta\, \|f\|_2\,.
\end{equation}
Substituting the first inequality of  \eqref{eq:828} and \eqref{eq:829}  into \eqref{eq:801}  
yields
\begin{equation}
\label{eq:830}
  \|\tilde{v}_\Df\|_2\leq C\,\gamma(\alpha,\beta)\big( \beta^{-3/8}\log\beta +|\mu|^{-1/4} \beta^{-1/2}\big)\, \|f\|_2\,.   
\end{equation}
Consequently, by  \eqref{eq:649}, \eqref{eq:799}, and
  \eqref{eq:562}
we obtain, from \eqref{eq:828}, \eqref{eq:829} and \eqref{eq:830}, 
\begin{equation}\label{eq:831}
 \|\phi^{\prime\prime}-\alpha^2\phi\|_2\leq C \gamma(\alpha,\beta)\big( \beta^{-1/4}\log\beta +|\mu|^{-1/4} \beta^{-3/8}\big)\    \|f\|_2\,.
 \end{equation}
Hence, we have proven,  under the additional condition that  $\alpha \leq \beta^{1/8}$,  that there exist
$\Upsilon>0$ and $\beta_0>0$   for either  $\Upsilon^{-1}\beta^{-1}\leq\mu \leq\Upsilon\beta^{-1/2}$ or $-\mu_0\leq\mu\leq
-\Upsilon^{-1}\beta^{-1}$
that
\begin{equation}
\label{eq:832}
   \|\phi^{\prime\prime}-\alpha^2\phi\|_2\leq C[\beta^{-1/4}\log \beta +|\mu|^{-1/4}\beta^{-3/8}]\|f\|_2 \,.
\end{equation}

For $\alpha\geq\beta^{1/8}$ \eqref{eq:831} is deficient (in this case
$\gamma(\alpha,\beta)=\Upsilon^{-5/2} \beta^{1/8}$), hence we use \eqref{eq:544} (with $v=\phi^{\prime\prime}-\alpha^2
  \phi$ and $f$ replaced by $f + i \beta U^{\prime\prime} \phi$) instead of
\eqref{eq:799},  to obtain
that 
\begin{displaymath}
  |\phi^{\prime\prime}(1)| \leq C \beta^{ 7/16}(\|\phi\|_2+\beta^{-1}\|f\|_2)  \,.  
\end{displaymath}
Then, with the aid of \eqref{eq:830} and  \eqref{eq:562}  we
establish \eqref{eq:832} for $\alpha\geq\beta^{1/8}$ as well.\\

{\em Step 6c: We prove \eqref{eq:788} for $|\mu|<\Upsilon^{-1}\beta^{-1}$,
where $\Upsilon>0$ has been determined in the previous step.}\\

\noindent Here we use a shifting argument. We begin by writing 
\begin{equation}
\label{eq:833}
  \B_{\lambda+2\Upsilon^{-1}\beta^{-1},\alpha}\phi=f+2\Upsilon^{-1}(\phi^{\prime\prime}-\alpha^2\phi)\,,
\end{equation}
and observe that $\hat \lambda:=\lambda+2\Upsilon^{-1}\beta^{-1}$ satisfies the assumptions of Step 6b. 
We then have by \eqref{eq:832}, with $\lambda$ replaced by $\hat \lambda$, 
\begin{displaymath}  
   \|\phi^{\prime\prime}-\alpha^2\phi\|_2\leq
   C(\beta^{-1/8}+ \hat \mu^{-1/4}\beta^{-3/8})[\|\phi^{\prime\prime}-\alpha^2\phi\|_2+\|f\|_2])\,.
\end{displaymath}
Consequently,
\begin{displaymath}
   \|\phi^{\prime\prime}-\alpha^2\phi\|_2\leq C\beta^{-1/8} \|f\|_2\,.
\end{displaymath}
We now apply \eqref{eq:828} to \eqref{eq:833} to obtain, with the aid
of the above inequality,
\begin{displaymath}
   \|\phi^\prime\|_2 \leq C \gamma (\alpha,\beta) \beta^{-1/2}
   (\|f\|_2+ \|\phi^{\prime\prime}-\alpha^2\phi\|_2)  \leq \widehat C \gamma(\alpha,\beta) \beta^{-1/2}\|f\|_2
\end{displaymath}
Combining the above with Poincar\'e's inequality  yields
\eqref{eq:788}.\\

{\em Step 7:  The case $\mu \leq -\mu_0$\,.}\\ 

\noindent Here we use \eqref{eq:520}, applied to the pair $(\phi^{\prime\prime} -\alpha^2 \phi ,f + i
\beta U^{\prime\prime} \phi)$, and \eqref{eq:6} to obtain that
\begin{equation}\label{eq:834}
  |\phi^{\prime\prime}(1)|\leq C(\beta^{1/2}\|\phi\|_2 + \beta^{-1/2}\|f\|_2)\,.
\end{equation}
We then use \eqref{eq:167} for the pair
$(\check{\phi},\phi^{\prime\prime}(1)(U+i\lambda)\hat{\psi})$, together with
\eqref{eq:614} for $s=1/2$ to conclude
\begin{equation}\label{eq:835}
\|\check{\phi}^\prime\|_2 \leq C \,|\lambda\beta|^{-3/4}  |\phi^{\prime\prime}(1)|\,.
\end{equation}
From \eqref{eq:834} and \eqref{eq:835},  we obtain, with the aid of
Poincar\'e's inequality, for sufficiently large $\beta_0$ 
\begin{equation}
\label{eq:836}
      \|\check{\phi}^\prime\|_2\leq
      C\, \big(\beta^{-1/4} \|\phi^\prime_\Df \|_2+\beta^{-5/4}\|f\|_2\big) \,.
\end{equation} 
To estimate $\|\phi^\prime_\Df \|_2$ we apply \eqref{eq:167} 
to the pair $(\phi_\Df,v_\Df)$ 
to obtain 
\begin{equation}
\label{eq:837}
   \|\phi_\Df^\prime\|_2 \leq C \, \|v_\Df\|_2  \,.
\end{equation}

We now use  \eqref{eq:650}, \eqref{eq:834},
\eqref{eq:661},  and \eqref{eq:364}, applied to the pair \break
$(\tilde{v}_{\Df}, i\beta U^{\prime\prime} \phi - f + \phi^{\prime\prime}(1)\hat{g})$, to
obtain,   that
\begin{equation}
\label{eq:838}
  \|\tilde{v}_\Df^\prime\|_2\leq
  C\, |\mu|^{-1/2}\big(\beta^{1/2}\|\phi^\prime\|_2+\beta^{-1/2}\|f\|_2\big)\,.
\end{equation}
By \eqref{eq:363} applied to the pair $(v_\Df,g_\Df)$ 
and   \eqref{eq:646}  it holds that
\begin{multline}\label{eq:928-aa}
  \|v_\Df\|_2\leq\frac{C}{\beta|\mu|}\Big( (1+|\mu|)\big(\|f\|_2 +
  |\phi^{\prime\prime}(1)|\,\|\hat{g}\|_2\big)  +\|\tilde{v}_\Df^\prime\|_2  +    \|\tilde{v}_\Df\|_2+ \|\phi^{\prime\prime}\|_2 + \|\phi\|_{1,2}
 \Big)\,. 
\end{multline}
 To bound $\|\phi^{\prime\prime}\|_2$ we use the identity
\begin{equation}\label{eq:839}
  \|\phi^{\prime\prime}\|_2^2+\alpha^2\|\phi^\prime\|_2^2 = \langle\phi^{\prime\prime},\phi^{\prime\prime}-\alpha^2\phi\rangle
  \,,
\end{equation}
to obtain,  with the aid of  \eqref{eq:649}, 
\begin{equation}
\label{eq:840}
 \|\phi^{\prime\prime}\|_2\leq   \|\tilde{v}_\Df\|_2+ |\phi^{\prime\prime}(1)|\,\|\hat{\psi}\|_2\,.
\end{equation} 
Consequently, we obtain from the substitution of \eqref{eq:661},
\eqref{eq:838}, \eqref{eq:840}, and \eqref{eq:562}, into
\eqref{eq:928-aa}
\begin{equation}
\label{eq:841}
    \|v_\Df\|_2  \leq  C\big(\beta^{-1/2}\|\phi^\prime\|_2+ \beta^{-1}\|f\|_2\big) \,. 
\end{equation}
Next we apply \eqref{eq:837} and \eqref{eq:841} to obtain  
\begin{displaymath}
   \|\phi_\Df^\prime\|_2 \leq C\|v_\Df\|_2\leq \widehat C\big(\beta^{-1/2}\|\phi^\prime\|_2+ \beta^{-1}\|f\|_2\big) \,. 
\end{displaymath}
Combining the above inequality with \eqref{eq:836} yields
\begin{equation}\label{eq:842}
\|\phi^\prime\|_2\leq C\, \beta^{-1}\|f\|_2\,. 
\end{equation}
The above inequality, combined with Poincar\'e's inequality,  completes the proof of 
\eqref{eq:844}. 
\end{proof}

\subsection{Resolvent estimates  for $ \beta^{-1/2} \ll
  U(0)-\nu<U(0)-U(1/2)$} 
\label{sec:5.7}
 We now consider the case where $ \beta^{-1/4}\ll x_\nu<1/2$.  More precisely, given some  positive
 $\Upsilon$ and $ \nu_1 < U(1/2)$, we consider for suitable $\nu_2>0$ and $\beta_0$ the zone
  \begin{multline}
  \label{eq:843}
\mathcal C_1(\nu_1,\nu_2, \Upsilon, \alpha_1, \beta_0)=
\begin{Bmatrix}
  (\lambda, \alpha, \beta) \in \mathbb
C \times \mathbb R_+^2 \,, \beta \geq \beta_0\,,\, \mu < \Upsilon \beta^{-1/2} \\\nu_1 \leq
\nu \leq U(0) -\nu_2 \beta^{-1/2}\,,\, 0\leq \alpha \leq \alpha_1 \beta^{1/3}
\end{Bmatrix}\,, 
 \end{multline}
 for some sufficiently small $\alpha_1>0$.

\begin{proposition}
\label{lem:nearly-quadratic}
Let $U\in C^4([0,1])$ satisfy \eqref{eq:10}. Let further  $\nu_1
< U(1/)2$ denote a positive constant. Then, there exist  $\Upsilon>0$,
$\alpha_1>0$, $\beta_0>0$, $\nu_2>0$ and $C>0$, such that for $ (\lambda, \alpha, \beta)
\in  \mathcal C_1(\nu_1,\nu_2, \Upsilon, \alpha_1, \beta_0)$ it holds that 
   \begin{equation}
\label{eq:844}
 \big\|(\B_{\lambda,\alpha,\beta}^{\mathfrak N,\Df})^{-1}\big\|+
       \big\|\frac{d}{dx}\, (\B_{\lambda,\alpha,\beta}^{\mathfrak
         N,\Df})^{-1}\big\|\leq C\,  \beta^{-1/2}\log \beta \,.
   \end{equation}
 \end{proposition}
 \begin{proof}~
  We refer to the notation introduced in \eqref{eq:644}-\eqref{eq:650}
  for $v_\Df$, $\tilde v_\Df$ and $g_\Df$.\\ 

{\em Step 1: We estimate $\tilde v_\Df$ and $\tilde v^\prime_\Df$ in $L^2$   for $\mu>-\mu_0$, for some,
sufficiently small, $\mu_0>0$.}\\  

\noindent For the
   convenience of the reader we repeat here once again \eqref{eq:651}
\begin{multline*}
  \Re\langle(U^{\prime\prime})^{-1}\tilde{v}_{\Df},(\LL_\beta^\Nf-\beta\lambda)\tilde{v}_{\Df}+i\beta U^{\prime\prime}\phi\rangle\\ =
  \|(U^{\prime\prime})^{-1/2}\tilde{v}_\Df^\prime\|_2^2
  +\Re\langle\big((U^{\prime\prime})^{-1}\big)^\prime\tilde{v}_\Df,\tilde{v}_\Df^\prime\rangle
 -\beta \mu\, \|\tilde{v}_\Df\|_2^2 + \beta
\Re\langle\phi^{\prime\prime}(1)\hat{\psi},i\phi\rangle\,.
\end{multline*}
As in \eqref{eq:798} we obtain that
\begin{equation}\label{eq:845}
  \beta\,
  \Re\langle\phi^{\prime\prime}(1)\hat{\psi},i\phi\rangle\geq-C\beta^{-3/4}|\phi^{\prime\prime}(1)|(\|\tilde{v}_\Df^\prime\|_2+  \beta^{2/3} \|\phi^\prime\|_2)\,.
\end{equation}
 Let $x_\nu$ be defined by \eqref{eq:defxnu}.   Since assumption
we have $ x_\nu\geq C\nu_2^{1/2}\beta^{-1/4}$ it holds by
\eqref{eq:768} that,  for sufficiently large $\nu_2$
\begin{multline}
  \label{eq:846}
   |\phi^{\prime\prime}(1)|\leq
    C\Big(\beta^{-1/3}x_\nu^{-5/6}\|f\|_2+  \\ \beta^{1/2}x_\nu^{-1}
      \log(\beta^{1/4}x_\nu) |\phi(x_\nu)|+\beta^{1/2}x_\nu^{-1/2}\|\phi^\prime\|_2\Big)\,. 
\end{multline}
 For $\alpha>\beta^{1/6}$ we use \eqref{eq:544} for the pair $(\phi^{\prime\prime} -\alpha^2 \phi ,f + i
\beta U^{\prime\prime} \phi)$ to obtain, with the aid of Poincar\'e's inequality, that
\begin{equation}
\label{eq:847}
   |\phi^{\prime\prime}(1)|\leq C\alpha^{-1/2}(\beta^{1/2}\|\phi^\prime\|_2+\beta^{-1/2}\|f\|_2)
\end{equation}
By substituting \eqref{eq:846} into \eqref{eq:845}, we get
\begin{multline}
\label{eq:848}
 \beta\, \Re\langle\phi^{\prime\prime}(1)\hat{\psi},i\phi\rangle\geq
 -C\,\Big(\beta^{-1/4}[x_\nu^{-1/2}\|\phi\|_{1,2}+ 
 x_\nu^{-1}  \log(\beta^{1/4}x_\nu) |\phi(x_\nu)| \\
 +\beta^{-13/12}x_\nu^{-5/6}\|f\|_2\Big)\,\Big(\|\tilde{v}_\Df^\prime\|_2+
  \beta^{2/3} \|\phi^\prime\|_2\Big)\,.
\end{multline}
Next, by \eqref{eq:650}, \eqref{eq:661}, \eqref{eq:846},
(\ref{eq:287}a), and \eqref{eq:412}, we obtain,  for the
  parameter range set in \eqref{eq:843}, 
 that (compare with
\eqref{eq:801})
\begin{equation}
\label{eq:849}
  \|\tilde{v}_\Df\|_2\leq
  C\,\big(x_\nu^{-1}\|\phi^\prime\|_2+\beta^{1/6}x_\nu^{-5/6}|\phi(x_\nu)|+[\beta x_\nu]^{-2/3}\|f\|_2\big)\,.
\end{equation}
Substituting \eqref{eq:849}  together with  \eqref{eq:848}, \eqref{eq:661},
and \eqref{eq:846}  into \eqref{eq:662} 
yields, 
with the aid of \eqref{eq:650}
\begin{multline*}
  \frac 1C    \|\tilde{v}_\Df^\prime\|_2^2 \leq 
     \big(x_\nu^{-1}\|\phi^\prime\|_2+\beta^{1/6}x_\nu^{-5/6}|\phi(x_\nu)|+[\beta x_\nu]^{-2/3}\|f\|_2\big)\\\quad  \qquad \big(\|f\|_2
   +\beta^{1/4}x_\nu^{-1/2}\|\phi^\prime\|_2+\beta^{1/4}x_\nu^{-1} \log(\beta^{1/4}x_\nu)|\phi(x_\nu)|\big)\\
  \qquad  +  \max(1,\mu \beta)
   \big(x_\nu^{-1}\|\phi^\prime\|_2+\beta^{1/6}x_\nu^{-5/6}|\phi(x_\nu)|+
   [\beta x_\nu]^{-2/3}\|f\|_2)^2\\ +
   (\beta^{-1/4}[ x_\nu^{-1/2}\|\phi\|_{1,2}+
     x_\nu^{-1}  \log(\beta^{1/4}x_\nu) |\phi(x_\nu)|]+\beta^{-13/12}x_\nu^{-5/6}\|f\|_2\big)\\ \big(\|\tilde{v}_\Df^\prime\|_2+  \beta^{2/3}\|\phi^\prime\|_2\big)\,.   
\end{multline*}
 Hence, 
\begin{multline}
\label{eq:850}
  \|\tilde{v}_\Df^\prime\|_2\leq C\, \Big[\big(  \beta^{1/8}
 +  \mu_+^{1/2}\beta^{1/2} \big)\big(x_\nu^{-1} \|\phi^\prime\|_2+\beta^{1/6}x_\nu^{-5/6}|\phi(x_\nu)|\big)\\
  +\beta^{5/24}x_\nu^{-1/4}\|\phi^\prime\|_2 +  \beta^{-1/8}\|f\|_2 \Big]\,.
\end{multline}

{\em Step 2:  We  prove that
\begin{multline}
  \label{eq:851}
 \|\phi^{\prime\prime}-\phi^{\prime\prime}(1)\hat{\psi}\|_2 \leq
   C \, \Big(\|\tilde{v}_\Df\|_2+ \\ \beta^{-11/12}x_\nu^{-5/6}\|f\|_2 +\beta^{-1/12}x_\nu^{-1}\log(\beta^{1/4}x_\nu)|\phi(x_\nu)|\big)\Big) \,.
\end{multline}
}

{\em Step  2.1:  Prove \eqref{eq:851} in the case $0 \leq
\alpha\leq\beta^{1/6}$.}\\

\noindent We write, as in
Step 5 of the proof of Proposition \ref{lem:quadratic} and with the
aid of \eqref{eq:846}  and \eqref{eq:614} (with $s=1/2$)
\begin{multline}\label{eq:850+}
  |\langle \phi^{\prime\prime}(1)\hat{\psi},\alpha^2\phi\rangle|\leq
  \alpha^2|\phi^{\prime\prime}(1)|\,\|(1-x)^{1/2}\hat{\psi}\|_1\,\|\phi^\prime\|_2\leq\\
  C\alpha^2\,\big(\beta^{-13/12}x_\nu^{-5/6}\|f\|_2+\beta^{-1/4}[x_\nu^{-1}{
    \log(\beta^{1/4}x_\nu)}|\phi(x_\nu)|\\ +x_\nu^{-1/2}\|\phi^\prime\|_2]\big)\|\phi^\prime\|_2\,.
\end{multline}
Substituting \eqref{eq:850+} into \eqref{eq:804} 
yields 
\begin{multline*}
  \|\phi^{\prime\prime}-\phi^{\prime\prime}(1)\hat{\psi}\|_2^2+\alpha^2\|\phi^\prime\|_2^2
  \\ \leq\|-\phi^{\prime\prime}+\phi^{\prime\prime}(1)\hat{\psi}\|_2\,\|\tilde{v}_\Df\|_2 \qquad\qquad\qquad\qquad\qquad\qquad\qquad \\
  +C\alpha^2\big(\beta^{-13/12}x_\nu^{-5/6}\|f\|_2+\beta^{-1/4}[x_\nu^{-1}{ \log(\beta^{1/4}x_\nu)}|\phi(x_\nu)|+x_\nu^{-1/2}\|\phi^\prime\|_2]\big)\|\phi^\prime\|_2\,.
\end{multline*}
 For sufficiently large $\beta_0$ we then obtain,  as $x_\nu>\beta^{-1/4}$,
\begin{displaymath}
   \|\phi^{\prime\prime}-\phi^{\prime\prime}(1)\hat{\psi}\|_2 \leq
   C \, \Big(\|\tilde{v}_\Df\|_2+\alpha\big(\beta^{-13/12}x_\nu^{-5/6}\|f\|_2 +\beta^{-1/4}x_\nu^{-1} \log(\beta^{1/4}x_\nu) |\phi(x_\nu)|\big)\Big) \,.
\end{displaymath}
For $\alpha<\beta^{1/6}$  \eqref{eq:851} is readily verified.\\

{\em Step  2.2: We prove \eqref{eq:851} in the case  $\beta^{1/6}\leq\alpha\leq\alpha_1\beta^{1/3}$.}\\

\noindent In this case, we obtain, instead of \eqref{eq:850+},  with the aid of
\eqref{eq:847} and  again \eqref{eq:614} (with $s=1/2$) that 
\begin{displaymath}
   |\langle \phi^{\prime\prime}(1)\hat{\psi},\alpha^2\phi\rangle|\leq
   C\alpha^{3/2}\,\big(\beta^{-5/4}\|f\|_2+\beta^{-1/4}\|\phi^\prime\|_2\big)\|\phi^\prime\|_2\,. 
\end{displaymath}
Substituting the above into  \eqref{eq:804} 
yields, as above
\begin{displaymath}
  \|\phi^{\prime\prime}-\phi^{\prime\prime}(1)\hat{\psi}\|_2 \leq
   C \, \Big(\|\tilde{v}_\Df\|_2+ \beta^{-7/6}\|f\|_2 )\,.
\end{displaymath}
Consequently, \eqref{eq:851} is valid also for
$\beta^{1/6}\leq\alpha<\alpha_1\beta^{1/3}$. \\

{\em Step 3. We estimate $v_\Df$ in $L^\infty$ for $\mu \geq -\mu_0$.}\\

Let $v_\Df$ be given by \eqref{eq:644}. Recall that by (\ref{eq:806})
$(\mathcal L_\beta -\beta \lambda)v_\Df= g_\Df$.  To obtain an estimate for
$\|v_\Df\|_\infty$ we begin, as in Step 6 in Proposition
\ref{lem:quadratic}, by rewriting (\ref{eq:806}b) as the sum of five
terms
\begin{multline}\label{eq:852}
   g_\Df=\big[ (U-\nu ) ( -f + \phi^{\prime\prime}(1)\hat{g})\big]   + \big[ i \mu ( -f + \phi^{\prime\prime}(1)\hat{g}) \big] 
  \\ - \big[ 2(U^\prime -U^\prime(x_\nu)) \tilde{v}_\Df^\prime\big] + \big[  -2
  U^\prime(x_\nu)  \tilde{v}_\Df^\prime  -U^{\prime\prime}(
  [\phi^{\prime\prime}-\phi^{\prime\prime}(1)\hat{\psi}]+\tilde{v}_\Df)  \\
 -  2U^{(3)}\phi^\prime -   U^{(4)}\phi \big] - \big[ U^{\prime\prime}\phi^{\prime\prime}(1)\hat{\psi}\big]\,.
\end{multline}
We separately estimate the contribution of each term on the
right-hand-side of \eqref{eq:852}, using the interpolation inequality
\eqref{eq:inter}.  For the first term on the right-hand-side of
\eqref{eq:852}, we apply \eqref{eq:432} with $f$ replaced by $-f +
\phi^{\prime\prime}(1)\hat{g}$. For the second term, we use (\ref{eq:287}a), and
(\ref{eq:287}b) with $p=2$ (both valid for sufficiently large $\nu_2$)
for $- \beta^{-1/2} \leq \mu < \Upsilon \beta^{-1/2} $ and
\eqref{eq:363}-\eqref{eq:364} for the case $\mu < - \beta^{-1/2}$.  For
the third term we use \eqref{eq:412} and \eqref{eq:413}.  For the
fourth term we use \eqref{eq:287} to obtain
\begin{displaymath}
  \|(\LL_\beta^{\Nf,\Df} -\beta\lambda)^{-1} (U^\prime(x_\nu) \tilde{v}_\Df^\prime)\|_\infty  \leq
  C\, x_\nu[\beta x_\nu]^{-1/2} \|\tilde{v}_\Df^\prime\|_2  \,.
\end{displaymath}
Finally, for the fifth term we use \eqref{eq:560} as in the proof of
(\ref{eq:807}d).
Combining the above yields for $x_\nu\geq C\nu_2^{1/2}\beta^{-1/4}$  
\begin{multline*}  \|v_\Df\|_\infty\leq
  C\,  \Big[\beta^{-3/4}(\|f\|_2+|\phi^{\prime\prime}(1)|\,\|\hat{g}\|_2) \\
 +[\beta x_\nu]^{-1/2}(\|\phi\|_{1,2}+\|\phi^{\prime\prime}-\phi^{\prime\prime}(1)\hat{\psi}\|_2+\|\tilde{v}_\Df\|_2)
  +\beta^{-1/2}x_\nu^{1/2}\|\tilde{v}_\Df^\prime\|_2  + \beta^{-1} |\phi^{\prime\prime}(1)| \Big]\,.
\end{multline*}
 Using \eqref{eq:851} we then obtain
\begin{multline}\label{eq:853}
  \|v_\Df\|_\infty\leq
  C\,  \Big[\beta^{-3/4}\big(\|f\|_2+|\phi^{\prime\prime}(1)|\,\|\hat{g}\|_2\big) \\
 +[\beta x_\nu]^{-1/2}\big(\|\phi\|_{1,2}+\|\tilde{v}_\Df\|_2+  \beta^{-1/12} x_\nu^{-1} \log(\beta^{1/4}x_\nu)\, |\phi(x_\nu)|\big) 
  \\ +\beta^{-1/2}x_\nu^{1/2}\|\tilde{v}_\Df^\prime\|_2  + \beta^{-1} |\phi^{\prime\prime}(1)| \Big]\,.
\end{multline}
For the estimate of $\|\tilde{v}_\Df\|_2$ we use a combination of
\eqref{eq:771} and \eqref{eq:849}.\\ Let $\check x_\mu:=\check x_\mu(\Upsilon,\beta)$ be defined by \eqref{eq:769}. 
Then, there exists $\Upsilon>0$
such that
\begin{equation}\label{eq:822-829}
  \|\tilde{v}_\Df\|_2\leq C
  \begin{cases}
 \Big[ \Upsilon^{-5/2}\beta^{-1/4}\|f\|_2 
  + ( \Upsilon^{-3/8}\mu_{+,\beta}^{ 3/4}+\Upsilon^{3/2}\beta^{3/16})|\phi(x_\nu)|  & \\
  \hspace*{11em} +
  (\mu_{+,\beta}^{1/2}+\Upsilon^{2}\beta^{1/8}) \|\phi^\prime\|_2 \Big] &
  x_\nu<\check x_\mu  
   \\
\big(x_\nu^{-1}\|\phi^\prime\|_2+\beta^{1/6}x_\nu^{-5/6}|\phi(x_\nu)|+[\beta
x_\nu]^{-2/3}\|f\|_2\big) & \text{otherwise}\,.
  \end{cases}
\end{equation}
Thus, we set
\begin{displaymath}
  \gamma(x_\nu,\check x_\mu )=
  \begin{cases}
    1 & x_\nu\geq \check x_\mu  \\
    0 & x_\nu<\check x_\mu  
  \end{cases}
\end{displaymath}
Substituting  \eqref{eq:846}, \eqref{eq:850},  \eqref{eq:822-829}, 
\eqref{eq:562}, and \eqref{eq:660} into \eqref{eq:853} yields
\begin{multline}
\label{eq:854}
 \|v_\Df\|_\infty \leq
 C\, \Big( \beta^{-5/8}(x_\nu^{1/2}+ \Upsilon^{-5/2}\beta^{-1/8}x_\nu^{-1/2})\|f\|_2+ \\
 [ \gamma(x_\nu,\check x_\mu ) \beta^{-1/3}x_\nu^{-4/3}+ \beta^{-5/24}x_\nu^{-1/3}+
  \mu_+^{1/2}\beta^{1/6}x_\nu^{-1/3}]|\phi(x_\nu)| \\
+[  \gamma(x_\nu,\check x_\mu ) \beta^{-1/2}x_\nu^{-3/2} +\beta^{-3/8}x_\nu^{-1/2}+ \beta^{-7/24}x_\nu^{1/4}+ 
  \mu_+^{1/2}x_\nu^{-1/2}]\|\phi\|_{1,2}  \Big)
 \,.  
\end{multline}

We now obtain an estimate of $|\phi(x_\nu)|$ (see \eqref{eq:857} below)\\

{\em Step 4:
With,  as in \eqref{eq:811}
\begin{equation}
\label{eq:855}
  \phi = \phi_\Df + \check{\phi} \,,
\end{equation}
\begin{displaymath}
  \phi_\Df =\A_{\lambda,\alpha}^{-1}v_\Df \quad ; \quad
 \check{\phi}=-\A_{\lambda,\alpha}^{-1}\big([U+i\lambda]\phi^{\prime\prime}(1)\hat{\psi} \big) \,,
\end{displaymath}
and 
\begin{multline}\label{eq:856}
\CC_2(\nu_1,\nu_2, \Upsilon, \alpha_1, \beta_0,\kappa_0)
:= \\\{(\lambda, \alpha,\beta)\in \mathcal C_1 \,|\, \beta^{-2}\leq|\mu| \leq \Upsilon\beta^{-1/2} \mbox{ or }   -\frac{|U(0)-\nu|}{\kappa_0}<\mu<-
\Upsilon\beta^{-1/2}\}\,,
\end{multline} 
 we prove   that 
there exist $C>0$,  $\Upsilon_0$ and $\hat \kappa_0$ such
that, for $0 < \Upsilon \leq \Upsilon_0$ and  $\kappa_0\geq \hat \kappa_0$, there exist 
$\hat \nu_2=\hat \nu_2(\Upsilon,\kappa_0)$ and $\beta_0=\beta_0(\Upsilon,\kappa_0)$ so that for
$\nu_2 \geq \hat \nu_2$ and $(\lambda, \alpha,\beta)\in \CC_2(\nu_1,\nu_2, \Upsilon, \alpha_1, \beta_0,\kappa_0)$ 
we have 
\begin{multline}
\label{eq:857}
    |\check{\phi}(x_\nu)|+  |\phi_\Df(x_\nu)| \leq \\
    \leq C \,  \log(|\mu|^{-1/2}x_\nu)\, \big([\beta^{-1/4}+\gamma(x_\nu,\check x_\mu ) \beta^{-1/2}x_\nu^{-3/2} +\mu_+^{1/2}x_\nu^{-1/2}]\|\phi^\prime\|_2 \\ 
  + \beta^{-5/8}[x_\nu^{1/2} + \Upsilon^{-5/2}  \beta^{-1/8}x_\nu^{-1/2}]\|f\|_2\big) \,.
\end{multline}}
\vspace*{1ex}

 By \eqref{eq:223} applied to the pair
$(\check \phi, \big([U+i\lambda]\phi^{\prime\prime}(1)\hat{\psi} \big)$,  which holds for
for $(\lambda, \alpha,\beta)\in \CC_2(\nu_1,\nu_2, \Upsilon, \alpha_1, \beta_0,\kappa_0)$ when $\nu_2\geq \Upsilon \kappa_0$, 
  \eqref{eq:846}, and \eqref{eq:614}  for $s=1/2$\,, 
we obtain that
\begin{multline}
\label{eq:858}
 |\check{\phi}(x_\nu)| \leq C\, x_\nu^{1/2} |\phi^{\prime\prime}(1)|
  \|(1-x)^{1/2}\hat{\psi}\|_1 
\leq \widehat C \, \big(\beta^{-1/4}\|\phi^\prime\|_2 +\\
  \beta^{-1/4}x_\nu^{-1/2}  \log(\beta^{1/4}x_\nu) |\phi(x_\nu)|+\beta^{-13/12}x_\nu^{-1/3}\|f\|_2\big)   \,.
\end{multline}
By \eqref{eq:854}
and (\ref{eq:241}) 
applied to the pair $(\phi_\Df,v_\Df)$, we obtain  that 
\begin{multline*} 
   |\phi_\Df(x_\nu)| \leq C\,  \log\Big(\frac{x_\nu}{|\mu|^{1/2}}\Big) \Big( \beta^{-5/8}(x_\nu^{1/2}+ \Upsilon^{-5/2} \beta^{-1/8}x_\nu^{-1/2})\|f\|_2+ \\
 [ \gamma(x_\nu,\check x_\mu ) \beta^{-1/3}x_\nu^{-4/3}+ \beta^{-5/24}x_\nu^{-1/3}+ 
  \mu_+^{1/2}\beta^{1/6}x_\nu^{-1/3}]|\phi(x_\nu)| \\
+[  \gamma(x_\nu,\check x_\mu ) \beta^{-1/2}x_\nu^{-3/2} +\beta^{-3/8}x_\nu^{-1/2}+ \beta^{-7/24}x_\nu^{1/4}+ 
  \mu_+^{1/2}x_\nu^{-1/2}]\|\phi\|_{1,2}  \Big)
 \,.  
\end{multline*}
Combining the above yields, using the fact
that $\beta^{-1/4}<x_\nu<1$  and that $|\mu|<\mu_0$
\begin{multline}\label{eq:859}
    |\check{\phi}(x_\nu)|+  |\phi_\Df(x_\nu)| \leq\\
   \leq  C\,  \log(|\mu|^{-1/2}x_\nu) \Big([ \beta^{-1/4} +\gamma(x_\nu,\check x_\mu ) \beta^{-1/2}x_\nu^{-3/2} +\mu_+^{1/2}x_\nu^{-1/2}]\|\phi^\prime\|_2+\\ 
  + [ \beta^{-1/4}x_\nu^{-1/2}  \log(\beta^{1/4}x_\nu) + \gamma(x_\nu,\check x_\mu ) \beta^{-1/3}x_\nu^{-4/3}+ \beta^{-5/24}x_\nu^{-1/3}+ 
  \mu_+^{1/2}\beta^{1/6}x_\nu^{-1/3}]|\phi(x_\nu)| \\ + \beta^{-5/8}[x_\nu^{1/2}+\Upsilon^{-5/2}\beta^{-1/8}x_\nu^{-1/2}]\|f\|_2\Big)  \,.  
\end{multline}

We proceed by showing that the coefficient of $|\phi(x_\nu)|$ on the
right hand side of \eqref{eq:859} can be made arbitrarily small by
choosing $\nu_2$ sufficiently large. We separately bound each term in
the brackets. 

  For the second term 
  we first observe that for $  x_\nu>\Upsilon\mu_+^{-1/2}\beta^{-1/2}$, we have 
 \begin{multline*} 
   \log(|\mu|^{-1/2}x_\nu) \, \beta^{-1/3}x_\nu^{-4/3}\leq
   \log(\Upsilon^{-1}x_\nu^2\beta^{1/2})[\beta^{1/2}x_\nu^2]^{-2/3}\\
   \leq
   [ \log(x_\nu^2\beta^{1/2})+\log(\Upsilon^{-1})][\beta^{1/2}x_\nu^2]^{-2/3}
   \leq C[\nu_2^{-1/3}+\log \Upsilon^{-1}\nu_2^{-2/3}]  \,.     
\end{multline*}
Furthermore,  for $\beta_0(\kappa_0)$ large enough and $(\lambda,\alpha,\beta) \in \CC_2(\nu_1,\nu_2, \Upsilon, \alpha_1, \beta_0,\kappa_0)$,
it holds that
  \begin{equation}
\label{eq:860}
    |\mu|^{-1/2}x_\nu\leq\beta\,.
  \end{equation}
 Hence  for  $x_\nu>\beta^{-1/8}$, 
we have 
\begin{displaymath}
   \log(|\mu|^{-1/2}x_\nu)\beta^{-1/3}x_\nu^{-4/3}\leq\beta^{-1/6}
   \log(|\mu|^{-1/2}x_\nu)  \leq C \beta^{-1/6} \log \beta \,.
\end{displaymath}

Combining the pair of estimates, we obtain
\begin{equation}
\label{eq:861}
  \log(|\mu|^{-1/2}x_\nu) \gamma(x_\nu,\check x_\mu )
  \beta^{-1/3}x_\nu^{-4/3}\leq \widetilde C \big(\nu_2^{-1/3}+\log
  \Upsilon^{-1}\nu_2^{-2/3}+ \beta^{-1/6}\log\beta \big)\,. 
\end{equation}

For the third term, we write, assuming that
$\nu_2 \Upsilon \geq 1$ (which implies \break $ \mu_+^{-1/2}x_\nu \geq  \frac 1C
\nu_2^{1/2} \Upsilon^{1/2}$ for $(\lambda, \alpha,\beta)\in \CC_2(\nu_1,\nu_2, \Upsilon, \alpha_1, \beta_0,\kappa_0)$),  
  \begin{multline*}
     \log(|\mu|^{-1/2}x_\nu)  \mu_+^{1/2}\beta^{1/6}x_\nu^{-1/3}
    = \log(|\mu|^{-1/2}x_\nu) [\mu_+^{-1/2}x_\nu]^{-1/3}
    \mu_+^{1/3}\beta^{1/6}\\  \leq C  \mu_+^{1/3}\beta^{1/6} \leq \widehat C \Upsilon^{1/3}\,.
  \end{multline*}
  For the first  term we  obtain with the aid of \eqref{eq:860}
\begin{displaymath} 
  \log(|\mu|^{-1/2}x_\nu)\beta^{-1/4}x_\nu^{-1/2} \log(\beta^{1/4}x_\nu) \leq C\beta^{-1/8}\log^2\beta \,.
\end{displaymath}
 We can now conclude \eqref{eq:857}.

{\it Step 5:  We prove  that 
there exist $C>0$,  $\Upsilon_0$ and $\hat \kappa_0$ such
that, for $0 < \Upsilon \leq \Upsilon_0$ and  $\kappa_0\geq \hat \kappa_0$, we have for some
$\hat \nu_2=\hat \nu_2(\Upsilon,\kappa_0)$ and $\beta_0=\beta_0(\Upsilon,\kappa_0)$ that,  for  $\nu_2 \geq \hat \nu_2$
and  $(\lambda, \alpha,\beta)\in \CC_2(\nu_1,\nu_2, \Upsilon, \alpha_1, \beta_0,\kappa_0)$,
\eqref{eq:844} holds true. }\\[1.5ex]

We first use (\ref{eq:178}b) (which holds $(\lambda, \alpha,\beta)\in \CC_2 $)
applied to the pair $(\phi_\Df,v_\Df)$ and \eqref{eq:854} to 
obtain that
\begin{multline}\label{eq:862}
   \|\phi_\Df^\prime\|_2 \leq   C\, \log(|\mu|^{-1/2}x_\nu)\Big( \beta^{-5/8}(1+ \Upsilon^{-5/2}  \beta^{-1/8}x_\nu^{-1})\|f\|_2 \\
+  [ \gamma(x_\nu,\check x_\mu ) \beta^{-1/3}x_\nu^{-11/6}+ \beta^{-5/24}x_\nu^{-5/6}+ 
  \mu_+^{1/2}\beta^{1/6}x_\nu^{-5/6}]|\phi(x_\nu)| \\
+[  \gamma(x_\nu,\check x_\mu ) \beta^{-1/2}x_\nu^{-2} +\beta^{-3/8}x_\nu^{-1}+ \beta^{-7/24}x_\nu^{-1/4}+ 
  \mu_+^{1/2}x_\nu^{-1}]\|\phi\|_{1,2}  \Big)
 \,.  
\end{multline} 

Using \eqref{eq:861}, (which holds for $(\lambda, \alpha,\beta)\in \CC_2 $), and
the fact that $\beta^{-1/4}< x_\nu$, we write 
\begin{multline}\label{eq:863}
   \log(|\mu|^{-1/2}x_\nu) \gamma(x_\nu,\check x_\mu ) \beta^{-1/2}x_\nu^{-2}
    \\ \leq \widetilde C \beta^{-1/6}x_\nu^{-2/3} \big(\nu_2^{-1/3}+\log
    \Upsilon^{-1}\nu_2^{-2/3}+ \beta^{-1/6}\log\beta  \big)\\ \leq \widetilde C \big(\nu_2^{-2/3}+\log
    \Upsilon^{-1}\nu_2^{-1}+ \beta^{-1/6}\log\beta  \big) \,.
\end{multline}
Substituting \eqref{eq:857} and \eqref{eq:863} into
\eqref{eq:862}, we obtain   
the existence of 
 $C>0$ and $\beta_0>0$ such that for all $\beta>\beta_0$ we have  
\begin{multline}
\label{eq:864}
\|\phi_\Df^\prime\|_2 \leq C \, \Big(\log (|\mu|^{-1/2}x_\nu)\beta^{-1/2} \|f\|_2+  \\
 \big[ \beta^{-1/6}x_\nu^{-2/3}+\beta^{-1/8}\log\beta  + \nu_2^{ -1/3}+\log
    \Upsilon^{-1}\nu_2^{-1} +\Upsilon^{1/4} \big]\,[\|\check{\phi}^\prime\|_2+\|\phi_\Df^\prime\|_2]\Big)
 \,.
\end{multline}
Note that in the proof we repeatedly  use the inequalities 
\begin{equation}\label{eq:865}
(|\mu|^{-1/2}x_\nu)^{-1/2}\log(|\mu|^{-1/2}x_\nu)\leq1
\end{equation}
(which holds since $|\mu|^{-1/2}x_\nu>1$ for sufficiently large $\nu_2$),
$\frac 1{\widehat C} \nu_2^{1/2} \beta^{-1/4}<x_\nu<1$ (which holds for
some $\widehat C>0$), $ \mu_+\leq \Upsilon \beta^{-1/2}$, and \eqref{eq:860}.

We now apply \eqref{eq:222}, (which holds for $(\lambda, \alpha,\beta)\in \CC_2$
with sufficiently large $\nu_2$) to the pair $(\check \phi,
\big([U+i\lambda]\phi^{\prime\prime}(1)\hat{\psi} \big)$, \eqref{eq:614} for $s=1/2$\,,
and \eqref{eq:846} as in \eqref{eq:858} to obtain that
\begin{equation}\label{eq:866}
  \|\check{\phi}^\prime\|_2\leq 
C\,\big(\beta^{-1/4}x_\nu^{-1/2}\|\phi^\prime\|_2 +
  \beta^{-1/4} x_\nu^{-1}|\phi(x_\nu)|+\beta^{-13/12}x_\nu^{-5/6}\|f\|_2\big)   \,.
\end{equation}
By  \eqref{eq:860}, \eqref{eq:857}, \eqref{eq:863},  and
\eqref{eq:865} we have   
\begin{multline*}
   |\phi(x_\nu)|\leq C \big(\{\beta^{-1/4}\log\beta +x_\nu^{1/2}[\beta^{-1/6}\log\beta
   +\nu_2^{-2/3}+\nu_2^{-1}\log\Upsilon^{-1}] + \Upsilon^{1/4} \beta^{-1/8}  \}\|\phi^\prime\|_2 \\ 
  + \log(|\mu|^{-1/2}x_\nu)\beta^{-5/8}\Upsilon^{-5/2}\|f\|_2\big) \,.
\end{multline*}
Hence, for $s\geq 1/2$, we have 
\begin{multline}\label{eq:867}
  (\beta^{-1/4}
  x_\nu^{-1})^s |\phi(x_\nu)|\leq C\big[(\nu_2^{-2/3} + \Upsilon^{1/4})\beta^{-1/8}
  \|\phi^\prime\|_2 \\
  +\beta^{-5/8} \,\log\beta \, \Upsilon^{-5/2} \nu_2^{-1/2}  \|f\|_2 \big]
\end{multline}
By \eqref{eq:866} and \eqref{eq:867} for $s=1$, we obtain 
\begin{equation}\label{eq:868} 
   \|\check{\phi}^\prime\|_2\leq C( \beta^{-1/2}\log\beta\|f\|_2+\beta^{-1/8}\|\phi^\prime\|_2) \,.
\end{equation}

Combining \eqref{eq:868} with \eqref{eq:864} yields if $(\lambda,
\alpha,\beta)\in \CC_2(\nu_1,\nu_2, \alpha_1,\beta_0, \Upsilon,\kappa_0,) $ that there exist
positive $\beta_0$ and $C$ such that for all $\beta>\beta_0$ (the reader is
referred for a precise description of the parameter space to the
statement of step 4)
\begin{equation}
\label{eq:869}
  \|\phi^\prime\|_2 \leq   \|\check{\phi}^\prime\|_2+\|\phi_\Df^\prime\|_2 \leq C
   \beta^{-1/2}\log\beta\|f\|_2\,. 
\end{equation}
 Together with Poincar\'e's inequality, \eqref{eq:869} implies \eqref{eq:844}
  under the assumptions of Step~4.\\

  To prove \eqref{eq:844} in the next step for $|\mu|<\beta^{-2} $, we
  need to obtain an estimate of $\|\phi^{\prime\prime}-\alpha^2\phi\|_2$ under the
  present condition on $\mu$. By \eqref{eq:867} for $s=5/6$ and
  \eqref{eq:869} we obtain under \eqref{eq:856} that for some $C>0$
  (here and in the sequel the explicit dependence of $ C=C(\kappa_0, \Upsilon)$
  on $(\kappa_0, \Upsilon)$ is of little concern)
\begin{equation}
\label{eq:870}
 x_\nu^{-5/6}  |\phi(x_\nu)|\leq C  \beta^{-5/12} \, \log \beta \,\|f\|_2 \,.
\end{equation}
Substituting \eqref{eq:869} and \eqref{eq:870} into \eqref{eq:849}
yields  that
\begin{displaymath}
  \|\tilde{v}_\Df\|_2\leq
  C\beta^{-1/4}\log\beta\|f\|_2 \,.   
\end{displaymath}
Consequently, by \eqref{eq:846}, \eqref{eq:649}, \eqref{eq:869},
 \eqref{eq:870},  and \eqref{eq:562}, it holds  under the
  assumptions  and conclusions of Steps 4  and 5 that
\begin{multline}
\label{eq:871}
   \|\phi^{\prime\prime}-\alpha^2\phi\|_2  \leq  \|\tilde{v}_\Df\|_2+
   |\phi^{\prime\prime}(1)|\,\|\hat{\psi}\|_2 \leq \\ C\beta^{-1/8}
     [x_\nu\beta^{1/4}]^{-1/6} \log(x_\nu\beta^{1/4}) \log\beta\|f\|_2\leq  C\beta^{-1/8}\log\beta\|f\|_2\,.
\end{multline}

{\it Step 6: We prove \eqref{eq:844} for $|\mu|\leq\beta^{-2}$.} \\ 

\noindent We begin by writing  as in \eqref{eq:833}
\begin{equation}
\label{eq:872}
  \B_{\lambda+2\beta^{-2},\alpha}\phi=f+2\beta^{-1}(\phi^{\prime\prime}-\alpha^2\phi)\,,
\end{equation}
Hence, by \eqref{eq:871} applied 
to the pair $(\phi, f+2\beta^{-1}(\phi^{\prime\prime}-\alpha^2\phi))$ with $\lambda$ replaced by  $\lambda+2\beta^{-2}$, we have  
\begin{displaymath}
   \|\phi^{\prime\prime}-\alpha^2\phi\|_2\leq
   C(\beta^{-9/8}\log\beta\|\phi^{\prime\prime}-\alpha^2\phi\|_2+\beta^{-1/8}\log\beta\|f\|_2)\,.
\end{displaymath}
Consequently, we obtain that
\begin{displaymath}
   \|\phi^{\prime\prime}-\alpha^2\phi\|_2\leq C\beta^{-1/8}\log\beta \|f\|_2\,.
\end{displaymath}
We now apply \eqref{eq:844} to \eqref{eq:872} to obtain
\begin{displaymath}
   \|\phi^\prime\|_2 \leq C\beta^{-1/2}\log\beta(\|f\|_2+\beta^{-1}\|\phi^{\prime\prime}-\alpha^2\phi\|_2)\,. 
\end{displaymath}
Hence, 
\begin{displaymath}
     \|\phi^\prime\|_2 \leq C\beta^{-1/2}\log \beta \|f\|_2 \,.
\end{displaymath}

{\it Step 7: We prove \eqref{eq:844} in  the case $
   -\mu_0\leq\mu<-\frac{U(0)-\nu}{\kappa_0}\,.$}\\

\noindent We note for below that,  under the above assumption, 
\begin{equation} \label{eq:873}
|\mu|>\frac{1}{C(\kappa_0)}\nu_2^{1/2}\beta^{-1/2}\,.
\end{equation}
In this case, we can apply \eqref{eq:530}, by \eqref{eq:655}, to the
pair $(\phi^{\prime\prime}-\alpha^2 \phi, i\beta U^{\prime\prime}\phi+f)$, so that the $L^\infty$ estimate
is applied to $ i\beta U^{\prime\prime}\phi$ and the $L^2$ estimate to $f$. We obtain
\begin{displaymath}
     |\phi^{\prime\prime}(1)|\leq
    C(\beta^{-1/2}|\mu|^{-3/4}\|f\|_2+\beta^{1/2}|\mu|^{-1/2}\|\phi\|_\infty)\,.
\end{displaymath}
Hence, by (\ref{eq:254}) applied to the pair
$(\check{\phi},\phi^{\prime\prime}(1)(U+i\lambda)\hat{\psi})$ (see \eqref{eq:813})
together with \eqref{eq:614} for $s=1/2$ and \eqref{eq:873}, we
obtain
\begin{displaymath}
    \|\check{\phi}^\prime\|_2\leq
    C(|\mu|^{-3/4}\beta^{-5/4}\|f\|_2+\beta^{-1/4}|\mu|^{-1/2}\|\phi\|_\infty) \leq \widehat C (\beta^{-7/8}\|f\|_2+\nu_2^{-1/4}\|\phi^\prime\|_2)  \,.
\end{displaymath}
From the above we conclude, using Poincar\'e's inequality, that for
$\nu_2$ large enough 
\begin{equation}
\label{eq:874}
   \|\check{\phi}^\prime\|_2+  |\check{\phi}(x_\nu)|\leq C (\beta^{-7/8}\|f\|_2+\nu_2^{-1/4}\|\phi^\prime_\Df\|_2)\,.
\end{equation}
We next use  (\ref{eq:270})  for  the pair $(\phi_\Df,v_\Df)$
(see (\ref{eq:673}b) ) to obtain from \eqref{eq:854}
\begin{equation*} 
   |\phi_\Df(x_\nu)| \leq C \|v_\Df\|_\infty \leq  \widehat C\big(\beta^{-5/8}\|f\|_2+ 
\beta^{-1/4}\|\phi^\prime\|_2+\beta^{-1/8}|\phi(x_\nu)|\big)\,.
\end{equation*}
which implies for $\beta \geq \beta_0$ with $\beta_0$ large enough
\begin{equation} \label{eq:875} 
   |\phi_\Df(x_\nu)| \leq C \big(\beta^{-5/8}\|f\|_2+ 
\beta^{-1/4}\|\phi^\prime\|_2+\beta^{-1/8}|\check \phi(x_\nu)|\big)\,.
\end{equation}
In the sequel the explicit dependence on $(\kappa_0, \Upsilon)$ is of
little concern to us and is therefore omitted.\\  
Similarly, we obtain, using (\ref{eq:242}) with
$p=+\infty$ for the pair $(\phi_\Df,v_\Df)$ and \eqref{eq:854}
\begin{equation*} 
\|\phi^\prime_\Df\|_2\leq C|\mu|^{-1/4}\|v_\Df\|_\infty \leq \widehat C \nu_2^{-1/8} \beta^{1/8} \big(\beta^{-5/8}\|f\|_2+ 
\beta^{-1/4}\|\phi^\prime\|_2+\beta^{-1/8}|\phi(x_\nu)|\big)\,.
\end{equation*}
Using Sobolev embedding for $\phi_\Df$, we obtain for sufficiently large
$\nu_2$
\begin{equation}\label{eq:876}
\|\phi^\prime_\Df\|_2\leq C \nu_2^{-1/8}  \big(\beta^{-1/2}\|f\|_2+ 
\beta^{- 1/8}\|\phi^\prime\|_2+ |\check \phi(x_\nu)|\big)\,.
\end{equation}
We may now continue as in the derivation of \eqref{eq:869} 
to obtain  by \eqref{eq:874}, \eqref{eq:875}  and \eqref{eq:876}
for $\nu_2$ and $\beta_0$ large enough and $\beta \geq \beta_0$ 
\begin{equation}
\label{eq:877}
  \|\phi^\prime\|_2 \leq \|\check{\phi}^\prime\|_2+\|\phi_\Df^\prime\|_2 \leq
  C\beta^{-1/2}\|f\|_2\,. 
\end{equation}

{\em Step 8:  The case $\mu < -\mu_0$\,.}\\
The proof of Step 7 in Proposition \ref{lem:quadratic}
furnishes  \eqref{eq:842} without any modification. Making use of
Poincar\'e's inequality we then establish \eqref{eq:844}. 
 \end{proof}

\subsection{Large $d(\Im\lambda,[0,U(0)])$}
\label{sec:5.8}
 In the following we consider the case where $\nu$ lies outside the
interval $[0,U(0)]$. We consider two different regimes: 
\begin{itemize}
\item We begin with
case $\nu-U(0)\gg\beta^{-1/2}$.
\item 
We  then continue by assuming $\beta^{1/3}(-\nu)\gg1$.
\end{itemize}
We begin by introducing for some positive constants  $\alpha_0$,  
$\Upsilon$, $\beta_0$ and $\kappa_1$  the zone
\begin{displaymath}
\mathcal F_1(\alpha_0, \beta_0,\Upsilon,\kappa_1) :=
\begin{Bmatrix}
  (\lambda,\alpha,\beta)\in \mathbb C\times\mathbb R_+^2\,,\,\beta \geq \beta_0\,,\,
  0\leq \alpha\leq\alpha_0\beta^{1/3}\,,\\[1.5ex]  \mu<\Upsilon\beta^{-1/2} \,,\,  U(0)+\kappa_1\beta^{-1/2}\leq\nu 
\end{Bmatrix}
\end{displaymath}

\begin{proposition}
\label{lem:large-nu}
  Let $U\in C^4([0,1])$ satisfy \eqref{eq:10} and  $\alpha_0$ and 
$\Upsilon$  denote positive constants. Then, there
exist $\beta_0>0$, $\kappa_1>0$, and $C>0$, such that, for \break $(\lambda,
\alpha,\beta)\in \mathcal F_1(\alpha_0, \beta_0,\Upsilon,\kappa_1)$, it holds that  
   \begin{equation}
\label{eq:878} 
 \big\|(\B_{\lambda,\alpha,\beta}^{\mathfrak N,\Df})^{-1}\big\|+
       \big\|\frac{d}{dx}\, (\B_{\lambda,\alpha,\beta}^{\mathfrak
         N,\Df})^{-1}\big\|\leq  C\beta^{-1/2}\,.
   \end{equation}
\end{proposition}
\begin{proof}~\\
  {\em Step 1:  For positive $\lambda_0$ and $\tilde{\alpha}_0$, we prove that
  for $\kappa_1$ and $\beta_0$ large enough, \eqref{eq:878} holds true under
  the additional conditions that $|\lambda| \leq \lambda_0$ and
  $\alpha\leq\tilde{\alpha}_0\beta^{1/4}$.}\\

 \noindent  Let $f\in L^2(0,1)$ and $\phi\in D(\B_{\lambda,\alpha,\beta}^{\mathfrak N,\Df})$ satisfy
  \begin{equation}
\label{eq:879}
    \B_{\lambda,\alpha,\beta}^{\mathfrak N,\Df}\phi=f
  \end{equation}
  Taking the scalar product of \eqref{eq:879} with $\frac{\phi}{U+i\lambda}$,
  and integrating by parts yields for the imaginary part
  \begin{multline}
\label{eq:880}
    -\Im\Big\langle\frac{\phi}{U+i\lambda},f\Big\rangle =
    \Im\Big\langle\Big(\frac{\phi}{U+i\lambda}\Big)^{\prime\prime},\phi^{\prime\prime}\Big\rangle\\
   +  \alpha^2\Im\Big\langle\Big(\frac{\phi}{U+i\lambda}\Big)^{\prime},\phi^{\prime}\Big\rangle +
    \beta\Big(\|\phi^\prime\|_2^2 +
    \alpha^2\|\phi\|_2^2+\Re\Big\langle\frac{U^{\prime\prime}\phi}{U+i\lambda},\phi\Big\rangle\Big) \,.
  \end{multline}
By \eqref{eq:10} we obtain the following auxiliary estimate, which we
repeatedly use in the sequel
\begin{equation}
\label{eq:881}
   \Big\|\frac{(U^\prime)^n}{(U+i\lambda)^m}\Big\|_\infty  \leq C
     \Big\|\frac{x^n}{(x^2+\check{\nu})^m}\Big\|_\infty  \leq C \,  \check \nu^{-(m-n/2)} \,, \; 
   \forall m\in\R_+,\;\forall  n\in[0,2m]\,,
\end{equation}
where
\begin{displaymath}
  \check \nu =\nu-U(0)>0\,.
\end{displaymath}
 As
\begin{equation}
\label{eq:882}
  \Big(\frac{\phi}{U+i\lambda}\Big)^{\prime\prime}= -
  2\frac{\phi^\prime U^\prime}{(U+i\lambda)^2}+ \frac{\phi^{\prime\prime}}{U+i\lambda}-
    \frac{\phi U^{\prime\prime}}{(U+i\lambda)^2} +2\frac{\phi(U^\prime)^2}{(U+i\lambda)^3} \,,
\end{equation}
we need to estimate four different terms to obtain a bound for the
first term on the right-hand-side of \eqref{eq:880}. We begin by
writing 
 \begin{displaymath}
  \Big|\Big\langle\frac{\phi^\prime U^\prime}{(U+i\lambda)^2},\phi^{\prime\prime}\Big\rangle\Big|  \leq
  \Big|\Big\langle\frac{\phi^\prime U^\prime}{(U+i\lambda)^2}, \phi^{\prime\prime}-\phi^{\prime\prime}(1)\hat{\psi}\Big\rangle\Big| +\Big|\Big\langle\frac{\phi^\prime U^\prime}{(U+i\lambda)^2},
\phi^{\prime\prime}(1)\hat{\psi}\Big\rangle\Big| \,,
\end{displaymath}
to obtain by \eqref{eq:562} and \eqref{eq:881} 
\begin{displaymath}
   \Big|\Big\langle\frac{\phi^\prime U^\prime}{(U+i\lambda)^2},\phi^{\prime\prime}\Big\rangle\Big|  \leq
   C{\check \nu^{-3/2}}\,\big(\|\phi^{\prime\prime}-\phi^{\prime\prime}(1)\hat{\psi}\|_2+ \beta^{-1/4}|\phi^{\prime\prime}(1)|\big)\,\|\phi^\prime\|_2\,.
\end{displaymath}

The contribution of the second term on the right-hand is obtained as
follows 
 \begin{displaymath}
   \Big|\Big\langle\frac{\phi^{\prime\prime}}{U+i\lambda},
\phi^{\prime\prime}\Big\rangle\Big| \leq \check \nu^{-1}\|\phi^{\prime\prime}\|_2^2\leq C\check \nu^{-1}\big(\|\phi^{\prime\prime}-\phi^{\prime\prime}(1)\hat{\psi}\|_2^2+\beta^{-1/2}|\phi^{\prime\prime}(1)|^2)
\end{displaymath}
To obtain an estimate for the contribution of the third term on the
right-hand-side of \eqref{eq:882} we write
\begin{multline*}
 \Big|\Big\langle\frac{\phi U^{\prime\prime}}{(U+i\lambda)^2},
\phi^{\prime\prime}\Big\rangle\Big| \leq C  \check \nu^{-1} \| (U+i\lambda)^{-1} \phi\|\, \|\phi^{\prime\prime} \| \\ \leq \widehat C  \check \nu^{-1} \| (U+i\lambda)^{-1} \phi\| [\|\phi^{\prime\prime}-\phi^{\prime\prime}(1)\hat{\psi}\|_2+\beta^{-1/4}|\phi^{\prime\prime}(1)|] 
\end{multline*}
A similar estimate is obtained for the last term on the
right-hand-side of \eqref{eq:882} by using \eqref{eq:881} with $m=2,
n=2$, i.e.,
\begin{multline*}
  \Big|\Im\Big\langle\Big(\frac{\phi}{U+i\lambda}\Big)^{\prime\prime},\phi^{\prime\prime}\Big\rangle\Big|\leq 
  C\check \nu^{-1}\Big(\|\phi^{\prime\prime}-\phi^{\prime\prime}(1)\hat{\psi}\|_2^2+\beta^{-1/2}|\phi^{\prime\prime}(1)|^2 \\ \Big[\Big\|\frac{\phi}{U+i\lambda}\Big\|_2+
  \check \nu^{-1/2}\|\phi^\prime\|_2\Big][\|\phi^{\prime\prime}-\phi^{\prime\prime}(1)\hat{\psi}\|_2+\beta^{-1/4}|\phi^{\prime\prime}(1)|] 
\Big)\,,
\end{multline*}
from which  we conclude that
\begin{multline} \label{eq:883}
  \Big|\Im\Big\langle\Big(\frac{\phi}{U+i\lambda}\Big)^{\prime\prime},\phi^{\prime\prime}\Big\rangle\Big|\leq
    C\check \nu^{-1}\Big(\|\phi^{\prime\prime}-\phi^{\prime\prime}(1)\hat{\psi}\|_2^2+\beta^{-1/2}|\phi^{\prime\prime}(1)|^2 \\ +\Big\|\frac{\phi}{U+i\lambda}\Big\|_2^2+
  \check \nu^{-1}\|\phi^\prime\|_2^2\Big)\,.
\end{multline}

For the second term on the right-hand-side \eqref{eq:880}, we use
\eqref{eq:881} to obtain 
\begin{equation}\label{eq:884}
  \alpha^2\Big|\Im\Big\langle\Big(\frac{\phi}{U+i\lambda}\Big)^{\prime},\phi^{\prime}\Big\rangle\Big|
  \leq C\, \alpha^2\check \nu^{-1/2}\Big(\Big\|\frac{\phi}{U+i\lambda}\Big\|_2\,+ 
\check \nu^{-1/2}  \|\phi^\prime\|_2\Big)\|\phi^\prime\|_2 \,. 
\end{equation}
Finally,  as $U^{\prime\prime}(x) (U(x)-\nu) >0$, we can deduce that 
\begin{equation}\label{eq:885}
  \Re\Big\langle\frac{U^{\prime\prime}\phi}{U+i\lambda},\phi\Big\rangle\geq\min_{x\in[0,1]}|U^{\prime\prime}(x)|
  \check \nu  |\Big\|\frac{\phi}{U+i\lambda}\Big\|_2^2 \,.
\end{equation}
Substituting \eqref{eq:883}-\eqref{eq:885} into \eqref{eq:880} yields 
\begin{multline*}
  \|\phi^\prime\|_2^2 + \alpha^2\|\phi\|_2^2 + \check \nu  \Big\|\frac{\phi}{U+i\lambda}\Big\|_2^2
  \leq  C \beta^{-1} \Big[\Big\|\frac{\phi}{U+i\lambda}\Big\|_2\,\|f\|_2\\  +
  \check \nu^{-1}\Big(\|\phi^{\prime\prime}-\phi^{\prime\prime}(1)\hat{\psi}\|_2^2+\beta^{-1/2}|\phi^{\prime\prime}(1)|^2  +
  [\check \nu^{-1}+\alpha^2]\|\phi^\prime\|_2^2\Big) \\ +  (\alpha^2+\check \nu^{-1})\Big\|\frac{\phi}{U+i\lambda}\Big\|_2^2\,\Big]\,.
\end{multline*}
From the above inequality we  conclude, for sufficiently large $\kappa_1$ and $\beta_0$,  
\begin{equation}
\label{eq:886}
  \|\phi^\prime\|_2 + \check \nu^{1/2}\Big\|\frac{\phi}{U+i\lambda}\Big\|_2\leq C
   \beta^{-1/2}\, \big[  \check \nu^{1/2} \|f\|_2
  + \check \nu^{-1/2} (\|\phi^{\prime\prime}-\phi^{\prime\prime}(1)\hat{\psi}\|_2+ \beta^{-1/4}|\phi^{\prime\prime}(1)|)\big] \,.
\end{equation}
Clearly
\begin{displaymath}
   (\LL_\beta^{\mathfrak z_\alpha}-\beta\lambda)(-\phi^{\prime\prime}+\alpha^2\phi)=f+i\beta U^{\prime\prime}\phi \,,
\end{displaymath}
where ${\mathfrak z_\alpha}$ is given by \eqref{eq:602}. For
sufficiently large $\kappa_1$, \eqref{eq:442} and \eqref{eq:443} hold and
since $\mu<\Upsilon\beta^{-1/2}$, we may use \eqref{eq:445} here to obtain
\begin{equation}
\label{eq:887}
  |\phi^{\prime\prime}(1)| \leq
  C(\beta^{-1/3}\check \nu^{-5/12}\|f\|_2+\beta^{1/2}\check \nu^{-1/2}\|\phi^\prime\|_2)\,.
\end{equation}
Substituting \eqref{eq:887} into \eqref{eq:886} yields for
sufficiently $\kappa_1$ and $\beta_0$
\begin{equation}
  \label{eq:888}
\|\phi^\prime\|_2 +\check \nu^{1/2}\Big\|\frac{\phi}{U+i\lambda}\Big\|_2\leq C  \beta^{-1/2}\,
\big[ \check \nu^{1/2} \|f\|_2
  +  \check \nu^{-1/2} \|\phi^{\prime\prime}-\phi^{\prime\prime}(1)\hat{\psi}\|_2\big] \,.
\end{equation}

To estimate $\|\phi^{\prime\prime}-\phi^{\prime\prime}(1)\hat{\psi}\|_2$ we set  (see
  also \eqref{eq:644}-\eqref{eq:650})
\begin{equation}
\label{eq:889}
  \hat{v}_\Df:=-\phi^{\prime\prime}+\alpha^2\phi
  +\frac{U^{\prime\prime}\phi}{U+i\lambda}+\phi^{\prime\prime}(1)\hat{\psi}=\frac{v_\Df}{U+i\lambda} \,.
\end{equation}
A simple computation yields
\begin{subequations}
\label{eq:890}
  \begin{equation} 
  (\LL_\beta^{\Nf,\Df} -\beta \lambda)\hat{v}_\Df= h\,,
\end{equation}
where
\begin{equation}
  h= -f- \left(\frac{U^{\prime\prime}\phi}{U+i
    \lambda}\right)^{\prime\prime}+\phi^{\prime\prime}(1)\hat{g}\,.
\end{equation}
\end{subequations}
By \cite[Theorem 1.3]{Hen2},  which can be applied to the even extension
of $\hat{v}_\Df$ and $h$ to $(-1,1)$,  it holds that
\begin{displaymath}
\|\hat{v}_\Df\|_2\leq C \beta^{-1/2} \|h\|_2\,.
\end{displaymath}
Hence, using again \eqref{eq:881} and \eqref{eq:562} yields, as 
$\|\phi^{\prime\prime}  \|_2 \leq   \|\phi^{\prime\prime} -\phi^{\prime\prime}(1) \hat \psi \|_2 + |\phi^{\prime\prime}(1)| \|\hat \psi\|_2\,,$
\begin{multline}
  \label{eq:891}
\|\hat{v}_\Df\|_2\leq C \beta^{-1/2} 
\Big(\|f\|_2+|\phi^{\prime\prime}(1)|\,\|\hat{g}\|_2+
\check \nu^{-1}\|\phi^{\prime\prime}-\phi^{\prime\prime}(1)\hat{\psi}\|_2 \\+ \beta^{-1/4}|\phi^{\prime\prime}(1)|+\check \nu^{-3/2}\|\phi^\prime\|_2+\check \nu^{-1}\Big\|\frac{\phi}{U+i\lambda}\Big\|_2\Big)\,.
\end{multline}
Consequently, by \eqref{eq:661} and  \eqref{eq:887} it holds that
\begin{multline}
  \label{eq:892}
\|\hat{v}_\Df\|_2\leq C \beta^{-1/2} \Big(\|f\|_2+
\check \nu^{-1}\|\phi^{\prime\prime}-\phi^{\prime\prime}(1)\hat{\psi}\|_2+ \\(\check
\nu^{-3/2}+\check \nu^{-1/2}\beta^{1/4})\|\phi^\prime\|_2+\check
\nu^{-1}\Big\|\frac{\phi}{U+i\lambda}\Big\|_2\Big)\,. 
\end{multline}
By \eqref{eq:804},
 it holds that
\begin{equation*}
\|\phi^{\prime\prime}-\phi^{\prime\prime}(1)\hat{\psi}\|_2^2+\alpha^2\|\phi^\prime\|_2^2 \leq
C(\|\tilde{v}_\Df\|_2^2
+\alpha^2|\phi^{\prime\prime}(1)|\,\|\phi^\prime\|_2\|(1-x)^{1/2}\hat{\psi}\|_1)\,.
\end{equation*}
Here we recall that
\begin{equation}\label{eq:893}
\tilde v_\Df =\hat v_\Df - \frac{U'' \phi}{U+i\lambda}\,.
\end{equation} 
Using \eqref{eq:614} and \eqref{eq:887}, we may conclude that for
sufficiently large $\kappa_1$, it holds that
\begin{equation*}
\|\phi^{\prime\prime}-\phi^{\prime\prime}(1)\hat{\psi}\|_2^2+\alpha^2\|\phi^\prime\|_2^2 \leq  \widehat
C\,\big(\|\tilde{v}_\Df\|_2^2+\alpha^2\beta^{-13/6}\check \nu
^{-5/6}\|f\|_2^2\big)\,. 
\end{equation*}
Then, we obtain from \eqref{eq:889} and  \eqref{eq:893} that
\begin{displaymath}
\|\phi^{\prime\prime}-\phi^{\prime\prime}(1)\hat{\psi}\|_2  \leq  C \big( \|\hat{v}_\Df\|_2
  + \Big\|\frac{\phi}{U+i\lambda}\Big\|_2+\beta^{-5/4}\|f\|_2\big)\,.
\end{displaymath}
By \eqref{eq:892} 
we then obtain for sufficiently large $\kappa_1$
\begin{displaymath}
\|\phi^{\prime\prime}-\phi^{\prime\prime}(1)\hat{\psi}\|_2  \leq   C\Big(\beta^{-1/2}\|f\|_2+
(\check \nu^{-3/2}\beta^{-1/2}+\check \nu^{-1/2}\beta^{-1/4})\|\phi^\prime\|_2+\Big\|\frac{\phi}{U+i\lambda}\Big\|_2\Big)\,.
\end{displaymath}
Substituting the above into \eqref{eq:888}
yields
\begin{displaymath}
  \|\phi^\prime\|_2 +\check \nu^{1/2}\Big\|\frac{\phi}{U+i\lambda}\Big\|_2\leq C \beta^{-1/2} \,
\Big[\check \nu^{1/2} \|f\|_2
  +  (\check \nu^{-2}\beta^{-1/2}+\check
  \nu^{-1}\beta^{-1/4})\|\phi^\prime\|_2+\check \nu^{-1/2} \Big\|\frac{\phi}{U+i\lambda}\Big\|_2\Big] \,.
\end{displaymath}
For sufficiently large $\kappa_1$ and $\beta_0$ we then obtain
 \eqref{eq:878}. \\

 {\em Step 2: We prove that there exists $\lambda_0>0$ and $\beta_0$ such that
 \eqref{eq:878} holds under the additional condition  $|\lambda|\geq\lambda_0$.}\\

\noindent Clearly, we must have either $\mu\leq-\lambda_0/2$ or $\nu>\lambda_0/2$. \\~\\
Consider
first the case where  $\mu\leq-\lambda_0/2$. As in \eqref{eq:748} we write
 \begin{equation}
\label{eq:894}
\Re\langle\phi,\B_{\lambda,\alpha,\beta}\phi\rangle=\|\phi^{\prime\prime}\|_2^2+|\mu|\beta
\,[\|\phi^\prime\|_2^2+\alpha^2\|\phi\|_2^2]+\beta\Im\langle U^\prime\phi,\phi^\prime\rangle \,.
  \end{equation}
Consequently, using Poincar\'e's  inequality, we obtain
\begin{equation}
\label{eq:895}
  \|\phi^\prime\|_2 \leq \frac{C}{\lambda_0}(\beta^{-1}\|f\|_2+\|\phi^\prime\|_2)\,.
\end{equation}
For sufficiently large $\lambda_0$ and $\beta_0$ we can then conclude
\eqref{eq:878}. 

For $\nu>\lambda_0/2$ we write 
\begin{equation}
\label{eq:896}
  \Im\langle\phi,\B_{\lambda,\alpha,\beta}\phi\rangle=\beta(-\langle (U-\nu)\phi^\prime,\phi^\prime\rangle+\alpha^2\langle
  (U-\nu)\phi^\prime,\phi^\prime\rangle-\Re \langle U^\prime\phi,\phi^\prime\rangle -\langle U^{\prime\prime}\phi,\phi\rangle)\,. 
\end{equation}
Using Poincar\'e's inequality we then obtain
\begin{displaymath}
   \|\phi^\prime\|_2 \leq \frac{C}{\lambda_0-U(0)}(\beta^{-1}\|f\|_2+\|\phi^\prime\|_2)\,,
\end{displaymath}
which validates \eqref{eq:878} for sufficiently large $\lambda_0$ and
$\beta_0$. \\ 

 {\em Step 3: We prove that there exist positive $\beta_0$,
$\tilde{\alpha}_0$, and $\alpha_0$  such that
 \eqref{eq:878} holds for
 $\tilde{\alpha}_0\beta^{1/4}\leq\alpha\leq\alpha_0\beta^{1/3}$. }\\ 

\noindent An integration by parts yields, in view of \eqref{eq:649} (see also step
2 of the proof of \ref{lem:nearly-quadratic}) 
\begin{displaymath}
  \langle\phi,\tilde{v}_\Df\rangle=\|\phi^\prime\|_2^2+\alpha^2\|\phi\|_2^2-\phi^{\prime\prime}(1)\langle\phi,\hat{\psi}\rangle\,.
\end{displaymath}
Since by \eqref{eq:614} 
\begin{displaymath}
  |\langle\phi,\hat{\psi}\rangle|\leq\|\phi\|_\infty\|\hat{\psi}\|_1\leq
  C\beta^{-1/2}\|\phi\|_\infty \,,
\end{displaymath}
we obtain 
\begin{displaymath}
  \|\phi^\prime\|_2^2+\alpha^2\|\phi\|_2^2\leq\|\phi\|_2\|\tilde{v}_\Df\|_2+
  C\beta^{-1/2}\|\phi\|_\infty|\phi^{\prime\prime}(1)|\,.
\end{displaymath} 
We can now conclude that
\begin{equation}
\label{eq:897}
  \|\phi^\prime\|_2^2+\alpha^2\|\phi\|_2^2\leq C(\alpha^{-2}\|\tilde{v}_\Df\|_2^2+\beta^{-1/2}|\phi^{\prime\prime}(1)|\,\|\phi\|_\infty)\,,
\end{equation}
yielding
\begin{displaymath}
  \|\phi\|_\infty^2\leq \|\phi^\prime\|_2 \|\phi\|_2\leq\frac{1}{2\alpha}(
  \|\phi^\prime\|_2^2+\alpha^2\|\phi\|_2^2)\leq C(\alpha^{-3}\|\tilde{v}_\Df\|_2^2+\alpha^{-1}\beta^{-1/2}|\phi^{\prime\prime}(1)|\,\|\phi\|_\infty)\,.
\end{displaymath} 
We may then infer that
\begin{displaymath}
  \|\phi\|_\infty\leq C(\alpha^{-3/2}\|\tilde{v}_\Df\|_2+\alpha^{-1}\beta^{-1/2}|\phi^{\prime\prime}(1)|)\,.
\end{displaymath}
By \eqref{eq:799} and \eqref{eq:801} (both remain valid in the present
case)
we obtain
\begin{displaymath}
   \|\phi\|_\infty\leq C\tilde{\alpha}_0^{-1}(\|\phi\|_\infty + \beta^{-1/8}\|\phi^\prime\|_2+\beta^{-7/8}\|f\|_2)\,.
\end{displaymath} 
For sufficiently large $\tilde \alpha_0$ it follows that
\begin{equation}\label{eq:892abc}
  \|\phi\|_\infty\leq C\tilde{\alpha}_0^{-1}(\beta^{-1/8}\|\phi^\prime\|_2+\beta^{-7/8}\|f\|_2)\,.
\end{equation}
Consequently, using \eqref{eq:799} and \eqref{eq:801} once again, it
holds that
\begin{equation}
  \label{eq:898}
 |\phi^{\prime\prime}(1)|+ \beta^{3/8}\|\tilde{v}_\Df\|_2\leq
    C \,\big(\beta^{-1/8}\|f\|_2+\beta^{5/8}\|\phi^\prime\|_2\big)\,.
\end{equation} 
Substituting \eqref{eq:892abc}  into \eqref{eq:897} then leads to 
\begin{displaymath}
 \|\phi^\prime\|_2^2\leq C(\alpha^{-2}\|\tilde{v}_\Df\|_2^2+\tilde\alpha_0^{-1}\beta^{-5/8}|\phi^{\prime\prime}(1)|\,[\|\phi^\prime\|_2+\beta^{-3/4}\|f\|_2])\,.
\end{displaymath}
 
Making use of \eqref{eq:898} we then obtain that
\begin{displaymath}
  \|\phi^\prime\|_2\leq C\tilde{\alpha}_0^{-1/2}(\|\phi^\prime\|_2+\beta^{-3/4}\|f\|_2)\,.
\end{displaymath}
For sufficiently large $\tilde \alpha_0$ we can then conclude \eqref{eq:878} under the conditions of this step. \\
The proposition is proved.
\end{proof}

We continue by introducing, for some positive constants  $\alpha_0$,  
$\Upsilon$, $\beta_0$ and $\kappa_2$,  the zone
\begin{multline}
\mathcal F_2(\alpha_0, \beta_0,\Upsilon,\kappa_2):=\\\big\{ (\lambda,\alpha,\beta)\in \mathbb C\times\mathbb R_+^2\,,\,\beta \geq \beta_0\,,\, 
  0\leq \alpha\leq\alpha_0\beta^{1/3}\,,\,  \mu<\Upsilon\beta^{-1/2} \,,\,  \nu \leq -\kappa_2 \beta^{-1/3} \big\}\,.
\end{multline}

\begin{proposition}
\label{large-nu-negative}
Let $U\in C^4([0,1])$ satisfy \eqref{eq:10}. Let further $\alpha_0$ and
$\Upsilon$ denote positive constants. Then, there exist $\beta_0>0$, $\kappa_2>0$,
and $C>0$, such that for all $(\lambda,\alpha,\beta) \in \mathcal F_2(\alpha_0,
\beta_0,\Upsilon,\kappa_2)$ it holds
   \begin{equation}
\label{eq:899}
 \big\|(\B_{\lambda,\alpha,\beta}^{\mathfrak N,\Df})^{-1}\big\|+
       \Big\|\frac{d}{dx}\, (\B_{\lambda,\alpha,\beta}^{\mathfrak
         N,\Df})^{-1}\Big\|\leq C  \beta^{-1} |\nu|^{-1}\log\beta \,.
   \end{equation}
\end{proposition}
\begin{proof}~\\
Taking the scalar product of \eqref{eq:879} with $w=(U-\nu)^{-1}\phi$,
  and integrating by parts yields for the imaginary part (see
  \eqref{eq:880})
  \begin{multline*}
    -\Im\langle w,f\rangle =
    \Im\langle w^{\prime\prime},\phi^{\prime\prime}\rangle
   +  \alpha^2\Im\langle w^\prime,\phi^\prime\rangle \\ +
    \beta\Big(\|(U-\nu)w^\prime\|_2^2 + \alpha^2\|\phi\|_2^2 
    \alpha^2\|\phi\|_2^2-|\mu|^2\Big\langle\frac{\phi}{U-\nu},\frac{U^{\prime\prime}\phi}{|U+i\lambda|^2}\Big\rangle\Big) \,.
  \end{multline*}
Hence, since $ (U-\nu)^{-1}U^{\prime\prime}<0$ we obtain that for any
$\delta>0$ there exists $C>0$ such that 
\begin{multline*}
  \frac{1}{2}(\|(U-\nu)w^\prime\|_2^2+\nu^2\|w^\prime\|_2^2) + \alpha^2\|\phi\|_2^2\\
  \leq C\big(\delta^{-1}\beta^{-2}\|f\|_2^2 +\delta\|w\|_2^2 +
  \alpha^2[\beta^{-4/3}\|w^\prime\|_2^2+\beta^{-2/3}\|\phi^\prime\|_2^2]+\beta^{-1}|\Im\langle w^{\prime\prime},\phi^{\prime\prime}\rangle|\big)
\end{multline*}

 {\em Step 1: With $\lambda_0>0$, we prove \eqref{eq:899} for $(\lambda,\alpha,\beta) \in \mathcal F_2(\alpha_0,
\beta_0,\Upsilon,\kappa_2)$ satisfying $|\lambda|\leq\lambda_0$.}\\

 \noindent  {\em Step 1a: We estimate $\hat{v}_\Df=\frac{v_\Df}{U+i\lambda}$ (as defined by
  \eqref{eq:889}).}\noindent \\

By \eqref{eq:890}
  and \eqref{eq:287} it holds that
\begin{equation}\label{eq:900}
  \|\hat{v}_\Df\|_2\leq C\, \beta^{-2/3}\, \|h\|_2\,,
\end{equation}
where $h$ is given by (\ref{eq:890}b).
 Hence
\begin{multline}\label{eq:901}
 \|h\|_2\leq C \Big(\|f\|_2+|\phi^{\prime\prime}(1)|\,\|\hat{g}\|_2+   \Big\|\frac{\phi}{(U+i\lambda)^3}\Big\|_2 + 
    \Big\|\frac{\phi^\prime}{(U+i\lambda)^2}\Big\|_2 + 
    \Big\|\frac{\phi^{\prime\prime}}{(U+i\lambda)}\Big\|_2 \Big)\\
    \leq C \Big(\|f\|_2+|\phi^{\prime\prime}(1)|\,\|\hat{g}\|_2 + |\lambda|^{-1} \|\phi''\|_2 +  \Big\|\frac{\phi}{(U+i\lambda)^3}\Big\|_2 + \Big\|\frac{\phi^\prime}{(U+i\lambda)^2}\Big\|_2   \Big) \,.
    \end{multline}
    To estimate the last two terms we use a decomposition of the
    interval of integration. 
    (Note that for $\nu\leq -1$ we have $\|\cdot
    \|_{L^2(0,1-|\nu|^{1/2})}=0$ and $\|\cdot\|_{L^2(1-|\nu|^{1/2},1)}=\|\cdot\|_2$.)
  In addition, we need the bounds
\begin{displaymath}
\| (U+ i\lambda)^{-m}\|_{L^\infty(0,1-|\nu|^{1/2})} \leq C |\nu|^{-m/2}\,,
\end{displaymath}
and 
\begin{displaymath}
\| (U+ i\lambda)^{-m}\|_{L^\infty(0,1)} \leq C |\lambda|^{-m}\,.
\end{displaymath}
We now write,
\begin{multline*}
 \|h\|_2\leq C \Big(\|f\|_2+|\phi^{\prime\prime}(1)|\,\|\hat{g}\|_2+
|\lambda|^{-1}\|\phi^{\prime\prime}\|_2+|\lambda|^{-2}\|\phi^\prime\|_{L^2(1-|\nu|^{1/2},1)}+|\nu|^{-1}\|\phi^\prime\|_{L^2(0,1-|\nu|^{1/2})}
\\  + |\lambda|^{-2}\Big\|\frac{\phi}{U+i\lambda}\Big\|_{L^2(1-|\nu|^{1/2},1)} +   
 |\nu|^{-1}\Big\|\frac{\phi}{U+i\lambda}\Big\|_{L^2(0,1-|\nu|^{1/2})}\Big)\,.
\end{multline*}
By Hardy's inequality \eqref{eq:18} it holds that
\begin{equation*}
\Big\|\frac{\phi}{U+i\lambda}\Big\|_{L^2(0,1-|\nu|^{1/2})} \leq 
  \Big\|\frac{\phi}{U}\Big\|_{L^2(0,1-|\nu|^{1/2})} \leq C \Big\|\frac{\phi}{1-x}\Big\|_{L^2(0,1-|\nu|^{1/2})} \leq \widehat C \, \|\phi^\prime\|_{L^2(0,1-|\nu|^{1/2})}\,.
\end{equation*}
On the interval $(1-|\nu|^{1/2},1)$, we have again by  \eqref{eq:18}
\begin{displaymath}
  \Big\|\frac{\phi}{U+i\lambda}\Big\|_{L^2(1-|\nu|^{1/2},1)} \leq  C\, \|\phi^\prime\|_2 \,.
\end{displaymath}
Combining the above yields
\begin{multline}\label{eq:902}
 \|h\|_2\leq C \Big(\|f\|_2+|\phi^{\prime\prime}(1)|\,\|\hat{g}\|_2+
|\lambda|^{-1}\|\phi^{\prime\prime}\|_2+|\lambda|^{-2}\|\phi^\prime\|_{L^2(1-|\nu|^{1/2},1)}+|\nu|^{-1}
  \|\phi^\prime\|_2
\Big)\,.
\end{multline}

By \eqref{eq:657} and  \eqref{eq:661}
it holds that
\begin{equation}\label{eq:903}
  |\phi^{\prime\prime}(1)|\,\|\hat{g}\|_2\leq C|\lambda|^{-3/4}(\beta^{1/4}\|\phi^\prime\|_2+\beta^{-7/12}\|f\|_2)\,.
\end{equation}
Substituting \eqref{eq:902} and \eqref{eq:903} into \eqref{eq:901} we obtain 
\begin{multline}
\label{eq:904}
    \|h\|_2\leq
C\big(\|f\|_2+ |\lambda|^{-1}\|\phi^{\prime\prime}\|_2 \\
+[|\nu|^{-1}+\beta^{1/4}|\lambda|^{-3/4}]\|\phi^\prime\|_2+|\lambda|^{-2}\|\phi^\prime\|_{L^2(1-|\nu|^{1/2},1)}\big)\,.
\end{multline}
To bound $\|\phi^{\prime\prime}\|_2$ we  use
  \eqref{eq:839}-\eqref{eq:840}  and \eqref{eq:893} to obtain
\begin{displaymath}
  \|\phi^{\prime\prime}\|_2\leq \|\phi^{\prime\prime}-\alpha^2\phi\|_2\leq
  \|\hat{v}_\Df\|_2+\|U^{\prime\prime}\|_\infty\Big\|\frac{\phi}{U+i\lambda}\Big\|_2+|\phi^{\prime\prime}(1)|\,\|\hat{\psi}\|_2\,.
\end{displaymath}
By \eqref{eq:562}, (\ref{eq:900}), \eqref{eq:904}, \eqref{eq:18}, and \eqref{eq:657}
(repeatedly  using the lower bound  $|\nu| \geq \kappa_2 \beta^{-1/3}$) it holds that
\begin{displaymath}
  \|\phi^{\prime\prime}\|_2\leq 
C\big(|\lambda|^{1/4}\beta^{-7/12}\|f\|_2+ \beta^{-2/3}|\lambda|^{-1}\|\phi^{\prime\prime}\|_2+|\lambda\beta|^{1/4}\|\phi^\prime\|_2)\,.
\end{displaymath}
For sufficiently large $\beta_0$ we then obtain
\begin{equation}\label{eq:905}
  \|\phi^{\prime\prime}\|_2\leq 
C\big(|\lambda|^{1/4}\beta^{-7/12}\|f\|_2+|\lambda\beta|^{1/4}\|\phi^\prime\|_2)\,.
\end{equation}
Substituting \eqref{eq:905}  into \eqref{eq:904} yields
\begin{equation}
  \label{eq:906}
    \|h\|_2\leq C \big(\|f\|_2
+[|\nu|^{-1}+\beta^{1/4}|\lambda|^{-3/4}]\|\phi^\prime\|_2+|\lambda|^{-2}\|\phi^\prime\|_{L^2(1-|\nu|^{1/2},1)}\big)\,.
\end{equation}
To use (\ref{eq:170}b) we must provide an estimate for 
  \begin{displaymath}
    N(v_\Df,\lambda)=\|(1-x)^{1/2}(U+i\lambda)^{-1}v_\Df\|_1=\|(1-x)^{1/2}\hat{v}_\Df\|_1\,.
  \end{displaymath}
Thus, we use \eqref{eq:354} and \eqref{eq:890} to obtain
\begin{multline*}
 \|(1-x)^{1/2}\hat{v}_\Df\|_1\leq  C\, \|(U-\nu)^{1/2}\hat{v}_\Df\|_1\\
 \leq C \,  \|(U-\nu)^{-1/2}\|_2 \,
  \|(U-\nu)\hat{v}_\Df\|_2\leq \widehat C\, \frac{\log \beta}{\beta}\|h\|_2\,.
\end{multline*}
Hence, we may conclude from \eqref{eq:906} that
\begin{multline}
\label{eq:907}
  \|(1-x)^{1/2}\hat{v}_\Df\|_1\leq   C \,\frac{\log \beta}{\beta} \big(\|f\|_2
+[|\nu|^{-1}+\beta^{1/4}|\lambda|^{-3/4}]\|\phi^\prime\|_2+|\lambda|^{-2}\|\phi^\prime\|_{L^2(1-|\nu|^{1/2},1)}\big)\,.
\end{multline}

 {\em Step 1b: We prove \eqref{eq:899}}\\

As in \eqref{eq:673},  we let $\phi=\phi_\Df+\check{\phi}$ where
\begin{displaymath}
\phi_\Df=\A_{\lambda,\alpha}^{-1}([U+i\lambda]\hat{v}_\Df)= \A_{\lambda,\alpha}^{-1}v_\Df 
\end{displaymath} 
and 
\begin{displaymath}
\check{\phi}=\A_{\lambda,\alpha}^{-1}(\phi^{\prime\prime}(1)[U+i\lambda]\hat{\psi})\,.
\end{displaymath}  
By
\eqref{eq:171} applied to the pair $(\check
  \phi,\phi^{\prime\prime}(1)[U+i\lambda]\hat{\psi})$ it holds that
\begin{displaymath}
  |\check{\phi}(x)|\leq C(1-x)^{1/2}[1+\nu^{-1/2}(1-x)^{1/2}] |\phi^{\prime\prime}(1)\langle\check{\phi},\hat{\psi}\rangle|^{1/2} \,.
\end{displaymath}
Hence,  integrating over $(0,1)$,
\begin{displaymath}
   |\langle\check{\phi},\hat{\psi}\rangle|\leq
   C(\|(1-x)^{1/2}\hat{\psi}\|_1+|\nu|^{-1/2}\|(1-x)\hat{\psi}\|_1)|\langle\check{\phi},\hat{\psi}\rangle|^{1/2}|\phi^{\prime\prime}(1)|^{1/2} \,,
\end{displaymath}
which implies
\begin{displaymath}
   |\langle\check{\phi},\hat{\psi}\rangle|\leq
   C(\|(1-x)^{1/2}\hat{\psi}\|_1+|\nu|^{-1/2}\|(1-x)\hat{\psi}\|_1)^2
   |\phi^{\prime\prime}(1)| \,. 
\end{displaymath}
By \eqref{eq:614} and \eqref{eq:657} 
we then obtain
\begin{equation}
\label{eq:908}
   |\langle\check{\phi},\hat{\psi}\rangle|\leq C\, |\lambda\beta|^{-1}\, (\|\phi^\prime\|_2+\beta^{-5/6}\|f\|_2)\,.
\end{equation}
Using (\ref{eq:172}), applied to the pair $(\check
\phi,\phi^{\prime\prime}(1)[U+i\lambda]\hat{\psi})$, together with \eqref{eq:175},
\eqref{eq:908}, and \eqref{eq:657} yields
\begin{displaymath}
  \|\check{\phi}^\prime\|_2^2 \leq C\, |\nu|^{-1}|\lambda|^{-1/2} \beta^{-1/2}\, (\|\phi^\prime\|_2+\beta^{-5/6}\|f\|_2)^2\,.
\end{displaymath}
For sufficiently large $\kappa_2$ (and $(\lambda,\alpha,\beta) \in \mathcal F_2(\alpha_0,
\beta_0,\Upsilon,\kappa_2)$) we then conclude that 
\begin{equation}
\label{eq:909}
   \|\check{\phi}^\prime\|_2 \leq C\, |\nu|^{-3/4}\beta^{-1/4}\, (\|\phi_\Df^\prime\|_2+\beta^{-5/6}\|f\|_2)\,.
\end{equation}
By \eqref{eq:177} and \eqref{eq:172}, applied again to the pair $(\check
\phi,\phi^{\prime\prime}(1)[U+i\lambda]\hat{\psi})$, we then obtain that
\begin{displaymath}
   \|\check{\phi}^\prime\|_{L^2(1-|\nu|^{1/2},1)}^2\leq C\, |\nu|^{-1/2}\, |\phi^{\prime\prime}(1)|\,   |\langle\check{\phi},\hat{\psi}\rangle| 
\end{displaymath}
Using \eqref{eq:908}, 
we then conclude that
\begin{equation}
\label{eq:910}
   \|\check{\phi}^\prime\|_{L^2(1-|\nu|^{1/2},1)}\leq C\, |\nu|^{-1/2}\beta^{-1/4}\, (\|\phi_\Df^\prime\|_2+\beta^{-5/6}\|f\|_2)\,.
\end{equation}

By
(\ref{eq:170}a) applied to the pair $(\phi_\Df, (U+i\lambda) \hat v_\Df)$,
\eqref{eq:907},
 \eqref{eq:909}, and  \eqref{eq:910}, 
it holds, as $\phi=\phi_\Df+\check{\phi}$,  that 
\begin{multline}\label{eq:911}
\|\phi_\Df\|_{1,2}\leq
C |\nu|^{-1} \|(1-x)^{1/2}\hat{v}_\Df\|_1
\leq 
\widehat C\frac{\log \beta }{\beta} |\nu|^{-1} \Big(\|f\|_2\\
+[|\nu|^{-1}+\beta^{1/4}|\nu|^{-3/4}+  |\nu|^{-5/2}\beta^{-1/4}]\|\phi^\prime_\Df \|_2+|\lambda|^{-2}\|\phi^\prime_\Df\|_{L^2(1-|\nu|^{1/2},1)}\Big)\,.
\end{multline}
For sufficiently large  $\beta_0$ we obtain from \eqref{eq:911} 
\begin{equation}
\label{eq:912}
  \|\phi_\Df\|_{1,2}\leq C
\frac{\log \beta}{\beta} |\nu|^{-1} \big(\|f\|_2+|\lambda|^{-2}\|\phi^\prime_\Df\|_{L^2(1-|\nu|^{1/2},1)}\big)\,.
\end{equation}
 For $\nu<-1/2$, we immediately obtain from
  \eqref{eq:912} that for sufficiently large  $\beta_0$ 
  \begin{equation}\label{eq:913}
      \|\phi_\Df\|_{1,2}\leq C
\frac{\log \beta}{\beta} |\nu|^{-1} \|f\|_2
  \end{equation}
For $\nu \geq -1/2$, we substitute \eqref{eq:912} into \eqref{eq:907} to conclude,
  with the  aid of \eqref{eq:909} and \eqref{eq:910}
\begin{displaymath} 
   \|(1-x)^{1/2}\hat{v}_\Df\|_1\leq C\frac{\log \beta}{\beta} \big(\|f\|_2
   +|\lambda|^{-2}\|\phi^\prime_\Df\|_{L^2(1-|\nu|^{1/2},1)}\big)\,.
\end{displaymath}
We can now use (\ref{eq:170}b) applied to the pair $(\phi_\Df, (U+i\lambda) \hat
v_\Df)$ to obtain 
\begin{displaymath}
  \|\phi^\prime_\Df\|_{L^2(1-|\nu|^{1/2},1)}\leq C
\frac{\log \beta}{\beta} |\nu|^{-3/4}\big(\|f\|_2+|\lambda|^{-2}\|\phi^\prime_\Df\|_{L^2(1-|\nu|^{1/2},1)}\big)\,.
\end{displaymath}
For sufficiently large $\beta_0$ we obtain that
\begin{displaymath}
  \|\phi^\prime_\Df\|_{L^2(1-|\nu|^{1/2},1)}\leq C
\frac{\log \beta}{\beta}|\nu|^{-3/4}\|f\|_2\,,
\end{displaymath}
which, when substituted into \eqref{eq:912}, yields in the case $-1/2 \leq \nu$
\begin{displaymath}
 \|\phi_\Df\|_{1,2}\leq C
\frac{\log \beta}{\beta} |\nu|^{-1} \,\|f\|_2\,.
\end{displaymath}
The above inequality, together with \eqref{eq:909} and \eqref{eq:913}
yields \eqref{eq:899} for $|\lambda|<\lambda_0$. \\

 {\em Step 2: We prove that there exists $\lambda_0>0$ such that
  \eqref{eq:899} holds true for any $(\lambda,\alpha,\beta) \in \mathcal F_2(\alpha_0,
\beta_0,\Upsilon,\kappa_2)$ satisfying $|\lambda|>\lambda_0$.}\\

\noindent The proof is almost identical with Step 2 of Proposition
\ref{lem:large-nu}. If $\mu<-\lambda_0/2$ we obtain \eqref{eq:895} from
\eqref{eq:894} and then \eqref{eq:899} for sufficiently large $\lambda_0$.
If $\nu<-\lambda_0/2$ we use \eqref{eq:896} to obtain \eqref{eq:895} once again.
\end{proof}

 \begin{remark}
  Note that there exists $\mu_1>0$ such that for all $\mu<-\mu_1\beta^{-1/3}$  \eqref{eq:909}
  remains valid even in the case where $\kappa_2$ is not necessarily
  large. Thus, we may conclude that under the conditions of
  Proposition \ref{large-nu-negative} for all $\kappa_2>0$, there exist $\beta_0>0$, $\mu_1>0$,
and $C>0$, such that for all $(\lambda,\alpha,\beta) \in \mathcal F_2(\alpha_0,
\beta_0,\Upsilon,\kappa_2)$ satisfying $\mu<-\mu_1\beta^{-1/3}$ \eqref{eq:899} holds
true. 
\end{remark}

\section{Proof of the main theorems}\label{sec:6}
The proofs of Theorem \ref{thm:small-alpha} and Theorem
\ref{thm:large-alpha} rely on a combination of the relevant results in Section
5. 

\begin{proof}[Proof of Theorem \ref{thm:small-alpha}]

We present the proof in  the following table, which gives the precise
range of parameters where each estimate is valid together with the
estimate itself 
\vspace{2ex}

\begin{tabular}{|c|c|c|c|c|}
\hline
  $\alpha$ & $\nu$ & $\mu$ & Stated in & Estimate \\
  \hline
  $\alpha\lesssim \beta^{1/3}$ & $-\nu\gg\beta^{-1/3}$ & $\mu\leq\beta^{-1/2}$ &Prop.
  \ref{large-nu-negative} &  $\beta^{-1} |\nu|^{-1}\log\beta$ \\
\hline
$\alpha\ll\beta^{-1/6}$ & $|\nu|\leq \nu_0$ & $\mu<\beta^{-1/2}$ &Prop.
\ref{cor:small-alpha}~&  $\beta^{-2/3}$ \\
\hline 
$\alpha\lesssim \beta^{1/3}$ &  
\(
\begin{array}{c}
 \beta^{-1/2}\ll U(0)-\nu \\
 \nu\geq\nu_0 
\end{array}
\)
& $\mu\ll\beta^{-1/2}$ &  Prop.
\ref{lem:nearly-quadratic} &  $\beta^{-1/2}\log \beta$\\
\hline 
$\alpha\lesssim \beta^{1/3}$ & $|U(0)-\nu|\lesssim\beta^{-1/2}$ & $\mu\ll\beta^{-1/2}$ & Prop.
\ref{lem:quadratic} &  $\min(\beta^{-5/8}|\mu|^{-1/4},\beta^{-3/8})$ \\
\hline
$\alpha\lesssim \beta^{1/3}$ & $\nu-U(0)\gg \beta^{-1/2}$ & $\mu\lesssim\beta^{-1/2}$ &Prop.
\ref{lem:large-nu} &  $\beta^{-1/2}$ \\
\hline
$\alpha\gtrsim\beta^{1/3}$ & $\nu\in\R$ & $ \mu\leq\mu^*$ & Prop.
\ref{lem:large-alpha}~ &  $\beta^{-1/2}$ \\
\hline
\end{tabular}
{ where $\mu^*=\min(\Upsilon\beta^{-1/2},[\mu_m-\hat{\Upsilon}]\beta^{-1/3}-\alpha^2\beta^{-1}/2)$
for some sufficiently small $\Upsilon>0$ and any $\hat{\Upsilon}>0$}
~\\

From the following table we learn that there exist positive $\alpha_L$,
$\beta_0$ and $\Upsilon$ such that for all $\beta>\beta_0$
   \begin{equation}
 \label{eq:915}
        \sup_{
          \begin{subarray}{c}
           0\leq \alpha \leq \alpha_L\beta^{-1/6}\\
       \Re \lambda \leq \Upsilon\beta^{-1/2} 
             \end{subarray}}
  \big\|(\B_{\lambda,\alpha,\beta}^{\Df,sym})^{-1}\big\|+
        \big\|\frac{d}{dx}\, (\B_{\lambda,\alpha,\beta}^{ \Df,sym})^{-1}\big\|\leq C\beta^{-3/8} \,.
   \end{equation}
Furthermore, for $\mu=\Upsilon\beta^{-1/2}$ it holds for all $\beta>\beta_0$ that 
   \begin{equation}
 \label{eq:916}
        \sup_{
          \begin{subarray}{c}
           0\leq \alpha \leq \alpha_L\beta^{-1/6}\\
       \Re \lambda =\Upsilon\beta^{-1/2} 
             \end{subarray}}
  \big\|(\B_{\lambda,\alpha,\beta}^{\Df,sym})^{-1}\big\|+
        \big\|\frac{d}{dx}\, (\B_{\lambda,\alpha,\beta}^{\Df,sym})^{-1}\big\|\leq
        C\beta^{-1/2}\log\beta  \,.
   \end{equation}
By \eqref{eq:915} $\B_{\lambda,\alpha}^{\Df,sym}$ depends holomorphically on $\lambda$ for all
$\mu\leq\Upsilon\beta^{-1/2}$, and hence we can use \eqref{eq:916} together with the
Phragm\'en-Lindel\"of Theorem  to obtain \eqref{eq:8}.
\end{proof}

\begin{proof}[Proof of Theorem \ref{thm:large-alpha}]
  As in the proof of Theorem \ref{thm:small-alpha}, we use the
  following table, which gives the precise range of parameters where
  each estimate is valid together with the estimate itself
  \vspace{2ex}

\begin{tabular}{|c|c|c|c|c|}
\hline
  $\alpha$ & $\nu$ & $\mu$ & Stated in & Estimate \\
  \hline
  $\alpha\lesssim \beta^{1/3}$ & $-\nu  \gg \beta^{-1/3}$ & $\mu\leq\beta^{-1/2}$ &Prop.
  \ref{large-nu-negative} &  $\beta^{-1} |\nu|^{-1}\log\beta$ \\
\hline
$\alpha\lesssim1$ & $\beta^{-1/5+\delta}\leq\nu<\nu_0$ & $\mu<\beta^{-2/5-\delta}$ &Prop. \ref{lem:right-curve}
 &  $\beta^{-1/2+\delta}$ \\
\hline
 $1\ll\alpha\ll\beta^{1/3}$ &  $|\nu|<\nu_0$ & $\mu\ll\beta^{-1/3}$ &
Prop. \ref{lem:intermediate-alpha}
 &   $\beta^{-5/6}$ \\
\hline
$\beta^{-1/10+\delta}\leq\alpha\lesssim1$ & $|\nu|\leq \beta^{-1/5+\delta}$ & $\mu<\beta^{-1/3-\delta}$ &Prop. \ref{lem:small-lambda-intermediate-alpha}
 &  $\beta^{-1/2+\delta}$ \\
\hline 
$\alpha\lesssim \beta^{1/3}$ & 
\(
\begin{array}{c}
 \beta^{-1/2}\ll U(0)-\nu \\
 \nu\geq\nu_0 
\end{array}
\)
& $\mu\ll\beta^{-1/2}$ & 
Prop. 
\ref{lem:nearly-quadratic} &  $\beta^{-1/2}\log \beta$\\
\hline 
$\alpha\lesssim \beta^{1/3}$ & $|U(0)-\nu|\lesssim\beta^{-1/2}$ & $\mu\ll\beta^{-1/2}$ & Prop.
\ref{lem:quadratic} &  $\min(\beta^{-5/8}|\mu|^{-1/4},\beta^{-3/8})$ \\
\hline
$\alpha\lesssim \beta^{1/3}$ & $\nu-U(0)\gg \beta^{-1/2}$ & $\mu\lesssim\beta^{-1/2}$ & Prop.
\ref{lem:large-nu} &  $\beta^{-1/2}$ \\
\hline
$\alpha\gtrsim\beta^{1/3}$ & $\nu\in\R$ & $\mu\lesssim\beta^{-1/2}$ & Prop.
\ref{lem:large-alpha} &  $\beta^{-1/2}$ \\
\hline
\end{tabular}
\vspace{2ex}
\\ 
From the following table we learn that there exist positive
$\beta_0$ and $\Upsilon$ such that for all $\beta>\beta_0$
   \begin{equation}
\label{eq:917}
        \sup_{
          \begin{subarray}{c}
           \beta^{-1/10+\delta }\leq\alpha \\
       \Re \lambda \leq \Upsilon\beta^{-1/2} 
             \end{subarray}}
  \big\|(\B_{\lambda,\alpha,\beta}^{\Df,sym})^{-1}\big\|+
        \big\|\frac{d}{dx}\, (\B_{\lambda,\alpha,\beta}^{ \Df,sym})^{-1}\big\|\leq C\beta^{-3/8} \,.
   \end{equation}
Furthermore, for $\mu=\Upsilon\beta^{-1/2}$ it holds for any  $\delta>0$ and
$\beta>\beta_0$ that  
   \begin{equation}
 \label{eq:918}
        \sup_{
          \begin{subarray}{c}
      \beta^{-1/10+\delta }\leq\alpha \\
       \Re \lambda =\Upsilon\beta^{-1/2} 
             \end{subarray}}
  \big\|(\B_{\lambda,\alpha,\beta}^{\Df,sym})^{-1}\big\|+
        \big\|\frac{d}{dx}\, (\B_{\lambda,\alpha,\beta}^{\Df,sym})^{-1}\big\|\leq
        C\beta^{-1/2+\delta} \,.
   \end{equation}
By \eqref{eq:917} $B_{\lambda,\alpha,\beta}^{\Df,sym}$ depends holomorphically on $\lambda$ for all
$\mu\leq\Upsilon\beta^{-1/2}$, and hence we can use \eqref{eq:918} together with the
Phragm\'en-Lindel\"of Theorem to obtain \eqref{eq:9}. 
\end{proof}

\end{document}